\documentclass[12pt, reqno]{amsart}
\usepackage{amsmath, amsthm, amscd, amsfonts, amssymb, graphicx, color, mathrsfs}
\usepackage[bookmarksnumbered, colorlinks, plainpages]{hyperref}
\usepackage[all]{xy}
\usepackage{slashed}
\usepackage{tikz-cd}
\usepackage{mathabx}
\usepackage{tipa}
\usepackage{soul}
\usepackage{cancel}
\usepackage{ulem}

\newtheorem{theorem}{Theorem}[section]
\newtheorem{lemma}[theorem]{Lemma}

\newtheorem{proposition}[theorem]{Proposition}
\newtheorem{corollary}[theorem]{Corollary}
\theoremstyle{definition}
\newtheorem{definition}[theorem]{Definition}
\newtheorem{example}[theorem]{Example}

\newtheorem{cri}[theorem]{Criterion}

\theoremstyle{remark}
\newtheorem{remark}[theorem]{Remark}
\numberwithin{equation}{section}

\begin{document}
\setcounter{page}{1}

\title[ Subelliptic pseudo-differential operators ]{ Subelliptic pseudo-differential operators and Fourier integral operators  on compact Lie groups}

\author[D. Cardona]{Duv\'an Cardona}
\address{
  Duv\'an Cardona:
  \endgraf
  Department of Mathematics: Analysis, Logic and Discrete Mathematics
  \endgraf
  Ghent University, Belgium
  \endgraf
  {\it E-mail address} {\rm duvanc306@gmail.com}
  }

\author[M. Ruzhansky]{Michael Ruzhansky}
\address{
  Michael Ruzhansky:
  \endgraf
  Department of Mathematics: Analysis, Logic and Discrete Mathematics
  \endgraf
  Ghent University, Belgium
  \endgraf
 and
  \endgraf
  School of Mathematical Sciences
  \endgraf
  Queen Mary University of London
  \endgraf
  United Kingdom
  \endgraf
  {\it E-mail address} {\rm ruzhansky@gmail.com}
  }

\subjclass[2020]{Primary 22E30; Secondary 58J40}.

\keywords{Sub-Laplacian, Compact Lie group, Pseudo-differential operator, Fourier analysis, Fourier integral operator}

\thanks{The authors are supported by the FWO Odysseus 1 grant G.0H94.18N: Analysis and Partial Differential Equations and by the Methusalem programme of the Ghent University Special Research Fund (BOF) (Grant number 01M01021). MR is also supported in parts by the EPSRC grant
EP/R003025/2.}

\begin{abstract} In this memoir we extend the theory of global pseudo-differential operators to the setting of arbitrary sub-Riemannian structures on a compact Lie group. More precisely, 
given a compact Lie group $G$, and the sub-Laplacian $\mathcal{L}$ associated to a system of vector fields $X=\{X_1,\cdots,X_k\}$ satisfying the H\"ormander condition, we introduce a (subelliptic) pseudo-differential calculus associated to $\mathcal{L},$ based on the matrix-valued quantisation process developed in \cite{Ruz}. This theory will be developed as follows. First, we will investigate the singular kernels of this calculus,  estimates of  $L^p$-$L^p$,  $H^1$-$L^1$, $L^\infty$-$BMO$ type and also the weak (1,1) boundedness of these subelliptic H\"ormander  classes. Between the obtained  estimates we prove subelliptic versions of the celebrated sharp Fefferman $L^p$-theorem  and  the Calder\'on-Vaillancourt theorem.  The obtained estimates will be used to establish the boundedness of subelliptic operators on subelliptic Sobolev and Besov spaces. We will investigate the ellipticity,  the construction of parametrices, the heat traces and the regularisation of traces  for the developed subelliptic calculus. A subelliptic global functional calculus will be established as well as a subelliptic version of Hulanicki theorem. This subelliptic functional calculus will be used to  prove a subelliptic version of the G\r{a}rding inequality, which we also use to study the global solvability for a class of  subelliptic pseudo-differential problems.  Finally, by using both, the matrix-valued symbols and also the notion of matrix-valued phases we study the $L^2$-boundedness of global Fourier integral operators. The approach established in characterising our subelliptic H\"ormander classes (by proving that the definition of these classes is independent of certain parameters) will be  also  applied in order to characterise the global H\"ormander classes on arbitrary graded Lie groups developed in \cite{FischerRuzhanskyBook}.
\end{abstract} \maketitle

\tableofcontents
\allowdisplaybreaks

\section{Introduction and historical remarks}

\subsection{Motivation}
In pure mathematics the Kohn-Nirenberg quantisation $a(x,D)$  was defined in \cite{KohnNirenberg1965}  for any measurable function $a$ defined on the phase space $\mathbb{R}^n\times \mathbb{R}^n, $ in terms of the Fourier transform $f\mapsto \widehat{f}$ of distributions on $\mathbb{R}^n.$ Indeed, the action of $a(x,D)$ in the space of test functions $f\in C^\infty_0(\mathbb{R}^n)$ is given by the integral formula
\begin{equation}
    a(x,D)f(x)=\smallint\limits_{\mathbb{R}^n}e^{2\pi i (x,\xi)}a(x,\xi)\widehat{f}(\xi)d\xi,
\end{equation}if for example, the symbol $a(x,\xi)$ satisfies the estimates
\begin{equation}\label{symbol:intro}
  \forall\alpha\in \mathbb{N}_0^n,\,\forall\beta\in \mathbb{N}_0^n,\exists C_{\alpha,\beta}>0, \,  |\partial_x^\beta\partial_\xi^\alpha a(x,\xi)|\leq C_{\alpha,\beta}(1+|\xi|)^{m-|\alpha|},
\end{equation} for some $m\in \mathbb{R}.$ By following the standard terminology one says that $a(x,D)$ is the pseudo-differential operator associated to the symbol $a(\cdot, \cdot).$ The class of symbols $S^{m}(\mathbb{R}^n\times \mathbb{R}^n)=\{a\in C^\infty(\mathbb{R}^{2n}): a\textnormal{  satisfies }\eqref{symbol:intro}\}$ begets the class of operators $\Psi^{m}(\mathbb{R}^n)=\{a(x,D): a\in S^{m}(\mathbb{R}^n\times \mathbb{R}^n)\}.$ These classes satisfy the following properties:
\begin{itemize}
    \item[1.] If $a(x,D)\in \Psi^{m_1}(\mathbb{R}^n)$  and  $b(x,D)\in \Psi^{m_2}(\mathbb{R}^n),$ then $c(x,D)=a(x,D)\circ b(x,D)$ is a pseudo-differential operator in the class  $\Psi^{m_1+m_2}(\mathbb{R}^n). $ This means that the classes  $\Psi^{m}(\mathbb{R}^n)$ are stable under compositions.
    \item[2.] The formal adjoint $a(x,D)^{*}$ of $a(x,D)\in \Psi^{m_1}(\mathbb{R}^n)$ is a pseudo-differential operator in the class  $\Psi^{m_1}(\mathbb{R}^n).$ The classes  $\Psi^{m}(\mathbb{R}^n)$ are  stable by taking adjoints.
    \item[3.] Any pseudo-differential operator in the class $\Psi^0(\mathbb{R}^n)$ is a bounded operator on $L^p(\mathbb{R}^n),$ for all $1<p<\infty.$ This fundamental $L^p$-estimate can be viewed as a variable coefficient version of the H\"ormander-Mihlin multiplier theorem \cite{Hormander1985III}.
    \item[4.] Any elliptic operator $a(x,D)$ of the class $\Psi^{m}(\mathbb{R}^n),$ i.e. when its symbol $a(\cdot,\cdot)$ has order $m$ from above and from  below, that is
    \[\exists C,C',R_0>0,\,\forall|\xi|\geq R_{0},\,\forall x\in \mathbb{R}^n,\, \ C'(1+|\xi|)^m\leq |a(x,\xi)|\leq C(1+|\xi|), \]
    admits the existence of a (approximate inverse) parametrix $a^{\#}(x,D)$ with the symbol  $a^{\#}\in \Psi^{-m}(\mathbb{R}^n)$ being elliptic of order $-m,$ and such that 
    \begin{equation}
        a(x,D)a^{\#}(x,D)=I+R(x,D),\,a^{\#}(x,D)a(x,D)=I+R'(x,D),
    \end{equation}where the ``error terms'' $R(x,D)$ and $R'(x,D)$ are smoothing/compact operators in the class $\Psi^{-\infty}(\mathbb{R}^n)=\cap_{m\in \mathbb{R}}\Psi^{m}(\mathbb{R}^n).$ 
\end{itemize}
A fundamental property of the classes $\Psi^{m}(\mathbb{R}^n)$ is that they include the family of partial differential operators
\begin{equation}
    p(x,D)=\sum_{|\alpha|\leq m}a_{\alpha}(x)\partial_{x}^{\alpha},
\end{equation}  of any order $m\in \mathbb{N},$  defined by the symbols $p(x,\xi)=\sum_{\alpha}a_{\alpha}(x)\xi^{\alpha},$ with bounded derivatives in $x,$ as well as the class of elliptic partial differential operators of order $m\in \mathbb{N},$ defined by symbols satisfying the ellipticity condition
    \[\exists C,C',R_0>0,\,\forall|\xi|\geq R_{0},\,\forall x\in \mathbb{R}^n,\, \ C'(1+|\xi|)^m\leq |p(x,\xi)|\leq C(1+|\xi|)^m. \]
Although the classes $\Psi^m(\mathbb{R}^n)$ have as a motivating example the elliptic partial differential operators, they exclude operators with unbounded coefficients like the harmonic oscillator $H=-\Delta_x+|x|^2$, and then other classes of pseudo-differential operators need to be defined.  For instance, adapted to the harmonic oscillator, Shubin (see e.g. \cite[Page 455]{FischerRuzhanskyBook}) has introduced the classes of symbols $\Sigma^m(\mathbb{R}^n\times \mathbb{R}^n),$ defined by the decay conditions
\begin{equation}\label{symbol:intro:2}
 a\in \Sigma^m(\mathbb{R}^n\times \mathbb{R}^n)\Longleftrightarrow \,\forall\alpha,\beta\in \mathbb{N}_0^n,\,  |\partial_x^\beta\partial_\xi^\alpha a(x,\xi)|\leq C_{\alpha,\beta}(1+|x|+|\xi|)^{m-|\alpha|-|\beta|}.
\end{equation} The classes of symbols $\Sigma^m(\mathbb{R}^n\times \mathbb{R}^n)$ also beget classes of pseudo-differential operators having similar properties to the ones satisfied by the Kohn-Nirenberg classes: stability under compositions, stability by taking adjoints, existence of parametrices  and $L^p$-boundedness for the class of a certain order, in this case of order zero. This last property can be different if one considers classes allowing loss of regularity, or allowing the analysis of hypoelliptic problems (instead of elliptic problems) as in the case of the H\"ormander classes $S^{m}_{\rho,\delta}(\mathbb{R}^n\times \mathbb{R}^n),$ $0\leq \delta\leq \rho\leq 1,$ see   Section \ref{Sect2} for details. Similarly one can define pseudo-differential operators on a compact manifold, like the localised Kohn-Nirenberg classes  $S^{m}(M,\text{loc})=S^{m}(T^{*}M),$ and in particular, one can consider the case related to this monograph, for instance related to the pseudo-differential classes on compact Lie groups allowing the analysis of subelliptic problems for H\"ormander sub-Laplacians $\mathcal{L}=-\sum_{j}X_{j}^2.$ We precise the contributions of this memoir in the rest of this introduction.

\subsection{Contributions of this memoir}
This monograph is devoted to the development of the pseudo-differential calculus for subelliptic  pseudo-differential  operators on arbitrary compact Lie groups and its applications. For instance, the theory developed here could remain valid in several compact non-commutative structures with the presence of symmetries  (see \cite[Part IV]{Ruz}). Indeed, in modern mathematics, the theory of pseudo-differential operators is a powerful branch in the analysis of linear partial  differential operators due to its interactions with several areas of mathematics. For instance, from the point of view of the theory of partial differential equations, pseudo-differential operators are used e.g. to study the global/local solvability of several partial differential problems, to understand the mapping properties of certain singular integral operators, to understand the propagation of singularities in distribution theory, and in the construction of fundamental solutions and parametrices. Also, in the interplay between differential geometry and algebraic topology, pseudo-differential operators are used to compute some geometric invariants arising in the  index theory. That is the case of analytical expressions for the Euler characteristic, the Hirzebruch signature and, in a more general context, the Atiyah-Singer index theorem (see e.g. Atiyah and Singer \cite{AS,AS1,AS2,AS3,AS4,AS5}, Kohn and Nirenberg  \cite{KohnNirenberg1965}, the fundamental book by H\"ormander \cite{Hormander1985III} and references therein). On the other hand, in the microlocal analysis, the theory of Fourier integral operators becomes  a prominent generalisation of  pseudo-differential operators, to study the spectral function for elliptic operators on vector bundles and in solving hyperbolic differential equations (see Duistermaat and H\"ormander \cite{DuiHor} and H\"ormander \cite{Hor71}).

In this work we  develop a subelliptic pseudo-differential calculus on compact Lie groups and some of its applications, by contributing with the notions and results of harmonic analysis on compact Lie groups, building up on the monograph \cite{Ruz} by V. Turunen and the second author, which was devoted to the development of the general theory of global pseudo-differential operators (with matrix-valued symbols) on spaces with symmetries. Starting our work, we investigate the action of the subelliptic calculus on $L^p,$ subelliptic Sobolev, and subelliptic Besov spaces and in the final part, we study the $L^2$-boundedness of global Fourier integral operators. Instead of the standard models of pseudo-differential operator theories  defined by the   notion of a symbol via localisations (see H\"ormander \cite{Hormander1985III}\footnote{ Where the argument of such symbols is defined in points  $(x,\xi)$ of the cotangent bundle $T^*G\cong G\times \mathfrak{g}^{*}$ over $G.$ Here, $\mathfrak{g}$ denotes the Lie algebra of $G,$ and we denote by $\widehat{G}$ its unitary dual.}) and other ones defined by sub-Laplacians (see e.g. Nagel and Stein \cite{NagelStein78}), we will follow the paradigm of the global quantisation on  compact Lie groups $G$ as introduced in \cite{Ruz}, which is a non-commutative extension of the classical Konh-Nirenberg quantisation \cite{KohnNirenberg1965}.  The global symbols according to the theory developed in \cite{Ruz} are defined on the non-commutative phase space $G\times \widehat{G},$ however, the idea of studying pseudo-differential operators on Lie groups as a generalisation of  multipliers of the Fourier transform can be traced back to  Taylor \cite{taylorNC}.   The pseudo-differential calculus associated to the usual $(\rho,\delta)$-classes and this global notion are equivalent for $1-\rho\leqslant  \delta<\rho\leqslant 1$, (see \cite{Fischer2015,RuzhanskyTurunenWirth2014}) with the global notion allowing also for the  range $0\leqslant \delta<\rho\leqslant 1/2,$ where the pseudo-differential calculus associated with the notion of a symbol defined via localisations is not operable\footnote{Indeed, in this case the $(\rho,\delta)$-classes on arbitrary $C^\infty$-manifolds are not stable under coordinate changes.}. So, in particular,  the resultant global calculus for  the range  $0\leqslant \delta<\rho\leqslant 1/2,$ can be applied to the treatment of inverses of complex vector fields, sub-Laplacians and a wide variety of pseudo-differential problems (see \cite{Ruz} for a complete description and also e.g.  the references of this work).

From its beginning, the theory of pseudo-differential operators was closely related to the theory of singular integral operators developed by Mihlin \cite{Mihlin}, Calder\'on and Zygmund \cite{CalderonZygmund}. However, in the case of $\mathbb{R}^n,$ and of other manifolds with symmetries, we can use the Fourier transform (as Kohn and Nirenberg in \cite{KohnNirenberg1965}, and the works of the second author with Turunen, Fischer, and    M\u{a}ntoiu \cite{Ruz,FischerRuzhanskyBook,M1,M2}) to define pseudo-differential operators  by using global symbols. In the case of compact Lie groups as the theory developed in \cite{Ruz}, these global symbols are matrix-valued with the size of the matrix growing according to the size of the representation spaces. In the other general cases, for example on graded Lie groups \cite{FischerRuzhanskyBook} and on general Lie groups of type I \cite{M1,M2}, the symbols become  operator-valued and densely defined on the possibly infinitely dimensional representation spaces. In the spirit of the theory of singular integrals of Coifman and Weiss \cite{CoifmanWeiss}, we will follow the criterion given by Coifman and  De Guzm\'an \cite{CoifmandeGuzman} and the approach developed by Fefferman \cite{Fefferman1973} to establish the mapping properties for subelliptic  pseudo-differential operators. Also, the classical Cotlar-Stein Lemma will be applied to obtain the Calder\'on-Vaillancourt theorem  and the $L^2$-boundedness of global Fourier integral operators (see Section \ref{Outline} for details).

Because the subelliptic pseudo-differential calculus developed here is a parallel theory to the ones developed in \cite{Ruz} and \cite{FischerRuzhanskyBook}, we will exploit that our operators have singular kernels in order to study their mapping properties and other spectral properties arising in the spectral geometry, specifically from the regularisation of traces. The singularity orders for the kernels of the obtained subelliptic calculus can be classified in terms of the Hausdorff dimension $Q\geqslant \dim(G),$ of a compact Lie group $G,$  associated with the Carnot-Carath\'eodory distance induced by the sub-Laplacian under consideration, so in local coordinates  we obtain  more singular  kernels that those obtained by the H\"ormander calculus  in the case of compact Lie groups \cite{Fischer2015,RuzhanskyTurunenWirth2014}. 

Let us consider on the compact Lie group $G,$ the positive sub-Laplacian  \[\mathcal{L}=-(X_1^2+\cdots +X_k^2),
\] which is  considered in such a way that the system of vector fields $X=\{X_i\}_{i=1}^{k}$ satisfies the H\"ormander condition of order $\kappa$.\footnote{This means that their iterated commutators of length $\leqslant \kappa$ span the Lie algebra $\mathfrak{g}$ of $G$.} In this memoir we develop a subelliptic pseudo-differential calculus associated with $\mathcal{L},$  by defining certain H\"ormander type classes ${S}^{m,\mathcal{L}}_{\rho,\delta}(G\times \widehat{G})=\mathscr{S}^{m\kappa,\mathcal{L}}_{\rho\kappa,\delta\kappa}(G),$ $0\leqslant \rho,\delta\leqslant 1,$ (stable under compositions and adjoints). We have opted for using the matrix-valued (global) quantisation process developed by the second author and V. Turunen in \cite{Ruz,RuzTurIMRN}. This viewpoint has shown to be a versatile tool in the analysis of pseudo-differential operators for describing their analytic and spectral properties (see e.g. \cite{deMoraes,RodriguezRuzhansky2020,RuzhanskyDelgado2017,Fischer2015,Garetto,GarettoRuzhansky2015,Ruz,RuzhanskyTurunenWirth2014,RuzTurIMRN,RuzhanskyWirth2014,RuzhanskyWirth2015}) and in the treatment of some  problems for PDE on compact Lie groups. Our motivation for using the matrix-valued quantisation came from the fact that the global symbols obtained with this procedure together with a suitable difference structure of the unitary dual $\widehat{G},$ of $G,$  characterise the H\"ormander classes of pseudo-differential operators defined by charts \cite{RuzhanskyTurunenWirth2014}.

The sub-Laplacian $\mathcal{L},$ associated with a system of vector fields $X=\{X_i\}_{i=1}^{k}$ satisfying the H\"ormander condition, endows  $G$ with a natural sub-Riemannian structure. The sub-bundle $\mathcal{H}=\textnormal{span}{(X)},$ of the tangent bundle $TG$ generated by the system $X,$ provides a natural setting to study subelliptic operators as $\mathcal{L},$ which is, in fact, hypoelliptic  by an application of the theorem of sum of squares of H\"ormander. This kind of sub-Riemannian structure appears in many areas, to say, describing  constrained systems in mechanics, or as limiting classes of Riemannian geometries (see e.g. Gordina and Laetsch \cite{GL} and references therein).  We also refer the reader to Bramanti \cite{Bramanti}, where the applications of this kind of sub-Riemannian manifolds are discussed, as its relation with  the Kolmogorov-Fokker-Plank equation, the $\overline{\partial}$-Neumann problem,  the  tangential Cauchy-Riemann Complex, the Kohn-Laplacian $\square_{b},$ and other differential problems. 
 
When we review the criteria obtained in terms of the matrix-valued quantization developed in \cite{Ruz,RuzTurIMRN}, we observe that some of them are given in terms of the decay of the matrix-valued symbol and its derivatives (or its differences) which is measured compared to the spectrum of the positive Laplace operator 
\[\mathcal{L}_{G}=-(X_1^2+\cdots +X_n^2),\,\,n=\dim(G).
\] 
Consequently, the symbolic calculus on compact Lie groups developed with the matrix-valued symbols enjoys  good properties when we look at its action on function spaces associated to the Laplacian.
From this view point, if we measure the decay of the global symbols used in the symbolic calculus associated  to the matrix-valued quantisation, with respect to the spectrum of the sub-Laplacian $\mathcal{L}$ (instead of using the spectrum of the Laplacian), we could provide sharp estimates for subelliptic problems without loss of regularity (see e.g. \cite{Garetto}, for the case of the wave equation associated to sub-Laplacians) on subelliptic function spaces by exploiting in this case the sub-Riemannian structure of $G$.

This analysis is organised as follows. 
\begin{itemize}
    \item In Section \ref{Outline} we will present the main results of this work and our contributions in relation to the existent literature for global pseudo-differential operators on compact Lie groups.
    \item In Section \ref{Sect2}, we present the preliminaries used throughout  this work. For instance, we will follow the original exposition in \cite{Ruz}.
    \item In Section \ref{Seccionsubelliptic} we define and develop the subelliptic pseudo-differential calculus on compact Lie groups, in terms of the matrix-valued quantisation. By using the Calder\'on-Zygmund estimates of Coifman and de Guzm\'an \cite{CoifmandeGuzman} (see also \cite{RuzhanskyWirth2015}), we  prove Theorem \ref{LpQL} and Theorem \ref{SubellipticLpestimate} in Section \ref{ps}.
    \item  By using the Littlewood-Paley theory and some estimates for commutators, Theorems \ref{parta} and \ref{parta2} will be proved in Section \ref{fs}. In Section \ref{LEH} we study the notion of ellipticity  associated with the developed subelliptic calculus. We provide the construction of subelliptic parametrices and we also study the heat traces and the regularised traces in the subelliptic context. 
    \item  A subelliptic global functional calculus will be developed in Section \ref{SFC}.
    \item Applications of this global functional calculus include the (subelliptic) G\r{a}rding inequality which will be used to study the global solvability  for  subelliptic evolution problems in Section \ref{GST}. 
    \item Finally, we will study the $L^2$-boundedness of global Fourier integral operators in Section \ref{FIOCLG}.
\end{itemize}

\section{Outline and main results}\label{Outline}

\subsection{Notation}In order to explain the main results of this work we will present some preliminaries on the matrix-valued quantisation. Indeed,  by following \cite{Ruz,RuzTurIMRN}, we associate  to a continuous linear operator $A:C^\infty(G)\rightarrow C^\infty(G),$ the global symbol $\sigma_A$ defined on the phase space $G\times \widehat{G}$ (here $\widehat{G}$ denotes the unitary dual of $G$) by the identity
 \[
     a(x,\xi)=\xi(x)^*A\xi(x), \,\,\,[\xi]\in \widehat{G}.\footnote{ Strictly speaking from every equivalence class $[\xi]$ we choose one and only one matrix-valued representation so that $a(x,[\xi]):=a(x,\xi)$.  }
 \]Then, the operator $A$ can be written in terms of this global symbol as
 \begin{equation}\label{RTq}
     Af(x)=\sum_{[\xi]\in \widehat{G}}d_\xi\textnormal{\textbf{Tr}}[\xi(x)a(x,\xi)\widehat{f}(\xi)],\,\,f\in C^\infty(G),
 \end{equation}where $\widehat{f}$ denotes the group Fourier transform of $f,$ $$\widehat{f}(\xi):=
\int\limits_{G}f(x)\xi(x)^{*}dx,$$ where $dx$ is the (normalised) Haar measure on $G.$  We denote by  $\Psi^{m}_{\rho,\delta}(G,\textnormal{loc}),$ $\delta<\rho,$ $\rho\geqslant   1-\delta,$  the H\"ormander class of order $m$ and of  type $(\rho,\delta)$\footnote{ i.e. the class of operators which in all local coordinate charts give operators in  $\Psi^{m}_{\rho,\delta}(\mathbb{R}^n).$ }. Then (see \cite{RuzhanskyWirth2014})
 \[
     A\in \Psi^{m}_{\rho,\delta}(G,\textnormal{loc})\textnormal{ \quad  if and only if \quad   }\Vert \Delta_\xi^\alpha\partial_{x}^\beta a(x,\xi) \Vert_{\textnormal{op}} \leqslant C_{\alpha,\beta}\langle \xi \rangle^{m-\rho|\alpha|+\delta|\beta|} ,
 \] for all multi-indices $\alpha,\beta.$   Here $\langle \xi \rangle=(1+\lambda_{[\xi]})^\frac{1}{2},$ and $\{\lambda_{[\xi]}\}_{[\xi]\in \widehat{G}}$ is the positive spectrum of the Laplacian $\mathcal{L}_{G},$ which can be enumerated by the unitary dual, and $\{\Delta_\xi^\alpha:\alpha\in\mathbb{N}_0^n\}$ is the collection of  the difference operators introduced in \cite{Ruz} that provide a difference structure on the unitary dual $\widehat{G}$. Furthermore, in \cite{Ruz}, the  H\"ormander classes  $\Psi^{m}_{\rho,\delta}(G)$ defined by
 \begin{equation}\label{HormanderclassRT}
     A\in \Psi^{m}_{\rho,\delta}(G),\,\,0\leqslant \delta,\rho\leqslant 1,\textnormal{   if and only if,   }\Vert \Delta_\xi^\alpha\partial_{x}^\beta a(x,\xi) \Vert_{\textnormal{op}} \leqslant C_{\alpha,\beta}\langle \xi \rangle^{m-\rho|\alpha|+\delta|\beta|} ,
 \end{equation} admit the complete range $0\leqslant \delta,\rho \leqslant 1$ and a symbolic calculus for these classes is possible for $0\leqslant \delta<\rho\leqslant 1$ without the standard restriction $\rho\geqslant   1-\delta.$ The symbol classes 
 \begin{equation}\label{smrd}
   \mathscr{S}^{m}_{\rho,\delta}(G\times \widehat{G}) :=\{ a(x,{\xi}): A=\textnormal{Op}(a)\in  \Psi^m_{\rho,\delta}(G)\},
 \end{equation}
  are useful tools in the analysis on compact Lie groups, where there  appear operators with symbols in local coordinates belonging to H\"ormander classes of type $(\rho,\delta)$ and  satisfying  the condition $\rho\leqslant 1-\delta,$\footnote{ e.g. these classes appear with symbols in the class $\mathscr{S}^{-1}_{\frac{1}{2},0}(G)$ where we have symbols of pseudo-differential parametrices of  sub-Laplacians,  or  the parametrix of the heat operator $\textnormal{D}_3-\textnormal{D}_1^2-\textnormal{D}_2^2$ on $\textnormal{SU}(2)$ (see \cite{RuzhanskyWirth2015}).}  under which these classes are not invariant under coordinate changes, and the classical methods, where one uses a  symbol  defined by local coordinate systems (see H\"ormander \cite{Hormander1985III}), could not be  applicable.  
 The subelliptic classes of pseudo-differential operators will be defined in Section \ref{Seccionsubelliptic}. They will be denoted by $$S^{m,\ell,\ell',\mathcal{L}}_{\rho,\delta}(G\times \widehat{G})=\mathscr{S}^{m\kappa,\ell,\ell',\mathcal{L}}_{\rho\kappa,\delta\kappa}(G),\,\,\, 0\leqslant \rho,\delta\leqslant 1,$$ by indicating that the symbols there have  order $m,$ which satisfy the $\rho$-type (subelliptic) conditions up to order $\ell\in\mathbb{N},$   and the $\delta$-type (subelliptic)  conditions up to order $\ell'$. So,
 \[   {S}^{m,\mathcal{L}}_{\rho,\delta}(G\times \widehat{G}):=\bigcap_{\ell,\ell'}  {S}^{m,\ell,\ell',\mathcal{L}}_{\rho,\delta}(G\times \widehat{G}),\,\,\,\,0\leqslant \rho,\delta\leqslant 1,
 \] denotes the contracted class of subelliptic smooth symbols  of order $m$. We also define in  Section \ref{Seccionsubelliptic} the class of symbols for subelliptic Fourier multipliers ${S}^{m,\ell,\mathcal{L}}_{\rho}(\widehat{G})=\mathscr{S}^{m\kappa,\ell,\mathcal{L}}_{\rho\kappa}(G),$ $0\leqslant\rho\leqslant 1.$ The corresponding classes of operators associated with these symbols classes are denoted by 
 \begin{equation}\label{C1}
     \textnormal{Op}( {S}^{m,\ell,\ell',\mathcal{L}}_{\rho,\delta}(G\times \widehat{G})),\,\,\textnormal{Op}({S}^{m,\ell,\mathcal{L}}_{\rho}(\widehat{G})),\,\,\,\,0\leqslant \rho,\delta\leqslant 1,
 \end{equation} and 
 \begin{equation}\label{C2}
     \textnormal{Op}( {S}^{m,\mathcal{L}}_{\rho,\delta}(G\times \widehat{G})):=\bigcap_{\ell,\ell'}\textnormal{Op}(  {S}^{m,\ell,\ell',\mathcal{L}}_{\rho,\delta}(G\times \widehat{G})),\,\,\,\,0\leqslant \rho,\delta\leqslant 1.
 \end{equation}
The reason for defining the contracted classes is that they will be useful in Section \ref{ps} in order to establish the $L^p$-boundedness of subelliptic Fourier multipliers.

 We will reserve as usually, the notation
 \begin{equation}\label{C3}
     \textnormal{Op}( \mathscr{S}^{m,\ell,\ell'}_{\rho,\delta}(G\times \widehat{G})),\,\,\textnormal{Op}(\mathscr{S}^{m,\ell}_{\rho}( \widehat{G})),\,\,\,\,0\leqslant \rho,\delta\leqslant 1,
 \end{equation} and 
 \begin{equation}\label{C4}
  \Psi^m_{\rho,\delta}(G)\equiv   \textnormal{Op}( \mathscr{S}^{m}_{\rho,\delta}(G\times \widehat{G})):=\bigcap_{\ell,\ell'}\textnormal{Op}(  \mathscr{S}^{m,\ell,\ell'}_{\rho,\delta}(G\times \widehat{G})),\,\,\,\,0\leqslant \rho,\delta\leqslant 1,
 \end{equation} for the corresponding classes of limited regularity and of smooth symbols.
 Here,    $\varkappa$ will be defined as the smallest even integer larger than $\dim(G)/2.$ It is interesting to note that in stark contrast to graded Lie groups, the symbol classes here may depend on the choice of a sub-Laplacian (see Remarks \ref{dependece1} and \ref{dependece2}).
 \subsection{Main results}
 In terms of the notations that we fix above, the main results of this memoir are, the symbolic calculus developed in Section  \ref{Seccionsubelliptic} and in Section \ref{LEH}, the subelliptic global functional calculus developed in Section \ref{SFC},  and the following subelliptic boundedness results/lower bounds/applications which we describe as follows and in the subsequent remarks.
\begin{itemize}
    \item (Subelliptic Marcinkiewicz multiplier theorem). Every  $A\in \textnormal{Op}({S}^{\,0,\varkappa,\mathcal{L}}_{1}(\widehat{G})),$  extends to an operator of weak type $(1,1)$ and is bounded on $L^p(G)$ for all $1<p<\infty$ (see Theorem \ref{LpQL}). Moreover,  if $A\in \textnormal{Op}({S}^{-(1-\rho),\varkappa,\mathcal{L}}_{\rho}(G)),$ $0\leqslant  \rho<1,$ then  $A$ extends to an operator of weak type $(1,1)$ and compact on $L^p(G)$ for all $1<p<\infty,$                   (see Theorem \ref{LpQL} for values of $\varkappa$).\footnote{The main point in this situation is the weak (1,1) estimate because the compactness is straightforward on $L^2(G)$ and by using the interpolation with the weak (1,1) estimate  the compactness on $L^p(G)$ also follows.} 
\item (Subelliptic Calder\'on-Vaillancourt theorem).  For  $0\leqslant \delta<\rho\leqslant  1,$ (or $0\leq \delta\leq \rho\leq 1,$ $\delta<1/\kappa$) let us consider a continuous linear operator $A:C^\infty(G)\rightarrow\mathscr{D}'(G)$ with symbol  $\sigma\in {S}^{0,\mathcal{L}}_{\rho,\delta}(G\times \widehat{G})$. Then $A$ extends to a bounded operator from $L^2(G)$ to  $L^2(G).$ The case $\rho=\delta=0$ can be deduced from Theorem 10.5.5 of \cite{Ruz}.  We observe that a similar $L^2$-theorem for $\delta=\rho=0,$ can be proved for global Fourier integral operators. But, we will return to this point in detail after presenting the mapping properties of the subelliptic calculus (see Theorem \ref{L2FIO}). 

\item   Let $A\in \textnormal{Op}({S}^{-(1-\rho),\varkappa,\left[\frac{n}{p}\right]+1,\mathcal{L}}_{\rho,0}(G)),$ $0\leqslant \rho\leqslant1.$  For $\rho=1,$ $A$ extends to a bounded operator on $L^p(G),$ and for $0\leqslant \rho<1,$  $A$ extends to a compact operator on $L^p(G),$ in both cases for all $1<p<\infty,$                   (Theorem \ref{SubellipticLpestimate}).\footnote{Again, the main point here is the boundedness estimate because the compactness is straighforward from the argument of interpolation.}
\item(Subelliptic Fefferman theorem).  For any compact Lie group $G,$  let us denote by $Q$ its Hausdorff
dimension  associated to the control distance associated to the sub-Laplacian $\mathcal{L}=\mathcal{L}_X,$ where  $X= \{X_{1},\cdots,X_k\} $ is a system of vector fields satisfying the H\"ormander condition of order $\kappa$.  For  $0\leqslant \delta<\rho\leqslant 1,$   let us consider a continuous linear operator $A:C^\infty(G)\rightarrow\mathscr{D}'(G)$ with symbol  $\sigma\in {S}^{-m,\mathcal{L}}_{\rho,\delta}(G\times \widehat{G})$. For the order  $m=\frac{Q(1-\rho) }{2},$ the operator $A$ is bounded from $L^\infty(G)$ to $\textnormal{BMO}^{\mathcal{L}}(G),$ from the subelliptic Hardy space $\textnormal{H}^{1,\mathcal{L}}(G)$ to $L^1(G),$ and from $L^p(G)$ into itself for all $1<p<\infty.$ Moreover,  for  $1<p<\infty,$ and 
\begin{equation}\label{sharporder''}
    m\geqslant   m_{p}:= Q(1-\rho)\left|\frac{1}{p} -\frac{1}{2}\right|,
\end{equation} the linear operator $A$ extends to a bounded operator on $L^p(G),$ see Theorem \ref{parta} and Theorem \ref{parta2}.
\item(Subelliptic G\r{a}rding inequality). Let $a(x,D):C^\infty(G)\rightarrow\mathscr{D}'(G)$ be an operator with symbol  $a\in {S}^{m,\mathcal{L}}_{\rho,\delta}(G\times \widehat{G})$, $0\leqslant \delta< \rho\leqslant 1,$   and let $m>0$. Let $\mathcal{M}=(1+\mathcal{L})^{\frac{1}{2}}$ be defined by the functional calculus and let $\{\widehat{\mathcal{M}}(\xi)\}_{[\xi]\in \widehat{G}}$ be its corresponding global symbol. Let us assume that $a(x,D)$ is strongly $\mathcal{L}$-elliptic which means that,
\[A(x,\xi):=\frac{1}{2}(a(x,\xi)+a(x,\xi)^{*}),\,(x,[\xi])\in G\times \widehat{G},\,\,a\in S^{m,\mathcal{L}}_{\rho,\delta}(G\times \widehat{G}), 
\] satisfies
\begin{equation}\label{garding2}    \Vert\widehat{\mathcal{M}}(\xi)^{m}A(x,\xi)^{-1} \Vert_{\textnormal{op}}\leqslant C_{0}.
\end{equation}Then, there exist $C_{1},C_{2}>0,$ such that the lower bound
\[    \textnormal{Re}(a(x,D)u,u) \geqslant C_1\Vert u\Vert_{{L}^{2,\mathcal{L}}_{\frac{m}{2}}(G)}-C_2\Vert u\Vert_{L^2(G)}^2,
\] holds true for every $u\in C^\infty(G).$ Here, $H^{\frac{m}{2},\mathcal{L}}(G)\equiv {L}^{2,\mathcal{L}}_{\frac{m}{2}}(G),$ is the subelliptc Sobolev space associated to $\mathcal{L}$ with regularity order $m/2.$ This subelliptic version of the G\r{a}rding inequality will be proved in Theorem \ref{GardinTheorem2}.
\item We will use the subelliptic G\r{a}rding inequality to study the well posedness for the Cauchy problem \begin{equation}\label{PVI2} \frac{\partial v}{\partial t}=K(t,x,D)v+f ,
v(0)=u_0,
\end{equation} where $K(t)=K(t,x,D)$ is strongly $\mathcal{L}$-elliptic for all $t\in [0,T].$ The simplest case $K(t)=\mathcal{L},$ corresponds to the heat equation for the sub-Laplacian. In particular we can take $K(t)=a(t,x)\mathcal{L}^s$ or $K(t)=a(t,x)(1+\mathcal{L})^{\frac{s}{2}},$ where $a:=a(t,x)\in C^\infty([0,T]\times G)$ is such that $|\textnormal{Re}(a(t,x))|\geqslant a_0>0,$ and $s>0.$  We refer the reader to Section \ref{GST} for details.
\end{itemize}
Now, we describe some $L^2$-estimates for  global Fourier integral operators on compact Lie groups which appear as continuous linear operators of the form
\begin{equation}\label{FIO}
    Af(x):=\sum\limits_{[\xi]\in \widehat{G}}\textnormal{\textbf{Tr}}(e^{i\phi(x,\xi)}\sigma(x,\xi)\widehat{f}(\xi)), \,\,\,f\in C^\infty(G),\,\,\,x\in G,
\end{equation}where $\phi:G\times \widehat{G}\rightarrow\cup_{[\xi]\in \widehat{G}}\mathbb{C}^{d_\xi\times d_\xi}$ is the matrix-valued phase function of $A.$ Global  Fourier integral operators (FIOs) appear as useful extensions of pseudo-differential operators (see Remark \ref{r1}) and arise in  solutions of some differential problems (see e.g. Remark \ref{r2}). We study essentially two kinds of symbol conditions:
\begin{itemize}
\item  First, to study the $L^2$-boundedness of a global FIO we need to impose reasonable conditions on the symbol $\sigma(x,\xi)$ and also on the matrix-valued phase function $\phi(x,\xi)=\textnormal{diag}(\phi_{jj}(x,\xi))$. To do so, if $\mathbb{X}=\{X_{1},\cdots,X_n\}$ is a basis of the Lie algebra $\mathfrak{g},$ and the corresponding gradient $ \nabla_{\mathbb{X}}$ is defined by
\[ 
    \nabla_{\mathbb{X}}\psi(x)=(X_{1}\psi,\cdots,X_n\psi)\in C^{\infty}(G)\times \cdots \times C^{\infty}(G),\,\,\,\psi\in C^\infty(G),
\]in Theorem \ref{L2FIO} we show that the conditions \begin{equation}\label{613intro}
  \sup_{(x,[\xi])\in G\times \widehat{G}}  \Vert X^{\alpha}_{x}\sigma(x,\xi)\Vert_{\textnormal{op}}<\infty,
\end{equation}for all $ |\alpha|\leqslant 5n/2,$ and 
\begin{equation}\label{614intro}
    \vert \nabla_{\mathbb{X}}\phi_{jj}(x,\xi)-\nabla_{\mathbb{X}}\phi_{j'j'}(x,\xi')\vert\asymp |\lambda_{[\xi]}^\tau-\lambda_{[\xi']}^\tau|,
\end{equation}uniformly in $([\xi],[\xi'])\in \widehat{G}\times \widehat{G}$ for some $\tau>0 ,$ imply the existence of a bounded extension of  $A$ (defined in \eqref{FIO})  on $L^2(G).$
\item For a global Fourier integral operator of the form
\begin{equation}\label{wavetype''}
 Af(x)=\sum\limits_{[\xi]\in \widehat{G}}d_\xi\textnormal{\textbf{Tr}}(\xi(x)e^{i\Phi(\xi)}\sigma(x,\xi)\widehat{f}(\xi)),\quad f\in C^\infty(G),
\end{equation}      where the function $\Phi:\widehat{G}\rightarrow \cup_{[\xi]\in \widehat{G}}\mathbb{C}^{d_\xi\times d_\xi},$ is such that $\Phi(\xi)=\Phi(\xi)^*$ for all $[\xi]\in \widehat{G},$ we will prove in Theorem \ref{L2FIO2}  that under the symbol inequalities
 \begin{equation}
    \sup_{(x,[\xi])\in G\times \widehat{G}}  \Vert X^{\alpha_1}_{i_1}\cdots X^{\alpha_k}_{i_k}\sigma(x,\xi)\Vert_{\textnormal{op}}<\infty,\quad1\leqslant i_1\leqslant i_2\leqslant\cdots \leqslant i_{k}\leqslant k,
\end{equation}  
$ |\alpha|\leqslant \varkappa_Q\footnote{ Here, $ \varkappa_Q$ is the smallest even integer larger than $Q/2. $ },$  if $X= \{X_{1},\cdots,X_k\} $ is a system of vector fields satisfying the H\"ormander condition, the operator $A$ extends to a bounded linear operator on $L^2(G).$ Moreover, if $1<p<\infty$ and $0<\rho\leq 1,$ in Theorem \ref{LpFIO1} we will prove that under the following  conditions:
 \begin{equation}\label{condFIO222}
    \sup_{(x,[\xi])\in G\times \widehat{G}}  \Vert\widehat{\mathcal{M}}(\xi)^{m+\rho|\gamma|}\Delta_{\xi}^{\gamma}(e^{i\Phi(\xi)} X^{\alpha_1}_{i_1}\cdots X^{\alpha_k}_{i_k}\sigma(x,\xi))\Vert_{\textnormal{op}}<\infty,\quad \gamma\in \mathbb{N}_0^n,
\end{equation}  
for all  $1\leqslant i_1\leqslant i_2\leqslant\cdots \leqslant i_{k}\leqslant k,$ and $ |\alpha|\leqslant \varkappa_{Q,p}\footnote{For all $1<p<\infty,$  $ \varkappa_{Q,p}$ is the smallest even integer larger than $Q/p. $ In particular $\varkappa_{Q}:=\varkappa_{Q,2}. $ }   ,$ the Fourier integral operator  $A$ extends to a bounded linear operator on $L^p(G),$ provided that
\begin{equation}
     m\geq m_{p}:=Q(1-\rho)\left|\frac{1}{p}-\frac{1}{2}\right|.
\end{equation}
 \item In particular,  for $\rho=\frac{1}{Q},$ Theorem \ref{LpFIO1} implies that under the condition $m\geq (Q-1)\left|\frac{1}{p}-\frac{1}{2}\right|,$ the operator $A\equiv \textnormal{FIO}(\sigma, \phi):C^\infty(G)\rightarrow\mathscr{D}'(G)$ associated with the symbol  $\sigma$  satisfying the family of inequalities in \eqref{condFIO222}, extends to a bounded operator from $L^p(G)$ to itself where $1<p<\infty$.

\end{itemize}
\begin{remark}
Let us explain the properties of the developed subelliptic calculus. In Definition \ref{contracted''} we define the (contracted) subelliptic H\"ormander classes for the sub-Laplacian $\mathcal{L},$ ${S}^{m,\mathcal{L}}_{\rho,\delta}(G\times \widehat{G}),$ $0\leqslant \rho,\delta\leqslant 1,$ (and the dilated classes in Definition \ref{dilated}). These classes are closed under compositions and adjoints, and we prove in Proposition \ref{CalderonZygmund}, that the operators associated to these classes have Calder\'on-Zygmund kernels in some sense.
\end{remark}
\begin{remark}
The singularity order for the right-convolution kernel of a subelliptic operator in our classes can be classified in terms of the Hausdorff dimension $Q$ of the Lie group $G$ with respect to $\mathcal{L}.$ Indeed, if  $A:C^\infty(G)\rightarrow\mathscr{D}'(G)$ is a continuous linear operator with symbol $\sigma\in {S}^{m,\mathcal{L}}_{\rho,\delta}(G\times \widehat{G}),$   then the  right-convolution kernel of $A,$ $x\mapsto k_{x}:G\rightarrow C^\infty(G\setminus\{e_G\}),$ defined by $k_{x}:=\mathscr{F}^{-1}\sigma(x,\cdot),$ satisfies the following estimates for $|y|<1$ (see Proposition \ref{CalderonZygmund}):
\[ 
    \left\{ \begin{array}{lcc}
             |k_{x}(y)|\lesssim_{m}  \Vert \sigma\Vert_{\ell, S^{m,\mathcal{L}}_{\rho,\delta}}|y|^{-\frac{Q+m}{\rho}},&   \textnormal{  if  }   m>-Q \\
             \\ |k_{x}(y)|\lesssim_{m} \Vert \sigma\Vert_{\ell, S^{m,\mathcal{L}}_{\rho,\delta}}|\log|y||,&  \textnormal{  if  }   m=-Q \\
             \\ |k_{x}(y)|\lesssim_{m} \Vert \sigma\Vert_{\ell, S^{m,\mathcal{L}}_{\rho,\delta}}, &  \textnormal{  if  }   m<-Q,
             \end{array}
   \right.
\]
 where $\ell\in \mathbb{N},$  is large enough and is independent of the symbol $\sigma.$ Let us remark that for $m+Q>0,$ in the context of the pseudo-differential classes defined by Nagel and Stein \cite{NagelStein78} for sub-Riemannian structures, the kernel estimates for the corresponding classes have been obtained in terms of the Carn\'ot-Carath\'eodory distance $|y|_{cc}=d_{cc}(y,e_G)$ whose orders should be  better than the reported in Proposition \ref{CalderonZygmund} in terms of the geodesic distance on $G.$ This is because of the topological inequalities
 \begin{equation}
     |y|\lesssim  |y|_{cc}\lesssim |y|^{\frac{1}{\kappa}},\,\,\,|y|<1. 
 \end{equation} Indeed, the classes of order $m$ in the calculus of Nagel and Stein \cite{NagelStein78}  include operators with kernels $k_{x}(y),$ behaving like $|y|_{cc}^{-(Q+m)}\lesssim |y|^{-(Q+m)},$ when $m>-Q,$ and with $|y|<1.$ Note that for $m=-Q,$ the operators of order $-Q$ behave like $|\log|y|_{cc}|\lesssim |\log|y||,$ for $|y|<1.$    However, for our purposes, the kernel estimates in Proposition \ref{CalderonZygmund} are good enough to describe the subelliptic calculus in terms of the difference operators $\Delta_\xi^\alpha$.
\end{remark}
\begin{remark}
  The condition \eqref{sharporder''} is sharp for connected and simply connected Lie groups $G$ in the following sense. If we replace the sub-Laplacian by the Laplace operator  on $G,$ $\mathcal{L}_{G},$   we recover Theorem 4.15 of \cite{RuzhanskyDelgado2017} and then (see Remark \ref{sharpexample}) the  critical order $m_p$ in  \eqref{sharporder''} is the best possible  in order to assure the $L^p(G)$-boundedness for operators in the  class $\textnormal{Op}(\mathscr{S}^m_{\rho,\delta}(G\times \widehat{G})).$ Indeed, in this case \eqref{sharporder''} with $\kappa=1,$ and $Q=\dim(G),$ is a necessary and sufficient condition for the $L^p(G)$-boundedness of $A$. The condition  \eqref{sharporder''} is an analogy of Theorem 1.2 in \cite{RuzhanskyDelgadoCardona2019} about the boundedness of the H\"ormander classes on arbitrary graded Lie groups which is an extension of the sharp theorem due to C. Fefferman  \cite{Fefferman1973}. 
  \end{remark}
  \begin{remark}
    Theorems \ref{parta} and \ref{parta2} are analogues on compact Lie groups of the boundedness   theorems due to Fefferman \cite{Fefferman1973} and Hirschman \cite{Hirschman1956} for the classical H\"ormander classes  on $\mathbb{R}^n,$ extensions of the classical Wainger $L^p$-estimates for oscillating  multipliers on the torus \cite{Wainger1965} and also extensions of the $L^p$-estimates for oscillating central multipliers for the Laplacian on compact connected and simply connected Lie groups proved in Chen and Fan  \cite[Theorem 1]{Chen}. 
  \end{remark}
\begin{remark}
   Theorem \ref{LpQL} and  Theorem \ref{SubellipticLpestimate} are extensions  of the weak (1,1) boundedness theorem proved  by the second author and J. Wirth in \cite{RuzhanskyWirth2014,RuzhanskyWirth2015}.
   Our main  $L^p$-subelliptic estimates can be used to prove estimates for pseudo-differential operators on subelliptic Sobolev and subelliptic Besov espaces (see Corollaries  \ref{SobL}, \ref{BesovL} and \ref{SobL2}). 
 \end{remark}
 \begin{remark} The condition \eqref{614intro} for the $L^2$-boundedness of Fourier integral operators is the non-commutative version of the usual local graph condition for Fourier integral operators, necessary for the local $L^2$-boundednes for Fourier integral operators on $\mathbb{R}^n$ (see e.g. Eskin \cite{Eskin} and H\"ormander \cite{Hor71}). Theorem  \ref{L2FIO} is the non-commutative extension of Theorem 4.14.2 for Fourier integral operators on the torus \cite{Ruz} (see Theorem \ref{L2FIOtorus}) and also extends the $L^2$-boundedness Theorem 10.5.5 in \cite{Ruz} for pseudo-differential operators on compact Lie groups.
 \end{remark}
For the general aspects of the theory of Fourier integral operators we refer the reader to H\"ormander \cite{Hor71} and Duistermaat and H\"ormander \cite{DuiHor}. The problem of the boundedness of Fourier integral operators has been treated in Fujiwara \cite{Fuji}, Asada and Fujiwara  \cite{AF}, Miyachi  \cite{Miyachi}, Peral \cite{Peral}, Seeger, Sogge, and Stein \cite{SSS91},  Tao \cite{Tao}, and also the references \cite{M. Ruzhansky}, \cite{CoRu1},  \cite{CoRu2} and \cite{RuzSugi01}. For necessary conditions of the weak boundedness of Fourier integral operators we refer the reader to \cite{CardonaRuzhansky2022FIO} and to \cite{RuzSugi2019} for the recent {\it local-to-global extension argument}. Results on the boundedness of Fourier integral operators on the torus can be found in \cite[Theorem 4.14.2]{Ruz} and in  \cite{CRS2018}.

\begin{remark}
We will consider a suitable notion of ellipticity associated to the sub-Laplacian $\mathcal{L}$ which we call $\mathcal{L}$-ellipticity. We found this notion consistent with the classical  notion of ellipticity from the point of view of construction of parametrices (see Section \ref{LEH}). 
\end{remark}
\begin{remark}[Heat traces and regularised traces for subelliptic operators]
We will study the asymptotic behaviour of the  heat traces and also of other regularised traces for  $\mathcal{L}$-elliptic    pseudo-differential  operators. Indeed, we will prove under reasonable conditions on a operator $A\in \Psi^{m,\mathcal{L}}_{\rho,\delta}(G\times \widehat{G}),$ $0\leqslant {\delta},\rho\leqslant 1,$ that 
\begin{equation}\label{summarisingformulas2}
     \textnormal{\textbf{Tr}}(Ae^{-t(1+\mathcal{L})^\frac{q}{2}})\sim t^{-\frac{m+Q}{q}}\sum_{k=0}^\infty a_kt^\frac{k}{q}-\frac{b_0}{q}\log(t),\quad t\rightarrow0^{+},
 \end{equation}for  $m\geqslant    -Q.$ If $m=-Q,$ then $a_k=0$ for every $k,$ and for $m>-Q,$ $b_0=0.$ If we consider the case of the Laplacian (see Corollary \ref{asymptotictracemultiplier2'}), and we restrict our attention to the case $(\rho,\delta)=(1,0)$ and $m=-n,$ $n:=\dim(G),$ it is known that the term $b_{0}$ in the asymptotic expansion \eqref{summarisingformulas2} agrees with the Wodzicki residue of $A$ (see e.g. Wodzicki \cite{Wodzicki} and Lesch \cite{Lesch}) and consequently with the Dixmier trace of $A,$ which is a consequence of a celebrated theorem due to A. Connes \cite{Connes94}. We   will also prove asymptotic expansions of the form (see Theorem \ref{asymptotictraceeta}),
 \begin{equation}\label{TC}
   \textnormal{\bf{Tr}}(A\psi(t E))=t^{-\frac{Q+m}{q}}\left(  \sum_{k=0}^{\infty}a_kt^k\right)+\frac{c_{Q}}{q}\int\limits_{0}^{\infty}\psi(s)\times\frac{ds}{s},\,\,t\rightarrow 0^{+} ,
\end{equation}for $m\geqslant -Q,$ where $E$ is an $\mathcal{L}$-elliptic left-invariant pseudo-differential operator of order $q>0,$ and $A\in S^{m,\mathcal{L}}_{\rho,\delta}(G\times \widehat{G}),$ $0\leqslant {\delta},\rho\leqslant 1,$ is a suitable operator in the subelliptic calculus. Again, if we consider the case of the Laplacian on $G,$  we replace $Q$ by $n,$ and for $(\rho,\delta)=(1,0)$ and $m=-n,$ it was proved e.g. in \cite{Fischertraces2}, that in the constant component 
$$\frac{c_{n}}{q}\int\limits_{0}^{\infty}\psi(s)\times \frac{ds}{s},$$ 
the term $c_{n}$  agrees with the Wodzicki residue of $A.$ Other singular traces on compact manifolds can be found   e.g. in the seminal paper of  Kontsevich and  Vishik \cite{KV} where  the canonical trace was introduced, and in other complete references on the subject such as  Fedosov, Golse, and  Leichtnam  \cite{FedosovGolseLeichtnam}, Grubb and Schrohe \cite{GrubbSchrohe}, Scott \cite{Scott}, and Paycha \cite{Paycha}.   We refer the reader to  \cite{Fischertraces1,Fischertraces2} for the treatment of regularised traces (in whose expansions appear the Wodzicki residue and the canonical trace) using the matrix-valued quantisation. A complete investigation about the spectral trace of global operators on compact Lie groups can be found in \cite{DelRuzTrace1,DelRuzTrace11,DelRuzTrace111,DelRuzTrace1111,CardonaWodzicki}.
 
 The second source of applications came from the functional calculus for subelliptic operators developed in Section  \ref{SFC}. Indeed, as an application of the subelliptic functional calculus we will deduce a subelliptic version of the G\r{a}rding inequality and we will study the Dixmier traceability of subelliptic operators.
\end{remark}
\begin{remark} One of the main features of the developed subelliptic calculus is that if we replace the role of the sub-Laplacian by the Laplace operator on the group $G,$ we recover the known properties for the global H\"ormander classes on compact Lie groups \cite{Fischer2015,Ruz,RuzhanskyTurunenWirth2014} which makes our calculus parallel to others existing in the literature, where the global symbols are used e.g. to develop the calculus on graded Lie groups by using Rockland operators \cite{FischerRuzhanskyBook}.
\end{remark}
\begin{remark}
We end this section by providing a short list of references about some recent works related to  the analysis of pseudo-differential operators on compact Lie groups and other related topics with the harmonic analysis of graded Lie groups. For instance,  we refer  to \cite{Poincare,RuzhanskyDelgadoCardona2022TracesDet} for the analysis of traces and determinants on compact Lie groups and on homogeneous vector bundles. For recent works about oscillating Fourier multipliers and criteria for their boundedness in terms of matrix-valued symbols, see \cite{CardonaRuzhanskyOscil1,CardonaRuzhanskyOscil2,CardonaRuzhanskyOscil3,CardonaRuzhanskyOscil4,Dyadic2023} as well as \cite{CardonaRuzhanksyFourierTriebel,CardonaRuzhanksyFourierTriebel2}. For extensions of the sharp G\r{a}rding inequality in the setting of compact Lie groups we refer to \cite{CardonaFedericoRuzhansky}. For applications of the sharp G\r{a}rding inequality to the analysis of fractional diffusion models on compact Lie groups  we refer to \cite{RuzhanskyDelgadoCardona2022JEv}. The  extension of the global functional calculus for the global H\"ormander classes on graded Lie groups can be found in \cite{RuzhanskyDelgadoCardona2022FC}. In \cite{BrinkerWirth}, Gelfand triples for the Kohn-Nirenberg quantization on homogeneous Lie groups have been analysed. For the analysis of diffusive wavelets on compact Lie groups and on homogeneous spaces, see \cite{EbertWirth}.  In \cite{control1,control2,Control3} applications of the pseudo-differential calculus on compact Lie groups to the control theory of heat equations for elliptic operators have been established and also analysed in the setting of arbitrary compact manifolds. Other works related to local Weyl formulas on compact Lie groups can be found in \cite{localweyl}. Finally,   necessary and sufficient conditions by employing this local formula for the membership of global pseudo-differential operators in the Schatten ideals have been investigated in \cite{Toft2023}.
\end{remark}

\section{Preliminaries}\label{Sect2} Throughout the memoir, we shall use the notation $A \lesssim B$ to indicate $A\leq cB $ for a suitable constant $c >0,$   whereas $A \asymp B$ if $A\leq cB$ and $B\leq d A$, for suitable $c, d >0.$

\subsection{Pseudo-differential operators via localisations}
 In this subsection we describe the  well-known  formulation of 
 pseudo-differential operators on compact manifolds (and in particular, on compact Lie groups) associated to symbols defined by local coordinate systems (see H\"ormander \cite{Hormander1985III} and e.g. the book of M. Taylor \cite{Taylorbook1981}). 
 
 If $U$ is an open topological subset of $\mathbb{R}^n,$ we say that  $a:U\times \mathbb{R}^n\rightarrow \mathbb{C},$ belongs to the H\"ormander class $S^m_{\rho,\delta}(U\times \mathbb{R}^n),$ $0\leqslant \rho,\delta\leqslant 1,$ if for every compact subset $K\subset U,$ the symbol inequalities,
\[ 
    |\partial_{x}^\beta\partial_{\xi}^\alpha a(x,\xi)|\leqslant C_{\alpha,\beta,K}(1+|\xi|)^{m-\rho|\alpha|+\delta|\beta|},
\] hold true uniformly in $x\in K$ and $\xi\in \mathbb{R}^n.$ Then, a continuous linear operator $A:C^\infty_0(U) \rightarrow C^\infty(U)$ 
is a pseudo-differential operator of order $m,$ of  $(\rho,\delta)$-type, if there exists
a function $a\in S^m_{\rho,\delta}(U\times \mathbb{R}^n),$ satisfying that
\[ 
    Af(x)=\int\limits_{\mathbb{R}^n}e^{2\pi i x\cdot \xi}a(x,\xi)(\mathscr{F}_{\mathbb{R}^n}{f})(\xi)d\xi,
\] for all $f\in C^\infty_0(U),$ where
\[ 
  (\mathscr{F}_{\mathbb{R}^n}{f})(\xi):=\int\limits_Ue^{-i2\pi x\cdot \xi}f(x)dx,  
\] is the  Euclidean Fourier transform of $f$ at $\xi\in \mathbb{R}^n.$ The class $S^m_{\rho,\delta}(U\times \mathbb{R}^n)$ on the phase space $U\times \mathbb{R}^n,$ is invariant under  changes of coordinates only if $\rho\geqslant   1-\delta,$ while a symbolic calculus (closed for products, adjoints, parametrices, etc.) is only possible for $\delta<\rho$ and $\rho\geqslant   1-\delta.$

In the case of a $C^\infty$-manifold $M,$ a linear continuous operator $A:C^\infty_0(M)\rightarrow C^\infty(M) $ is a pseudo-differential operator of order $m,$ of $(\rho,\delta)$-type, $ \rho\geqslant   1-\delta, $ if for every local  coordinate patch $\omega: M_{\omega}\subset M\rightarrow U\subset \mathbb{R}^n,$
and for every $\phi,\psi\in C^\infty_0(U),$ the operator
\[ 
    Tu:=\psi(\omega^{-1})^*A\omega^{*}(\phi u),\,\,u\in C^\infty(U),\footnote{As usually, $\omega^{*}$ and $(\omega^{-1})^*$ are the pullbacks induced by the maps $\omega$ and $\omega^{-1},$ respectively.}
\] is a pseudo-differential operator with symbol $a_T\in S^m_{\rho,\delta}(U\times \mathbb{R}^n).$ In this case we write that $A\in \Psi^m_{\rho,\delta}(M,\textnormal{loc}).$

\subsection{The positive sub-Laplacian} Let $G$ be a compact Lie group  with Lie algebra $\mathfrak{g}.$ Under the identification $\mathfrak{g}\simeq T_{e_G}G,  $ where $e_{G}$ is the identity element of $G,$ let us consider  a system of $C^\infty$-vector fields $X=\{X_1,\cdots,X_k \}\in \mathfrak{g}$. For all $I=(i_1,\cdots,i_\omega)\in \{1,2,\cdots,k\}^{\omega},$ of length $\omega\geqslant   2,$ denote $$X_{I}:=[X_{i_1},[X_{i_2},\cdots [X_{i_{\omega-1}},X_{i_\omega}]\cdots]],$$ and for $\omega=1,$ $I=(i),$ $X_{I}:=X_{i}.$ Let $V_{\omega}$ be the subspace generated by the set $\{X_{I}:|I|\leqslant \omega\}.$ That $X$ satisfies the H\"ormander condition,  means that there exists $\kappa'\in \mathbb{N}$ such that $V_{\kappa'}=\mathfrak{g}.$ Certainly, we consider the smallest $\kappa'$ with this property and we denote it by $\kappa$ which will be  later called the step of the system $X.$ We also say that $X$ satisfies the H\"ormander condition of order $\kappa.$ 

Note that  the sum of squares
\[ 
    \mathcal{L}\equiv \mathcal{L}_{X}:=-(X_{1}^2+\cdots +X_{k}^2),
\] is a subelliptic operator which by following the usual terminology is called the subelliptic Laplacian associated with the family $X.$ For short we refer to $\mathcal{L}$ as the sub-Laplacian. In view of the H\"ormander theorem on sums of the  squares of vector fields (see H\"ormander \cite{Hormander1967}) it is a hypoelliptic operator (i.e. if $ \mathcal{L}u\in C^\infty(G)$ with $u\in  \mathscr{D}'(G)$ then $u\in C^\infty(G),$ and also locally at all points). 

For other aspects on the analysis of the sub-Laplacian we refer the reader to Agrachev et al. \cite{Agrachev2008}, Bismut \cite{Bismut2008}, Domokos et al. \cite{Domokos} as well as to  the fundamental book of Montgomery \cite{Montgomery}.

A central notion in the analysis of the sub-Laplacian is that of the Hausdorff dimension, in this case, associated to $\mathcal{L}$. Indeed, for all $x\in G,$ denote by $H_{x}^\omega G$ the subspace of the tangent space $T_xG$ generated by the $X_i$'s and all the Lie brackets  $$ [X_{j_1},X_{j_2}],[X_{j_1},[X_{j_2},X_{j_3}]],\cdots, [X_{j_1},[X_{j_2}, [X_{j_3},\cdots, X_{j_\omega}] ] ],$$ with $\omega\leqslant \kappa.$ The H\"ormander condition can be stated as $H_{x}^\kappa G=T_xG,$ $x\in G.$ We have the filtration
\[ 
H_{x}^1G\subset H_{x}^2G \subset H_{x}^3G\subset \cdots \subset H_{x}^{\kappa-1}G\subset H_{x}^\kappa G= T_xG,\,\,x\in G.
\] In our case,  the dimension of every $H_x^\omega G$ does not depend on $x$ and we write that $\dim H^\omega G:=\dim H_{x}^\omega G,$ for any $x\in G.$ So, the Hausdorff dimension can be defined as (see e.g. \cite[p. 6]{HK16}),
\begin{equation}\label{Hausdorff-dimension}
    Q:=\dim(H^1G)+\sum_{i=1}^{\kappa-1} (i+1)(\dim H^{i+1}G-\dim H^{i}G ).
\end{equation}
Explicit examples of sub-Laplacians on compact Lie groups are presented in Section \ref{examplessublaplacians} in the case of $\mathbb{S}^3\cong \textnormal{SU}(2),$ $\textnormal{SU}(3),$ $\textnormal{SO}(4)$ and $\textnormal{Spin}(4)\cong \textnormal{SU}(2)\times \textnormal{SU}(2).$

\subsection{Pseudo-differential operators via global symbols}
In this work we are interested in developing a pseudo-differential calculus associated to the sub-Laplacian $\mathcal{L}$. We will use the quantisation process developed by the second author and V. Turunen in \cite{Ruz}.  We explain it as follows. First, let us record the notion of the unitary dual $\widehat{G}$ of a compact Lie group $G.$ To do this, we start with the fundamental object in this setting which is the definition of a irreducible representation. Indeed, a continuous, unitary and irreducible  representation of $G,$ is:
\begin{itemize}
    \item a mapping $\xi\in \textnormal{Hom}(G, \textnormal{U}(H_{\xi})),$ for some finite-dimensional vector space $H_\xi\cong \mathbb{C}^{d_\xi},$ i.e. $\xi(xy)=\xi(x)\xi(y)$ and for the  adjoint of $\xi(x),$ $\xi(x)^*=\xi(x^{-1}),$ for every $x,y\in G.$
    \item The map $(x,v)\mapsto \xi(x)v, $ from $G\times H_\xi$ into $H_\xi$ is continuous.
    \item For every $x\in G,$ and for any subspace $W_\xi\subset H_\xi,$ if $\xi(x)W_{\xi}\subset W_{\xi},$ then $W_\xi=H_\xi$ or  $W_\xi=\{0\}.$
\end{itemize} Let $\textnormal{Rep}(G)$ be the set of unitary, continuous and irreducible representations of $G.$ The relation, {\small{
\[ 
    \xi_1\sim \xi_2\textnormal{ if and only if, there exists } A\in \textnormal{End}(H_{\xi_1},H_{\xi_2}),\textnormal{ such that }A\xi_{1}(x)A^{-1}=\xi_2(x), 
\]}}for every $x\in G,$ is an equivalence relation and the unitary dual of $G,$ denoted by $\widehat{G}$ is defined via
$$
    \widehat{G}:={\textnormal{Rep}(G)}/{\sim}.
$$By a suitable change of basis, we always can assume that every $\xi$ is matrix-valued and that $H_{\xi}=\mathbb{C}^{d_\xi}.$ 

Note that if a representation $\xi$ is unitary, then $$\xi(G):=  \{\xi(x):x\in G \}$$ is a subgroup (of the group of  matrices $\mathbb{C}^{d_\xi\times d_\xi}$) which is isomorphic to the original group $G$. Thus the homomorphism $\xi$ allows us to represent the compact Lie group $G$ as a group of matrices. This is the motivation for the term `representation'. 

Now, let us follow \cite[Chapter 10]{Ruz} to introduce the analysis of operators on the phase space $G\times \widehat{G}.$ Indeed, if $A$ is a continuous linear operator on $C^\infty(G),$  there exists a function \begin{equation}\label{symbol}a:G\times \widehat{G}\rightarrow \cup_{\ell\in \mathbb{N}} \mathbb{C}^{\ell\times \ell},\end{equation} such that for every equivalence class $[\xi]\in \widehat{G},$ $a(x,\xi):=a(x,[\xi])\in \mathbb{C}^{d_\xi\times d_\xi},$ (where $d_\xi$ is the dimension of the continuous, unitary and irreducible  representation $\xi:G\rightarrow \textnormal{U}(\mathbb{C}^{d_\xi})$) and satisfying
\begin{equation}\label{RuzhanskyTurunenQuanti}
    Af(x)=\sum_{[\xi]\in \widehat{G}}d_\xi\textnormal{\textbf{Tr}}[\xi(x)a(x,\xi)\widehat{f}(\xi)],\,\,f\in C^\infty(G).
\end{equation}Here 
\[ 
    \widehat{f}(\xi)\equiv (\mathscr{F}f)(\xi):=\int\limits_{G}f(x)\xi(x)^*dx\in  \mathbb{C}^{d_\xi\times d_\xi},\,\,\,[\xi]\in \widehat{G},
\]is the matrix-valued Fourier transform of $f$ at $\xi=(\xi_{ij})_{i,j=1}^{d_\xi}.$
The function $a$ in \eqref{symbol}, satisfying \eqref{RuzhanskyTurunenQuanti} is unique and satisfies the identity,\footnote{It is well known that the functions $\xi_{ij}$ are of $C^\infty$-class, that they are eigenfunctions of the positive Laplace operator $\mathcal{L}_G$, and that $\mathcal{L}_G\xi_{ij}=\lambda_{[\xi]}\xi_{ij}$.}
\[ 
    a(x,\xi)=\xi(x)^{*}(A\xi)(x),\,\, A\xi:=(A\xi_{ij})_{i,j=1}^{d_\xi},\,\,\,\,[\xi]\in \widehat{G}.
\]In general, we refer to the function $a$ as the (global or matrix) {\it{symbol}} of the operator $A.$ 

\begin{remark}
Let us denote by $\mathscr{S}(\widehat{G}):=\mathscr{F}(C^\infty(G))$ the Schwartz space on the unitary dual. Then the Fourier transform on the group $\mathscr{F}$ is a bijective mapping from $C^\infty(G)$ into $\mathscr{S}(\widehat{G})$ (see \cite[Page 541]{Ruz}), and in terms of the Fourier transform we have
\[ 
    Af(x)=\mathscr{F}^{-1}[ a(x,\cdot )(\mathscr{F}{f})\,](x),
\]for every $f\in C^{\infty}(G).$ In particular, if $a(x,\xi)=I_{d_\xi}$ is the identity matrix in every representation space, $A\equiv I$ is the identity operator on $C^\infty(G),$ and we recover the Fourier inversion formula
\[ 
    f(x)=\sum_{[\xi]\in \widehat{G}}d_\xi\textnormal{\textbf{Tr}}[\xi(x)\widehat{f}(\xi)],\,\,f\in C^\infty(G).
\]
\end{remark}

In order to classify symbols in the H\"ormander classes  developed in \cite{Ruz}, the notion of {\it{ difference operators}} on the unitary dual, by endowing $\widehat{G}$ with a difference structure, is an instrumental tool. Indeed,   by following  \cite{RuzhanskyWirth2015},   a difference operator $Q_\xi$ of order $k,$  is defined by
\begin{equation}\label{taylordifferences}
    Q_\xi\widehat{f}(\xi)=\widehat{qf}(\xi),\,[\xi]\in \widehat{G}, 
\end{equation} for all $f\in C^\infty(G),$ for some function $q$ vanishing of order $k$ at the identity $e=e_G.$ We will denote by $\textnormal{diff}^k(\widehat{G})$  the set of all difference operators of order $k.$ For a  fixed smooth function $q,$ the associated difference operator will be denoted by $\Delta_q:=Q_\xi.$

For our further analysis, we will choose an admissible collection of difference operators (see e.g. \cite{RuzhanskyDelgado2017,RuzhanskyWirth2015}),
\[ 
  \Delta_{\xi}^\alpha:=\Delta_{q_{(1)}}^{\alpha_1}\cdots   \Delta_{q_{(i)}}^{\alpha_{i}},\,\,\alpha=(\alpha_j)_{1\leqslant j\leqslant i}, 
\]
where
\[ 
    \textnormal{rank}\{\nabla q_{(j)}(e):1\leqslant j\leqslant i \}=\textnormal{dim}(G), \textnormal{   and   }\Delta_{q_{(j)}}\in \textnormal{diff}^{1}(\widehat{G}).
\]We say that this admissible collection is strongly admissible if 
\[ 
    \bigcap_{j=1}^{i}\{x\in G: q_{(j)}(x)=0\}=\{e_G\}.
\]

\begin{remark}\label{remarkD} Matrix components of unitary representations induce difference operators. Indeed, if $\xi_{1},\xi_2,\cdots, \xi_{k},$ are  fixed irreducible and unitary  representation of $G$, which not necessarily belong to the same equivalence class, then each coefficient of the matrix
\begin{equation}
 \xi_{\ell}(g)-I_{d_{\xi_{\ell}}}=[\xi_{\ell}(g)_{ij}-\delta_{ij}]_{i,j=1}^{d_{\xi_\ell}},\, \quad g\in G, \,\,1\leq \ell\leq k,
\end{equation} 
that is each function 
$q^{\ell}_{ij}(g):=\xi_{\ell}(g)_{ij}-\delta_{ij}$, $ g\in G,$ defines a difference operator
\begin{equation}\label{Difference:op:rep}
    \mathbb{D}_{\xi_\ell,i,j}:=\mathscr{F}(\xi_{\ell}(g)_{ij}-\delta_{ij})\mathscr{F}^{-1}.
\end{equation}
We can fix $k\geq \mathrm{dim}(G)$ of these representations in such a way that the corresponding  family of difference operators is admissible, that is, 
\[ 
    \textnormal{rank}\{\nabla q^{\ell}_{i,j}(e):1\leqslant \ell\leqslant k \}=\textnormal{dim}(G).
\]
To define higher order difference operators of this kind, let us fix a unitary irreducible representation $\xi_\ell$.
Since the representation is fixed we omit the index $\ell$ of the representations $\xi_\ell$ in the notation that will follow.
Then, for any given multi-index $\alpha\in \mathbb{N}_0^{d_{\xi_\ell}^2}$, with 
$|\alpha|=\sum_{i,j=1}^{d_{\xi_\ell}}\alpha_{i,j}$, we write
$$\mathbb{D}^{\alpha}:=\mathbb{D}_{1,1}^{\alpha_{11}}\cdots \mathbb{D}^{\alpha_{d_{\xi_\ell},d_{\xi_\ell}}}_{d_{\xi_\ell}d_{\xi_\ell}}
$$ 
for a difference operator of order $|\alpha|$.
\end{remark}
\begin{remark}[Leibniz rule for difference operators]\label{Leibnizrule} The difference structure on the unitary dual $\widehat{G},$ induced by the difference operators acting on the momentum variable $[\xi]\in \widehat{G},$  implies the following Leibniz rule 
\begin{align*}
    \Delta_{q_\ell}(a_{1}a_{2})(x_0,\xi) =\sum_{ |\gamma|,|\varepsilon|\leqslant \ell\leqslant |\gamma|+|\varepsilon| }C_{\varepsilon,\gamma}(\Delta_{q_\gamma}a_{1})(x_0,\xi) (\Delta_{q_\varepsilon}a_{2})(x_0,\xi), \quad (x_{0},[\xi])\in G\times \widehat{G},
\end{align*} for $a_{1},a_{2}\in C^{\infty}(G, \mathscr{S}'(\widehat{G})).$ For details we refer the reader to  \cite{Ruz,RuzhanskyTurunenWirth2014}.
\end{remark}
\begin{remark}
 Every $X\in\mathfrak{g},$ can be identified with the differential operator $X:C^\infty(G)\rightarrow C^\infty(G)$  defined by
 \[ 
     (X_{x}f)(x):=\frac{d}{dt}f(x\exp(tX) )|_{t=0},\quad x\in G.
 \]
\end{remark}
If $A\in \Psi^m_{\rho,\delta}(G,\textnormal{loc}),$ $\rho\geqslant   1-\delta,$ the matrix-valued symbol $\sigma_A$ of $A$ satisfies (see \cite{Ruz,RuzhanskyTurunenWirth2014}),
\begin{equation}\label{HormanderSymbolMatrix}
    \Vert {X}_x^\beta \Delta_{q_\gamma} \sigma_A(x,\xi)\Vert_{\textnormal{op}}\leqslant C_{\alpha,\beta}
    \langle \xi \rangle^{m-\rho|\gamma|+\delta|\beta|}\end{equation} for all $\beta$ and  $\gamma $ multi-indices and all $(x,[\xi])\in G\times \widehat{G}$. Now, if $0\leqslant \delta,\rho\leqslant 1,$
we say that $\sigma_A\in \mathscr{S}^m_{\rho,\delta}(G\times \widehat{G}),$ if the global symbol inequalities \eqref{HormanderSymbolMatrix}  hold true. So, for $\sigma_A\in \mathscr{S}^m_{\rho,\delta}(G\times \widehat{G})$ we write $A\in\Psi^m_{\rho,\delta}(G)\equiv\textnormal{Op}(\mathscr{S}^m_{\rho,\delta}(G\times \widehat{G})).$  As we mentioned early in the introduction we have the following equality of classes,
\[ 
   \textnormal{Op}(\mathscr{S}^m_{\rho,\delta}(G\times \widehat{G}))= \Psi^m_{\rho,\delta}(G,\textnormal{loc}),\,\,\,0\leqslant \delta<\rho\leqslant 1,\,\rho\geqslant   1-\delta.
\]

\subsection{Calder\'on-Zygmund type estimates for multipliers}
In order to provide $L^p$-estimates for multipliers in the subelliptic context, we will use the techniques developed by the second author and J. Wirth in \cite{RuzhanskyWirth2015}, where a special case (compatible with the notion of difference operators) of a statement of Coifman and de Guzm\'an (\cite{CoifmandeGuzman}, Theorem 2) was established. We record it as follows  (see   \cite[p. 630]{RuzhanskyWirth2015}).  
    \begin{cri}\label{CoifDeGuzCrit} Assume that $A:L^2(G)\rightarrow L^2(G)$ is a left-invariant operator on $G$ satisfying 
 \begin{equation}\label{CoifDeGuz}
  \Vert A\psi_r\Vert_{L^2(G,\rho(x)^{n(1+\varepsilon)}dx )}  := \left(\,\int\limits_{G}|A\psi_{r}(x)|^2\rho(x)^{n(1+\varepsilon)}dx\right)^\frac{1}{2}\leqslant Cr^{\frac{\varepsilon}{2}},
 \end{equation}for some constants $C>0$ and $\varepsilon>0,$ uniformly in $r.$ Then $A$ is of weak type $(1,1)$ and bounded on $L^p(G),$ for all $1<p<\infty.$
 \end{cri}
 
 The family $\{\psi_r\}_{r>0}$  that appears in Criterion \eqref{CoifDeGuzCrit} is defined by  $\psi_{r}=\phi_{r}-\phi_{r/2},$ where the functions in the net $\{\phi_r\}_{r>0},$ satisfy, among other things,  the following properties (see \cite[p. 140]{CoifmandeGuzman}):
 \begin{itemize}
     \item $\int\limits_{G}\phi_{r}(x)dx=1,$ 
     \item $\int\limits_{G}\phi_{r}^2(x)dx=O(\frac{1}{r}),$
     \item $\phi_r\ast\phi_s=\phi_s\ast \phi_r,$ $r,s>0.$
 \end{itemize}
  The function $\rho:x\mapsto\rho(x),$  appearing in \eqref{CoifDeGuz}, is a suitable pseudo-distance defined on $G.$ If $G$ is semi-simple (this means that the centre of $G,$ $Z(G)$ is trivial), it is defined by
  \begin{equation}\label{therhofucntion}
      \rho(x)^2:=\dim(G)-\textnormal{\textbf{Tr}}(\textnormal{Ad}(x))=\sum_{\xi\in \Delta_0}(d_\xi-\textnormal{\textbf{Tr}}(\xi(x))),\,\,x\in G,
  \end{equation}
 where $\textnormal{Ad}:G\rightarrow\textnormal{U}(\mathfrak{g}),$ and $\Delta_0$ is the system of positive roots. It can be decomposed into irreducible representations as,
 \[ 
     \textnormal{Ad}=[\textnormal{rank}(G)e_{\widehat{G}}]\oplus\left( \bigoplus_{\xi\in \Delta_{0}}\xi\right),
 \]
 where $e_{\widehat{G}}$ is the trivial representation.
With the consideration on the centre $Z(G)=\{e_{G}\},$ it can be shown (see Lemma 3.1 of \cite{RuzhanskyWirth2015}) that 
\begin{itemize}
    \item $\rho^2(x)\geqslant   0$ and $\rho(x)=0$ if and only if $x=e_{G}.$
    \item $\Delta_{\rho^2}\in \textnormal{diff}^2(\widehat{G}).$
\end{itemize}
 If $G$ is not semi-simple,  we refer the reader to \cite[Remark 3.2]{RuzhanskyWirth2015} for the modifications in the definition of $\rho,$ in this particular case. For our further analysis, we will use the following lemma which exploits the properties of the functions $\psi_r,$ (see Lemma 3.4 of \cite[p. 630]{RuzhanskyWirth2015}).
 \begin{lemma}\label{LemmaRuzhWirth} Let $q\in C^\infty(G)$ be a smooth function vanishing to order $\tilde\ell\in \mathbb{R} ,$ at $e_G.$ Then
 \[ 
     \Vert q(x)\psi_{r}\Vert_{H^{-s}(G)}\leqslant C_{q,s} r^{\frac{\tilde\ell+s}{n}-\frac{1}{2}},
 \] for all $0\leqslant s\leqslant 1+\frac{n}{2}.$
 \end{lemma}

\subsection{$L^p$-multipliers and $L^p$-bounded pseudo-differential operators }\label{Lpcompact1}
We record the $L^p$-estimates for multipliers on compact Lie groups through  the methods developed by the second author and J. Wirth in \cite{RuzhanskyWirth2015} by using Criterion \ref{CoifDeGuzCrit}.    We will denote by $\Sigma(\widehat{G}\times G)$ and $\Sigma(\widehat{G}) $ the space of  matrix-valued functions,
\[ 
    \Sigma(G\times \widehat{G}):=\{ \sigma: G\times \widehat{G}\rightarrow \cup_{[\xi]\in \widehat{G}}\mathbb{C}^{d_\xi\times d_\xi}\},
\]
 \[ 
    \Sigma(\widehat{G}):=\{ \sigma:  \widehat{G}\rightarrow \cup_{[\xi]\in \widehat{G}}\mathbb{C}^{d_\xi\times d_\xi}\}.
\]

\begin{theorem}\label{TRW2015ZIP}
Let $G$ be a compact Lie group   and let $\varkappa\in 2\mathbb{N}$ be such that $\varkappa>\frac{n}{2}$. Let $a\in \Sigma(\widehat{G})$ be a symbol, satisfying
\[ 
    \Vert\mathbb{D}^{\alpha} a(\xi)\Vert_{\textnormal{op}}\leqslant C_\alpha\langle \xi\rangle^{-|\alpha|},\,\,|\alpha|\leqslant \varkappa.
\] Then $A=\textnormal{Op}(a)$ is of weak type $(1,1)$  and bounded on $L^p(G)$ for all $1<p<\infty.$ Moreover, if $0\leqslant \rho<1,$ and $a$ satisfies 
\[ 
    \Vert\mathbb{D}^{\alpha} a(\xi)\Vert_{\textnormal{op}}\leqslant C_\alpha\langle \xi \rangle^{-\rho|\alpha|},\,\,|\alpha|\leqslant \varkappa,
\] then $A$ extends to a bounded operator from $L^p_{r}(G)$ into $L^p(G)$ for all $1<p<\infty$ and $r=\varkappa(1-\rho) \left|\frac{1}{p}-\frac{1}{2}\right |.$ Here $L^p_r(G)$ denotes the Sobolev space of order $r$ over $L^p(G).$
\end{theorem}

From the proof of  Corollary 5.1 of \cite{RuzhanskyWirth2015}, one has the following version of Theorem \ref{TRW2015ZIP}.
\begin{theorem}\label{LpQ}
Let us assume that $G$ is a compact Lie group of dimension $n=2d$ or $n=2d+1,$ and that  $d$ is odd. Let  $0< \rho\leqslant  1,$ and   $a\in \Sigma(\widehat{G})$ be a symbol satisfying
\[ 
    \Vert\mathbb{D}^{\alpha} a(\xi)\Vert_{\textnormal{op}}\leqslant C_\alpha\langle \xi\rangle^{-\varkappa(1-\rho)-\rho|\alpha|},\,\,|\alpha|\leqslant \varkappa:=d+1,\, 
\] then $A=\textnormal{Op}(a)$ extends to a linear operator of weak type $(1,1).$  Moreover, if the dimension of the group is $\dim(G)=2d$ or $\dim(G)=2d+1,$ and   $d$ is even, the conclusion on $A$ is the same provided that
\[ 
    \Vert\mathbb{D}^{\alpha} a(\xi)\Vert_{\textnormal{op}}\leqslant C_\alpha\langle \xi\rangle^{-\varkappa(1-\rho)-\rho|\alpha|},\,\,|\alpha|\leqslant \varkappa:=d+2. 
\]
\end{theorem}
The argument  developed  in  \cite{RuzhanskyWirth2015} using the Sobolev embedding theorem for extending the $L^p$-estimates from multipliers to pseudo-differential operators, allows us to present the following consequence of  Theorem \ref{LpQ} (see Theorem 5.2 of \cite{RuzhanskyWirth2015}).
\begin{theorem}\label{LpQnoninvariant}
Let us assume that $G$ is a compact Lie group of dimension $n=2d$ or $n=2d+1,$ and that $d$ is odd. Let     $a\in \Sigma(G\times \widehat{G})$ be a non-invariant symbol satisfying
\[ 
    \Vert{X}_x^\beta\mathbb{D}^{\alpha} a(x,\xi)\Vert_{\textnormal{op}}\leqslant C_\alpha\langle \xi\rangle^{-|\alpha|},\,\,|\alpha|\leqslant \varkappa:=d+1,\,\, |\beta|\leqslant \left[\frac{n}{p}\right]+1,
\] then $A=\textnormal{Op}(a)$ extends to a bounded operator on $L^p(G)$ for all $1<p<\infty.$ Moreover, if the dimension of the group is $\dim(G)=2d$ or $\dim(G)=2d+1,$ and  $d$ is even, the conclusion on $A$ is the same provided that
\[ 
    \Vert{X}_x^\beta\mathbb{D}^{\alpha} a(x,\xi)\Vert_{\textnormal{op}}\leqslant C_\alpha\langle \xi\rangle^{-|\alpha|},\,\,|\alpha|\leqslant \varkappa:=d+2,\,\,|\beta|\leqslant \left[\frac{n}{p}\right]+1. 
\]
\end{theorem}
The following theorem records the action of the H\"ormander classes $\mathscr{S}^{m}_{\rho,\delta}(G\times \widehat{G})$ (see \eqref{smrd}) on $L^p(G)$ spaces (see  \cite{RuzhanskyDelgado2017}).
\begin{theorem}\label{DelgadoRuzhanskyLppseudo}
Let $G$ be a compact Lie group of dimension $n.$ Let $0\leqslant \delta<\rho\leqslant 1.$ Let $\sigma\in \mathscr{S}^{-\nu}_{\rho,\delta} (G\times \widehat{G}) ,$ where $\nu\in \mathbb{R}.$ Then $A\equiv \sigma(x,D)$ extends to a bounded operator on $L^p(G)$ provided that
\begin{equation}\label{sharporder}
    \nu\geqslant   n(1-\rho)\left|\frac{1}{p}-\frac{1}{2}\right|.
\end{equation}
\end{theorem}
\begin{remark}[Sharpness of Theorem \ref{DelgadoRuzhanskyLppseudo}]\label{sharpexample} Let $G$ be a connected, simply connected, compact semi-simple Lie group of dimension $n.$ Theorem \ref{DelgadoRuzhanskyLppseudo} is sharp in the  following sense. Let $\Lambda$ be the lattice of  the highest weights, $\Delta_0$ be the system of positive roots and $\delta$ be the half sum of all positive roots. There exists a correspondence $\lambda\mapsto [\xi_\lambda]$ between $\Lambda$ and the unitary dual $\widehat{G}.$ For $0<\beta<1,$ $\alpha>0,$ let us define the oscillating  Fourier multiplier,
\[ 
    T_{\alpha,\beta,\Omega}f(x)=\sum_{\lambda+\delta\in \Lambda\setminus\{0\}}d_{\xi_\lambda}\textnormal{\textbf{Tr}}[\xi_\lambda(x)\sigma_{\alpha,\beta}(\xi_\lambda)\widehat{f}(\xi_\lambda)],\,\,
\] with symbol
\[ 
    \sigma_{\alpha,\beta}(\xi_\lambda):=\frac{ e^{i\Vert \lambda+\delta\Vert^{\beta}} \Omega(\frac{\lambda+\delta}{\Vert \lambda+\delta \Vert})}{ \Vert \lambda+\delta\Vert^{\alpha} }I_{d_{\xi_\lambda}},
\]
where $\Omega$ is a sufficiently smooth function that is symmetric under the dual Weyl group. Here, $\Vert \alpha \Vert^2$ is the quadratic form $\Vert \alpha\Vert^2=\langle \alpha,\alpha\rangle$ on $\mathfrak{g}^{*}$ induced from the Killing form. It was proved by Chen and Fan (see \cite[Theorem 1]{Chen}) that $T_{\alpha,\beta,\Omega}$ is bounded on $L^p(G),$ $1<p<\infty,$ if and only if,  $$ \alpha\geqslant  \alpha_{p}:= n\beta\left|\frac{1}{2}-\frac{1}{p}\right|.$$ So, for $\alpha< \alpha_{p}:=n\beta\left|\frac{1}{2}-\frac{1}{p}\right|,$ there exists a bounded sequence $f_n$ in $L^p(G)$ such that $\Vert T_{\alpha,\beta,\Omega}f_n \Vert_{L^p(G)}\rightarrow \infty.$ The main point in this remark is that $\sigma_{\alpha,\beta}\in \mathscr{S}^{-\nu}_{\rho,0}(G\times \widehat{G})$ where $\nu=\alpha$ and $\rho=1-\beta.$ Consequently,  $T_{\alpha,\beta,\Omega}$ is bounded on $L^p(G),$ $1<p<\infty,$ if and only if,  $ \nu\geqslant   n(1-\rho)|\frac{1}{2}-\frac{1}{p}|,$ which implies that the condition \eqref{sharporder} is sharp. The aforementioned  
Chen and Fan's estimate is an  extension of the classical estimate due to Wainger \cite{Wainger1965}.

\end{remark}

We end this section with the following theorem where the important cases of symbols with limited regularity are considered, see \cite{RuzhanskyDelgado2017}.
\begin{theorem}\label{DelgadoRuzhanskylimitedregularity}
Let $0\leqslant \delta, \rho\leqslant 1,$ and let $\varkappa$ be the smallest even integer larger than $n/2,$ $n:=\dim(G).$ If $A:C^\infty(G)\rightarrow\mathscr{D}'(G)$ is a continuous operator such that its matrix symbol $\sigma_A$ satisfies 
\[ 
    \Vert {X}_x^\gamma\Delta^\alpha \sigma_A(x,\xi)\Vert_{\textnormal{op}}\leqslant C_{\gamma,\beta}\langle \xi \rangle^{-\nu_0-\rho|\alpha|+\delta|\gamma|},\,\,[\xi]\in \widehat{G},\,|\alpha|\leqslant \varkappa, |\gamma|\leqslant [n/p]+1,
\]
with
\[ 
    \nu_0\geqslant   \varkappa(1-\rho)\left|\frac{1}{p}-\frac{1}{2}\right|+\delta\left(\left[\frac{n}{p}\right]+1\right),
\]
then  $A$ extends to a bounded operator from $L^{p}(G)$ into $L^p(G)$ for all $1<p<\infty.$
\end{theorem}

\subsection{The subelliptic spaces $H^1$ and $BMO$ on compact Lie groups}\label{BMOH1}
 Let $G$ be a compact Lie group. Let us consider a sub-Laplacian $\mathcal{L}=-(X_1^2+\cdots +X_k^2)$ on $G,$ where the system of vector fields $X=\{X_i\}_{i=1}^{k}$ satisfies the H\"ormander condition of step $\kappa$. For every point $g\in G,$ let us denote  $X_{g}=\{X_{i,g}\}_{i=1}^{k},$  $\mathcal{H}_{g}=\textnormal{span}\{X_{g}\}.$ We say that a curve $\gamma:[0,1]\rightarrow G$ is  horizontal if  $$\Dot{\gamma}(t)\in \mathcal{H}_{\gamma(t)},\textnormal{   for a.e.  }t\in (0,1).$$ The Carnot-Carath\'eodory distance associated to the sub-Riemannian structure induced by $X,$ is defined by
 \[ 
     d_{s}(g_{0},g_{1}):=\inf_{\gamma \textnormal{ horizontal   }}\{l(\gamma):=\int\limits_{0}^{1}|\Dot{\gamma}(t)|dt:\,\,\gamma(0)=g_0,\,\gamma(1)=g_1\},\,\,g_{0},g_{1 }\in G.
 \]
 
 We will fix a subelliptic norm on $G,$ $|\cdot|,$ defined by the Carnot-Carath\'eodory distance in the natural way: $|g|=d_s(g, e_G),$ 
where $e_{G}$ is the identity element of $G.$  As usual, the ball of radius $r>0,$ is defined as 
\[ 
    B(x,r)=\{y\in G:|y^{-1}x|<r\}.
\]Then subelliptic $BMO$ space on $G,$ which we denote by $\textnormal{BMO}^{\mathcal{L}}(G),$ is the space of locally integrable functions $f$ satisfying
\[ 
    \Vert f\Vert_{\textnormal{BMO}^{\mathcal{L}}(G)}:=\sup_{\mathbb{B}}\frac{1}{|\mathbb{B}|}\int\limits_{\mathbb{B}}|f(x)-f_{\mathbb{B}}|dx<\infty,\textnormal{ where  } f_{\mathbb{B}}:=\frac{1}{|\mathbb{B}|}\int\limits_{\mathbb{B}}f(x)dx,
\]
and $\mathbb{B}$ ranges over all balls $B(x_{0},r),$ with $(x_0,r)\in G\times (0,\infty).$ 

On the other hand, the subelliptic  Hardy space $\textnormal{H}^{1,\mathcal{L}}(G)$ will be defined via the atomic decomposition. Thus, $f\in \textnormal{H}^{1,\mathcal{L}}(G)$ if and only if $f$ can be expressed as $$f=\sum_{j=1}^\infty c_{j}a_{j},$$ where $\{c_j\}_{j=1}^\infty$ is a sequence in $\ell^1(\mathbb{N}),$ and every function $a_j$ is an atom, i.e., $a_j$ is supported in some ball $B_j,$ ($a_j$ satisfies the cancellation property) $$\int\limits_{B_j}a_{j}(x)dx=0,$$ and 
\[ 
    \Vert a_j\Vert_{L^\infty(G)}\leqslant \frac{1}{|B_j|}.
\] The norm $\Vert f\Vert_{\textnormal{H}^{1,\mathcal{L}}(G)}$ is the infimum over  all possible series $\sum_{j=1}^\infty|c_j|.$ Furthermore $\textnormal{BMO}^{\mathcal{L}}(G)$ is the dual space of $\textnormal{H}^{1,\mathcal{L}}(G)$. This can be understood in the following sense:
\begin{itemize}
    \item[(a)] if $\phi\in \textnormal{BMO}^{\mathcal{L}}(G), $ then $$\Phi: f\mapsto \int\limits_{G}f(x)\phi(x)dx,$$ admits a bounded extension on $\textnormal{H}^{1,\mathcal{L}}(G).$
    \item[(b)] Conversely, every continuous linear functional $\Phi$ on $\textnormal{H}^{1,\mathcal{L}}(G)$ arises as in $\textnormal{(a)}$ with a unique element $\phi\in \textnormal{BMO}^{\mathcal{L}}(G).$
\end{itemize} The norm of $\phi$ as a linear functional on $\textnormal{H}^{1,\mathcal{L}}(G)$ is equivalent with the $\textnormal{BMO}^{\mathcal{L}}(G)$-norm. Important properties of the $\textnormal{BMO}^{\mathcal{L}}(G)$ and the $\textnormal{H}^{1,\mathcal{L}}(G)$ norms are the following,
\begin{equation}\label{BMOnormduality}
 \Vert f \Vert_{\textnormal{BMO}^{\mathcal{L}}(G)}  =\sup_{\Vert g\Vert_{\textnormal{H}^{1,\mathcal{L}}(G)}=1} 
\left| \int\limits_{G}f(x)g(x)dx\right|,\end{equation}
\begin{equation}\label{BMOnormduality'}\Vert g \Vert_{\textnormal{H}^{1,\mathcal{L}}(G)}  =\sup_{\Vert f\Vert_{\textnormal{BMO}^{\mathcal{L}}(G)}=1} 
\left| \int\limits_{G}f(x)g(x)dx\right|.
\end{equation} 
If we replace $\mathcal{L}$ by the Laplacian $\mathcal{L}_G$ in the definitions above, we will write  $BMO(G)$ and $H^1(G),$ defined by the distance induced by the usual bi-invariant Riemannian metric on $G.$ The subelliptic Fefferman-Stein interpolation theorem in this case can be stated as follows (see e.g. Carbonaro, Mauceri and Meda \cite{CMM}). 
\begin{theorem}
Let $G$ be a compact Lie group. Let us consider a sub-Laplacian $\mathcal{L}=-(X_1^2+\cdots +X_k^2)$ on $G,$ where the system of vector fields $X=\{X_i\}_{i=1}^{k}$ satisfies the H\"ormander condition of step $\kappa$. For every $\theta\in (0,1),$ we have,
\begin{itemize}
    \item If $p_\theta=\frac{2}{1-\theta},$ then $(L^2,\textnormal{BMO}^{\mathcal{L}})_{[\theta]}(G)=L^{p_\theta}(G).$ 
    \item If $p_\theta=\frac{2}{2-\theta},$ then $(\textnormal{H}^{1,\mathcal{L}},L^2)_{[\theta]}(G)=L^{p_\theta}(G).$ 
\end{itemize}
\end{theorem}

\section{Subelliptic pseudo-differential operators}\label{Seccionsubelliptic}
\subsection{Subelliptic  symbols on compact Lie groups}

In order to establish the basic properties of the subelliptic symbolic calculus, and as it was pointed out in \cite{Ruz}, we will use as in the Euclidean case, the expansion of smooth functions in the Taylor series. Although it has been constructed in \cite[Section 10.6]{Ruz} for arbitrary compact Lie groups, and we apply it in further sections, we  will explain this notion in the case of compact connected Lie groups. 
 
\begin{remark}[Local Taylor series on compact connected Lie groups] If $V$ is an open and convex subset of $\mathbb{R}^{{n}},$ and $h\in C^{\infty}(V),$ the Taylor polynomial of order $N$ at $y_0\in V,$ is given by 
 \begin{align*} h(y)=(p_{y_0,N}h)(y)+(R_{N}h)(y,y_0),   \,\,(p_{y_0,N}h)(y):=\sum_{|\alpha|\leqslant N}  \frac{(y-y_0)^{\alpha}}{\alpha!}\frac{\partial^\alpha h}{\partial y^\alpha}(y_0),
\end{align*} for all $y\in V,$ where  
\[ 
   (R_{N}h)(y, y_0):= (N+1)\sum_{|\alpha|=N+1}\frac{(y-y_0)^{\alpha}}{\alpha!}\int\limits_{0}^{1}(1-t)^N\frac{\partial^\alpha h}{\partial y^\alpha}(y_0+t(y-y_0))dt.
\] Now, let $G$ be a connected compact Lie group with Lie algebra $\mathfrak{g}=\textnormal{Lie}(G),$ and of dimension $n.$  The exponential map \begin{equation}\label{exponentialmap}\textnormal{exp}: \mathfrak{g}\rightarrow G,
\end{equation}    is a local diffeomorphism from an open  neighbourhood of $0_{\mathfrak{g}}$ into an open set containing the identity $e_G$ of $G.$   Moreover, both hypothesis, connectedess and compactness, assure the completeness of $G$ and hence  the exponential map is surjective according to the Hopf-Rinow Theorem.

So, if $f\in C^\infty(G),$ and $x\in G,$
$f(x)=f(\textnormal{exp}(X))$ for some  $X\in\mathfrak{g}.$ By defining the function $\tilde{f}:=f\circ \textnormal{exp},$ and under the identification $\mathfrak{g}\simeq \mathbb{R}^n,$ we can think that $\tilde{f}$ is defined on $\mathbb{R}^n$ and consequently, we have the expansion at $0_{\mathfrak{g}}\sim {0}\in \mathbb{R}^n,$
\begin{align*}
    \tilde{f}(X)=(p_{0,N} \tilde{f})(X)+(R_{0,N} \tilde{f})(X,0),\,\,\,X\in\mathfrak{g}.
\end{align*}Returning to the coordinates of $G$ with the inverse exponential mapping, we have
\begin{align*}
    f(x)=\sum_{|\alpha|\leqslant N}  \frac{x^{\alpha}}{\alpha!} & \frac{\partial^\alpha  (f\circ\textnormal{exp})}{\partial y^\alpha}(0)\\
    &+(N+1)\sum_{|\alpha|=N+1}\frac{x^{\alpha}}{\alpha!}\int\limits_{0}^{1}(1-t)^N\frac{\partial^\alpha (f\circ \exp)}{\partial y^\alpha}(t\exp^{-1}(x))dt,
\end{align*}where, for a basis $B=\{X_1,X_2,\cdots,X_n\}$ of $ \mathfrak{g}$ we have used the multi-index notation  $x^\alpha:=\omega(X_1)^{\alpha_1}\cdots \omega(X_n)^{\alpha_n}=y_1^{\alpha_1}\cdots y_n^{\alpha_n},$ again using the identification between  $\mathfrak{g} $ and $\mathbb{R}^n,$ under the diffeomorphism $\omega:\mathfrak{g}\rightarrow\mathbb{R}^{n},$ $y_{i}=\omega(X_i),$ where $y_{i},$ $1\leqslant i\leqslant n,$ are the coordinate functions on $\mathbb{R}^n.$ We will use the notation
 \end{remark} 
 \begin{equation}\label{diffoper}
    \partial_{x_i}^{\alpha_i}f:= \frac{\partial^{\alpha_i}  (f\circ\textnormal{exp})}{\partial y_i^{\alpha_i}},\,\,\,  \partial_{x}^{\alpha}f:= \frac{\partial^\alpha  (f\circ\textnormal{exp})}{\partial y^\alpha}\equiv \partial_{x_1}^{\alpha_1}\partial_{x_2}^{\alpha_2}\cdots \partial_{x_n}^{\alpha_n}f ,
 \end{equation}
 for the local differential operators  appearing in the local Taylor series. 
 However, in order to introduce our subelliptic classes, we need a suitable Taylor expansion associated with a suitable system of vector fields. So, we present it in the following Lemma (see Lemma 7.4 in \cite{Fischer2015}).

 \begin{lemma}[Global Taylor Series on compact Lie groups]\label{Taylorseries} Let $G$ be a compact Lie group of dimension $n.$ Let us consider an strongly  admissible collection of difference operators $\mathfrak{D}=\{\Delta_{q_{(j)}}\}_{1\leqslant j\leqslant n}$, which means that 
\[ 
    \textnormal{rank}\{\nabla q_{(j)}(e):1\leqslant j\leqslant n \}=n, \,\,\,\bigcap_{j=1}^{n}\{x\in G: q_{(j)}(x)=0\}=\{e_G\}.
\]Then there exists a basis $X_{\mathfrak{D}}=\{X_{1,\mathfrak{D}},\cdots ,X_{n,\mathfrak{D}}\}$ of $\mathfrak{g},$ such that $X_{j,\mathfrak{D}}q_{(k)}(\cdot^{-1})(e_G)=\delta_{jk}.$ Moreover, by using the multi-index notation $$\partial_{X}^{(\beta)}=\partial_{X_{i,\mathfrak{D}}}^{\beta_1}\cdots \partial_{X_{n,\mathfrak{D}}}^{\beta_n}, \,\beta\in\mathbb{N}_0^n,$$
where $$\partial_{X_{i,\mathfrak{D}}}f(x):=  \frac{d}{dt}f(x\exp(tX_{i,\mathfrak{D}}) )|_{t=0},\,\,f\in C^{\infty}(G),$$ and denoting
for every $f\in C^{\infty}(G)$
\[ 
    R_{x,N}^{f}(y):=f(xy)-\sum_{|\alpha|<N}q_{(1)}^{\alpha_1}(y^{-1})\cdots q_{(n)}^{\alpha_n}(y^{-1})\partial_{X}^{(\alpha)}f(x),
\] we have the inequality
\[ 
    | R_{x,N}^{f}(y)|\leqslant C|y|^{N}\max_{|\alpha|\leqslant N}\Vert \partial_{X}^{(\alpha)}f\Vert_{L^\infty(G)}.
\]The constant $C>0$ is dependent on $N,$ $G$ and $\mathfrak{D},$ but not on $f\in C^\infty(G).$ Also, we have the identity $\partial_{X}^{(\beta)}|_{x_1=x}R_{x_1,N}^{f}=R_{x,N}^{\partial_{X}^{(\beta)}f}$ as well as the inequality 
\[ 
    | \partial_{X}^{(\beta)}|_{y_1=y}R_{x,N}^{f}(y_1)|\leqslant C|y|^{N-|\beta|}\max_{|\alpha|\leqslant N-|\beta|}\Vert \partial_{X}^{(\alpha+\beta)}f\Vert_{L^\infty(G)},
\]provided that $|\beta|\leqslant N.$
 \end{lemma} 
 Now with the notation above and with the following one $\Delta_{\xi}^\alpha:=\Delta_{q_{(1)}}^{\alpha_1}\cdots   \Delta_{q_{(n)}}^{\alpha_{n}},$ for arbitrary compositions of difference operators, we introduce the subelliptic H\"ormander class of symbols of order $m\in \mathbb{R},$ in the $(\rho,\delta)$-class.  Here, $\widehat{ \mathcal{M}}$ is the matrix-valued symbol of the subelliptic Bessel potential $$\mathcal{M}:=(1+\mathcal{L})^{\frac{1}{2}},$$ and  for every $[\xi]\in \widehat{G},$ we define
   \[ 
       \widehat{ \mathcal{M}}(\xi)^{s}:=\textnormal{diag}[(1+\nu_{ii}(\xi)^2)^{s/2 }]_{1\leqslant i\leqslant d_\xi}, \,s\in \mathbb{R},
   \] where $$\widehat{\mathcal{L}}(\xi)=:\textnormal{diag}[\nu_{ii}(\xi)^2]_{1\leqslant i\leqslant d_\xi}$$ is the symbol of the sub-Laplacian $\mathcal{L}$ at $[\xi].$ 
\begin{definition}[Dilated subelliptic H\"ormander classes]\label{dilated}
   Let $G$ be a compact Lie group and let  $0\leqslant \delta,\rho\leqslant \kappa.$ Let us consider a sub-Laplacian $\mathcal{L}=-(X_1^2+\cdots +X_k^2)$ on $G,$ where the system of vector fields $X=\{X_i\}_{i=1}^{k}$ satisfies the H\"ormander condition of step $\kappa$. The dilated class $\mathscr{S}^{m,\mathcal{L}}_{\rho,\delta}(G)$ of subelliptic H\"ormander order $m$ and of type $(\rho,\delta),$ consists of those functions $\sigma \in \Sigma(G\times \widehat{G}),$ satisfying the symbol inequalities
   \begin{equation}\label{InI}
      p_{\alpha,\beta,\rho,\delta,m,\textnormal{left}}(a):= \sup_{(x,[\xi])\in G\times \widehat{G} }\Vert \widehat{ \mathcal{M}}(\xi)^{\frac{1}{\kappa}(\rho|\alpha|-\delta|\beta|-m)}\partial_{X}^{(\beta)} \Delta_{\xi}^{\alpha}a(x,\xi)\Vert_{\textnormal{op}} <\infty,
   \end{equation}
   and 
   \begin{equation}\label{InII}
      p_{\alpha,\beta,\rho,\delta,m,\textnormal{right}}(a):= \sup_{(x,[\xi])\in G\times \widehat{G} }\Vert (\partial_{X}^{(\beta)} \Delta_{\xi}^{\alpha} a(x,\xi) ) \widehat{ \mathcal{M}}(\xi)^{\frac{1}{\kappa}(\rho|\alpha|-\delta|\beta|-m)}\Vert_{\textnormal{op}} <\infty.
   \end{equation}
\end{definition}
\begin{remark}Note that in contrast with the usual conditions $0\leqslant \rho,\delta\leqslant 1,$  appearing  in  Section \ref{Sect2} for H\"ormander classes on manifolds (by using charts),  in the definition of dilated subelliptic H\"ormander classes, we allow $0\leqslant \rho,\delta\leqslant \kappa,$ in view of the normalisation factor $\frac{1}{\kappa}$   in \eqref{InI} and \eqref{InII}. To establish a calculus with the usual conditions on $\rho$ and $\delta$ we will define the following contracted  classes. The reason for their definition will be clear in Example \ref{powesfirstpart}.
\end{remark}
\begin{definition}[Contracted subelliptic H\"ormander classes]\label{contracted''}
   Let $G$ be a compact Lie group and let $0\leqslant \delta,\rho\leqslant 1.$ Let us consider a sub-Laplacian $\mathcal{L}=-(X_1^2+\cdots +X_k^2)$ on $G,$ where the system of vector fields $X=\{X_i\}_{i=1}^{k}$ satisfies the H\"ormander condition of step $\kappa$. A symbol $\sigma$ belongs to the contracted class ${S}^{m,\mathcal{L}}_{\rho,\delta}(G\times \widehat{G})$ if $\sigma\in \mathscr{S}^{m\kappa,\mathcal{L}}_{\rho\kappa,\delta\kappa}(G).$ This means that $\sigma $ satisfies  the symbol inequalities
   \begin{equation}\label{InIC}
      p_{\alpha,\beta,\rho,\delta,m,\textnormal{left}}(\sigma)':= \sup_{(x,[\xi])\in G\times \widehat{G} }\Vert \widehat{ \mathcal{M}}(\xi)^{(\rho|\alpha|-\delta|\beta|-m)}\partial_{X}^{(\beta)} \Delta_{\xi}^{\alpha}\sigma(x,\xi)\Vert_{\textnormal{op}} <\infty,
   \end{equation}
   and 
   \begin{equation}\label{InIIC}
      p_{\alpha,\beta,\rho,\delta,m,\textnormal{right}}(\sigma)':= \sup_{(x,[\xi])\in G\times \widehat{G} }\Vert (\partial_{X}^{(\beta)} \Delta_{\xi}^{\alpha} \sigma(x,\xi) ) \widehat{ \mathcal{M}}(\xi)^{(\rho|\alpha|-\delta|\beta|-m)}\Vert_{\textnormal{op}} <\infty.
   \end{equation}
  \end{definition}
By following the usual terminology we  define:
\[ 
    \textnormal{Op}(\mathscr{S}^{m,\mathcal{L}}_{\rho,\delta}(G)):=\{A:C^{\infty}(G)\rightarrow \mathscr{D}'(G):\sigma_A\equiv\widehat{A}(x,\xi)\in \mathscr{S}^{m,\mathcal{L}}_{\rho,\delta}(G) \},
\] with
\[ 
    Af=\sum_{[\xi]\in \widehat{G}}d_\xi \textnormal{\textbf{Tr}}(\xi(\cdot)\widehat{A}(\cdot,\xi)\widehat{f}(\xi)),\,\,\,f\in C^\infty(G).  
\]
We also define
\[ 
    \Psi^{m,\mathcal{L}}_{\rho, \delta}(G\times \widehat{G})\equiv \textnormal{Op}({S}^{m,\mathcal{L}}_{\rho,\delta}(G\times \widehat{G})):=\textnormal{Op}(\mathscr{S}^{m\kappa,\mathcal{L}}_{\rho\kappa,\delta\kappa}(G)),\,\,\,0\leq \delta,\rho\leq 1.
\]

In order to illustrate the connections of the spectral properties of the sub-Laplacian $\mathcal{L}$ with these  classes, we show in the following example that the real powers of the operator $(1+\mathcal{L})^\frac{1}{2},$ have subelliptic symbols.

\begin{example}[Positive powers of $(1+\mathcal{L})^\frac{1}{2}$ in subelliptic H\"ormander classes]\label{powesfirstpart} Although it will be proved in Theorem \ref{orders:theorems} that for every $s\in \mathbb{R},$ 
\begin{equation}\label{realpowerssharp}
  \mathcal{M}_{s}:=  (1+\mathcal{L})^\frac{s}{2}\in \Psi^{s,\mathcal{L}}_{1,0}(G\times \widehat{G}),
\end{equation} let us illustrate quickly how to include positive powers of $(1+\mathcal{L})^\frac{1}{2}$ to the subelliptic classes just giving a short argument in order to deduce \eqref{1useful}. Certainly \eqref{realpowerssharp} is a better conclusion than  \eqref{1useful}. 
We will apply the inequality
\begin{equation}\label{GarettoRuzhanskyIneq}
  \langle \xi\rangle^{\frac{1}{\kappa}}\lesssim  (1+\nu_{ii}(\xi)^2)^{\frac{1}{2}}\lesssim \langle \xi\rangle,
\end{equation}  proved in Proposition 3.1 of \cite{GarettoRuzhansky2015}.  Let us observe that for $0<s<\infty,$ 
 \begin{equation}\label{1useful}
     \mathcal{M}_{s}=(1+\mathcal{L})^\frac{s}{2}\in \textnormal{Op}(\mathscr{S}^{s\kappa,\mathcal{L}}_{1,0}(G))\equiv\textnormal{Op}({S}^{s,\mathcal{L}}_{\frac{1}{\kappa},0}(G\times \widehat{G})).
 \end{equation} Indeed, we can prove \eqref{1useful} as follows. For $0<s<1,$
  if $m=m_s$ is the subelliptic H\"ormander  order of $m_{s},$ according to \eqref{InI}, if we set  $|\alpha|=0,$ then  we have that
 $$
 \sup_{[\xi]\in G }\Vert \widehat{ \mathcal{M}}(\xi)^{\frac{1}{\kappa}(-m_s)}\widehat{\mathcal{M}}_s(\xi)\Vert_{\textnormal{op}}=\sup_{1\leqslant i\leqslant d_\xi} (1+\nu_{ii}(\xi)^2)^{\frac{1}{2}(- \frac{m_s}{\kappa}+s)}<\infty,$$ if and only if $m_s\geqslant   \kappa s.$ This suggests that perhaps $m_s=\kappa s.$ Indeed, this is the case.  Because, for $0<s<1,$ $\mathcal{M}_s\in \textnormal{Op}(\mathscr{S}^{s}_{1,0}(G)), $ we have the estimates,
 \[ 
     \Vert \Delta_\xi^\alpha \widehat{\mathcal{M}}_{s}(\xi) \Vert_{\textnormal{op}}\leqslant C_\alpha \langle \xi\rangle^{s-|\alpha|}.
 \] So, in order to deduce that $\mathcal{M}_{s}\in \textnormal{Op}(\mathscr{S}^{s\kappa,\mathcal{L}}_{1,0}(G)),$
 it remains to show that
 \[ 
    \textnormal{I'}:=\sup_{[\xi]\in  \widehat{G} }\Vert \widehat{ \mathcal{M}}(\xi)^{\frac{1}{\kappa}(|\alpha|-s\kappa)} \Delta_{\xi}^{\alpha}\widehat{\mathcal{M}}_{s}(\xi)\Vert_{\textnormal{op}} <\infty,
\] and
\[ 
   \textnormal{II'}:= \sup_{[\xi]\in \widehat{G} }\Vert ( \Delta_{\xi}^{\alpha}\widehat{\mathcal{M}}_{s}(\xi)) \widehat{ \mathcal{M}}(\xi)^{\frac{1}{\kappa}(|\alpha|-s\kappa)} \Vert_{\textnormal{op}} <\infty,
\] for all $|\alpha|\geqslant   1.$ However, in this case $s-|\alpha|<1-|\alpha|\leqslant 0,$ and the estimate $\langle \xi\rangle^{s-|\alpha|}\leqslant (1+\nu_{ii}(\xi)^2)^{\frac{1}{2}(s-|\alpha|)},$ leads to 
\[ 
    \Vert  \Delta_{\xi}^{\alpha}\widehat{\mathcal{M}}_{s}(\xi)\Vert_{\textnormal{op}} \Vert \widehat{ \mathcal{M}}(\xi)^{\frac{1}{\kappa}(|\alpha|-s\kappa)} \Vert_{\textnormal{op}}\leqslant C_\alpha\sup_{1\leqslant i\leqslant d_\xi} (1+\nu_{ii}(\xi)^2)^{\frac{1}{2}(s-|\alpha|+\frac{|\alpha|}{\kappa}-s)}\leqslant C_\alpha.
\] This analysis proves that $\textnormal{I'},\textnormal{II'}\leqslant C_\alpha.$ In the general case $s\geqslant   1,$ we also have $\mathcal{M}_{s}\in \textnormal{Op}(\mathscr{S}^{s\kappa,\mathcal{L}}_{1,0}(G)). $ This will be proved in Remark \ref{sgeq1} as a consequence of the subelliptic symbolic calculus. 
\end{example}

\begin{remark}
{ One of the reasons to introduce contracted classes  $S^{m,\mathcal{L}}_{\rho,\delta}(G\times \widehat{G})$ is that when studying the qualitative properties of the subelliptic Fourier multipliers, one uses the inequality \eqref{GarettoRuzhanskyIneq} and the powers of the operator $\mathcal{M}^{\frac{1}{\kappa}}$ appear. These contracted classes will be useful to establish the subelliptic $L^p$-multipliers in Section \ref{ps}. Furthermore, the contracted classes allow for the inclusion \eqref{realpowerssharp} which is what one expects in terms of the orders of operators.  }
\end{remark}

For symbols of Fourier multipliers we will use the following notation,
\[ 
     \mathscr{S}^{m,\mathcal{L}}_{\rho}(\widehat{G})=\{\sigma\in  \Sigma(\widehat{G}):\sigma\in \mathscr{S}^{m,\mathcal{L}}_{\rho,0}(G)\}.
 \]
 
  We also define the class of symbols  of order $m,$ satisfying the $\rho$-type conditions up to order $\ell\in\mathbb{N},$  $ \mathscr{S}^{m,\ell,\mathcal{L}}_{\rho}(\widehat{G}),$ by those symbols satisfying,
\[  \sup_{[\xi]\in \widehat{G} }\Vert \widehat{ \mathcal{M}}(\xi)^{\frac{1}{\kappa}(\rho|\alpha|-m)} \Delta_{\xi}^{\alpha}a(\xi)\Vert_{\textnormal{op}} <\infty,\,\,|\alpha|\leqslant \ell,\]
and 
\[  \sup_{[\xi]\in \widehat{G} }\Vert  \Delta_{\xi}^{\alpha}a(\xi)\widehat{ \mathcal{M}}(\xi)^{\frac{1}{\kappa}(\rho|\alpha|-m)}\Vert_{\textnormal{op}} <\infty,\,\,|\alpha|\leqslant \ell.
\]

In a similar way,  we also define the class of non-invariant symbols  with order $m,$ satisfying the $\rho$-type conditions up to order $\ell\in\mathbb{N},$   and $\delta$-type conditions up to order $\ell'$, $ \mathscr{S}^{m,\ell,\ell',\mathcal{L}}_{\rho,\delta}(G),$ by those symbols satisfying,
\[ \sup_{[\xi]\in \widehat{G} }\Vert \widehat{ \mathcal{M}}(\xi)^{\frac{1}{\kappa}(\rho|\alpha|-\delta|\beta|-m)} \partial_{X}^{(\beta)}\Delta_{\xi}^{\alpha}a(x,\xi)\Vert_{\textnormal{op}} <\infty,\,\,|\alpha|\leqslant \ell,\,|\beta|\leqslant \ell',\]
and 
\[  \sup_{[\xi]\in \widehat{G} }\Vert  \partial_{X}^{(\beta)}  \Delta_{\xi}^{\alpha}a(x,\xi)\widehat{ \mathcal{M}}(\xi)^{\frac{1}{\kappa}(\rho|\alpha|-\delta|\beta|-m)}\Vert_{\textnormal{op}} <\infty,\,\,|\alpha|\leqslant \ell,\,|\beta|\leqslant \ell'.\]

\begin{lemma}\label{lemadecaying1}
Let $G$ be a compact Lie group and  let $0\leqslant \delta,\rho\leqslant \kappa.$ If $a\in \mathscr{S}^{m,\mathcal{L}}_{\rho,\delta}(G)$ then for every $\alpha,\beta\in \mathbb{N}_0^n,$ there exists $C_{\alpha,\beta}>0$ satisfying the estimates
\[ 
    \Vert \partial_{X}^{(\beta)} \Delta_{\xi}^{\alpha}a(x,\xi)\Vert_{\textnormal{op}}\leqslant C_{\alpha,\beta}\sup_{1\leqslant i\leqslant d_\xi}(1+\nu_{ii}(\xi))^{\frac{m-\rho|\alpha|+\delta|\beta|}{\kappa }},
\]uniformly in $(x,[\xi])\in G\times \widehat{G}.$ 
\end{lemma}
\begin{proof}
Let us assume that  $a\in \mathscr{S}^{m,\mathcal{L}}_{\rho,\delta}(G).$ Then for every $\alpha,\beta\in \mathbb{N}_0^n,$ we have 
$$\Vert \widehat{ \mathcal{M}}(\xi)^{\frac{1}{\kappa}(\rho|\alpha|-\delta|\beta|-m)}\partial_{X}^{(\beta)} \Delta_{\xi}^{\alpha}a(x,\xi)\Vert_{\textnormal{op}} \leqslant C_{\alpha,\beta}<\infty.$$
On the other hand, let us estimate
\begin{align*}
  \Vert \partial_{X}^{(\beta)} \Delta_{\xi}^{\alpha}a(x,\xi)\Vert_{\textnormal{op}} &=   \Vert \widehat{ \mathcal{M}}(\xi)^{-\frac{1}{\kappa}(\rho|\alpha|-\delta|\beta|-m)}     \widehat{ \mathcal{M}}(\xi)^{\frac{1}{\kappa}(\rho|\alpha|-\delta|\beta|-m)} \partial_{X}^{(\beta)} \Delta_{\xi}^{\alpha}a(x,\xi)\Vert_{\textnormal{op}}\\
  &\leqslant \Vert \widehat{ \mathcal{M}}(\xi)^{-\frac{1}{\kappa}(\rho|\alpha|-\delta|\beta|-m)}  \Vert_{\textnormal{op}} \Vert   \widehat{ \mathcal{M}}(\xi)^{\frac{1}{\kappa}(\rho|\alpha|-\delta|\beta|-m)} \partial_{X}^{(\beta)} \Delta_{\xi}^{\alpha}a(x,\xi)\Vert_{\textnormal{op}}\\
  &\leqslant  C_{\alpha,\beta}  \Vert   \widehat{ \mathcal{M}}(\xi)^{-\frac{1}{\kappa}(\rho|\alpha|-\delta|\beta|-m)} \Vert_{\textnormal{op}}.
\end{align*}Finally, there exists $i\in \{1,2,\cdots, d_\xi\}$ depending on $[\xi],$ such that  
\begin{align*}
    \Vert   \widehat{ \mathcal{M}}(\xi)^{-\frac{1}{\kappa}(\rho|\alpha|-\delta|\beta|-m)} \Vert_{\textnormal{op}} &=(1+\nu_{ii}(\xi)^2)^{\frac{m-\rho|\alpha|+\delta|\beta|}{2\kappa }}\leqslant \sup_{1\leqslant i\leqslant d_\xi}(1+\nu_{ii}(\xi)^2)^{\frac{m-\rho|\alpha|+\delta|\beta|}{2\kappa }}\\
    &\asymp \sup_{1\leqslant i\leqslant d_\xi}(1+\nu_{ii}(\xi))^{\frac{m-\rho|\alpha|+\delta|\beta|}{\kappa }}. 
\end{align*}
 The proof is complete.
\end{proof}
 
The following Lemma is straightforward.
\begin{lemma}\label{Lemma47}
Let $G$ be a compact Lie group and  $0\leqslant \delta,\rho\leqslant 1.$ Let us consider a sub-Laplacian $\mathcal{L}=-(X_1^2+\cdots +X_k^2)$ on $G,$ where the system of vector fields $X=\{X_i\}_{i=1}^{k}$ satisfies the H\"ormander condition of step $\kappa$.
Let us consider $\sigma\in {S}^{m,\mathcal{L}}_{\rho,\delta}(G\times \widehat{G})$ and let  $k_{x}$ be its associated  right-convolution kernel. Then,
\begin{enumerate}
    \item if $\Delta_q\in \textnormal{diff}^{\ell}(\widehat{G}),$ then the right convolution kernel associated with the symbol  $\partial_{X}^{(\beta)} \Delta_q\sigma(x,\xi)\in {S}^{m-\rho \ell+\delta|\beta|,\mathcal{L}}_{\rho,\delta}(G\times \widehat{G}) $ is given by $q(\cdot)\partial_{X}^{(\beta)}k_{x}.$  
    \item If $0\leqslant \delta\leqslant \delta'\leqslant 1$ and  $0\leqslant \rho\leqslant \rho'\leqslant 1,$ then 
    \[ 
       {S}^{m,\mathcal{L}}_{\rho',\delta}(G\times \widehat{G})\subset  {S}^{m,\mathcal{L}}_{\rho,\delta}(G\times \widehat{G})\subset {S}^{m,\mathcal{L}}_{\rho,\delta'}(G\times \widehat{G}), 
    \]with continuous inclusions. In particular, 
    \[ 
        {S}^{m,\mathcal{L}}_{1,0}(G\times \widehat{G})\subset  {S}^{m,\mathcal{L}}_{\rho,\delta}(G\times \widehat{G})\subset {S}^{m,\mathcal{L}}_{0,1}(G\times \widehat{G}).
    \]
\end{enumerate}
\end{lemma}

In the following proposition we compare some subelliptic classes with (elliptic) classes on compact Lie groups (we will use \eqref{C1} and \eqref{C3} for the corresponding classes of matrix symbols of limited regularity).
\begin{proposition}\label{lemitapequeño}  Let $G$ be a compact Lie group and  $0\leqslant \delta,\rho\leqslant 1.$ Let us consider a sub-Laplacian $\mathcal{L}=-(X_1^2+\cdots +X_k^2)$ on $G,$ where the system of vector fields $X=\{X_i\}_{i=1}^{k}$ satisfies the H\"ormander condition of step $\kappa$.
Let us assume that $\nu\geqslant   0.$  Then, for every $\ell\in \mathbb{N}$ and all $0\leqslant \rho\leqslant 1, $ we have
\[ 
      \mathscr{S}^{-\nu\kappa,\ell,\mathcal{L}}_{\rho \kappa,0}(G)\subset \mathscr{S}^{-\frac{\nu}{\kappa},\ell}_{\frac{\rho}{\kappa},0 }(G\times \widehat{G}),\,\,\mathscr{S}^{-\nu,\ell}_{\rho,0 }(G\times  \widehat{G})\subset \mathscr{S}^{-\nu\kappa,\ell,\mathcal{L}}_{{\rho}\kappa,0 }(G)
  \]  with continuous inclusions.
\end{proposition}
\begin{proof}
 Because we are considering symbols with $\delta=0,$ it is sufficient to prove  Proposition \ref{lemitapequeño} for classes of invariant symbols. Let us assume that $a\in  \mathscr{S}^{-\nu,\ell,\mathcal{L}}_{\rho }(\widehat{G})$ where $0\leqslant \rho\leqslant \kappa.$ Then we have 
\[  \sup_{[\xi]\in \widehat{G} }\Vert  \Delta_{\xi}^{\alpha}a(\xi)\widehat{ \mathcal{M}}(\xi)^{\frac{1}{\kappa}(\rho|\alpha|+\nu)}\Vert_{\textnormal{op}},\sup_{[\xi]\in \widehat{G} }\Vert \widehat{ \mathcal{M}}(\xi)^{\frac{1}{\kappa}(\rho|\alpha|+\nu)} \Delta_{\xi}^{\alpha}a(\xi)\Vert_{\textnormal{op}} <\infty,\,\,|\alpha|\leqslant \ell.\]
Now, let us note that
\begin{align*}
    \Vert \Delta_\xi^\alpha a(\xi)\Vert_{\textnormal{op}}& =\Vert \widehat{ \mathcal{M}}(\xi)^{-\frac{1}{\kappa}(\rho|\alpha|+\nu)}\widehat{ \mathcal{M}}(\xi)^{\frac{1}{\kappa}(\rho|\alpha|+\nu)} \Delta_\xi^\alpha a(\xi)\Vert_{\textnormal{op}}\\
    &\leqslant \Vert \widehat{ \mathcal{M}}(\xi)^{-\frac{1}{\kappa}(\rho|\alpha|+\nu)}\Vert_{\textnormal{op}}\Vert \widehat{ \mathcal{M}}(\xi)^{\frac{1}{\kappa}(\rho|\alpha|+\nu)} \Delta_\xi^\alpha a(\xi)\Vert_{\textnormal{op}}\\
     &=\sup_{1\leqslant i\leqslant d_{\xi}}(1+\nu_{ii}(\xi)^2)^{-\frac{1}{2\kappa}(\rho|\alpha|+\nu)}\Vert \widehat{ \mathcal{M}}(\xi)^{\frac{1}{\kappa}(\rho|\alpha|+\nu)} \Delta_\xi^\alpha a(\xi)\Vert_{\textnormal{op}}\\
     \\
     &\leqslant \langle \xi\rangle^{-\frac{1}{\kappa^2}(\rho|\alpha|+\nu)}   \sup_{[\xi]\in \widehat{G}}\Vert \widehat{ \mathcal{M}}(\xi)^{\frac{1}{\kappa}(\rho|\alpha|+\nu)} \Delta_\xi^\alpha a(\xi)\Vert_{\textnormal{op}},
\end{align*}
where in the last line  we have used \eqref{GarettoRuzhanskyIneq}. So, we have proved that
\begin{equation}\label{444}
     \mathscr{S}^{-\nu,\ell,\mathcal{L}}_{\rho }(\widehat{G})\subset \mathscr{S}^{-\frac{\nu}{\kappa^2},\ell}_{\frac{\rho}{\kappa^2} }(\widehat{G}),\,\,0\leqslant \rho\leqslant \kappa.
\end{equation}So, if we replace $\rho$ in \eqref{444} by $\rho\kappa,$ where $0\leqslant \rho\leqslant 1,$ we obtain the inclusion,
 \[ 
     \mathscr{S}^{-\nu\kappa,\ell,\mathcal{L}}_{\rho\kappa}(\widehat{G})\subset \mathscr{S}^{-\frac{\nu}{\kappa},\ell}_{\frac{\rho}{\kappa} }(\widehat{G}),\,\,\,0\leqslant \rho\leqslant 1.
 \]The inclusion $  \mathscr{S}^{-\nu,\ell}_{\rho }(\widehat{G})\subset \mathscr{S}^{-\nu\kappa,\ell,\mathcal{L}}_{{\rho}\kappa }(\widehat{G}) $ for $0\leqslant \rho\leqslant 1,$ can be proved in an analogous way. Indeed, if $a\in \mathscr{S}^{-\nu,\ell}_{\rho }(\widehat{G}), $ we have
 \begin{align*}
     \Vert \widehat{ \mathcal{M}}(\xi)^{\frac{1}{\kappa}(\kappa{\rho}|\alpha|+\kappa\nu)}  \Delta_\xi^\alpha a(\xi)\Vert_{\textnormal{op}} & =\Vert \widehat{ \mathcal{M}}(\xi)^{\frac{1}{\kappa}(\kappa{\rho}|\alpha|+\kappa\nu)}\langle \xi\rangle^{-\rho|\alpha|-\nu}\langle \xi\rangle^{\rho|\alpha|+\nu}  \Delta_\xi^\alpha a(\xi)\Vert_{\textnormal{op}}\\
     &\leqslant \Vert \widehat{ \mathcal{M}}(\xi)^{\frac{1}{\kappa}(\kappa{\rho}|\alpha|+\kappa\nu)}\langle \xi\rangle^{-\rho|\alpha|-\nu} \Vert_{\textnormal{op}}  \Vert \langle \xi\rangle^{\rho|\alpha|+\nu}  \Delta_\xi^\alpha a(\xi)\Vert_{\textnormal{op}}.
 \end{align*} Because $a\in \mathscr{S}^{-\nu,\ell}_{\rho }(\widehat{G}), $
 \[ 
     \sup_{[\xi]\in \widehat{G}}\Vert \langle \xi\rangle^{\rho|\alpha|+\nu}  \Delta_\xi^\alpha a(\xi)\Vert_{\textnormal{op}}<\infty,\,\,|\alpha|\leqslant \ell.
 \]So, we only need to check that $\Vert \widehat{ \mathcal{M}}(\xi)^{\frac{1}{\kappa}(\kappa{\rho}|\alpha|+\kappa\nu)}\langle \xi\rangle^{-\rho|\alpha|-\nu} \Vert_{\textnormal{op}}$ is uniformly bounded in $[\xi]\in \widehat{G}.$ For this, we can estimate
 \begin{align*}
     \Vert \widehat{ \mathcal{M}}(\xi)^{\frac{1}{\kappa}(\kappa{\rho}|\alpha|+\kappa\nu)}\langle \xi\rangle^{-\rho|\alpha|-\nu} \Vert_{\textnormal{op}}=\sup_{1\leqslant i\leqslant d_\xi}(1+\nu_{ii}(\xi)^2)^{\frac{1}{2}({\rho|\alpha|}+\nu)}\langle \xi\rangle^{-\rho|\alpha|-\nu}\leqslant 1,
 \end{align*} because $(1+\nu_{ii}(\xi)^2)^{\frac{1}{2}({\rho|\alpha|}+\nu)}\lesssim \langle \xi \rangle^{\rho|\alpha|+\nu}$ in view of \eqref{GarettoRuzhanskyIneq}. So, we finish the proof.
\end{proof}
\begin{corollary}\label{lemitapequeñocontracted}
Let us assume that $\nu\geqslant   0.$  Then, for every $\ell\in \mathbb{N}$ and all $0\leqslant \rho\leqslant 1, $ we have
\[ 
      {S}^{-\nu,\ell,\mathcal{L}}_{\rho ,0}(G\times \widehat{G})\subset \mathscr{S}^{-\frac{\nu}{\kappa},\ell}_{\frac{\rho}{\kappa},0 }(G\times \widehat{G}),\,\,\mathscr{S}^{-\nu,\ell}_{\rho,0 }(G\times \widehat{G})\subset {S}^{-\nu,\ell,\mathcal{L}}_{{\rho},0 }(G\times \widehat{G})
  \]  with continuous inclusions.
\end{corollary}

Now, we will prove some useful identities in order to characterise  the subelliptic H\"ormander classes, by showing that   \eqref{InI} and \eqref{InII} are equivalent in some sense. To do so, we will use the following version of the Corach-Porta-Recht inequality (see  Corach, Porta, and Recht \cite{CorachPortaRecht90} and Seddik \cite[Theorem 2.3]{Seddik} for \eqref{Recht1} and Andruchow, Corach, and  Stojanoff \cite[page 297]{Andruchow} for \eqref{Recht2}).

\begin{proposition}
Let $H$ be a complex Hilbert space and let $A,P,Q,X\in \mathscr{B}(H)$ be bounded operators on $H.$ Let us assume that $P$ and $Q$ are positive and invertible operators with $PQ=QP,$ and that $A$ is self-adjoint. Then we have the norm inequalities
\begin{equation}\label{Recht1}
    2\Vert X\Vert_{\textnormal{op}}\leqslant \max\{ \Vert PXP^{-1}+Q^{-1}XQ\Vert_{\textnormal{op}}, \Vert PX^*P^{-1}+Q^{-1}X^*Q\Vert_{\textnormal{op}}  \}, 
\end{equation}and 
\begin{equation}\label{Recht2}
    \Vert X \Vert_{\textnormal{op}}\leqslant \Vert AXA+(1+A^2)^{\frac{1}{2}}X(1+A^2)^{\frac{1}{2}} \Vert_{\textnormal{op}}.
\end{equation}
\end{proposition}
\begin{remark}
It was proved also in Andruchow, Corach, and  Stojanoff \cite[page 302]{Andruchow} that the inequalities \eqref{Recht1} and \eqref{Recht2} are equivalent and $2$ is the best constant in  \eqref{Recht1} if $P=Q$.
\end{remark}
So, we are ready to prove the following characterization of dilated subelliptic H\"ormander classes.
\begin{theorem}\label{cor}   Let $G$ be a compact Lie group and let  $0\leqslant \delta,\rho\leqslant \kappa.$    The following conditions are equivalent.
\begin{itemize}
    \item[A.] For every $\alpha,\beta\in \mathbb{N}_0^n,$ \begin{equation}\label{InI2}
      p_{\alpha,\beta,\rho,\delta,m,\textnormal{left}}(a):= \sup_{(x,[\xi])\in G\times \widehat{G} }\Vert \widehat{ \mathcal{M}}(\xi)^{\frac{1}{\kappa}(\rho|\alpha|-\delta|\beta|-m)}\partial_{X}^{(\beta)} \Delta_{\xi}^{\alpha}a(x,\xi)\Vert_{\textnormal{op}} <\infty.
   \end{equation}
   \item[B.]For every $\alpha,\beta\in \mathbb{N}_0^n,$ \begin{equation}\label{InII2}
      p_{\alpha,\beta,\rho,\delta,m,\textnormal{right}}(a):= \sup_{(x,[\xi])\in G\times \widehat{G} }\Vert (\partial_{X}^{(\beta)} \Delta_{\xi}^{\alpha} a(x,\xi) ) \widehat{ \mathcal{M}}(\xi)^{\frac{1}{\kappa}(\rho|\alpha|-\delta|\beta|-m)}\Vert_{\textnormal{op}} <\infty.
   \end{equation}
   \item[C.] For all $r\in \mathbb{R},$ $\alpha,\beta\in \mathbb{N}_0^n,$
    \begin{equation}\label{InI2X}
      p_{\alpha,\beta,\rho,\delta,m,r}(a):= \sup_{(x,[\xi])\in G\times \widehat{G} }\Vert \widehat{ \mathcal{M}}(\xi)^{\frac{1}{\kappa}(\rho|\alpha|-\delta|\beta|-m-r)}\partial_{X}^{(\beta)} \Delta_{\xi}^{\alpha}a(x,\xi)\widehat{ \mathcal{M}}(\xi)^{\frac{r}{\kappa}}\Vert_{\textnormal{op}} <\infty.
   \end{equation}
   \item[D.] There exists $r_0\in \mathbb{R},$ such that for every $\alpha,\beta\in \mathbb{N}_0^n,$
    \begin{equation}\label{InI2X''}
      p_{\alpha,\beta,\rho,\delta,m,r_0}(a):= \sup_{(x,[\xi])\in G\times \widehat{G} }\Vert \widehat{ \mathcal{M}}(\xi)^{\frac{1}{\kappa}(\rho|\alpha|-\delta|\beta|-m-r_0)}\partial_{X}^{(\beta)} \Delta_{\xi}^{\alpha}a(x,\xi)\widehat{ \mathcal{M}}(\xi)^{\frac{r_0}{\kappa}}\Vert_{\textnormal{op}} <\infty.
   \end{equation}
   \item[E.] $a\in \mathscr{S}^{m,\mathcal{L}}_{\rho,\delta}(G).$
\end{itemize}
\begin{proof}

We only need to prove that 
$\textnormal{D}\Longrightarrow \textnormal{C}.$ Let us assume that \eqref{InI2X''} holds true for some $r_0\in \mathbb{R}$ and let $r\in \mathbb{R}$ be a real number. Let us assume first that $r>r_0.$ Let us note that the operator $A(x,\xi):=\widehat{ \mathcal{M}}(\xi)^{\frac{1}{\kappa}(r_0-r)}$ is self-adjoint. Let us denote
\[ 
    X_{\alpha,\beta,r_0}(x,\xi)=  \widehat{ \mathcal{M}}(\xi)^{\frac{1}{\kappa}(\rho|\alpha|-\delta|\beta|-m-r_0)}\partial_{X}^{(\beta)} \Delta_{\xi}^{\alpha}a(x,\xi)\widehat{ \mathcal{M}}(\xi)^{\frac{r_0}{\kappa}}.
\]

From the Corach-Porta-Recht inequality \eqref{Recht2}, we have 
\begin{align*}
    &\Vert \widehat{ \mathcal{M}}(\xi)^{\frac{1}{\kappa}(\rho|\alpha|-\delta|\beta|-m-r)}\partial_{X}^{(\beta)} \Delta_{\xi}^{\alpha}a(x,\xi)\widehat{ \mathcal{M}}(\xi)^{\frac{r}{\kappa}}\Vert_{\textnormal{op}}\\
    &=\Vert \widehat{ \mathcal{M}}(\xi)^{\frac{1}{\kappa}(r_0-r)} \widehat{ \mathcal{M}}(\xi)^{\frac{1}{\kappa}(\rho|\alpha|-\delta|\beta|-m-r_0)}\partial_{X}^{(\beta)} \Delta_{\xi}^{\alpha}a(x,\xi)\widehat{ \mathcal{M}}(\xi)^{\frac{r_0}{\kappa}}\widehat{ \mathcal{M}}(\xi)^{\frac{1}{\kappa}(r-r_0)}\Vert_{\textnormal{op}} \\
    &=\Vert \widehat{ \mathcal{M}}(\xi)^{\frac{1}{\kappa}(r_0-r)} X_{\alpha,\beta,r_0}(x,\xi)\widehat{ \mathcal{M}}(\xi)^{\frac{1}{\kappa}(r-r_0)}\Vert_{\textnormal{op}}\\
    &\leqslant \Vert A(x,\xi) \widehat{ \mathcal{M}}(\xi)^{\frac{1}{\kappa}(r_0-r)}  X_{\alpha,\beta,r_0}(x,\xi)\widehat{ \mathcal{M}}(\xi)^{\frac{1}{\kappa}(r-r_0)}A(x,\xi)\\
    &+(1+A(x,\xi)^{2})^{\frac{1}{2}}\widehat{ \mathcal{M}}(\xi)^{\frac{1}{\kappa}(r_0-r)} X_{\alpha,\beta,r_0}(x,\xi)\widehat{ \mathcal{M}}(\xi)^{\frac{1}{\kappa}(r-r_0)} (1+A(x,\xi)^{2})^{\frac{1}{2}} \Vert_{\textnormal{op}}\\
    &= \Vert \widehat{ \mathcal{M}}(\xi)^{\frac{2}{\kappa}(r_0-r)}  X_{\alpha,\beta,r_0}(x,\xi)\\
   &+(1+\widehat{ \mathcal{M}}(\xi)^{\frac{2}{\kappa}(r_0-r)})^{\frac{1}{2}}\widehat{ \mathcal{M}}(\xi)^{\frac{1}{\kappa}(r_0-r)} X_{\alpha,\beta,r_0}(x,\xi)\widehat{ \mathcal{M}}(\xi)^{\frac{1}{\kappa}(r-r_0)} (1+\widehat{ \mathcal{M}}(\xi)^{\frac{2}{\kappa}(r_0-r)})^{\frac{1}{2}} \Vert_{\textnormal{op}}.
    \end{align*} 
Taking into account that $r_0-r<0,$ and the following facts (see  \cite[Remark 5.9]{CardonaRuzhansky2019I})
\[ 
  \widehat{ \mathcal{M}}(\xi)^{\frac{2}{\kappa}(r_0-r)}\in \mathscr{S}^{\frac{2(r_0-r)}{\kappa^{2}}}_{\frac{1}{\kappa},0}(G\times\widehat{G}),\,\,\, \,(1+\widehat{ \mathcal{M}}(\xi)^{\frac{2}{\kappa}(r_0-r)})^{\frac{1}{2}} \widehat{ \mathcal{M}}(\xi)^{\frac{1}{\kappa}(r_0-r)}\in \mathscr{S}^{\frac{2(r_0-r)}{\kappa^{2}}}_{\frac{1}{\kappa},0}(G\times\widehat{G}) ,
\]  
\[ 
    \widehat{ \mathcal{M}}(\xi)^{\frac{1}{\kappa}(r-r_0)} (1+\widehat{ \mathcal{M}}(\xi)^{\frac{2}{\kappa}(r_0-r)})^{\frac{1}{2}} \in \mathscr{S}^{0}_{\frac{1}{\kappa},0}(G\times\widehat{G}),
\]
we deduce that
\[ 
  \mathscr{D}_1:=  \sup_{[\xi]\in \widehat{G}} \Vert  \widehat{ \mathcal{M}}(\xi)^{\frac{2}{\kappa}(r_0-r)}\Vert_{\textnormal{op}}, \, \mathscr{D}_2:=  \sup_{[\xi]\in \widehat{G}} \Vert  (1+\widehat{ \mathcal{M}}(\xi)^{\frac{2}{\kappa}(r_0-r)})^{\frac{1}{2}} \widehat{ \mathcal{M}}(\xi)^{\frac{2}{\kappa}(r_0-r)}  \Vert_{\textnormal{op}}  <\infty,
\] and 
\[ 
  \mathscr{D}_3:=  \sup_{[\xi]\in \widehat{G}}\Vert \widehat{ \mathcal{M}}(\xi)^{\frac{1}{\kappa}(r-r_0)} (1+\widehat{ \mathcal{M}}(\xi)^{\frac{2}{\kappa}(r_0-r)})^{\frac{1}{2}}\Vert_{\textnormal{op}}<\infty.
\]
Consequently, 
\begin{align*}
    &\Vert \widehat{ \mathcal{M}}(\xi)^{\frac{1}{\kappa}(\rho|\alpha|-\delta|\beta|-m-r)}\partial_{X}^{(\beta)} \Delta_{\xi}^{\alpha}a(x,\xi)\widehat{ \mathcal{M}}(\xi)^{\frac{r}{\kappa}}\Vert_{\textnormal{op}}\\
    &\leqslant  \Vert \widehat{ \mathcal{M}}(\xi)^{\frac{2}{\kappa}(r_0-r)}  X_{\alpha,\beta,r_0}(x,\xi)\\
   &+(1+\widehat{ \mathcal{M}}(\xi)^{\frac{2}{\kappa}(r_0-r)})^{\frac{1}{2}}\widehat{ \mathcal{M}}(\xi)^{\frac{1}{\kappa}(r_0-r)} X_{\alpha,\beta,r_0}(x,\xi)\widehat{ \mathcal{M}}(\xi)^{\frac{1}{\kappa}(r-r_0)} (1+\widehat{ \mathcal{M}}(\xi)^{\frac{2}{\kappa}(r_0-r)})^{\frac{1}{2}} \Vert_{\textnormal{op}}\\
   &\leqslant  \Vert \widehat{ \mathcal{M}}(\xi)^{\frac{2}{\kappa}(r_0-r)}\Vert_{\textnormal{op}} \Vert  X_{\alpha,\beta,r_0}(x,\xi)\Vert_{\textnormal{op}}\\
   &+\Vert (1+\widehat{ \mathcal{M}}(\xi)^{\frac{2}{\kappa}(r_0-r)})^{\frac{1}{2}}\widehat{ \mathcal{M}}(\xi)^{\frac{1}{\kappa}(r_0-r)}\Vert_{\textnormal{op}}\Vert X_{\alpha,\beta,r_0}(x,\xi)\Vert_{\textnormal{op}}\\
   &\hspace{7cm}\times\Vert \widehat{ \mathcal{M}}(\xi)^{\frac{1}{\kappa}(r-r_0)} (1+\widehat{ \mathcal{M}}(\xi)^{\frac{2}{\kappa}(r_0-r)})^{\frac{1}{2}} \Vert_{\textnormal{op}}\\
   &\leqslant (\mathscr{D}_1+\mathscr{D}_2\times \mathscr{D}_3)\Vert X_{\alpha,\beta,r_0}(x,\xi)\Vert_{\textnormal{op}}.
\end{align*}The previous argument shows that $\textnormal{D}\Longrightarrow \textnormal{C}$ for $r>r_0.$ In the  case where $r<r_0,$ we can define $A(x,\xi)=\widehat{ \mathcal{M}}(\xi)^{\frac{1}{\kappa}(r-r_0)}.$ By repeating the argument above we can  deduce that $\textnormal{D}\Longrightarrow \textnormal{C}$ for $r<r_0.$ Indeed, by using again the Corach-Porta-Recht inequality \eqref{Recht2}, we have 
\begin{align*}
    &\Vert \widehat{ \mathcal{M}}(\xi)^{\frac{1}{\kappa}(\rho|\alpha|-\delta|\beta|-m-r)}\partial_{X}^{(\beta)} \Delta_{\xi}^{\alpha}a(x,\xi)\widehat{ \mathcal{M}}(\xi)^{\frac{r}{\kappa}}\Vert_{\textnormal{op}}\\
    &=\Vert \widehat{ \mathcal{M}}(\xi)^{\frac{1}{\kappa}(r_0-r)} \widehat{ \mathcal{M}}(\xi)^{\frac{1}{\kappa}(\rho|\alpha|-\delta|\beta|-m-r_0)}\partial_{X}^{(\beta)} \Delta_{\xi}^{\alpha}a(x,\xi)\widehat{ \mathcal{M}}(\xi)^{\frac{r_0}{\kappa}}\widehat{ \mathcal{M}}(\xi)^{\frac{1}{\kappa}(r-r_0)}\Vert_{\textnormal{op}} \\
    &=\Vert \widehat{ \mathcal{M}}(\xi)^{\frac{1}{\kappa}(r_0-r)} X_{\alpha,\beta,r_0}(x,\xi)\widehat{ \mathcal{M}}(\xi)^{\frac{1}{\kappa}(r-r_0)}\Vert_{\textnormal{op}}\\
    &\leqslant \Vert A(x,\xi) \widehat{ \mathcal{M}}(\xi)^{\frac{1}{\kappa}(r_0-r)}  X_{\alpha,\beta,r_0}(x,\xi)\widehat{ \mathcal{M}}(\xi)^{\frac{1}{\kappa}(r-r_0)}A(x,\xi)\\
    &+(1+A(x,\xi)^{2})^{\frac{1}{2}}\widehat{ \mathcal{M}}(\xi)^{\frac{1}{\kappa}(r_0-r)} X_{\alpha,\beta,r_0}(x,\xi)\widehat{ \mathcal{M}}(\xi)^{\frac{1}{\kappa}(r-r_0)} (1+A(x,\xi)^{2})^{\frac{1}{2}} \Vert_{\textnormal{op}}\\
    &= \Vert   X_{\alpha,\beta,r_0}(x,\xi)\widehat{ \mathcal{M}}(\xi)^{\frac{2}{\kappa}(r-r_0)}\\
   &+(1+\widehat{ \mathcal{M}}(\xi)^{\frac{2}{\kappa}(r-r_0)})^{\frac{1}{2}}\widehat{ \mathcal{M}}(\xi)^{\frac{1}{\kappa}(r_0-r)} X_{\alpha,\beta,r_0}(x,\xi)\widehat{ \mathcal{M}}(\xi)^{\frac{1}{\kappa}(r-r_0)} (1+\widehat{ \mathcal{M}}(\xi)^{\frac{2}{\kappa}(r-r_0)})^{\frac{1}{2}} \Vert_{\textnormal{op}}.
    \end{align*} 
Since $r-r_0$ is negative, we have the following facts (see  \cite[Remark 5.9]{CardonaRuzhansky2019I})
\[ 
  \widehat{ \mathcal{M}}(\xi)^{\frac{2}{\kappa}(r-r_0)}\in \mathscr{S}^{\frac{2}{\kappa^2}(r-r_0)}_{\frac{1}{\kappa},0}(G\times\widehat{G}), \,  \widehat{ \mathcal{M}}(\xi)^{\frac{1}{\kappa}(r-r_0)} (1+\widehat{ \mathcal{M}}(\xi)^{\frac{2}{\kappa}(r-r_0)})^{\frac{1}{2}}        \in \mathscr{S}^{\frac{2(r-r_0)}{\kappa^{2}}}_{\frac{1}{\kappa},0}(G\times\widehat{G}) ,
\]  
\[ 
  (1+\widehat{ \mathcal{M}}(\xi)^{\frac{2}{\kappa}(r-r_0)})^{\frac{1}{2}} \widehat{ \mathcal{M}}(\xi)^{\frac{1}{\kappa}(r_0-r)}   \in \mathscr{S}^{0}_{\frac{1}{\kappa},0}(G\times\widehat{G}),
\]
we deduce that
\[ 
  \mathscr{D}'_1:=  \sup_{[\xi]\in \widehat{G}} \Vert   \widehat{ \mathcal{M}}(\xi)^{\frac{2}{\kappa}(r-r_0)}  \Vert_{\textnormal{op}}, \, \mathscr{D}_2':=  \sup_{[\xi]\in \widehat{G}} \Vert  (1+\widehat{ \mathcal{M}}(\xi)^{\frac{2}{\kappa}(r-r_0)})^{\frac{1}{2}} \widehat{ \mathcal{M}}(\xi)^{\frac{1}{\kappa}(r_0-r)}    \Vert_{\textnormal{op}}  <\infty,
\] and 
\[ 
  \mathscr{D}_3':=  \sup_{[\xi]\in \widehat{G}}\Vert  \widehat{ \mathcal{M}}(\xi)^{\frac{1}{\kappa}(r-r_0)} (1+\widehat{ \mathcal{M}}(\xi)^{\frac{2}{\kappa}(r-r_0)})^{\frac{1}{2}}    \Vert_{\textnormal{op}}<\infty.
\]
Consequently, 
\begin{align*}
    &\Vert \widehat{ \mathcal{M}}(\xi)^{\frac{1}{\kappa}(\rho|\alpha|-\delta|\beta|-m-r)}\partial_{X}^{(\beta)} \Delta_{\xi}^{\alpha}a(x,\xi)\widehat{ \mathcal{M}}(\xi)^{\frac{r}{\kappa}}\Vert_{\textnormal{op}}\\
    &\leqslant  \Vert   X_{\alpha,\beta,r_0}(x,\xi) \widehat{ \mathcal{M}}(\xi)^{\frac{2}{\kappa}(r-r_0)}\\
   &+(1+\widehat{ \mathcal{M}}(\xi)^{\frac{2}{\kappa}(r-r_0)})^{\frac{1}{2}}\widehat{ \mathcal{M}}(\xi)^{\frac{1}{\kappa}(r_0-r)} X_{\alpha,\beta,r_0}(x,\xi)\widehat{ \mathcal{M}}(\xi)^{\frac{1}{\kappa}(r-r_0)} (1+\widehat{ \mathcal{M}}(\xi)^{\frac{2}{\kappa}(r-r_0)})^{\frac{1}{2}} \Vert_{\textnormal{op}}\\
   &\leqslant  \Vert \widehat{ \mathcal{M}}(\xi)^{\frac{2}{\kappa}(r-r_0)}\Vert_{\textnormal{op}} \Vert  X_{\alpha,\beta,r_0}(x,\xi)\Vert_{\textnormal{op}}\\
   &+\Vert (1+\widehat{ \mathcal{M}}(\xi)^{\frac{2}{\kappa}(r-r_0)})^{\frac{1}{2}}\widehat{ \mathcal{M}}(\xi)^{\frac{1}{\kappa}(r_0-r)}\Vert_{\textnormal{op}}\Vert X_{\alpha,\beta,r_0}(x,\xi)\Vert_{\textnormal{op}}\\
   &\hspace{7cm}\times\Vert \widehat{ \mathcal{M}}(\xi)^{\frac{1}{\kappa}(r-r_0)} (1+\widehat{ \mathcal{M}}(\xi)^{\frac{2}{\kappa}(r-r_0)})^{\frac{1}{2}} \Vert_{\textnormal{op}}\\
   &\leqslant (\mathscr{D}_1'+\mathscr{D}_2'\times \mathscr{D}_3')\Vert X_{\alpha,\beta,r_0}(x,\xi)\Vert_{\textnormal{op}}.
\end{align*}The previous argument shows that $\textnormal{D}\Longrightarrow \textnormal{C}$ for $r_0>r.$
The proof is complete.
\end{proof}
\end{theorem}
Theorem \ref{cor} implies the following characterization for the contracted subelliptic classes.

\begin{corollary}\label{corC}   Let $G$ be a compact Lie group and let  $0\leqslant \delta,\rho\leqslant 1.$   The following conditions are equivalent.
\begin{itemize}
    \item[A.] For every $\alpha,\beta\in \mathbb{N}_0^n,$ \begin{equation}\label{InI2C}
      p_{\alpha,\beta,\rho,\delta,m,\textnormal{left}}(a)':= \sup_{(x,[\xi])\in G\times \widehat{G} }\Vert \widehat{ \mathcal{M}}(\xi)^{(\rho|\alpha|-\delta|\beta|-m)}\partial_{X}^{(\beta)} \Delta_{\xi}^{\alpha}a(x,\xi)\Vert_{\textnormal{op}} <\infty.
   \end{equation}
   \item[B.] For every $\alpha,\beta\in \mathbb{N}_0^n,$ \begin{equation}\label{InII2C}
      p_{\alpha,\beta,\rho,\delta,m,\textnormal{right}}(a)':= \sup_{(x,[\xi])\in G\times \widehat{G} }\Vert (\partial_{X}^{(\beta)} \Delta_{\xi}^{\alpha} a(x,\xi) ) \widehat{ \mathcal{M}}(\xi)^{(\rho|\alpha|-\delta|\beta|-m)}\Vert_{\textnormal{op}} <\infty.
   \end{equation}
   \item[C.] For all $r\in \mathbb{R},\alpha,\beta\in \mathbb{N}_0^n,$
    \begin{equation}\label{InI2XC}
      p_{\alpha,\beta,\rho,\delta,m,r}(a)':= \sup_{(x,[\xi])\in G\times \widehat{G} }\Vert \widehat{ \mathcal{M}}(\xi)^{(\rho|\alpha|-\delta|\beta|-m-r)}\partial_{X}^{(\beta)} \Delta_{\xi}^{\alpha}a(x,\xi)\widehat{ \mathcal{M}}(\xi)^{r}\Vert_{\textnormal{op}} <\infty.
   \end{equation}
   \item[D.] There exists $r_0\in \mathbb{R},$ such that for every $\alpha,\beta\in \mathbb{N}_0^n,$
    \begin{equation}\label{InI2X''C}
      p_{\alpha,\beta,\rho,\delta,m,r_0}(a)':= \sup_{(x,[\xi])\in G\times \widehat{G} }\Vert \widehat{ \mathcal{M}}(\xi)^{(\rho|\alpha|-\delta|\beta|-m-r_0)}\partial_{X}^{(\beta)} \Delta_{\xi}^{\alpha}a(x,\xi)\widehat{ \mathcal{M}}(\xi)^{r_0}\Vert_{\textnormal{op}} <\infty.
   \end{equation}
   \item[E.]  $a\in {S}^{m,\mathcal{L}}_{\rho,\delta}(G\times \widehat{G}).$
\end{itemize}
\end{corollary}
\begin{remark}
We will prove an analogy of  Theorem \ref{cor} for arbitrary graded Lie groups    in Theorem
\ref{corgraded} extending Theorem 5.5.20 of \cite{FischerRuzhanskyBook}. For this we will use the same approach that in the proof of   Theorem \ref{cor}.
\end{remark}

To study, for example, the classification of negative powers of subelliptic operators we need to study the inversion of symbols in the (dilated) subelliptic H\"ormander classes. So, we have the following theorem.

\begin{theorem}\label{IesT} Let $m\in \mathbb{R},$ and let $0\leqslant \delta<\rho\leqslant \kappa.$  Let  $a=a(x,\xi)\in \mathscr{S}^{m,\mathcal{L}}_{\rho,\delta}(G).$  Assume also that $a(x,\xi)$ is invertible for every $(x,[\xi])\in G\times\widehat{G},$ and satisfies
\begin{equation}\label{Ies}
   \sup_{(x,[\xi])\in G\times \widehat{G}} \Vert \widehat{\mathcal{M}}(\xi)^{\frac{m}{\kappa}}a(x,\xi)^{-1}\Vert_{\textnormal{op}}<\infty.
\end{equation}Then, $a^{-1}:=a(x,\xi)^{-1}\in \mathscr{S}^{-m,\mathcal{L}}_{\rho,\delta}(G).$ 
\end{theorem}
\begin{proof}
Let us estimate $\partial^{(\beta)}_{X_i}a^{-1}$ first. Suppose we have proved that 
\[ 
 I:= \sup_{|\beta|\leqslant \ell}\sup_{(x,[\xi])\in G\times \widehat{G} }\Vert \widehat{ \mathcal{M}}(\xi)^{\frac{1}{\kappa}(-\delta|\beta|+m)}\partial_{X}^{(\beta)} a(x,\xi)^{-1}\Vert_{\textnormal{op}} <\infty,
   \] 
   for some  $ \ell\in \mathbb{N}.$ We proceed by mathematical induction. Let us analyse the cases $|\tilde\beta|=\ell+1.$ If we write $\partial_{X}^{(\tilde\beta)}= \partial_{X}\partial_{X}^{(\beta)}$ where $|\beta|\leqslant \ell,$ then  $\partial^{\tilde\beta}_{X_i}a^{-1}=\partial_{X}\partial_{X}^{(\beta)}a^{-1}.$ From the identity $a(x,\xi)a(x,\xi)^{-1}=I_{d_\xi}$ we have
   \[ 
       a(x,\xi)\partial_{X}\partial_{X}^{(\beta)}a^{-1}(x,\xi)=-\sum_{\beta_1+\beta_2=\beta+e_{j},|\beta_2|\leqslant  |\beta|}C_{\beta_1,\beta_2}(\partial_{X}^{(\beta_1)}a(x,\xi))(\partial_{X}^{(\beta_2)}a^{-1}(x,\xi)).
   \]
Consequently,
\[ 
       \partial_{X}\partial_{X}^{(\beta)}a^{-1}(x,\xi)=-a(x,\xi)^{-1}\sum_{\beta_1+\beta_2=\beta+e_{j},|\beta_2|\leqslant  |\beta|}C_{\beta_1,\beta_2}(\partial_{X}^{(\beta_1)}a(x,\xi))(\partial_{X}^{(\beta_2)}a^{-1}(x,\xi)).
   \] We want to prove the estimate
   \[ 
 \sup_{(x,[\xi])\in G\times \widehat{G} }\Vert \widehat{ \mathcal{M}}(\xi)^{\frac{1}{\kappa}(-\delta(|\beta|+1)+m)}\partial_{X}\partial_{X}^{(\beta)} a(x,\xi)^{-1}\Vert_{\textnormal{op}} <\infty.
   \]
   For this, we only need to show that for every $\beta_1$ and $\beta_2$ such that $\beta_1+\beta_2=\beta+e_{j},|\beta_2|\leqslant  |\beta|,$
\[ 
    \sup_{(x,[\xi])\in G\times \widehat{G} }\Vert \widehat{ \mathcal{M}}(\xi)^{\frac{1}{\kappa}(-\delta(|\beta|+1)+m)} a(x,\xi)^{-1}(\partial_{X}^{(\beta_1)}a(x,\xi))(\partial_{X}^{(\beta_2)}a^{-1}(x,\xi))   \Vert_{\textnormal{op}} <\infty.
\] 
Observe that
\begin{align*}
 &\widehat{ \mathcal{M}}(\xi)^{\frac{1}{\kappa}(-\delta(|\beta|+1)+m)} a(x,\xi)^{-1}(\partial_{X}^{(\beta_1)}a(x,\xi))(\partial_{X}^{(\beta_2)}a^{-1}(x,\xi))   \\
 &=\widehat{ \mathcal{M}}(\xi)^{\frac{1}{\kappa}(-\delta(|\beta_1|+|\beta_2|)+m)} a(x,\xi)^{-1}(\partial_{X}^{(\beta_1)}a(x,\xi))(\partial_{X}^{(\beta_2)}a^{-1}(x,\xi)).
\end{align*} First, let us prove that
 \begin{align*}
    & \Vert \widehat{ \mathcal{M}}(\xi)^{\frac{1}{\kappa}(-\delta(|\beta_1|+|\beta_2|)} \widehat{ \mathcal{M}}(\xi)^{\frac{m}{\kappa}} a(x,\xi)^{-1}(\partial_{X}^{(\beta_1)}a(x,\xi))(\partial_{X}^{(\beta_2)}a^{-1}(x,\xi)) \Vert_{\textnormal{op}}\\
    &\lesssim \Vert(\partial_{X}^{(\beta_1)}a(x,\xi))(\partial_{X}^{(\beta_2)}a^{-1}(x,\xi))  \widehat{ \mathcal{M}}(\xi)^{\frac{1}{\kappa}(-\delta(|\beta_1|+|\beta_2|)}\Vert_{\textnormal{op}}.
 \end{align*} Indeed, by using \eqref{Recht2} with $A= \widehat{ \mathcal{M}}(\xi)^{\frac{1}{\kappa}(-\delta(|\beta_1|+|\beta_2|)} $ and  $$X=\widehat{ \mathcal{M}}(\xi)^{\frac{1}{\kappa}(-\delta(|\beta_1|+|\beta_2|)} \widehat{ \mathcal{M}}(\xi)^{\frac{m}{\kappa}} a(x,\xi)^{-1}(\partial_{X}^{(\beta_1)}a(x,\xi))(\partial_{X}^{(\beta_2)}a^{-1}(x,\xi)) ,$$ we obtain
 \begin{align*}
  &\Vert \widehat{ \mathcal{M}}(\xi)^{\frac{1}{\kappa}(-\delta(|\beta_1|+|\beta_2|)} \widehat{ \mathcal{M}}(\xi)^{\frac{m}{\kappa}} a(x,\xi)^{-1}(\partial_{X}^{(\beta_1)}a(x,\xi))(\partial_{X}^{(\beta_2)}a^{-1}(x,\xi))\Vert_{\textnormal{op}}\\
  &\lesssim \Vert \widehat{ \mathcal{M}}(\xi)^{\frac{2}{\kappa}(-\delta(|\beta_1|+|\beta_2|)} \\
  &\times \widehat{ \mathcal{M}}(\xi)^{\frac{m}{\kappa}} a(x,\xi)^{-1}(\partial_{X}^{(\beta_1)}a(x,\xi))(\partial_{X}^{(\beta_2)}a^{-1}(x,\xi)) \widehat{ \mathcal{M}}(\xi)^{\frac{1}{\kappa}(-\delta(|\beta_1|+|\beta_2|)}\\
  &\hspace{1cm}+ (1+\widehat{ \mathcal{M}}(\xi)^{\frac{2}{\kappa}(-\delta(|\beta_1|+|\beta_2|)})^{\frac{1}{2}}X(1+\widehat{ \mathcal{M}}(\xi)^{\frac{2}{\kappa}(-\delta(|\beta_1|+|\beta_2|)})^{\frac{1}{2}}   \Vert_{\textnormal{op}} \\
  &\lesssim  \Vert \widehat{ \mathcal{M}}(\xi)^{\frac{2}{\kappa}(-\delta(|\beta_1|+|\beta_2|)} \Vert_{\textnormal{op}}\\
  &\times \Vert \widehat{ \mathcal{M}}(\xi)^{\frac{m}{\kappa}} a(x,\xi)^{-1}\Vert_{\textnormal{op}}\Vert (\partial_{X}^{(\beta_1)}a(x,\xi))(\partial_{X}^{(\beta_2)}a^{-1}(x,\xi)) \widehat{ \mathcal{M}}(\xi)^{\frac{1}{\kappa}(-\delta(|\beta_1|+|\beta_2|)}\Vert_{\textnormal{op}}\\
  &\hspace{1cm}+ \Vert(1+\widehat{ \mathcal{M}}(\xi)^{\frac{2}{\kappa}(-\delta(|\beta_1|+|\beta_2|)})^{\frac{1}{2}}\Vert_{\textnormal{op}}\Vert X(1+\widehat{ \mathcal{M}}(\xi)^{\frac{2}{\kappa}(-\delta(|\beta_1|+|\beta_2|)})^{\frac{1}{2}}   \Vert_{\textnormal{op}}\\
  &\lesssim  \Vert (\partial_{X}^{(\beta_1)}a(x,\xi))(\partial_{X}^{(\beta_2)}a^{-1}(x,\xi)) \widehat{ \mathcal{M}}(\xi)^{\frac{1}{\kappa}(-\delta(|\beta_1|+|\beta_2|)}\Vert_{\textnormal{op}}\\
  &\hspace{1cm}+ \Vert X \widehat{ \mathcal{M}}(\xi)^{\frac{1}{\kappa}(-\delta(|\beta_1|+|\beta_2|)} \widehat{ \mathcal{M}}(\xi)^{\frac{1}{\kappa}(\delta(|\beta_1|+|\beta_2|)} (1+\widehat{ \mathcal{M}}(\xi)^{\frac{2}{\kappa}(-\delta(|\beta_1|+|\beta_2|)})^{\frac{1}{2}}   \Vert_{\textnormal{op}}\\
  &\lesssim  \Vert (\partial_{X}^{(\beta_1)}a(x,\xi))(\partial_{X}^{(\beta_2)}a^{-1}(x,\xi)) \widehat{ \mathcal{M}}(\xi)^{\frac{1}{\kappa}(-\delta(|\beta_1|+|\beta_2|)}\Vert_{\textnormal{op}}\\
  &\hspace{1cm}+ \Vert X \widehat{ \mathcal{M}}(\xi)^{\frac{1}{\kappa}(-\delta(|\beta_1|+|\beta_2|)}\Vert_{\textnormal{op}}\Vert  \widehat{ \mathcal{M}}(\xi)^{\frac{1}{\kappa}(\delta(|\beta_1|+|\beta_2|)} (1+\widehat{ \mathcal{M}}(\xi)^{\frac{2}{\kappa}(-\delta(|\beta_1|+|\beta_2|)})^{\frac{1}{2}}   \Vert_{\textnormal{op}}\\
  &\lesssim   \Vert (\partial_{X}^{(\beta_1)}a(x,\xi))(\partial_{X}^{(\beta_2)}a^{-1}(x,\xi)) \widehat{ \mathcal{M}}(\xi)^{\frac{1}{\kappa}(-\delta(|\beta_1|+|\beta_2|)}\Vert_{\textnormal{op}}\\
  &+ \Vert \widehat{ \mathcal{M}}(\xi)^{\frac{1}{\kappa}(-\delta(|\beta_1|+|\beta_2|)}\\
  &\times \widehat{ \mathcal{M}}(\xi)^{\frac{m}{\kappa}} a(x,\xi)^{-1}(\partial_{X}^{(\beta_1)}a(x,\xi))(\partial_{X}^{(\beta_2)}a^{-1}(x,\xi)) \widehat{ \mathcal{M}}(\xi)^{\frac{1}{\kappa}(-\delta(|\beta_1|+|\beta_2|)}\Vert_{\textnormal{op}}\\
  &\lesssim   \Vert (\partial_{X}^{(\beta_1)}a(x,\xi))(\partial_{X}^{(\beta_2)}a^{-1}(x,\xi)) \widehat{ \mathcal{M}}(\xi)^{\frac{1}{\kappa}(-\delta(|\beta_1|+|\beta_2|)}\Vert_{\textnormal{op}}\\
  &+ \Vert \widehat{ \mathcal{M}}(\xi)^{\frac{1}{\kappa}(-\delta(|\beta_1|+|\beta_2|)} \Vert_{\textnormal{op}}\\
  &\times \Vert \widehat{ \mathcal{M}}(\xi)^{\frac{m}{\kappa}} a(x,\xi)^{-1}\Vert_{\textnormal{op}}\Vert (\partial_{X}^{(\beta_1)}a(x,\xi))(\partial_{X}^{(\beta_2)}a^{-1}(x,\xi)) \widehat{ \mathcal{M}}(\xi)^{\frac{1}{\kappa}(-\delta(|\beta_1|+|\beta_2|)}\Vert_{\textnormal{op}}\\
  &\lesssim   \Vert (\partial_{X}^{(\beta_1)}a(x,\xi))(\partial_{X}^{(\beta_2)}a^{-1}(x,\xi)) \widehat{ \mathcal{M}}(\xi)^{\frac{1}{\kappa}(-\delta(|\beta_1|+|\beta_2|)}\Vert_{\textnormal{op}}.
 \end{align*}  Observe that we have estimated $$\Vert  \widehat{ \mathcal{M}}(\xi)^{\frac{1}{\kappa}(\delta(|\beta_1|+|\beta_2|)} (1+\widehat{ \mathcal{M}}(\xi)^{\frac{2}{\kappa}(-\delta(|\beta_1|+|\beta_2|)})^{\frac{1}{2}}   \Vert_{\textnormal{op}}=O(1),$$ because of
 \begin{align*}
     \widehat{ \mathcal{M}}(\xi)^{\frac{1}{\kappa}(\delta(|\beta_1|+|\beta_2|)} (1+\widehat{ \mathcal{M}}(\xi)^{\frac{2}{\kappa}(-\delta(|\beta_1|+|\beta_2|)})^{\frac{1}{2}} \in \mathscr{S}^{0}_{\frac{1}{\kappa},0}(G\times \widehat{G}).
 \end{align*}
 Again, by using Theorem \ref{cor}, we have

\begin{align*}
   & \Vert(\partial_{X}^{(\beta_1)}a(x,\xi))(\partial_{X}^{(\beta_2)}a^{-1}(x,\xi))  \widehat{ \mathcal{M}}(\xi)^{\frac{1}{\kappa}(-\delta(|\beta_1|+|\beta_2|)}\Vert_{\textnormal{op}}\\
   &\leqslant \Vert(\partial_{X}^{(\beta_1)}a(x,\xi))  \widehat{ \mathcal{M}}(\xi)^{\frac{1}{\kappa}(-m-\delta|\beta_1|)}\Vert_{\textnormal{op}} \\
   &\hspace{4cm}\times\Vert\widehat{ \mathcal{M}}(\xi)^{\frac{1}{\kappa}(m+\delta|\beta_1|)}   (\partial_{X}^{(\beta_2)}a^{-1}(x,\xi))  \widehat{ \mathcal{M}}(\xi)^{\frac{1}{\kappa}(-\delta(|\beta_1|+|\beta_2|)}  \Vert_{\textnormal{op}}\\
   &\lesssim  \Vert\widehat{ \mathcal{M}}(\xi)^{\frac{1}{\kappa}(m+\delta|\beta_1|)}   (\partial_{X}^{(\beta_2)}a^{-1}(x,\xi))  \widehat{ \mathcal{M}}(\xi)^{\frac{1}{\kappa}(-\delta(|\beta_1|+|\beta_2|)}  \Vert_{\textnormal{op}}\\
   &=\Vert\widehat{ \mathcal{M}}(\xi)^{\frac{1}{\kappa}(m+\delta|\beta_1|)}   (\partial_{X}^{(\beta_2)}a^{-1}(x,\xi))  \widehat{ \mathcal{M}}(\xi)^{\frac{1}{\kappa}(-\delta(|\beta_2|+m)}\widehat{ \mathcal{M}}(\xi)^{\frac{1}{\kappa}(-\delta|\beta_1|-m)}\Vert_{\textnormal{op}}\\
   &\lesssim \Vert   (\partial_{X}^{(\beta_2)}a^{-1}(x,\xi))  \widehat{ \mathcal{M}}(\xi)^{\frac{1}{\kappa}(-\delta(|\beta_2|+m)}\Vert_{\textnormal{op}}\\
  & \lesssim 1,
\end{align*}  uniformly in $(x,[\xi]).$
A similar analysis using the Leibniz rule for difference operators can be used in order to estimate the terms $\Delta_\xi^{\alpha}a^{-1}.$ For this, we need the following two estimates,
\begin{equation}
     \Vert \widehat{\mathcal{M}}(\xi)^{\frac{1}{\kappa}(\rho|\gamma|-\delta|\beta|)} a(x,\xi)^{-1}[\Delta_{\xi}^{\gamma}\partial_{X}^{(\beta)}a(x,\xi)]\Vert_{\textnormal{op}}=O(1),
\end{equation}
\begin{equation}
     \Vert  [\Delta_{\xi}^{\gamma}\partial_{X}^{(\beta)}a(x,\xi)]a(x,\xi)^{-1}\widehat{\mathcal{M}}(\xi)^{\frac{1}{\kappa}(\rho|\gamma|-\delta|\beta|)}\Vert_{\textnormal{op}}\\ =O(1).
\end{equation}
For the proof, let us use  \eqref{Recht2}, observing that
\begin{align*}
     \Vert \widehat{\mathcal{M}}(\xi)^{\frac{1}{\kappa}(\rho|\gamma|-\delta|\beta|)} &a(x,\xi)^{-1}[\Delta_{\xi}^{\gamma}\partial_{X}^{(\beta)}a(x,\xi)]\Vert_{\textnormal{op}}\\
    &= \Vert \widehat{\mathcal{M}}(\xi)^{\frac{1}{\kappa}(\rho|\gamma|-\delta|\beta|-m)}\widehat{\mathcal{M}}(\xi)^{\frac{m}{\kappa}} a(x,\xi)^{-1}[\Delta_{\xi}^{\gamma}\partial_{X}^{(\beta)}a(x,\xi)]\Vert_{\textnormal{op}}\\
     &\lesssim \Vert \widehat{\mathcal{M}}(\xi)^{\frac{m}{\kappa}} a(x,\xi)^{-1}[\Delta_{\xi}^{\gamma}\partial_{X}^{(\beta)}a(x,\xi)]\widehat{\mathcal{M}}(\xi)^{\frac{1}{\kappa}(\rho|\gamma|-\delta|\beta|-m)}\Vert_{\textnormal{op}}\\ 
     &\lesssim \Vert \widehat{\mathcal{M}}(\xi)^{\frac{m}{\kappa}} a(x,\xi)^{-1}\Vert_{\textnormal{op}}\Vert[\Delta_{\xi}^{\gamma}\partial_{X}^{(\beta)}a(x,\xi)]\widehat{\mathcal{M}}(\xi)^{\frac{1}{\kappa}(\rho|\gamma|-\delta|\beta|-m)}\Vert_{\textnormal{op}}\\ 
     &= O(1).
\end{align*}On the other hand,
\begin{align*}
     \Vert  [\Delta_{\xi}^{\gamma}\partial_{X}^{(\beta)}a(x,\xi)]&a(x,\xi)^{-1}\widehat{\mathcal{M}}(\xi)^{\frac{1}{\kappa}(\rho|\gamma|-\delta|\beta|)}\Vert_{\textnormal{op}}\\
    &= \Vert  [\Delta_{\xi}^{\gamma}\partial_{X}^{(\beta)}a(x,\xi)]   \widehat{\mathcal{M}}(\xi)^{-\frac{m}{\kappa}}\widehat{\mathcal{M}}(\xi)^{\frac{m}{\kappa}} a(x,\xi)^{-1}\widehat{\mathcal{M}}(\xi)^{\frac{1}{\kappa}(\rho|\gamma|-\delta|\beta|)}\Vert_{\textnormal{op}}\\
     &\lesssim \Vert \widehat{\mathcal{M}}(\xi)^{\frac{1}{\kappa}(\rho|\gamma|-\delta|\beta|)} [\Delta_{\xi}^{\gamma}\partial_{X}^{(\beta)}a(x,\xi)]\widehat{\mathcal{M}}(\xi)^{-\frac{m}{\kappa}}\widehat{\mathcal{M}}(\xi)^{\frac{m}{\kappa}}a(x,\xi)^{-1}\Vert_{\textnormal{op}}\\ 
     &\lesssim \Vert \widehat{\mathcal{M}}(\xi)^{\frac{1}{\kappa}(\rho|\gamma|-\delta|\beta|)} [\Delta_{\xi}^{\gamma}\partial_{X}^{(\beta)}a(x,\xi)]\widehat{\mathcal{M}}(\xi)^{-\frac{m}{\kappa}}\Vert_{\textnormal{op}} \Vert\widehat{\mathcal{M}}(\xi)^{\frac{m}{\kappa}}a(x,\xi)^{-1}\Vert_{\textnormal{op}}\\ 
     &\lesssim \Vert \widehat{\mathcal{M}}(\xi)^{\frac{1}{\kappa}(-m+\rho|\gamma|-\delta|\beta|)} [\Delta_{\xi}^{\gamma}\partial_{X}^{(\beta)}a(x,\xi)]\Vert_{\textnormal{op}} \\ 
     &= O(1).
\end{align*}
Now, we will estimate in $\frac{1}{\kappa}(m+\rho\ell)$ the subelliptic order for the differences  $\Delta_{q_{\ell}}a^{-1},$ in the dilated classes. To do so, we will use the mathematical induction. The case $\ell=0$ holds true from the hypothesis of Theorem \ref{IesT}. To study the differences of higher order we will use the Leibniz rule (see Remark \ref{Leibnizrule}),
\begin{align*}
   \Delta_{q_\ell}[ a_{1}a_{2}](x,\xi) =\sum_{ |\gamma|,|\varepsilon|\leqslant \ell\leqslant |\gamma|+|\varepsilon| }C_{\varepsilon,\gamma}(\Delta_{q_\gamma}a_{1})(x,\xi) (\Delta_{q_\varepsilon}a_{2})(x,\xi),
\end{align*} for $a_{i}\in C^{\infty}(G)\times \mathscr{S}'(\widehat{G}).$
From the identity, $a(x,\xi)a(x,\xi)^{-1}=I_{d_\xi},$ we deduce that
\begin{align*}
  (\Delta_{q_1}a)(x,\xi) & a(x,\xi)^{-1}+a(x,\xi) (\Delta_{q_1}a^{-1})(x,\xi)\\
   &=-\sum_{ 1= |\nu|,|\nu'| }C_{\nu,\nu'}(\Delta_{q_{(\nu)}}a)(x,\xi) (\Delta_{q_{(\nu')}}a^{-1})(x,\xi),
\end{align*}and consequently
\begin{align*}
   &  (\Delta_{q_1}a^{-1})(x,\xi)=-a(x,\xi)^{-1}(\Delta_{q_1}a)(x,\xi)a^{-1}(x,\xi)\\
   &\hspace{2cm}-\sum_{ 1= |\nu|,|\nu'| }C_{\nu,\nu'}a(x,\xi)^{-1}(\Delta_{q_{(\nu)}}a)(x,\xi) (\Delta_{q_{(\nu')}}a^{-1})(x,\xi).
\end{align*} The differences of higher order $\Delta_{q_{\ell+1}}a^{-1},$ can be estimated e.g. from difference operators of the form $$\Delta_{q_{\ell+1}}=\Delta_{q_{\ell}}\Delta_{q_{1}}.$$ 
Then, by applying the difference operator $\Delta_{q_{\ell}}$ to $(\Delta_{q_1}a^{-1})(x,\xi)$ we essentially obtain linear combinations of terms of the following kind as a consequence of the Leibniz rule:
\begin{itemize}
    \item $IV_{(1)}(x,\xi):=\Delta_{q_{\ell}}a(x,\xi)^{-1}\times (\Delta_{q_1}a)(x,\xi)\times a^{-1}(x,\xi)$\\
    \item $IV_{(2)}(x,\xi):=a(x,\xi)^{-1}\times \Delta_{q_{\ell+1}}a(x,\xi)\times a^{-1}(x,\xi)$\\
     \item $IV_{(3)}(x,\xi):=a(x,\xi)^{-1}\times \Delta_{q_{1}}a(x,\xi)\times \Delta_{q_{\ell}}a^{-1}(x,\xi)$\\
      \item $IV_{(4)}(x,\xi):=\Delta_{q_{\ell_1}} a(x,\xi)^{-1}\times(\Delta_{q_{\ell_2+1}}a)(x,\xi)\times \Delta_{q_{\ell_3+1}}a^{-1}(x,\xi),$ $1\leqslant \ell_{1},\ell_2,\ell_3\leqslant \ell\leqslant \ell_1+\ell_2+\ell_3,$\\
      \item $V_{(1)}(x,\xi):=\Delta_{q_{\ell}}a(x,\xi)^{-1}\times (\Delta_{q_{(\nu) } }a)(x,\xi)\times \Delta_{q_{(\nu')}}a^{-1}(x,\xi)$\\
    \item $V_{(2)}(x,\xi):=a(x,\xi)^{-1}\times \Delta_{q_{\ell+\nu}}a(x,\xi)\times \Delta_{q_{\nu'}}a^{-1}(x,\xi)$\\
     \item $V_{(3)}(x,\xi):=a(x,\xi)^{-1}\times \Delta_{q_{(\nu)}}a(x,\xi)\times \Delta_{q_{\ell+\nu'}}a^{-1}(x,\xi)$\\
      \item $V_{(4)}(x,\xi):=\Delta_{q_{\ell_1}} a(x,\xi)^{-1}\times(\Delta_{q_{\ell_2+\nu}}a)(x,\xi)\times \Delta_{q_{\ell_3+\nu'}}a^{-1}(x,\xi),$ $1\leqslant \ell_{1},\ell_2,\ell_3\leqslant \ell\leqslant \ell_1+\ell_2+\ell_3,$\\
\end{itemize}which we can estimate as follows. First, assume that for every $\ell\geqslant  1,$ $\ell\in \mathbb{N},$ we have
$$\sup_{(x,[\xi])}  \Vert\widehat{\mathcal{M}}(\xi)^{\frac{1}{\kappa}(m+\rho\ell)} (\Delta_{q_{\ell}}a^{-1})(x,\xi)\Vert_{\textnormal{op}}  <\infty. $$ We need to prove that
$$\sup_{(x,[\xi])}  \Vert\widehat{\mathcal{M}}(\xi)^{\frac{1}{\kappa}(m+\rho(\ell+1))} (\Delta_{q_{\ell+1}}a^{-1})(x,\xi)\Vert_{\textnormal{op}}  <\infty. $$
For this, it is enough to prove that 
$$ \sup_{(x,[\xi])}  \Vert \widehat{\mathcal{M}}(\xi)^{\frac{1}{\kappa}(m+\rho(\ell+1))} IV_{(i)}(x,\xi)\Vert_{\textnormal{op}},\, \sup_{(x,[\xi])}  \Vert \widehat{\mathcal{M}}(\xi)^{ \frac{1}{\kappa}(m+\rho(\ell+1))} V_{(j)}(x,\xi)\Vert_{\textnormal{op}} <\infty, $$ for all $1\leqslant i,j\leqslant 4.$ Next, let us omit the argument $(x,\xi)$ in order to simplify the notation.
So, we have
\begin{align*}
    &\Vert \widehat{\mathcal{M}}(\xi)^{\frac{1}{\kappa}(m+\rho(\ell+1))} IV_{(1)}\Vert_{\textnormal{op}}\\
    &\hspace{4cm}+     \Vert \widehat{\mathcal{M}}(\xi)^{\frac{1}{\kappa}(m+\rho(\ell+1))} IV_{(2)}(x,\xi)\Vert_{\textnormal{op}}\\
   &\hspace{4cm}+   \Vert \widehat{\mathcal{M}}(\xi)^{\frac{1}{\kappa}(m+\rho(\ell+1))} IV_{(3)}(x,\xi)\Vert_{\textnormal{op}}\\
    &\hspace{4cm}+\Vert \widehat{\mathcal{M}}(\xi)^{\frac{1}{\kappa}(m+\rho(\ell+1))} V_{(1)}(x,\xi)\Vert_{\textnormal{op}}\\
    &\hspace{4cm}+\Vert \widehat{\mathcal{M}}(\xi)^{\frac{1}{\kappa}(m+\rho(\ell+1))} V_{(2)}(x,\xi)\Vert_{\textnormal{op}}\\
    & =\Vert \widehat{\mathcal{M}}(\xi)^{\frac{1}{\kappa}(m+\rho(\ell+1))} \Delta_{q_{\ell}}a^{-1}\times (\Delta_{q_1}a)\times a^{-1}\Vert_{\textnormal{op}}\\&\hspace{2cm}+     \Vert \widehat{\mathcal{M}}(\xi)^{\frac{1}{\kappa}(m+\rho(\ell+1))} a^{-1}\times \Delta_{q_{\ell+1}}a\times a^{-1}\Vert_{\textnormal{op}}\\
   &\hspace{2cm}+   \Vert \widehat{\mathcal{M}}(\xi)^{\frac{1}{\kappa}(m+\rho(\ell+1))} a^{-1}\times \Delta_{q_{1}}a\times \Delta_{q_{\ell}}a^{-1}\Vert_{\textnormal{op}}\\
    &\hspace{2cm}+\Vert \widehat{\mathcal{M}}(\xi)^{\frac{1}{\kappa}(m+\rho(\ell+1))} \Delta_{q_{\ell}}a^{-1}\times (\Delta_{q_{(\nu) } }a)\times \Delta_{q_{(\nu')}}a^{-1}\Vert_{\textnormal{op}}\\
    &\hspace{2cm}+\Vert \widehat{\mathcal{M}}(\xi)^{\frac{1}{\kappa}(m+\rho(\ell+1))} a^{-1}\times \Delta_{q_{\ell+\nu}}a\times \Delta_{q_{\nu'}}a^{-1}\Vert_{\textnormal{op}}\\
    \end{align*}
    
    \begin{align*}
     \lesssim\Vert &\widehat{\mathcal{M}}(\xi)^{\frac{1}{\kappa}(m+\rho\ell)} \Delta_{q_{\ell}}a^{-1}\times (\Delta_{q_1}a)\times a^{-1}\widehat{\mathcal{M}}(\xi)^{\frac{\rho}{\kappa}}\Vert_{\textnormal{op}}\\&\hspace{2cm}+     \Vert \widehat{\mathcal{M}}(\xi)^{\frac{m}{\kappa}} a^{-1}\times \Delta_{q_{\ell+1}}a\times a^{-1} \widehat{\mathcal{M}}(\xi)^{\frac{1}{\kappa}(\rho(\ell+1))} \Vert_{\textnormal{op}}\\
   &\hspace{2cm}+   \Vert \widehat{\mathcal{M}}(\xi)^{\frac{\rho}{\kappa}} a^{-1}\times \Delta_{q_{1}}a\times \Delta_{q_{\ell}}a^{-1}\widehat{\mathcal{M}}(\xi)^{\frac{1}{\kappa}(m+\rho\ell)}\Vert_{\textnormal{op}}\\
    &\hspace{2cm}+\Vert \widehat{\mathcal{M}}(\xi)^{\frac{1}{\kappa}(m+\rho\ell)} \Delta_{q_{\ell}}a^{-1}\Vert_{\textnormal{op}}\Vert  (\Delta_{q_{(\nu) } }a)\times \Delta_{q_{(\nu')}}a^{-1} \widehat{\mathcal{M}}(\xi)^{\frac{\rho}{\kappa}}\Vert_{\textnormal{op}}\\
    &\hspace{2cm}+\Vert \widehat{\mathcal{M}}(\xi)^{\frac{1}{\kappa}(m+\rho(\ell+1))} a^{-1}\times \Delta_{q_{\ell+\nu}}a\times \Delta_{q_{\nu'}}a^{-1}\Vert_{\textnormal{op}}\\
    \\
    & \lesssim\Vert (\Delta_{q_1}a)\times a^{-1}\widehat{\mathcal{M}}(\xi)^{\frac{\rho}{\kappa}}\Vert_{\textnormal{op}}\\&\hspace{2cm}+      \Vert  \Delta_{q_{\ell+1}}a\times a^{-1} \widehat{\mathcal{M}}(\xi)^{\frac{1}{\kappa}(\rho(\ell+1))} \Vert_{\textnormal{op}}\\
   &\hspace{2cm}+   \Vert \widehat{\mathcal{M}}(\xi)^{\frac{\rho}{\kappa}} a^{-1}\times \Delta_{q_{1}}a\Vert_{\textnormal{op}}\\
    &\hspace{2cm}+\Vert  (\Delta_{q_{(\nu) } }a)\times a^{-1} \widehat{\mathcal{M}}(\xi)^{\frac{\rho}{\kappa}}\Vert_{\textnormal{op}}\\
    &\hspace{2cm}+\Vert \widehat{\mathcal{M}}(\xi)^{\frac{1}{\kappa}(\rho(\ell+1))} a^{-1}\times \Delta_{q_{\ell+\nu}}a\Vert_{\textnormal{op}}\\
     &\hspace{2cm}\lesssim 1,
\end{align*}
where we have used  the estimates
\begin{equation}\label{casobacano}
    \Vert  (\Delta_{q_{(\nu) } }a)\times \Delta_{q_{(\nu')}}a^{-1} \widehat{\mathcal{M}}(\xi)^{\frac{\rho}{\kappa}}\Vert_{\textnormal{op}}\lesssim 1,
\end{equation}
and 
\begin{equation}\label{casobacano2}
    \Vert \widehat{\mathcal{M}}(\xi)^{\frac{1}{\kappa}(m+\rho(\ell+1))} a^{-1}\cdot \Delta_{q_{\ell+\nu}}a\cdot \Delta_{q_{\nu'}}a^{-1}\Vert_{\textnormal{op}}\lesssim \Vert \widehat{\mathcal{M}}(\xi)^{\frac{1}{\kappa}(\rho(\ell+1))} a^{-1}\cdot \Delta_{q_{\ell+\nu}}a\Vert_{\textnormal{op}}. 
\end{equation}
Indeed, for the proof of \eqref{casobacano}  observe that from the induction hypothesis, we have
\begin{align*}
  \Vert  (\Delta_{q_{(\nu) } }a)\times &\Delta_{q_{(\nu')}}a^{-1} \widehat{\mathcal{M}}(\xi)^{\frac{\rho}{\kappa}}\Vert_{\textnormal{op}}\\
  &\lesssim   \Vert \widehat{\mathcal{M}}(\xi)^{\frac{1}{\kappa}(-m+\rho)} (\Delta_{q_{(\nu) } }a)\times \Delta_{q_{(\nu')}}a^{-1} \widehat{\mathcal{M}}(\xi)^{\frac{1}{\kappa}(m+\rho)} \widehat{\mathcal{M}}(\xi)^{-\frac{\rho}{\kappa}} \Vert_{\textnormal{op}}   \\
  &=   \Vert \widehat{\mathcal{M}}(\xi)^{\frac{1}{\kappa}(-m+\rho)} (\Delta_{q_{(\nu) } }a)\Vert_{\textnormal{op}} \Vert \Delta_{q_{(\nu')}}a^{-1} \widehat{\mathcal{M}}(\xi)^{\frac{1}{\kappa}(m+\rho)}\Vert_{\textnormal{op}} \Vert \widehat{\mathcal{M}}(\xi)^{-\frac{\rho}{\kappa}} \Vert_{\textnormal{op}}   \\
  &\lesssim 1.
\end{align*}In order to prove \eqref{casobacano2}, observe that
\begin{align*}
   & \Vert \widehat{\mathcal{M}}(\xi)^{\frac{1}{\kappa}(m+\rho(\ell+1))} a^{-1}\times \Delta_{q_{\ell+\nu}}a\times \Delta_{q_{\nu'}}a^{-1}\Vert_{\textnormal{op}}\\
    &\lesssim \Vert \widehat{\mathcal{M}}(\xi)^{\frac{1}{\kappa}(\rho(\ell+1))} a^{-1}\times \Delta_{q_{\ell+\nu}}a\times \Delta_{q_{\nu'}}a^{-1}\widehat{\mathcal{M}}(\xi)^{\frac{m}{\kappa}}\Vert_{\textnormal{op}}\\
     &\lesssim \Vert \widehat{\mathcal{M}}(\xi)^{\frac{1}{\kappa}(\rho(\ell+1))} a^{-1}\times \Delta_{q_{\ell+\nu}}a\Vert_{\textnormal{op}}\Vert \Delta_{q_{\nu'}}a^{-1}\widehat{\mathcal{M}}(\xi)^{\frac{1}{\kappa}(m+\rho)}\Vert_{\textnormal{op}}\Vert \widehat{\mathcal{M}}(\xi)^{-\frac{\rho}{\kappa}}\Vert_{\textnormal{op}} \\
     &\lesssim \Vert \widehat{\mathcal{M}}(\xi)^{\frac{1}{\kappa}(\rho(\ell+1))} a^{-1}\times \Delta_{q_{\ell+\nu}}a\Vert_{\textnormal{op}}.
\end{align*}
A similar analysis can be used to study $IV_{(4)}(x,\xi),V_{(3)}(x,\xi)$ and $V_{(4)}(x,\xi).$ Thus, we end the proof.
\end{proof}
Theorem \ref{IesT} implies the following result for the contracted subelliptic classes.
\begin{corollary}\label{IesTC} Let $m\in \mathbb{R}
,$ and let $0\leqslant \delta<\rho\leqslant 1.$  Let  $a=a(x,\xi)\in {S}^{m,\mathcal{L}}_{\rho,\delta}(G\times \widehat{G}).$  Assume also that $a(x,\xi)$ is invertible for every $(x,[\xi])\in G\times\widehat{G},$ and satisfies
\begin{equation}\label{IesC}
   \sup_{(x,[\xi])\in G\times \widehat{G}} \Vert \widehat{\mathcal{M}}(\xi)^{m}a(x,\xi)^{-1}\Vert_{\textnormal{op}}<\infty.
\end{equation}Then, $a^{-1}:=a(x,\xi)^{-1}\in {S}^{-m,\mathcal{L}}_{\rho,\delta}(G\times \widehat{G}).$ 
\end{corollary}
\begin{definition}[$ \mathcal{L} $-ellipticity]\label{lellipticity} In view of Theorem \ref{IesT} and Corollary \ref{IesTC}, it is justified to say that the symbols $a\in {S}^{m,\mathcal{L}}_{\rho,\delta}(G\times \widehat{G})=\mathscr{S}^{m\kappa,\mathcal{L}}_{\rho\kappa,\delta\kappa}(G) $ satisfying \eqref{IesC} are $\mathcal{L}$-elliptic of order $m$ and of type $(\rho,\delta),$ $0\leqslant \delta,\rho\leqslant 1,$ in the subelliptic contracted classes (or  $\mathcal{L}$-elliptic of order $m\kappa$ and of type $(\rho\kappa,\delta\kappa),$ $0\leqslant \delta,\rho\leqslant 1,$ in the subelliptic dilated classes).
 \end{definition}

\begin{example}[Negative powers of $(1+\mathcal{L})^\frac{1}{2}$]\label{powessecondtpart} As mentioned before, it will be proved in Theorem \ref{orders:theorems} that for every $s\in \mathbb{R},$ 
\begin{equation}\label{thenegativepower2}
  \mathcal{M}_{s}:=  (1+\mathcal{L})^\frac{s}{2}\in \Psi^{s,\mathcal{L}}_{1,0}(G\times \widehat{G}).
\end{equation}
However, let us mention that in Example \ref{powesfirstpart} we have deduced the weak conclusion, $\mathcal{M}_s\in \textnormal{Op}(\mathscr{S}^{s\kappa,\mathcal{L}}_{1,0}(G)),$ from which we deduced the following fact,
\begin{equation}\label{subellipticorder}
    \mathcal{M}_{-s}\in \textnormal{Op}(\mathscr{S}^{-s\kappa,\mathcal{L}}_{1,0}(G))=\textnormal{Op}({S}^{-s,\mathcal{L}}_{\frac{1}{\kappa},0}(G\times \widehat{G})).
\end{equation} Indeed, we have from Example \ref{powesfirstpart} that $\mathcal{M}_s\in \textnormal{Op}(\mathscr{S}^{s\kappa,\mathcal{L}}_{1,0}(G)).$ Since
\[ 
    \sup_{[\xi]\in \widehat{G}} \Vert \widehat{\mathcal{M}}(\xi)^{\frac{s\kappa}{\kappa}}\widehat{\mathcal{M}}_{-s}(\xi)\Vert_{\textnormal{op}}=\sup_{[\xi]\in \widehat{G}} \Vert I_{d_\xi\times d_\xi}\Vert_{\textnormal{op}}=1,
\]we have proved \eqref{Ies} for $m=s\kappa$. So, from Theorem \ref{IesT}, we deduce \eqref{subellipticorder}. Certainly, \eqref{thenegativepower2} is a better conclusion than \eqref{subellipticorder} but its proof will require the subelliptic version of the Hulanicki theorem (see Lemma \ref{HulanickiTheorem}).

\end{example}

\begin{remark}[Dependence of the subelliptic H\"ormander classes on the choice of a sub-Laplacian]\label{dependece1} In general, if we take two sub-Laplacians $\mathcal{L}_1$ and $\mathcal{L}_2$ on $G,$ associated with two systems of vector fields satisfying the H\"ormander condition of order $\kappa,$ the corresponding subelliptic H\"ormander classes $S^{m,\mathcal{L}_1}_{\rho,\delta}(G\times \widehat{G} )$ and $S^{m,\mathcal{L}_2}_{\rho,\delta}(G\times \widehat{G} )$ do not  necessarily agree. This implies that the subelliptic classes may in general  depend on the choice of the sub-Laplacian as we will see in Remark \ref{dependece2}, even for collections of vector fields of the same step.

\end{remark}
\subsection{Singular kernels of subelliptic pseudo-differential operators}
In order to study the behaviour of the kernel of a subelliptic operator near the origin we will consider a strongly  admissible collection of  difference operators $\{\Delta_{q(j)}\}$ and $\gamma_0\in\mathbb{N}_0^i,$ $|\gamma_0|=\ell,$ such that
\[ 
\Delta_{ q_{\ell}}\equiv \Delta_{\xi}^{\gamma_0}:=\Delta_{q_{(1)}}^{\gamma_1}\cdots   \Delta_{q_{(i)}}^{\gamma_{i}},\,\,\gamma_0=(\gamma_j)_{1\leqslant j\leqslant i}. 
\]
The strong admissibility means that
\[ 
    \textnormal{rank}\{\nabla q_{(j)}(e):1\leqslant j\leqslant i \}=\textnormal{dim}(G), \textnormal{  }\Delta_{q_{(j)}}\in \textnormal{diff}^{1}(\widehat{G}),
\] and that $x=e_{G}$ is the only common zero of the functions $q_{(j)},$ i.e. 
\[ 
    \bigcap_{j=1}^{i}\{x\in G: q_{(j)}(x)=0\}=\{e_G\}.
\]
\begin{remark}\label{weyl}
For our further analysis we will use that in view of  the Weyl-eigenvalue counting formula for the sub-Laplacian, we have
\[ 
     N(\lambda):=    \sum_{[\xi]\in \widehat{G}:(1+\nu_{ii}(\xi)^2)^{\frac{s}{2}} \leqslant \lambda,\,\forall 1\leqslant i\leqslant d_\xi} d_\xi^2=O(\lambda^{\frac{Q}{s}}),
     \]for every $\lambda>0,$ and $s\in \mathbb{N}$ (following e.g. Hassannezhad and Kokarev \cite[Theorem 3.5]{HK16}).
   \end{remark}
   \begin{remark}
   Let us note that from Remark \ref{weyl}, we can deduce that $d_\xi^2\lesssim (1+\nu_{ii}(\xi)^2)^{\frac{s}{2}\times \frac{Q}{s}},$ for $\langle\xi\rangle\rightarrow\infty,$ which also implies that
   \begin{equation}\label{XI}
       d_\xi\lesssim (1+\nu_{ii}(\xi)^2)^{\frac{Q}{4}}\asymp (1+\nu_{ii}(\xi))^{\frac{Q}{2}},\quad \,\langle\xi\rangle\rightarrow\infty.
   \end{equation}
   \end{remark}
We summarise the properties for the kernels of subelliptic operators in the following theorem. We will denote $|\cdot|$ the metric induced by the geodesic distance on $G\times G,$ measuring the distance from the identity element $e_G.$

\begin{proposition}\label{CalderonZygmund}
Let $G$ be a compact Lie group of dimension $n$ and let $0\leqslant \delta,\rho\leqslant 1.$  Let  $A:C^\infty(G)\rightarrow\mathscr{D}'(G)$ be a continuous linear operator with symbol $\sigma\in {S}^{m,\mathcal{L}}_{\rho,\delta}(G\times \widehat{G})$   in the contracted subelliptic H\"ormander class of order $m$ and of type $(\rho,\delta)$. Then, the right-convolution kernel of $A,$ $x\mapsto k_{x}:G\rightarrow C^\infty(G\setminus\{e_G\}),$ defined by $k_{x}:=\mathscr{F}^{-1}\sigma(x,\cdot),$ satisfies the following estimates for $|y|<1$:
\begin{itemize}
    \item[(i)] if $m>-Q,$ there exists $\ell\in \mathbb{N},$ independent of $\sigma,$ such that  
        \[ 
            |k_{x}(y)|\lesssim_{m}  \Vert \sigma\Vert_{\ell, S^{m,\mathcal{L}}_{\rho,\delta}}|y|^{-\frac{Q+m}{\rho}}.
        \]
         \item[(ii)] If $m=-Q,$ there exists $\ell\in \mathbb{N},$ independent of $\sigma,$ such that  
        \[ 
            |k_{x}(y)|\lesssim_{m} \Vert \sigma\Vert_{\ell, S^{m,\mathcal{L}}_{\rho,\delta}}|\log|y||.
        \]
        \item[(iii)] If $m<-Q,$ there exists $\ell\in \mathbb{N},$ independent of $\sigma,$ such that  
        \[ 
            |k_{x}(y)|\lesssim_{m} \Vert \sigma\Vert_{\ell, S^{m,\mathcal{L}}_{\rho,\delta}}.
        \]
\end{itemize}
\end{proposition}
The proof of Proposition \ref{CalderonZygmund} requires some preliminary results.
\begin{remark}\label{Remarksk2}
Let us consider $k\in \mathscr{D}'(G),$ and let $s>Q/2.$ Then,
 \begin{equation}\label{inequality}
     \Vert k \Vert_{L^2(G)}\lesssim_{s}\sup_{[\xi]\in \widehat{G}}\Vert \widehat{\mathcal{M}}(\xi)^{s}\widehat{k}(\xi)\Vert_{\textnormal{op}},
 \end{equation}in the sense that $k\in L^2(G)$ when the right hand side is finite. Indeed, let us observe that
\begin{align*}
    \Vert k\Vert_{L^2(G)}&=\Vert \widehat{k} \Vert_{L^2(\widehat{G})}= \Vert  \widehat{\mathcal{M}}(\xi)^{-s} \widehat{\mathcal{M}}(\xi)^{s} \widehat{k} \Vert_{L^2(\widehat{G})}\\
    & \leqslant \sup_{[\xi]\in \widehat{G}}\Vert \widehat{\mathcal{M}}(\xi)^{s}\widehat{k}(\xi)\Vert_{\textnormal{op}}\Vert \widehat{\mathcal{M}}(\xi)^{-s} \Vert_{L^2(\widehat{G})}.
\end{align*} 
Observe that 
\begin{align*}
   \Vert \widehat{\mathcal{M}}(\xi)^{-s} \Vert^2_{L^2(\widehat{G})}&=\sum_{[\xi]\in \widehat{G}}d_\xi\Vert \widehat{\mathcal{M}}(\xi)^{-s}\Vert^2_{\textnormal{HS}}\\
   &=\sum_{j=0}^\infty\sum_{[\xi]:2^j\leqslant (1+\nu_{kk}(\xi)^2)^\frac{1}{2}<2^{j+1} ,\,\,\forall 1\leqslant  k\leqslant d_\xi} d_\xi\sum_{i=1}^{d_\xi}(1+\nu_{ii}(\xi)^2)^{-s}\\
   &\asymp \sum_{j=0}^\infty\sum_{[\xi]:2^j\leqslant (1+\nu_{kk}(\xi)^2)^\frac{1}{2}<2^{j+1} ,\,\,\forall 1\leqslant  k\leqslant d_\xi} d_\xi^22^{-2js}\\
   &\leqslant  \sum_{j=0}^\infty\sum_{[\xi]:2^j\leqslant (1+\nu_{kk}(\xi)^2)^\frac{1}{2}<2^{j+1} ,\,\,\forall 1\leqslant  k\leqslant d_\xi} d_\xi^22^{-2js}\\
   &= \sum_{j=0}^\infty N(2^{j+1}) 2^{-2js},
\end{align*} where $N(\lambda)$ denotes the Weyl function for the sub-Laplacian. From Remark \ref{weyl}, we have that $N(2^{j+1})\asymp N(2^{j})\simeq 2^{jQ}.$ Consequently $\Vert \widehat{\mathcal{M}}(\xi)^{-s} \Vert^2_{L^2(\widehat{G})}<\infty,$ for $s>Q/2.$ Thus, we conclude the proof.
\end{remark}
\begin{remark}\label{sset}Let us observe that for $s>Q/2,$  we have the embedding $$H^{s,\mathcal{L}}(G)\equiv L^{2,\mathcal{L}}_s(G)\hookrightarrow C(G).$$ This is well known, but we will provide a proof for completeness. If $f\in C^\infty(G),$ then we have
\begin{align*}
    |f(x)| &\leqslant \sum_{[\xi]\in \widehat{G}}d_\xi|\textnormal{\textbf{Tr}}[\xi(x)\widehat{f}(\xi)]|=\sum_{[\xi]\in \widehat{G}}d_\xi|\textnormal{\textbf{Tr}}[\xi(x)\widehat{\mathcal{M}}(\xi)^{-s}\widehat{\mathcal{M}}(\xi)^{s}  \widehat{f}(\xi)]|\\
    &\leqslant \sum_{[\xi]\in \widehat{G}}d_\xi^{\frac{1}{2}}\Vert \xi(x)\widehat{\mathcal{M}}(\xi)^{-s}\Vert_{\textnormal{HS}}d_\xi^{\frac{1}{2}}\Vert \widehat{\mathcal{M}}(\xi)^{s}  \widehat{f}(\xi)\Vert_{\textnormal{HS}}\\
   &\leqslant \left( \sum_{[\xi]\in \widehat{G}}d_\xi\Vert \xi(x)\widehat{\mathcal{M}}(\xi)^{-s}\Vert_{\textnormal{HS}}^2 \right)^{\frac{1}{2}}\left( \sum_{[\xi]\in \widehat{G}}d_\xi\Vert \widehat{\mathcal{M}}(\xi)^{s}  \widehat{f}(\xi)\Vert_{\textnormal{HS}}^2 \right)^{\frac{1}{2}}\\
   &\leqslant \left( \sum_{[\xi]\in \widehat{G}}d_\xi\Vert\widehat{\mathcal{M}}(\xi)^{-s}\Vert_{\textnormal{HS}}^2 \right)^{\frac{1}{2}}\left( \sum_{[\xi]\in \widehat{G}}d_\xi\Vert \widehat{\mathcal{M}}(\xi)^{s}  \widehat{f}(\xi)\Vert_{\textnormal{HS}}^2 \right)^{\frac{1}{2}}\\
   &=\Vert \widehat{\mathcal{M}}(\xi)^{-s}  \Vert_{L^2(\widehat{G})}\Vert f\Vert_{L^{2,\mathcal{L}}_s(G)}.
\end{align*}The condition $s>Q/2,$ implies that $\Vert \widehat{\mathcal{M}}(\xi)^{-s}  \Vert_{L^2(\widehat{G})}<\infty.$ Now, if $f_i\rightarrow f$ in $H^s(G),$ the previous inequality shows that $f_{i}\rightarrow f$ uniformly on $G,$ and it shows that $f$ is continuous.
\end{remark}

 \begin{lemma}\label{lemma3.18}
  Let $G$ be a compact Lie group of Hausdorff dimension $Q$ and let  $0\leqslant \delta,\rho\leqslant 1.$  Let  $A:C^\infty(G)\rightarrow\mathscr{D}'(G)$ be a continuous linear operator with symbol $\sigma\in {S}^{m,\mathcal{L}}_{\rho,\delta}(G\times \widehat{G})$   in the contracted subelliptic H\"ormander class of order $m$ and of type $(\rho,\delta)$. Let $k_{x}$ be the right-convolution kernel of $A,$ and let us define the  function \[ 
      f_{\beta',\beta,\alpha}(z):=(1+\mathcal{L}_{z})^{|\beta|}([q_{(1)}^{\alpha(1)}\cdots q_{(i)}^{\alpha(i)} ]\partial_{x}^{\beta'}k_{x}(z))
  \] where $\{\Delta_{q(i)}\},$   is an admissible family of difference operators with  $x=e_{G}$ being the only common zero of the functions $q_{(i)}.$ Then, $f_{\beta',\beta,\alpha}$ is continuous and bounded provided that
  \[ 
       Q+2|\beta|+m+\delta|\beta'|<\rho|\alpha|.
  \]Moreover
  \[ 
     \sup_{z\in G}|f_{\beta',\beta,\alpha}(z)|\lesssim \Vert \sigma \Vert_{|\alpha|+|\beta'|,S^{m,\mathcal{L}}_{\rho,\delta}}.
  \]
 \end{lemma} 
\begin{proof}
Let us fix $s>Q/2,$ and let $a=|\alpha|.$ By using the subelliptic Sobolev embedding theorem (see Remark \ref{sset}), if we prove that there exists $C_s>0,$ such that
 \begin{equation}\label{Cs}
     \Vert f_{\beta',\beta,\alpha}(z) \Vert_{H^{s,\mathcal{L}}(G)}\lesssim C_s,
 \end{equation}then $f_{\beta',\beta,\alpha}(z)$ is a continuous and  bounded function, and 
 \[ 
     \sup_{z\in G}|f_{\beta',\beta,\alpha}(z)|\lesssim_{s} \Vert f_{\beta',\beta,\alpha}(z) \Vert_{H^{s,\mathcal{L}}(G)}.
 \]
 Observe that
 \begin{align*}
     &\Vert f_{\beta',\beta,\alpha}(z) \Vert_{H^{s,\mathcal{L}  }(G)}=\Vert (1+\mathcal{L})^{\frac{ s+2|\beta| }{2}} [q_{(1)}^{\alpha(1)}\cdots q_{(i)}^{\alpha(i)} ]\partial_{x}^{\beta'}k_{x}(z) \Vert_{L^2(G)}.
 \end{align*}  Fix $s'>Q/2.$ From
 \eqref{inequality}, we deduce
 \begin{align*}
   &  \Vert (1+\mathcal{L})^{\frac{ s+2|\beta| }{2}} [q_{(1)}^{\alpha(1)}\cdots q_{(i)}^{\alpha(i)} ]\partial_{x}^{\beta'}k_{x}(z) \Vert_{L^2(G)}\\
    &\lesssim \sup_{[\xi]\in \widehat{G}}  \Vert \widehat{\mathcal{M}}(\xi)^{s'}\mathscr{F}[ (1+\mathcal{L})^{\frac{ s+2|\beta| }{2}} [q_{(1)}^{\alpha(1)}\cdots q_{(i)}^{\alpha(i)} ]\partial_{x}^{\beta'}k_{x}(z)   ]\Vert_{\textnormal{op}}\\
     &= \sup_{[\xi]\in \widehat{G}} \Vert\widehat{\mathcal{M}}(\xi)^{s'+s+2|\beta|} \Delta_{q_{a}}\partial_{x}^{\beta'}\sigma(x,\xi) \Vert_{\textnormal{op}}.
 \end{align*} Since $\sigma\in {S}^{m,\mathcal{L}}_{\rho,\delta}(G\times \widehat{G})=\mathscr{S}^{m\kappa,\mathcal{L}}_{\rho\kappa,\delta\kappa}(G\times \widehat{G}),$ we can write,
 \begin{align*}
  &\Vert\widehat{\mathcal{M}}(\xi)^{s'+s+2|\beta|} \Delta_{q_{a}}\partial_{x}^{\beta'}\sigma(x,\xi) \Vert_{\textnormal{op}}\\
  &=\Vert\widehat{\mathcal{M}}(\xi)^{s'+s+2|\beta|+m-\rho|\alpha|+\delta|\beta'|} \widehat{\mathcal{M}}(\xi)^{-m+\rho|\alpha|-\delta|\beta'|}\Delta_{q_{a}}\partial_{x}^{\beta'}\sigma(x,\xi) \Vert_{\textnormal{op}}\\
 &\leqslant \Vert\widehat{\mathcal{M}}(\xi)^{s'+s+2|\beta|+m-\rho|\alpha|+\delta|\beta'|}\Vert_{\textnormal{op}} \Vert \widehat{\mathcal{M}}(\xi)^{-m+\rho|\alpha|-\delta|\beta'|}\Delta_{q_{a}}\partial_{x}^{\beta'}\sigma(x,\xi) \Vert_{\textnormal{op}}\\
 &\lesssim \Vert\widehat{\mathcal{M}}(\xi)^{s'+s+2|\beta|+m-\rho|\alpha|+\delta|\beta'|}\Vert_{\textnormal{op}}\Vert \sigma \Vert_{|\alpha|+|\beta'|,S^{m,\mathcal{L}}_{\rho,\delta}}.
 \end{align*}To conclude the proof, we need only that 
 \[ 
    I_{s,s'}:= \sup_{[\xi]\in \widehat{G}}\Vert\widehat{\mathcal{M}}(\xi)^{s'+s+2|\beta|+m-\rho|\alpha|+\delta|\beta'|}\Vert_{\textnormal{op}}<\infty.
 \]Indeed, in \eqref{Cs}, we can take $C_s:=I_{s,s'}\Vert \sigma \Vert_{|\alpha|+|\beta'|,S^{m,\mathcal{L}}_{\rho,\delta}}.$ But, in order to assure that $I_{s,s'}<\infty,$ we only need to impose that $s'+s+2|\beta|+m-\rho|\alpha|+\delta|\beta'|\leqslant 0,$ restricted to $s+s'>\frac{Q}{2}+\frac{Q}{2}=Q.$ The inequality, $Q+2|\beta|<s'+s+2|\beta|,$ shows that the required inequality holds if, $Q+2|\beta|+m-\rho|\alpha|+\delta|\beta'|< 0.$ In this case, we can assure the existence of $s,$ and $s'$ close enough  to $Q/2$, as above such that
 $$ Q+2|\beta|+m-\rho|\alpha|+\delta|\beta'| <s'+s+2|\beta|+m-\rho|\alpha|+\delta|\beta'|\leqslant  0.$$
 Thus, we finish the proof.
 \end{proof} From the previous lemma we deduce the following immediate consequence.

\begin{corollary}\label{Corollary}
Let $G$ be a compact Lie group of Hausdorff dimension $Q$ and let  $0\leqslant \delta,\rho\leqslant 1.$  Let  $A:C^\infty(G)\rightarrow\mathscr{D}'(G)$ be a continuous linear operator with symbol $\sigma\in {S}^{m,\mathcal{L}}_{\rho,\delta}(G\times \widehat{G})$   in the contracted subelliptic H\"ormander class of order $m$ and of type $(\rho,\delta)$. Let $k_{x}$ be the right-convolution kernel of $A,$ and let us define the  function \[ 
      f(z):=k_{x}(z).
  \]  Then, $f$ is continuous and bounded provided that
  \[ 
       Q+m<0.
  \]
\end{corollary}
Finally we need the following Lemma (see Lemma 6.8 of \cite{Fischer2015}).
\begin{lemma}\label{LemmaFischer} Let $\sigma\in \mathscr{S}^{m}_{\rho,\delta}(G\times \widehat{G}),$ with $0\leqslant \delta\leqslant \rho\leqslant 1.$ Let $\eta\in C^\infty_0(\mathbb{R}).$ For every $t\in (0,1),$ we define the symbol $\sigma_t$ by $\sigma_t(x,\xi)=\sigma(x,\xi)\eta(t\langle \xi\rangle).$ Then, for every $m_1\in \mathbb{R},$ we have
\[ 
    \Vert\sigma_t \Vert_{a+b,S^{m_1}_{\rho,\delta}}\leqslant C\Vert\sigma  \Vert_{a+b,\mathscr{S}^{m}_{\rho,\delta}}t^{\frac{m_1-m}{2}}.
\] where $C=C_{m,m_1,a,b,\eta}$ does not depend on $\sigma$ and $t\in (0,1).$ 
 
\end{lemma}  
  \begin{proof}[Proof of Proposition \ref{CalderonZygmund}]  In the light of Remark \ref{Remarksk2} and  Lemma  \ref{lemma3.18} we will follow the approach in the proof of Proposition 6.8 in \cite{Fischer2015}, adapted to the subelliptic case. Note that for $m<-Q,$  Proposition \ref{CalderonZygmund} follows from Corollary \ref{Corollary}. So, we will assume that $m\geqslant   -Q.$ In order to prove the theorem in this case, we will use a Littlewood-Paley decomposition. Let us choose  $\eta_{0}$ and $\eta_{1},$ supported in $[-1,1]$ and $[\frac{1}{2},2]$ respectively, with $0\leqslant \eta_0,\eta_1\leqslant 1,$ such that
  \[ 
      \forall t\geqslant   0,\,\,\,\sum_{\ell=0}^{\infty}\eta_{\ell}(t)=1,\,\,\,\textnormal{ where for all }\ell\in\mathbb{N},\,\,\,\eta_\ell(t):=\eta_{1}(2^{-(\ell-1)}t).
  \] For every $\ell\in \mathbb{N}_0,$ let us denote $\sigma_\ell(x,\xi):=\sigma(x,\xi)\eta_{\ell}(\langle \xi\rangle).$ Let us denote by $k_{x}$ and $k_{x,\ell}$ the right-convolutions kernels associated with $A$ and $T_{\sigma_\ell}=\textnormal{Op}(\sigma_\ell),$ respectively. 
  We have the (possibly unbounded summation),
  \[ 
      |k_x(y)|\leqslant \sum_{\ell=0}^\infty|k_{x,\ell}(y)|.
  \] By Lemma \ref{lemma3.18}, we have
  \[ 
      \sup_{z\in G}|[q_{(1)}^{\alpha(1)}\cdots q_{(i)}^{\alpha(i)} ]k_{x,\ell}(z)|\lesssim \Vert \sigma_{\ell} \Vert_{|\alpha|,S^{m_1,\mathcal{L}}_{\rho,\delta}}.
  \]Although  we can take an arbitrary $m_1\in \mathbb{R}$   because $k_{x,\ell}$ is a smooth function in view that $\sigma_\ell\in \mathscr{S}^{-\infty}_{1,0}:=\cap_{\nu\in \mathbb{R}}\mathscr{S}^{\nu}_{1,0}$, the hypothesis in Lemma \ref{lemma3.18} suggests us to fix the following condition $$m_1+Q<\rho|\alpha|.$$ 
  From \eqref{Fischersub} in Remark \ref{Remark}, applied with $s=2^{-(\ell-1)}$ and $\omega=(m_1-m)/2,$ we obtain the following estimate for the  subelliptic seminorm $\Vert\cdot \Vert_{|\alpha|,S^{m_1,\mathcal{L}}_{\rho,\delta}}$ of $\sigma_{\ell},$
  \[ 
      \Vert \sigma_{\ell} \Vert_{|\alpha|,S^{m_1,\mathcal{L}}_{\rho,\delta}}\lesssim \Vert \sigma \Vert_{|\alpha|,S^{m,\mathcal{L}}_{\rho,\delta}} 2^{-(\ell-1)(\frac{m_1-m}{2})} \lesssim  \Vert \sigma \Vert_{|\alpha|,S^{m,\mathcal{L}}_{\rho,\delta}} 2^{-\ell(\frac{m_1-m}{2})}.  
  \] As is Lemma 5.6 of \cite{Fischer2015}, the fact that (see \cite[page 20]{Fischer2015}) \[ 
    \textnormal{rank}\{\nabla q_{(j)}(e):1\leqslant j\leqslant i \}=\textnormal{dim}(G), \textnormal{  }\Delta_{q_{(j)}}\in \textnormal{diff}^{1}(\widehat{G}),
\] and by using that $x=e_{G}$ is the only common zero of the functions $q_{(j)},$ implies that (see \cite[page 20]{Fischer2015})
\[ 
    |z|^{a}\lesssim \sum_{|\alpha|=a}|q_{(1)}^{\alpha(1)}\cdots q_{(i)}^{\alpha(i)}|.
\]Consequently, for every $a\in 2\mathbb{N}_0,$ and $m_1\in \mathbb{R}$ with $m_1+Q<\rho a,$ we have proved the estimate
\[ 
    |z|^{a}|k_{x,\ell}(z)|\lesssim \Vert \sigma \Vert_{a,S^{m,\mathcal{L}}_{\rho,\delta}} 2^{\ell(\frac{m-m_1}{2})}.
\]Because our analysis is local, and $G$ is compact, we only need to consider the case $|z|<1.$ So, let us choose $\ell_0\in\mathbb{N},$ such that $|z|\sim 2^{-\ell_0}.$ So, we consider the following two situations.
\begin{itemize}
    \item Case 1. $m+Q>0.$
    \item Case 2. $m+Q=0.$
\end{itemize}
  In the first case, for $\ell\leqslant \ell_0,$ we can choose $m_1\in \mathbb{R},$ such that 
  \[ 
      \frac{m+Q}{\rho}>a\geqslant   \frac{m+Q}{\rho}-2, \textnormal{   and,  }\frac{m-m_1}{2}=\frac{m+Q}{\rho}-a.
  \] Because $m+Q>\rho a,$ $\frac{m+Q}{\rho}-a=\frac{m-m_1}{2}>0.$ So, $m>m_1,$ and 
  \begin{align*}
      \sum_{\ell=0}^{\ell_0}|k_{x,\ell}(z)|&\lesssim  \sum_{\ell=0}^{\ell_0} |z|^{-a}\Vert \sigma \Vert_{a,S^{m,\mathcal{L}}_{\rho,\delta}} 2^{\ell(\frac{m-m_1}{2})} \\&
    \leqslant   \sum_{\ell=0}^{\ell_0} |z|^{-a}\Vert \sigma \Vert_{a,S^{m,\mathcal{L}}_{\rho,\delta}} 2^{\ell_0(\frac{m-m_1}{2})}\\&
      =   |z|^{-a}\Vert \sigma \Vert_{a,S^{m,\mathcal{L}}_{\rho,\delta}} \ell_02^{\ell_0(\frac{m-m_1}{2})}\\
      &\lesssim \Vert \sigma \Vert_{a,S^{m,\mathcal{L}}_{\rho,\delta}} 2^{\ell_0(\frac{ m+Q}{\rho}})\asymp |z|^{  - \frac{ m+Q}{\rho}}.
  \end{align*}On the other hand, if $\ell>\ell_0,$   we choose $m_1$ satisfying
   \[ 
      \frac{m+Q}{\rho}>a\geqslant   \frac{m+Q}{\rho}-2, \textnormal{   and,  }\frac{m-m_1}{2}=\frac{m+Q}{\rho}-a-2.
  \]
  
  So,  $m<m_1,$ and again we have
  \begin{align*}
      \sum_{\ell=\ell_0+1}^{\infty}|k_{x,\ell}(z)|&\lesssim  \sum_{\ell=\ell_0+1}^{\infty} |z|^{-a}\Vert \sigma \Vert_{a,S^{m,\mathcal{L}}_{\rho,\delta}} 2^{\ell(\frac{m-m_1}{2})} \\&
    \lesssim    |z|^{-a}\Vert \sigma \Vert_{a,S^{m,\mathcal{L}}_{\rho,\delta}} 2^{\ell_0(\frac{m-m_1}{2})}\\&
      \asymp  |z|^{-a}\Vert \sigma \Vert_{a,S^{m,\mathcal{L}}_{\rho,\delta}} |z|^{-(\frac{m-m_1}{2})}=\Vert \sigma \Vert_{a,S^{m,\mathcal{L}}_{\rho,\delta}} |z|^{-(\frac{m-m_1}{2}+a)}\\
      &=\Vert \sigma \Vert_{a,S^{m,\mathcal{L}}_{\rho,\delta}} |z|^{ -( \frac{m+Q}{\rho}-2) }\\ &\leqslant \Vert \sigma \Vert_{a,S^{m,\mathcal{L}}_{\rho,\delta}}  |z|^{  - \frac{ m+Q}{\rho}}.
  \end{align*}So, we have proved the statement for Case 1. Now, if we consider $m=-Q,$ then we choose $a=0,$ $m_1=m$ and we proceed as above,
  \begin{align*}
      \sum_{\ell=0}^{\ell_0}|k_{x,\ell}(z)|&\lesssim  \sum_{\ell=0}^{\ell_0} |z|^{-a}\Vert \sigma \Vert_{a,S^{m,\mathcal{L}}_{\rho,\delta}} 2^{\ell(\frac{m-m_1}{2})} 
    =  \sum_{\ell=0}^{\ell_0} \Vert \sigma \Vert_{a,S^{m,\mathcal{L}}_{\rho,\delta}} \\&
      =   \Vert \sigma \Vert_{a,S^{m,\mathcal{L}}_{\rho,\delta}} \cdot\ell_0\\
      &\asymp \Vert \sigma \Vert_{a,S^{m,\mathcal{L}}_{\rho,\delta}} |\log(2^{-\ell_0})|\\
      &\asymp \Vert \sigma \Vert_{a,S^{m,\mathcal{L}}_{\rho,\delta}} |\log|z||.
  \end{align*}
 On the other hand, if $\ell>\ell_0,$ we choose $m_1=m+4,$ and $a=2.$ Then, $\frac{m_1-m}{2}=2=a,$ and we can estimate
  \begin{align*}
      \sum_{\ell=\ell_0+1}^{\infty}|k_{x,\ell}(z)|&\lesssim  \sum_{\ell=\ell_0+1}^{\infty} |z|^{-a}\Vert \sigma \Vert_{a,S^{m,\mathcal{L}}_{\rho,\delta}} 2^{-2\ell}= \Vert \sigma \Vert_{a,S^{m,\mathcal{L}}_{\rho,\delta}}2^{2\ell_0}\sum_{\ell=\ell_0+1}^{\infty}  2^{-2\ell} \\&
      \asymp \Vert \sigma \Vert_{a,S^{m,\mathcal{L}}_{\rho,\delta}} 2^{2\ell_0}\times \frac{2^{-2(\ell_0+1)}}{1-2^{-2}}\lesssim \Vert \sigma \Vert_{a,S^{m,\mathcal{L}}_{\rho,\delta}}\leq \Vert \sigma \Vert_{a,S^{m,\mathcal{L}}_{\rho,\delta}}\ell_0\\
      &\asymp \Vert \sigma \Vert_{a,S^{m,\mathcal{L}}_{\rho,\delta}} |\log|z||.
  \end{align*}Thus, we finish the proof.
  \end{proof}
 
\subsection{Calder\'on-Vaillancourt Theorem for subelliptic classes} 
The aim of this subsection is to prove the following subelliptic version of the Calder\'on-Vaillancourt theorem. 
\begin{theorem}\label{CVT}
Let $G$ be a compact Lie group and let us consider the sub-Laplacian $\mathcal{L}=\mathcal{L}_X,$ where  $X= \{X_{1},\cdots,X_k\} $ is a system of vector fields satisfying the H\"ormander condition of order $\kappa$.  For  $0\leqslant \delta< \rho\leqslant    1,$  let us consider a continuous linear operator $A:C^\infty(G)\rightarrow\mathscr{D}'(G)$ with symbol  $\sigma\in {S}^{0,\mathcal{L}}_{\rho,\delta}(G\times \widehat{G})$. Then $A$ extends to a bounded operator from $L^2(G)$ to  $L^2(G).$ Moreover,
\begin{equation}\label{eq1CVT}
    \Vert  A\Vert_{\mathscr{B}(L^2(G))}\leqslant C \Vert \sigma\Vert_{\ell, {S}^{0,\mathcal{L}}_{\rho,\delta} },
\end{equation} for $\ell \in \mathbb{N}$ large enough.  In the case where $\delta < 1/\kappa,$ and $\rho\leq 1,$ the condition $\delta<\rho$ can be improved to $\delta\leq \rho$ in order to obtain the $L^2(G)$-boundedness of $A.$
\end{theorem} 
 In order to prove Theorem \ref{CVT}, we start with the following subelliptic bilinear estimate. For every $[\tau]\in \widehat{G},$ we will denote by $H_{\lambda_{[\tau]}}(G)$ the eigenspace of the Laplacian $\mathcal{L}_{G}$ associated with the eigenvalue $\lambda_{[\tau]}.$
 \begin{lemma}\label{bilinearestimate} Let  $f\in H_{\lambda_{[\xi]}}(G),$  $g\in H_{\lambda_{[\xi']}}(G),$ and let $\gamma\in \mathbb{R},$ $s>Q/2,$ be such that $2\gamma+s\leqslant 0.$ Then, the following subelliptic bilinear  estimate
 \[ 
     \Vert  (1+\mathcal{L})^{\gamma}(fg)\Vert_{L^2(G)}\leqslant   C_{s,\gamma}  (1+ |\lambda_{[\xi']}-\lambda_{[\xi]} |)^{\frac{1}{\kappa}( \gamma+\frac{s}{2} +  \frac{Q}{8} ) } \Vert g \Vert_{L^2(G)}\Vert f \Vert_{L^2(G)}
 \]holds true uniformly in $[\xi],[\xi']\in \widehat{G}.$ Here $n$ is the topological dimension of $G.$
 \end{lemma}
 \begin{proof} Let us consider $f\in H_{\lambda_{[\xi]}}(G)$ and $g\in H_{\lambda_{[\xi']}}(G).$ Without loss of generality  let us consider $\lambda_{[\xi']}\geqslant   \lambda_{[\xi]}. $ 
 From the Fourier inversion formula we have,
 \[ 
     (fg)(x):=f(x)g(x)=\sum_{[\eta]\in \widehat{G}: \lambda_{[\eta]}= \lambda_{[\xi']}    }  d_{\eta}\textnormal{\textbf{Tr}}[\eta(x)
   f(x)\widehat{g}(\eta) ].
 \]So, if the apply $(1+\mathcal{L})^{\gamma}$ to both sides, we have
 \begin{align*}
     (1+\mathcal{L})^{\gamma}(fg)(x)=\sum_{[\eta]\in \widehat{G}: \lambda_{[\eta]}= \lambda_{[\xi']}    }  d_{\eta}\textnormal{\textbf{Tr}}[(1+\mathcal{L})^{\gamma}(\eta(x)
   f(x))\widehat{g}(\eta) ].
 \end{align*}The Cauchy-Schwarz inequality gives,
 \begin{align*}
     |(1+\mathcal{L})^{\gamma}(fg)(x)|&\leqslant \sum_{[\eta]\in \widehat{G}: \lambda_{[\eta]}= \lambda_{[\xi']}    }  d_{\eta}|\textnormal{\textbf{Tr}}[(1+\mathcal{L})^{\gamma}(\eta(x)
   f(x))\widehat{g}(\eta) ]|\\
   &\leqslant \Vert g\Vert_{L^2(G)}\left( \sum_{[\eta]\in \widehat{G}: \lambda_{[\eta]}= \lambda_{[\xi']}    }  d_{\eta}\Vert (1+\mathcal{L})^{\gamma}(\eta(x)
   f(x))\Vert_{\textnormal{HS}}^2      \right)^{\frac{1}{2}},
 \end{align*}and consequently
 
 \begin{equation}\label{384}
     \Vert  (1+\mathcal{L})^{\gamma}(fg)\Vert_{L^2(G)}^2\leqslant \Vert g \Vert^2_{L^2(G)}\sum_{[\eta]\in \widehat{G}: \lambda_{[\eta]}= \lambda_{[\xi']}    }  d_{\eta}\int\limits_{G}\Vert (1+\mathcal{L})^{\gamma}(\eta(x)
   f(x))\Vert_{\textnormal{HS}}^2   dx.
 \end{equation}To estimate the integral on the right-hand side we can consider a basis $\{e_{\tau,i}\}_{1\leqslant i\leqslant d_\tau}$ of every representation space $\mathbb{C}^{d_\tau},$ where the symbol of operator $\mathcal{M}_{2\gamma}:=(1+\mathcal{L})^{\gamma}$ is diagonal
 \[ 
     \widehat{\mathcal{M}}_{2\gamma}(\tau)e_{\tau,i}=\widehat{\mathcal{M}}(\tau)^{2\gamma}e_{\tau,i}=(1+\nu_{ii}(\tau)^2)^{\gamma}e_{\tau,i},\,\,\forall 1\leqslant i\leqslant d_\tau,\,\,[\tau]\in \widehat{G}.
 \]
 So, from the Plancherel formula, we have
 \begin{align*}
   &  \int\limits_{G}\Vert (1+\mathcal{L})^{\gamma}(\eta(x)
   f(x))\Vert_{\textnormal{HS}}^2   dx = \sum_{1\leqslant k,\ell\leqslant d_\eta}\int\limits_{G}|  (1+\mathcal{L})^{\gamma}(f(x)\eta_{\ell,k}(x))   |^2dx\\
   &= \sum_{1\leqslant k,\ell\leqslant d_\eta}\sum_{[\tau]\in \widehat{G}}d_\tau\Vert  \widehat{\mathcal{M}}(\tau)^{2\gamma}\mathscr{F}[f\eta_{\ell,k}](\tau)  \Vert_{\textnormal{HS}}^2\\
   &= \sum_{1\leqslant k,\ell\leqslant d_\eta}\sum_{[\tau]\in \widehat{G}}\sum_{\ell',k'=1}^{d_\tau}d_\tau\vert  (1+\nu_{\ell'\ell'}(\tau)^2)^{\gamma}\mathscr{F}[f\eta_{\ell,k}](\tau_{\ell'k'})  \vert^2.
 \end{align*}
 Observe that
 \begin{align*}
     \mathscr{F}[f\eta_{\ell,k}](\tau_{\ell'k'}) &=\int\limits_{G}f(y)\eta_{\ell,k}(y)\tau^{*}_{\ell'k'}(y)dy=\int\limits_{G}f(y^{-1})\tau_{\ell'k'}(y)\eta^{*}_{\ell,k}(y)dy\\
     &=\mathscr{F}[\tilde{f}\tau_{\ell'k'}](\eta_{\ell,k}),
 \end{align*}where $\tilde{f}(\cdot):=f(\cdot^{-1}).$
 So,  by changing the summation order we have,
 \begin{align*}
   &  \sum_{1\leqslant k,\ell\leqslant d_\eta} \sum_{[\tau]\in \widehat{G}}\sum_{\ell',k'=1}^{d_\tau}d_\tau\vert  (1+\nu_{\ell'\ell'}(\tau)^2)^{\gamma}\mathscr{F}[f\eta_{\ell,k}](\tau_{\ell'k'})  \vert^2\\
     &=\sum_{[\tau]\in \widehat{G}}d_\tau \sum_{\ell',k'=1}^{d_\tau}(1+\nu_{\ell'\ell'}(\tau)^2)^{2\gamma}\sum_{1\leqslant k,\ell\leqslant d_\eta}\vert  \mathscr{F}[\tilde{f}\tau_{\ell'k'}](\eta_{\ell,k})  \vert^2\\
     &=\sum_{[\tau]\in \widehat{G}}d_\tau \sum_{\ell',k'=1}^{d_\tau}(1+\nu_{\ell'\ell'}(\tau)^2)^{2\gamma}\Vert  \mathscr{F}[\tilde{f}\tau_{\ell'k'}](\eta)  \Vert^2_{\textnormal{HS}}.
 \end{align*}
 Returning to \eqref{384}, and using the Plancherel theorem, we can estimate
 
 \begin{align*}
   &  \Vert  (1+\mathcal{L})^{\gamma}(fg)\Vert_{L^2(G)}^2\leqslant \Vert g \Vert^2_{L^2(G)}\sum_{[\eta]\in \widehat{G}: \lambda_{[\eta]}= \lambda_{[\xi']}    }  d_{\eta}\int\limits_{G}\Vert (1+\mathcal{L})^{\gamma}(\eta(x)
   f(x))\Vert_{\textnormal{HS}}^2   dx\\
   &\leqslant \Vert g \Vert^2_{L^2(G)}\sum_{[\eta]\in \widehat{G}: \lambda_{[\eta]}= \lambda_{[\xi']}    } d_\eta\sum_{[\tau]\in \widehat{G}}d_\tau \sum_{\ell',k'=1}^{d_\tau}(1+\nu_{\ell'\ell'}(\tau)^2)^{2\gamma}\Vert  \mathscr{F}[\tilde{f}\tau_{\ell'k'}](\eta)  \Vert^2_{\textnormal{HS}}\\
   &= \Vert g \Vert^2_{L^2(G)} \sum_{[\tau]\in \widehat{G}}d_\tau \sum_{\ell',k'=1}^{d_\tau}(1+\nu_{\ell'\ell'}(\tau)^2)^{2\gamma}\sum_{[\eta]\in \widehat{G}: \lambda_{[\eta]}= \lambda_{[\xi']}    }d_\eta\Vert  \mathscr{F}[\tilde{f}\tau_{\ell'k'}](\eta)  \Vert^2_{\textnormal{HS}}\\
   &= \Vert g \Vert^2_{L^2(G)} \sum_{[\tau]\in \widehat{G}}d_\tau \sum_{\ell',k'=1}^{d_\tau}(1+\nu_{\ell'\ell'}(\tau)^2)^{2\gamma}\Vert P_{\xi'}[\tilde{f}\tau_{\ell'k'}] \Vert^2_{L^2(G)},
 \end{align*}
 where $P_{\xi'}:L^2(G)\rightarrow H_{\lambda_{[\eta]}}(G),$ denotes the orthogonal projection on the subspace $H_{\lambda_{[\xi^{'}]}}(G),$ with $\lambda_{[\eta]}=\lambda_{[\xi']}. $ 
 By using that $f\in H_{\lambda_{[\xi]}}(G),$ and $\tau_{\ell'k'}\in H_{\lambda_{[\tau']}}(G).$ It was proved in \cite[page 45]{Fischer2015}, that $f \tau_{\ell'k'}\in \bigoplus_{\lambda_{[\omega]}\leqslant \max\{\lambda_{[\xi]},\lambda_{[\tau]} \} }H_{\lambda_{[\omega]}}.$ So, if $\lambda_{[\xi]}+\lambda_{[\tau]}<\lambda_{[\xi']},$ then $P_{\xi'}[f\tau_{\ell'k'}]\equiv 0.$ From this analysis we deduce
 \begin{align*}
   & \Vert  (1+\mathcal{L})^{\gamma}(fg)\Vert_{L^2(G)}^2\leqslant \Vert g \Vert^2_{L^2(G)}\sum_{[\eta]\in \widehat{G}: \lambda_{[\eta]}= \lambda_{[\xi']}    }  d_{\eta}\int\limits_{G}\Vert (1+\mathcal{L})^{\gamma}(\eta(x)
   f(x))\Vert_{\textnormal{HS}}^2   dx\\
   &\leqslant  \Vert g \Vert^2_{L^2(G)} \sum_{[\tau]\in \widehat{G}:\lambda_{[\tau]}\geqslant  \lambda_{[\xi']}-\lambda_{[\xi]}    }d_\tau \sum_{\ell',k'=1}^{d_\tau}(1+\nu_{\ell'\ell'}(\tau)^2)^{2\gamma}\Vert P_{\xi'}[\tilde{f}\tau_{\ell'k'}] \Vert^2_{L^2(G)}\\
   &\leqslant  \Vert g \Vert^2_{L^2(G)} \sum_{[\tau]\in \widehat{G}:\lambda_{[\tau]}\geqslant  \lambda_{[\xi']}-\lambda_{[\xi]}    }d_\tau\sum_{1\leqslant \ell_0\leqslant d_\tau}(1+\nu_{\ell'_0\ell'_0}(\tau)^2)^{2\gamma}\sup_{1\leqslant \ell'\leqslant d_\tau} \sum_{k'=1}^{d_\tau}\Vert P_{\xi'}[\tilde{f}\tau_{\ell'k'}] \Vert^2_{L^2(G)}.
 \end{align*}
 
 Since $\Vert P_{\xi'}\Vert_{\mathscr{B}(L^2(G))}=1,$ and $|\tau_{\ell'k'}(x)|\leq \Vert\tau_{\ell'k'}(x) \Vert_{\textnormal{op}}=1,$ we have
 \begin{align*}
   & \sum_{k'=1}^{d_\tau}\Vert P_{\xi'}[\tilde{f}\tau_{\ell'k'}] \Vert^2_{L^2(G)}\leqslant d_\tau  \sup_{k'=1,\cdots,d_\tau}\Vert \tau_{\ell'k'} \Vert^2_{L^\infty(G)} \Vert \tilde{f} \Vert^2_{L^2(G)}\leqslant d_\tau\Vert \tilde{f} \Vert_{L^2(G)}^2\\
    &=d_\tau\Vert {f} \Vert^2_{L^2(G)}.
 \end{align*}
 Consequently,
 \begin{align*}
     &\Vert  (1+\mathcal{L})^{\gamma}(fg)\Vert_{L^2(G)}^2\leqslant \Vert g \Vert^2_{L^2(G)}\sum_{[\eta]\in \widehat{G}: \lambda_{[\eta]}= \lambda_{[\xi']}    }  d_{\eta}\int\limits_{G}\Vert (1+\mathcal{L})^{\gamma}(\eta(x)
   f(x))\Vert_{\textnormal{HS}}^2   dx\\
  & \leqslant  \Vert g \Vert^2_{L^2(G)}\sum_{[\tau]\in \widehat{G}:\lambda_{[\tau]}\geqslant  \lambda_{[\xi']}-\lambda_{[\xi]}    }d_\tau \sum_{1\leqslant \ell_0\leqslant d_\tau}(1+\nu_{\ell'_0\ell'_0}(\tau)^2)^{2\gamma}\times d_\tau\Vert {f} \Vert^2_{L^2(G)}\\
  &= \Vert g \Vert^2_{L^2(G)}\Vert {f} \Vert^2_{L^2(G)}\sum_{[\tau]\in \widehat{G}:\lambda_{[\tau]}\geqslant  \lambda_{[\xi']}-\lambda_{[\xi]}    }d_\tau^{2} \sum_{1\leqslant \ell_0\leqslant d_\tau}(1+\nu_{\ell'_0\ell'_0}(\tau)^2)^{2\gamma}.
 \end{align*}
  Hence, we have obtained the following estimate,
 \begin{align*}
   & \Vert  (1+\mathcal{L})^{\gamma}(fg)\Vert_{L^2(G)}^2\leqslant   \Vert g \Vert^2_{L^2(G)}\Vert f \Vert^2_{L^2(G)} I ,
 \end{align*}
 where,
 \[ 
    I:= \sum_{[\tau]\in \widehat{G}:\lambda_{[\tau]}\geqslant  \lambda_{[\xi']}-\lambda_{[\xi]}    }    {d_\tau}^{2}(1+\nu_{\ell'_0\ell'_0}(\tau)^2)^{2\gamma}<\infty.
 \]
 In order to finish the proof we need to show that $I<\infty.$ Let us fix $\gamma$ and $s$ such that  $2\gamma+s +  \frac{Q}{4}\leqslant 0,$ and $s>Q/2.$ Observe that
 \begin{align*}
   & \sum_{[\tau]\in \widehat{G}:\lambda_{[\tau]}\geqslant  \lambda_{[\xi']}-\lambda_{[\xi]}    }    {d_\tau}^{2}(1+\nu_{\ell'_0\ell'_0}(\tau)^2)^{2\gamma}\\
   &= \sum_{[\tau]\in \widehat{G}:\lambda_{[\tau]}\geqslant  \lambda_{[\xi']}-\lambda_{[\xi]}    }    {d_\tau}\times d_\tau(1+\nu_{\ell'_0\ell'_0}(\tau)^2)^{2\gamma+s}(1+\nu_{\ell'_0\ell'_0}(\tau)^2)^{-s}\\
    &\leq  \sum_{[\tau]\in \widehat{G}:\lambda_{[\tau]}\geqslant  \lambda_{[\xi']}-\lambda_{[\xi]}    } \sum_{1\leqslant \ell'_1\leqslant d_\tau}   {d_\tau}(1+\nu_{\ell'_0\ell'_0}(\tau)^2)^{2\gamma+s}    {d_\tau}(1+\nu_{\ell'_1,\ell'_1}(\tau)^2)^{-s},
    \end{align*} and  by using the estimate $d_{\tau}\lesssim (1+\nu_{\ell'_0\ell'_0}(\tau)^2)^{\frac{Q}{4}}$ (see \eqref{XI}) we have
    \begin{align*}
    &\sum_{[\tau]\in \widehat{G}:\lambda_{[\tau]}\geqslant  \lambda_{[\xi']}-\lambda_{[\xi]}    } \sup_{1\leqslant \ell'\leqslant d_\tau}   {d_\tau}(1+\nu_{\ell'_0\ell'_0}(\tau)^2)^{2\gamma+s } \sum_{1\leqslant \ell'_1\leqslant d_\tau}   {d_\tau}(1+\nu_{\ell'_1\ell'_1}(\tau)^2)^{-s} \\
     &\leqslant \sum_{[\tau]\in \widehat{G}:\lambda_{[\tau]}\geqslant  \lambda_{[\xi']}-\lambda_{[\xi]}    } \sup_{1\leqslant \ell'\leqslant d_\tau}   (1+\nu_{\ell'_0\ell'_0}(\tau)^2)^{2\gamma+s +  \frac{Q}{4} } \sup_{1\leqslant \ell'_1\leqslant d_\tau}   {d_\tau}(1+\nu_{\ell'_1\ell'_1}(\tau)^2)^{-s} \\
     &\lesssim \sum_{[\tau]\in \widehat{G}:\lambda_{[\tau]}\geqslant  \lambda_{[\xi']}-\lambda_{[\xi]}    }    (1+ \lambda_{[\xi']}-\lambda_{[\xi]} )^{\frac{1}{\kappa}( 2\gamma+s +  \frac{Q}{4} ) } \sup_{1\leqslant \ell'_1\leqslant d_\tau}   {d_\tau}(1+\nu_{\ell'_1\ell'_1}(\tau)^2)^{-s}\\
     &\leqslant  (1+ \lambda_{[\xi']}-\lambda_{[\xi]} )^{\frac{1}{\kappa}( 2\gamma+s +  \frac{Q}{4} ) } \sum_{[\tau]\in \widehat{G}:}    \sum_{1\leqslant \ell'_1\leqslant d_\tau}   {d_\tau}(1+\nu_{\ell'_1\ell'_1}(\tau)^2)^{-s}\\
     &= (1+ \lambda_{[\xi']}-\lambda_{[\xi]} )^{ \frac{1}{\kappa}( 2\gamma+s +  \frac{Q}{4} ) } \Vert \widehat{\mathcal{M}}_{-s}(\xi)\Vert^2_{L^2(\widehat{G})}<\infty,
 \end{align*}where we have used that $\Vert \widehat{\mathcal{M}}_{-s}(\xi)\Vert_{L^2(\widehat{G})}<\infty$ for $s>Q/2.$ Thus, we conclude that
 $$   \Vert  (1+\mathcal{L})^{\gamma}(fg)\Vert_{L^2(G)}^2\lesssim   \Vert g \Vert^2_{L^2(G)}\Vert f \Vert^2_{L^2(G)} (1+ \lambda_{[\xi']}-\lambda_{[\xi]} )^{ \frac{1}{\kappa}( 2\gamma+s +  \frac{Q}{4} )   }.  $$ So, we finish the proof. 
 \end{proof} \begin{proof}[Proof of Theorem \ref{CVT}] Let us fix $0\leq \delta< 1/\kappa,$ $0< \rho\leq 1,$ with $\delta\leq \rho.$  Indeed,  Theorem \ref{CVT} for $\rho=0$ follows from Theorem 10.5.5 of \cite{Ruz}. We will give an additional proof for the case $0\leqslant \delta<\rho\leqslant 1$  in Corollary \ref{CL2''}.  Let us choose  $\eta_{0}$ and $\eta_{1},$ supported in $[-1,1]$ and $[\frac{1}{2},2]$ respectively, with $0\leqslant \eta_0,\eta_1\leqslant 1,$ such that
  \[ 
      \forall t\geqslant   0,\,\,\,\sum_{\ell=0}^{\infty}\eta_{\ell}(t)=1,\,\,\,\textnormal{ where for all }\ell\in\mathbb{N},\,\,\,\eta_\ell(t):=\eta_{1}(2^{- \frac{\ell}{n} }t),
  \]  For every $\ell\in \mathbb{N}_0,$ let us denote $\sigma_\ell(x,\xi):=\sigma(x,\xi)\eta_{\ell}(\langle \xi\rangle).$ Let us denote by $k_{x}$ and $k_{x,\ell}$ the right-convolutions kernels associated with $A$ and $T_{\sigma_\ell}=\textnormal{Op}(\sigma_\ell).$ From the properties of the dyadic decomposition we have (see e.g. \cite[page 30]{Fischer2015})
 \begin{align*}
     &\Vert A\Vert^2_{\mathscr{B}(L^2(G))}\\
     &\lesssim \sup_{\ell\in \mathbb{N}_0}\Vert T_{\sigma_{\ell}}\Vert^2_{\mathscr{B}(L^2(G))}+\sum_{\ell\neq \ell'\,\ell,\ell'\in 2\mathbb{N}_0} \Vert T_{\sigma_{\ell}}^{*}  T_{\sigma_{\ell'}}\Vert_{\mathscr{B}(L^2(G))}+\sum_{\ell\neq \ell'\,\ell,\ell'\in 2\mathbb{N}_0+1} \Vert T_{\sigma_{\ell}}^{*}  T_{\sigma_{\ell'}}\Vert_{\mathscr{B}(L^2(G))}\\
     &:=I+II+III,
 \end{align*}where $I:=\sup_{\ell\in \mathbb{N}_0}\Vert T_{\sigma_{\ell}}\Vert_{\mathscr{B}(L^2(G))},$ and  $II:=\sum_{\ell\neq \ell'\,\ell,\ell'\in 2\mathbb{N}_0} \Vert T_{\sigma_{\ell}}^{*}  T_{\sigma_{\ell'}}\Vert_{\mathscr{B}(L^2(G))}.$
 In order to prove that $I<\infty,$ we will use that the exponential map\footnote{The exponential of $X\in \mathfrak{g}$ is given by $\textnormal{exp}(X)=\gamma_X(1),$ where $\gamma_X:\mathbb{R}\rightarrow G,$ is the unique one parameter subgroup of $G$ whose tangent vector at the identity is equal to $X.$} \begin{equation}\textnormal{exp}:\mathfrak{v}\subset  \mathfrak{g}\simeq \mathbb{R}^{n}\rightarrow B(e_G,\varepsilon_0),
\end{equation}    is a local diffeomorphism from an open  neighbourhood $\mathfrak{v}$ of $0\in \mathbb{R}^n,$ $n=\dim(G),$ into an  ball $B(e_G,\varepsilon_0)$  containing the identity $e_G$ of $G.$ The ball $B(e_G,\varepsilon_0)$  is defined by the geodesic distance on $G$ and $\varepsilon_0>0.$ By the compactness of $G,$ there exists a finite number of elements $x_{0}=e_{G},x_{i},$ $1\leqslant i\leqslant N_0,$ in $G,$ and some smooth functions $\chi_{j}\in C^\infty(G,[0,1]),$ supported in $B(e_G,\varepsilon_0/2),$ such that
\begin{equation}\label{chij}
    G=\bigcup_{j=1}^{N_0}B(x_{j},\varepsilon_0/4),\,\,\,\textnormal{   and  }\sum_{j=0}^{N_0}\chi_{j}(x_j^{-1}x)=1,\,\,\, x\in G.
\end{equation}
 For every $0<r\leqslant 1,$ let us define the local dilation $D_{r}:B(e_G,\varepsilon_0/2)\rightarrow B(e_G, \varepsilon_0)$ given by \[ 
     D_{r}(x)\equiv r\cdot x:=\exp(r \exp^{-1}(x)).
 \] Observe that, 
 \begin{align*}
    & \Vert T_{\ell}\Vert_{\mathscr{B}(L^2(G))}=\Vert \textnormal{Op}(\sigma_{\ell}(x,\xi)) \Vert_{\mathscr{B}(L^2(G))}\leqslant \sum_{j=0}^{N_0}\Vert \textnormal{Op}(\sigma_{\ell}(x,\xi)\chi_j(x_j^{-1}x)) \Vert_{\mathscr{B}(L^2(G))}\\
    &=\sum_{j=0}^{N_0}\Vert \textnormal{Op}(\sigma_{\ell}(x_jx,\xi)\chi_j(x)) \Vert_{\mathscr{B}(L^2(G))},
 \end{align*}with the possibly unbounded sequence in $\ell\in \mathbb{N}_0$ in the right-hand side of the inequalities. The idea, is to show that this is not the case, and we will estimate the operator norm $\Vert \textnormal{Op}(\sigma_{\ell}(x_jx,\xi)\chi_j(x)) \Vert_{\mathscr{B}(L^2(G))}.$ To simplify the notation, let us write $\sigma_{j,\ell}(x,\xi):=\sigma_{\ell}(x_jx,\xi)\chi_j(x).$ So, we have 
 \[ 
   \sup_{\ell\in \mathbb{N}_0}  \Vert T_{\ell}\Vert_{\mathscr{B}(L^2(G))}\leqslant \sum_{j=0}^{N_0}  \sup_{\ell\in \mathbb{N}_0}\Vert \textnormal{Op}(\sigma_{j,\ell}) \Vert_{\mathscr{B}(L^2(G))}.
 \]Because the family of symbols are compactly supported in both variables, $x$ and $[\xi],$ the fact that $\sigma\in S^{0,\mathcal{L}}_{\rho,\delta}(G\times \widehat{G}),$ also implies that $\sigma_{j,\ell}\in S^{0,\mathcal{L}}_{\rho,\delta}(G\times \widehat{G}).$

Let us define some  dilations of every dyadic part $\sigma_{j,\ell}$ in the spirit of the classical proof of the Calder\'on-Vaillancourt Theorem on $\mathbb{R}^n$ (see \cite{Calderon1,Calderon2}). The support of $\sigma_{j,\ell}$ is contained in the set of $(x,[\xi]),$ such that $\langle \xi\rangle\sim 2^{  \frac{\ell}{n}  }$ and $x\in B(e_G,\varepsilon_0),$ so we can define the following dilation of $\sigma_{j,\ell},$
\[ 
    \tilde{\sigma}_{j,\ell}(x,\xi):=\sigma_{j,\ell}(2^{- \frac{\rho\ell}{n} }\cdot x,\xi)\times 1_{B(e_G,\varepsilon_0)}.
\]Taking into account that $\tilde{\sigma}_{j,\ell},$ keep fixed the $[\xi]$-variables of the symbol $\sigma_{j,\ell},$ we deduce that $\tilde{\sigma}_{j,\ell}$ satisfies the subelliptic  conditions of type $\rho.$ We claim that $\tilde{\sigma}_{j,\ell}\in S^{0,\mathcal{L}}_{\rho,\delta}(G\times \widehat{G}).$ So, we need to check that $\tilde{\sigma}_{j,\ell}$ satisfies the subelliptic  conditions of type $\delta,$ for $\delta\leqslant \rho.$ For this, observe that,
\begin{align*}
    \partial_{X_{i}}^{(1)}\tilde{\sigma}_{j,\ell}(x,\xi)&= \partial_{X_{i}}^{(1)}{\sigma}_{j,\ell}(2^{- \frac{\rho\ell}{n} }\cdot x,\xi)=\partial_{X_{i}}^{(1)}{\sigma}_{j,\ell}(\exp({ 2^{- \frac{\rho\ell}{n} }\exp^{-1}{(x)}} ),\xi)\\
    &=(\partial_{X_{i}}^{(1)}{\sigma}_{j,\ell})(\exp({ 2^{- \frac{\rho\ell}{n} }\exp^{-1}{(x)}} ),\xi)\cdot (\partial_{X_i}^{(1)}\exp)({ 2^{- \frac{\rho\ell}{n} }\exp^{-1}{(x)}} )\\
    &\hspace{3cm}\times 2^{- \frac{\rho\ell}{n} }(\partial_{X_i}^{(1)}\exp^{-1})(x)\\
    &=2^{- \frac{\rho\ell}{n} }(\partial_{X_{i}}^{(1)}{\sigma}_{j,\ell})(\exp({ 2^{- \frac{\rho\ell}{n} }\exp^{-1}{(x)}} ),\xi)\cdot (\partial_{X_i}^{(1)}\exp)({ 2^{- \frac{\rho\ell}{n} }\exp^{-1}{(x)}} )\\
    &\hspace{3cm}\times (\partial_{X_i}^{(1)}\exp^{-1})(x).
\end{align*}Now, from the fact that $\sigma_{j,\ell}\in S^{0,\mathcal{L}}_{\rho,\delta}(G\times \widehat{G}),$ we deduce
\begin{align*}
   & \Vert \widehat{\mathcal{M}}(\xi)^{-\delta}\partial_{X_{i}}^{(1)}\tilde{\sigma}_{j,\ell}(z'',\xi)\Vert_{\textnormal{op}}\\
   &\leqslant 2^{- \frac{\rho\ell}{n} }\sup_{z'',z',z\in G}|(\partial_{X_i}^{(1)}\exp)(z )(\partial_{X_i}^{(1)}\exp^{-1})(z')|\Vert \widehat{\mathcal{M}}(\xi)^{-\delta}(\partial_{X_{i}}^{(1)}{\sigma}_{j,\ell})(z'',\xi) \Vert_{\textnormal{op}}\\
   &\lesssim 2^{- \frac{\rho\ell}{n} } \Vert \widehat{\mathcal{M}}(\xi)^{-\delta}(\partial_{X_{i}}^{(1)}{\sigma}_{j,\ell})(z'',\xi) \Vert_{\textnormal{op}}\\
   &\lesssim 2^{- \frac{\rho\ell}{n} } \Vert \widehat{\mathcal{M}}(\xi)^{-\delta}(\partial_{X_{i}}^{(1)}{\sigma})(z'',\xi) \Vert_{\textnormal{op}}.
\end{align*} A similar argument can help us to deduce the following estimate for higher derivatives
\begin{equation}\label{dilationestimate}
   \sup_{(x,\xi)\in G\times \widehat{G}} \Vert \widehat{\mathcal{M}}(\xi)^{-\delta|\beta|}\partial_{X_{i}}^{(\beta)}\tilde{\sigma}_{j,\ell}\Vert_{\textnormal{op}}\lesssim 2^{- \frac{\rho\ell}{n} |\beta|} \sup_{(x,\xi)\in G\times \widehat{G}} \Vert \widehat{\mathcal{M}}(\xi)^{-\delta|\beta|}(\partial_{X_{i}}^{(1)}{\sigma})(z'',\xi) \Vert_{\textnormal{op}}<\infty.
\end{equation}The symbol $\sigma_{j,\ell}$ and its convolution kernel $k_{j,\ell,x}$ are supported in $x$ in $B(0,\varepsilon_0),$ and dilating the $x$-argument to $x\in B(e_G, 2^{{ \frac{\rho\ell}{n}}  }\varepsilon_0),$ implies the identity
\[ 
    (\textnormal{Op}(\sigma_{j,\ell})f)(2^{-{ \frac{\rho\ell}{n}}  }\cdot x)=f\ast k_{j,\ell, 2^{-{ \frac{\rho\ell}{n}}  }\cdot x  }=f\ast \tilde{k}_{j,\ell,  x  } (2^{-{ \frac{\rho\ell}{n}}  }\cdot x)= (\textnormal{Op}(\tilde{\sigma}_{j,\ell})f)(2^{-{ \frac{\rho\ell}{n}}  }\cdot x),
\]where we have denoted by $\tilde{k}_{j,\ell,  x  }=\mathscr{F}^{-1}\tilde{\sigma}_{j,\ell}$ the right-convolution kernel associated with $\tilde{\sigma}_{j,\ell}.$ From the usual Sobolev embedding theorem we have
\begin{align*}
    |(\textnormal{Op}(\sigma_{j,\ell})f)(2^{-{ \frac{\rho\ell}{n}}  }\cdot x)|\lesssim \sum_{|\beta|\leqslant [\frac{n}{2}]+1}\Vert \partial_{Z'}^{(\beta)} ( f\ast \tilde{k}_{j,\ell,  z'  } ) (2^{-{ \frac{\rho\ell}{n}}  }\cdot x) \Vert_{L^2(G, dz')},
 \end{align*} and we deduce
\begin{align*}
    &\Vert (\textnormal{Op}(\sigma_{j,\ell})f)(2^{-{ \frac{\rho\ell}{n}}  }\cdot x)\Vert_{L^2(B(e_G,\,2^{ \frac{\rho\ell}{n} }\varepsilon_0),dx)}\\
    &\hspace{2cm}\lesssim \sum_{|\beta|\leqslant [\frac{n}{2}]+1}\Vert \partial_{Z'}^{(\beta)}  (f\ast \tilde{k}_{j,\ell,  z'  })  (2^{-{ \frac{\rho\ell}{n}}  }\cdot x) \Vert_{L^2(B(e_G,\,2^{ \frac{\rho\ell}{n} }\varepsilon_0)\times G, dxdz')}.
\end{align*}By the  change of variables, for the transformation $x'=2^{- \frac{\rho\ell}{n} }\cdot x,$ we have
\begin{align*}
 &\Vert (\textnormal{Op}(\sigma_{j,\ell})f)(2^{-{ \frac{\rho\ell}{n}}  }\cdot x)\Vert_{L^2(B(e_G,\,2^{ \frac{\rho\ell}{n} }\varepsilon_0),dx)}=\left(\int\limits_{B(e_G,\,2^{ \frac{\rho\ell}{n} }\varepsilon_0)}     |(\textnormal{Op}(\sigma_{j,\ell})f)(2^{-{ \frac{\rho\ell}{n}}  }\cdot x)|^2dx\right)^{\frac{1}{2}} \\
 &=\left(\int\limits_{B(e_G,\,\varepsilon_0)}     |(\textnormal{Op}(\sigma_{j,\ell})f)(x')|^2\left|\det\left( \frac{\partial x}{\partial x'}\right)\right|dx'\right)^{\frac{1}{2}}\\
 &\asymp  2^{\frac{\rho\ell}{2}}\left(\int\limits_{B(e_G,\,\varepsilon_0)}     |(\textnormal{Op}(\sigma_{j,\ell})f)(x')|^2dx'\right)^{\frac{1}{2}}.
\end{align*}A similar argument implies that
\begin{align*}
   & \sum_{|\beta|\leqslant [\frac{n}{2}]+1}\Vert \partial_{Z'}^{(\beta)}  (f\ast \tilde{k}_{j,\ell,  z'  })  (2^{- { \frac{\rho\ell}{n}}  }\cdot x) \Vert_{L^2(B(e_G,\,2^{ \frac{\rho\ell}{n} }\varepsilon_0)\times G, dxdz')}\\
    &\hspace{2cm}\asymp 2^{\frac{\rho\ell}{2}}\sum_{|\beta|\leqslant [\frac{n}{2}]+1}\Vert \partial_{Z'}^{(\beta)}  (f\ast \tilde{k}_{j,\ell,  z'  })  ( x') \Vert_{L^2(B(e_G,\,\varepsilon_0), dx'dz')}
\end{align*} and consequently we have proved that,
\[ 
 \Vert (\textnormal{Op}(\sigma_{j,\ell})f)( x')\Vert_{L^2(B(e_G,\,\varepsilon_0),dx')}   \lesssim \sum_{|\beta|\leqslant [\frac{n}{2}]+1}\Vert \partial_{Z'}^{(\beta)}  (f\ast \tilde{k}_{j,\ell,  z'  })  (x') \Vert_{L^2(B(e_G,\varepsilon_0)\times G, dx'dz')}.
\] Now, let us estimate the right-hand side of the previous inequality,
\begin{align*}
   & \Vert \partial_{Z'}^{(\beta)}  (f\ast \tilde{k}_{j,\ell,  z'  })  (x') \Vert_{L^2(B(e_G,\,)\times G, dx'dz')}=\Vert  f\ast \partial_{Z'}^{(\beta)} \tilde{k}_{j,\ell,  z'  } (x') \Vert_{L^2(B(e_G,\varepsilon_0)\times G,
   dx'dz')}\\
   &\leqslant \sup_{z'\in G}\Vert  f\ast \partial_{Z'}^{(\beta)} \tilde{k}_{j,\ell,  z'  } (x') \Vert_{L^2(B(e_G,\varepsilon_0)dx')}\\
   &=\sup_{z'\in G}\Vert  \textnormal{Op}(\partial_{Z'}^{(\beta)} \tilde{\sigma}_{j,\ell})f (x') \Vert_{L^2(B(e_G,\varepsilon_0)dx')}\\
   &\leqslant \sup_{z'\in G}\Vert  \textnormal{Op}(\partial_{Z'}^{(\beta)} \tilde{\sigma}_{j,\ell}) \Vert_{\mathscr{B}(L^2(G))}\Vert f\Vert_{L^2(G)}.
\end{align*}Observe that from \eqref{dilationestimate}, for $|\beta|\geqslant   1,$ we have
\begin{align*}
    &\sup_{z'\in G}\Vert  \textnormal{Op}(\partial_{Z'}^{(\beta)} \tilde{\sigma}_{j,\ell}) \Vert_{\mathscr{B}(L^2(G))}=\sup_{z'\in G, \,\langle \xi\rangle\sim 2^{  \frac{\ell}{n}  }}\Vert \partial_{Z'}^{(\beta)} \tilde{\sigma}_{j,\ell}(z',\xi)\Vert_{\textnormal{op}} \\
    &=\sup_{z''\in G, \,\langle \xi\rangle\sim 2^{  \frac{\ell}{n}  }}\Vert \widehat{\mathcal{M}}(\xi)^{\delta|\beta|} \widehat{\mathcal{M}}(\xi)^{-\delta|\beta|}\partial_{Z'}^{(\beta)} \tilde{\sigma}_{j,\ell}(z',\xi)\Vert_{\textnormal{op}} \\
    &\lesssim \sup_{z''\in G, \,\langle \xi\rangle\sim 2^{  \frac{\ell}{n}  }} \Vert \widehat{\mathcal{M}}(\xi)^{\delta|\beta|} \Vert_{\textnormal{op}}  2^{- \frac{\rho\ell|\beta|}{n} } \Vert \widehat{\mathcal{M}}(\xi)^{-\delta|\beta|}(\partial_{X}^{(\beta)}{\sigma})(z'',\xi) \Vert_{\textnormal{op}}\\
     &= \sup_{z''\in G, \,\langle \xi\rangle\sim 2^{  \frac{\ell}{n}  }}\sup_{1\leqslant i\leqslant d_\xi}  (1+\nu_{ii}(\xi)^2)^{\frac{\delta|\beta|}{2}}  2^{- \frac{\rho\ell|\beta|}{n} } \Vert \widehat{\mathcal{M}}(\xi)^{-\delta|\beta|}(\partial_{X}^{(\beta)}{\sigma})(z'',\xi) \Vert_{\textnormal{op}}\\
     &\lesssim  \sup_{z''\in G, \,\langle \xi\rangle\sim 2^{  \frac{\ell}{n}  }} \langle \xi\rangle^{\delta|\beta|}  2^{- \frac{\rho\ell|\beta|}{n} } \Vert \widehat{\mathcal{M}}(\xi)^{-\delta|\beta|}(\partial_{X}^{(\beta)}{\sigma})(z'',\xi) \Vert_{\textnormal{op}}\\
      &\asymp  \sup_{z''\in G} 2^{ \frac{\delta\ell|\beta|}{n} }  2^{- \frac{\rho\ell|\beta|}{n} } \Vert \widehat{\mathcal{M}}(\xi)^{-\delta|\beta|}(\partial_{X}^{(\beta)}{\sigma})(z'',\xi) \Vert_{\textnormal{op}}\\
      &\leqslant   \sup_{z''\in G} 2^{(\delta\ell-\rho\ell)|\beta|/n} \Vert \widehat{\mathcal{M}}(\xi)^{-\delta|\beta|}(\partial_{X}^{(\beta)}{\sigma})(z'',\xi) \Vert_{\textnormal{op}}<\infty,
\end{align*}where we have used that $\delta-\rho\leqslant 0.$ A similar argument applied for $|\beta|=0,$ allows us to deduce that
\[ 
    \Vert (\textnormal{Op}(\sigma_{j,\ell})f)( x')\Vert_{L^2(B(e_G,\,\varepsilon_0),dx')} \lesssim \sum_{|\beta|\leqslant [\frac{n}{2}]+1} \sup_{(z'',[\xi])}\Vert \widehat{\mathcal{M}}(\xi)^{-\delta|\beta|}(\partial_{X}^{(\beta)}{\sigma})(z'',\xi) \Vert_{\textnormal{op}}<\infty.
\]So, we have proved that
\begin{align*}
  & I:=\sup_{\ell\in \mathbb{N}}\Vert T_{\ell} \Vert_{\mathscr{B}(L^2(G))}\leqslant  \sup_{\ell\in \mathbb{N}}\sum_{j=1}^{N_0}\Vert \textnormal{Op}(\sigma_{j,\ell} )\Vert_{\mathscr{B}(L^2(G))}\leqslant \\
  &\lesssim \sum_{|\beta|\leqslant [\frac{n}{2}]+1} \sup_{(z'',[\xi])}\Vert \widehat{\mathcal{M}}(\xi)^{-\delta|\beta|}(\partial_{X}^{(\beta)}{\sigma})(z'',\xi) \Vert_{\textnormal{op}}<\infty.
\end{align*} In order to finish the proof, we need to show that $II$ and $III$ are finite. To do so, we will use the bilinear estimate in Lemma \ref{bilinearestimate}. Let $k_{\ell}$ be the convolution kernel of $T_\ell$ and let us denote by $K_{\ell,\ell'},$ the integral kernel of $T^{*}_\ell T_{\ell'}.$ From the Schur Lemma, we can estimate
\[ 
    \Vert T^{*}_\ell T_{\ell'} \Vert_{\mathscr{B}(L^2(G))}\lesssim \sup_{(x,y)\in G\times G}|K_{\ell,\ell'}(x,y)|.
\]  Let us fix $N\in \mathbb{N},$ and let  $s_{0}>\frac{Q}{4}.$ Note that
\[ 
    K_{\ell,\ell'}(x,y)= \int\limits_{G}\overline{k}_{\ell,z}(x^{-1}z)k_{\ell',z_1}(y^{-1}z_2)dz.
\]
So, from the subelliptic Sobolev embedding theorem (see Remark \ref{sset}), and the  continuous inclusion $H^{s}(G)\subset H^{s,\mathcal{L}}(G)$ for $s\geqslant   0,$ we have,
\begin{align*}
   & |K_{\ell,\ell'}(x,y)|\\
   &=\left| \int\limits_{G}(1+\mathcal{L})^{N}_{z_1=z}(1+\mathcal{L})^{-N}_{z_2=z}[\overline{k}_{\ell,z_1}(x^{-1}z_2)k_{\ell',z_1}(y^{-1}z_2)]dz\right|\\
   &\leqslant \int\limits_{G}\sup_{z_1\in G}\left|(1+\mathcal{L})^{N}_{z_1=z}(1+\mathcal{L})^{-N}_{z_2=z}[\overline{k}_{\ell,z_1}(x^{-1}z_2)k_{\ell',z_1}(y^{-1}z_2)]\right|dz_2\\
   &\lesssim\int\limits_{G}\Vert(1+\mathcal{L})^{N+s_0}_{z_1=z}(1+\mathcal{L})^{-N}_{z_2=z}[\overline{k}_{\ell,z_1}(x^{-1}z_2)k_{\ell',z_1}(y^{-1}z_2)]\Vert_{L^2(G,dz_1)} dz_2\\
   &\leqslant\left(\int\limits_{G}\Vert(1+\mathcal{L})^{N+s_0}_{z_1=z}(1+\mathcal{L})^{-N}_{z_2=z}[\overline{k}_{\ell,z_1}(x^{-1}z_2)k_{\ell',z_1}(y^{-1}z_2)]\Vert^2_{L^2(G,dz_1)} dz_2\right)^{\frac{1}{2}}\\
   &=\left(\int\limits_{G}\Vert(1+\mathcal{L})^{-N}_{z_2=z}[\overline{k}_{\ell,z_1}(x^{-1}z_2)k_{\ell',z_1}(y^{-1}z_2)]\Vert^2_{H^{2N+2s_0,\mathcal{L}}(G,dz_1)} dz_2\right)^{\frac{1}{2}}\\
   &\lesssim \left(\int\limits_{G}\Vert(1+\mathcal{L})^{-N}_{z_2=z}[\overline{k}_{\ell,z_1}(x^{-1}z_2)k_{\ell',z_1}(y^{-1}z_2)]\Vert^2_{H^{2N+2s_0}(G,dz_1)} dz_2\right)^{\frac{1}{2}}.
\end{align*}Taking into account the following equivalence of norms
\begin{align*}
   & \Vert(1+\mathcal{L})^{-N}_{z_2=z}[\overline{k}_{\ell,z_1}(x^{-1}z_2)k_{\ell',z_1}(y^{-1}z_2)]\Vert^2_{H^{2N+2s_0}(G,dz_1)}\\
    &\asymp \sum_{|\alpha_1|+|\alpha_2|\leqslant 2(N+s_0)} \Vert(1+\mathcal{L})^{-N}_{z_2=z}[\partial_{Z_1}^{(\alpha_1)}\overline{k}_{\ell,z_1}(x^{-1}z_2)\partial_{Z_1}^{(\alpha_2)}k_{\ell',z_1}(y^{-1}z_2)]\Vert^2_{L^2(G,dz_1)},
\end{align*}from Lemma \ref{bilinearestimate}, and by using that $k_{\ell',z_1}\in H_{\lambda_{[\xi']}}(G),$ $k_{\ell,z_1}\in H_{\lambda_{[\xi]}}(G),$  (see Lemma \ref{bilinearestimate} for this notation) with $\lambda_{[\xi']}\sim\langle \xi' \rangle^2\sim 2^{ \frac{2\ell'}{n} }$ and   $\lambda_{[\xi]}\sim\langle \xi \rangle^2\sim  2^{  \frac{2\ell}{n}  },$ we estimate for $s>Q/2,$
\begin{align*}
  & \Vert(1+\mathcal{L})^{-N}_{z_2=z}[\partial_{Z_1}^{(\alpha_1)}\overline{k}_{\ell,z_1}(x^{-1}z_2)\partial_{Z_1}^{(\alpha_2)}k_{\ell',z_1}(y^{-1}z_2)]\Vert_{L^2(G,dz_1)}\\
   &\lesssim(1+|2^{ \frac{2\ell}{n} }-2^{ \frac{2\ell'}{n} }|)^{\frac{1}{\kappa}(-N+\frac{s}{2}+\frac{Q}{8})}\Vert \partial_{Z_1}^{(\alpha_2)}k_{\ell',z_1}(y^{-1}z_2) \Vert_{L^2(G)}\Vert \partial_{Z_1}^{(\alpha_1)}\overline{k}_{\ell,z_1}(x^{-1}z_2) \Vert_{L^2(G)}\\
   &\lesssim 2^{-(\max(\ell,\ell')\frac{1}{n}\cdot \frac{2}{\kappa}(N-\frac{s}{2}- \frac{Q}{8})}\Vert \partial_{Z_1}^{(\alpha_2)}k_{\ell',z_1}(y^{-1}z_2) \Vert_{L^2(G)}\Vert \partial_{Z_1}^{(\alpha_1)}\overline{k}_{\ell,z_1}(x^{-1}z_2) \Vert_{L^2(G)}\\
   &\lesssim 2^{-\frac{1}{\kappa}(\max(\ell,\ell')\frac{1}{n}(2N-s-\frac{Q}{4})}\Vert \partial_{Z_1}^{(\alpha_2)}k_{\ell',z_1}(y^{-1}z_2) \Vert_{L^2(G)}\Vert \partial_{Z_1}^{(\alpha_1)}\overline{k}_{\ell,z_1}(x^{-1}z_2) \Vert_{L^2(G)}.
\end{align*}From \eqref{inequality}, for $s>Q/2$ we have
\begin{align*}
    &\Vert \partial_{Z_1}^{(\alpha_1)}\overline{k}_{\ell,z_1}(x^{-1}z_2) \Vert_{L^2(G)}\lesssim_{s}\sup_{[\xi]\in \widehat{G}}\Vert \widehat{\mathcal{M}}(\xi)^s\partial_{Z_1}^{(\alpha_1)}\sigma_{\ell}(z_1,\xi)\Vert_{\textnormal{op}}\\
    &\asymp \sup_{\langle \xi\rangle\sim 2^{  \frac{\ell}{n}  }}\Vert \widehat{\mathcal{M}}(\xi)^s\partial_{Z_1}^{(\alpha_1)}\sigma_{\ell}(z_1,\xi)\Vert_{\textnormal{op}}.
\end{align*}Note that for $\langle \xi\rangle\sim 2^{  \frac{\ell}{n}  },$ the following estimates hold
\begin{align*}
  \Vert \widehat{\mathcal{M}}(\xi)^s\partial_{Z_1}^{(\alpha_1)}\sigma_{\ell}(z_1,\xi)\Vert_{\textnormal{op}}&= \Vert \widehat{\mathcal{M}}(\xi)^{s+\delta|\alpha_1|}\widehat{\mathcal{M}}(\xi)^{-\delta|\alpha_1|}\partial_{Z_1}^{(\alpha_1)}\sigma_{\ell}(z_1,\xi)\Vert_{\textnormal{op}} \\
  &\leqslant \Vert \widehat{\mathcal{M}}(\xi)^{s+\delta|\alpha_1|}\Vert_{\textnormal{op}}\Vert \widehat{\mathcal{M}}(\xi)^{-\delta|\alpha_1|}\partial_{Z_1}^{(\alpha_1)}\sigma_{\ell}(z_1,\xi)\Vert_{\textnormal{op}} \\
  &= \Vert \widehat{\mathcal{M}}(\xi)^{s+\delta|\alpha_1|}\Vert_{\textnormal{op}}\Vert \widehat{\mathcal{M}}(\xi)^{-\delta|\alpha_1|}\partial_{Z_1}^{(\alpha_1)}\sigma(z_1,\xi)\eta_1(2^{-\frac{\ell}{n}}\langle\xi\rangle)\Vert_{\textnormal{op}}\\
  &\lesssim \sup_{\langle \xi\rangle\sim 2^{  \frac{\ell}{n}  }} \Vert \widehat{\mathcal{M}}(\xi)^{s+\delta|\alpha_1|}\Vert_{\textnormal{op}}\Vert \widehat{\mathcal{M}}(\xi)^{-\delta|\alpha_1|}\partial_{Z_1}^{(\alpha_1)}\sigma(z_1,\xi)\Vert_{\textnormal{op}}\\
  &=\sup_{\langle \xi\rangle\sim 2^{  \frac{\ell}{n}  }}\sup_{1\leqslant j\leqslant d_\xi} (1+\nu_{jj}(\xi)^2)^{\frac{s+\delta|\alpha_1|}{2}}\Vert \widehat{\mathcal{M}}(\xi)^{-\delta|\alpha_1|}\partial_{Z_1}^{(\alpha_1)}\sigma(z_1,\xi)\Vert_{\textnormal{op}}\\
  &\lesssim \sup_{\langle \xi\rangle\sim 2^{  \frac{\ell}{n}  }}\langle \xi\rangle^{{s+\delta|\alpha_1|}}\Vert \widehat{\mathcal{M}}(\xi)^{-\delta|\alpha_1|}\partial_{Z_1}^{(\alpha_1)}\sigma(z_1,\xi)\Vert_{\textnormal{op}},
\end{align*} and  we have 
\begin{align*}
   & \Vert \partial_{Z_1}^{(\alpha_1)}\overline{k}_{\ell,z_1}(x^{-1}z_2) \Vert_{L^2(G)}\lesssim_{s} \sup_{\langle \xi\rangle\sim 2^{  \frac{\ell}{n}  } } \langle \xi\rangle^{{s+\delta|\alpha_1|}}\Vert \widehat{\mathcal{M}}(\xi)^{-\delta|\alpha_1|}\partial_{Z_1}^{(\alpha_1)}\sigma(z_1,\xi)\Vert_{\textnormal{op}}\\
  &\asymp 2^{\ell(s+\delta|\alpha_1|)/n}  \sup_{[\xi]\in \widehat{G}}\Vert \widehat{\mathcal{M}}(\xi)^{-\delta|\alpha_1|}\partial_{Z_1}^{(\alpha_1)}\sigma(z_1,\xi)\Vert_{\textnormal{op}}<\infty.
\end{align*}Thus, we have obtained that
\[ 
    \Vert \partial_{Z_1}^{(\alpha_1)}\overline{k}_{\ell,z_1}(x^{-1}z_2) \Vert_{L^2(G)}\lesssim  2^{\ell(s+\delta|\alpha_1|)/n}.
\]In a similar way we can prove 
\[ 
    \Vert \partial_{Z_1}^{(\alpha_2)}k_{\ell',z_1}(y^{-1}z_2) \Vert_{L^2(G)}\lesssim  2^{\ell(s+\delta|\alpha_2|)/n}.
\]
Thus, the condition $|\alpha_1|+|\alpha_2|\leqslant 2(N+s_0)$ implies 
\[ 
  \Vert \partial_{Z_1}^{(\alpha_1)}\overline{k}_{\ell,z_1}(x^{-1}z_2) \Vert_{L^2(G)}\Vert \partial_{Z_1}^{(\alpha_2)}k_{\ell',z_1}(y^{-1}z_2) \Vert_{L^2(G)}\lesssim  2^{\max(\ell,\ell')(s+\delta(2N+2s_0))/n}. 
\]
Now, we have  the estimate
\begin{align*}
   & \Vert(1+\mathcal{L})^{-N}_{z_2=z}[\partial_{Z_1}^{(\alpha_1)}\overline{k}_{\ell,z_1}(x^{-1}z_2)\partial_{Z_1}^{(\alpha_2)}k_{\ell',z_1}(y^{-1}z_2)]\Vert_{L^2(G,dz_1)}\\
   &\lesssim 2^{-\frac{1}{\kappa}\max(\ell,\ell')\cdot \frac{1}{n}\cdot (2N-s-\frac{Q}{4})}\times 2^{\max(\ell,\ell')(s+\delta(2N+2s_0))/n}\\
&=2^{-\frac{1}{n\kappa}(\max(\ell,\ell'))(2N-s-\frac{Q}{4}-s\kappa-2N\delta\kappa -2s_0\delta \kappa   )   }\\
&=2^{-\frac{1}{n\kappa}(\max(\ell,\ell'))(2N(1-\delta\kappa)-s(\kappa+1)-\frac{Q}{4}-2\delta\kappa s_0)}.
\end{align*}
Thus, we have proved 
\begin{align*}
    &|K_{\ell,\ell'}(x,y)|\lesssim 2^{-\frac{1}{n\kappa}(\max(\ell,\ell'))(2N(1-\delta\kappa)-s(\kappa+1)-\frac{Q}{4}-2\delta\kappa s_0)}.
\end{align*}This shows that $II,III<\infty,$ for  $2N(1-\delta\kappa)-s(\kappa+1)-\frac{Q}{4}-2\delta\kappa s_0>0.$ So, we only need to choose 
\[ 
    N>\frac{ s(\kappa+1)+\frac{Q}{4}+2\delta\kappa s_0 }{2(1-\delta\kappa)},
\]where we have used that $\delta< 1/\kappa.$ Thus, the proof is complete, with the remaining case of $0\leq \delta<\rho\leq 1$ proved in Corollary \ref{L2''}.
 \end{proof}

\subsection{The formal adjoint of subelliptic operators}
In Theorems \ref{Adjoint} and  \ref{Subellipticcomposition}   we will show that the subelliptic classes introduced before are closed under compositions and adjoints. 
If $A:C^\infty(G)\rightarrow C^\infty(G)$ is a continuous operator, its formal adjoint is the operator $A^*,$ defined by
\[ 
 (Af,g)_{L^2(G)}=(f, A^*g)_{L^2(G)},\,\,f,g\in C^{\infty}(G).
\]Next, we study the formal adjoint of subelliptic pseudo-differential operators.

 \begin{theorem}\label{Adjoint}
Let $0\leqslant \delta<\rho\leqslant 1.$ If $A:C^\infty(G)\rightarrow C^\infty(G)$ is a continuous operator, $A\in \textnormal{Op}({S}^{m,\mathcal{L}}_{\rho,\delta}(G\times \widehat{G})),$ then $A^*\in \textnormal{Op}({S}^{m,\mathcal{L}}_{\rho,\delta}(G\times \widehat{G})).$  The  symbol of $A^*,$ $\widehat{A^{*}}(x,\xi)$ satisfies the asymptotic expansion,
 \begin{equation}\label{asyexp}
    \widehat{A^{*}}(x,\xi)\sim \sum_{|\alpha|= 0}^\infty \Delta_{\xi}^{\alpha} \partial_{X}^{(\alpha)} (\widehat{A}(x,\xi)^{*}).
 \end{equation} This means that, for every $N\in \mathbb{N},$ and all $\ell\in \mathbb{N},$
\[ 
   \Small{ \Delta_{\xi}^{\alpha_\ell}\partial_{X}^{(\beta)}\left(\widehat{A^{*}}(x,\xi)-\sum_{|\alpha|\leqslant N}\Delta_{\xi}^\alpha\partial_{X}^{(\alpha)} (\widehat{A}(x,\xi)^{*}) \right)\in {S}^{m-(\rho-\delta)(N+1)-\rho\ell+\delta|\beta|,\mathcal{L}}_{\rho,\delta}(G\times\widehat{G}) },
\] where $|\alpha_\ell|=\ell.$
 \end{theorem}
 \begin{proof} Let $\Delta_{\xi}^\alpha\equiv \Delta_{q^{\alpha}}$ be a difference operator of order $|\alpha|.$ We note that the formula \eqref{asyexp} was established in \cite[Theorem 10.7.8]{Ruz} so that we only need to prove the remainder estimate.   By following  \cite[p. 569]{Ruz}, the right kernel of $A^*,$ $k^{*}_{(\cdot)}(\cdot),$ satisfies the identity,
 \[ 
     k^{*}_{x}(v)=\overline{k_{xv^{-1}}(v^{-1})},\,\,x,v\in G,
 \] where $k_{(\cdot)}(\cdot)$ is the right-kernel associated to $A.$ Let us prove that for any multi-index $\beta,\beta_0,\alpha_0\in \mathbb{N}_0,$ there exists $N_0\in\mathbb{N}_0$ such that for any integer $N>N_0,$ we have the estimate
 \begin{equation}\label{estimateadjoint}
    \sup_{x\in G} \Vert \partial_{Y}^{(\beta)} \partial_{X}^{(\beta_0)}(q^{\alpha_0}(y))(k_{x}^{*}(y)-\sum_{|\alpha|<N}q^{\alpha(y)}\partial_{X}^{(\alpha)} k_{x}^{*}(y))\Vert_{L^1(G)}\leqslant C \Vert \sigma\Vert_{\ell',S^{m,\mathcal{L}}_{\rho,\delta}},
 \end{equation} for $\ell'$ large enough. Later, we will conclude the proof using this estimate. 
 In the notation of Lemma \ref{Taylorseries}, we have
 \[ 
  R_{x,N}^{k_{(\cdot)}^*(y)}(y^{-1})=k_{x}^{*}(y)-\sum_{|\alpha|<N}q^{\alpha}(y)\partial_{X}^{(\alpha)} k_{x}^{*}(y),   
 \]where 
 $
q^{\alpha}:=  q_{(1)}^{\alpha_1}\cdots q_{(n)}^{\alpha_n}.
 $
 By using the remainder estimates in Lemma \ref{Taylorseries} and the Leibniz rule, we have
 \begin{align*}
  &|\partial_{Y}^{(\beta)} \partial_{X}^{(\beta_0)}(q^{\alpha_0}(y))(k_{x}^{*}(y)-\sum_{|\alpha|<N}q^{\alpha}(y)\partial_{X}^{(\alpha)} k_{x}^{*}(y))|\\
  &=|\partial_{Y}^{(\beta)}( R_{x,N}^{\partial_{X}^{(\beta_0)}(q^{\alpha_0}(y)k_{(\cdot)}^*(y))}(y^{-1}) )|   \\
  &\lesssim \sum_{|\beta_1|+|\beta_2|=\beta}|\partial_{Y}^{(\beta_2)}R_{x,N}^{\partial_{X}^{(\beta_0)}\partial_{Y}^{(\beta_1)}(q^{\alpha_0}(y)k_{(\cdot)}^*(y))}(y^{-1}) )|\\
  &\lesssim \sum_{|\beta_1|+|\beta_2|=\beta}|y|^{(N-|\beta_2|)_{+}}\max_{|\alpha|\leq N}\Vert\partial_{X}^{(\alpha)}\partial_{X}^{(\beta_0)}\partial_{Y}^{(\beta_1)}(q^{\alpha_0}(y)k_{(\cdot)}^*(y))\Vert_{L^\infty(G)},
 \end{align*}where we have used the notation $(r)_{+}:=\max\{r,0\}.$ Now, by using the kernel estimates in Proposition \ref{CalderonZygmund}, and the fact that   $\partial_{X}^{(\alpha+\beta_0)}\partial_{Y}^{(\beta_1)}(q^{\alpha_0}(y)k_{(\cdot)}^*(y))$ is the right-convolution kernel of some symbol $\sigma$ with subellitic order equal to $m+\delta(|\beta_0|+|\alpha|)+|\beta_1|-\rho|\alpha_0|,$  we have that there exists $\ell'\in \mathbb{N},$ such that
 \begin{itemize}
    \item[(i)] if $s>0,$  
        \[ 
          \Vert\partial_{X}^{(\alpha+\beta_0)}\partial_{Y}^{(\beta_1)}(q^{\alpha_0}(y)k_{(\cdot)}^*(y))\Vert_{L^\infty(G)}\lesssim_{m}  \Vert \sigma\Vert_{\ell, S^{m,\mathcal{L}}_{\rho,\delta}}|y|^{-\frac{s}{\rho}}.
        \]
         \item[(ii)] If $s=0,$  
        \[ 
           \Vert\partial_{X}^{(\alpha+\beta_0)}\partial_{Y}^{(\beta_1)}(q^{\alpha_0}(y)k_{(\cdot)}^*(y))\Vert_{L^\infty(G)}\lesssim_{m} \Vert \sigma\Vert_{\ell, S^{m,\mathcal{L}}_{\rho,\delta}}|\log|y||.
        \]
        \item[(iii)] If $s<0,$   
        \[ 
   \Vert\partial_{X}^{(\alpha+\beta_0)}\partial_{Y}^{(\beta_1)}(q^{\alpha_0}(y)k_{(\cdot)}^*(y))\Vert_{L^\infty(G)}  \lesssim_{m} \Vert \sigma\Vert_{\ell, S^{m,\mathcal{L}}_{\rho,\delta}}.
        \]  
\end{itemize}with $s=Q+m+\delta(|\beta_0|+N)+|\beta_1|-\rho|\alpha_0|.$ Observe that
\begin{align*}
    s&=Q+m+\delta(|\beta_0|+|\alpha|)+|\beta_1|-\rho|\alpha_0|\\
   &=Q+m+\delta(|\beta_0|+|\alpha|)+|\beta_1|-\rho|\alpha_0|-\rho |\alpha|+\rho |\alpha|\\
   &=Q+m+\delta|\beta_0|+(\delta-\rho)|\alpha|+|\beta_1|-\rho|\alpha_0|+\rho |\alpha|=:s'+(\delta-\rho)|\alpha|.
\end{align*}Consequently,
\begin{align*}
    &\int\limits_{G} \Vert\partial_{X}^{(\alpha+\beta_0)}\partial_{Y}^{(\beta_1)}(q^{\alpha_0}(y)k_{(\cdot)}^*(y))\Vert_{L^\infty(G)}dy\lesssim_{m}  \Vert \sigma\Vert_{\ell, S^{m,\mathcal{L}}_{\rho,\delta}}\int\limits_{G}|y|^{-\frac{s}{\rho}}|y|^{(N-|\beta_2|)_+}dy\\
    & \lesssim  \Vert \sigma\Vert_{\ell, S^{m,\mathcal{L}}_{\rho,\delta}}\int\limits_{G}|y|^{-\frac{s'+(\delta-\rho)|\alpha|}{\rho}}dy
   <\infty,
 \end{align*}provided that $s'<n\rho+(\rho-\delta)|\alpha| .$ To assure this condition, observe that   $n\rho+(\rho-\delta)|\alpha|\leq n\rho+(\rho-\delta)N,$ so we only need to choose $N_0\in \mathbb{N}_0$ such that
 \begin{align*}
N\geq N_{0}> \frac{s'-n\rho}{\rho-\delta},
 \end{align*}where we have used that $\delta<\rho.$ So, we conclude the proof of the estimate \eqref{estimateadjoint}. Now, in order to conclude the proof of Theorem \ref{Adjoint}, let us define
 \begin{equation}
     \widehat{A}_{N}(x,\xi):=\widehat{A^{*}}(x,\xi)-\sum_{|\alpha|\leqslant N}\Delta_{\xi}^\alpha\partial_{X}^{(\alpha)} (\widehat{A}(x,\xi)^{*}).
 \end{equation} We need to prove that  $   \Delta_{\xi}^{\alpha_0}\partial_{X}^{(\beta)}\widehat{A}_{N}(x,\xi)\in {S}^{m-(\rho-\delta)(N+1)-\rho|\alpha_0|+\delta|\beta|,\mathcal{L}}_{\rho,\delta}(G\times\widehat{G}). $ Set $m':=m-(\rho-\delta)(N+1)-\rho|\alpha_0|+\delta|\beta|.$  Observe  that for any $M'>m',$ with $M'\equiv 0\textnormal{ mod }(2),$ we have
 \begin{align*}
    & \Vert \Delta_{\xi}^{\alpha_0}\partial_{X}^{(\beta)}\widehat{A}_{N}(x,\xi) \widehat{\mathcal{M}}(\xi)^{-m'} \Vert_{\textnormal{op}}=\Vert \Delta_{\xi}^{\alpha_0}\partial_{X}^{(\beta)}\widehat{A}_{N}(x,\xi) \widehat{\mathcal{M}}(\xi)^{M'}\widehat{\mathcal{M}}(\xi)^{m'-M'} \Vert_{\textnormal{op}}\\
     &\lesssim \Vert \Delta_{\xi}^{\alpha_0}\partial_{X}^{(\beta)}\widehat{A}_{N}(x,\xi)\widehat{\mathcal{M}}(\xi)^{M'} \Vert_{\textnormal{op}}\\
    & \lesssim \Vert (1+\mathcal{L}_y)^{\frac{M'}{2}} R_{x,N+1}^{\partial_{X}^{(\beta)}(q^{\alpha_0}(y)k_{(\cdot)}^*(y))}(y^{-1}) )\Vert_{L^1(G_y)}\\
     &\lesssim \sum_{1\leqslant i_1\leqslant i_2\leqslant\cdots \leqslant i_{k}\leqslant k\,,|\alpha|\leqslant M'}\Vert X^{\alpha_1}_{i_1,y}\cdots X^{\alpha_k}_{i_k,y} R_{x,N+1}^{\partial_{X}^{(\beta)}(q^{\alpha_0}(y)k_{(\cdot)}^*(y))}(y^{-1}) )\Vert_{L^1(G_y)}.
 \end{align*} By Lemma \ref{Taylorseries}, and using that $X_{\mathfrak{D}}=\{X_{1,\mathfrak{D}},\cdots ,X_{n,\mathfrak{D}}\}$ is a basis of the Lie algebra  $\mathfrak{g},$ we can express every vector field $X_{i_j,y}$ as a linear combination of elements in $X_{\mathfrak{D}},$ so that we can estimate
 \begin{align*}
  \Vert \Delta_{\xi}^{\alpha_0}\partial_{X}^{(\beta)}\widehat{A}_{N}(x,\xi) \widehat{\mathcal{M}}(\xi)^{-m'} \Vert_{\textnormal{op}}   &\lesssim \sum_{|\alpha|\leqslant M'}\Vert \partial_{Y}^{(\alpha)} R_{x,N+1}^{\partial_{X}^{(\beta)}(q^{\alpha_0}(y)k_{(\cdot)}^*(y))}(y^{-1}) )\Vert_{L^1(G_y)}\\
  &\leqslant  C \Vert \sigma\Vert_{\ell',S^{m,\mathcal{L}}_{\rho,\delta}},
 \end{align*} for $\ell'$ large enough, where in the last line we have used \eqref{estimateadjoint}. The proof is complete.
 \end{proof}

 \begin{corollary}\label{AdjointC}
Let $0\leqslant \delta<\rho\leqslant \kappa.$ If $A:C^\infty(G)\rightarrow C^\infty(G)$ is a continuous operator, $A\in \textnormal{Op}(\mathscr{S}^{m,\mathcal{L}}_{\rho,\delta}(G)),$ then $A^*\in \textnormal{Op}(\mathscr{S}^{m,\mathcal{L}}_{\rho,\delta}(G)).$  The  symbol of $A^*,$ $\widehat{A^{*}}(x,\xi)$ satisfies the asymptotic expansion,
 \[ 
    \widehat{A^{*}}(x,\xi)\sim \sum_{|\alpha|= 0}^\infty \Delta_{\xi}^{\alpha} \partial_{X}^{(\alpha)} (\widehat{A}(x,\xi)^{*}).
 \] This means that, for every $N\in \mathbb{N},$ and all $\ell\in \mathbb{N},$
\[ 
   \Small{ \Delta_{\xi}^{\alpha_\ell}\partial_{X}^{(\beta)}\left(\widehat{A^{*}}(x,\xi)-\sum_{|\alpha|\leqslant N}\Delta_{\xi}^\alpha\partial_{X}^{(\alpha)} (\widehat{A}(x,\xi)^{*}) \right)\in \mathscr{S}^{m-(\rho-\delta)(N+1)-\rho\ell+\delta|\beta|,\mathcal{L}}_{\rho,\delta}(G) },
\] where $|\alpha_\ell|=\ell.$
 \end{corollary}

 \subsection{Composition of subelliptic pseudo-differential operators}
Next we prove the stability of the $(\rho,\delta)$-classes  subelliptic classes by taking compositions.

\begin{theorem}\label{Subellipticcomposition} Let  $0\leqslant \delta<\rho\leqslant 1.$ If $A_{i}\in \textnormal{Op}({S}^{m_i,\mathcal{L}}_{\rho,\delta}(G\times \widehat{G})), $ $A_{i}:C^\infty(G)\rightarrow C^\infty(G), $ $i=1,2,$ then the composition operator $A:=A_{1}\circ A_{2}:C^\infty(G)\rightarrow C^\infty(G)$ belongs to the subelliptic class $\textnormal{Op}({S}^{m_1+m_2,\mathcal{L}}_{\rho,\delta}(G\times \widehat{G})).$ The symbol of $A,$ $\widehat{A}(x,\xi),$ satisfies the asymptotic expansion,
\begin{equation}\label{ASYMP2}
    \widehat{A}(x,\xi)\sim \sum_{|\alpha|= 0}^\infty(\Delta_{\xi}^\alpha\widehat{A}_{1}(x,\xi))(\partial_{X}^{(\alpha)} \widehat{A}_2(x,\xi)),
\end{equation}this means that, for every $N\in \mathbb{N},$ and all $\ell \in\mathbb{N},$
\begin{align*}
    &\Delta_{\xi}^{\alpha_\ell}\partial_{X}^{(\beta)}\left(\widehat{A}(x,\xi)-\sum_{|\alpha|\leqslant N}  ( \Delta_{\xi}^{\alpha} \widehat{A}_{1}(x,\xi))(\partial_{X}^{(\alpha)} \widehat{A}_2(x,\xi))  \right)\\
    &\hspace{2cm}\in {S}^{m_1+m_2-(\rho-\delta)(N+1)-\rho\ell+\delta|\beta|,\mathcal{L}}_{\rho,\delta}(G\times \widehat{G}),
\end{align*}for every  $\alpha_\ell \in \mathbb{N}_0^n$ with $|\alpha_\ell|=\ell.$
\end{theorem}
\begin{proof} Let $\Delta_{\xi}^\alpha\equiv \Delta_{q^{\alpha}}$ be a difference operator of order $|\alpha|.$ As in Theorem 10.7.8 in \cite[p. 569]{Ruz} we have the formula \ref{ASYMP2}, so we only need to prove the remainder estimate.
If $K_{A}\in C^\infty(G)\widehat{\otimes}_{\pi}\mathscr{D}'(G) $ is the Schwartz kernel of $A,$ we denote by $k_{x}(y):=K_{A}(x, xy^{-1})$ and $k_{x,{i}}$ the right-convolution kernels of $A$ and $A_{i}.$ Therefore,
\begin{align*}
    \widehat{A}(x,\xi) &=\int\limits_{G}k_{x}(y)\xi(y)^*dy\\
    &=\int\limits_{G}\int\limits_{G}k_{x,1}(x,z^{-1})^*\xi(z^{-1})^{*}k_{x,2}(xz,yz)\xi(yz)^{*}dzdy,
\end{align*}
where we have used the identity
\begin{align*}
    Af(x)&=\int\limits_{G}A_{2}f(xz)k_{x}(z^{-1})dz=\int\limits_{G}f(xy^{-1})\underbrace{\int\limits_{G} k_{xz,2}(yz)k_{x,1}(z^{-1})dz}_{:=k_{x}(y)}dy.
\end{align*} So, we have
\[ 
  k_{x}(y):=\int\limits_{G}k_{xz^{-1},2}(yz^{-1})k_{x,1}(z)dz.
  \]
Observe that,
\begin{align*}
    k_{x}(y)&-\sum_{|\alpha|<N}(\partial_{X}^{(\alpha)}k_{x,2})\ast (q^{\alpha}k_{x,1})(y)\\
    &=\int\limits_{G}(k_{2,xz^{-1}}(yz^{-1})-\sum_{|\alpha|<N}q^{\alpha}(z)\partial_{Z}^{(\alpha)}k_{2,x}(z^{-1}))k_{1,x}(z)dz\\
    &=\int\limits_{G}R_{x,N}^{k_{\cdot,2}(yz^{-1})}(z^{-1})k_{1,x}(z)dz.
\end{align*}So, by applying the group Fourier transform we obtain
\[ 
    \widehat{A}(x,\xi)-\sum_{|\alpha|<N}\Delta_{\xi}^\alpha \widehat{A}_1(x,\xi) (\partial_{X}^{(\alpha)}\widehat{A}_2(x,\xi))=\int\limits_G k_{x,1}(z)\xi(z)^{*}R_{x,N}^{\widehat{A}_{2}(\cdot,\xi)}(z^{-1})dz,
\]
where we have denoted $R_{x,N}^{\widehat{A}_{2}(\cdot,\xi)}=[R_{x,N}^{\widehat{A}_{2}(\cdot,\xi)_{ij}}]_{i,j=1}^{d_\xi}.$ Now, the central part in the proof is to show that there exists $N_{0}\in \mathbb{N}$ such that the following estimate
\begin{equation}\label{compest}
    \Vert  \partial_{X}^{(\beta_0)}\Delta_\xi^{\alpha_0}[\int\limits_G k_{x,1}(z)\xi(z)^{*}R_{x,N}^{\widehat{A}_{2}(\cdot,\xi)}(z^{-1})dz]\widehat{\mathcal{M}}(\xi)^{b_{1}} \Vert_{\textnormal{op}} \leqslant C\Vert \widehat{A}_{1}\Vert_{\ell',S^{m_1,\mathcal{L}}}\Vert \widehat{A}_{2}\Vert_{\ell',S^{m_2,\mathcal{L}}},
\end{equation}holds true for all $N\geqslant N_{0}$ and all $b_1>0.$ 
Indeed, if we assume \eqref{compest}, by defining
\begin{equation}
     \widehat{A}_{N}(x,\xi):= \widehat{A}(x,\xi)-\sum_{|\alpha|<N}\Delta_{\xi}^\alpha \widehat{A}_1(x,\xi) (\partial_{X}^{(\alpha)}\widehat{A}_2(x,\xi)),
\end{equation}and taking $b_{1}=-m_1-m_{2}+(\rho-\delta)N+\rho|\alpha_0|-\delta|\beta_0|\geq -m_1-m_{2}+(\rho-\delta)N_0+\rho|\alpha_0|-\delta|\beta_0|, $ the condition $\rho>\delta$ and the choice of  $N_0$ large enough implies that $b_1>0,$ and from \eqref{compest}  the statement of  Theorem \ref{Subellipticcomposition} follows. So, let us fix $b\in \mathbb{N},$ with $b\equiv 0\textnormal{  mod  }(2).$ By using the Leibniz rule we can write
\begin{align*}
& \partial_{X}^{(\beta_0)}\Delta_\xi^{\alpha_0}[\int\limits_G k_{x,1}(z)\xi(z)^{*}R_{x,N}^{\widehat{A}_{2}(\cdot,\xi)}(z^{-1})dz] \\
 &=\sum_{\substack{|\alpha_0|\leq |\alpha_1|+|\alpha_2|\leq 2|\alpha_0|\\ |\beta_{0,1}|+|\beta_{0,2}|=|\beta_0| }}\int\limits_G \partial_{X}^{(\beta_{0,1})}q^{\alpha_{1}}(z)k_{x,1}(z)\xi(z)^{*}R_{x,N}^{\Delta_\xi^{\alpha_{2}}\partial_{X}^{(\beta_{0,2})}\widehat{A}_{2}(\cdot,\xi)}(z^{-1})dz,
\end{align*}and from the identity $\xi(z)^{*}=\widehat{\mathcal{M}}^{-b}(\xi)(1+\mathcal{L}_{z})^{\frac{b}{2}}\xi(z)^{*},$ by writing $I=\{\alpha_{i},\beta_{0,i}:\alpha_1|+|\alpha_2|\leq 2|\alpha_0|,\, |\beta_{0,1}|+|\beta_{0,2}|,\,i=1,2\},$ we can estimate
\begin{align*}
 &\Vert  \partial_{X}^{(\beta_0)}\Delta_\xi^{\alpha_0}[\int\limits_G k_{x,1}(z)\xi(z)^{*}R_{x,N}^{\widehat{A}_{2}(\cdot,\xi)}(z^{-1})dz]\widehat{\mathcal{M}}(\xi)^{b_1} \Vert_{\textnormal{op}}\\
 &\lesssim \sum_{I}\Vert\int\limits_G \partial_{X}^{(\beta_{0,1})}q^{\alpha_{1}}(z)k_{x,1}(z)\widehat{\mathcal{M}}(\xi)^{-b}(1+\mathcal{L}_z)^{\frac{b}{2}}\xi(z)^{*}R_{x,N}^{\Delta_\xi^{\alpha_{2}}\partial_{X}^{(\beta_{0,2})}\widehat{A}_{2}(\cdot,\xi)}(z^{-1})dz\widehat{\mathcal{M}}(\xi)^{b_1} \Vert_{\textnormal{op}}\\
 &= \sum_{I}\Vert \widehat{\mathcal{M}}(\xi)^{-b}\int\limits_G \partial_{X}^{(\beta_{0,1})}q^{\alpha_{1}}(z)k_{x,1}(z)(1+\mathcal{L}_z)^{\frac{b}{2}}\xi(z)^{*}R_{x,N}^{\Delta_\xi^{\alpha_{2}}\partial_{X}^{(\beta_{0,2})}\widehat{A}_{2}(\cdot,\xi)}(z^{-1})dz\widehat{\mathcal{M}}(\xi)^{b_1} \Vert_{\textnormal{op}}\\
 &\lesssim  \sum_{I}\Vert \widehat{\mathcal{M}}(\xi)^{b_1-b} \int\limits_G \partial_{X}^{(\beta_{0,1})}q^{\alpha_{1}}(z)k_{x,1}(z)(1+\mathcal{L}_z)^{\frac{b}{2}}\xi(z)^{*}R_{x,N}^{\Delta_\xi^{\alpha_{2}}\partial_{X}^{(\beta_{0,2})}\widehat{A}_{2}(\cdot,\xi)}(z^{-1})dz \Vert_{\textnormal{op}}.
\end{align*}From the equality
\begin{align*}
    &\int\limits_G \partial_{X}^{(\beta_{0,1})}q^{\alpha_{1}}(z)k_{x,1}(z)(1+\mathcal{L}_z)^{\frac{b}{2}}\xi(z)^{*}R_{x,N}^{\Delta_\xi^{\alpha_{2}}\partial_{X}^{(\beta_{0,2})}\widehat{A}_{2}(\cdot,\xi)}(z^{-1})dz \\
    &=\int\limits_G[ (1+\mathcal{L}_z)^{\frac{b}{2}}\xi(z)^{*}] \partial_{X}^{(\beta_{0,1})}q^{\alpha_{1}}(z)k_{x,1}(z)R_{x,N}^{\Delta_\xi^{\alpha_{2}}\partial_{X}^{(\beta_{0,2})}\widehat{A}_{2}(\cdot,\xi)}(z^{-1})dz ,
\end{align*}we have
\begin{align*}
 &\Vert  \widehat{\mathcal{M}}(\xi)^{b_1-b}\int\limits_G \partial_{X}^{(\beta_{0,1})}q^{\alpha_{1}}(z)k_{x,1}(z)(1+\mathcal{L}_z)^{\frac{b}{2}}\xi(z)^{*}R_{x,N}^{\Delta_\xi^{\alpha_{2}}\partial_{X}^{(\beta_{0,2})}\widehat{A}_{2}(\cdot,\xi)}(z^{-1})dz   \Vert_{\textnormal{op}}\\
 &=\Vert \widehat{\mathcal{M}}(\xi)^{b_1-b} \int\limits_G \xi(z)^{*} (1+\mathcal{L}_z)^{\frac{b}{2}}[\partial_{X}^{(\beta_{0,1})}q^{\alpha_{1}}(z)k_{x,1}(z)R_{x,N}^{\Delta_\xi^{\alpha_{2}}\partial_{X}^{(\beta_{0,2})}\widehat{A}_{2}(\cdot,\xi)}(z^{-1})]dz  \Vert_{\textnormal{op}} \\
 &\leqslant  \int\limits_G \Vert  \widehat{\mathcal{M}}(\xi)^{b_1-b} (1+\mathcal{L}_z)^{\frac{b}{2}}[\partial_{X}^{(\beta_{0,1})}q^{\alpha_{1}}(z)k_{x,1}(z)R_{x,N}^{\Delta_\xi^{\alpha_{2}}\partial_{X}^{(\beta_{0,2})}\widehat{A}_{2}(\cdot,\xi)}(z^{-1})]\Vert_{\textnormal{op}}dz \\
 &\lesssim \sum_{J\,,|\beta_1|+|\beta_2| \leqslant b}\int\limits_G \Vert  \widehat{\mathcal{M}}(\xi)^{b_1-b} X_{z,J}^{\beta_1}[\partial_{X}^{(\beta_{0,1})}q^{\alpha_{1}}(z)k_{x,1}(z)]X_{z_{1}=z^{-1},J}^{\beta_2}R_{x,N}^{\Delta_\xi^{\alpha_{2}}\partial_{X}^{(\beta_{0,2})}\widehat{A}_{2}(\cdot,\xi)}]\Vert_{\textnormal{op}}dz,
\end{align*} where $J=\{(i_1, i_2,\cdots , i_{k}):1\leqslant i_1\leqslant i_2\leqslant\cdots \leqslant i_{k}\leqslant k\}.$ By  Lemma \ref{Taylorseries}, and using that $X_{\mathfrak{D}}=\{X_{1,\mathfrak{D}},\cdots ,X_{n,\mathfrak{D}}\}$ is a basis of the Lie algebra  $\mathfrak{g},$ as in the proof of the asymptotic expansion for the adjoint, we can express every vector field $X_{i_j}$ as a linear combination of elements in $X_{\mathfrak{D}}.$ So, we have
\begin{align*}
    &\sum_{J\,,|\beta_1|+|\beta_2| \leqslant b}\int\limits_G \Vert  \widehat{\mathcal{M}}(\xi)^{b_1-b} X_{z,J}^{\beta_1}[\partial_{X}^{(\beta_{0,1})}q^{\alpha_{1}}(z)k_{x,1}(z)]X_{z_{1}=z^{-1},J}^{\beta_2}R_{x,N}^{\Delta_\xi^{\alpha_{2}}\partial_{X}^{(\beta_{0,2})}\widehat{A}_{2}(\cdot,\xi)}]\Vert_{\textnormal{op}}dz\\
    &\lesssim \sum_{|\beta_1|+|\beta_2| \leqslant b}\int\limits_G \Vert  \widehat{\mathcal{M}}(\xi)^{b_1-b} \partial_{Z}^{(\beta_1)}[\partial_{X}^{(\beta_{0,1})}q^{\alpha_{1}}(z)k_{x,1}(z)]\partial_{Z}^{(\beta_2)}[R_{x,N}^{\Delta_\xi^{\alpha_{2}}\partial_{X}^{(\beta_{0,2})}\widehat{A}_{2}(\cdot,\xi)}]|_{z^{-1}}\Vert_{\textnormal{op}}dz.
\end{align*} Observe that from Lemma \ref{Taylorseries} we have
\begin{align*}
  & \Vert  \widehat{\mathcal{M}}(\xi)^{b_1-b}\partial_{Z}^{(\beta_2)}[R_{x,N}^{\Delta_\xi^{\alpha_{2}}\partial_{X}^{(\beta_{0,2})}\widehat{A}_{2}(\cdot,\xi)}]|_{z^{-1}}\Vert_{\textnormal{op}}\\
   &\lesssim |z|^{(N-|\beta_2|)_{+}} \sup_{x\in G,\,|\beta_2|\leq N}\Vert  \widehat{\mathcal{M}}(\xi)^{b_1-b} \partial_{X}^{(\beta_2)}\Delta_\xi^{\alpha_{2}}\partial_{X}^{(\beta_{0,2})}\widehat{A}_{2}(x,\xi) \Vert\\
   &\lesssim |z|^{(N-|\beta_2|)_{+}} \Vert  \widehat{\mathcal{M}}(\xi)^{b_1-b+m_2+\delta(N+|\beta_{0,2}|)-\rho|\alpha_2|}  \widehat{A}_{2} \Vert_{N+|\beta_{0,2}|+|\alpha|,\,S^{m_2,\mathcal{L}}_{\rho,\delta}} .
\end{align*} By Proposition \ref{CalderonZygmund} applied to $\sigma=\partial_{Z}^{(\beta_1)}[\partial_{X}^{(\beta_{0,1})}q^{\alpha_{1}}(z)k_{x,1}(z)]$,  there exists $\ell'\in \mathbb{N},$ such that
 \begin{itemize}
    \item[(i)] if $s>0,$  
        \[ 
          |\partial_{Z}^{(\beta_1)}[\partial_{X}^{(\beta_{0,1})}q^{\alpha_{1}}(z)k_{x,1}(z)]|\lesssim_{m_1}  \Vert \widehat{A}_2\Vert_{\ell', S^{m_1,\mathcal{L}}_{\rho,\delta}}|z|^{-\frac{s}{\rho}}.
        \]
         \item[(ii)] If $s=0,$  
        \[ 
          |\partial_{Z}^{(\beta_1)}[\partial_{X}^{(\beta_{0,1})}q^{\alpha_{1}}(z)k_{x,1}(z)]|\lesssim_{m_1} \Vert \widehat{A}_2\Vert_{\ell', S^{m_1,\mathcal{L}}_{\rho,\delta}}|\log|z||.
        \]
        \item[(iii)] If $s<0,$   
        \[ 
   |\partial_{Z}^{(\beta_1)}[\partial_{X}^{(\beta_{0,1})}q^{\alpha_{1}}(z)k_{x,1}(z)]|  \lesssim_{m_1} \Vert \widehat{A}_2\Vert_{\ell', S^{m_1,\mathcal{L}}_{\rho,\delta}}.
        \]  
\end{itemize}where $s=Q+m_1+\delta|\beta_{0,1}|+|\beta_1|-\rho|\alpha_1|.$ First, let us assume that $$  b>\max\{b_1+m_2+\delta(N+|\beta_{0,2}|-\rho|\alpha_2|)\}.$$
This choice of $b$ implies that $\Vert  \widehat{\mathcal{M}}(\xi)^{b_1-b+m_2+\delta(N+|\beta_{0,2}|)-\rho|\alpha_2|} \Vert_{\textnormal{op}}\leqslant 1.$ In particular, $$b>b_1+m_2+\delta(N+|\beta_{0}|)-\rho|\alpha_0|.$$  Finally, a similar analysis as in the final part of the proof of Theorem \ref{Adjoint},  and the hypothesis $\delta<\rho,$ can be used to guarantee the existence of $N_{0}$ such that for any $N\geqslant N_{0}$ the integral $\int\limits_{G}I_{s}(z)dz$ converges, where  $I_{s}(z):=|z|^{(N-|\beta_2|)_{+}-\frac{s}{\rho}}$ if $s>0,$  and $I_{s}(z):=|z|^{(N-|\beta_2|)_{+}}$ if $s<0.$ Thus, we end the proof.  
\end{proof}

\begin{corollary}\label{SubellipticcompositionC} Let  $0\leqslant \delta<\rho\leqslant \kappa.$ If $A_{i}\in \textnormal{Op}(\mathscr{S}^{m_i,\mathcal{L}}_{\rho,\delta}(G)), $ $A_{i}:C^\infty(G)\rightarrow C^\infty(G), $ $i=1,2,$ then the composition operator $A:=A_{1}\circ A_{2}:C^\infty(G)\rightarrow C^\infty(G)$ belongs to the subelliptic class $\textnormal{Op}(\mathscr{S}^{m_1+m_2,\mathcal{L}}_{\rho,\delta}(G)).$ The symbol of $A,$ $\widehat{A}(x,\xi),$ satisfies the asymptotic expansion,
\[ 
    \widehat{A}(x,\xi)\sim \sum_{|\alpha|= 0}^\infty(\Delta_{\xi}^\alpha\widehat{A}_{1}(x,\xi))(\partial_{X}^{(\alpha)} \widehat{A}_2(x,\xi)),
\]this means that, for every $N\in \mathbb{N},$ and all $\ell \in\mathbb{N},$
\begin{align*}
    &\Delta_{\xi}^{\alpha_\ell}\partial_{X}^{(\beta)}\left(\widehat{A}(x,\xi)-\sum_{|\alpha|\leqslant N}  ( \Delta_{\xi}^{\alpha} \widehat{A}_{1}(x,\xi))(\partial_{X}^{(\alpha)} \widehat{A}_2(x,\xi))  \right)\\
    &\hspace{2cm}\in \mathscr{S}^{m_1+m_2-(\rho-\delta)(N+1)-\rho\ell+\delta|\beta|,\mathcal{L}}_{\rho,\delta}(G),
\end{align*}for every  $\alpha_\ell \in \mathbb{N}_0^n$ with $|\alpha_\ell|=\ell.$
\end{corollary}

\begin{remark}\label{sgeq1}
 Now, we return to the last part of Example \ref{powesfirstpart} where we claimed that $\mathcal{M}_{s}\in \textnormal{Op}(\mathscr{S}^{s\kappa,\mathcal{L}}_{1,0}(G)),$ for all $s>0.$ For $1\leqslant s<\infty,$   if $N_{0}\in \mathbb{N}$ satisfies $1\leqslant s< N_0,$ then, from Example \ref{powesfirstpart}, $(1+\mathcal{L})^{\frac{s}{2N_0}}\in \textnormal{Op}(\mathscr{S}^{ \frac{s\kappa}{N_0} ,\mathcal{L}}_{1,0}(G)). $ From Theorem \ref{Subellipticcomposition}, we deduce that
 \[ 
     \mathcal{M}_{s}=\underbrace{(1+\mathcal{L})^{\frac{s}{2N_0}}\cdots (1+\mathcal{L})^{\frac{s}{2N_0}}  }_{N_0-times}\in \textnormal{Op}(\mathscr{S}^{ \overbrace{\frac{s\kappa}{N_0}+\cdots+\frac{s\kappa}{N_0} }^{N_0-times} ,\mathcal{L}}_{1,0}(G)).
 \]Thus, $\mathcal{M}_{s}\in \textnormal{Op}(\mathscr{S}^{s\kappa,\mathcal{L}}_{1,0}(G))=\textnormal{Op}({S}^{s,\mathcal{L}}_{\frac{1}{\kappa},0}(G\times \widehat{G})).$ However, in view of \eqref{realpowerssharp}, ones can improve this inclusion to the following one:  $\mathcal{M}_{s}\in \textnormal{Op}({S}^{s,\mathcal{L}}_{1,0}(G\times \widehat{G})).$
\end{remark}

\section{Boundedness of subelliptic operators with non-smooth symbols}\label{ps} 
In this section we extend the results in \cite{RuzhanskyWirth2015} for the Laplacian, to the subelliptic setting. 
For our further analysis  we will apply the following lemma (see Persson \cite{Person}):
\begin{lemma}\label{Thelemaofcompact}
 Let us assume that $\Omega$ is a locally compact topological space and $\mu$ a positive measure on $\Omega.$ If $A$ extends to a bounded operator on $L^r(\Omega,\mu)$ and to a compact operator on $L^q(\Omega,\mu),$ then $A$ extends to a compact operator on $L^p(\Omega,\mu)$ for all $p$ between $q$ and $r,$ (i.e such that $1/p=\theta/q+(1-\theta)/r$ for some $\theta\in (0,1)$).
\end{lemma}For $0< p<\infty,$ the subelliptic $L^p$-Sobolev space of order $s\in\mathbb{R},$ is defined by the family of distributions $f\in \mathscr{D}'(G)$ such that  
\[ 
     \Vert f \Vert_{ L^{p,\mathcal{L}}_{s}(G) }:=\Vert  \mathcal{M}_{s} f \Vert_{ L^{p}(G) }<\infty.
 \]
We will develop an $L^p$-multiplier theorem for invariant operators in the dilated subelliptic classes with a limited number of regularity conditions. In particular, we will impose  conditions of the type
\[ 
      \sup_{(x,[\xi])\in G\times \widehat{G} }\Vert (\partial_{X}^{(\beta)} \Delta_{\xi}^{\alpha} a(x,\xi) ) \widehat{ \mathcal{M}}(\xi)^{\frac{1}{\kappa}(\rho|\alpha|-\delta|\beta|+\varkappa(\kappa-\rho))}\Vert_{\textnormal{op}} <\infty,
   \]which from Theorem \ref{cor}, is equivalent to the following ones 
\begin{itemize}
   \item \[ 
       \sup_{(x,[\xi])\in G\times \widehat{G} }\Vert \widehat{ \mathcal{M}}(\xi)^{\frac{1}{\kappa}(\rho|\alpha|-\delta|\beta|+\varkappa(\kappa-\rho))}\partial_{X}^{(\beta)} \Delta_{\xi}^{\alpha}a(x,\xi)\Vert_{\textnormal{op}} <\infty.
   \] 
   \item For all $r\in \mathbb{R},$
    \[ 
      \sup_{(x,[\xi])\in G\times \widehat{G} }\Vert \widehat{ \mathcal{M}}(\xi)^{\frac{1}{\kappa}(\rho|\alpha|-\delta|\beta|+\varkappa(\kappa-\rho)-r)}\partial_{X}^{(\beta)} \Delta_{\xi}^{\alpha}a(x,\xi)\widehat{ \mathcal{M}}(\xi)^{\frac{r}{\kappa}}\Vert_{\textnormal{op}} <\infty.
   \]
   \item There exists $r_0\in \mathbb{R},$ such that
    \[ \
      \sup_{(x,[\xi])\in G\times \widehat{G} }\Vert \widehat{ \mathcal{M}}(\xi)^{\frac{1}{\kappa}(\rho|\alpha|-\delta|\beta|+\varkappa(\kappa-\rho)-r_0)}\partial_{X}^{(\beta)} \Delta_{\xi}^{\alpha}a(x,\xi)\widehat{ \mathcal{M}}(\xi)^{\frac{r_0}{\kappa}}\Vert_{\textnormal{op}} <\infty.
   \]
\end{itemize}
The next theorem will be proved using the difference operators $\mathbb{D}^\alpha_\xi$ in Remark \ref{remarkD}.

\begin{theorem}\label{LpQL}
Let us assume that $G$ is a compact Lie group of dimension $n=2d$ or $n=2d+1,$ and that $d$ is odd. Let  $a\in \Sigma(\widehat{G})$ be a  symbol satisfying
\[ 
    \Vert \mathbb{D}^\alpha_\xi a(\xi) \widehat{\mathcal{M}}(\xi)^{\frac{\rho|\alpha|+\varkappa(\kappa-\rho)}{\kappa}}\Vert_{\textnormal{op}}\leqslant C_\alpha,\,\,|\alpha|\leqslant \varkappa:=d+1.
\] Then $A=\textnormal{Op}(a)$ extends to an operator of weak type $(1,1).$ Moreover, if the dimension of the group is $\dim(G)=2d$ or $\dim(G)=2d+1,$ and  $d$ is even, the conclusion on $A$ is the same provided that 
\[ 
    \Vert \mathbb{D}^\alpha_\xi a(\xi) \widehat{\mathcal{M}}(\xi)^{\frac{\rho|\alpha|+\varkappa(\kappa-\rho)}{\kappa}}\Vert_{\textnormal{op}}\leqslant C_\alpha,\,\,|\alpha|\leqslant \varkappa:=d+2.
\]In both cases, if $\rho=\kappa,$ $A$ extends to a bounded operator on $L^p(G),$ while for $0\leqslant \rho< \kappa,$ $A$ extends to a compact operator on $L^p(G)$ for all $1<p<\infty.$

\end{theorem}
\begin{proof} Let us consider the difference operator associated to $q_{m}(x):=\rho(x)^{2m},$ $\Delta_{q_m}\in \textnormal{diff}^{2m}(\widehat{G}),$ where $\varkappa:=2m\in \mathbb{N}$, and $\rho(x)$ is as in \eqref{therhofucntion}.   Let us define  $\varepsilon>0$ by the equality
 $$n(1+\varepsilon)=4m=2\varkappa.$$ The main step of the proof is to observe that the condition \eqref{CoifDeGuz}, is equivalent to showing that
 \begin{equation}\label{toproof}
     \Vert \Delta_{q_m}(a(\xi)\widehat{\psi_r})\Vert_{\ell^2(\widehat{G})}\lesssim r^{\frac{\varepsilon}{2}}=r^{\frac{2m}{  n}-\frac{1}{2}},
 \end{equation} where we have used that $\varepsilon=\frac{4 m}{ n}-1.$ The Leibniz rule applied for every $[\xi]\in \widehat{G},$ gives
 \[ 
     \Delta_{q_m}(a(\xi)\widehat{\psi_r}(\xi))=\Delta_{q_m}(a(\xi))\widehat{\psi_r}(\xi)+a(\xi)\Delta_{q_m}(\widehat{\psi_r}(\xi))+\sum_{\ell=1}^{m-1}\sum_{j}(Q_{\ell,j}a(\xi))(\tilde{Q}_{\ell,j}\widehat{\psi_r}(\xi)),
 \] for some differences operators $Q_{\ell,j}\in \textnormal{diff}^{\ell}(\widehat{G})$ and $\tilde{Q}_{\ell,j}\in \textnormal{diff}^{2m-\ell}(\widehat{G}).$ Because for every $\nu\geqslant   0$, we have
 \begin{align*}
      &\Delta_{q_m}(a(\xi)\widehat{\psi_r}(\xi))=\Delta_{q_m}(a(\xi))\widehat{\mathcal{M}}(\xi)^{2m\rho+\nu}\widehat{\mathcal{M}}(\xi)^{-2m\rho-\nu}\widehat{\psi_r}(\xi)\\
      &+a(\xi)\widehat{\mathcal{M}}(\xi)^{\nu}\widehat{\mathcal{M}}(\xi)^{-\nu}\Delta_{q_m}(\widehat{\psi_r}(\xi))\\
      &+\sum_{\ell=1}^{2m-1}\sum_{j}(Q_{\ell,j}a(\xi))\widehat{\mathcal{M}}(\xi)^{ \rho\ell+\nu }\widehat{\mathcal{M}}(\xi)^{-\rho\ell-\nu}(\tilde{Q}_{\ell,j}\widehat{\psi_r}(\xi)),
 \end{align*} taking the norm inequalities with $\nu:=\varkappa(\kappa-\rho),$ we have

     \begin{align*}
      \Vert \Delta_{q_m}(a(\xi)\widehat{\psi_r}(\xi))\Vert_{\textnormal{HS}}&\leqslant \Vert \Delta_{q_m}(a(\xi))\widehat{\mathcal{M}}(\xi)^{2m\rho+\nu}\Vert_{\textnormal{op}} \Vert \widehat{\mathcal{M}}(\xi)^{-2m\rho-\nu} \widehat{\psi_r}(\xi)\Vert_{\textnormal{HS}}\\
      &+\Vert a(\xi)\widehat{\mathcal{M}}(\xi)^{\nu}\Vert_{\textnormal{op}} \Vert \widehat{\mathcal{M}}(\xi)^{-\nu} \Delta_{q_m}(\widehat{\psi_r}(\xi))\Vert_{\textnormal{HS}}\\
      &+\sum_{\ell=1}^{2m-1}\sum_{j}\Vert (Q_{\ell,j}a(\xi))\widehat{\mathcal{M}}(\xi)^{\rho\ell+\nu}\Vert_{\textnormal{op}} \Vert \widehat{\mathcal{M}}(\xi)^{-\rho\ell-\nu} (\tilde{Q}_{\ell,j}\widehat{\psi_r}(\xi))\Vert_{\textnormal{HS}}\\
      &=\Vert \Delta_{q_m}(a(\xi))\widehat{\mathcal{M}}(\xi)^{2m\rho+\nu}\Vert_{\textnormal{op}} \Vert \widehat{\mathcal{M}}(\xi)^{-2m\rho-\nu} \widehat{\psi_r}(\xi)\Vert_{\textnormal{HS}}\\
      &\,\,+\sum_{\ell=0}^{2m-1}\sum_{j}\Vert (Q_{\ell,j}a(\xi))\widehat{\mathcal{M}}(\xi)^{\rho\ell+\nu}\Vert_{\textnormal{op}} \Vert \widehat{\mathcal{M}}(\xi)^{-\rho\ell-\nu} (\tilde{Q}_{\ell,j}\widehat{\psi_r}(\xi))\Vert_{\textnormal{HS}}\\
      &\lesssim  \Vert \widehat{\mathcal{M}}(\xi)^{-2m\rho-\nu} \widehat{\psi_r}(\xi)\Vert_{\textnormal{HS}}+\sum_{\ell=0}^{2m-1}\sum_{j}\Vert\textnormal{II}_{\ell,j}(\xi)\Vert_{\textnormal{HS}},\\
      &=\sum_{\ell=0}^{2m}\sum_{j}\Vert\textnormal{II}_{\ell,j}(\xi)\Vert_{\textnormal{HS}}
 \end{align*} where $\textnormal{II}_{\ell,j}(\xi)=  \widehat{\mathcal{M}}(\xi)^{-\rho\ell-\nu} (\tilde{Q}_{\ell,j}\widehat{\psi_r}(\xi)),$  and  $\textnormal{II}_{2m,j}(\xi)=  \widehat{\mathcal{M}}(\xi)^{-2m\rho-\nu}  \widehat{\psi_r}(\xi),$ with $j=0$ being the unique possible index for $\ell=0,2m.$ 
 Now, we study the terms  $\textnormal{II}_{\ell,j}(\xi).$ Let $\tilde{q}_{2m-\ell,j}$ be the associated function to the difference operator $\tilde{Q}_{\ell,j}$ vanishing with order $2m-\ell$ at $e_{G}$. Then,
 \[ 
     \Vert \tilde{q}_{2m-\ell,j}\psi_r \Vert_{H^{-\rho\ell-\nu,\mathcal{L}}(G)}=\left(\sum_{[\xi]\in \widehat{G}}d_{\xi}\Vert \widehat{\mathcal{M}}(\xi)^{-\rho\ell-\nu} (\tilde{Q}_{\ell,j}\widehat{\psi_r}(\xi)) \Vert_{\textnormal{HS}}^{2} \right)^{\frac{1}{2}}.
 \]
Let us fix $\nu=\varkappa(\kappa-\rho).$ By the embedding $H^{-\frac{s}{\kappa}}(G)\hookrightarrow H^{-s,\mathcal{L}}(G),$ (see  Proposition 3.1 of \cite{GarettoRuzhansky2015}), we have the following inequality,
 \[ 
     \Vert \tilde{q}_{2m-\ell,j}\psi_r \Vert_{H^{-\rho\ell-\nu,\mathcal{L}}}\lesssim \Vert \tilde{q}_{2m-\ell,j}\psi_r \Vert_{H^{-\frac{\rho\ell+\nu}{\kappa}}}\lesssim r^{\frac{2m-\ell+\frac{\rho\ell+\nu}{\kappa}}{n}-\frac{1}{2}}
 \]where we have used Lemma \ref{LemmaRuzhWirth}, with $\tilde{\ell}=2m-\ell$ and $s= \frac{\rho\ell+\nu}{\kappa},$ with $0\leqslant  \frac{\rho\ell+\nu}{\kappa} \leqslant 1+\frac{n}{2}. $ Indeed, note that $\frac{\rho\ell+\nu}{\kappa}\leqslant \frac{\rho\varkappa+\nu}{\kappa},$ and  the condition $0\leqslant    \frac{\rho\kappa+\nu}{\kappa}\leqslant 1+\frac{n}{2}, $ is equivalent to prove that: $\varkappa\leqslant 1+\frac{n}{2},$ which  holds true for $\dim(G)=2d$ or $\dim(G)=2d+1,$ with $d$ odd. For $0<r<1,$ and $ \frac{1}{2}\varepsilon_\ell:=\frac{2m-\ell+\frac{\rho\ell+\nu}{\kappa}}{n}-\frac{1}{2},$ the condition $\nu=\varkappa(\kappa-\rho)$ assures that $\frac{1}{2}\varepsilon\leqslant \frac{1}{2}\varepsilon_\ell ,$ and consequently, $r^{\frac{1}{2}\varepsilon_\ell}\leqslant r^{\frac{1}{2}\varepsilon}.$ Hence, we conclude that
\begin{equation}\label{1}
 \sum_{\ell=0}^{2m-1}\sum_{j}\Vert\textnormal{II}_{\ell,j}\Vert_{\ell^2(\widehat{G})}= \sum_{\ell=0}^{2m-1}\sum_{j}\Vert \tilde{q}_{2m-\ell,j}\psi_r  \Vert_{H^{-\frac{\ell}{\kappa},\mathcal{L}}}\lesssim r^\frac{\varepsilon}{2}.
\end{equation} 
Now, if $r\geqslant   1,$ we observe that
\begin{align*}
  \Vert\tilde{q}_{2m-\ell,j}\psi_{r} \Vert_{H^{-\frac{\ell}{\kappa},\mathcal{L}}}&\lesssim   \Vert\tilde{q}_{2m-\ell,j}\psi_{r} \Vert_{H^{-\frac{\ell}{\kappa^2}}}\lesssim \Vert\tilde{q}_{2m-\ell,j}\psi_{r} \Vert_{L^2(G)}\\
  &\leqslant  \Vert \tilde{q}_{2m-\ell,j}\Vert_{L^\infty(G)}\Vert \psi_{r} \Vert_{L^2(G)}. 
\end{align*} Because  \begin{align*}\Vert \psi_{r}\Vert_{L^2(G)}&=\Vert \phi_{r}-\phi_{\frac{r}{2}}\Vert_{L^2(G)}\leqslant \Vert \phi_{r}\Vert_{L^2(G)}+\Vert \phi_{\frac{r}{2}}\Vert_{L^2(G)}\\
&\lesssim  r^{-1/2}+(\frac{r}{2})^{-1}\leqslant 3r^{-\frac{1}{2}},
\end{align*} (for instance, see \cite[p. 140]{CoifmandeGuzman} or Lemma 3.3 of \cite[p. 629]{RuzhanskyWirth2015}), the condition $r\geqslant   1,$ implies
\[ 
      \Vert\tilde{q}_{2m-\ell,j}\psi_{r} \Vert_{H^{-\frac{\ell}{\kappa},\mathcal{L}}}\lesssim  3r^{-\frac{1}{2}}\Vert \tilde{q}_{2m-\ell,j}\Vert_{L^\infty(G)}\lesssim  \Vert \tilde{q}_{2m-\ell,j}\Vert_{L^\infty(G)}\lesssim r^{\frac{\varepsilon}{2}},
\] which proves \eqref{toproof} for $r\geqslant   1.$
Now, we end the proof of the weak (1,1) inequality  by observing (according to the proof of Theorem 2.1 of \cite{RuzhanskyWirth2015}) that the difference operator $\Delta_{q_m}$ is a linear combination of operators of the form $\mathbb{D}_\xi^\gamma,$ with $|\gamma|=2m:=\varkappa$. The proof of the weak (1,1) inequality  for $\dim(G)=2d$ or  $\dim(G)=2d+1,$ $d$ even, can be analysed in a similar way. Let us observe that in view of Lemma \ref{lemadecaying1} we have the estimate
\[ 
    \Vert a(\xi)\Vert_{\textnormal{op}} \lesssim (1+\nu_{ii}(\xi)^2)^{-\frac{\varkappa(\kappa-\rho)}{2\kappa}},
\] for some $1\leqslant i\leqslant d_\xi.$ So, if $\rho=\kappa,$ $A$ extends to a bounded operator on $L^2(G),$ while for $0\leqslant \rho<\kappa,$ we deduce that $A$ extends to a compact operator on $L^2(G)$\footnote{ Indeed, let us observe that the fact that $\lim_{\langle \xi\rangle\rightarrow \infty }\Vert a(\xi)\Vert_{\textnormal{op}}\lesssim \lim_{\langle \xi\rangle\rightarrow \infty }\max_{1\leqslant i\leqslant d_\xi }(1+\nu_{ii}(\xi)^2)^{-\frac{\varkappa(\kappa-\rho)}{2\kappa}}\leqslant \lim_{\langle \xi\rangle\rightarrow \infty }\langle \xi\rangle^{ -\frac{\varkappa(\kappa-\rho)}{2\kappa^2} }=0, $ implies the existence of a compact extension of $A$ on $L^2(G)$ (see \cite{ARLG}).}. If $\rho=\kappa, $ in view of the Marcinkiewicz interpolation theorem and the duality argument we have that $A$ extends to a bounded operator on $L^p(G).$ However, for $0\leqslant \rho<\kappa,$ the compactness of $A$ on $L^2(G),$ the weak (1,1) type of $A,$ the Marcinkiewicz interpolation theorem, and the compactness interpolation theorem (Theorem \ref{Thelemaofcompact}) allow us to conclude that $A$ extends to a compact operator on $L^p(G)$ for all $1<p<\infty.$
\end{proof}

The argument  used in the proof of Theorem \ref{LpQnoninvariant} via the Sobolev embedding theorem can be easily adapted to prove the following result for non-invariant operators, but, in the subelliptic context.

\begin{theorem}\label{SubellipticLpestimate}
Let us assume that $G$ is a compact Lie group of dimension $n=2d$ or $n=2d+1,$ with $d$ odd. Let $0\leqslant \rho\leqslant \kappa,$ and $a\in \Sigma(G\times \widehat{G})$ be a symbol satisfying
\[ 
    \Vert (\partial_X^{(\beta)}\mathbb{D}^\alpha_\xi a(x,\xi)) \widehat{\mathcal{M}}(\xi)^{\frac{\rho|\alpha|+\varkappa(\kappa-\rho)}{\kappa}}\Vert_{\textnormal{op}}\leqslant C_\alpha,\,\,|\alpha|\leqslant \varkappa:=d+1,\,|\beta|\leqslant \left[\frac{n}{p}\right]+1.
\] Then $A=\textnormal{Op}(a)$ extends to a bounded operator  on $L^p(G)$ for all $1<p<\infty$ for $\rho=\kappa,$ and to a compact linear operator  on $L^p(G)$ for all $1<p<\infty$ when $0\leqslant \rho<\kappa.$   Moreover, if the dimension of the group is $\dim(G)=2d$ or $\dim(G)=2d+1,$ with $d$ even, the conclusion on $A$ is the same provided that \[ 
    \Vert(\partial_X^{(\beta)} \mathbb{D}^\alpha_\xi a(x,\xi)) \widehat{\mathcal{M}}(\xi)^{\frac{\rho|\alpha|+\varkappa(\kappa-\rho)}{\kappa}}\Vert_{\textnormal{op}}\leqslant C_\alpha,\,\,|\alpha|\leqslant \varkappa:=d+2,\,\,|\beta|\leqslant \left[\frac{n}{p}\right]+1.
\]
\end{theorem}
\begin{proof}
Now, for every $z\in G,$ let us define the family of invariant operators $\{A_z\}_{z\in G},$ by

\[ 
    A_zf(x)=\sum_{[\xi]\in \widehat{G}}d_\xi\textnormal{\textbf{Tr}}[\xi(x)a(z,\xi)\widehat{f}(\xi)],\,\,f\in C^\infty(G).
\] By the identity $A_xf(x)=Af(x),$ $x\in G,$  the Sobolev embedding Theorem gives
\[ 
    \sup_{z\in G}|A_zf(x)|\lesssim \sum_{|\beta|\leqslant [\frac{n}{p}]+1}\Vert\partial_{Z}^\beta A_zf(\cdot)\Vert_{L^p(G)}=\sum_{|\beta|\leqslant [\frac{n}{p}]+1}\left(\int\limits_{G} \vert\partial_{Z}^\beta A_zf(x)\vert^p \, dz\right)^\frac{1}{p}.
\] Every operator $A_{z,\beta}:=\partial_{Z}^\beta A_z$ has an invariant symbol $a_{z,\beta}:=\partial_{Z}^\beta a(z,\cdot),$ and the estimates 
\[ 
    \Vert \widehat{\mathcal{M}}(\xi)^{\frac{1}{\kappa}(\varkappa(\kappa-\rho)+\rho|\alpha|  )} (\partial_{X}^{(\beta)}\mathbb{D}^{\alpha} a(x,\xi))\Vert_{\textnormal{op}}\leqslant C_\alpha, \,\,\,|\alpha|\leqslant \varkappa,\,\, |\beta|\leqslant [\frac{n}{p}]+1,
\] are equivalent to the following ones,
\[ 
   \Vert \widehat{\mathcal{M}}(\xi)^{\frac{1}{\kappa}(\varkappa(\kappa-\rho)+\rho|\alpha|)} \mathbb{D}^{\alpha} a_{z,\beta}(\xi)\Vert_{\textnormal{op}}\leqslant C_\alpha,\,\,\,|\alpha|\leqslant \varkappa,\,\, |\beta|\leqslant [\frac{n}{p}]+1.
\] Consequently the family of operators $\{A_{z,\alpha}\}_{z\in G,\,|\beta|\leqslant [\frac{n}{p}]+1},$ satisfies the conclusions in Theorem  \ref{LpQL}. Moreover, for every $|\beta|\leqslant [\frac{n}{p}]+1,$  the function $z\mapsto A_{z,\beta},$ is a continuous function from $G$ into $\mathscr{B}(L^p(G)),$ the set of bounded linear operators on $L^p(G),$ for all $1<p<\infty.$ The compactness of $G$ implies that 
\[ 
   \sup_{z\in G}\Vert A_{z,\beta}\Vert_{\mathscr{B}(L^p(G))}= \sup_{z\in G}\Vert\partial_{Z}^\beta A_z\Vert_{\mathscr{B}(L^p(G))}=\Vert\partial_{Z}^\beta A_{z_{0,\beta}}\Vert_{\mathscr{B}(L^p(G))}
\] for some $z_{0,\beta}\in G.$
Consequently, we can estimate the $L^p(G)$-norm of $Af,$ $f\in C^\infty(G),$ by
\begin{align*}
\Vert Af \Vert^p_{L^p(G)}&=\int\limits_{G}|A_xf(x)|^pdx\leqslant \int\limits_{G}\sup_{z\in G}|A_zf(x)|^pdx\\
&\lesssim \sum_{|\beta|\leqslant [\frac{n}{p}]+1} \int\limits_{G}\int\limits_{G} \vert\partial_{Z}^\beta A_zf(x)\vert^p \, dzdx\\
&= \sum_{|\beta|\leqslant [\frac{n}{p}]+1} \int\limits_{G}\int\limits_{G} \vert\partial_{Z}^\beta A_zf(x)\vert^p \, dxdz\\
&\leqslant \sum_{|\beta|\leqslant [\frac{n}{p}]+1}\sup_{z\in G}\Vert\partial_{Z}^\beta A_z\Vert^p_{\mathscr{B}(L^p(G))} \Vert f\Vert^p_{L^p(G)} \\
&=\sum_{|\beta|\leqslant [\frac{n}{p}]+1}\Vert\partial_{Z}^\beta A_{z_{0,\beta}}\Vert^p_{\mathscr{B}(L^p(G))} \Vert f\Vert^p_{L^p(G)} . 
\end{align*}  So, we have the boundedness of $A$ on $L^p(G),$ for all $1<p<\infty,$ and all $0\leqslant \rho\leqslant \kappa.$ However, for   $0\leqslant \rho< \kappa,$ $A$ extends to a compact operator on  $L^p(G).$ For the proof, we can take a sequence $\{f_j\}$ in $L^p(G),$ converging weakly to zero. We need to show that $Af_j$ converges to zero in the  $L^p(G)$-norm. The previous analysis gives us the inequality,
\[ 
    \Vert Af_j \Vert^p_{L^p(G)}\lesssim  \sum_{|\beta|\leqslant [\frac{n}{p}]+1} \int\limits_{G}\int\limits_{G} \vert\partial_{Z}^\beta A_zf_j(x)\vert^p \, dxdz=\sum_{|\beta|\leqslant [\frac{n}{p}]+1} \int\limits_{G} \Vert\partial_{Z}^\beta A_zf_j\Vert^p_{L^p(G)} dz.
\]
Because, every weakly convergent sequence is a bounded sequence,\footnote{ this is a known fact, indeed, it can be proved  by using the  Uniform Boundedness Principle.} we have that $ \sup_j \Vert f_j\Vert^p_{L^p(G)}<\infty,$ and  the estimate
\[ 
    \Vert\partial_{Z}^\beta A_zf_j\Vert\leqslant \Vert\partial_{Z}^\beta A_{z_{0,\beta}}\Vert_{\mathscr{B}(L^p(G))} \sup_j \Vert f_j\Vert_{L^p(G)}
\]
allows us to use the dominated convergence Theorem in order to conclude that 
\begin{align*}
\lim_{j\rightarrow \infty }   \Vert Af_j \Vert^p_{L^p(G)} &\lesssim    \sum_{|\beta|\leqslant [\frac{n}{p}]+1} \lim_{j\rightarrow \infty } \int\limits_{G} \Vert\partial_{Z}^\beta A_zf_j\Vert^p_{L^p(G)} dz\\
 &=\sum_{|\beta|\leqslant [\frac{n}{p}]+1}  \int\limits_{G} \lim_{j\rightarrow \infty }\Vert\partial_{Z}^\beta A_zf_j\Vert^p_{L^p(G)} dz.
\end{align*} Now, for every $z\in G,$ $ \partial_{Z}^\beta A_z$ is a Fourier multiplier satisfying the hypothesis in Theorem \ref{LpQL}, and consequently they admit  compact extensions on $L^p(G),$ for all $1<p<\infty,$ provided that $0\leqslant \rho<\kappa.$ This, and our assumption that $\{f_j\}$ in $L^p(G)$ converges weakly to zero, implies that $\lim_{j\rightarrow \infty }\Vert\partial_{Z}^\beta A_zf_j\Vert_{L^p(G)}=0.$ So, we conclude that $\lim_{j\rightarrow \infty }   \Vert Af_j \Vert_{L^p(G)}=0.$ The proof is complete.
\end{proof}
\begin{remark}
As we can see in Theorem \ref{LpQL},  the class of symbols with limited regularity $\mathscr{S}^{-\varkappa(\kappa-\rho),\varkappa,\mathcal{L}}_{\rho}(G)$ begets operators of weak type $(1,1).$ The same conclusion is valid for operators with smooth symbols in the class $\mathscr{S}^{-\varkappa(\kappa-\rho),\mathcal{L}}_{\rho}(G).$ Also, for $\rho=\kappa,$ we have
\[ 
  \textnormal{Op}(\mathscr{S}^{0,\mathcal{L}}_{\kappa,0}(G))\subset   \textnormal{Op}(\mathscr{S}^{0,\varkappa, [\frac{n}{p}]+1,\mathcal{L}}_{\kappa,0}(G))\subset \mathscr{B}(L^p(G)),\,\,1<p<\infty,
\] or equivalently,
\[ 
  \textnormal{Op}({S}^{0,\mathcal{L}}_{1,0}(G\times \widehat{G}))\subset   \textnormal{Op}({S}^{0,\varkappa, [\frac{n}{p}]+1,\mathcal{L}}_{1,0}(G))\subset \mathscr{B}(L^p(G)),\,\,1<p<\infty.
\]
 \end{remark}

 \begin{corollary}\label{SobL}
Let  $\varkappa$ be the smallest even integer larger than $\frac{n}{2},$ $n:=\dim(G).$ Let  $A:C^\infty(G)\rightarrow\mathscr{D}'(G)$ be a continuous operator such that its matrix symbol $\sigma_A$ satisfies 
\begin{equation}\label{estimatessub2}
    \Vert( \partial_{X}^{(\beta)}\Delta^\gamma \sigma_A(x,\xi) )\widehat{\mathcal{M}}(\xi)^{|\alpha|}   \Vert_{\textnormal{op}}\leqslant C_{\gamma},\,\,|\gamma|\leqslant \varkappa, \,|\beta|\leqslant  \left[\frac{n}{p}\right]+1.
\end{equation}
If $s\in \mathbb{R},$ then  $A$ extends to a bounded operator from $L^{p,\mathcal{L}}_{s}(G)$ into $L^{p,\mathcal{L}}_{s}(G)$ for  all $1<p<\infty.$
\end{corollary}
 \begin{proof}
  Note that
 \[ 
     \Vert Af \Vert_{ L^{p,\mathcal{L}}_{s}(G) }=\Vert \mathcal{M}_{s }A \mathcal{M}_{-s} \mathcal{M}_{s} f \Vert_{ L^{p}(G) }.
 \]By taking into account Remarks \ref{powesfirstpart} and  \ref{powessecondtpart}, we have that $\mathcal{M}_{s }\in \textnormal{Op}(\mathscr{S}^{s\kappa,\mathcal{L}}_{1,0}(G))\subset \textnormal{Op}(\mathscr{S}^{s\kappa,\mathcal{L}}_{\kappa,0}(G)) ,$ $ A\in \textnormal{Op}(\mathscr{S}^{0,\varkappa, [\frac{n}{p}]+1,\mathcal{L}}_{\kappa,0}(G)),$ and
 \[\mathcal{M}_{-s} \textnormal{Op}(\mathscr{S}^{-s\kappa,\mathcal{L}}_{1,0}(G))\subset \textnormal{Op}(\mathscr{S}^{-s\kappa,\mathcal{L}}_{\kappa,0}(G)) , \] and  we conclude that
 \[ 
    A_s:= \mathcal{M}_{s }A \mathcal{M}_{-s} \in \textnormal{Op}(\mathscr{S}^{0,\varkappa, [\frac{n}{p}]+1,\mathcal{L}}_{\kappa,0}(G)).
 \] Moreover, from Theorem \ref{SubellipticLpestimate} we deduce that $A_{s}$ extends to a bounded operator on $L^p(G)$ for all $1<p<\infty.$ Consequently we deduce the estimate
 \[ 
      \Vert Af \Vert_{ L^{p,\mathcal{L}}_{s}(G) }\leqslant \Vert A_s \Vert_{\mathscr{B}(L^p(G))}\Vert f\Vert_{L^{p,\mathcal{L}}_{s}(G)}.
 \] So, $A$ extends to a bounded operator from $L^{p,\mathcal{L}}_{s}(G)$ into $L^{p,\mathcal{L}}_{s}(G).$ Thus, we finish the proof.
 \end{proof}
 Corollary \ref{SobL} implies the following result on the boundedness of pseudo-differential operators on Besov spaces (we refer the reader to Appendix \ref{Finalsect} for the definition and  to \cite{CardonaRuzhansky2019I} for the interpolation properties of subelliptic  Besov spaces).
  \begin{corollary}\label{BesovL}
Let  $\varkappa$ be the smallest even integer larger than $\frac{n}{2},$ $n:=\dim(G).$ Let  $A:C^\infty(G)\rightarrow\mathscr{D}'(G)$ be a continuous operator such that its matrix symbol $\sigma_A$ satisfies 
\begin{equation}\label{estimatessub2Besov}
    \Vert( \partial_{X}^{(\beta)}\Delta^\gamma \sigma_A(x,\xi) )\widehat{\mathcal{M}}(\xi)^{|\gamma|}   \Vert_{\textnormal{op}}\leqslant C_{\gamma,\beta},\,\,|\gamma|\leqslant \varkappa, \,|\beta|\leqslant  \left[\frac{n}{p}\right]+1.
\end{equation}
If $s\in\mathbb{R},$ then  $A$ extends to a bounded operator from $B^{s,\mathcal{L}}_{p,q}(G)$ into $B^{s,\mathcal{L}}_{p,q}(G)$ for  all $1<p<\infty,$ and $0<q<\infty.$
\end{corollary}
 \begin{proof}
 We will present a proof based on the real interpolation analysis. Observe that  Corollary \ref{SobL}, shows that  $A$ extends to a bounded operator from $L^{p,\mathcal{L}}_{s}(G)$ into $L^{p,\mathcal{L}}_{s}(G)$ for every $1<p<\infty.$ In particular, if $1<p_0<p_1<\infty$ and $\theta\in (0,1)$ satisfies $1/p=\theta/p_0+(1-\theta)/p_1,$ then from the boundedness of the following bounded extensions of $A,$
 \[ 
     A:L^{p_0,\mathcal{L}}_{s}(G)\rightarrow L^{p_0,\mathcal{L}}_{s}(G),\,\,\,\,A:L^{p_1,\mathcal{L}}_{s}(G)\rightarrow L^{p_1,\mathcal{L}}_{s}(G)
 \] by the real interpolation, we deduce 
 \[ 
     A:(L^{p_0,\mathcal{L}}_{s}(G),L^{p_1,\mathcal{L}}_{s}(G))_{(\theta,q)}\rightarrow (L^{p_0,\mathcal{L}}_{s}(G),L^{p_1,\mathcal{L}}_{s}(G))_{(\theta,q)},\,\,\,\,0<q<\infty.
 \] Because $(L^{p_0,\mathcal{L}}_{r}(G),L^{p_1,\mathcal{L}}_{r}(G))_{(\theta,q)}=B^{r,\mathcal{L}}_{p,q}(G)$ for every $r\in \mathbb{R},$ (see \cite[Theorem 6.2]{CardonaRuzhansky2019I}) we conclude that  $A$ extends to a bounded operator from $B^{s,\mathcal{L}}_{p,q}(G)$ into $B^{s,\mathcal{L}}_{p,q}(G).$  The proof is complete.
 \end{proof}

\section{Boundedness of subelliptic pseudo-differential operators}\label{fs} 

 In this section we develop a Fefferman type $L^p$-estimate for subelliptic classes on compact Lie groups. In constrast with the previous subsection, we will use the calculus for subelliptic operators e.g. in the duality argument. First, we will prove an estimate of $L^\infty$-$BMO$ type and we will provide the $H^1$-$L^1$ estimate by the duality argument. Finally, by using the Fefferman-Stein interpolation theorem we will obtain the $L^p$-estimate for subelliptic operators. This gives an extension of the results in \cite{RuzhanskyDelgado2017} for the Laplacian, to the subelliptic setting (proving a parallel result to the main theorem in   \cite{RuzhanskyDelgadoCardona2019} in  the setting of graded Lie groups).
 
 \subsection{$L^\infty$-$BMO$ boundedness for subelliptic H\"ormander classes  }The boundedness of pseudo-differntial operators on $L^\infty,$ can be estimated in terms of the $L^1$-norm of the right convolution kernel as follows.
 \begin{remark}[$L^\infty(G)$-boundedness of pseudo-differential operators]\label{Linftyremark} Let $x\mapsto k_x,$ the right convolution kernel of $A,$ this means that for every smooth function $f$ on $G,$  \[ 
     Af(x)=(f\ast k_{x})(x),\,\,x\in G.
 \] If $\sigma(x,\xi)$ is  the matrix-valued symbol of $A,$ we have that $k_{x}=\mathscr{F}^{-1}_G{\sigma}(x,\cdot).$ If we assume the following condition
   \[ 
       \sup_{x\in G}\Vert\mathscr{F}_{G}^{-1}\sigma(x,\cdot) \Vert_{L^1(G)}=\sup_{x\in G}\Vert k_{x} \Vert_{L^1(G)}<\infty,
   \] then 
   \[ 
       |Af(x)|\leqslant \int\limits_{G}|k_x(y^{-1}x)||f(y)|dy\leqslant  \sup_{x\in G}\Vert k_{x} \Vert_{L^1(G)}\Vert f\Vert_{L^\infty(G)}.
   \]Consequently,
   $$\Vert A\Vert_{\mathscr{B}(L^\infty(G))}\leqslant \sup_{x\in G}\Vert k_{x} \Vert_{L^1(G)}. $$
   \end{remark}
   
   \begin{remark}
   For every smooth function $f\in C^\infty(0,\infty),$ we will define
   \begin{equation}\label{f(M)}
       f(\widehat{\mathcal{M}})(\xi)\equiv f(\widehat{\mathcal{M}}(\xi)):=\textnormal{diag}[f((1+\nu_{ii}^2)^{\frac{1}{2}})]_{1\leqslant i\leqslant d_\xi}\equiv(f((1+\nu_{ii}^2)^{\frac{1}{2}})\delta_{ij})_{1\leqslant i,j\leqslant d_\xi}
   \end{equation} for every $[\xi]\in \widehat{G}.$ According to the symbolic calculus developed in \cite{RuzhanskyWirth2014}, \eqref{f(M)} defines the matrix-symbol $\{\widehat{f(\mathcal{M})}(\xi)\}_{[\xi]\in \widehat{G}}$ of the operator $f(\mathcal{M})$ defined by the functional calculus. 
   \end{remark}
 \begin{lemma}\label{FundamentallemmaI}
Let $G$ be a compact Lie group and let us denote by $Q$ the Hausdorff
dimension of $G$ associated to the control distance associated to the sub-Laplacian $\mathcal{L}=\mathcal{L}_X,$ where  $X= \{X_{1},\cdots,X_k\} $ is a system of vector fields satisfying the H\"ormander condition of order $\kappa$. For  $0\leqslant \delta<\rho\leqslant 1,$ let $\sigma\in {S}^{-\frac{Q(1-\rho) }{2},\mathcal{L}}_{\rho,\delta}(G\times \widehat{G})$ be a symbol satisfying
\begin{equation}\label{spectralabsorbsion}
  \sigma(x,\xi)\psi_{j}(\widehat{\mathcal{M}}(\xi))=\sigma(x,\xi),
\end{equation} where $\psi_{j}(\lambda)=\psi_0(2^{-j}\lambda),$ for some test function $\psi_{0}\in C^\infty_0(0,\infty),$ and some fixed integer $j\in \mathbb{N}_0.$ Then $A=\textnormal{Op}(\sigma)$ extends to a bounded operator from $L^\infty(G)$ to $L^\infty(G),$ and for $\ell:=[\frac{Q}{2}]+1,$ we have
\[ 
    \Vert  A\Vert_{\mathscr{B}(L^\infty(G))}\lesssim \left(\sup_{(x,[\xi])\in G\times\widehat{G},|\alpha|\leqslant \ell} \Vert [\Delta^{\alpha}_\xi \sigma(x,\xi)]\widehat{\mathcal{M}}(\xi)^{\frac{Q(1-\rho)}{2}+\rho|\alpha|})  \Vert_{\textnormal{op}}     \right).
\]
with the positive constant $C$  independent of $j$ and $\sigma.$ 
\end{lemma}
\begin{proof}
Let us fix $j\in \mathbb{N}_0, $ and allow us to define $a:=1-\rho.$ Let $b=R^{a-1},$ $R=2^j.$ Let us consider  $\sigma\in {S}^{-\frac{Qa }{2},\mathcal{L}}_{\rho,\delta}(G\times \widehat{G}) =\mathscr{S}^{-\frac{Qa\kappa }{2},\mathcal{L}}_{\rho\kappa,\delta\kappa}(G).$   In view of Remark \ref{Linftyremark}, we only need to prove that  \[ 
    \sup_{x\in G}\Vert k_x \Vert_{L^1(G)}\leqslant C,
\] where $C$ is a positive constant independent of $R,$ and $k_x,$ as usually, is the right-convolution kernel of $A.$ Let us denote by $|\cdot|$ the seminorm induced by the Carnot Carath\'eodory distance on $G.$ First, let us split the $L^1(G)$-norm of $k_x$ as,
\[ 
    \int\limits_{G}|k_x(z)|dz= \int\limits_{|z|\leqslant b}|k_x(z)|dz+ \int\limits_{|z|> b}|k_x(z)|dz.
\] By using the H\"older inequality we estimate the first integral as follows,

\begin{align*}
    \int\limits_{|z|\leqslant b}|k_x(z)|dz&\leqslant \left(\int\limits_{|z|\leqslant b}dz\right)^{\frac{1}{2}}\left(\int\limits_{|z|\leqslant b}|k_x(z)|^2dz\right)^{\frac{1}{2}}\\
    &\asymp R^{\frac{Q(a-1)}{2}}\Vert k_{x}\Vert_{L^2(G)}
    \\
    &= 2^{-\frac{j(1-a)Q}{2}}\Vert k_{x}\Vert_{L^2(G)}.
\end{align*}The Plancherel theorem   and the definition of the right-convolution kernel: $k_{x}=\mathscr{F}_{G}^{-1}(\sigma(x,\cdot)),$ for every $x\in G,$ together with \eqref{spectralabsorbsion} imply

\begin{align*}
   \Vert k_{x}\Vert_{L^2(G)} &=\left( \sum\limits_{[\xi]\in \widehat{G}} d_\xi\Vert \sigma(x,\xi) \Vert_{\textnormal{HS}}^2 \right)^\frac{1}{2}\\&=\left( \sum\limits_{[\xi]\in \widehat{G}}d_\xi\Vert \sigma(x,\xi)\widehat{\mathcal{M}}(\xi)^{\frac{Qa}{2}}  \widehat{\mathcal{M}}(\xi)^{-\frac{Qa}{2}}\Vert_{\textnormal{HS}}^2 \right)^\frac{1}{2}\\
    &=\left( \sum\limits_{[\xi]\in \widehat{G}}d_\xi\Vert \sigma(x,\xi)\psi_{j}(\widehat{\mathcal{M}}(\xi))\widehat{\mathcal{M}}(\xi)^{ \frac{Qa}{2} }  \widehat{\mathcal{M}}(\xi)^{-\frac{Qa}{2} } \Vert_{\textnormal{HS}}^2 \right)^\frac{1}{2}\\
     &= \left( \sum\limits_{[\xi]\in \widehat{G}}d_\xi\Vert \sigma(x,\xi)\widehat{\mathcal{M}}(\xi)^{\frac{Qa}{2}}\psi_{j}(\widehat{\mathcal{M}}(\xi))\widehat{\mathcal{M}}(\xi)^{-\frac{Qa}{2}} \Vert_{\textnormal{HS}}^2 \right)^\frac{1}{2}\\
     &\leqslant \left((\sup_{[\xi]\in\widehat{G}}\Vert \sigma(x,\xi)\widehat{\mathcal{M}}(\xi)^{ \frac{Qa}{2} } \Vert_{\textnormal{op}}\right) \left( \sum\limits_{[\xi]\in \widehat{G}} d_\xi\Vert\psi_{j}(\widehat{\mathcal{M}}(\xi))\widehat{\mathcal{M}}(\xi)^{- \frac{Qa}{2}  } \Vert_{\textnormal{HS}}^2 \right)^\frac{1}{2}.
\end{align*}
Let us denote, for every $j\geqslant   1,$
$$
  I_j:=  \left( \sum\limits_{[\xi]\in \widehat{G}}d_\xi\Vert\psi_{j}(\widehat{\mathcal{M}}(\xi))\widehat{\mathcal{M}}(\xi)^{-\frac{Qa}{2}  } \Vert_{\textnormal{HS}}^2 \right)^\frac{1}{2}.
$$
So, we can estimate $I_j$ for every $j\geqslant   1.$ Indeed, the Weyl eigenvalue counting formula for the sub-Laplacian \eqref{weyl}, gives
\begin{align*}
    I_j^2&\leqslant  \sum\limits_{[\xi]\in \widehat{G}}d_\xi\Vert\psi_{j}(\widehat{\mathcal{M}}(\xi))\widehat{\mathcal{M}}(\xi)^{-\frac{Qa}{2}  } \Vert_{\textnormal{HS}}^2\\
    &\leqslant  \sum\limits_{[\xi]:2^{j-1}\leqslant (1+\nu_{kk}(\xi)^2)^{\frac{1}{2}}<2^{j+1}, \forall 1\leq k\leq d_\xi}d_\xi\sum_{i=1}^{d_\xi} (1+\nu_{ii}(\xi)^2)^{-\frac{Qa}{2}} \\
     &\lesssim  \sum\limits_{[\xi]:2^{j-1}\leqslant (1+\nu_{kk}(\xi)^2)^{\frac{1}{2}}<2^{j+1}, \forall 1\leq k\leq d_\xi } d_{\xi}^2 2^{-jQa} \lesssim 2^{jQ-jQa}.
\end{align*}
Consequently, 
\[ 
    I_j\lesssim 2^{\frac{jQ(1-a)}{2}}=R^{\frac{ Q(1-a)}{2}}.
\]
This analysis and the fact that $0<a\leqslant 1,$ allows us to deduce that
\begin{align*}
     \left(\int\limits_{|z|\leqslant b}dz\right)^{\frac{1}{2}}\left(\int\limits_{|z|\leqslant b}|k_x(z)|^2dz\right)^{\frac{1}{2}}\asymp 2^{\frac{jQ(a-1)}{2}}2^{\frac{ jQ(1-a)}{2}} =1,
\end{align*} and consequently we estimate
\[ 
    \int\limits_{|z|\leqslant b}|k_x(z)|dz=O(1).
\]
    On the other hand, if $\Delta_q$ is a difference operator associated to $q$ that vanishes at $e_G$ of order $\ell,$
      observe that
\begin{align*}
    \int\limits_{|z|> b}|k_x(z)|dz &\leqslant \left(\int\limits_{|z|> b}q(z)^{-2}dz\right)^{\frac{1}{2}}\left(\int\limits_{|z|> b}|q(z)k_x(z)|^2dz\right)^{\frac{1}{2}}\\
    &\leqslant \left(\int\limits_{|z|> b}|z|^{-2\ell}dz\right)^{\frac{1}{2}}\left(\int\limits_{|z|> b}|q(z)k_x(z)|^2dz\right)^{\frac{1}{2}}\\
    &\lesssim b^{\frac{Q-2\ell}{2}}\left(\sum\limits_{[\xi]\in \widehat{G}}d_\xi\Vert \Delta_q\sigma(x,\xi)\Vert^2_{\textnormal{HS}}\right)^{\frac{1}{2}}=b^{\frac{Q}{2}-\ell}\Vert \Delta_q \sigma(x,\xi)\Vert_{L^2(\widehat{G})}\\
    &=2^{j(a-1)(\frac{Q}{2}-\ell)}\Vert \Delta_q \sigma(x,\xi)\Vert_{L^2(\widehat{G})}.  
\end{align*}
We observe that the condition $\ell>\frac{Q}{2}$ is a necessary and suficient condition in order that $\int\limits_{|z|> b}|z|^{-2\ell}dz<\infty,$ which can be deduced from an argument using polar coordinates in order to estimate 
\[ 
     \left(\int\limits_{|z|> b}|z|^{-2\ell}dz\right)^{\frac{1}{2}}\sim b^{\frac{Q-2\ell}{2}}.
\]
In order to control the $L^2(\widehat{G})$-norm of $\Delta_q \sigma(x,\xi)$ we proceed  as follows, \begin{align*}
   & \Vert \Delta_q \sigma(x,\xi)\Vert_{L^2(\widehat{G})}^2 &\\
   &\leqslant \sum_{  [\xi]:2^{j-1}\leqslant (1+\nu_{kk}(\xi)^2)^{\frac{1}{2}}<2^{j+1}, \forall 1\leq k\leq d_\xi } d_\xi\Vert  \Delta_q\sigma(x,\xi)\widehat{\mathcal{M}}(\xi)^{\frac{Qa}{2} +(1-a)\ell} \widehat{\mathcal{M}}(\xi)^{-\frac{Qa}{2} -(1-a)\ell}\Vert_{\textnormal{HS}}^2\\
    &\leqslant \sum_{  [\xi]:2^{j-1}\leqslant (1+\nu_{kk}(\xi)^2)^{\frac{1}{2}}<2^{j+1}, \forall 1\leq k\leq d_\xi } d_\xi\Vert  \Delta_q\sigma(x,\xi)\widehat{\mathcal{M}}(\xi)^{\frac{Qa}{2} +(1-a)\ell}\Vert^2_{\textnormal{op}}\\
    &\times \Vert \widehat{\mathcal{M}}(\xi)^{-\frac{Qa}{2} -(1-a)\ell}\Vert_{\textnormal{HS}}^2\\ 
    &\lesssim \sum_{  [\xi]:2^{j-1}\leqslant (1+\nu_{kk}(\xi)^2)^{\frac{1}{2}}<2^{j+1}, \forall 1\leq k\leq d_\xi } d_\xi\sum_{i=1}^{d_\xi}(1+\nu_{ii}(\xi)^2)^{-(\frac{Qa}{2} +(1-a)\ell)}\\ 
     &\lesssim \sum_{  [\xi]:2^{j-1}\leqslant (1+\nu_{kk}(\xi)^2)^{\frac{1}{2}}<2^{j+1}, \forall 1\leq k\leq d_\xi } d_\xi^2 2^{-2j(\frac{Qa}{2} +(1-a)\ell)}\\
     &\lesssim 2^{-2j(\frac{Qa}{2} +(1-a)\ell)}\times 2^{jQ}=2^{j(1-a)(Q-2\ell)},
\end{align*} where we have used again the Weyl eigenvalue counting formula for the sub-Laplacian (Remark \ref{weyl}). So, we have
\[ 
    \Vert \Delta_q \sigma(x,\xi)\Vert_{L^2(\widehat{G})}\lesssim 2^{j(1-a)(\frac{Q}{2}-\ell)}.
\] 
The preceding  analysis allows us to conclude that for $\ell:=[\frac{Q}{2}]+1>\frac{Q}{2},$
\begin{align*}
     \int\limits_{|z|> b}|k_x(z)|dz &\lesssim  2^{j(a-1)(\frac{Q}{2}-\ell)}\Vert \Delta_q \sigma(x,\xi)\Vert_{L^2(\widehat{G})}\lesssim  2^{j(a-1)(\frac{Q}{2}-\ell)}2^{j(1-a)(\frac{Q}{2}-\ell)}=1.
\end{align*}
Thus, the proof is complete.
\end{proof}
 
 \begin{lemma}\label{FundamentallemmaII}
Let $G$ be a compact Lie group and let us denote by $Q$ the Hausdorff
dimension of $G$ associated to the control distance associated to the sub-Laplacian $\mathcal{L}=\mathcal{L}_X,$ where  $X= \{X_{1},\cdots,X_k\} $ is a system of vector fields satisfying the H\"ormander condition of order $\kappa$.  For  $0\leqslant \delta<\rho\leqslant 1,$ and $\varepsilon\in \mathbb{R},$  let us consider a continuous linear operator $A:C^\infty(G)\rightarrow\mathscr{D}'(G)$ with symbol  $\sigma\in {S}^{\varepsilon,\mathcal{L}}_{\rho,\delta}(G\times \widehat{G})$. Let $\eta$ be a smooth function supported in $\{\lambda:R\leqslant\lambda\leqslant 3R\},$ for some $R>1.$ Then for all $\alpha\in\mathbb{N}_0^n$ with $|\alpha|\leqslant \ell,$ there exists $C>0,$ such that for every $\omega\geq 0,$ 
\[ 
   \sup_{(x,[\xi])\in G\times \widehat{G}}  \Vert \widehat{\mathcal{M}}(\xi)^{ -\varepsilon+\rho|\alpha| }\Delta^{\alpha}_\xi [\sigma(x,\xi)\eta(s\langle\xi \rangle)]     \Vert_{\textnormal{op}}\leqslant C     \Vert \sigma\Vert_{\ell,S^{\varepsilon,\mathcal{L}}_{\rho,\delta}}(s\langle\xi\rangle)^{-\omega},
\]
with the positive constant $C$  independent of $s>0,$ $R$ and $\sigma.$ 
\end{lemma}
\begin{proof} Let us  consider $\sigma\in {S}^{\varepsilon,\mathcal{L}}_{\rho,\delta}(G\times \widehat{G})=\mathscr{S}^{\varepsilon\kappa,\mathcal{L}}_{\kappa\rho,\kappa\delta}(G).$ The Leibniz rule allows us to write
\[ 
   \Delta^{\alpha}_\xi [\sigma(x,\xi)\eta(s\langle\xi \rangle)I_{d_\xi}]=\sum_{\alpha_1+\alpha_2=\alpha}C_{\alpha_1,\alpha_2}[ \Delta_{\alpha_1}\sigma](x,\xi)[\Delta_{\alpha_2}\eta(s\langle\xi\rangle)I_{d_\xi}],
\] because $\eta$ has support in $\{\lambda: R\leqslant \lambda \leqslant 3R\}$, we can estimate (see Lemma 6.8 of \cite{Fischer2015}),
\begin{equation}\label{516}
    \Vert \Delta_{\alpha_2}\eta(s\langle\xi\rangle)I_{d_\xi}\Vert_{\textnormal{op}}\leqslant C_\omega(s\langle\xi\rangle)^{-\omega},
\end{equation}for every $\omega\in \mathbb{R}_{0}^+.$
So, we deduce
\begin{align*}
     & \Vert   \Delta^{\alpha}_\xi[\sigma(x,\xi)\eta(s\langle\xi \rangle)I_{d_\xi}]\Vert_{\textnormal{op}} \\
     &\leqslant 
       \sum_{\alpha_1+\alpha_2=\alpha}C_{\alpha_1,\alpha_2} \Vert \widehat{\mathcal{M}}(\xi)^{ -\varepsilon+\rho|\alpha_1| }[ \Delta_{\alpha_1}\sigma](x,\xi)] \Vert_{\textnormal{op}}\Vert\Delta_{\alpha_2}[  \eta(s\langle\xi\rangle)I_{d_\xi}]\Vert_{\textnormal{op}}\\
       &\lesssim 
       \sum_{\alpha_1+\alpha_2=\alpha}C_{\alpha_1,\alpha_2} \Vert \widehat{\mathcal{M}}(\xi)^{ -\varepsilon+\rho|\alpha_1| }[ \Delta_{\alpha_1}\sigma](x,\xi)] \Vert_{\textnormal{op}} (s\langle\xi\rangle)^{-\omega}\\
       &\lesssim   \Vert \sigma\Vert_{\ell,S^{\varepsilon,\mathcal{L}}_{\rho,\delta}}(s\langle\xi\rangle)^{-\omega},
\end{align*}proving Lemma \ref{FundamentallemmaII}.
\end{proof}
\begin{remark}\label{Remark}
Let us mention that for $s\in (0,1),$ (instead of the general case $s>0$ considered in the previous lemma), from \cite[page 3434]{Fischer2015}, \eqref{516} can be replaced by the following estimate,
\begin{equation}\label{516'}
    \Vert \Delta_{\alpha_2}\eta(s\langle\xi\rangle)I_{d_\xi}\Vert_{\textnormal{op}}\leqslant C_{\omega,\eta} s^{\omega},
\end{equation} for all  $\omega\in \mathbb{R}.$ So, from the proof of Lemma \ref{FundamentallemmaII}, we deduce the following estimate which is interesting by itself, 
\begin{equation}\label{Fischersub}
   \sup_{(x,[\xi])\in G\times \widehat{G}}  \Vert \widehat{\mathcal{M}}(\xi)^{ -\varepsilon+\rho|\alpha| }\Delta^{\alpha}_\xi [\sigma(x,\xi)\eta(s\langle\xi \rangle)]     \Vert_{\textnormal{op}}\lesssim_{\omega,\eta} C     \Vert \sigma\Vert_{\ell,S^{\varepsilon,\mathcal{L}}_{\rho,\delta}}s^{\omega},\,\,0<s<1.
\end{equation} This estimate is a subelliptic analogy of Lemma \ref{LemmaFischer}.
\end{remark}
 Now, we will present one of our fundamental theorems in this memoir. The subelliptic BMO space and the subelliptic Hardy space $\textnormal{BMO}^{\mathcal{L}}$ and $\textnormal{H}^{1,\mathcal{L}}$ were defined in Subsection \ref{BMOH1}.
 \begin{theorem}\label{parta}
Let $G$ be a compact Lie group and let us denote by $Q$ the Hausdorff
dimension of $G$ associated to the control distance associated to the sub-Laplacian $\mathcal{L}=\mathcal{L}_X,$ where  $X= \{X_{1},\cdots,X_k\} $ is a system of vector fields satisfying the H\"ormander condition of order $\kappa$.  For  $0\leqslant \delta<\rho\leqslant    1,$   let us consider a continuous linear operator $A:C^\infty(G)\rightarrow\mathscr{D}'(G)$ with symbol  $\sigma\in {S}^{-\frac{Q(1-\rho) }{2},\mathcal{L}}_{\rho,\delta}(G\times \widehat{G})$. Then $A$ extends to a bounded operator from $L^\infty(G)$ to $\textnormal{BMO}^{\mathcal{L}}(G),$ and from $\textnormal{H}^{1,\mathcal{L}}(G)$ to $L^1(G).$ Moreover,
\begin{equation}\label{eq1}
    \Vert  A\Vert_{\mathscr{B}(L^\infty(G),\textnormal{BMO}^{\mathcal{L}}(G))}\leqslant C \max\{\Vert \sigma\Vert_{\ell, {S}^{-\frac{Q(1-\rho) }{2},\mathcal{L}}_{\rho,\delta} },\Vert  \sigma(\cdot,\cdot) \widehat{\mathcal{M}}(\xi)^{{ Q(1-\rho) } }\Vert_{\ell, {S}^{0,\mathcal{L}}_{\rho,\delta} }\}
\end{equation} and 
\begin{equation}\label{eq2}
      \,\,\,\,\,\,\,\,\,\,\Vert  A\Vert_{\mathscr{B}(\textnormal{H}^{1,\mathcal{L}}(G),L^1(G))}\leqslant C \max\{\Vert \sigma^*\Vert_{\ell, {S}^{-\frac{Q(1-\rho) }{2},\mathcal{L}}_{\rho,\delta} }, \Vert \sigma^*(\cdot,\cdot) \widehat{\mathcal{M}}(\xi)^{{ Q(1-\rho) } }\Vert_{\ell, {S}^{0,\mathcal{L}}_{\rho,\delta} }\}
\end{equation}
for some integer $\ell,$ where $\sigma^{*}$ denotes the matrix-valued symbol of the formal adjoint $A^*$ of $A.$ 
\end{theorem}
\begin{proof}
Let us consider  $0\leqslant a:=(1-\rho)\leqslant 1,$  and let $f\in L^\infty(G).$  Let us fix a ball $B(x_0,r)$ where $x_0\in G,$ and $r>0.$ We will prove that there exists a constant $C>0,$ independent of $f$ and $r,$ such that
\[ 
    \frac{1}{|B(x_0,r)|}\int\limits_{B(x_0,r)}|Af(x)-(Af)_{B(x_0,r)}|dx\leqslant C\Vert \sigma\Vert_{\ell,{S}^{-\frac{Q(1-\rho)}{2}}_{\rho,\delta}}\Vert f\Vert_{L^\infty(G)}
\]
 where 
\[ 
    (Af)_{B(x_0,r)}:=\frac{1}{|B(x_0,r)|}\int\limits_{B(x_0,r)}Af(x)dx.
\]  
 We will prove the existence of two symbols $\sigma^0$ and $\sigma^1$ such that (see \eqref{split})
\[ 
    \sigma(x,\xi)=\sigma^0(x,\xi)+\sigma^1(x,\xi),\,\,\,\sigma^{j}(x,\xi)\in \mathbb{C}^{d_\xi\times d_\xi},\,j=0,1,
\] in a such way that both, $\sigma^0(x,\xi)$  and  $\sigma^1(x,\xi),$ satisfy the estimate
\begin{equation}\label{symbolestimate}
    \Vert \sigma^{j}\Vert_{\ell,\mathscr{S}^{-\frac{Q(1-\rho)}{2}}_{\rho\kappa,\delta\kappa}}\leqslant C_{j,\ell}\Vert \sigma\Vert_{\ell,{S}^{-\frac{Q(1-\rho)}{2}}_{\rho,\delta}},\,\,j=0,1,\,\ell\geqslant   1. 
\end{equation} Now, if $A^j:=\textnormal{Op}(\sigma^j),$ for $j=0,1,$ then $A=A^0+A^1$ on $C^\infty(G),$ and we only need to prove that\begin{equation}\label{splitestimate}
    \frac{1}{|B(x_0,r)|}\int\limits_{B(x_0,r)}|A^jf(x)-(A^jf)_{B(x_0,r)}|dx\leqslant C\Vert \sigma^j\Vert_{\ell,{S}^{-\frac{Q(1-\rho)}{2},\mathcal{L}}_{\rho,\delta}}\Vert f\Vert_{L^\infty(G)}.
\end{equation} Then, if we prove \eqref{symbolestimate}, we can deduce \eqref{eq1} and by the duality argument we could obtain  \eqref{eq2} proving Theorem \ref{parta}. Now, we proceed to prove the existence of $\sigma^0$ and $\sigma^1$ satisfying the requested properties.
  Let us define
 \begin{equation}\label{split}
     \sigma^0(x,\xi)=\sigma(x,\xi)\tilde{\gamma}(zz\xi),\textnormal{    }\sigma^{1}(x,\xi)=\sigma(x,\xi)-\sigma^0(x,\xi),\,(x,\xi)\in G\times \widehat{G},
 \end{equation}where $\tilde{\gamma}(\xi):=\gamma(r\langle\xi\rangle ),$ and $\gamma\in C^\infty_0(\mathbb{R},\mathbb{R}^+_0),$ is a smooth function supported in $\{t:|t|\leqslant 1\},$ satisfying $\gamma(t)=1,$ for $|t|\leqslant \frac{1}{2}.$ To estimate the integral
\[ 
    I_0:=\frac{1}{|B(x_0,r)|}\int\limits_{B(x_0,r)}|A^0f(x)-(A^0f)_{B(x_0,r)}|dx,
\]we will use the Mean value Theorem. If $d_G(x,y)$ is the geodesic distance between $x$ and $y,$  observe that
\begin{align}\label{r}
    |A^0f(x)-A^0f(y)| &\leqslant C_{0}\sum_{k=1}^{n}\Vert X_kA^0f\Vert_{L^\infty(G)}d_G(x,y)
    &\leqslant r\sum_{k=1}^{\dim(G)}\Vert X_kA^0f\Vert_{L^\infty(G)},
\end{align} where $\{X_k\}_{k=1}^n$ is a basis for the Lie algebra $\mathfrak{g}$ of $G.$
In order to estimate the $L^\infty$-norm of $X_kA^0f,$ first let us observe that the matrix-valued symbol of $X_kA^0=\textnormal{Op}(\sigma'_k)$ is given by
\begin{equation}\label{symbolderivative}
\sigma'_k(x,\xi):=   \sigma_{X_k}(\xi)\sigma^0(x,\xi)+(X_k\sigma^0(x,\xi)).
\end{equation}
Indeed, the Leibniz law gives
\begin{align*}
    X_kA^0f(x)&=\sum\limits_{[\xi]\in \widehat{G}}\textnormal{\textbf{Tr}}(X_k(\xi(x)\sigma^0(x,\xi))\widehat{f}(\xi))\\
    &=\sum\limits_{[\xi]\in \widehat{G}}\textnormal{\textbf{Tr}}([X_k(\xi(x))\sigma^0(x,\xi)+\xi(x)X_k\sigma^0(x,\xi))]\widehat{f}(\xi)).
\end{align*}
The identity $\sigma_{X_k}(\xi)=\xi(x)^*X_k\xi(x),$ implies  $X_k\xi(x)=\xi(x)\sigma_{X_k}(\xi),$ and we obtain
\begin{align*}
     X_kA^0f(x)&=\sum\limits_{[\xi]\in \widehat{G}}\textnormal{\textbf{Tr}}([\xi(x)\sigma_{X_k}(\xi)\sigma^0(x,\xi)+\xi(x)X_k\sigma^0(x,\xi))]\widehat{f}(\xi)),
\end{align*}  which proves \eqref{symbolderivative}. By using a suitable partition of the unity we will decompose the matrix $\sigma'_k(x,\xi)$ in the following way,
\[ 
    \sigma'_k(x,\xi)=\sum_{j=1}^\infty\rho_{j,k}(x,\xi).
\]
To construct the family of operators $\rho_{j,k}(x,\xi)$ we will proceed as follows. We choose a smooth real  function $\eta$ satisfying $\eta(t)\equiv 0$ for $|t|\leqslant 1$ and $\eta(t)\equiv 1$ for $|t|\geqslant   2.$ Set
\[ 
    \rho(t)=\eta(t)-\eta(\frac{t}{2}).
\] The support of $\rho$ satisfies $\textnormal{supp}(\rho)\subset[1,4] .$ One can check that 
\[ 
    1=\eta(t)+\sum_{j=1}^\infty\rho(2^{j}t),\,\,\,\,\textnormal{ for all }t\in \mathbb{R}.
\] Indeed, 
\[ 
    \eta(t)+\sum_{j=1}^{ \infty}  (2^{j}t)=\eta(t)+\sum_{j=1}^\ell\eta(2^{j}t)-\eta(2^{j-1}t)=\eta(2^{\ell}t)\rightarrow 1,\,\,\ell\rightarrow\infty.
\]
For $t=r\langle\xi \rangle$ we have
\[ 
    1=\eta(r\langle\xi \rangle)+\sum_{j=1}^\infty\rho(2^{j}r\langle\xi \rangle).
\] 
 Observe that
 \[ 
     \sigma'_k(x,\xi)=\eta(r\langle\xi \rangle)\sigma'_k(x,\xi)+\sum_{j=1}^\infty\rho(2^{j}r\langle\xi \rangle)\sigma'_k(x,\xi).
 \] Because $\eta(r\langle\xi \rangle)\sigma'_k(x,\xi)=0,$ in view that the support of $\tilde\gamma$ and $\eta$ are disjoint sets, we have
  \[ 
     \sigma'_k(x,\xi)=\sum_{j=1}^\infty\rho_{j,k}(x,\xi),\,\,\,\rho_{j,k}(x,\xi):=\rho(2^{j}r\langle\xi \rangle)\sigma'_k(x,\xi).
 \] 
 From the Mean value Theorem we have,
 \begin{align*}
    |A^0f(x)-A^0f(y)| &\leqslant C_{0}\sum_{k=1}^{\textnormal{dim}(G)}|y^{-1}x|\cdot\Vert X_kA^0f\Vert_{L^\infty(G)}\\
    &\lesssim r\sum_{k=1}^{\textnormal{dim}(G)}\Vert X_kA^0f\Vert_{L^\infty(G)}.
\end{align*}

 Now, in order to prove \eqref{splitestimate} for $j=0,$ we proceed by using Lemma \ref{FundamentallemmaII} with $\omega=1,$ Lemma \ref{FundamentallemmaI} and \eqref{r}. So, we have,
 \begin{align*}
     I_0:&=\frac{1}{|B(x_0,r)|}\int\limits_{B(x_0,r)}|A^0f(x)-(A^0f)_{B(x_0,r)}|dx\\
     &= \frac{1}{|B(x_0,r)|}\int\limits_{B(x_0,r)}\left|\frac{1}{|B(x_0,r)|}\int\limits_{B(x_0,r)}(A^0f(x)-A^0f(y))dy\right|dx\\
     &\leqslant \frac{1}{|B(x_0,r)|^{2}}\int\limits_{B(x_0,r)}\int\limits_{B(x_0,r)}|A^0f(x)-A^0f(y)|dydx \\
     &\lesssim r \sup_{1\leqslant k\leqslant \dim(G)}\Vert X_kA^0f\Vert_{L^\infty(G)}\\
     &=r\sup_{1\leqslant k\leqslant \dim(G)}\Vert \textnormal{Op}(\sigma'_{k})f\Vert_{L^\infty(G)}\leqslant r\sup_{1\leqslant k\leqslant \dim(G)}\sum_{j=1}^\infty\Vert \textnormal{Op}(\rho_{j,k})f\Vert_{L^\infty(G)}\\
     &\lesssim r \sum_{j=1}^\infty r^{-1}2^{-j} \Vert \sigma\Vert_{\ell,S^{-\frac{Qa}{2},\mathcal{L}}_{\rho,\delta}}\Vert f\Vert_{L^\infty(G)}\\
      &\lesssim rr^{-1}\Vert \sigma\Vert_{\ell,S^{-\frac{Qa}{2},\mathcal{L}}_{\rho,\,\delta}}\Vert f\Vert_{L^\infty(G)}\\
      &=\Vert \sigma\Vert_{\ell,S^{-\frac{Qa}{2},\mathcal{L}}_{\rho,\delta}}\Vert f\Vert_{L^\infty(G)}.
\end{align*} In order to finish the proof, we need to prove \eqref{splitestimate} for $j=1.$
In order to obtain a similar $L^\infty(G)$-$\textnormal{BMO}^{\mathcal{L}}(G)$ estimate for $A^{1},$ we will proceed as follows. Let $\phi$ be a smooth function compactly supported in $B(x_0,2r)$ satisfying that
\[ 
    \phi(x)=1, \textnormal{  for  }\,\,x\in B(x_0,r),\textnormal{ and }0\leqslant \phi\leqslant 10.
\]
Note that
\[ 
   |B(x_0,r)|\leqslant \int\limits_{B(x_0,r)}\phi(x)^2dx\leqslant  \int\limits_{B(x_0,2r)}\phi(x)^2dx =\Vert \phi\Vert_{L^2(G)}^2\leqslant 100  |B(x_0,2r)|,
\] and consequently
\begin{equation}\label{doubling}
     |B(x_0,r)|^{\frac{1}{2}}\leqslant \Vert \phi\Vert_{L^2(G)}\leqslant 10|B(x_0,2r)|^{\frac{1}{2}}\leqslant 10C|B(x_0,r)|^{\frac{1}{2}},
\end{equation}where in the last inequality we have used  that the measure on the group satisfies the doubling property. Taking into account that
\begin{align*}
    \frac{1}{|B(x_0,r)|}\int\limits_{B(x_0,r)}|A^1f(x)-(A^1f)_{B(x_0,r)}|dx\leqslant \frac{2}{|B(x_0,r)|}\int\limits_{B(x_0,r)}|A^1f(x)|dx,
\end{align*}
we will estimate the right-hand side. Indeed, let us observe that
\begin{align*}
 &\frac{1}{|B(x_0,r)|} \int\limits_{B(x_0,r)}|A^1f(x)|dx\leqslant  \frac{1}{|B(x_0,r)|}\int\limits_{B(x_0,r)}|\phi(x)A^1f(x)|dx\\
  &\leqslant  \frac{1}{|B(x_0,r)|}\int\limits_{B(x_0,r)}|A^1[\phi f](x)|dx+\frac{1}{|B(x_0,r)|} \int\limits_{B(x_0,r)}|[M_{\phi},A^1 ] f(x)|dx\\
 :&=I+II,
\end{align*}where $M_\phi$ is the multiplication  operator by $\phi.$ To estimate $I,$ observe that, in view of the Cauchy-Schwartz inequality, we have
\[ 
    \frac{1}{|B(x_0,r)|}\int\limits_{B(x_0,r)}|A^1[\phi f](x)|dx\leqslant \frac{1}{|B(x_0,r)|^{\frac{1}{2}}}\left(\int\limits_{B(x_0,r)} |A^1[\phi f](x)|^2dx   \right)^{\frac{1}{2}}.
\]
For $0< \rho\leqslant 1,$ let $L:=(1+\mathcal{L})^{\frac{Qa\varepsilon}{2}}\in {S}^{\frac{Qa\varepsilon}{2}}_{  \rho  ,0}(G\times\widehat{G})\subset \mathscr{S}^{\frac{Qa\varepsilon\kappa}{2}}_{\rho ,0}(G\times\widehat{G})\subset \mathscr{S}^{\frac{Qa\varepsilon}{2}}_{\rho ,\delta}(G\times\widehat{G}),$ where $a:=1-\rho$ and $0<\varepsilon<1.$
For a moment, allow us to assume that we have
\begin{equation}\label{L2}
    \textnormal{Op}({S}^{-m,\mathcal{L}}_{\rho',\delta'}(G\times\widehat{G}))\subset \mathscr{B}(L^2(G)),
\end{equation} for $m\geqslant   0$ and $0\leqslant \delta'<\rho'\leqslant 1.$ Because $A^{1}\in {S}^{-\frac{Qa}{2},\mathcal{L}}_{ \rho, ,\delta}(G\times\widehat{G}), $ the subelliptic symbolic calculus gives
\[ 
    A^{1}L\in {S}^{-\frac{Qa}{2}(1-\varepsilon),\mathcal{L}}_{\rho ,\delta}(G\times\widehat{G}).
\]
Because $m=\frac{Qa}{2}(1-\varepsilon)>0, $ it follows from \eqref{L2} and the condition $\delta<1,$  that $A^{1}L$ is bounded on $L^2(G).$  Consequently,
\begin{align*}
 &  \frac{1}{|B(x_0,r)|^{\frac{1}{2}}}\left(\int\limits_{B(x_0,r)} |A^1L[L^{-1}(\phi f)](x)|^2dx   \right)^{\frac{1}{2}}\leqslant \frac{\Vert A^1L[L^{-1}(\phi f)] \Vert_{L^2(G)}}{|B(x_0,r)|^{\frac{1}{2}}}\\
   &\leqslant \frac{  C\Vert \sigma_{A^1L}\Vert_{k,{S}^{-\frac{Qa}{2}(1-\varepsilon),\mathcal{L}}_{\rho,\delta}}    \Vert L^{-1}( \phi f )\Vert_{L^2(G)}}{|B(x_0,r)|^{\frac{1}{2}}}.
\end{align*}
By observing that 
\[ 
    \Vert L^{-1}(\phi f) \Vert_{L^2(G)}=\Vert \phi f\Vert_{H^{-\frac{Qa\varepsilon}{2},\mathcal{L}}(G)},
\]
where $H^{-\frac{Qa\varepsilon}{2},\mathcal{L}}(G)$ is the Sobolev space of order $-\frac{Qa\varepsilon}{2},$ associated with $\mathcal{L},$ the embedding $L^2(G)\hookrightarrow H^{-\frac{Qa\varepsilon}{2},\mathcal{L}}(G), $ (see \cite{GarettoRuzhansky2015}) implies that
\[ 
 \Vert L^{-1}(\phi f) \Vert_{L^2(G)}=\Vert \phi f\Vert_{H^{-\frac{Qa\varepsilon}{2},\mathcal{L}}(G)}\lesssim \Vert \phi f \Vert_{L^2(G)} .    
\]
Moreover, from \eqref{doubling}, we deduce the inequality
\[ 
     \Vert \phi f \Vert_{L^2(G)}\leqslant \Vert f\Vert_{L^\infty(G)}\Vert \phi\Vert_{L^2(G)}\lesssim   10\Vert f\Vert_{L^\infty(G)}|B(x_0,r)|^{\frac{1}{2}}.
\]
So, we conclude
\[ 
    I:= \frac{1}{|B(x_0,r)|}\int\limits_{B(x_0,r)}|A^1[\phi f](x)|dx\leqslant {  C\Vert \sigma_{A^1L}\Vert_{\ell',{S}^{0,\mathcal{L}}_{\rho,\delta}} }\Vert f\Vert_{L^\infty(G)},
\]
which is the desired estimate for $I.$ Now, we will estimate $II.$ For this, observe that the symbol of $[M_{\phi},A^1 ]=\textnormal{Op}(\theta) ,$ is given by
\begin{equation}\label{theta}
    \theta(x,\xi)=\int\limits_{G}(\phi(x)-\phi(xy^{-1}))k_{x}(y)\xi(y)^{*}dy,
\end{equation}
where $x\mapsto k_{x},$ is the right-convolution kernel of $A^1.$ Indeed, the equality \eqref{theta}  was shown in \cite[page 554]{RuzhanskyDelgado2017}.
Using the Taylor expansion we obtain
\[ 
    \phi(xy^{-1})=\phi(x)+\sum_{|\alpha|=1}(X_{x}^{\alpha}\phi)(x)\tilde{q}_{\alpha}(y),
\]
where, every  $\tilde{q}_\alpha$ is a smooth function vanishing with order $1$ at $e_G.$
So, we can write
\[ 
    \theta(x,\xi)=\sum_{|\alpha|=1}{X_{x}^{\alpha}\phi(x)}\Delta_{ \tilde{q}_{\alpha}} \sigma(x,\xi) .
\]The hypothesis $\sigma \in {S}^{-\frac{Q(1-\rho) }{2},\mathcal{L}}_{\rho,\delta}(G\times \widehat{G}), $ implies that $\Delta_{ \tilde{q}_{\alpha}} \sigma(x,\xi) \in  {S}^{-\frac{Q(1-\rho) }{2}-\rho,\mathcal{L}}_{\rho,\delta}(G\times \widehat{G}),$ and the fact that $\phi$ is of compact support, implies the same conclusion for $\theta:=\theta(x,\xi) \in  {S}^{-\frac{Q(1-\rho) }{2}-\rho,\mathcal{L}}_{\rho,\delta}(G\times \widehat{G}).$
From  Lemma \ref{FundamentallemmaI}, one has   
\begin{align*}
   &\Vert\textnormal{Op}(\theta) \Vert_{\mathscr{B}(L^\infty(G))}  \lesssim  \left(\sup_{(x,[\xi])\in G\times\widehat{G},|\alpha|\leqslant \ell'} \Vert [\Delta^{\alpha}_\xi \theta(x,\xi)]\widehat{\mathcal{M}}(\xi)^{\frac{Qa}{2}+\rho+\rho|\alpha|})  \Vert_{\textnormal{op}}     \right)\\
   &\lesssim\sup_{|\alpha|\leqslant \ell} \Vert {X_{x}^{\alpha}\phi} \Vert_{L^\infty(G)} \left(\sup_{(x,[\xi])\in G\times\widehat{G},|\alpha|\leqslant \ell'} \Vert [\Delta^{\alpha}_\xi \sigma(x,\xi)]\widehat{\mathcal{M}}(\xi)^{\frac{Qa}{2}+\rho+\rho|\alpha|})  \Vert_{\textnormal{op}}     \right)\\
   &\lesssim_{\ell'} \Vert \sigma\Vert_{\ell'+1,S^{-\frac{Qa}{2}+\rho,\mathcal{L}   }_{\rho,\delta}}.
\end{align*}
Observing that
\begin{align*}
    \frac{1}{|B(x_0,r)|} \int\limits_{B(x_0,r)}|[M_{\phi},A^1 ] f(x)| &\leqslant   \frac{1}{|B(x_0,r)|} \int\limits_{B(x_0,r)}dx   \Vert [M_{\phi},A^1 ] f \Vert_{L^\infty(G)}
    \\&=  \Vert [M_{\phi},A^1 ] f \Vert_{L^\infty(G)},
\end{align*} we deduce,
\begin{align*}
     \frac{1}{|B(x_0,r)|} \int\limits_{B(x_0,r)}|[M_{\phi},A^1 ] f(x)|&\leqslant\Vert\textnormal{Op}(\theta)f \Vert_{L^\infty(G)}\\
     &\lesssim \Vert \sigma\Vert_{\ell'+1,S^{-\frac{Qa}{2},\mathcal{L}   }_{\rho,\delta}} \Vert f \Vert_{L^\infty(G)}.
\end{align*}
Thus, we obtain 
\[ 
    II:= \frac{1}{|B(x_0,r)|}\int\limits_{B(x_0,r)}|[M_\phi,A^1 ]f(x)|dx\leqslant {  C\Vert \sigma_{A}\Vert_{\ell'+1,{S}^{-\frac{Qa}{2},\mathcal{L}}_{\rho,\delta}} }\Vert f\Vert_{L^\infty(G)}.
\] 
Now, in order to finish the proof, we only need to prove \eqref{L2}. We will follow  the classical argument of H\"ormander. Indeed, assume first that $p(x,\xi)\in {S}^{-m_0,\mathcal{L}}_{\rho',\delta'}(G\times \widehat{G}),$ where $m_0>0.$ The  kernel of $p(x,D)=\textnormal{Op}(p),$ $K_{p}(x,y),$ belong to $L^\infty(G\times G)$ for $m_0$ large enough.
Hence, $$ \sup_{x\in G}\int\limits_G|K_p(x,y)|dy\, , \, \sup_{y\in G}\int\limits_G|K_p(x,y)|dx<\infty,$$ and the $L^2(G)$-continuity of $p(x,D)$ follows from Schur lemma. Next we prove by induction that $p(x,D)$ is $L^2$-bounded if $p(x,\xi)\in  {S}^{-m_0,\mathcal{L}}_{\rho',\delta'}(G\times \widehat{G}),$ for $m_0< m\leqslant -(\rho'-\delta').$ To do so we form for $u\in C^{\infty}(G)$
\begin{align*}
\Vert p(x,D)u \Vert^2_{L^2(G)}&=( p(x,D)u, p(x,D)u)_{L^2(G)}\\
&=( p^{*}(x,D)p(x,D)u, u)_{L^2(G)}\\
&=( b(x,D)u,u )_{L^2(G)} ,
\end{align*}
where $b(x,D)=p^{*}(x,D)p(x,D)$ has symbol in ${S}^{2m,\mathcal{L}}_{\rho',\delta'}(G\times\widehat{G}),$ for $0\leqslant \delta'<\rho'\leqslant 1.$ From the induction hypothesis the continuity of $p(x,D)$ for all $p\in {S}^{2m,\mathcal{L}}_{\rho',\delta'}$ now follows successively for $m\leqslant-\frac{m_0}{2},\cdots, -\frac{m_0}{4},\cdots , -\frac{m_0}{2^{\ell_0}},\cdots ,$ $\ell_0\in\mathbb{N},$ and hence for $m\leqslant -\frac{m_0}{2^{\ell_0}}$ where $\frac{m_0}{2^{\ell_0}}<\rho'-\delta',$ after a finite number of steps. 
 Now, assume that $p(x,\xi)\in {S}^{0,\mathcal{L}}_{\rho',\delta'}(G\times\widehat{G})$ and chose  $$ M>2\sup_{(x,[\xi])} \Vert p(x,\xi) \Vert^{2}_{\textnormal{op}}.$$ Then, in view of the subelliptic fucntional calculus, in particular by Corollary \ref{1/2}, we have that  $c(x,\xi)=(MI_{d_\xi}-p(x,\xi)p(x,\xi)^*
 )^{1/2}\in {S}^{0,\mathcal{L}}_{\rho',\delta'}(G\times \widehat{G}).$
Now,
$$  c(x,D)^{*}c(x,D)=M-p^{*}(x,D)p(x,D)+r(x,D)$$
where $r\in {S}^{-(\rho'-\delta')}_{{\rho',\delta'}}(G\times\widehat{G}).$
Hence, $\Vert p(x,D) \Vert_{\mathscr{B}(L^2)}\leqslant M+\Vert r(x,D)\Vert_{\mathscr{B}(L^2)}. $
Thus, the proof is complete.
\end{proof}
In view of \eqref{L2} we have the following $L^2(G)$-estimate.
\begin{corollary}\label{CL2''}Let us assume that $G$ is a compact Lie group. Then,
\begin{equation}\label{L2''}
    \textnormal{Op}({S}^{0,\mathcal{L}}_{\rho',\delta'}(G\times\widehat{G}))\subset \mathscr{B}(L^2(G)),
\end{equation} for all $0\leqslant \delta'<\rho'\leqslant 1.$ In the case where $\delta'< 1/\kappa,$ and $\rho'\leq 1,$ the condition $\delta'<\rho'$ can be improved to $\delta'\leq \rho'$ in order to obtain the $L^2(G)$-boundedness of $A$ (see the subelliptic Calder\'on-Vaillancourt Theorem \ref{CVT}).
\end{corollary}

\subsection{$L^p$-Sobolev and Besov boundedness for subelliptic  classes}

\begin{remark}
Under those hypothesis in   Theorem \ref{parta}, every continuous linear operator $A:C^\infty(G)\rightarrow\mathscr{D}'(G)$ with symbol  $\sigma\in {S}^{-\frac{Q(1-\rho)}{2},\mathcal{L}}_{\rho,\delta}(G\times \widehat{G})$ extends to a bounded operator from $L^\infty(G)$ to $\textnormal{BMO}^{\mathcal{L}}(G),$ and from $\textnormal{H}^{1,\mathcal{L}}(G)$ to $L^1(G).$ The simple argument of interpolation gives that $A$ also extends to a bounded operator on $L^p(G)$ for all $1<p<\infty.$ However, an argument via the Fefferman-Stein interpolation theorem gives a more precise result that we will now present in Theorem \ref{parta2}.
\end{remark}
 \begin{theorem}\label{parta2}
Let $G$ be a compact Lie group and let us denote by $Q$ the Hausdorff
dimension of $G$ associated to the control distance associated to the sub-Laplacian $\mathcal{L}=\mathcal{L}_X,$ where  $X= \{X_{1},\cdots,X_k\} $ is a system of vector fields satisfying the H\"ormander condition of order $\kappa$.  For  $0\leqslant \delta<\rho\leqslant 1,$  let us consider a continuous linear operator $A:C^\infty(G)\rightarrow\mathscr{D}'(G)$ with symbol  $\sigma\in {S}^{-m,\mathcal{L}}_{\rho,\delta}(G\times \widehat{G})$, $m\geqslant   0$. Then $A$ extends to a bounded operator on $L^p(G)$ provided that 
\[ 
    m\geqslant   m_p:= Q(1-\rho)\left|\frac{1}{p}-\frac{1}{2}\right|.
\]
\end{theorem}
\begin{proof}
Now, having proved Theorem \ref{parta}, the proof is verbatim  the proof of Theorem 4.15 of \cite{RuzhanskyDelgado2017}. Let us write $a:=1-\rho.$ We will present the argument here, for completeness. We will use the complex Fefferman-Stein interpolation theorem. We only need to prove the theorem for $m=m_p$ in view of the inclusion $S^{-m,\mathcal{L}}_{\rho,\delta}(G\times \widehat{G})\subset S^{-m_p,\mathcal{L}}_{\rho,\delta}(G\times \widehat{G})  $  for $m>m_p.$ Let us consider the complex family of operators indexed by $z\in \mathbb{C},$ $ \mathfrak{Re}(z)\in [0,1],$ given by
\[ 
    T_{z}:=\textnormal{Op}(\sigma_{z}),\,\,\,\, \sigma_{z}(x,\xi):=e^{z^2}\sigma(x,\xi)\widehat{\mathcal{M}}(\xi)^{{\frac{Qa}{2}(z-1)}}.
\] The family of operators $\{T_{z}\}$ defines an analytic family of operator from $ \mathfrak{Re}(z)\in (0,1),$ (resp. continuous for $ \mathfrak{Re}(z)\in [0,1]$), into the algebra of bounded operators on $L^2(G).$  Let us observe that $\sigma_0(x,\xi)=\sigma(x,\xi)\widehat{\mathcal{M}}(\xi)^{{m-\frac{Qa}{2}}},$ and $\sigma_{1}(x,\xi)=e\sigma(x,\xi)\widehat{\mathcal{M}}(\xi)^{m}.$ Because $T_0$ is bounded from $L^\infty(G)$ into $\textnormal{BMO}^{\mathcal{L}}(G)$ and $T_1$ is bounded on $L^2(G),$ the Fefferman-Stein interpolation theorem implies that $T_t$ extends to a bounded operator on $L^p(G),$ for $p=\frac{2}{t}$ and all $0<t\leqslant 1.$ Because $0\leqslant m\leqslant\frac{Qa}{2}, $ there exists $t_0\in (0,1)$ such that $m=m_p=\frac{Qa}{2}(1-t_{0}).$ So, $T_{t_0}=e^{t_0^2}A$ extends to a bounded operator on $L^{\frac{2}{t_0}}.$ The Fefferman-Stein interpolation theorem, the $L^2(G)$-boundedness and the  $L^{\frac{2}{t_0}}$-boundedness  of $A$ give the $L^p(G)$-boundedness of $A$ for all $2\leqslant p\leqslant \frac{2}{t_0},$ and interpolating the $L^{\frac{2}{t_0}}(G)$-boundedness with the $L^\infty(G)$-$\textnormal{BMO}^{\mathcal{L}}(G)$ boundedness of $A$ we obtain the boundedneess of $A$ on $L^p(G)$ for all $\frac{2}{t_0}\leqslant p<\infty.$ So, $A$ extends to a bounded operator on $L^p(G)$ for all $2\leqslant p<\infty.$ The $L^p(G)$-boundedness of $A$ for $1<p\leqslant 2$ now follows by the duality argument.
\end{proof}
\begin{corollary}\label{SobL2}
Let $G$ be a compact Lie group and let us denote by $Q$ the Hausdorff
dimension of $G$ associated to the control distance associated to the sub-Laplacian $\mathcal{L}=\mathcal{L}_X,$ where  $X= \{X_{1},\cdots,X_k\} $ is a system of vector fields satisfying the H\"ormander condition of order $\kappa$.  For  $0\leqslant \delta<\rho\leqslant 1,$  let us consider a continuous linear operator $A:C^\infty(G)\rightarrow\mathscr{D}'(G)$ with symbol  $\sigma\in {S}^{-m,\mathcal{L}}_{\rho,\delta}(G\times \widehat{G})$, $m\geqslant   0$. 
If $s\in \mathbb{R},$ then  $A$ extends to a bounded operator from $L^{p,\mathcal{L}}_{s}(G)$ into $L^{p,\mathcal{L}}_{s}(G),$   and also from $B^{s,\mathcal{L}}_{p,q}(G)$ into $B^{s,\mathcal{L}}_{p,q}(G),$ for all $1<p<\infty,$ and $0<q<\infty,$  provided that 
\[ 
    m\geqslant   m_p:= Q(1-\rho)\left|\frac{1}{p}-\frac{1}{2}\right|.
\]
\end{corollary}
 \begin{proof}
 Let us assume that $s\in\mathbb{R}.$ Note that
 \[ 
     \Vert Af \Vert_{ L^{p,\mathcal{L}}_{s}(G) }=\Vert \mathcal{M}_{s }A \mathcal{M}_{-s} \mathcal{M}_{s} f \Vert_{ L^{p}(G) }.
 \] Because $\mathcal{M}_{s }\in \textnormal{Op}({S}^{s,\mathcal{L}}_{1,0}(G\times \widehat{G})),$   $ A\in \textnormal{Op}({S}^{-m,\mathcal{L}}_{\rho,\delta}(G\times \widehat{G})),$  and  $\mathcal{M}_{-s} \textnormal{Op}(\mathscr{S}^{-s,\mathcal{L}}_{1,0}(G\times \widehat{G})), $ we have that
 \[ 
    A_s:= \mathcal{M}_{s }A \mathcal{M}_{-s} \in \textnormal{Op}({S}^{-m,\mathcal{L}}_{\rho,\delta}(G\times \widehat{G}),
 \] and from Theorem \ref{parta}, we deduce that $A_{s}$ extends to a bounded operator on $L^p(G)$ for all $1<p<\infty.$ Consequently we deduce the estimate
 \[ 
      \Vert Af \Vert_{ L^{p,\mathcal{L}}_{s}(G) }\leqslant \Vert A_s \Vert_{\mathscr{B}(L^p(G))}\Vert f\Vert_{L^{p,\mathcal{L}}_{s}(G)}.
 \] So, $A$ extends to a bounded operator from $L^{p,\mathcal{L}}_{s}(G)$ into $L^{p,\mathcal{L}}_{s}(G).$ In a similar way it can be proved that $A$ extends to a bounded operator from  $L^{p,\mathcal{L}}_{-s}(G)$ into $L^{p,\mathcal{L}}_{-s}(G)$ for all $1<p<\infty.$  The interpolation argument as in Corollary \ref{BesovL} allows us to deduce the corresponding boundedness result on Besov spaces. Thus, we finish the proof.
 \end{proof}
\begin{corollary}\label{SobL3}
Let $G$ be a compact Lie group and let us denote by $Q$ the Hausdorff
dimension of $G$ associated to the control distance associated to the sub-Laplacian $\mathcal{L}=\mathcal{L}_X,$ where  $X= \{X_{1},\cdots,X_k\} $ is a system of vector fields satisfying the H\"ormander condition of order $\kappa$.  For  $0\leqslant \delta<\rho\leqslant 1,$ let us consider a continuous linear operator $A:C^\infty(G)\rightarrow\mathscr{D}'(G)$ with symbol  $\sigma\in {S}^{m,\mathcal{L}}_{\rho,\delta}(G\times \widehat{G})$, $m\in\mathbb{R}$. 
If $s\in \mathbb{R},$ then  $A$ extends to a bounded operator from $L^{2,\mathcal{L}}_{s}(G)$ into $L^{2,\mathcal{L}}_{s-m}(G).$
\end{corollary}
\begin{proof}    Observe that $\mathcal{M}_{-m}A $ extends to a bounded operator on $L^2(G)$ in view of the subelliptic Calder\'on-Vaillancourt Theorem (Theorem \ref{CVT}). In view of Corollary \ref{SobL3} applied to $p=2$ we have that $A \mathcal{M}_{-m}$ extends to a bounded operator from  $L^{2,\mathcal{L}}_{s}(G)$ to $L^{p,\mathcal{L}}_{s}(G)$ which is equivalent to the boundedness of $A$  from $L^{2,\mathcal{L}}_{s}(G)$ into $L^{2,\mathcal{L}}_{s-m}(G).$  Thus, we finish the proof.
\end{proof}

\section{Construction of parametrices and regularisation of traces}\label{LEH}

In this section we will study the notion of ellipticity  associated to the subelliptic calculus. As in the theory of pseudo-differential operators on compact manifolds (see H\"ormander \cite{Hormander1985III}) the ellipticity notion can be applied to study some singularity order appearing in heat traces and regularisation of traces (see Wodzicki \cite{Wodzicki} and Kontsevich and Vishik \cite{KV}). So,  we will study the analogy of such  kind of traces for subelliptic operators.
\subsection{Construction of parametrices} Now, we will present a technical result about the existence of parametrices for $\mathcal{L}$-elliptic operators (see Definition \ref{lellipticity}) in the subelliptic calculus. We denote $ {S}^{-\infty,\mathcal{L}}(G\times \widehat{G})=\cap_{m\in \mathbb{R}}{S}^{m,\mathcal{L}}_{\rho,\delta}(G\times \widehat{G}).$
\begin{proposition}\label{IesTParametrix} Let $m\in \mathbb{R},$ and let $0\leqslant \delta<\rho\leqslant 1.$  Let  $a=a(x,\xi)\in {S}^{m,\mathcal{L}}_{\rho,\delta}(G\times \widehat{G}).$  Assume also that $a(x,\xi)$ is invertible for every $(x,[\xi])\in G\times\widehat{G},$ and satisfies
\begin{equation}\label{Iesparametrix}
   \sup_{(x,[\xi])\in G\times \widehat{G}} \Vert \widehat{\mathcal{M}}(\xi)^{m}a(x,\xi)^{-1}\Vert_{\textnormal{op}}<\infty.
\end{equation}Then, there exists $B\in {S}^{-m,\mathcal{L}}_{\rho,\delta}(G\times \widehat{G}),$ such that $AB-I,BA-I\in {S}^{-\infty,\mathcal{L}}(G\times \widehat{G}). $ Moreover, the symbol of $B$ satisfies the following asymptotic expansion
\begin{equation}\label{AE}
    \widehat{B}(x,\xi)\sim \sum_{N=0}^\infty\widehat{B}_{N}(x,\xi),\,\,\,(x,[\xi])\in G\times \widehat{G},
\end{equation}where $\widehat{B}_{N}(x,\xi)\in {S}^{-m-(\rho-\delta)N,\mathcal{L}}_{\rho,\delta}(G\times \widehat{G})$ obeys to the recursive  formula
\begin{equation}\label{conditionelip}
    \widehat{B}_{N}(x,\xi)=-a(x,\xi)^{-1}\left(\sum_{k=0}^{N-1}\sum_{|\gamma|=N-k}(\Delta_{\xi}^\gamma a(x,\xi))(\partial_{X}^{(\gamma)}\widehat{B}_{k}(x,\xi))\right),\,\,N\geqslant 1,
\end{equation}with $ \widehat{B}_{0}(x,\xi)=a(x,\xi)^{-1}.$
\end{proposition} 
\begin{proof}
The idea is to find a symbol $\widehat{B}$ such that if $\mathcal{I}=AB,$ then $\mathcal{I}-I$ is a smoothing operator, where
\begin{align*}
  \widehat{\mathcal{I}}(x,\xi)  \sim \sum_{|\alpha|= 0}^\infty( \Delta_{\xi}^{\alpha} a(x,\xi))(\partial_{X}^{(\alpha)} \widehat{B}(x,\xi)).
\end{align*} The asymptotic expansion means that for every $N\in \mathbb{N},$
\begin{align*}
    &\Delta_{\xi}^{\alpha_\ell}\partial_{X}^{(\beta)}\left(\widehat{\mathcal{I}}(x,\xi)-\sum_{|\alpha|\leqslant N}  ( \Delta_{\xi}^{\alpha} a(x,\xi))(\partial_{X}^{(\alpha)} \widehat{B}(x,\xi))  \right)\\
    &\hspace{3cm}\in {S}^{-(\rho-\delta)(N+1)-\rho\ell+\delta|\beta|,\mathcal{L}}_{\rho,\delta}(G\times \widehat{G}),
\end{align*}for every $\alpha_\ell\in \mathbb{N}_0$ of order $\ell\in\mathbb{N}_0,$ where $\widehat{B}$ is requested to satisfy the asymptotic expansion  \eqref{AE}. So, formally we can write
\begin{align*}
  \widehat{\mathcal{I}}(x,\xi) &\sim \sum_{|\alpha|= 0}^\infty( \Delta_{\xi}^{\alpha} a(x,\xi))(\partial_{X}^{(\alpha)} \widehat{B}(x,\xi))=\sum_{|\alpha|= 0}^\infty\sum_{N=0}^{\infty}( \Delta_{\xi}^{\alpha} a(x,\xi))(\partial_{X}^{(\alpha)} \widehat{B}_N(x,\xi)).
\end{align*}
  Observe that the fact that $\widehat{B}_{0}(x,\xi)\in {S}^{-m,\mathcal{L}}_{\rho,\delta}(G\times \widehat{G}) $    follows from Corollary \ref{IesTC}. Now, one can check easily that $\widehat{B}_{N}(x,\xi)\in {S}^{-m-(\rho-\delta)N,\mathcal{L}}_{\rho,\delta}(G\times \widehat{G}),$ for all $N\geqslant 1$ by using induction. Consequently,
  \begin{align*}
      \widehat{B}(x,\xi)- \sum_{j=0}^{N-1}\widehat{B}_{j}(x,\xi)\in {S}^{-m-(\rho-\delta)N,\mathcal{L}}_{\rho,\delta}(G\times \widehat{G}).
  \end{align*}This analysis allows us to deduce that
  \begin{align*}
    \widehat{\mathcal{I}}(x,\xi)-\sum_{k=0}^{N-1}\sum_{|\gamma|<N}(\Delta_{\xi}^\gamma a(x,\xi))(\partial_{X}^{(\gamma)}\widehat{B}_{k}(x,\xi))\in {S}^{-(\rho-\delta)N,\mathcal{L}}_{\rho,\delta}(G\times \widehat{G}).
  \end{align*}On the other hand, observe that
  \begin{align*}
   &\sum_{k=0}^{N-1}\sum_{|\gamma|< N}(\Delta_{\xi}^\gamma a(x,\xi))(\partial_{X}^{(\gamma)}\widehat{B}_{k}(x,\xi))\\
  & =I_{d_\xi}+\sum_{k=1}^{N-1}\left(a(x,\xi)\widehat{B}_{k}(x,\xi)+\sum_{ |\gamma|\leqslant N,\,|\gamma|\geqslant 1}(\Delta_{\xi}^\gamma a(x,\xi))(\partial_{X}^{(\gamma)}\widehat{B}_{k}(x,\xi))\right)\\
  & =I_{d_\xi}+\sum_{k=1}^{N-1}\left(a(x,\xi)\widehat{B}_{k}(x,\xi)+\sum_{ |\gamma|= N-j,\,j<k}(\Delta_{\xi}^\gamma a(x,\xi))(\partial_{X}^{(\gamma)}\widehat{B}_{j}(x,\xi))\right)\\
  &\hspace{2cm}+\sum_{ |\gamma|+j\geqslant  N,\,|\gamma|<N,\,j<N}(\Delta_{\xi}^\gamma a(x,\xi))(\partial_{X}^{(\gamma)}\widehat{B}_{k}(x,\xi))\\
   &=I_{d_\xi}+\sum_{ |\gamma|+j\geqslant  N,\,|\gamma|<N,\,j<N}(\Delta_{\xi}^\gamma a(x,\xi))(\partial_{X}^{(\gamma)}\widehat{B}_{j}(x,\xi)),
  \end{align*}where we have used that 
  $$ a(x,\xi)\widehat{B}_{k}(x,\xi)+\sum_{ |\gamma|= k-j,\,j<k}(\Delta_{\xi}^\gamma a(x,\xi))(\partial_{X}^{(\gamma)}\widehat{B}_{j}(x,\xi))\equiv 0,  $$
  in view of \eqref{conditionelip}. Because, for $|\gamma|+j\geqslant  N,$ $(\Delta_{\xi}^\gamma a(x,\xi))(\partial_{X}^{(\gamma)}\widehat{B}_{k}(x,\xi))\in {S}^{-(\rho-\delta)N,\mathcal{L}}_{\rho,\delta}(G\times \widehat{G}), $ it follows that
  \begin{align*}
      \sum_{k=0}^{N-1}\sum_{|\gamma|< N}(\Delta_{\xi}^\gamma a(x,\xi))(\partial_{X}^{(\gamma)}\widehat{B}_{k}(x,\xi))-I_{d_\xi}\in  {S}^{-(\rho-\delta)N,\mathcal{L}}_{\rho,\delta}(G\times \widehat{G}),
  \end{align*}and consequently, $\widehat{\mathcal{I}}(x,\xi)-I_{d_\xi}\in  {S}^{-(\rho-\delta)N,\mathcal{L}}_{\rho,\delta}(G\times \widehat{G}),$ for every $N\in \mathbb{N}.$ 
 So, we have proved that  $AB-I\in {S}^{-\infty,\mathcal{L}}(G\times \widehat{G}). $ A similar analysis can be used to prove that $BA-I\in {S}^{-\infty,\mathcal{L}}(G\times \widehat{G}). $  
\end{proof}

\subsection{Parameter $\mathcal{L}$-ellipticity} To develop the subelliptic fucntional calculus we need a more wide notion of ellipticity. By following the  approach in \cite{RuzhanskyWirth2014}, we present in our subelliptic context the notion of parameter $\mathcal{L}$-ellipticity. 
\begin{definition} Let $m>0,$ and let $0\leqslant \delta<\rho\leqslant 1.$
Let $\Lambda=\{\gamma(t):t\in I\footnote{where $I=[a,b],$ $-\infty<a\leqslant b<\infty,$ $I=[a,\infty),$ $I=(-\infty,b]$ or $I=(-\infty,\infty).$}\}$  be an analytic curve in the complex plane $\mathbb{C}.$ If $I$ is a finite interval we assume that $\Lambda$ is a closed curve. For  simplicity, if $I$ is an infinite interval we assume that $\Lambda$ is homotopy equivalent to the line $\Lambda_{i\mathbb{R}}:=\{iy:-\infty<y<\infty\}.$  Let  $a=a(x,\xi)\in {S}^{m,\mathcal{L}}_{\rho,\delta}(G\times \widehat{G}).$  Assume also that $R_{\lambda}(x,\xi)^{-1}:=a(x,\xi)-\lambda $\footnote{We have denoted $a(x,\xi)-\lambda:=a(x,\xi)-\lambda I_{d_\xi}$ to simplify the notation.} is invertible for every $(x,[\xi])\in G\times\widehat{G},$ and $\lambda \in  \Lambda.$ We say that $a$ is parameter $\mathcal{L}$-elliptic with respect to $\Lambda,$ if
\[ 
    \sup_{\lambda\in \Lambda}\sup_{(x,[\xi])\in G\times \widehat{G}}\Vert (|\lambda|^{\frac{1}{m}}+\widehat{\mathcal{M}}(\xi))^{m}R_{\lambda}(x,\xi)\Vert_{\textnormal{op}}<\infty.
\]
\end{definition}
\begin{remark}
Observe that for $a=b=0,$ $I=\{0\},$ and for the trivial curve $\gamma(t)=0,$ that $a$ is parameter $\mathcal{L}$-elliptic with respect to $\Lambda=\{0\},$ is equivalent to saying that $a$ is  $\mathcal{L}$-elliptic.
\end{remark}
The following theorem classifies the matrix resolvent $R_{\lambda}(x,\xi)$ of a parameter $\mathcal{L}$-elliptic symbol $a.$
\begin{theorem}\label{lambdalambdita} Let $m>0,$ and let $0\leqslant \delta<\rho\leqslant 1.$ If $a$ is parameter $\mathcal{L}$-elliptic with respect to $\Lambda,$ the following estimate 
\[ 
   \sup_{\lambda\in \Lambda}\sup_{(x,[\xi])\in G\times \widehat{G}}\Vert (|\lambda|^{\frac{1}{m}}+\widehat{\mathcal{M}}(\xi))^{m(k+1)}\widehat{\mathcal{M}}(\xi)^{\rho|\alpha|-\delta|\beta|}\partial_{\lambda}^k\partial_{X}^{(\beta)}\Delta_{\xi}^{\alpha}R_{\lambda}(x,\xi)\Vert_{\textnormal{op}}<\infty,
\]holds true for all $\alpha,\beta\in \mathbb{N}_0^n$ and $k\in \mathbb{N}_0.$

\end{theorem}
\begin{proof}
    We will split the proof in the cases $|\lambda|\leqslant 1,$  and $|\lambda|> 1,$ where $\lambda\in \Lambda.$ It is possible however that one of these two cases could be trivial in the sense that $\Lambda_{1}:=\{\lambda\in \Lambda:|\lambda|\leqslant 1\}$ or $\Lambda_{1}^{c}:=\{\lambda\in \Lambda:|\lambda|> 1\}$ could be empty sets. In such a case the proof is self-contained in the situation that we will consider where we assume that $\Lambda_{1}$ and $\Lambda_{1}^c$ are not trivial sets.   For $|\lambda|\leqslant 1,$ observe that
\begin{align*}
    &\Vert (|\lambda|^{\frac{1}{m}}+\widehat{\mathcal{M}}(\xi))^{m(k+1)}\widehat{\mathcal{M}}(\xi)^{\rho|\alpha|-\delta|\beta|}\partial_{\lambda}^k\partial_{X}^{(\beta)}\Delta_{\xi}^{\alpha}R_{\lambda}(x,\xi)\Vert_{\textnormal{op}}\\
    &=\Vert (|\lambda|^{\frac{1}{m}}+\widehat{\mathcal{M}}(\xi))^{m(k+1)}\widehat{\mathcal{M}}(\xi)^{-m(k+1)}\widehat{\mathcal{M}}(\xi)^{m(k+1)+\rho|\alpha|-\delta|\beta|}\partial_{\lambda}^k\partial_{X}^{(\beta)}\Delta_{\xi}^{\alpha}R_{\lambda}(x,\xi)\Vert_{\textnormal{op}}\\
    &\leqslant \Vert (|\lambda|^{\frac{1}{m}}+\widehat{\mathcal{M}}(\xi))^{m(k+1)}\widehat{\mathcal{M}}(\xi)^{-m(k+1)}\Vert_{\textnormal{op}}\\
    &\hspace{5cm}\times\Vert\widehat{\mathcal{M}}(\xi)^{m(k+1)+\rho|\alpha|-\delta|\beta|}\partial_{\lambda}^k\partial_{X}^{(\beta)}\Delta_{\xi}^{\alpha}R_{\lambda}(x,\xi)\Vert_{\textnormal{op}}.
\end{align*} Note that
\begin{align*}
    &\Vert (|\lambda|^{\frac{1}{m}}+\widehat{\mathcal{M}}(\xi))^{m(k+1)}\widehat{\mathcal{M}}(\xi)^{-m(k+1)}\Vert_{\textnormal{op}}\\
    &= \Vert (|\lambda|^{\frac{1}{m}}\widehat{\mathcal{M}}(\xi)^{-1}+I_{d_\xi})^{m(k+1)}\Vert_{\textnormal{op}}\leqslant \Vert |\lambda|^{\frac{1}{m}}\widehat{\mathcal{M}}(\xi)^{-1}+I_{d_\xi}\Vert^{m(k+1)}_{\textnormal{op}}\\
    &\leqslant \sup_{|\lambda|\in [0,1]}\sup_{1\leqslant j\leqslant d_\xi}(|\lambda|^{\frac{1}{m}}(1+\nu_{jj}(\xi)^2)^{-\frac{1}{2}}+1))^{k(m+1)}\\
    &=O(1).
\end{align*} On the other hand, we can prove that
\begin{align*}
   \Vert\widehat{\mathcal{M}}(\xi)^{m(k+1)+\rho|\alpha|-\delta|\beta|}\partial_{\lambda}^k\partial_{X}^{(\beta)}\Delta_{\xi}^{\alpha}R_{\lambda}(x,\xi)\Vert_{\textnormal{op}}=O(1).
\end{align*} For $k=1,$ $\partial_{\lambda}R_{\lambda}(x,\xi)= R_{\lambda}(x,\xi)^{2}.$ This can be deduced from the Leibniz rule, indeed,
\begin{align*}
 0&=\partial_{\lambda}(R_{\lambda}(x,\xi)(a(x,\xi)-\lambda))=(\partial_{\lambda}R_{\lambda}(x,\xi)) (a(x,\xi)-\lambda)+ R_{\lambda}(x,\xi)\partial_{\lambda}(a(x,\xi)-\lambda)\\
 &=(\partial_{\lambda}R_{\lambda}(x,\xi)) (a(x,\xi)-\lambda)+ R_{\lambda}(x,\xi)(-1) 
\end{align*}implies that
    \begin{align*}
 -\partial_{\lambda}(R_{\lambda}(x,\xi))(a(x,\xi)-\lambda)= R_{\lambda}(x,\xi)(-1). 
\end{align*} Because $(a(x,\xi)-\lambda)=R_{\lambda}(x,\xi)^{-1}$ the identity for the first derivative of $R_\lambda,$ $\partial_{\lambda}R_{\lambda}$ it follows. So, from the chain rule we obtain that the term of higher order expanding the derivative   $ \partial_{\lambda}^kR_{\lambda} $ is $ R_{\lambda}^{k+1}.$ From  Corollary \ref{IesTC} we deduce that $R_{\lambda}\in S^{-m,\mathcal{L}}_{\rho,\delta}(G\times \widehat{G}).$ The subelliptic calculus implies that $R_{\lambda}^{k+1}\in S^{-m(k+1),\mathcal{L}}_{\rho,\delta}(G\times \widehat{G}).$ This fact, and the compactness of $\Lambda_1\subset \mathbb{C},$ provide us the uniform estimate 
    \[ 
   \sup_{\lambda\in \Lambda_1}\sup_{(x,[\xi])\in G\times \widehat{G}}\Vert\widehat{\mathcal{M}}(\xi)^{m(k+1)+\rho|\alpha|-\delta|\beta|}\partial_{\lambda}^k\partial_{X}^{(\beta)}\Delta_{\xi}^{\alpha}R_{\lambda}(x,\xi)\Vert_{\textnormal{op}}<\infty.
\]Now, we will analyse the situation for $\lambda\in \Lambda_1^c.$ We will use induction over $k$ in order to prove that
\[ 
   \sup_{\lambda\in \Lambda_1^c}\sup_{(x,[\xi])\in G\times \widehat{G}}\Vert (|\lambda|^{\frac{1}{m}}+\widehat{\mathcal{M}}(\xi))^{m(k+1)}\widehat{\mathcal{M}}(\xi)^{\rho|\alpha|-\delta|\beta|}\partial_{\lambda}^k\partial_{X}^{(\beta)}\Delta_{\xi}^{\alpha}R_{\lambda}(x,\xi)\Vert_{\textnormal{op}}<\infty.
\]
For $k=0$ notice that
\begin{align*}
     &\Vert (|\lambda|^{\frac{1}{m}}+\widehat{\mathcal{M}}(\xi))^{m(k+1)}\widehat{\mathcal{M}}(\xi)^{\rho|\alpha|-\delta|\beta|}\partial_{\lambda}^k\partial_{X}^{(\beta)}\Delta_{\xi}^{\alpha}R_{\lambda}(x,\xi)\Vert_{\textnormal{op}}\\
     &=\Vert (|\lambda|^{\frac{1}{m}}+\widehat{\mathcal{M}}(\xi))^{m}\widehat{\mathcal{M}}(\xi)^{\rho|\alpha|-\delta|\beta|}\partial_{X}^{(\beta)}\Delta_{\xi}^{\alpha}(a(x,\xi)-\lambda)^{-1}\Vert_{\textnormal{op}},
\end{align*}and denoting $\theta=\frac{1}{|\lambda|},$ $\omega=\frac{\lambda}{|\lambda|},$ we have
 \begin{align*}
     &\Vert (|\lambda|^{\frac{1}{m}}+\widehat{\mathcal{M}}(\xi))^{m(k+1)}\widehat{\mathcal{M}}(\xi)^{\rho|\alpha|-\delta|\beta|}\partial_{\lambda}^k\partial_{X}^{(\beta)}\Delta_{\xi}^{\alpha}R_{\lambda}(x,\xi)\Vert_{\textnormal{op}}\\
     &=\Vert (|\lambda|^{\frac{1}{m}}+\widehat{\mathcal{M}}(\xi))^{m}|\lambda|^{-1}\widehat{\mathcal{M}}(\xi)^{\rho|\alpha|-\delta|\beta|}\partial_{X}^{(\beta)}\Delta_{\xi}^{\alpha}(\theta\times  a(x,\xi)-\omega)^{-1}\Vert_{\textnormal{op}}\\
     &=\Vert (1+|\lambda|^{-\frac{1}{m}}\widehat{\mathcal{M}}(\xi))^{m}\widehat{\mathcal{M}}(\xi)^{\rho|\alpha|-\delta|\beta|}\partial_{X}^{(\beta)}\Delta_{\xi}^{\alpha}(\theta\times  a(x,\xi)-\omega)^{-1}\Vert_{\textnormal{op}}\\
     &=\Vert (1+\theta^{\frac{1}{m}}\widehat{\mathcal{M}}(\xi))^{m}\widehat{\mathcal{M}}(\xi)^{\rho|\alpha|-\delta|\beta|}\partial_{X}^{(\beta)}\Delta_{\xi}^{\alpha}(\theta\times  a(x,\xi)-\omega)^{-1}\Vert_{\textnormal{op}}\\
     &=\Vert (1+\theta^{\frac{1}{m}}\widehat{\mathcal{M}}(\xi))^{m}\widehat{\mathcal{M}}(\xi)^{-m}\widehat{\mathcal{M}}(\xi)^{m+\rho|\alpha|-\delta|\beta|}\partial_{X}^{(\beta)}\Delta_{\xi}^{\alpha}(\theta\times  a(x,\xi)-\omega)^{-1}\Vert_{\textnormal{op}}\\
     &\leqslant \Vert (1+\theta^{\frac{1}{m}}\widehat{\mathcal{M}}(\xi))^{m}\widehat{\mathcal{M}}(\xi)^{-m}\Vert_{\textnormal{op}}\Vert\widehat{\mathcal{M}}(\xi)^{m+\rho|\alpha|-\delta|\beta|}\partial_{X}^{(\beta)}\Delta_{\xi}^{\alpha}(\theta\times  a(x,\xi)-\omega)^{-1}\Vert_{\textnormal{op}}.
\end{align*}   Because $ (1+\theta^{\frac{1}{m}}\widehat{\mathcal{M}}(\xi))^{m}\widehat{\mathcal{M}}(\xi)^{-m} \in S^{0,\mathcal{L}}_{\rho,\delta}(G\times \widehat{G}) ,$ we have that the operator norm $ \Vert (1+\theta^{\frac{1}{m}}\widehat{\mathcal{M}}(\xi))^{m}\widehat{\mathcal{M}}(\xi)^{-m}\Vert_{\textnormal{op}}$ is uniformly bounded in $\theta\in [0,1].$ The same argument can be applied to the operator norm $$ \Vert\widehat{\mathcal{M}}(\xi)^{m+\rho|\alpha|-\delta|\beta|}\partial_{X}^{(\beta)}\Delta_{\xi}^{\alpha}(\theta\times  a(x,\xi)-\omega)^{-1}\Vert_{\textnormal{op}}, $$
    by using that $(\theta\times  a(x,\xi)-\omega)^{-1}\in S^{-m,\mathcal{L}}_{\rho,\delta}(G\times \widehat{G}), $ with $\theta\in [0,1]$ and $\omega$ being an element of the complex circle. The case  $k\geqslant 1$ for $\lambda\in \Lambda_1^c$ can be proved in a analogous way.
\end{proof} Combining Proposition \ref{IesTParametrix} and Theorem \ref{lambdalambdita} we obtain the following corollaries.
\begin{corollary}\label{parameterparametrix}
Let $m>0,$ and let $0\leqslant \delta<\rho\leqslant 1.$ Let  $a$ be a parameter $\mathcal{L}$-elliptic symbol with respect to $\Lambda.$ Then  there exists a parameter-dependent parametrix of $A-\lambda I,$ with symbol $a^{-\#}(x,\xi,\lambda)$ satisfying the estimates
\[ 
   \sup_{\lambda\in \Lambda}\sup_{(x,[\xi])\in G\times \widehat{G}}\Vert (|\lambda|^{\frac{1}{m}}+\widehat{\mathcal{M}}(\xi))^{m(k+1)}\widehat{\mathcal{M}}(\xi)^{\rho|\alpha|-\delta|\beta|}\partial_{\lambda}^k\partial_{X}^{(\beta)}\Delta_{\xi}^{\alpha}a^{-\#}(x,\xi,\lambda)\Vert_{\textnormal{op}}<\infty,
\]for all $\alpha,\beta\in \mathbb{N}_0^n$ and $k\in \mathbb{N}_0.$
\end{corollary}
\begin{corollary}\label{resolv}
Let $m>0,$ and let $a\in S^{m,\mathcal{L}}_{\rho,\delta}(G\times \widehat{G}) $ where  $0\leqslant \delta<\rho\leqslant 1.$ Let us assume that $\Lambda$ is a subset of the $L^2$-resolvent set of $A,$ $\textnormal{Resolv}(A):=\mathbb{C}\setminus \textnormal{Spec}(A).$ Then $A-\lambda I$ is invertible on $\mathscr{D}'(G)$ and the symbol of the resolvent operator $\mathcal{R}_{\lambda}:=(A-\lambda I)^{-1},$ $\widehat{\mathcal{R}}_{\lambda}(x,\xi)$ belongs to $S^{-m,\mathcal{L}}_{\rho,\delta}(G\times \widehat{G}).$ 
\end{corollary}

\subsection{Asymptotic expansions for  regularised traces}
In this subsection we will study the trace for the heat semigroup associated with $\mathcal{L}$-elliptic positive left-invariant operators and also  regularised traces of subelliptic operators. So, we start with the following Pleijel type formula. 
 \begin{theorem}\label{asymptotictracemultiplier}
Let $G$ be a compact Lie group and let us denote by $Q$ the Hausdorff
dimension of $G$ associated to the control distance associated to the sub-Laplacian $\mathcal{L}=\mathcal{L}_X,$ where  $X= \{X_{1},\cdots,X_k\} $ is a system of vector fields satisfying the H\"ormander condition.  For  $0\leqslant \rho\leqslant 1,$ let us consider a positive  left-invariant $\mathcal{L}$-elliptic continuous linear operator $A:C^\infty(G)\rightarrow\mathscr{D}'(G)$ with symbol  $\sigma\in {S}^{m,\mathcal{L}}_{\rho}( \widehat{G})$, $m> 0$. If $A$ commutes with $\mathcal{L},$ then the heat trace of $A$ has an asymptotic behaviour of the form
\begin{equation}\label{asymp1}
    \textnormal{\textbf{Tr}}(e^{-tA})\sim c_{m,Q} t^{-\frac{Q}{m}}\times \int\limits_{t^{\frac{1}{m}}}^{\infty}e^{-s^m}s^{Q-1}ds,\,\,\forall t>0.
\end{equation} Moreover, we have the following asymptotic expansion,
\begin{equation}
\textnormal{\textbf{Tr}}(e^{-sA})=s^{-\frac{Q}{m}}\left(  \sum_{k=0}^{\infty}a_ks^{\frac{k}{m}}\right),\,\,s\rightarrow 0^{+}  .
\end{equation}
\end{theorem}
 \begin{proof}Note that $A$ is densely defined and positive on $L^2(G),$ so it admits a self-adjoint extension.
At the level of symbols, if $A$ commutes with $\mathcal{L},$ for every $[\xi]\in \widehat{G},$ $\sigma(\mathcal{\xi})$ commutes with $\widehat{\mathcal{L}}(\xi)$  and consequently, $\sigma(\mathcal{\xi})$ and  $\widehat{\mathcal{L}}(\xi)$ are simultaneously diagonalisable on every  representation space. So, in a suitable basis of the representation space we can write,
\[ 
    \sigma(\xi)=\textnormal{diag}[\sigma_{jj}(\xi)]_{j=1}^{d_\xi},\,\,\,\widehat{\mathcal{L}}(\xi)=\textnormal{diag}[(1+\nu_{jj}(\xi)^2)^{\frac{1}{2}}]_{j=1}^{d_\xi},
\]where $\sigma_{jj}(\xi),$ $1\leqslant k\leqslant d_{\xi},$ is the system of positive eigenvalues of $\sigma(\xi),$ $[\xi]\in \widehat{G}.$ The spectral mapping theorem implies that
\[ 
    \textnormal{spectrum}(e^{-tA})=\{e^{-t\sigma_{jj}(\xi)}:1\leqslant j\leqslant d_\xi,\,\,[\xi]\in \widehat{G}\}.
    \]So, we have
    \begin{align*}
        \textnormal{\textbf{Tr}}(e^{-tA})=\sum_{[\xi]\in  \widehat{G}}\sum_{j=1}^{d_{\xi}}e^{-t\sigma_{jj}(\xi)}.\,\,\,\,
    \end{align*} The $\mathcal{L}$-ellipticity of $A,$ implies that,
    \begin{align*}
        \sup_{1\leqslant j\leqslant d_\xi}\sigma_{jj}(\xi)^{-1}(1+\nu_{jj}(\xi)^2)^{\frac{m}{2}}=\Vert \sigma(\xi)^{-1}\widehat{\mathcal{M}}(\xi)\Vert_{\textnormal{op}}
        \leqslant  \sup_{[\xi]\in \widehat{G}}\Vert \sigma(\xi)^{-1}\widehat{\mathcal{M}}(\xi)\Vert_{\textnormal{op}}<\infty.
    \end{align*}Consequently,
    \[ 
        \inf_{1\leqslant j\leqslant d_\xi}\sigma_{jj}(\xi)(1+\nu_{jj}(\xi)^2)^{-\frac{m}{2}}\geqslant     \sup_{[\xi]\in \widehat{G}}\Vert \sigma(\xi)^{-1}\widehat{\mathcal{M}}(\xi)^{m}\Vert_{\textnormal{op}}^{-1}.
    \]
Now, observe that from the hypothesis $\sigma\in {S}^{m,\mathcal{L}}_{\rho}( \widehat{G})$ we have,
\[ 
  \sup_{1\leqslant j\leqslant d_\xi}\sigma_{jj}(\xi)(1+\nu_{jj}(\xi)^2)^{-\frac{m}{2}}\leqslant   \sup_{[\xi]\in \widehat{G}}\Vert \sigma(\xi)\widehat{\mathcal{M}}(\xi)^{-m}\Vert_{\textnormal{op}}.   
\]These inequalities reduce the problem of computing the trace  $\textnormal{\textbf{Tr}}(e^{-tA})$ to compute $\textnormal{\textbf{Tr}}(e^{-t(1+\mathcal{L})^{\frac{m}{2}}}).$ Indeed,
\begin{align*}
     &\textnormal{\textbf{Tr}}(e^{-tA})=\sum_{[\xi]\in  \widehat{G}}\sum_{j=1}^{d_{\xi}}e^{-t\sigma_{jj}(\xi)}=\sum_{[\xi]\in  \widehat{G}}\sum_{j=1}^{d_{\xi}}e^{-t\sigma_{jj}(\xi) (1+\nu_{jj}(\xi)^2)^{-\frac{m}{2}}(1+\nu_{jj}(\xi)^2)^{\frac{m}{2}}      }\\
     &\asymp \sum_{[\xi]\in  \widehat{G}}\sum_{j=1}^{d_{\xi}}e^{-t(1+\nu_{jj}(\xi)^2)^{\frac{m}{2}}      }=\textnormal{\textbf{Tr}}(e^{-t(1+\mathcal{L})^{\frac{m}{2}}}).
\end{align*}Now, we will use the Weyl-law for the sub-Laplacian (see e.g. Remark \ref{weyl}). Observe that,
\begin{align*}
  \textnormal{\textbf{Tr}}(e^{-t(1+\mathcal{L})^{\frac{m}{2}}})=\sum_{k=0}^{\infty}\sum_{[\xi]:2^{k}\leqslant (1+\nu_{j'j'}(\xi)^2)^\frac{1}{2} <2^{k+1},\, \forall 1\leqslant j'\leqslant d_{\xi} } \sum_{j=1}^{d_{\xi}}e^{-t(1+\nu_{jj}(\xi)^2)^{\frac{m}{2}}      }.
\end{align*}Because,
\begin{align*}
   & \sum_{[\xi]:2^{k}\leqslant (1+\nu_{j'j'}(\xi)^2)^\frac{1}{2} <2^{k+1},\,\forall 1\leqslant j'\leqslant d_{\xi} } \sum_{j=1}^{d_{\xi}}e^{-t(1+\nu_{jj}(\xi)^2)^{\frac{m}{2}}      }\\
    &\asymp \sum_{[\xi]:2^{k}\leqslant (1+\nu_{j'j'}(\xi)^2)^\frac{1}{2} <2^{k+1},\,\forall 1\leqslant j'\leqslant d_{\xi} } {d_{\xi}}e^{-t2^{km}      }, 
\end{align*}we have 
\begin{align*}
    \textnormal{\textbf{Tr}}(e^{-t(1+\mathcal{L})^{\frac{m}{2}}})&=\sum_{k=0}^{\infty} e^{-t2^{km}}\sum_{[\xi]:2^{k}\leqslant (1+\nu_{j'j'}(\xi)^2)^\frac{1}{2} <2^{k+1},\,\forall 1\leqslant j'\leqslant d_{\xi} } d_\xi\\
    &=\sum_{k=0}^{\infty} e^{-t2^{km}}N(2^k)=\sum_{k=0}^{\infty} e^{-t2^{km}}2^{kQ}\\
    &=\sum_{k=0}^{\infty} e^{-t2^{km}}2^{k(Q-1)}2^{k}.
\end{align*}Observe that
\begin{align*}
    \sum_{k=0}^{\infty} e^{-t2^{km}}2^{k(Q-1)}2^{k}\asymp \int\limits_{1}^{\infty}e^{-t\lambda^m}\lambda^{Q-1}d\lambda=t^{-\frac{Q}{m}}\int\limits_{t^{\frac{1}{m}}}^{\infty}e^{-s^m}s^{Q-1}ds.
\end{align*}The condition $m>0,$ implies that $g(t):=\int\limits_{t}^{\infty}e^{-s^m}s^{Q-1}ds,$ is smooth and real-analytic on $\mathbb{R}^{+}:=(0,\infty),$ admitting a Taylor expansion of the form 
\begin{align*}
  g(s)=  \sum_{k=0}^{\infty}a_k's^k\,\,\,s\rightarrow 0^+.
\end{align*}
So,  we have the estimate $\textnormal{\textbf{Tr}}(e^{-sA})\sim c_{m,Q}s^{-\frac{Q}{m}}g(s),$ for some positive constant $c_{m,Q}.$ On the other hand, we deduce that $F(s):=s^{\frac{Q}{m}}\textnormal{\textbf{Tr}}(e^{-sA})$ is a real-analytic function and its Taylor expansion at $s=0,$ has the form: $\sum_{k=0}^{\infty}a_k's^{\frac{k}{m}},$ which implies the following expansion,
\[ 
\textnormal{\textbf{Tr}}(e^{-sA})=s^{-\frac{Q}{m}}\left(  \sum_{k=0}^{\infty}a_ks^{\frac{k}{m}}\right),\,\,s\rightarrow 0^{+}  .
\]
Thus, we end the proof.
\end{proof}
 
\begin{remark}Observe that under the conditions of Theorem \ref{asymptotictracemultiplier}, we have
\begin{equation}
    \textnormal{\textbf{Tr}}(e^{-tA})\sim c_{m,Q,t} t^{-\frac{Q}{m}},\,\,\forall t>0,
\end{equation} where $c_{m,Q,t}:=c_{m,Q} \int\limits_{t^{\frac{1}{m}} }^{\infty}e^{-s^m}s^{Q-1}ds.$ For $t\rightarrow\infty,$ $c_{m,Q,t}\rightarrow 0^{+},$ and in general,
\[ 
    0<c_{m,Q,t}\leqslant \int\limits_{0}^{\infty}e^{-s^m}s^{Q-1}ds=o( 1),\,\,\,0\leqslant t<\infty.
\]So, \eqref{asymp1} implies the following estimate
\[ 
     \textnormal{\textbf{Tr}}(e^{-tA})\sim c_{m,Q} t^{-\frac{Q}{m}}.
\]
\end{remark} 
 Now, we study other kind of singularities appearing in traces of the form $\textnormal{\textbf{Tr}}(Ae^{-t(1+\mathcal{L})^\frac{q}{2}}).$ To illustrate the importance of computing such traces let us recall an interesting situation that comes from spectral geometry. If $M$ is an orientable and  compact manifold without boundary, and if $E$ is a positive elliptic pseudo-differential operator of order $q>0,$ for every elliptic and positive pseudo-differential operator $A$ with order $m,$ $m\geqslant -\dim(M),$ we have 
 \begin{equation}\label{subellipticWR}
     \textnormal{\textbf{Tr}}(Ae^{-tE})\sim t^{-\frac{m+\dim(M)}{q}}\sum_{k=0}^\infty a_kt^k-\frac{b_0}{q}\log(t)+O(1).
 \end{equation}If $m>-\dim(M),$ $b_0=0,$ and for $m=-\dim(M),$ $a_{k}=0$ for every $k,$ and $b_0=\textnormal{res}(A)$ is the Wodzicki residue of $A,$ see e.g. Wodzicki \cite{Wodzicki} and Lesch \cite{Lesch}.
 Let us recall that a matrix $M\in \mathbb{K}^{\ell\times \ell},$ $\mathbb{K}=\mathbb{R}$ or $\mathbb{K}=\mathbb{C},$  is  positive if
 \[ 
   (M v,v)_{\mathbb{K}^{\ell}} \geqslant    0,\,\,\forall v\in \mathbb{K}^{\ell}.
\]
 Now, we will compute an analogy of \eqref{subellipticWR} for subelliptic operators.
\begin{theorem}\label{asymptotictracemultiplier2}
Let $G$ be a compact Lie group and let us denote by $Q$ the Hausdorff
dimension of $G$ associated to the control distance associated to the sub-Laplacian $\mathcal{L}=\mathcal{L}_X,$ where  $X= \{X_{1},\cdots,X_k\} $ is a system of vector fields satisfying the H\"ormander condition of order $\kappa$.  For  $0\leqslant\delta, \rho\leqslant 1,$ let us consider an  $\mathcal{L}$-elliptic continuous linear operator $A:C^\infty(G)\rightarrow\mathscr{D}'(G)$ with symbol  $\sigma\in {S}^{m,\mathcal{L}}_{\rho,\delta}(G\times  \widehat{G})$, $m\in \mathbb{R} $. Let $\mathcal{M}_q=(1+\mathcal{L})^{\frac{q}{2}}$ be the subelliptic Bessel potential of order $q>0$. If   $\sigma (x,[\xi])\geqslant     0,$ for every $(x,[\xi])\in G\times \widehat{G},$ 
then 
\begin{equation}\label{asymp122}
    \textnormal{\textbf{Tr}}(Ae^{-t\mathcal{M}_q})\sim c_{m,Q} t^{-\frac{Q+m}{q}}\times \int\limits_{t^{\frac{1}{q}}}^{\infty}e^{-s^m}s^{Q+m-1}ds,\,\,\forall t>0.
\end{equation}In particular, for $m=-Q,$ we have
\begin{equation}\label{asymp1212}
    \textnormal{\textbf{Tr}}(Ae^{-t\mathcal{M}_q})\sim -\frac{c_{Q}}{q}\log(t),\,\,\forall t\in (0,1),
\end{equation}while for $m>-Q,$ we have the asymptotic expansion
\begin{equation}\label{asymp1212'}
\textnormal{\textbf{Tr}}(Ae^{-t\mathcal{M}_q})=t^{-\frac{Q+m}{q}}\left(  \sum_{k=0}^{\infty}a_kt^\frac{k}{q}\right),\,\,t\rightarrow 0^{+}  .
\end{equation}
\end{theorem}
\begin{proof}
We will follow the same approach as in Theorem \ref{asymptotictracemultiplier}. Because the trace of $Ae^{-t\mathcal{M}_q}$ is the integral of its Schwartz kernel over the diagonal (this is a consequence of the main results in \cite{DelRuzTrace1111}), we have

 \begin{align*}
     \textnormal{\textbf{Tr}}(Ae^{-t\mathcal{M}_q})&=\int\limits_{G}\sum_{[\xi]\in  \widehat{G}}d_{\xi}\textnormal{\textbf{Tr}}[\sigma(x,\xi)e^{-t\widehat{\mathcal{M}}(\xi)^{q}}]dx\\
     &=\int\limits_{G}\sum_{[\xi]\in  \widehat{G}}d_{\xi}\textnormal{\textbf{Tr}}[\sigma(x,\xi)  \widehat{\mathcal{M}}(\xi)^{-m}\widehat{\mathcal{M}}(\xi)^{m} e^{-t\widehat{\mathcal{M}}(\xi)^{q}}]dx.
\end{align*}In a suitable basis of the representation space we can diagonalise the operator $\sigma(x,\xi)  \widehat{\mathcal{M}}(\xi)^{-m},$ and we can write in  such a basis,
\begin{equation}
    \sigma(x,\xi)  \widehat{\mathcal{M}}(\xi)^{-m}=\textnormal{diag}[\lambda_{jj}(x,\xi)]_{j=1}^{d_\xi},\,\,\widehat{\mathcal{M}}(\xi)^{m} e^{-t\widehat{\mathcal{M}}(\xi)^{q}}=[\Omega_{ij,t}(\xi)]_{i,j=1}^{d_\xi}.
\end{equation}
Now, we can write
\begin{align*}
     \textnormal{\textbf{Tr}}(Ae^{-t\mathcal{M}_q})
     &=\int\limits_{G}\sum_{[\xi]\in  \widehat{G}}d_{\xi}\textnormal{\textbf{Tr}}[\sigma(x,\xi)  \widehat{\mathcal{M}}(\xi)^{-m}\widehat{\mathcal{M}}(\xi)^{m} e^{-t\widehat{\mathcal{M}}(\xi)^{q}}]dx\\
     &=\sum_{[\xi]\in  \widehat{G}}\sum_{j,j'=1}^{d_\xi}d_{\xi}\int\limits_{G}\lambda_{j'j'}(x,\xi)dx\,\Omega_{j'j,t}(\xi).
\end{align*}
The $\mathcal{L}$-ellipticity of $A$ and the positivity  of its symbol, imply that 
\begin{align*}
  \sup_{(x,[\xi])\in \widehat{G}}\Vert \sigma(x,\xi)^{-1}\widehat{\mathcal{M}}(\xi)^{m}\Vert_{\textnormal{op}}^{-1}&\leqslant   \inf_{x\in G}\inf_{j',j',[\xi]\in \widehat{G}}\lambda_{j'j'}(x,\xi)\\
  &\leqslant \sup_{x\in G}\sup_{j',j',[\xi]\in \widehat{G}}\lambda_{j'j'}(x,\xi)\\ &=\sup_{(x,[\xi])\in \widehat{G}}\Vert \sigma(x,\xi)\widehat{\mathcal{M}}(\xi)^{-m}\Vert_{\textnormal{op}},\hspace{2cm}
\end{align*}from which we deduce the following estimate,
\begin{align*}
     \textnormal{\textbf{Tr}}(Ae^{-t\mathcal{M}_q})
     &\asymp \sum_{[\xi]\in  \widehat{G}}\sum_{j,j'=1}^{d_\xi}d_{\xi}\Omega_{j'j,t}(\xi).
\end{align*} Because
$\widehat{\mathcal{M}}(\xi)^{m} e^{-t\widehat{\mathcal{M}}(\xi)^{q}}=[\Omega_{ij,t}(\xi)]_{i,j=1}^{d_\xi}$ is a symmetric matrix written in the basis that allows to write in a  diagonal form the operator $\sigma(x,\xi)\widehat{\mathcal{M}}(\xi)^{-m},$ we can find a matrix $P(\xi)$ such that
\[ 
    \widehat{\mathcal{M}}(\xi)^{m} e^{-t\widehat{\mathcal{M}}(\xi)^{q}}=P(\xi)^{-1}  \textnormal{diag}[  (1+\nu_{jj}(\xi)^2)^{\frac{m}{2}} e^{-t(1+\nu_{jj}(\xi)^2)^{\frac{q}{2}}}]_{j=1}^{d_\xi}   P(\xi),
\]and 
\begin{align*}
   & \sum_{j,j'=1}^{d_\xi}\Omega_{j'j,t}(\xi)=\sum_{j,j'=1}^{d_\xi}[P(\xi)^{-1}  \textnormal{diag}[  (1+\nu_{ss}(\xi)^2)^{\frac{m}{2}} e^{-t(1+\nu_{ss}(\xi)^2)^{\frac{q}{2}}}]_{s=1}^{d_\xi}   P(\xi)]_{j'j}\\
    &=\sum_{j,j's=1}^{d_\xi}P(\xi)^{-1}_{j'j}    (1+\nu_{j'j'}(\xi)^2)^{\frac{m}{2}} e^{-t(1+\nu_{j'j'}(\xi)^2)^{\frac{q}{2}}}   P(\xi)]_{j's}\\
    &=\textnormal{\textbf{Tr}}[P(\xi)^{-1}  \textnormal{diag}[  (1+\nu_{jj}(\xi)^2)^{\frac{m}{2}} e^{-t(1+\nu_{jj}(\xi)^2)^{\frac{q}{2}}}]_{j=1}^{d_\xi}   P(\xi)]=\textnormal{\textbf{Tr}}[\widehat{\mathcal{M}}(\xi)^{m} e^{-t\widehat{\mathcal{M}}(\xi)^{q}}].
\end{align*}
Now, as above, we will use the Weyl-law for the sub-Laplacian (see e.g. Remark \ref{weyl}). Observe that,
\begin{align*}
  &\textnormal{\textbf{Tr}}( (1+\mathcal{L})^{\frac{m}{2}} e^{-t(1+\mathcal{L})^{\frac{q}{2}}})\\
  \quad&=\sum_{k=0}^{\infty}\sum_{[\xi]:2^{k}\leqslant (1+\nu_{j'j'}(\xi)^2)^\frac{1}{2} <2^{k+1},\,\forall 1\leqslant j'\leqslant d_{\xi} } \sum_{j=1}^{d_{\xi}}(1+\nu_{jj}(\xi)^2)^\frac{m}{2}e^{-t(1+\nu_{jj}(\xi)^2)^{\frac{q}{2}}      }.
\end{align*}Because,
\begin{align*}
    &\sum_{[\xi]:2^{k}\leqslant (1+\nu_{j'j'}(\xi)^2)^\frac{1}{2} <2^{k+1},\,\forall 1\leqslant j'\leqslant d_{\xi} } \sum_{j=1}^{d_{\xi}}(1+\nu_{jj}(\xi)^2)^\frac{m}{2}e^{-t(1+\nu_{jj}(\xi)^2)^{\frac{q}{2}}      }\\
    \quad &\asymp \sum_{[\xi]:2^{k}\leqslant (1+\nu_{j'j'}(\xi)^2)^\frac{1}{2} <2^{k+1},\,\forall 1\leqslant j'\leqslant d_{\xi} } {d_{\xi}} 2^{km} e^{-t2^{kq}      }, 
\end{align*}we can write 
\begin{align*}
    \textnormal{\textbf{Tr}}(Ae^{-t\mathcal{M}_q})&\asymp \sum_{k=0}^{\infty} 2^{km}e^{-t2^{kq}}\sum_{[\xi]:2^{k}\leqslant (1+\nu_{j'j'}(\xi)^2)^\frac{1}{2} <2^{k+1},\,\forall 1\leqslant j'\leqslant d_{\xi} } d_\xi\\
    &=\sum_{k=0}^{\infty} 2^{km}e^{-t2^{kq}}N(2^k)=\sum_{k=0}^{\infty} 2^{km}e^{-t2^{kq}}2^{kQ}\\
    &=\sum_{k=0}^{\infty} e^{-t2^{kq}}2^{k(Q+m-1)}2^{k}.
\end{align*}From the estimate
\begin{align*}
    \sum_{k=0}^{\infty} e^{-t2^{kq}}2^{k(Q+m-1)}2^{k}\asymp \int\limits_{1}^{\infty}e^{-t\lambda^q}\lambda^{Q+m-1}d\lambda=t^{-\frac{Q+m}{q}}\int\limits_{t^{\frac{1}{q}}}^{\infty}e^{-s^q}s^{Q+m-1}ds,
\end{align*} we have proved the first part of the theorem. Now, in particular, for $m=-Q,$ we have
\begin{align*}
   \textnormal{\textbf{Tr}}(Ae^{-t\mathcal{M}_q})\sim \int\limits_{t^{\frac{1}{q}}}^{\infty}e^{-s^q}s^{-1}ds. 
\end{align*}Observe that for $0<t<1,$ the main contribution in the integral $\int\limits_{t^{\frac{1}{q}}}^{\infty}e^{-s^q}s^{-1}ds$ is the integral of $G(s):=e^{-s^q}s^{-1},$ on the interval $[t^{\frac{1}{q}},1).$ Indeed, $\int\limits_{1}^{\infty}e^{-s^q}s^{-1}ds=o(1)$ for $q>0.$ Now, we can compute
\begin{align*}
 \int\limits_{t^{\frac{1}{q}}}^{1}e^{-s^q}s^{-1}ds  \sim  \int\limits_{t^{\frac{1}{q}}}^{1}s^{-1}ds=-\frac{1}{q}\log(t).
\end{align*}In the case $m>-Q,$ we have that the function $g(t)=\int\limits_{t}^{\infty}e^{-s^q}s^{Q+m-1}ds<\infty,$ is real analytic in $[0,\infty),$ and for $t\rightarrow 0^+,$ $g(t)=\sum_{k=0}^\infty b_{k}t^k,$ which implies 
\begin{align*}
    \textnormal{\textbf{Tr}}(Ae^{-t\mathcal{M}_q})\sim t^{-\frac{Q+m}{q}}\left(  \sum_{k=0}^{\infty}b_k't^{\frac{k}{q}}\right),\,\,t\rightarrow 0^{+} .
\end{align*}
So, we end the proof.
\end{proof}
\begin{remark}\label{remarkresidue}
Observe that we can summarise \eqref{asymp1212} and \eqref{asymp1212'} by writing 
\begin{equation}\label{summarisingformulas}
     \textnormal{\textbf{Tr}}(Ae^{-t(1+\mathcal{L})^\frac{q}{2}})\sim t^{-\frac{m+Q}{2}}\sum_{k=0}^\infty a_kt^\frac{k}{q}-\frac{b_0}{q}\log(t),\quad t\rightarrow0^{+},
 \end{equation}for  $m\geqslant    -Q.$ If $m=-Q,$ then $a_k=0$ for every $k,$ and for $m>-Q,$ $b_0=0.$
\end{remark}

\begin{example}Let us assume that $a(x)$ is an integrable  function over $G.$ Let $P=a(x)A,$ where $A\in S^{m,\mathcal{L}}_\rho(\widehat{G}),$ $0\leqslant \rho\leqslant 1,$ is a positive pseudo-differential operator of order $m\geqslant    -Q.$ Let us assume that $A$ is an $\mathcal{L}$-elliptic operator which commutes with $\mathcal{L}.$ Because the $L^2$-trace of $Pe^{-t(1+\mathcal{L})^\frac{q}{2}}$ is the integral of its kernel over the diagonal (see \cite{DelRuzTrace1111}), we have
\begin{align*}
  \textnormal{\textbf{Tr}}(Pe^{-t(1+\mathcal{L})^\frac{q}{2}}) =\int\limits_{G}\sum_{[\xi]\in \widehat{G}}\textnormal{\textbf{Tr}}(a(x)\sigma(\xi)e^{-t\widehat{\mathcal{M}}_q(\xi)})dx=\int\limits_{G}a(g)dg\times \textnormal{\textbf{Tr}}(Ae^{-t(1+\mathcal{L})^\frac{q}{2}}). 
\end{align*}This implies that $P$ also admits an asymptotic expansion of the form
\[ \textnormal{\textbf{Tr}}(Pe^{-t(1+\mathcal{L})^\frac{q}{2}})\sim t^{-\frac{m+Q}{q}}\sum_{k=0}^\infty a_kt^\frac{k}{q}-\frac{b_0}{q}\int\limits_{G}a(x)dx\log(t),\quad t\rightarrow0^{+}.
\]
\end{example}
Now, if in Theorem \ref{asymptotictracemultiplier2}  we replace the role of the sub-Laplacian $\mathcal{L}$ by using the Laplacian $\mathcal{L}_G$ on $G,$ we obtain the following corollary.

\begin{corollary}\label{asymptotictracemultiplier2'}
Let $G$ be a compact Lie group and let us denote by $n$ its dimension.  For  $0\leqslant\delta, \rho\leqslant 1,$ let us consider an  elliptic continuous linear operator $A:C^\infty(G)\rightarrow\mathscr{D}'(G)$ with symbol  $\sigma\in \mathscr{S}^{m}_{\rho,\delta}(G\times \widehat{G})$, $m\in \mathbb{R} $. Let $B_q=(1+\mathcal{L}_G)^{\frac{q}{2}}$ be the  Bessel potential of order $q>0$. If   $\sigma (x,[\xi])\geqslant     0,$ for every $(x,[\xi])\in G\times \widehat{G},$ 
then 
\begin{equation}\label{asymp122''''}
    \textnormal{\textbf{Tr}}(Ae^{-tB_q})\sim c_{m,n} t^{-\frac{n+m}{q}}\times \int\limits_{t^{\frac{1}{q}}}^{\infty}e^{-s^m}s^{n+m-1}ds,\,\,\forall t>0.
\end{equation}In particular, for $m=-n,$ we have
\begin{equation}\label{asymp1212''''}
    \textnormal{\textbf{Tr}}(Ae^{-tB_q})\sim -\frac{c_{n}}{q}\log(t),\,\,\forall t\in (0,1),
\end{equation}while for $m>-n,$ we have the asymptotic expansion
\begin{equation}\label{asymp1212'''}
\textnormal{\textbf{Tr}}(Ae^{-tB_q})=t^{-\frac{n+m}{q}}\left(  \sum_{k=0}^{\infty}a_kt^\frac{k}{q}\right),\,\,t\rightarrow 0^{+}  .
\end{equation}
\end{corollary}
\begin{remark}
It is obvious that Corollary \ref{asymptotictracemultiplier2'} follows from
\eqref{subellipticWR} in the case $(\rho,\delta)=(1,0).$ However, the notion of ellipticity in the H\"ormander classes allows us to extend this kind of asymptotic expansions in the complete range $0\leqslant \delta, \rho\leqslant 1,$ without the natural assumption $\delta<\rho.$
\end{remark}
Now, we will study regularised traces of the form $\textnormal{\bf{Tr}}(A\psi(t E))$ where $t\in \mathbb{R},$ $\psi$ is a compactly supported real-valued function and $E$ is an $\mathcal{L}$-elliptic positive left-invariant operator of order $q>0.$

\begin{theorem}\label{asymptotictraceeta}
Let $G$ be a compact Lie group and let us denote by $Q$ the Hausdorff
dimension of $G$ associated to the control distance associated to the sub-Laplacian $\mathcal{L}=\mathcal{L}_X,$ where  $X= \{X_{1},\cdots,X_k\} $ is a system of vector fields satisfying the H\"ormander condition.  For  $0\leqslant\delta, \rho\leqslant 1,$ let us consider an  $\mathcal{L}$-elliptic continuous linear operator $A:C^\infty(G)\rightarrow\mathscr{D}'(G)$ with symbol  $\sigma\in {S}^{m,\mathcal{L}}_{\rho,\delta}(G\times \widehat{G})$, $m\in \mathbb{R} $. Let $E$ be positive $\mathcal{L}$-elliptic left-invariant operator  of order $q>0$. If   $\sigma (x,[\xi])\geqslant     0,$ for every $(x,[\xi])\in G\times \widehat{G},$ 
then 
\begin{equation}\label{asymp122222}
    \textnormal{\bf{Tr}}(A\psi(t E))\sim \frac{1}{q}\int\limits_{0}^{\infty}\psi(s)\times\frac{ds}{s},\,\,\forall t>0,
\end{equation}provided that $\psi\in L^{1}(\mathbb{R}^{+}_0;\frac{ds}{s})\cap C^{\infty}_0(\mathbb{R}^{+}_0),$ and $m=-Q.$ On the other hand, for $m>-Q$ and $\psi\in C^{\infty}_{0}(\mathbb{R}^+_0),$ we have
\begin{equation}\label{asymp1212222}
     \textnormal{\bf{Tr}}(A\psi(t E))\sim \frac{t^{-\frac{1}{q}(Q+m)}}{q}\int\limits_{0}^\infty\psi(s)s^{\frac{Q+m}{q}}\times \frac{ds}{s},\,\,\forall t>0.
\end{equation}So, we have the asymptotic expansion
\begin{equation}\label{asymp1212'22}
\textnormal{\bf{Tr}}(A\psi(t E))=t^{-\frac{Q+m}{q}}\left(  \sum_{k=0}^{\infty}a_kt^k\right)+\frac{c_{Q}}{q}\int\limits_{0}^{\infty}\psi(s)\frac{ds}{s},\,\,t\rightarrow 0^{+} ,
\end{equation}for $m\geqslant -Q,$
\end{theorem}
\begin{proof}
By writing the trace of $A\psi(t E)$ as the integral of its Schwartz kernel at  the diagonal (see \cite{DelRuzTrace1111}), we have

 \begin{align*}
     \textnormal{\textbf{Tr}}(A\psi(t E))&=\int\limits_{G}\sum_{[\xi]\in  \widehat{G}}d_{\xi}\textnormal{\textbf{Tr}}[\sigma(x,\xi)\psi(t\widehat{E}(\xi))]dx\\
     &=\int\limits_{G}\sum_{[\xi]\in  \widehat{G}}d_{\xi}\textnormal{\textbf{Tr}}[\sigma(x,\xi)  \widehat{\mathcal{M}}(\xi)^{-m}\widehat{\mathcal{M}}(\xi)^{m} \psi(t\widehat{E}(\xi))]dx\\
     &\asymp \sum_{[\xi]\in  \widehat{G}}d_{\xi}\textnormal{\textbf{Tr}}[\widehat{\mathcal{M}}(\xi)^{m} \psi(t\widehat{E}(\xi))],
\end{align*}where the last line will be justified by using the $\mathcal{L}$-ellipticity of $A$ and the positivity  of its symbol. Indeed, as in the proof of Theorem \ref{asymptotictracemultiplier2}, in a suitable basis of the representation space we can diagonalise the operator $\sigma(x,\xi)  \widehat{\mathcal{M}}(\xi)^{-m},$ and we can write in such a basis,
\begin{equation}
    \sigma(x,\xi)  \widehat{\mathcal{M}}(\xi)^{-m}=\textnormal{diag}[\lambda_{jj}(x,\xi)]_{j=1}^{d_\xi},\,\,\widehat{\mathcal{M}}(\xi)^{m}\psi(t\widehat{E})=[\Omega_{ij,t}(\xi)]_{i,j=1}^{d_\xi}.
\end{equation}
Now, we can write
\begin{align*}
     \textnormal{\textbf{Tr}}(A\psi(t{E}))
     &=\sum_{[\xi]\in  \widehat{G}}\sum_{j,j'=1}^{d_\xi}d_{\xi}\int\limits_{G}\lambda_{j'j'}(x,\xi)dx\times\Omega_{j'j,t}(\xi).
\end{align*}
The $\mathcal{L}$-ellipticity of $A$ and the positivity  of its symbol, implies that \begin{align*}
 1&\lesssim  \sup_{(x,[\xi])\in \widehat{G}}\Vert \sigma(x,\xi)^{-1}\widehat{\mathcal{M}}(\xi)^{m}\Vert_{\textnormal{op}}^{-1}\leqslant   \inf_{x\in G}\inf_{j',j',[\xi]\in \widehat{G}}\lambda_{j'j'}(x,\xi)\\
 &\leqslant \sup_{(x,[\xi])\in \widehat{G}}\Vert \sigma(x,\xi)\widehat{\mathcal{M}}(\xi)^{-m}\Vert_{\textnormal{op}}=\sup_{x\in G}\sup_{j',j',[\xi]\in \widehat{G}}\lambda_{j'j'}(x,\xi)\lesssim 1,
\end{align*}and we consequently deduce the  estimate,
\begin{align*}
     \textnormal{\textbf{Tr}}(A\psi(t{E}))
     &\asymp \sum_{[\xi]\in  \widehat{G}}\sum_{j,j'=1}^{d_\xi}d_{\xi}\Omega_{j'j,t}(\xi).
\end{align*} Because
$\widehat{\mathcal{M}}(\xi)^{m} \psi(t\widehat{E}(\xi))=[\Omega_{ij,t}(\xi)]_{i,j=1}^{d_\xi}$ is a symmetric matrix written in the basis that allows to write in a  diagonal form the operator $\sigma(x,\xi)\widehat{\mathcal{M}}(\xi)^{-m},$ we can find a matrix $P(\xi)$ such that
\[ 
    \widehat{\mathcal{M}}(\xi)^{m} \psi(t\widehat{E}(\xi))=P(\xi)^{-1}  \textnormal{diag}[  \Lambda_{t,jj}(\xi)]_{j=1}^{d_\xi}   P(\xi),
\]where $\Lambda_{t,jj}(\xi)$ is the sequence of eigenvalues of the matrix $\widehat{\mathcal{M}}(\xi)^{m} \psi(t\widehat{E}(\xi)).$ Observe that 
\begin{align*}
   &\sum_{j,j'=1}^{d_\xi}\Omega_{j'j,t}(\xi)=\sum_{j,j'=1}^{d_\xi}[P(\xi)^{-1}  \textnormal{diag}[    \Lambda_{t,ss}(\xi)]_{s=1}^{d_\xi}   P(\xi)]_{j'j}\\
    &=\sum_{j,j',s=1}^{d_\xi}P(\xi)^{-1}_{j'j}    \times   \Lambda_{t,j'j'}(\xi)  \times P(\xi)_{j's}\\
    &=\textnormal{\textbf{Tr}}[P(\xi)^{-1}  \textnormal{diag}[  \Lambda_{t,jj}(\xi)]_{j=1}^{d_\xi}   P(\xi)]=\textnormal{\textbf{Tr}}[\widehat{\mathcal{M}}(\xi)^{m} \psi(t\widehat{E}(\xi))].
\end{align*}
Now, we will use the Weyl-law for the sub-Laplacian. Observe that in a suitable basis of the representation spaces the operator $\widehat{E}(\xi)$ is diagonal and from the $\mathcal{L}$-ellipticity of $E$ and its positivity, we have
$$ t \widehat{E}_{jj}(\xi)\sim t(1+\nu_{jj}(\xi)^2)^\frac{q}{2},\,\quad\forall  1\leqslant j\leqslant d_\xi .$$
So, we get,
\begin{align*}
  &\textnormal{\textbf{Tr}}(\widehat{\mathcal{M}}(\xi)^{m} \psi(t\widehat{E}(\xi)))\\
  \quad&=\sum_{k=0}^{\infty}\sum_{[\xi]:2^{k}\leqslant (1+\nu_{j'j'}(\xi)^2)^\frac{1}{2} <2^{k+1},\,\forall 1\leqslant j'\leqslant d_{\xi} } \sum_{j=1}^{d_{\xi}}(1+\nu_{jj}(\xi)^2)^\frac{m}{2}\psi(t\widehat{E}_{jj}(\xi))\\
  \quad&\asymp\sum_{k=0}^{\infty}\sum_{[\xi]:2^{k}\leqslant (1+\nu_{j'j'}(\xi)^2)^\frac{1}{2} <2^{k+1},\,\forall 1\leqslant j'\leqslant d_{\xi} } \sum_{j=1}^{d_{\xi}}(1+\nu_{jj}(\xi)^2)^\frac{m}{2}\psi(t(1+\nu_{jj}(\xi)^2)^\frac{q}{2}).
\end{align*}Now, we can deduce that
\begin{align*}
    &\sum_{[\xi]:2^{k}\leqslant (1+\nu_{j'j'}(\xi)^2)^\frac{1}{2} <2^{k+1},\,\forall 1\leqslant j'\leqslant d_{\xi} } \sum_{j=1}^{d_{\xi}}(1+\nu_{jj}(\xi)^2)^\frac{m}{2}\psi(t(1+\nu_{jj}(\xi)^2)^{\frac{q}{2}}      )\\
    \quad &\asymp \sum_{[\xi]:2^{k}\leqslant (1+\nu_{j'j'}(\xi)^2)^\frac{1}{2} <2^{k+1},\,\forall 1\leqslant j'\leqslant d_{\xi} } {d_{\xi}} 2^{km}\psi(t2^{kq}) , 
\end{align*}and consequently,
\begin{align*}
    \textnormal{\textbf{Tr}}(\widehat{\mathcal{M}}(\xi)^{m} \psi(t\widehat{E}(\xi)))&\asymp \sum_{k=0}^{\infty} 2^{km}\psi(t2^{kq}) \sum_{[\xi]:2^{k}\leqslant (1+\nu_{j'j'}(\xi)^2)^\frac{1}{2} <2^{k+1},\,\forall 1\leqslant j'\leqslant d_{\xi} } d_\xi\\
    &\asymp\sum_{k=0}^{\infty} 2^{km}\psi(t2^{kq}) N(2^k)\asymp\sum_{k=0}^{\infty} 2^{km}\psi(t2^{kq}) 2^{kQ}\\
    &=\sum_{k=0}^{\infty} \psi(t2^{kq}) 2^{k(Q+m-1)}2^{k}.
\end{align*}Estimating the sums in $k$ as a integral, we have
\begin{align*}
  & \sum_{k=0}^{\infty} \psi(t2^{kq}) 2^{k(Q+m-1)}2^{k}\asymp \int\limits_{0}^{\infty}\psi(t\lambda^{q})\lambda^{Q+m-1}d\lambda\\
   &=\frac{t^{-\frac{1}{q}(Q+m)}}{q}\int\limits_{0}^\infty\psi(s)s^{\frac{Q+m}{q}}\frac{ds}{s}.
\end{align*}In particular, for $m=-Q,$ we have
\begin{align*}
   \textnormal{\textbf{Tr}}(\widehat{\mathcal{M}}(\xi)^{m} \psi(t\widehat{E}(\xi))\sim \frac{1}{q}\int\limits_{0}^{\infty}\psi(s)\frac{ds}{s},
\end{align*} provided that the compactly supported function $\psi$ on $\mathbb{R}^+_0$ belongs to $L^1(\mathbb{R}_0^+,\frac{ds}{s}).$
Observe that the integral $\int\limits_{0}^\infty\psi(s)s^{\frac{Q+m}{q}}\times \frac{ds}{s}$ makes sense if $\psi$  is smooth and it has compact support in $(0,\infty).$ However if $\psi(0)\neq 0,$ in order to assure that $$\int\limits_{0}^\infty\psi(s)s^{\frac{Q+m}{q}}\times \frac{ds}{s}<\infty,\,\,\,\psi\in C^{\infty}_0(\mathbb{R}^+_0),$$ we require the condition $1-\frac{Q+m}{q}<1,$ or equivalently that $Q+m> 0.$  So, in  such a situation, the function $$G(s):=s^{\frac{Q+m}{q}}\textnormal{\textbf{Tr}}(\widehat{\mathcal{M}}(\xi)^{m} \psi(s\widehat{E}(\xi))),\quad s>0,$$ is real-analytic and we can deduce the asymptotic formula  \eqref{asymp1212'22}. Thus, we end the proof.
\end{proof}

\section{Subelliptic global functional calculus and applications}\label{SFC}
In this section we develop the global functional calculus for subelliptic operators. The calculus will be applied to obtaining a subelliptic G\r{a}rding inequality and for studying the Dixmier trace of subelliptic operators.  

\subsection{Functions of symbols vs functions of operators}\label{S8}
Let $a\in S^{m,\mathcal{L}}_{\rho,\delta}(G\times \widehat{G})$ be a parameter $\mathcal{L}$-elliptic symbol  of order $m>0$ with respect to the sector $\Lambda\subset\mathbb{C}.$ For $A=\textnormal{Op}(a),$ let us define the operator $F(A)$  by the (Dunford-Riesz) complex functional calculus
\begin{equation}\label{F(A)}
    F(A)=-\frac{1}{2\pi i}\oint\limits_{\partial \Lambda_\varepsilon}F(z)(A-zI)^{-1}dz,
\end{equation}where
\begin{itemize}
    \item[(CI).] $\Lambda_{\varepsilon}:=\Lambda\cup \{z:|z|\leqslant \varepsilon\},$ $\varepsilon>0,$ and $\Gamma=\partial \Lambda_\varepsilon\subset\textnormal{Resolv}(A)$ is a positively oriented curve in the complex plane $\mathbb{C}$.
    \item[(CII).] $F$ is a holomorphic function in $\mathbb{C}\setminus \Lambda_{\varepsilon},$ and continuous on its closure. 
    \item[(CIII).] We will assume a decay order of $F$ along $\partial \Lambda_\varepsilon$ in order that the operator \eqref{F(A)} will be densely defined on $C^\infty(G)$ in the strong sense of the topology on $L^2(G).$
\end{itemize} Now, we will compute the matrix-valued symbols for operators defined by this complex functional calculus.
\begin{lemma}\label{LemmaFC}
Let $a\in S^{m,\mathcal{L}}_{\rho,\delta}(G\times \widehat{G})$ be a parameter $\mathcal{L}$-elliptic symbol  of order $m>0$ with respect to the sector $\Lambda\subset\mathbb{C}.$ Let $F(A):C^\infty(G)\rightarrow \mathscr{D}'(G)$ be the operator defined by the analytical functional calculus as in \eqref{F(A)}. Under the assumptions $\textnormal{(CI)}$, $\textnormal{(CII)}$, and $\textnormal{(CIII)}$, the matrix-valued symbol of $F(A),$ $\sigma_{F(A)}(x,\xi)$ is given by,
\[ 
    \sigma_{F(A)}(x,\xi)=-\frac{1}{2\pi i}\oint\limits_{\partial \Lambda_\varepsilon}F(z)\widehat{\mathcal{R}}_z(x,\xi)dz,
\]where $\mathcal{R}_z=(A-zI)^{-1}$ denotes the resolvent of $A,$ and $\widehat{\mathcal{R}}_z(x,\xi)\in S^{-m,\mathcal{L}}_{\rho,\delta}(G\times \widehat{G}) $ its symbol.
\end{lemma}
\begin{proof}
 From Corollary \ref{resolv}, we have that  $\widehat{\mathcal{R}}_z(x,\xi)\in S^{-m,\mathcal{L}}_{\rho,\delta}(G\times \widehat{G}) .$ Now, observe that  \begin{align*}
  \sigma_{F(A)}(x,\xi)=\xi(x)^*F(A)\xi(x)=-\frac{1}{2\pi i}\oint\limits_{\partial \Lambda_\varepsilon}F(z)\xi(x)^*(A-zI)^{-1}\xi(x)dz.  \end{align*} We finish the proof by observing that $\widehat{\mathcal{R}}_z(x,\xi)=\xi(x)^*(A-zI)^{-1}\xi(x),$ for every $z\in \textnormal{Resolv}(A).$
\end{proof}
Assumption (CIII) will be clarified in the following theorem where we show that the subelliptic calculus is stable under the action of the complex functional calculus.
\begin{theorem}\label{DunforRiesz}
Let $m>0,$ and let $0\leqslant \delta<\rho\leqslant 1.$ Let  $a\in S^{m,\mathcal{L}}_{\rho,\delta}(G\times \widehat{G})$ be a parameter $\mathcal{L}$-elliptic symbol with respect to $\Lambda.$ Let us assume that $F$ satisfies the  estimate $|F(\lambda)|\leqslant C|\lambda|^s$ uniformly in $\lambda,$ for some $s<0.$  Then  the symbol of $F(A),$  $\sigma_{F(A)}\in S^{ms,\mathcal{L}}_{\rho,\delta}(G\times \widehat{G}) $ admits an asymptotic expansion of the form
\begin{equation}\label{asymcomplex}
    \sigma_{F(A)}(x,\xi)\sim 
     \sum_{N=0}^\infty\sigma_{{B}_{N}}(x,\xi),\,\,\,(x,[\xi])\in G\times \widehat{G},
\end{equation}where $\sigma_{{B}_{N}}(x,\xi)\in {S}^{ms-(\rho-\delta)N,\mathcal{L}}_{\rho,\delta}(G\times \widehat{G})$ and 
\[ 
    \sigma_{{B}_{0}}(x,\xi)=-\frac{1}{2\pi i}\oint\limits_{\partial \Lambda_\varepsilon}F(z)(a(x,\xi)-z)^{-1}dz\in {S}^{ms,\mathcal{L}}_{\rho,\delta}(G\times \widehat{G}).
\]Moreover, 
\[ 
     \sigma_{F(A)}(x,\xi)\equiv -\frac{1}{2\pi i}\oint\limits_{\partial \Lambda_\varepsilon}F(z)a^{-\#}(x,\xi,\lambda)dz \textnormal{  mod  } {S}^{-\infty,\mathcal{L}}(G\times \widehat{G}),
\]where $a^{-\#}(x,\xi,\lambda)$ is the symbol of the parametrix to $A-\lambda I,$   in Corollary \ref{parameterparametrix}.
\end{theorem}
\begin{proof}
    First, we need to prove that the condition $|F(\lambda)|\leqslant C|\lambda|^s$ uniformly in $\lambda,$ for some $s<0,$ is enough in order to guarantee that \[ 
    \sigma_{{B}_{0}}(x,\xi):=-\frac{1}{2\pi i}\oint\limits_{\partial \Lambda_\varepsilon}F(z)(a(x,\xi)-z)^{-1}dz,
\] is a well defined matrix-symbol.
From Theorem \ref{lambdalambdita} we deduce that $(a(x,\xi)-z)^{-1}$ satisfies the estimate
\[ 
   \Vert (|z|^{\frac{1}{m}}+\widehat{\mathcal{M}}(\xi))^{m(k+1)}\widehat{\mathcal{M}}(\xi)^{\rho|\alpha|-\delta|\beta|}\partial_{z}^k\partial_{X}^{(\beta)}\Delta_{\xi}^{\alpha}(a(x,\xi)-z)^{-1}\Vert_{\textnormal{op}}<\infty.
\]
Observe that 
\begin{align*}
    &\Vert(a(x,\xi)-z)^{-1}\Vert_{\textnormal{op}}\\
    & =\Vert (|z|^{\frac{1}{m}}+\widehat{\mathcal{M}}(\xi))^{-m}(|z|^{\frac{1}{m}}+\widehat{\mathcal{M}}(\xi))^{m}(a(x,\xi)-z)^{-1}\Vert_{\textnormal{op}} \\
    &\lesssim  \sup_{1\leqslant j\leqslant d_\xi}(|z|^{\frac{1}{m}}+(1+\nu_{jj}(\xi)^2)^{\frac{1}{2}})^{-m}\\&\leqslant |z|^{-1},
\end{align*} and the condition $s<0$ implies
\begin{align*}
    \left|\frac{1}{2\pi i}\oint\limits_{\partial \Lambda_\varepsilon}F(z)(a(x,\xi)-z)^{-1}dz\right|\lesssim \oint\limits_{\partial \Lambda_\varepsilon}|z|^{-1+s}|dz|<\infty,
\end{align*}uniformly in $(x,[\xi])\in G\times \widehat{G}.$ In order to check that $\sigma_{B_0}\in {S}^{ms,\mathcal{L}}_{\rho,\delta}(G\times \widehat{G})$ let us analyse the cases $-1<s<0$ and $s\leqslant -1$ separately. So, let us analyse first the situation of $-1<s<0.$ We observe that
\begin{align*}
   &\Vert \widehat{\mathcal{M}}(\xi)^{-ms+\rho|\alpha|-\delta|\beta|}\partial_{X}^{(\beta)}\Delta_{\xi}^{\alpha}\sigma_{B_0}(x,\xi)\Vert_{\textnormal{op}}\\
   &\leqslant \frac{C}{2\pi }\oint\limits_{\partial \Lambda_\varepsilon} |z|^{s}\Vert      \widehat{\mathcal{M}}(\xi)^{-ms+\rho|\alpha|-\delta|\beta|}\partial_{X}^{(\beta)}\Delta_{\xi}^{\alpha}(a(x,\xi)-z)^{-1}\Vert_{\textnormal{op}} |dz|.
\end{align*}Now, we will estimate the operator norm inside of the integral. Indeed, the identity
\begin{align*}
    &\Vert      \widehat{\mathcal{M}}(\xi)^{-ms+\rho|\alpha|-\delta|\beta|}\partial_{X}^{(\beta)}\Delta_{\xi}^{\alpha}(a(x,\xi)-z)^{-1}\Vert_{\textnormal{op}}=\\
    &\Vert (|z|^{\frac{1}{m}}+\widehat{\mathcal{M}}(\xi))^{-m}(|z|^{\frac{1}{m}}+\widehat{\mathcal{M}}(\xi))^{m}\widehat{\mathcal{M}}(\xi)^{-ms+\rho|\alpha|-\delta|\beta|}\partial_{X}^{(\beta)}\Delta_{\xi}^{\alpha}(a(x,\xi)-z)^{-1}\Vert_{\textnormal{op}}
\end{align*}implies that
\begin{align*}
    &\Vert      \widehat{\mathcal{M}}(\xi)^{-ms+\rho|\alpha|-\delta|\beta|}\partial_{X}^{(\beta)}\Delta_{\xi}^{\alpha}(a(x,\xi)-z)^{-1}\Vert_{\textnormal{op}} \lesssim  \Vert (|z|^{\frac{1}{m}}+\widehat{\mathcal{M}}(\xi))^{-m}\widehat{\mathcal{M}}(\xi)^{-ms}\Vert_{\textnormal{op}}
\end{align*}where we have used that
\begin{align*}
\sup_{z\in \partial \Lambda_\varepsilon}  \sup_{(x,\xi)} \Vert (|z|^{\frac{1}{m}}+\widehat{\mathcal{M}}(\xi))^{m}\widehat{\mathcal{M}}(\xi)^{\rho|\alpha|-\delta|\beta|}\partial_{X}^{(\beta)}\Delta_{\xi}^{\alpha}(a(x,\xi)-z)^{-1}\Vert_{\textnormal{op}} <\infty.
\end{align*}
Consequently, by using that  $s<0,$ we deduce
\begin{align*}
   & \frac{C}{2\pi }\oint\limits_{\partial \Lambda_\varepsilon} |z|^{s}\Vert      \widehat{\mathcal{M}}(\xi)^{ms+\rho|\alpha|-\delta|\beta|}\partial_{X}^{(\beta)}\Delta_{\xi}^{\alpha}(a(x,\xi)-z)^{-1}\Vert_{\textnormal{op}} |dz|\\
    &\lesssim \frac{C}{2\pi }\oint\limits_{\partial \Lambda_\varepsilon} |z|^{s}\Vert (|z|^{\frac{1}{m}}+\widehat{\mathcal{M}}(\xi))^{-m} \widehat{\mathcal{M}}(\xi)^{-ms}\Vert_{\textnormal{op}} |dz|\\
    &= \frac{C}{2\pi }\oint\limits_{\partial \Lambda_\varepsilon} |z|^{s}\sup_{1\leqslant j\leqslant d_\xi} (|z|^{\frac{1}{m}}+(1+\nu_{jj}(\xi)^{2})^{\frac{1}{2}})^{-m}(1+\nu_{jj}(\xi)^{2})^{-\frac{ms}{2}}|dz|.
\end{align*}To study the convergence of the last contour integral we only need to check the convergence of $\int\limits_{1}^{\infty}r^s(r^{\frac{1}{m}}+\varkappa)^{-m}\varkappa^{-ms}dr,$ where $\varkappa>1$ in a parameter. The change of variable $r=\varkappa^{m}t$ implies that
\begin{align*}
   \int\limits_{1}^{\infty}r^s(r^{\frac{1}{m}}+\varkappa)^{-m}\varkappa^{-ms}dr&=\int\limits_{\varkappa^{-m}}^{\infty}\varkappa^{ms}t^s(\varkappa t^{\frac{1}{m}}+\varkappa)^{-m}\varkappa^{-ms}\varkappa^mdt=\int\limits_{\varkappa^{-m}}^{\infty}t^s(t^{\frac{1}{m}}+1)^{-m}dt\\
   &\lesssim \int\limits_{\varkappa^{-m}}^{1}t^sdt+\int\limits_{1}^{\infty}t^{-1+s}<\infty.
\end{align*}Indeed, for $t\rightarrow\infty,$ $t^s(t^{\frac{1}{m}}+1)^{-m}\lesssim t^{-1+s}$ and we conclude the estimate because $\int\limits_{1}^{\infty} t^{-1+s'}dt<\infty,$ for all $s'<0.$ On the other hand, the condition $-1<s<0$ implies that
\begin{align*}
 \int\limits_{\varkappa^{-m}}^{1}t^sdt=\frac{1}{1+s}-\frac{\varkappa^{-m(1+s)}}{1+s}=   O(1).
\end{align*} In the case where $s\leqslant -1,$ we can find an analytic function $\tilde{G}(z)$ such that it is a holomorphic function in $\mathbb{C}\setminus \Lambda_{\varepsilon},$ and continuous on its closure and additionally satisfying that $F(\lambda)=\tilde{G}(\lambda)^{1+[-s]}.$\footnote{ $[-s]$ denotes  the integer part of $-s.$} In this case,  $\tilde{G}(A)$ defined by the complex functional calculus 
\begin{equation}\label{G(A)}
    \tilde{G}(A)=-\frac{1}{2\pi i}\oint\limits_{\partial \Lambda_\varepsilon}\tilde{G}(z)(A-zI)^{-1}dz,
\end{equation}
has symbol belonging to ${S}^{\frac{sm}{1+[-s]},\mathcal{L}}_{\rho,\delta}(G\times \widehat{G})$ because $\tilde{G}$ satisfies the estimate $|G(\lambda)|\leqslant C|\lambda|^{\frac{s}{1+[-s]}},$ with $-1<\frac{s}{1+[-s]}<0.$ 
By observing that
\begin{align*}
    \sigma_{F(A)}(x,\xi)&=-\frac{1}{2\pi i}\oint\limits_{\partial \Lambda_\varepsilon}F(z)\widehat{\mathcal{R}}_z(x,\xi)dz=-\frac{1}{2\pi i}\oint\limits_{\partial \Lambda_\varepsilon}\tilde{G}(z)^{1+[-s]}\widehat{\mathcal{R}}_z(x,\xi)dz\\
    &=\sigma_{\tilde{G}(A)^{1+[-s]}}(x,\xi),
\end{align*}and computing the symbol $\sigma_{\tilde{G}(A)^{1+[-s]}}(x,\xi)$ by iterating $1+[-s]$-times  the asymptotic formula for the composition in the subelliptic calculus (see Corollary \ref{SubellipticcompositionC}), we can see that the term with higher order in such expansion is $\sigma_{\tilde{G}(A)}(x,\xi)^{1+[-s]}\in {S}^{ms,\mathcal{L}}_{\rho,\delta}(G\times \widehat{G}).$ Consequently we have proved that $\sigma_{F(A)}(x,\xi)\in {S}^{ms,\mathcal{L}}_{\rho,\delta}(G\times \widehat{G}).$
This completes the proof for the first part of the theorem.
For the second part of the proof, let us denote by $a^{-\#}(x,\xi,\lambda)$  the symbol of the parametrix to $A-\lambda I,$   in Corollary \ref{parameterparametrix}. Let $P_{\lambda}=\textnormal{Op}(a^{-\#}(\cdot,\cdot,\lambda)).$ Because $\lambda\in \textnormal{Resolv}(A)$ for $\lambda\in \partial \Lambda_\varepsilon,$ $(A-\lambda)^{-1}-P_{\lambda}$ is a smoothing operator. Consequently, from Lemma \ref{LemmaFC} we deduce that
\begin{align*}
   & \sigma_{F(A)}(x,\xi)\\
   &=-\frac{1}{2\pi i}\oint\limits_{\partial \Lambda_\varepsilon}F(z)\widehat{\mathcal{R}}_z(x,\xi)dz\\
    &=-\frac{1}{2\pi i}\oint\limits_{\partial \Lambda_\varepsilon}F(z)a^{-\#}(x,\xi,z)dz-\frac{1}{2\pi i}\oint\limits_{\partial \Lambda_\varepsilon}F(z)(\widehat{\mathcal{R}}_z(x,\xi)-a^{-\#}(x,\xi,z))dz\\
    &\equiv -\frac{1}{2\pi i}\oint\limits_{\partial \Lambda_\varepsilon}F(z)a^{-\#}(x,\xi,z)dz  \textnormal{  mod  } {S}^{-\infty,\mathcal{L}}(G\times \widehat{G}).
\end{align*}The asymptotic expansion \eqref{asymcomplex} can be deduced from the construction of the parametrix in the subelliptic calculus (see Proposition \ref{IesTParametrix}).
\end{proof}

\subsection{G\r{a}rding inequality} In this section we will prove the G\r{a}rding inequality for the subelliptic calculus. To do so, we need some preliminary propositions. 
\begin{proposition}
Let $0\leqslant \delta<\rho\leqslant 1.$  Let  $a\in S^{m,\mathcal{L}}_{\rho,\delta}(G\times \widehat{G})$ be an $\mathcal{L}$-elliptic matrix-valued symbol where $m\geqslant 0$ and let us assume that $a$ is positive definite. Then $a$ is parameter-elliptic with respect to $\mathbb{R}_{-}:=\{z=x+i0:x<0\}\subset\mathbb{C}.$ Furthermore, for any number $s\in \mathbb{C},$ 
\[ 
    \widehat{B}(x,\xi)\equiv a(x,\xi)^s:=\exp(s\log(a(x,\xi))),\,\,(x,[\xi])\in G\times \widehat{G},
\]defines a symbol $\widehat{B}(x,\xi)\in S^{m\times\textnormal{Re}(s),\mathcal{L}}_{\rho,\delta}(G\times \widehat{G}).$
\end{proposition}
\begin{proof} From the estimates 
\[ 
     \sup_{(x,[\xi])}\Vert\widehat{\mathcal{M}}(\xi)^{-m} a(x,\xi) \Vert_{\textnormal{op}}<\infty,\quad \sup_{(x,[\xi])}\Vert \widehat{\mathcal{M}}^m(\xi)a(x,\xi)^{-1} \Vert_{\textnormal{op}}<\infty,
\]
   we deduce that
    \[ 
       \sup_{x\in G}\Vert a(x,\xi) \Vert_{\textnormal{op}}\asymp\Vert  \widehat{\mathcal{M}}(\xi)^{m}\Vert_{\textnormal{op}}, \quad \sup_{x\in G}\Vert a(x,\xi)^{-1} \Vert_{\textnormal{op}}\asymp\Vert  \widehat{\mathcal{M}}(\xi)^{-m}\Vert_{\textnormal{op}}. 
    \] Consequently we have
    \[ 
      \Vert  \widehat{\mathcal{M}}(\xi)^{m}\Vert_{\textnormal{op}}^{-1} \sup_{x\in G}\Vert a(x,\xi) \Vert_{\textnormal{op}}\asymp 1, \quad \Vert  \widehat{\mathcal{M}}(\xi)^{-m}\Vert_{\textnormal{op}}^{-1} \sup_{x\in G}\Vert a(x,\xi)^{-1} \Vert_{\textnormal{op}}\asymp 1, 
    \] from which we deduce that
    \[ 
        \Vert  \widehat{\mathcal{M}}(\xi)^{-m}\Vert_{\textnormal{op}}^{-1}\textnormal{Spectrum}(a(x,\xi))\subset [c,C],
    \]where $c,C>0$ are positive real numbers. The matrix $a(x,[\xi])$ is normal, so that for every $\lambda\in \mathbb{R}_{-}$ we have
    \begin{align*}
      &   \Vert (|\lambda|^{\frac{1}{m}}+\widehat{\mathcal{M}}(\xi))^m (a(x,\xi)-\lambda)^{-1}\Vert_{\textnormal{op}}\\
         &\asymp \Vert (|\lambda|^{\frac{1}{m}}+\widehat{\mathcal{M}}(\xi))^m (\widehat{\mathcal{M}}(\xi)^m-\lambda)^{-1}(\widehat{\mathcal{M}}(\xi)^m-\lambda)(a(x,\xi)-\lambda)^{-1}\Vert_{\textnormal{op}} \\
         &\lesssim \Vert (|\lambda|^{\frac{1}{m}}+\widehat{\mathcal{M}}(\xi))^m (\widehat{\mathcal{M}}(\xi)^m-\lambda)^{-1}\Vert_{\textnormal{op}} \Vert (\widehat{\mathcal{M}}(\xi)^m-\lambda)(a(x,\xi)-\lambda)^{-1}\Vert_{\textnormal{op}}\\
         &\lesssim \Vert (|\lambda|^{\frac{1}{m}}+\widehat{\mathcal{M}}(\xi))^m (\widehat{\mathcal{M}}(\xi)^m-\lambda)^{-1}\Vert_{\textnormal{op}}.
    \end{align*}
    Let us note that the condition $m\geq 0,$ implies that $\Vert \widehat{\mathcal{M}}(\xi)^{-m} \Vert_{\textnormal{op}}\leq 1.$ 
    So, if $|\lambda|\in [0,1/2],$ then 
    $$ |\lambda|\Vert \widehat{\mathcal{M}}(\xi)^{-m} \Vert_{\textnormal{op}} \leq 1/2, $$ which implies that for all $0\leq |\lambda|\leq 1/2,$ $(1-\lambda  \widehat{\mathcal{M}}(\xi)^{-m} )$ is invertible and from the first von-Neumann identity,
    \begin{align*}
        \Vert (1-\lambda  \widehat{\mathcal{M}}(\xi)^{-m} )^{-1} \Vert_{\textnormal{op}}\leq (1-\Vert\lambda  \widehat{\mathcal{M}}(\xi)^{-m}\Vert_{\textnormal{op}} )^{-1}=(1-|\lambda|\Vert  \widehat{\mathcal{M}}(\xi)^{-m}\Vert_{\textnormal{op}} )^{-1}\leq 2.
    \end{align*}
    Now, fixing again $|\lambda|\in \mathbb{R}_{-}$ observe that from the compactness of $[0,1/2]$ we deduce that
    \begin{align*}
        \sup_{0\leqslant |\lambda|\leqslant 1/2}\Vert (|\lambda|^{\frac{1}{m}}+\widehat{\mathcal{M}}(\xi))^m (\widehat{\mathcal{M}}(\xi)^m-\lambda)^{-1}\Vert_{\textnormal{op}}&\asymp \sup_{0\leqslant |\lambda|\leqslant 1/2}\Vert \widehat{\mathcal{M}}(\xi)^m (\widehat{\mathcal{M}}(\xi)^m-\lambda)^{-1}\Vert_{\textnormal{op}}  \\
         &\asymp \sup_{0\leqslant |\lambda|\leqslant 1/2}\Vert  (I_{d_\xi}-\lambda\widehat{\mathcal{M}}(\xi)^{-m})^{-1}\Vert_{\textnormal{op}} \\
          &\lesssim 1,
    \end{align*}where in the last line we have used the continuity of the function $U(\lambda):=\Vert  (I_{d_\xi}-\lambda\widehat{\mathcal{M}}(\xi)^{-m})^{-1}\Vert_{\textnormal{op}},$ and the fact that it is bounded on $[0,1/2].$ On the other hand,
    \begin{align*}
    &  \sup_{|\lambda|\geqslant 1/2}\Vert (|\lambda|^{\frac{1}{m}}+\widehat{\mathcal{M}}(\xi))^m (\widehat{\mathcal{M}}(\xi)^m-\lambda)^{-1}\Vert_{\textnormal{op}}  \\
      &=\sup_{|\lambda|\geqslant 1/2}\Vert (|\lambda|^{\frac{1}{m}}\widehat{\mathcal{M}}(\xi)^{-1}+I_{d_\xi})^m (I_{d_\xi}-\widehat{\mathcal{M}}(\xi)^{-m}\lambda)^{-1}\Vert_{\textnormal{op}}\\
      &=\sup_{|\lambda|\geqslant 1/2}\Vert (\widehat{\mathcal{M}}(\xi)^{-1}+|\lambda|^{-\frac{1}{m}}I_{d_\xi})^m |\lambda|(I_{d_\xi}-\widehat{\mathcal{M}}(\xi)^{-m}\lambda)^{-1}\Vert_{\textnormal{op}}\\
      &\lesssim \sup_{|\lambda|\geqslant 1/2}\Vert \widehat{\mathcal{M}}(\xi)^{-m} |\lambda|(-\lambda)^{-1}\widehat{\mathcal{M}}(\xi)^{m}\Vert_{\textnormal{op}}\\
      &=1.
    \end{align*}So, we have proved that $a$ is parameter-elliptic with respect to $\mathbb{R}_{-}.$ To prove that $\widehat{B}(x,\xi)\in S^{m\times\textnormal{Re}(s),\mathcal{L}}_{\rho,\delta}(G\times \widehat{G}),$ we can observe that for $\textnormal{Re}(s)<0,$ we can apply Theorem \ref{DunforRiesz}. If  $\textnormal{Re}(s)\geqslant 0,$ we can find $k\in \mathbb{N}$ such that $\textnormal{Re}(s)-k<0$ and consequently from the spectral calculus of matrices we deduce that $a(x,\xi)^{\textnormal{Re}(s)-k}\in S^{m\times(\textnormal{Re}(s)-k),\mathcal{L}}_{\rho,\delta}(G\times \widehat{G}).$ So, from the calculus we conclude that   $$a(x,\xi)^{s}=a(x,\xi)^{s-k}a(x,\xi)^{k}\in S^{m\times\textnormal{Re}(s),\mathcal{L}}_{\rho,\delta}(G\times \widehat{G}).$$ Thus the proof is complete.
\end{proof}
\begin{corollary}\label{1/2}
Let $0\leqslant \delta<\rho\leqslant 1.$  Let  $a\in S^{m,\mathcal{L}}_{\rho,\delta}(G\times \widehat{G}),$  be an $\mathcal{L}$-elliptic symbol  where $m\geqslant 0$ and let us assume that $a$ is positive definite. Then 
$\widehat{B}(x,\xi)\equiv a(x,\xi)^\frac{1}{2}:=\exp(\frac{1}{2}\log(a(x,\xi)))\in S^{\frac{m}{2},\mathcal{L}}_{\rho,\delta}(G\times \widehat{G}).$
\end{corollary}
Now, let us assume that 
\[ 
    A(x,\xi):=\frac{1}{2}(a(x,\xi)+a(x,\xi)^{*}),\,(x,[\xi])\in G\times \widehat{G},\,\,a\in S^{m,\mathcal{L}}_{\rho,\delta}(G\times \widehat{G}), 
\]satisfies
\begin{align}\label{eqi}
    \Vert\widehat{\mathcal{M}}(\xi)^{m}A(x,\xi)^{-1} \Vert_{\textnormal{op}}\leqslant C_{0}.
\end{align} Observe that \eqref{eqi} implies that
\[ 
    \lambda(x,\xi):=\inf\{\tilde{\lambda}(x,[\xi])^{-1}:\det(\widehat{\mathcal{M}}(\xi)^{-m}A(x,\xi)-\tilde{\lambda}(x,\xi)I_{d_\xi})=0\}\leqslant C_{0}.
\]So, $\lambda(x,\xi)^{-1}\geqslant \frac{1}{C_0}$ and consequently
\begin{align*}
  \widehat{\mathcal{M}}(\xi)^{-m}A(x,\xi)\geqslant\frac{1}{C_0}I_{d_\xi}.
\end{align*}This implies that
\begin{align*}
  A(x,\xi)\geqslant\frac{1}{C_0}\widehat{\mathcal{M}}(\xi)^{m},
\end{align*} provided that $A(x,\xi)$ commutes with the symbol $\widehat{\mathcal{L}}$ of the sub-Laplacian $\mathcal{L}$.  For $C_1\in(0, \frac{1}{C_0})$ we have that
\begin{align*}
 A(x,\xi)-C_{1}  \widehat{\mathcal{M}}(\xi)^{m}\geqslant \left(\frac{1}{C_0}-C_1\right) \widehat{\mathcal{M}}(\xi)^{m}>0.
\end{align*}If  $0\leqslant \delta<\rho\leqslant 1,$   from Corollary \ref{1/2}, we have that
\begin{align*}
    q(x,\xi):=(A(x,\xi)-C_{1}  \widehat{\mathcal{M}}(\xi)^{m})^{\frac{1}{2}}\in  S^{\frac{m}{2},\mathcal{L}}_{\rho ,\delta}(G\times \widehat{G}).
\end{align*}From the symbolic calculus we obtain
\begin{align*}
  q(x,\xi)q(x,\xi)^*= A(x,\xi)-C_{1}  \widehat{\mathcal{M}}(\xi)^{m}+r(x,\xi),\,\,   r(x,\xi)\in  S^{m-(\rho-\delta),\mathcal{L}}_{\rho ,\delta}(G\times \widehat{G}).
\end{align*}Now, let us assume that $u\in C^\infty(G).$ Then we have
\begin{align*}
    \textnormal{Re}(a(x,D)u,u)&=\frac{1}{2}((a(x,D)+\textnormal{op}(a^*))u,u)\\
    &=C_{1}(\mathcal{M}_{m}u,u)+(q(x,D)q(x,D)^{*}u,u)+(r(x,D)u,u)\\
    &=C_{1}(\mathcal{M}_{m}u,u)+(q(x,D)^{*}u,q(x,D)^*u)-(r(x,D)u,u)\\
    &\geqslant C_{1}\Vert u\Vert_{{L}^{2,\mathcal{L}}_{\frac{m}{2}}(G)}-(r(x,D)u,u)\\
     &= C_{1}\Vert u\Vert_{{L}^{2,\mathcal{L}}_{\frac{m}{2}}(G)}-(\mathcal{M}_{-\frac{m-(\rho-\delta)}{2}}r(x,D)u,\mathcal{M}_{\frac{m-(\rho-\delta)}{2}}u).
\end{align*}Observe that
\begin{align*}
    (\mathcal{M}_{-\frac{m-(\rho-\delta)}{2}}r(x,D)u,\mathcal{M}_{\frac{m-(\rho-\delta)}{2}}u)&\leqslant \Vert \mathcal{M}_{-\frac{m-(\rho-\delta)}{2}}r(x,D)u \Vert_{L^2(G)}\Vert u\Vert_{L^{2,\mathcal{L}  }_{\frac{m-(\rho-\delta)}{2}}(G)}\\
    &= \Vert r(x,D)u \Vert_{L^{2,\mathcal{L}  }_{-\frac{m-(\rho-\delta)}{2}}(G)}\Vert u\Vert_{L^{2,\mathcal{L}  }_{\frac{m-(\rho-\delta)}{2}}(G)}\\
     &\leqslant C_1\Vert u \Vert_{L^{2,\mathcal{L}  }_{\frac{m-(\rho-\delta)}{2}}(G)}\Vert u\Vert_{L^{2,\mathcal{L}  }_{\frac{m-(\rho-\delta)}{2}}(G)},
\end{align*}where in the last line we have used the subelliptic Sobolev boundedness of $r(x,D)$ from $L^{2,\mathcal{L}  }_{\frac{m-(\rho-\delta)}{2}}(G)$ into $L^{2,\mathcal{L}  }_{-\frac{m-(\rho-\delta)}{2}}(G),$ in view of Corollary \ref{SobL3}. Consequently, we deduce the lower bound
\begin{align*}
    \textnormal{Re}(a(x,D)u,u) \geqslant C_{1}\Vert u\Vert_{{L}^{2,\mathcal{L}}_{\frac{m}{2}}(G)}-C\Vert u\Vert_{L^{2,\mathcal{L}  }_{\frac{m-(\rho-\delta)}{2}}(G)}^2.
\end{align*} If we assume for a moment that for every $\varepsilon>0,$ there exists $C_{\varepsilon}>0,$ such that
\begin{equation}\label{lemararo}
    \Vert u\Vert_{{L}^{2,\mathcal{L}}_{\frac{m-(\rho-\delta)}{2}}(G)}^2\leqslant \varepsilon\Vert u\Vert_{{L}^{2,\mathcal{L}}_{\frac{m}{2}}(G)}^2+C_{\varepsilon}\Vert u \Vert_{L^2(G)}^2,
\end{equation} for $0<\varepsilon<C_{1}$ we have
\begin{align*}
    \textnormal{Re}(a(x,D)u,u) \geqslant (C_{1}-\varepsilon)\Vert u\Vert_{{L}^{2,\mathcal{L}}_{\frac{m}{2}}(G)}-C_{\varepsilon}\Vert u\Vert_{L^2(G)}^2.
\end{align*}So, with the exception of the proof of \eqref{lemararo} we have deduced the following estimate which is the main result of this subsection.
\begin{theorem}[Subelliptic G\r{a}rding inequality]\label{GardinTheorem} Let $G$ be a compact Lie group and let us denote by $Q$ the Hausdorff
dimension of $G$ associated to the control distance associated to the sub-Laplacian $\mathcal{L}=\mathcal{L}_X,$ where  $X= \{X_{1},\cdots,X_k\} $ is a system of vector fields satisfying the H\"ormander condition.  For $0\leqslant \delta<\rho\leqslant 1,$  let $a(x,D):C^\infty(G)\rightarrow\mathscr{D}'(G)$ be an operator with symbol  $a\in {S}^{m,\mathcal{L}}_{\rho,\delta}( G\times \widehat{G})$ of order  $m>0$. Let us assume that 
\[ 
    A(x,\xi):=\frac{1}{2}(a(x,\xi)+a(x,\xi)^{*}),\,(x,[\xi])\in G\times \widehat{G},\,\,a\in S^{m,\mathcal{L}}_{\rho,\delta}(G\times \widehat{G}), 
\] commutes with the symbol $\widehat{\mathcal{L}}$ of the sub-Laplacian $\mathcal{L},$ and that it satisfies
\begin{align*}\label{garding}
    \Vert\widehat{\mathcal{M}}(\xi)^{m}A(x,\xi)^{-1} \Vert_{\textnormal{op}}\leqslant C_{0}.
\end{align*}Then, there exist $C_{1},C_{2}>0,$ such that the lower bound
\begin{align}
    \textnormal{Re}(a(x,D)u,u) \geqslant C_1\Vert u\Vert_{{L}^{2,\mathcal{L}}_{\frac{m}{2}}(G)}-C_2\Vert u\Vert_{L^2(G)}^2,
\end{align}holds true for every $u\in C^\infty(G).$ Moreover, the commutativity condition can be removed if
\begin{equation}
     A(x,\xi)\geqslant C\widehat{\mathcal{M}}(\xi)^{m},
\end{equation}for all $(x,[\xi])\in G\times \widehat{G},$ for some $C>0.$
\end{theorem}
In view of the analysis above, for the proof of Theorem \ref{GardinTheorem} we only need to prove \eqref{lemararo}. However we will deduce it from the following more general lemma.
\begin{lemma}Let us assume that $s\geqslant t\geqslant 0$ or that $s,t<0.$ Then, for every $\varepsilon>0,$ there exists $C_\varepsilon>0$ such that 
\begin{equation}\label{lemararo2}
    \Vert u\Vert_{{L}^{2,\mathcal{L}}_{t}(G)}^2\leqslant \varepsilon\Vert u\Vert_{{L}^{2,\mathcal{L}}_{s}(G)}^2+C_{\varepsilon}\Vert u \Vert_{L^2(G)}^2,
\end{equation}holds true for every $u\in C^\infty(G).$ 
\end{lemma}
\begin{proof}
    Let $\varepsilon>0.$ Then, there exists $C_{\varepsilon}>0$ such that
    \[ 
        (1+\nu_{ii}(\xi)^2)^{t}-\varepsilon (1+\nu_{ii}(\xi)^2)^s\leqslant C_{\varepsilon},
    \]uniformly in $[\xi]\in \widehat{G}.$ Then \eqref{lemararo2} it follows from the Plancherel theorem. Indeed,
    \begin{align*}
     \Vert u\Vert_{{L}^{2,\mathcal{L}}_{\frac{t}{2}}(G)}^2& =\sum_{[\xi]\in \widehat{G}}d_\xi\sum_{i,j=1}^{d_\xi} (1+\nu_{ii}(\xi)^2)^{t}|\widehat{u}_{ij}(\xi)|^{2}  \\
     &\leqslant \sum_{[\xi]\in \widehat{G}}d_\xi\sum_{i,j=1}^{d_\xi} (\varepsilon(1+\nu_{ii}(\xi)^2)^{s}+C_\varepsilon)|\widehat{u}_{ij}(\xi)|^{2}  \\
    &= \varepsilon\Vert u\Vert_{{L}^{2,\mathcal{L}}_{s}(G)}^2+C_{\varepsilon}\Vert u \Vert_{L^2(G)}^2,
    \end{align*}completing the proof.
\end{proof}
\begin{corollary}\label{GardinTheorem2} Let $G$ be a compact Lie group and let us denote by $Q$ the Hausdorff
dimension of $G$ associated to the control distance associated to the sub-Laplacian $\mathcal{L}=\mathcal{L}_X,$ where  $X= \{X_{1},\cdots,X_k\} $ is a system of vector fields satisfying the H\"ormander condition. Let $a(x,D):C^\infty(G)\rightarrow\mathscr{D}'(G)$ be an operator with symbol  $a\in {S}^{m,\mathcal{L}}_{\rho,\delta}( G\times \widehat{G})$, $m>0$, $0\leq \delta<\rho\leq 1$. Let us assume that 
\[ 
    a(x,\xi)\geqslant 0,\,(x,[\xi])\in G\times \widehat{G}, 
\]commutes with the symbol $\widehat{\mathcal{L}}$ of the sub-Laplacian $\mathcal{L},$ and that it satisfies
\begin{align*}\label{garding22}
    \Vert\widehat{\mathcal{M}}(\xi)^{m}a(x,\xi)^{-1} \Vert_{\textnormal{op}}\leqslant C_{0}.
\end{align*}Then, there exist $C_{1},C_{2}>0,$ such that the lower bound
\begin{align}
    \textnormal{Re}(a(x,D)u,u) \geqslant C_1\Vert u\Vert_{{L}^{2,\mathcal{L}}_{\frac{m}{2}}(G)}^2-C_2\Vert u\Vert_{L^2(G)}^2,
\end{align}holds true for every $u\in C^\infty(G).$
\end{corollary}
\subsection{Dixmier traces} Now, we will apply the subelliptic functional calculus to study the membership of subelliptic operators in the Dixmier ideal on $L^2(G).$ 

 By following Connes \cite{Connes94}, if $H$ is a Hilbert space (we are interested in $H=L^2(G)$ for instance), the class $\mathcal{L}^{(1,\infty)}(H)$ consists of those compact linear operators $A\in\mathcal{L}(H)$  satisfying
\begin{equation}
\sum_{1\leq n\leq N}s_{n}(A)=O(\log(N)),\,\,\,N\rightarrow \infty,
\end{equation}
where $\{s_{n}(A)\}$ denotes the sequence of singular values of $A$,  i.e. the square roots of the eigenvalues of the non-negative self-adjoint operator $A^\ast A.$ 
So, $\mathcal{L}^{(1,\infty)}(H)$ is endowed with the norm
\begin{equation}\label{dixmier}
\Vert A \Vert_{\mathcal{L}^{(1,\infty)}(H)}=\sup_{N\geq 2}\frac{1}{\log(N)}\sum_{1\leq n\leq N}s_{n}(A).
\end{equation}
We define the functional  $$\textnormal{\bf{Tr}}_{\textnormal{Dix}}(A):=\lim_{N\rightarrow\infty }\frac{1}{\log(N)}\sum_{1\leq n\leq N}s_{n}(A)\in [0,\infty],$$ for the family of bounded operators $A$ in $\mathcal{L}^{(1,\infty)}(H)$  (see \cite{Connes94} or  \cite{Dixmier}).
We recall the main result in Comptes Rendus's note of Dixmier  \cite{Dixmier} regarding the existence of a trace functional different from
the spectral trace defined on the ideal  $\mathcal{L}^{(1,\infty)}(\mathcal{H})$ which contains the ideal of trace class operators $S_1(\mathcal{H})$.

\begin{theorem}[Dixmier   \cite{Dixmier}]\label{Dixm:Theorem} Let $\omega \in \ell^{\infty}(\mathbb{N}\setminus\{ 0\})^{*}$ be a bounded functional satisfying the following conditions:

\begin{itemize}
    \item[(D1):] $\omega$ is a positive linear functional that satisfies $\omega(1,1,\cdots, 1,\cdots)=1.$  
    \item[(D2):]  $\omega((a_n)_{n\in \mathbb{N}\setminus\{ 0\}})=0,$ for a sequence $(a_n)_{n\in \mathbb{N}\setminus\{ 0\}}\in \ell^{\infty}(\mathbb{N}\setminus\{ 0\})$ if $$ \lim_{N\rightarrow \infty}a_{n}=0.$$ 
    \item[(D3):] For any  $(a_n)_{n\in \mathbb{N}\setminus\{ 0\}}\in \ell^{\infty}(\mathbb{N}\setminus\{ 0\}),$
    $$ \omega(a_1,a_2,\cdots ,a_n,\cdots )= \omega(a_1,a_1,a_2,a_2,\cdots ,a_n,a_n,\cdots ). $$
\end{itemize} For a compact positive operator $A$ in $\mathcal{L}^{(1,\infty)}(\mathcal{H})$ we set
\begin{equation}
    \textnormal{\bf Tr}_\omega(A):=\omega \left(\left\{\frac{1}{\log(N)}\sum_{j=1}^{N}s_{j}(A)\right\}_{N\geq 1} \right).
\end{equation} Then, $ \textnormal{\bf Tr}_\omega$ extends by linearity to a trace on $\mathcal{L}^{(1,\infty)}(\mathcal{H}),$ and 
\begin{equation}\label{TrD}
     \textnormal{\bf Tr}_\omega(A)= \textnormal{\bf Tr}_{\textnormal{Dix}}(A):=\lim_{N\rightarrow \infty} \frac{1}{\log(N)}\sum_{j=1}^{N}s_{j}(A),
\end{equation}provided that the  limit in the right hand side exists. In this case, the value of $ \textnormal{\bf Tr}_\omega(A)$ is independent of the choice of $\omega.$ Moreover, $\textnormal{\bf Tr}_\omega(A)=0,$ if $A\in S_{1}(\mathcal{H}).$
\end{theorem}
\begin{remark}\label{Deco:remark} Let us describe the extension of $\textnormal{\bf Tr}_\omega$ from the positive elements in $\mathcal{L}^{(1,\infty)}(\mathcal{H})$ to general compact operators in this ideal, see \cite{Connes94}.
We use the decomposition of $A$ into its real and imaginary part,
\[ 
    \textnormal{Re}(A):=\frac{A+A^*}{2},\,\, \textnormal{Im}(A):=\frac{A-A^*}{2i},
\]and the decomposition of $\textnormal{Re}(A)$ and $\textnormal{Im}(A)$ into their positive and negative parts,
\begin{eqnarray*}
  \textnormal{Re}(A)^{+}:=\frac{\textnormal{Re}(A)+|\textnormal{Re}(A)|}{2},\,\, \textnormal{Re}(A)^{-}:=\frac{|\textnormal{Re}(A)|-\textnormal{Re}(A)}{2},
\end{eqnarray*}
and 
\begin{eqnarray*}
  \textnormal{Im}(A)^{+}:=\frac{\textnormal{Im}(A)+|\textnormal{Im}(A)|}{2},\,\, \textnormal{Im}(A)^{-}:=\frac{|\textnormal{Im}(A)|-\textnormal{Im}(A)}{2}.
\end{eqnarray*}Now, the operator $A$ can be written as
\begin{align*}
    A&= \textnormal{Re}(A)+i\textnormal{Im}(A)\\
    &=\left(\textnormal{Re}(A)^{+}-\textnormal{Re}(A)^{-}\right)+i\left(\textnormal{Im}(A)^{+}-\textnormal{Im}(A)^{-}\right).
\end{align*}So, by the linearity of the Dixmier trace $\textnormal{Tr}_\omega$ we have
\begin{align*}
    \textnormal{Tr}_\omega (A)&= \textnormal{Tr}_\omega(\textnormal{Re}(A))+i\textnormal{Tr}_\omega(\textnormal{Im}(A))\\
    &=\left(\textnormal{Tr}_\omega(\textnormal{Re}(A)^{+})-\textnormal{Tr}_\omega(\textnormal{Re}(A)^{-})\right)+i\left(\textnormal{Tr}_\omega(\textnormal{Im}(A)^{+})-\textnormal{Tr}_\omega(\textnormal{Im}(A)^{-})\right).
\end{align*}
\end{remark}

Our starting point is the following lemma.
\begin{lemma}\label{dixmierlema1}
Let $G$ be a compact Lie group and let us denote by $Q$ the Hausdorff
dimension of $G$ associated to the control distance associated to the sub-Laplacian $\mathcal{L}=\mathcal{L}_X,$ where  $X= \{X_{1},\cdots,X_k\} $ is a system of vector fields satisfying the H\"ormander condition.  For  $0\leqslant \rho\leqslant 1,$ let us consider a positive  left-invariant $\mathcal{L}$-elliptic continuous linear operator $A:C^\infty(G)\rightarrow\mathscr{D}'(G)$ with symbol  $\sigma\in {S}^{m,\mathcal{L}}_{\rho}( \widehat{G})$, $m\in \mathbb{R}$. Let us assume that $A$ commutes with $\mathcal{L}.$ Then, $A$ belongs to the Dixmier ideal $\mathscr{L}^{1,\infty}(L^2(G))$ with $\textnormal{\bf{Tr}}_{w}(A)<\infty,$ if and only if $m\leqslant -Q.$ If $A\neq 0,$ $\textnormal{\bf{Tr}}_{w}(A)\asymp\frac{1}{Q}$ for $m=-Q,$ and for $m<-Q,$ $\textnormal{\bf{Tr}}_{w}(A)=0.$

\end{lemma}
\begin{proof}
Let us use the positivity of $A$ computing the Dixmier trace of $A$ from the identity (see e.g. Sukochev and Usachev \cite[Page 114]{Sukochev})
\[ 
\textnormal{\bf{Tr}}_{w}(A)=\lim_{p\rightarrow 1^{+}}    (p-1)\textnormal{\bf{Tr}}(A^p).
\]
At the level of symbols, if $A$ commutes with $\mathcal{L},$ for every $[\xi]\in \widehat{G},$ $\sigma(\mathcal{\xi})$ commutes with $\widehat{\mathcal{L}}(\xi)$  and consequently, $\sigma(\mathcal{\xi})$ and  $\widehat{\mathcal{L}}(\xi)$ are simultaneously diagonalisable on every  representation space. So, in a suitable basis of the representation space we can write,
\[ 
    \sigma(\xi)=\textnormal{diag}[\sigma_{jj}(\xi)]_{j=1}^{d_\xi},\,\,\,\widehat{\mathcal{L}}(\xi)=\textnormal{diag}[(1+\nu_{jj}(\xi)^2)^{\frac{1}{2}}]_{j=1}^{d_\xi},
\]where $\sigma_{jj}(\xi),$ $1\leqslant k\leqslant d_{\xi},$ is the system of positive eigenvalues of $\sigma(\xi),$ $[\xi]\in \widehat{G}.$ The spectral mapping theorem implies that
\[ 
    \textnormal{spectrum}(A^p)=\{{\sigma_{jj}(\xi)^p}:1\leqslant j\leqslant d_\xi,\,\,[\xi]\in \widehat{G}\}.
    \]So, we have
    \begin{align*}
        \textnormal{\textbf{Tr}}(A^p)=\sum_{[\xi]\in  \widehat{G}}\sum_{j=1}^{d_{\xi}}{\sigma_{jj}(\xi)}^p,\,\,\,\,
    \end{align*} The $\mathcal{L}$-ellipticity of $A,$ implies that,
    \[ 
        \sup_{1\leqslant j\leqslant d_\xi}\sigma_{jj}(\xi)^{-1}(1+\nu_{jj}(\xi)^2)^{\frac{m}{2}}=\Vert \sigma(\xi)^{-1}\widehat{\mathcal{M}}(\xi)\Vert_{\textnormal{op}}\leqslant  \sup_{[\xi]\in \widehat{G}}\Vert \sigma(\xi)^{-1}\widehat{\mathcal{M}}(\xi)\Vert_{\textnormal{op}}<\infty.
    \]Consequently,
    \[ 
        \inf_{1\leqslant j\leqslant d_\xi}\sigma_{jj}(\xi)(1+\nu_{jj}(\xi)^2)^{-\frac{m}{2}}\geqslant     \sup_{[\xi]\in \widehat{G}}\Vert \sigma(\xi)^{-1}\widehat{\mathcal{M}}(\xi)^{m}\Vert_{\textnormal{op}}^{-1}.
    \]
Now, observe that from the hypothesis $\sigma\in {S}^{m,\mathcal{L}}_{\rho}( \widehat{G})$ we have,
\[ 
  \sup_{1\leqslant j\leqslant d_\xi}\sigma_{jj}(\xi)(1+\nu_{jj}(\xi)^2)^{-\frac{m}{2}}\leqslant   \sup_{[\xi]\in \widehat{G}}\Vert \sigma(\xi)\widehat{\mathcal{M}}(\xi)^{-m}\Vert_{\textnormal{op}}.   
\]These inequalities imply that
\begin{align*}
     \textnormal{\textbf{Tr}}(A^p)&=\sum_{[\xi]\in  \widehat{G}}\sum_{j=1}^{d_{\xi}}{\sigma_{jj}(\xi)^p (1+\nu_{jj}(\xi)^2)^{-\frac{pm}{2}}(1+\nu_{jj}(\xi)^2)^{\frac{mp}{2}}      }\\
     &\asymp \sum_{[\xi]\in  \widehat{G}}\sum_{j=1}^{d_{\xi}}  (1+\nu_{jj}(\xi)^2)^{\frac{mp}{2}} \\
     &=\textnormal{\textbf{Tr}}((1+\mathcal{L})^{\frac{mp}{2}}).
\end{align*}Now, as in the previous section, we will use the Weyl-law for the sub-Laplacian (see  Remark \ref{weyl}). Observe that,
\begin{align*}
  \textnormal{\textbf{Tr}}((1+\mathcal{L})^{\frac{mp}{2}})=\sum_{k=0}^{\infty}\sum_{[\xi]:2^{k}\leqslant (1+\nu_{j'j'}(\xi)^2)^\frac{1}{2} <2^{k+1},\,\forall 1\leqslant j'\leqslant d_{\xi} } \sum_{j=1}^{d_{\xi}}{(1+\nu_{jj}(\xi)^2)^{\frac{mp}{2}}      }.
\end{align*}Because,
\begin{align*}
     \sum_{j=1}^{d_{\xi}}{(1+\nu_{jj}(\xi)^2)^{\frac{mp}{2}}      }\asymp  {d_{\xi}}{2^{kmp}      }, 
\end{align*}we have 
\begin{align*}
    \textnormal{\textbf{Tr}}((1+\mathcal{L})^{\frac{mp}{2}})&=\sum_{k=0}^{\infty} {2^{kmp}}\sum_{[\xi]:2^{k}\leqslant (1+\nu_{j'j'}(\xi)^2)^\frac{1}{2} <2^{k+1},\,\forall 1\leqslant j'\leqslant d_{\xi} } d_\xi\\
    &=\sum_{k=0}^{\infty} {2^{kmp}}N(2^k)=\sum_{k=0}^{\infty} e^{2^{kmp}}2^{kQ}\\
    &=\sum_{k=0}^{\infty} {2^{kmp}}2^{k(Q-1)}2^{k}.
\end{align*}Observe that 
\begin{align*}
    \sum_{k=0}^{\infty} 2^{k(Q+mp-1)}2^{k}\asymp \int\limits_{1}^{\infty}\lambda^{Q+mp-1}d\lambda<\infty,
\end{align*} for all $p>1,$ if and only if $m\leqslant  -Q.$ So, from the identity
\[ 
    \int\limits_{1}^{\infty}\lambda^{Q+mp-1}d\lambda =-\frac{1}{Q+mp}.
\]
we deduce that
\[ 
    \textnormal{\bf{Tr}}_{w}(A)\asymp \lim_{p\rightarrow 1^{+}}(p-1)\times \frac{(-1)}{Q+mp}=\delta_{m,-Q}\times \frac{1}{Q},\,\,\,m\leqslant -Q, 
\] where $\delta_{m,-Q}$ is the Kronecker delta.
Thus, we end the proof.
\end{proof}
\begin{lemma}\label{lemmadixmier}
Let $G$ be a compact Lie group and let us denote by $Q$ the Hausdorff
dimension of $G$ associated to the control distance associated to the sub-Laplacian $\mathcal{L}=\mathcal{L}_X,$ where  $X= \{X_{1},\cdots,X_k\} $ is a system of vector fields satisfying the H\"ormander condition.  For  $0\leqslant \rho\leqslant 1,$ let us consider a positive  left-invariant $\mathcal{L}$-elliptic continuous linear operator $A:C^\infty(G)\rightarrow\mathscr{D}'(G)$ with symbol  $\sigma\in {S}^{m,\mathcal{L}}_{\rho}( \widehat{G})$, $m\in \mathbb{R}$. Then, $A$ belongs to the Dixmier ideal $\mathscr{L}^{1,\infty}(L^2(G))$ with $\textnormal{\bf{Tr}}_{w}(A)<\infty,$ if and only if $m\leqslant -Q.$ If $A\neq 0,$ $\textnormal{\bf{Tr}}_{w}(A)\asymp\frac{1}{Q}$ for $m=-Q,$ and for $m<-Q,$ $\textnormal{\bf{Tr}}_{w}(A)=0.$

\end{lemma}
\begin{proof}
Let us fix $p>1.$ We will compute the trace of $A^p$  using the formula,
\begin{align*}
     \textnormal{\textbf{Tr}}(A^p)&=\sum_{[\xi]\in  \widehat{G}}d_{\xi}\textnormal{\textbf{Tr}}[\widehat{A^p}(\xi)].
\end{align*} From the positivity  of $A,$
 in a suitable basis of the representation space we can diagonalise the operator $\sigma(\xi)  \widehat{\mathcal{M}}(\xi)^{-m},$ and we can write in such a basis,
\begin{equation}
    \sigma(\xi)^p  \widehat{\mathcal{M}}(\xi)^{-mp}=\textnormal{diag}[\lambda_{jj}(\xi)]_{j=1}^{d_\xi},\,\,\widehat{\mathcal{M}}(\xi)^{mp} =[\Omega_{ij}(\xi)]_{i,j=1}^{d_\xi}.
\end{equation}
Now, we can write
\begin{align*}
     \textnormal{\textbf{Tr}}(A^p)
     &=\sum_{[\xi]\in  \widehat{G}}d_{\xi}\textnormal{\textbf{Tr}}[\sigma(\xi)^p  \widehat{\mathcal{M}}(\xi)^{-mp}\widehat{\mathcal{M}}(\xi)^{mp} ]dx\\
     &=\sum_{[\xi]\in  \widehat{G}}\sum_{j,j'=1}^{d_\xi}d_{\xi}\lambda_{j'j'}(\xi)\Omega_{jj'}(\xi).
\end{align*}
The $\mathcal{L}$-ellipticity of $A$ implies the  $\mathcal{L}$-ellipticity of $A^p$ and the positivity  of its symbol, implies that
\begin{align*}
  \sup_{[\xi]\in \widehat{G}}\Vert \sigma(\xi)^{-p}\widehat{\mathcal{M}}(\xi)^{mp}\Vert_{\textnormal{op}}^{-1}&\leqslant   \inf_{j',j',[\xi]\in \widehat{G}}\lambda_{j'j'}(\xi)\\
  &\leqslant \sup_{j',j',[\xi]\in \widehat{G}}\lambda_{j'j'}(\xi)\\ &=\sup_{[\xi]\in \widehat{G}}\Vert \sigma(\xi)^p\widehat{\mathcal{M}}(\xi)^{-pm}\Vert_{\textnormal{op}},\hspace{2cm}
\end{align*}from which we deduce the following estimate,
\begin{align*}
     \textnormal{\textbf{Tr}}(A)
     &\asymp\sum_{[\xi]\in  \widehat{G}}\sum_{j,j'=1}^{d_\xi}d_{\xi}\Omega_{jj'}(\xi).
\end{align*}
Because
$\widehat{\mathcal{M}}(\xi)^{mp} =[\Omega_{ij}(\xi)]_{i,j=1}^{d_\xi}$ is a symmetric matrix written in the basis that allows to write in a  diagonal form the operator $\sigma(\xi)^p\widehat{\mathcal{M}}(\xi)^{-mp},$ we can find a matrix $P(\xi)$ such that
\[ 
    \widehat{\mathcal{M}}(\xi)^{mp} =P(\xi)^{-1}  \textnormal{diag}[  (1+\nu_{jj}(\xi)^2)^{\frac{mp}{2}} ]_{j=1}^{d_\xi}   P(\xi),
\]and 
\begin{align*}
   & \sum_{j,j'=1}^{d_\xi}\Omega_{j'j}(\xi)=\sum_{j,j'=1}^{d_\xi}[P(\xi)^{-1}  \textnormal{diag}[  (1+\nu_{ss}(\xi)^2)^{\frac{mp}{2}} ]_{s=1}^{d_\xi}   P(\xi)]_{j'j}\\
    &=\sum_{j,j',s=1}^{d_\xi}P(\xi)^{-1}_{j'j}    (1+\nu_{j'j'}(\xi)^2)^{\frac{mp}{2}}    P(\xi)_{j's}\\
    &=\textnormal{\textbf{Tr}}[P(\xi)^{-1}  \textnormal{diag}[  (1+\nu_{jj}(\xi)^2)^{\frac{mp}{2}} ]_{j=1}^{d_\xi}   P(\xi)]=\textnormal{\textbf{Tr}}[\widehat{\mathcal{M}}(\xi)^{mp}].
\end{align*}
Consequently, from Lemma \ref{dixmierlema1} we deduce that
\begin{align*}
     \textnormal{\textbf{Tr}}(A^p)
     &\asymp \textnormal{\textbf{Tr}}((1+\mathcal{L})^{\frac{mp}{2}}),
\end{align*}from which we deduce that
\begin{align*}
     \textnormal{\textbf{Tr}}_w(A)
     &\asymp \textnormal{\textbf{Tr}}_w((1+\mathcal{L})^{\frac{mp}{2}})=\frac{1}{Q}\delta_{m,-Q},\,\,\,m\leqslant -Q.
\end{align*}Thus, we end the proof.
\end{proof}

\begin{corollary}\label{beautifulproof}
For  $0\leqslant\delta< \rho\leqslant 1,$ (or for  $0\leq \delta<\frac{1}{\kappa},$  $0\leq \rho\leq 1,$ $\delta\leq \rho$)  let us consider a continuous linear operator $A:C^\infty(G)\rightarrow\mathscr{D}'(G)$ with symbol  $\sigma\in {S}^{m,\mathcal{L}}_{\rho,\delta}( G\times \widehat{G})$, with $m<-Q.$ Then $\textnormal{\bf{Tr}}_{w}(A)=0.$
\end{corollary}
\begin{proof}
    We will use the notation $$s_{1}(T)\leq s_{2}(T)\leq \cdots\leq s_{n}(T)\leq \cdots  $$  for the sequence of singular values of a compact operator $T$ on a Hilbert space $H.$ Then, the following inequality holds (see \cite[Page 75]{Bathia}): $s_{n}(CB)\leq \Vert C\Vert_{\textnormal{op}}s_{n}(B),$ for $C$ a bounded linear operator and $B$ a compact linear operator. From the definition of the functional $\textnormal{\bf{Tr}}_{w},$ we conclude easily that $0\leq \textnormal{\bf{Tr}}_{w}(CB)\leq \Vert C\Vert_{\textnormal{op}}\textnormal{\bf{Tr}}_{w}(B) .$ Now, let us use this inequality in our setting. From the subelliptic Calder\'on-Vaillancourt Theorem, we have that $A\mathcal{M}^{-m}\in  {S}^{0,\mathcal{L}}_{\rho,\delta}( G\times \widehat{G})$ extends to a bounded operator on $L^2(G),$ where $\mathcal{M}:=(1+\mathcal{L})^{\frac{1}{2}}.$ Consequently, 
    \begin{align*}
    0\leq \textnormal{\bf{Tr}}_{w}(A\mathcal{M}^{-m}\mathcal{M}^{m})\leq \Vert A\mathcal{M}^{-m}\Vert_{\mathscr{B}(L^2(G))}\textnormal{\bf{Tr}}_{w}(\mathcal{M}^{m})=0,    
    \end{align*}where we have used that $\textnormal{\bf{Tr}}_{w}(\mathcal{M}^{m})=0$ in view of Lemma \ref{lemmadixmier}. This implies that $\textnormal{\bf{Tr}}_{w}(A)=0.$ The proof is complete.
\end{proof}
\begin{remark}
It was proved in \cite{CdC2}
 that a Fourier multiplier $A\in \mathcal{L}^{(1,\infty)}(L^2(G)),$ if and only if,
 \begin{equation}
    \Vert \sigma_A(\xi)\Vert_{\mathcal{L}^{(1,\infty)}(\widehat{G})}:=\lim_{N\rightarrow\infty}\frac{1}{\log N}\sum_{[\xi]:\langle \xi\rangle\leq N}d_{\xi}\textnormal{\bf{Tr}}(|\sigma_A(\xi)|)<\infty,
\end{equation} and that  $\textnormal{\bf{Tr}}_\omega(A)= \Vert \sigma_A(\xi)\Vert_{\mathcal{L}^{(1,\infty)}(\widehat{G})}.$ The following result proves a more general statement. We follow the notation in Remark \ref{Deco:remark}.
\end{remark}

\begin{theorem}\label{Dtrace}
Let   $A\in \Psi^{-Q,\mathcal{L}}_{\rho,0}(G\times  \widehat{G})$ be an $\mathcal{L}$-elliptic pseudo-differential operator  and let $0<\rho\leq 1$. Assume that the symbol of $A$ admits an asymptotic expansion
\[ 
    \sigma_A(x,[\xi])\sim \sum_{k=-\infty}^{-Q}\sigma_{k}(x,[\xi]),\,
\]in components with decreasing order, which means that, for any $N\in \mathbb{N},$
\[ 
     \sigma_A(x,[\xi])- \sum_{k=-N-Q}^{-Q}\sigma_{k}(x,[\xi])\in S^{-(N+1)\rho-Q}_{\rho,0}(G\times \widehat{G}).
\]
Then $A\in \mathcal{L}^{(1,\infty)}(L^2(G)),$ and its Dixmier trace is given by
\begin{eqnarray*}
     \textnormal{\bf{Tr}}_{\omega}(A)=\int\limits_{G}\left(\Vert \textnormal{Re}(\sigma_{-Q}(x,[\xi]))^{+} \Vert_{\mathcal{L}^{(1,\infty)}(\widehat{G})}-\Vert \textnormal{Re}(\sigma_{-Q}(x,[\xi]))^{-} \Vert_{\mathcal{L}^{(1,\infty)}(\widehat{G})}
     \right)dx\\
    +i\int\limits_{G}\left(\Vert \textnormal{Im}(\sigma_{-Q}(x,[\xi]))^{+} \Vert_{\mathcal{L}^{(1,\infty)}(\widehat{G})}-\Vert \textnormal{Im}(\sigma_{-Q}(x,[\xi]))^{-} \Vert_{\mathcal{L}^{(1,\infty)}(\widehat{G})}
    \right)dx.
\end{eqnarray*}
\end{theorem}
\begin{proof} For the proof of Theorem 8.14  using the subelliptic functional calculus in Subsection \ref{S8} and Corollary \ref{beautifulproof}  we refer the reader to \cite{CDR21b}.
\end{proof}
\begin{remark}
Other results about the classification of the Dixmier trace of pseudo-differential operators on compact manifolds, with or without boundary (or on the lattice $\mathbb{Z}^n$) can be found in \cite{CdC1,CdC2} and \cite{CdC3}. 
\end{remark}

\subsection{Subelliptic operators in Schatten classes in $L^2(G)$}\label{SC} As a consequence of our analysis on the Dixmier traceability of subelliptic operators we will explain its consequences in the classification of subelliptic operators in Schatten classes. Let us record that, 
 if $A$ is a compact operator on a Hilbert space $H$ and  $\{s_{n}(A)\}$ denotes the sequence of its singular values, the Schatten von-Neumann class of order $r>0,$ $S_{r}(H),$ consists of all compact operators $A$ on $H$ such that $$\Vert A\Vert_{S_r(H)}:=\left(\sum_{n=0}^{\infty}s_n(A)^r\right)^{\frac{1}{r}}<\infty.$$ Let us recall that the following inequality holds (see e.g. \cite[Page 75]{Bathia}): $s_{n}(CB)\leq \Vert C\Vert_{\textnormal{op}}s_{n}(B),$ for  a bounded linear operator $C$ and  a compact linear operator $B$. 
\begin{corollary}
Let $G$ be a compact Lie group and let us denote by $Q$ the Hausdorff
dimension of $G$ associated to the control distance associated to the sub-Laplacian $\mathcal{L}=\mathcal{L}_X,$ where  $X= \{X_{1},\cdots,X_k\} $ is a system of vector fields satisfying the H\"ormander condition. Let $0\leq \delta<\rho\leq 1,$  (or let $0\leq \delta\leq \rho\leq 1, $ $\delta<1/\kappa$) and let $r>0$. If $T\in \textnormal{Op}({S}^{-m,\mathcal{L}}_{\rho,\delta}(G\times \widehat{G})),$  then $T\in S_{r}(L^2(G)) $ if  $m>Q/r,$ and this condition on the order is sharp.
\end{corollary}
\begin{proof}
Indeed, the inequality 
\begin{align*}
    s_{n}(T)=s_{n}(T\mathcal{M}_{m}\mathcal{M}_{-m})\leq \Vert T\mathcal{M}_{m}\Vert_{\mathscr{B}(L^2(G))}s_{n}(\mathcal{M}_{-m}),
\end{align*}implies that $$ \Vert T\Vert_{S_r(L^2(G))} \leq \Vert T\mathcal{M}_{m}\Vert_{\mathscr{B}(L^2(G))}\Vert \mathcal{M}_{-m}\Vert_{S_r(L^2(G))}=\Vert T\mathcal{M}_{m}\Vert_{\mathscr{B}(L^2(G))}\Vert \mathcal{M}_{-mr}\Vert^{\frac{1}{r}}_{S_1(L^2(G))} . $$   We can use the subelliptic Calder\'on-Vaillancourt Theorem (see Theorem \ref{CVT}) to deduce that $\Vert T\mathcal{M}_{m}\Vert_{\mathscr{B}(L^2(G))}<\infty.$ Because of Corollary \ref{beautifulproof}, $\textnormal{\bf Tr}_{w}(\mathcal{M}_{-mr})=0,$ for $mr>Q,$ which means that $m>Q/r.$ But, it is well known that the Dixmier functional $\textnormal{\bf Tr}_{w}$ vanishes on the Schatten class of order $r=1,$ (see e.g. Connes \cite{Connes94}). This fact, that $m>r/Q$ and Corollary \ref{beautifulproof} implies that $\Vert \mathcal{M}_{-mr}\Vert_{S_1(L^2(G))}<\infty,$ which implies that $ \Vert T\Vert_{S_r(L^2(G))}  <\infty.$ The sharpness of the order $m=Q/r,$ can be explained as follows. For  $T=\mathcal{M}_{-Q/r}\in \textnormal{Op}( S^{-Q/r,\mathcal{L}}_{\frac{1}{\kappa},0}(G\times \widehat{G})),$ (see Remark \ref{sgeq1})  then from Theorem \ref{Dtrace}, $\textnormal{\bf Tr}_{w}(T^r)\asymp 1/Q,$ which certainly implies (see Connes \cite{Connes94}) that, $\Vert \mathcal{M}_{-Q}\Vert_{S_{1}(L^2(G))}=\Vert \mathcal{M}_{-Q/r}\Vert^r_{S_r(L^2(G))}=\infty.$
\end{proof}
A complete investigation about the spectral trace of global operators on compact Lie groups can be found in \cite{DelRuzTrace1,DelRuzTrace11,DelRuzTrace111,DelRuzTrace1111}.
\subsection{Compactness and Gohberg lemma}
 In this subsection, the compactness of subelliptic pseudo-differential operators is discussed. So, we will deduce necessary conditions for the compactness of subelliptic pseudo-differetial operators by using the Gohberg Lemma proved for arbitrary compact Lie groups in \cite{ARLG}. Gohberg lemmas are results about the compactness  of continuous linear operators. In the case of a compact Lie group $G,$ a left-invariant operator with subelliptic negative order is compact on $L^2(G),$  (as a consequence of the Plancherel theorem). We will explain why a similar result is valid for pseudo-differential operators. 
 The original Gohberg Lemma was proved by Israel Gohberg in his investigation of integral operators with kernels (see \cite{Goh60}). On the other hand,  its version on $L^2(\mathbb{S}^1)$ was proved in  \cite{Mo} and it was extended to $L^p(\mathbb{S}^1),$ $1<p<\infty,$ in \cite{VR}. The Gohberg Lemma for compact Lie groups was obtained in \cite{ARLG} and it was generalised for $L^p$-spaces on compact manifolds  for the pseudo-differential calculus developed by the second author and Tokmagambetov (see \cite{Ruz-Tok,ProfRuzM:TokN:20017}), in \cite{RuzVR}.  Let us record (see Theorem 3.1 in \cite{ARLG}, but for the following statement see  Remark 4.2 of  \cite{ARLG})  the following theorem.

 \begin{theorem}\label{ADMR}
 Let $A:C^\infty(G)\rightarrow\mathscr{D}'(G)$ be a continuous linear operator admitting a bounded extension on $L^2(G).$ If the matrix valued symbol of $A,$ $a(x,\xi)$ satisfies the following conditions
 \begin{align}\label{Diana}
     \Vert a(x,\xi)\Vert_{\textnormal{op}}+\Vert \partial_{X}^{(\alpha)}a(x,\xi)\Vert_{\textnormal{op}}\leq C,\quad \Vert \Delta_{\xi}^{\beta}a(x,\xi)\Vert_{\textnormal{op}}\leq C\langle\xi\rangle^{-\epsilon},
 \end{align}where $\alpha, \beta\in \mathbb{N}_0,$ $|\alpha|=|\beta|=1$ and for some $\epsilon\in (0,1),$ then for all  compact operators $K$ on $L^2(G),$
 \[ 
     \Vert A-K \Vert_{\mathscr{B}(L^2(G))}\geq d_{\textnormal{min}},
 \]where
 \[ 
     d_{\textnormal{min}}:=\limsup_{\langle \xi\rangle\rightarrow\infty}\sup_{x\in G}\frac{\Vert a(x,\xi)a(x,\xi)^*\Vert_{\textnormal{min}} }{  \Vert a(x,\xi) \Vert_{\textnormal{op}} }\lesssim \limsup_{\langle \xi\rangle\rightarrow\infty}\sup_{x\in G}\Vert a(x,\xi)\Vert_{\textnormal{op}} ,
 \] with  $\Vert a(x,\xi)a(x,\xi)^*\Vert_{\textnormal{min}}:=\min\{\lambda\geq 0:\textnormal{det}(a(x,\xi)a(x,\xi)^*-\lambda I_{d_\xi})= 0$\}.

 \end{theorem}
 As a consequence of Theorem \ref{ADMR} we have:

 \begin{corollary}\label{ADMR2}
 Let $A:C^\infty(G)\rightarrow\mathscr{D}'(G)$ be a continuous linear operator admitting a bounded extension on $L^2(G).$ If the matrix valued symbol of $A,$ $a(x,\xi)$ satisfies the following conditions
 \begin{align*}
     \Vert a(x,\xi)\Vert_{\textnormal{op}}+\Vert \partial_{X}^{(\alpha)}a(x,\xi)\Vert_{\textnormal{op}}\leq C,\quad \Vert\widehat{\mathcal{M}}(\xi)^{\tau} \Delta_{\xi}^{\beta}a(x,\xi)\Vert_{\textnormal{op}}\leq C,
 \end{align*}where $\alpha, \beta\in \mathbb{N}_0,$ $|\alpha|=|\beta|=1$ and for some $\tau\in (0,1),$ then for all  compact operators $K$ on $L^2(G),$
 \[ 
     \Vert A-K \Vert_{\mathscr{B}(L^2(G))}\geq d_{\textnormal{min}}.
 \]

 \end{corollary} 
\begin{proof}
In view of Theorem \ref{ADMR}, we only need to prove that the condition $$\Vert\widehat{\mathcal{M}}(\xi)^{\tau} \Delta_{\xi}^{\beta}a(x,\xi)\Vert_{\textnormal{op}}\leq C,$$ implies the existence of $\epsilon>0,$ such that the second inequality  of \eqref{Diana} holds true. Indeed, observe that for $|\beta|=1$,
\begin{align*}
    \Vert \Delta_{\xi}^{\beta}a(x,\xi)\Vert_{\textnormal{op}}&=\Vert\widehat{\mathcal{M}}(\xi)^{-\tau}\widehat{\mathcal{M}}(\xi)^{\tau} \Delta_{\xi}^{\beta}a(x,\xi)\Vert_{\textnormal{op}}\\
    &\leq C\Vert\widehat{\mathcal{M}}(\xi)^{-\tau}\Vert_{\textnormal{op}}\leq \langle \xi\rangle^{-\frac{\tau}{\kappa}}, 
\end{align*}so, we can pick $\epsilon=\frac{\tau}{\kappa}$ in order to use Theorem \ref{ADMR}. The proof is complete.
\end{proof}
\begin{remark}
If $A$ satisfying the hypothesis in Corollary \ref{ADMR2} admits a compact linear extension on $L^2(G),$ then $d_{\textnormal{min}}=0.$ By observing the approach in \cite{RuzVR} we expect that the condition  $\limsup_{\langle \xi\rangle\rightarrow\infty}\sup_{x\in G}\Vert a(x,\xi)\Vert_{\textnormal{op}}=0$ implies that $A$ is compact on $L^p(G)$ for all $1<p<\infty$. However, the characterisation of the $L^p$-compactness of subelliptic pseudo-differential operators is still an open problem.
\end{remark}

\subsection{Multipliers of the sub-Laplacian and Hulanicki Theorem}

The aim of this subsection is to prove that the subelliptic classes $S^{m,\mathcal{L}}_{1,0}(G\times \widehat{G})$ contain the spectral multipliers of the sub-Laplacian $\mathcal{L}.$ For $m\in \mathbb{R},$ let us consider the Euclidean class of symbols $S^{m}(\mathbb{R}^{+}_0),$ defined by the countable family of seminorms 
\begin{equation}
    \Vert f \Vert_{d,\,S^{m}(\mathbb{R}^{+}_0)}:=\sup_{\ell\leq d}\sup_{\lambda\geq 0}(1+\lambda)^{-m+\ell}|\partial^\ell_{\lambda}f(\lambda)|,\,d\in \mathbb{N}_0.
\end{equation}We are going to prove that the subelliptic calculus is stable under the spectral functions of the sub-Laplacian.
\begin{theorem}\label{orders:theorems}Let $f\in S^{\frac{m}{2}}(\mathbb{R}^{+}_0), $ $m\in \mathbb{R}.$ Then, for all $t>0,$  $f(t\mathcal{L})\in S^{m,\mathcal{L}}_{1,0}(G\times \widehat{G}).$ Moreover, for $m<0,$
\begin{equation}
   \sup_{|\alpha|\leq d,\,[\xi]\in \widehat{G}} \Vert \widehat{\mathcal{M}}(\xi)^{-m+|\alpha|}\Delta_\xi^{\alpha}f(t\widehat{\mathcal{L}}(\xi)) \Vert_{\textnormal{op}}\lesssim t^{\frac{m}{2}} \Vert f \Vert_{d,\,S^{\frac{m}{2}}(\mathbb{R}^{+}_0)},\,\,0<t\leq 1,
\end{equation}uniformly in $t.$ In particular for $f_s(\lambda):=(1+\lambda)^{\frac{s}{2}},$ the symbol of $\mathcal{M}_{s}=f_s(\mathcal{L})$ belongs to $ S^{s,\mathcal{L}}_{1}( \widehat{G})\subset S^{s,\mathcal{L}}_{1,0}(G\times \widehat{G}) $ for all $s\in \mathbb{R}.$ In other words, we have $\mathcal{M}_s:=(1+\mathcal{L})^{\frac{s}{2}}\in \Psi^{s,\mathcal{L}}_{1,0}(G\times \widehat{G})$ for all $s\in \mathbb{R}.$
 
\end{theorem}
Before  starting with the proof of Theorem \ref{orders:theorems}, let us record some properties of the subelliptic heat kernel $p_{t}:=e^{-t\mathcal{L}}\delta,$ that is the right convolution kernel of the operator $e^{-t\mathcal{L}}.$
    \begin{remark}[The subelliptic heat kernel]\label{remarksub:heatkernel}Let $t>0,$ and let $\alpha\in \mathbb{N}_0.$ Let $|\cdot|_{cc}$ be the Carnot-Carath\'eodory distance associated to the control distance defined by the sub-Laplacian. Let $V(r)\sim r^{Q}$ be the volume of the ball with radious $r>0$ centred at $e_G\in G.$ The following statements hold (see e.g.  \cite[Section 2]{MartiniOtaziVallarino}):
\begin{itemize}
    \item[(A)] $\int\limits_{G}p_{t}(x)dx=1,$ $p_{t}(x^{-1})=p_{t}(x)$ and $p_{t}\ast p_{s}=p_{s}\ast p_{t}.$\\
    
   \item[(B)] $|p_{t}(x)|\lesssim V(\sqrt{t})^{-1}e^{-|x|^2/Ct},$ $x\in G,$ $t>0.$ \\
   
    \item[(C)] $|\partial_{X}^{(\alpha)}p_{t}(x)|\lesssim \sqrt{t}^{-Q-|\alpha|}e^{-|x|^2/Ct},$ $x\in G,$ $t>0.$ \\
    \item[(D)] $\int\limits_{G} e^{-|x|^2/Ct}\lesssim V(\sqrt{t}).$
    \end{itemize} Also,  
    \begin{itemize}
        \item[(E)] for $q\in C^{\infty}(G),$ with $|x|^{s}\lesssim q(x)\lesssim |x|^{s},$ $s\geq 0,$ one has \begin{equation}\label{qonball}
            \Vert q\Vert_{L^2(B(r))}\lesssim_q \min(1,r^{s+\frac{Q}{2}}),\,r>0.
        \end{equation}Indeed, using polar coordinates one gets
        $$  \Vert q\Vert_{L^2(B(r))}\asymp \int\limits_{|x|<r}|x|^{2s}dx
        \lesssim \int\limits_{0}^{r}s^{2s}s^{Q-1}ds\lesssim r^{2s+Q}. $$ Since $q\in C^{\infty}(G),$ $\Vert q\Vert_{L^2(B(r))}\leq \Vert q\Vert_{L^{\infty}(G)}\lesssim_q 1.$ The previous estimates yield \eqref{qonball}.
    \end{itemize}
One of the remarkable properties of the sub-Laplacian is the finite propagation speed property. We record it as follows.
\begin{remark}[Finite propagation speed property] Let $f:\mathbb{R}\rightarrow\mathbb{C}$ be an even, bounded and measurable  function and let  $k_{f(\mathcal{L})}:=f(\mathcal{L})\delta$ be  the right-convolution kernel of the operator $f(\mathcal{L}).$ The finite  propagation speed property for the sub-Laplacian was proved e.g. by   Melrose \cite[Section 3]{Melrose} and  Jerison and S\'anchez-Calle \cite{JerisonSanchezCalle}. It say that,     for any $t>0,$
\begin{equation}
    \textnormal{supp}(k_{\cos(t\mathcal{L})})\subset B(c|t|),
\end{equation} for some $c>0.$ Then, as it is well known, (see e.g.  Lemma 2.1 of Cowling and Sikora \cite{CowlingSikora}) the following finite propagation speed property holds
\begin{equation}
    \textnormal{supp}(k_{f(\mathcal{L})})\subset B(cR),
\end{equation} for some constant $c>0,$ independent of $f,$ and $r,$ provided that $\textnormal{supp}(\widehat{f})\subset [-r,r],$ $r>0.$ Here, $\widehat{f}(\xi):=\int\limits_{-\infty}^{\infty}e^{-i2\pi x\cdot \xi}f(t)dt$ is the Euclidean Fourier transform of $f.$

\end{remark}

\end{remark}
The main tool of this subsection is the following subelliptic version of the Hulanicki Theorem. We denote $B(r):=B(e_G,r),$ the ball with radious $r>0,$ centred at $e_G.$

\begin{lemma}[Subelliptic Hulanicki Theorem]\label{HulanickiTheorem} Let $G$ be a compact Lie group and let $q$ be a smooth function on $G.$ Let us consider a continuous function $f$ on $\mathbb{R}$ supported in $[0,2].$ The following statements hold.
\begin{itemize}
    \item Let $\epsilon_0>0$  and let $t\geq \epsilon_0.$  Then
    \begin{equation}
        \int\limits_{G}|q(x)f(t\mathcal{L})\delta (x)|dx\lesssim_{\epsilon_0}\Vert f \Vert_{L^{\infty}(G)}.
    \end{equation}
    \item If in addition $f\in C^{d}(\mathbb{R}),$ and $\Delta_q\in \textnormal{diff}^{a}(\widehat{G}),$ then the following estimate holds
    \begin{equation}
        \int\limits_{G}|q(x)\partial_{X}^{(\beta)}(f(t\mathcal{L})\delta)(x)|dx\lesssim t^{\frac{a-|\beta|}{2}}\sup_{0\leq \ell\leq d}\sup_{0\leq \lambda\leq 2}|\partial^{\ell}_{\lambda}f(\lambda)|.
    \end{equation}
\end{itemize}

\end{lemma}
\begin{proof} We follow the approach of the proof for the Hulanicki theorem on graded Lie groups done in \cite{FischerRuzhanskyBook}. Let $q$ be a smooth function on $G.$ Let us consider   $f\in C^{\infty}(\mathbb{R}_0^{+})$ supported in  $[0,2]$. 
We will divide the proof of Lemma \ref{HulanickiTheorem} in four steps.
\begin{itemize}
    \item \textbf{Step I:} 
Fix $t>0.$ Let us consider the function $h_t:[0,\infty)\rightarrow \mathbb{C}$ defined by
\begin{equation}
\label{pf_lem_prop_mult_1_def_ht}
h_t(\mu)=e^{-t\mu^2}f(t\mu^2),
\quad \mu\geq 0.
\end{equation}
We have
$$
  \Vert  h_t  \Vert  _{L^\infty} \leq C   \Vert  f  \Vert  _{L^\infty}
\quad\mbox{and}\quad
f(t\lambda)=h_t(\sqrt \lambda ) e^{t\lambda}.
$$
By the spectral Functional calculus we have  
$$
f(t\mathcal{L})\delta=h_t(\sqrt\mathcal{L}) p_t
\quad\mbox{and}\quad
  \Vert  f(t\mathcal{L})\delta  \Vert  _{L^2(G)}\leq 
  \Vert  h_t   \Vert  _{L^\infty}   \Vert  p_t  \Vert  _{L^2(G)}.
$$
For the $L^2$-norm of the heat kernel, we use Remark \ref{remarksub:heatkernel} to observe that 
$$
  \Vert  p_t  \Vert  _{L^2(G)} \leq C V(\sqrt t)^{-\frac12}.
$$ Indeed, (B) and (D) in Remark \ref{remarksub:heatkernel} imply
\begin{align*}
    \int\limits_G|p_t(x)|^2dx &\lesssim  V(\sqrt{t})^{-2}\int\limits_Ge^{-|x|^2/Ct}e^{-|x|^2/Ct}dx\\
    &\lesssim V(\sqrt{t})^{-2}V(\sqrt{t})\int\limits_Ge^{-|x|^2/Ct}dx\\
     &\lesssim V(\sqrt{t})^{-1}.
\end{align*}

This implies that  $f(t\mathcal{L})\delta \in L^2(G)$ and that:
\begin{equation}
\label{pf_lem_prop_mult_1_step1}
  \Vert  f(t\mathcal{L})\delta  \Vert  _{L^2(G)}\lesssim 
  \Vert  f  \Vert  _{L^\infty} V(\sqrt t)^{-\frac12}.
\end{equation}
\item \textbf{Step II:} Let us show that the integral in the statement on a ball of radius $\sqrt t,$ $B(\sqrt{t}),$ for $t\rightarrow 0^{+}$ can be estimated by
\begin{equation}
\int\limits_{|x|<\sqrt t} 
| q(x) \ f(t\mathcal{L})\delta(x)| dx 
\lesssim_q \min \{1, t^{\frac {a} 2}\}
  \Vert  f  \Vert  _{L^\infty}.
\label{pf_lem_prop_mult_1_step2}
\end{equation}
To analyse
this, observe that the   Cauchy-Schwarz inequality implies
$$
\int\limits_{|x|<\sqrt t} 
| q(x) \ f(t\mathcal{L})\delta(x)| dx 
\leq
  \Vert  q  \Vert  _{L^2(B(\sqrt t))}
  \Vert  f(t\mathcal{L})\delta  \Vert  _{L^2(B(\sqrt t))}.
$$
The first $L^2$-norm of the right-hand side
may be estimated 
using Remark \ref{remarksub:heatkernel}
and the second with \eqref{pf_lem_prop_mult_1_step1}:
$$
  \Vert  f(t\mathcal{L})\delta  \Vert  _{L^2(B(\sqrt t))}
\leq 
  \Vert  f(t\mathcal{L})\delta  \Vert  _{L^2(G)}
\lesssim 
  \Vert  f  \Vert  _{L^\infty} V(\sqrt t)^{-\frac12}.
$$
Hence, Part (D) in Remark \ref{remarksub:heatkernel} implies $$
\int\limits_{|x|<\sqrt t} 
| q(x) \ f(t\mathcal{L})\delta(x)| dx 
\lesssim_q \min (1,{\sqrt t} ^{a +\frac Q2})
  \Vert  f  \Vert  _{L^\infty} V(\sqrt t)^{-\frac12}.
$$
Using that $V(r)\sim r^{Q},$ for $r\rightarrow 0^{+}$, 
we have proved \eqref{pf_lem_prop_mult_1_step2}.
\item \textbf{Step III:}  Let us consider the ball $B(e_{G},\epsilon_0)$ and let $t\geq \epsilon_{0}.$
Then the first part of Lemma \ref{HulanickiTheorem} has been proved because $t\geq \epsilon_0$. Let $\Delta_{q}\in \textnormal{diff}^{\,a}(\widehat{G}),$ with $a\in \mathbb{N}.$
Let us now consider the case of a multi-index $\beta\in \mathbb{N}_0^n$, 
and still the case when $\sqrt t\asymp \epsilon_0$.
Same as in Steps I and II, we get
$$
  \Vert  \partial_{X}^{(\beta)} f(t\mathcal{L})\delta  \Vert  _{L^2(G)}
=
  \Vert   h_t(\sqrt\mathcal{L}) \{\partial_{X}^{(\beta)} p_t\}  \Vert  _{L^2(G)}
\leq 
  \Vert  h_t   \Vert  _{L^\infty}   \Vert  \partial_{X}^{(\beta)} p_t  \Vert  _{L^2(G)},
$$
and (C) in Remark 8.18 implies 
\begin{align*}
  \Vert  \partial_{X}^{(\beta)} p_t  \Vert  _{L^2(G)} &\leq \Vert \sqrt{t}^{-Q-|\beta|} e^{-|x|^2/Ct} \Vert_{L^2(G)}\\
  & \leq  \sqrt{t}^{-Q-|\beta|} \Vert e^{-|x|^2/Ct} \Vert_{L^\infty(G)}\\
  &\lesssim \sqrt{\epsilon_0}^{-Q-|\beta|} \\ &\lesssim_{\epsilon_0} 1. 
\end{align*}
Consequently,
$$
  \Vert  f(t\mathcal{L})\delta  \Vert  _{L^2(G)}\lesssim_{\epsilon_0}  
   \Vert  f  \Vert  _{L^\infty}.
$$
Thus the second part of Lemma \ref{HulanickiTheorem} is proved for $t\asymp 1$.
We therefore may assume that $t$ is small enough and then let us fix a multi-index $\beta\in \mathbb{N}_0^n$. 

\item \textbf{Step IV:}  
In order to finish the proof, we only need to prove that
\begin{equation}
\forall \sqrt t<\epsilon_0,\qquad
\int\limits_{|x|\geq \sqrt t} 
| q(x) \partial_{X}^{(\beta)}\{f(t\mathcal{L})\delta\}(x)| dx 
\lesssim
C_q  t^{\frac {a-|\beta|} 2}
  \Vert  f  \Vert  _{C^d}.
\label{pf_lem_prop_mult_1_step3}
\end{equation} Indeed, using \eqref{pf_lem_prop_mult_1_step3} we can estimate $\int\limits_{|x|\leq \sqrt t} 
| q(x) \partial_{X}^{(\beta)}\{f(t\mathcal{L})\delta\}(x)| dx, $ as follows,
\begin{align*}
    \int\limits_{|x|\leq \sqrt t} 
| q(x) \partial_{X}^{(\beta)}\{f(t\mathcal{L})\delta\}(x)| dx &=\sum_{\ell=0}^{\infty}\int\limits_{\frac{ \sqrt t}{2^{\ell+1}}\leq |x|\leq \frac{ \sqrt t}{2^{\ell}} } 
| q(x) \partial_{X}^{(\beta)}\{f(t\mathcal{L})\delta\}(x)| dx \\
&=\sum_{\ell=0}^{\infty}\int\limits_{\frac{ \sqrt t}{2^{\ell+1}}\leq |x|} 
| q(x) \partial_{X}^{(\beta)}\{f(t\mathcal{L})\delta\}(x)| dx\\
&\leq \sum_{\ell=0}^{\infty}C_q  \left( \frac{ \sqrt t}{2^{\ell+1}}\right)^{\frac {a-|\beta|} 2}
  \Vert  f  \Vert_{C^d}\\
  &\lesssim   {\sqrt t}^{\frac {a-|\beta|} 2}
  \Vert  f  \Vert_{C^d}.
\end{align*}So, in order to prove the estimate \eqref{pf_lem_prop_mult_1_step3} we will decompose the function inside the integral as follows,
$$
\partial_{X}^{(\beta)} \{f(t\mathcal{L})\delta\}=h_t(\sqrt\mathcal{L}) \partial_{X}^{(\beta)} p_t
= 
h_t(\sqrt\mathcal{L})\sum_{j=0}^\infty 
 \{\partial_{X}^{(\beta)} p_t\} \ 1_{B(2^{j-1}\sqrt t)} 
 + \{\partial_{X}^{(\beta)} p_t\} \ 1_{B(2^{j-1}\sqrt t)^c)}.
$$
Here $1_{B(r)}$ and $1_{B(r)^c}$
denote the characteristic functions of the  ball $B(r)$ centred at $e_{G}$ and its complement in $G,$ $B(r)^c:=G\setminus B(r)$. Let us record that
the function $h_t$ has been defined in \eqref{pf_lem_prop_mult_1_def_ht}.
Note that the sum over $j$
is finite but the number of terms is the smaller integer $J$ 
such that $$2^{J+1} \sqrt t > R_0:=\textnormal{Diameter}_{cc}(G)\equiv \sup_{x,\,y\in G}{|y^{-1}x|_{cc}},$$ thus $J=J(t)$ depends on $t$.
The decomposition above  yields
\begin{equation}
\label{pf_lem_prop_mult_1_step3_dec}
\int\limits_{|x|\geq \sqrt t} 
| q(x) \partial_{X}^{(\beta)}\{f(t\mathcal{L})\delta\}(x)| dx 
\leq
\sum_{j=0}^\infty 
\int\limits_{S_{t,j}} |q \ \mathcal{M}_{t,j}^{(1)}|
+
\int\limits_{S_{t,j}} |q \ \mathcal{M}_{t,j}^{(2)}|,
\end{equation}
where the spherical sectors $S_{t,j}$ are defined by
$$
S_{t,j}
:=
\{x\in G:\ 2^j\sqrt t < |x| \leq  2^{j+1}\sqrt t\}
= B(2^{j+1}\sqrt t) \backslash B(2^{j}\sqrt t),
$$
and 
$$
\mathcal{M}_{t,j}^{(1)}
:=
h_t(\sqrt{ \mathcal{L}}) \left\{\partial_{X}^{(\beta)} p_t \ 1_{B(2^{j-1}\sqrt t)}\right\}
\qquad\mbox{and}\qquad
\mathcal{M}_{t,j}^{(2)}:=
h_t(\sqrt{ \mathcal{L}}) \left\{\partial_{X}^{(\beta)} p_t \ 1_{B(2^{j-1}\sqrt t)^c}\right\}.
$$

In both cases $i=1,2$, 
we will use Cauchy-Schwarz' inequality
$$
\int\limits_{S_{t,j}} |q \ \mathcal{M}_{t,j}^{(i)}|
\leq
  \Vert  q  \Vert  _{L^2(S_{t,j})}
  \Vert  \mathcal{M}_{t,j}^{(i)}  \Vert  _{L^2(S_{t,j})}.
$$
For the first $L^2$-norm, we use Remark \ref{remarksub:heatkernel}
(with $t\rightarrow 0^{+}$):
$$
  \Vert  q  \Vert  _{L^2(S_{t,j})}
\leq 
  \Vert  q  \Vert  _{L^2(B(2^{j+1}\sqrt t))}
\lesssim C_q  (2^{j+1} \sqrt t)^{a+\frac Q 2}.
$$ 
\begin{itemize}
    \item \textbf{Step IV-(a):}  
For the second $L^2$-norm, in the case $i=2$, we have
$$
  \Vert  \mathcal{M}_{t,j}^{(2)}  \Vert  _{L^2(S_{t,j})}
\leq
  \Vert  \mathcal{M}_{t,j}^{(2)}  \Vert  _{L^2(G)}
\leq
  \Vert  h_t(\sqrt{ \mathcal{L}})  \Vert  _{ L^2(G) } 
  \Vert  \partial_{X}^{(\beta)} p_t \ 1_{B(2^{j-1}\sqrt t)^c}  \Vert  _{L^2(G)}.
$$
On the one hand, we have 
by the spectral theorem 
$$
  \Vert  h_t(\sqrt{ \mathcal{L}})  \Vert  _{ L^2(G) }\leq 
  \Vert  h_t  \Vert  _{L^\infty}
\leq e^2   \Vert  f  \Vert  _{L^\infty}.
$$
On the other hand, the estimate 
 for the heat kernel in Remark \ref{remarksub:heatkernel} yields
\begin{eqnarray*}
  \Vert  \partial_{X}^{(\beta)} p_t \ 1_{B(2^{j-1}\sqrt t)^c}  \Vert  _{L^2(G)}^2
&\leq &
\sup_{|x|\geq 2^{j-1}\sqrt t} |\partial_{X}^{(\beta)} p_t(x)|
\int_G  |\partial_{X}^{(\beta)} p_t(x)| dx
\\
&\lesssim &
\sqrt t^{-Q-|\beta|} e^{-\frac{2^{2(j-1)}}C}
\int_G \sqrt t^{-Q-|\beta|} e^{-\frac{|x|^2}{C t}} dx
\\
&\lesssim &
\sqrt t^{-Q-|\beta|} e^{-\frac{2^{2(j-1)}}C}\sqrt t^{-Q-|\beta|} V(\sqrt t)\\
&\lesssim &
\sqrt t^{-Q-|\beta|} e^{-\frac{2^{2(j-1)}}C}\sqrt t^{-Q-|\beta|} (\sqrt t)^{Q}\\
&\lesssim &
\sqrt t^{-Q-2|\beta|} e^{-\frac{2^{2(j-1)}}C}.
\end{eqnarray*} Thus we have obtained
$$
  \Vert  \partial_{X}^{(\beta)} p_t \ 1_{B(2^{j-1}\sqrt t)^c}  \Vert  _{L^2(G)}
\lesssim
\sqrt t^{-\frac Q2 -|\beta|} e^{-\frac{2^{2(j-1)}} {2C} },
$$
and 
$$
  \Vert  \mathcal{M}_{t,j}^{(2)}  \Vert  _{L^2(S_{t,j})}
\lesssim   \Vert  f  \Vert  _{L^\infty}  \sqrt t^{-\frac Q2-|\beta|} e^{-\frac{2^{2(j-1)}}{2C}}.
$$
In view of the  previous analysis we have,
\begin{eqnarray*}
\int\limits_{S_{t,j}} |q \ \mathcal{M}_{t,j}^{(2)}|dx
&\lesssim&
C_q  (2^{j+1} \sqrt t)^{a+\frac Q 2}
  \Vert  f  \Vert  _{L^\infty}  \sqrt t^{-\frac Q2-|\beta|} e^{-\frac{2^{2(j-1)}}{2C}}
\\
&\lesssim& C_q    \Vert  f  \Vert  _{L^\infty} \sqrt t^{a - |\beta|}
\ 
2^{(j+1)(a+\frac Q2)} e^{-\frac{2^{2(j-1)}}{ 2C}}.
\end{eqnarray*}
The exponential decay allows us to sum up over $j$ and to obtain:
\begin{equation}
\sum_{j=0}^\infty
\int\limits_{S_{t,j}} |q \ \mathcal{M}_{t,j}^{(2)}|
\lesssim
 C_q    \Vert  f  \Vert  _{L^\infty} \sqrt t^{a - |\beta|}.
\label{pf_lem_prop_mult_1_step3a}
\end{equation}

\item \textbf{Step IV-(b):}  
The case of $i=1$, that is, the estimate of $  \Vert  \mathcal{M}_{t,j}^{(1)}  \Vert  _{L^2(S_{t,j})}$, 
can be analysed as follows.
The function $h_t$ is even and has compact support. 
Assuming that $f\in C^d[0,+\infty),$ $d\geq 2$, 
 $h_t\in C^d(\mathbb{R}^d)$ admits an integrable Euclidean  Fourier transform   $\widehat h_t\in L^1(\mathbb{R})$ it follows that the following formula holds for all $\mu\in \mathbb{R}$
$$
h_t(\mu) = \frac 1{2\pi} \int\limits_{-\infty}^{\infty} \cos (s\mu)\ \widehat h_t(s)  ds, 
\quad \mu\in \mathbb{R},
$$
with a convergent integral.
The spectral  functional calculus of the operator $\sqrt{L}$ then implies 
$$
h_t(\sqrt{ \mathcal{L}})= \frac 1{2\pi} \int\limits_{-\infty}^{\infty} \cos (s\sqrt{ \mathcal{L}}) \ \widehat h_t(s) ds
$$
and also
\begin{equation}
\label{pf_lem_prop_mult_1_step3_Mtj1}
\mathcal{M}_{t,j}^{(1)}(x)=
\frac 1{2\pi} \int\limits_{-\infty}^{\infty} 
 \cos (s\sqrt{ \mathcal{L}}) 
\left\{\partial_{X}^{(\beta)} p_t \ 1_{B(2^{j-1}\sqrt t)}\right\}(x) \  \widehat h_t(s) ds.
\end{equation}

The operator $\cos (s\sqrt\mathcal{L})$ satisfies that
$\textnormal{supp}( \{\cos (s\sqrt\mathcal{L})\delta\}) \subset B(|s|)$.
This implies the following property,

$$
x\in S_{t,j} \ \mbox{and}\  |s|\leq 2^{j-1}\sqrt t
\ \Longrightarrow\
\cos (s\sqrt{ \mathcal{L}}) 
\left\{\partial_{X}^{(\beta)} p_t \ 1_{B(2^{j-1}\sqrt t)}\right\}(x)
=0.
$$
Before continuing with the proof, we will use the following lemma (see \cite[Page 3466]{Fischer2015}).
\begin{lemma}
\label{lem_alexo_hg}Let us consider an even function 
 $g \in \mathscr{S} (\mathbb{R})$   
such that its (Euclidean) Fourier transform satisfies:
$$
\widehat g\in C^{\infty}_0(\mathbb{R}), 
\qquad
\widehat g\big|_{[-\frac 12,\frac 12]} \equiv1,
\qquad\mbox{and}\qquad 
\widehat g\big|_{(-\infty,1]\cup[1,\infty)} \equiv0.
$$
Then, for every $d\in \mathbb{N}$ and
 any $h\in \mathscr{S}'(\mathbb{R})$ 
with  $h\in C^d(\mathbb{R})$ and $  \Vert   \partial_\lambda^{d} h   \Vert  _{L^\infty}<\infty$,
we have
$$
\forall\epsilon>0,\qquad
  \Vert  h - h \ast g_\epsilon  \Vert  _{L^\infty} \leq 
\frac {\epsilon^d}{d!} \int\limits_{-\infty}^{\infty} |y|^d |g(y)| dy  \   \Vert   \partial_\lambda^{d} h   \Vert  _{L^\infty},
$$
where $g_\epsilon$ is the dilation of $g$  given by $g_\epsilon(x):=\epsilon^{-1} g(\epsilon^{-1}x),$ $x\in \mathbb{R}$. 
\end{lemma}
We use this lemma in the following way. 
Let $g \in \mathscr{S} (\mathbb{R})$ 
and $g_\epsilon=\epsilon^{-1} g(\epsilon^{-1}\cdot)$
be functions
as in Lemma \ref{lem_alexo_hg}.
As $\textnormal{supp}(  \widehat g_{(2^{j-1} \sqrt t)^{-1}}) 
\subset 
[-2^{j-1}\sqrt t, 2^{j-1}\sqrt t]$,
the finite propagation speed property implies
$$
x\in S_{t,j} \ \Longrightarrow\
 \int\limits_{-\infty}^{\infty} 
 \cos (s\sqrt{ \mathcal{L}}) 
\left\{\partial_{X}^{(\beta)} p_t \ 1_{B(2^{j-1}\sqrt t)}\right\}(x) \  
\widehat h_t(s) \widehat g_{(2^{j-1} \sqrt t)^{-1}}  (s) ds = 0.
$$
Hence we can rewrite \eqref{pf_lem_prop_mult_1_step3_Mtj1}
for any $x\in S_{t,j}$ as follows
\begin{align*}
\mathcal{M}_{t,j}^{(1)}(x)
&=
\frac 1{2\pi} \int\limits_{-\infty}^{\infty} 
 \cos (s\sqrt{ \mathcal{L}}) 
\left\{\partial_{X}^{(\beta)} p_t \ 1_{B(2^{j-1}\sqrt t)}\right\}(x) 
\\
&\left(  \widehat h_t(s) - \widehat h_t(s) \widehat g_{(2^{j-1} \sqrt t)^{-1}}
\right) ds
\\
&=
\left(h_t - h_t* g_{(2^{j-1} \sqrt t)^{-1}}\right)(\sqrt{ \mathcal{L}})
\left\{\partial_{X}^{(\beta)} p_t \ 1_{B(2^{j-1}\sqrt t)}\right\}(x), 
\end{align*}
having used the spectral theorem and the inverse Fourier formula for even  functions on $\mathbb{R}$.
Applying the $L^2$-norm on $S_{t,j}$, we obtain
\begin{eqnarray*}
  \Vert  \mathcal{M}_{t,j}^{(1)}  \Vert  _{L^2(S_{t,j})}
&\leq&
  \Vert  \left(h_t - h_t* g_{(2^{j-1} \sqrt t)^{-1}}\right)(\sqrt{ \mathcal{L}})
\left\{\partial_{X}^{(\beta)} p_t \ 1_{B(2^{j-1}\sqrt t)}\right\}  \Vert  _{L^2(G)}
\\
&\leq&
  \Vert   h_t - h_t* g_{(2^{j-1} \sqrt t)^{-1}}  \Vert  _{L^\infty}
  \Vert   \partial_{X}^{(\beta)} p_t \ 1_{B(2^{j-1}\sqrt t)}  \Vert  _{L^2(G)},
\end{eqnarray*}
by the spectral theorem.
We estimate the supremum norm with the result of Lemma \ref{lem_alexo_hg}:
$$
  \Vert   h_t - h_t* g_{(2^{j-1} \sqrt t)^{-1}}  \Vert  _{L^\infty}
\lesssim (2^{j-1} \sqrt t)^{-d}   \Vert  \partial_\lambda^{d} h_t   \Vert  _{L^\infty},
$$ 
and one checks easily 
$$
  \Vert  \partial_\lambda^{d} h_t   \Vert  _{L^\infty} 
= t^{\frac d2}   \Vert  h_1^{(d)}  \Vert  _{L^\infty} 
\lesssim t^{\frac d2}  \max_{\ell =0,1,\ldots,d} 
  \Vert  \partial_\lambda^{\ell} f  \Vert  _{L^\infty}.
$$
For the $L^2$-norm, the estimates in  Remark \ref{remarksub:heatkernel} for the heat kernel imply
\begin{align*}
     \Vert   \partial_{X}^{(\beta)} p_t \ 1_{B(2^{j-1}\sqrt t)}  \Vert  _{L^2(G)}
&\lesssim
\sqrt t^{-Q -|\beta|} V(2^{j-1}\sqrt t)^{\frac 12}\\
& \lesssim \sqrt t^{-Q -|\beta|}\gamma_0^{\frac j2}  V(\sqrt t)^{\frac 12}\\
& \lesssim \sqrt t^{-Q -|\beta|}\gamma_0^{\frac j2}(\sqrt t)^{\frac Q2}\\
&= \gamma_0^{\frac j2}  \sqrt t^{-\frac Q2 -|\beta|},
\end{align*}
 where we have used the doubling property:
$$
\gamma_0:=\sup_{r>0} \frac {V(2r)}{V(r)} \in (0,\infty).
$$
Hence we obtain
$$
  \Vert  \mathcal{M}_{t,j}^{(1)}  \Vert  _{L^2(S_{t,j})}
\lesssim 
(2^{j-1} \sqrt t)^{-d} t^{\frac d2}  \max_{\ell =0,1,\ldots,d} 
  \Vert  \partial_\lambda^{\ell} f  \Vert  _{L^\infty}
 \gamma_0^{\frac j2}  \sqrt t^{-\frac Q2 -|\beta|}.
$$

We can now go back to 
\begin{align*}
&\int\limits_{S_{t,j}} |q \ \mathcal{M}_{t,j}^{(1)}|\\
&\lesssim
C_q  (2^{j+1} \sqrt t)^{a+\frac Q 2}
(2^{j-1} \sqrt t)^{-d} t^{\frac d2}  \max_{\ell =0,1,\ldots,d} 
  \Vert  \partial_\lambda^{\ell} f  \Vert  _{L^\infty}
 \gamma_0^{\frac j2}  \sqrt t^{-\frac Q2 -|\beta|}
 \\
&\lesssim
C_q \max_{\ell =0,1,\ldots,d} 
  \Vert  \partial_\lambda^{\ell} f  \Vert  _{L^\infty}
2^{j(a+\frac Q2 -d + \frac {\ln \gamma_0}2)} \sqrt t^{a-|\beta|}.
\end{align*}
We choose $d$ to be the smallest positive integer such that 
$d>a+\frac Q 2+ \frac {\ln \gamma_0}2$ so that we can sum up over $j$ 
to obtain
$$
\sum_{j=0}^\infty \int\limits_{S_{t,j}} |q \ \mathcal{M}_{t,j}^{(1)}|
\lesssim
C_q \max_{\ell =0,1,\ldots,d} 
  \Vert  \partial_\lambda^{\ell} f  \Vert  _{L^\infty}
\sqrt t^{a-|\beta|}.
$$
Using \eqref{pf_lem_prop_mult_1_step3_dec} and \eqref{pf_lem_prop_mult_1_step3a}, 
this shows \eqref{pf_lem_prop_mult_1_step3}.
\end{itemize}

\end{itemize}

This concludes the proof of  Lemma \ref{HulanickiTheorem}.
\end{proof}
\begin{remark}
\label{rem_lem_prop_mult_1}

\begin{enumerate}

\item
\label{item_rem_lem_prop_mult_1_smooth}
Observe that if $f$ is compactly supported in $\mathbb{R}^{+}_0$, 
then the kernel of the operator $f({\mathcal{L}} )$ is smooth
and  the integrals  in Lemma \ref{HulanickiTheorem} are finite. 
However Lemma \ref{HulanickiTheorem}  yields
 bounds for these integrals in terms of  $f$ and $t$
 which we will use in our further analysis. 

\item
\label{item_rem_lem_prop_mult_1_Xq}
The second part of  
Lemma \ref{HulanickiTheorem} implies that for any difference operator $\Delta_q\in \textnormal{diff}^{\,a}(\widehat{G})$, $\beta,\gamma\in \mathbb{N}_0^n$, we have,
$$
\int_G | \partial_{X}^{(\gamma)} \{ q(x) X^{\beta}  \ f(t{\mathcal{L}} )\delta\}(x)| dx
\lesssim  t^{\frac {a-|\beta| -|\gamma|}  2}
\max_{\ell =0,1,\ldots,d} 
  \Vert  \partial_\lambda^{\ell} f  \Vert  _{L^\infty},
$$
This follows  from the Leibniz rule
$$
\partial_{X}^{(\gamma)} (q\phi)(x) 
= \sum_{|\gamma_1|+|\gamma_2| =|\gamma|}
c_{\gamma_1,\gamma_2}\partial_{X}^{(\gamma_1)}q(x) \partial_{X}^{(\gamma_2)} \phi(x).
$$
Indeed, note that $\partial_{X}^{(\gamma_1)}q$ vanishing at $e_G$ up to order $a-1 -|\gamma|$. Here, we denote $\phi=f(t{\mathcal{L}} )\delta$.
\end{enumerate}

\end{remark}

\begin{proof}[Proof of Theorem \ref{orders:theorems}] Let us start the proof making the following remark.
\begin{remark}\label{remarkofproof}
In view of the subelliptic functional calculus, we can divide the proof of Theorem \ref{orders:theorems} in three steps. 
\begin{itemize}
    \item Step 1: to prove Theorem \ref{orders:theorems} for $m<0.$ This fact will be proved below.
    \item Step 2: to prove Theorem \ref{orders:theorems} for $m=0.$ This can be done as follows. We can write for $f\in S^{0}_{1,0}(\mathbb{R}_0^+),$ $f(\lambda)=f_{-\frac{m}{2}}(\lambda)(1+\lambda)^{\frac{m}{2}},$ with $f_{-\frac{m}{2}}(\lambda):=f(\lambda)(1+\lambda)^{-\frac{m}{2}},$  $m>0.$ In view of Step 1, $(1+t\mathcal{L})^{-\frac{m}{2}}\in {\Psi}^{-m,\mathcal{L}}_{1,0}(G\times \widehat{G}).$  Because of Corollary \ref{IesTC}, that $(1+t\mathcal{L})^{-\frac{m}{2}}\in {\Psi}^{-m,\mathcal{L}}_{1,0}(G\times \widehat{G})$ implies that $(1+t\mathcal{L})^{\frac{m}{2}}\in {\Psi}^{m,\mathcal{L}}_{1,0}(G\times \widehat{G}).$ The fact that $$ f_{-\frac{m}{2}}=f(\lambda)(1+\lambda)^{-\frac{m}{2}}\in S^{0}_{1,0}(\mathbb{R}_0^+)\times S^{-\frac{m}{2}}_{1,0}(\mathbb{R}_0^+)=S^{-\frac{m}{2}}_{1,0}(\mathbb{R}_0^+)$$   implies that $f_{-\frac{m}{2}}(t\mathcal{L})\in {\Psi}^{-m,\mathcal{L}}_{1,0}(G\times \widehat{G}).  $ So, the subelliptic calculus gives
    $$ f(t\mathcal{L})=f_{-\frac{m}{2}}(t\mathcal{L})(1+t\mathcal{L})^{\frac{m}{2}}\in  \Psi^{-m,\mathcal{L}}_{1,0}(G\times \widehat{G})\circ \Psi^{m,\mathcal{L}}_{1,0}(G\times \widehat{G}) = \Psi^{0,\mathcal{L}}_{1,0}(G\times \widehat{G}). $$
    This analysis proves Step 2.
    \item Case 3: $m>0.$ Taking $N\in \mathbb{R},$ such that $m<N$ for $f\in S^{\frac{m}{2}}_{1,0}(\mathbb{R}_0^+),$ let us define $f_{N}(\lambda):=f(\lambda)(1+\lambda)^{-\frac{N}{2}}\in S^{\frac{m-N}{2}}_{1,0}(\mathbb{R}_0^+).$ In view of Case 1, $f_N(t\mathcal{L})\in \Psi^{m-N,\mathcal{L}}_{1,0}(G\times \widehat{G}). $ Since $(1+t\mathcal{L})^{\frac{N}{2}}\in \Psi^{N,\mathcal{L}}_{1,0}(G\times \widehat{G}), $ we conclude from the subelliptic calculus that  
    $$  f(t\mathcal{L})=f_{N}(t\mathcal{L})(1+t\mathcal{L})^{\frac{m}{2}}\in  \Psi^{m-N,\mathcal{L}}_{1,0}(G\times \widehat{G})\circ \Psi^{N,\mathcal{L}}_{1,0}(G\times \widehat{G}) = \Psi^{m,\mathcal{L}}_{1,0}(G\times \widehat{G}), $$ proving Case 3.
\end{itemize}
\end{remark}
\begin{remark} We can assume that $f(\lambda)=0$ for  $\lambda\in [0,1],$ as a consequence of the following property:

\end{remark}
\begin{lemma}
\label{lem_prop_mult_2}
Let $m\in \mathbb{R}$ and let us fix $a\in \mathbb{N}_0$.
There exists $d=d(a,m)\in \mathbb{N}_0$ such that, for any $\Delta_q\in \textnormal{diff}^{\,a}(\widehat{G})$ there exists $C>0$
satisfying  
 for any function  $f\in C^d[0,\infty)$ with support in $[0,1]$ the following estimate 
$$
  \Vert  \widehat{\mathcal{M}}(\xi)^{a-m} \Delta_q f(t\widehat{\mathcal{L}}(\xi))  \Vert  _{\textnormal{op}} 
\lesssim   t^{\frac {m}  2} \max_{\ell =0,1,\ldots,d} 
  \Vert  \partial_\lambda^{\ell} f  \Vert  _{L^\infty},
$$ uniformly in $t\in (0,1).$
\end{lemma}
\begin{proof}[Proof of Lemma \ref{lem_prop_mult_2}]
Observe that in terms of the sub-Laplacian $\mathcal{L}=-\sum_{i=1}^{k}X_{i}^2$, 
 we have:
\begin{eqnarray*}
&&  \Vert   \widehat{\mathcal{M}}(\xi)^{2N}\Delta_q f(t\widehat{\mathcal{M}}(\xi))  \Vert  _{\textnormal{op}} 
=
  \Vert  (1+\widehat{\mathcal{L}}(\xi))^N \Delta_q f(t\widehat{\mathcal{M}}(\xi))  \Vert  _{\textnormal{op}} 
\\&& \qquad \leq
\int_G |(1+\mathcal{L})^N  q(x) \{(1+\mathcal{L})^N f(t\mathcal{L})\delta \}(x)|dx
\\&& \qquad\lesssim \sum_{1\leqslant i_1\leqslant i_2\leqslant\dots \leqslant i_{k}\leqslant k\,,|\beta|\leqslant 2N}
\int\limits_G |  X^{\beta_1}_{i_1,z}\dots X^{\beta_k}_{i_k,z} q(x) \{ f(t\mathcal{L})\delta \}(x)|dx
\\&& \quad
\ \lesssim  \sum_{|\beta|\leq 2N} t^{\frac{a -|\beta|}2}
\max_{\ell =0,1,\ldots,d} 
  \Vert  \partial_\lambda^{\ell} f  \Vert  _{L^\infty},
\end{eqnarray*}
having used Lemma \ref{HulanickiTheorem}
and Remark \ref{rem_lem_prop_mult_1} \eqref{item_rem_lem_prop_mult_1_Xq}.
Hence we have obtained
$$
\forall  t\in (0,1)\qquad
  \Vert   (1+\widehat{\mathcal{M}}(\xi))^{\frac{m_1}2}\Delta_q f(t\widehat{\mathcal{M}}(\xi))  \Vert  _{\textnormal{op}}  
\leq C t^{\frac {a-m_1}  2}
\max_{\ell =0,1,\ldots,d} 
  \Vert  \partial_\lambda^{\ell} f  \Vert  _{L^\infty},
$$
for any $m_1=2N\in 2\mathbb{N}$.
The properties of interpolation and duality of the subelliptic Sobolev spaces imply the result for any $m_1\in \mathbb{R}$.
We then choose $m_1=m-a$.
\end{proof}
Now, we continue the proof of Step 1 in Remark \ref{remarkofproof} in order to prove Theorem \ref{orders:theorems}.  Let us make use of the following notation, 
$$
  \Vert  \kappa  \Vert  _{\star} :=   \Vert  T_\kappa  \Vert  _{\mathscr{B}(L^2(G))}
=
\sup_{ [\xi]\in \widehat{G} }   \Vert  \mathscr{F}_G \kappa(\xi)  \Vert  _{\textnormal{op}},
$$
for an $L^{2}$-Fourier multiplier with kernel $\kappa.$
Let $\Delta_q\in \textnormal{diff}^{a}(\widehat{G})$, $m<0$,
and  $f\in C^d[0,\infty)$ supported in $[1,\infty)$.
The properties of the Sobolev spaces imply that 
  it suffices to show
\begin{equation}
\label{eq_pf_prop_mult2}
  \Vert   \mathcal{L}^{\frac b2} \{q\ f(t\mathcal{L})\delta \}  \Vert  _{\star} \lesssim  C 
t^{\frac{m}2} \,    \,   \sup_{\substack{\lambda\geq 1\\ \ell =0,\ldots,d} } \,    \,   
\lambda^{-\frac m2 + \ell} |\partial^\ell_\lambda f(\lambda)|,
\ \mbox{for}\ b=0,  -m+a.
\end{equation}
Let us fix a dyadic decomposition,
that is, a function $\chi_1\in \mathscr{D}(\mathbb{R})$ 
satisfying
$$
0\leq \chi_1\leq 1,\quad
\chi_1\big|_{[\frac 34,\frac 32]}=1,\quad
\textnormal{supp} \chi_1 \subset [\frac 12,2],
$$
and 
$$
\forall \lambda \geq 1 \quad 
\sum_{j=1}^\infty \chi_j(\lambda)=1,
\qquad\mbox{where}\quad \chi_j(\lambda)=\chi_1(2^{-j}\lambda)\ \mbox{for} \ j\in \mathbb{N}.
$$
We then define for all $j\in \mathbb{N},$ and all $\lambda\geq 0,$
$$
f_j(\lambda)
:=
\lambda^{-\frac m2}f(\lambda) \chi_j(\lambda)
\quad\mbox{and}\quad
g_j(\lambda)
:=
\lambda^{\frac{m}{2}} f_j(2^j \lambda).
$$
Note that, for any $j\in \mathbb{N}_0$,  
$g_j$ is smooth, supported in $[\frac 12,2]$, 
and satisfies the following estimates (see \cite[Page 3473]{Fischer2015})
\begin{equation}
\label{eq_pf_prop_mult_bdgj}
\forall d\in \mathbb{N}_0,\quad
  \Vert  \partial_\lambda^{d} g_j  \Vert  _{L^\infty} \lesssim_m
\sup_{\substack {\lambda\geq 1 \\ \ell\leq d}}
\lambda^{-\frac m2+\ell}  |\partial_\lambda^{\ell} f(\lambda)|.
\end{equation}

The sum $f(\lambda) =\sum_{j=1}^\infty 2^{j\frac m2} g_j(2^{-j}\lambda)$
is finite for any $\lambda\geq0$ and  locally finite on $[0,\infty)$.
Using \eqref{eq_pf_prop_mult_bdgj} 
and  that $\sum_j 2^{j\frac m2}<\infty$ for $m<0$, 
we obtain
$$
  \Vert  f(t\mathcal{L})  \Vert  _{\mathscr{B}(L^2(G))} 
\leq \sum_{j=1}^\infty 2^{j\frac m2}   \Vert   g_j(2^{-j}t\mathcal{L})  \Vert  _{\mathscr{B}(L^2(G))}
 \lesssim \sup_{\lambda\geq 1} \lambda^{-\frac m2} |f(\lambda)|
 <\infty.
$$
Hence we can write 
$$
f(t\mathcal{L}) 
=
\sum_{j=1}^\infty 2^{j\frac m2} g_j(2^{-j}t\mathcal{L}) 
\quad\mbox{in}\ \mathscr{B}(L^2(G)),
\quad\mbox{so}\quad
f(t\mathcal{L})\delta  
=
\sum_{j=1}^\infty 2^{j\frac m2} g_j(2^{-j}t\mathcal{L})\delta 
 ,
$$ in $\mathscr{D}'(G),$
with each function $g_j(2^{-j}\mathcal{L})\delta $ being smooth, 
see Remark \ref{rem_lem_prop_mult_1} \eqref{item_rem_lem_prop_mult_1_smooth}.
This justifies the following estimates:
$$
  \Vert   \partial_{X}^{(\beta)} q f(t\mathcal{L})\delta   \Vert  _{L^1(G)}
\leq
\sum_{j=1}^\infty 2^{j\frac m2}
  \Vert   \partial_{X}^{(\beta)} q g_j(2^{-j}t\mathcal{L})\delta   \Vert  _{L^1(G)}.
$$
By Lemma \ref{HulanickiTheorem}
and Remark \ref{rem_lem_prop_mult_1} \eqref{item_rem_lem_prop_mult_1_Xq}, 
we have
\begin{eqnarray}
  \Vert   \partial_{X}^{(\beta)} q g_j(2^{-j}t\mathcal{L})\delta   \Vert  _{L^1(G)}
&\lesssim &
(2^{-j}t )^{\frac {a-|\beta|}  2}
\max_{\ell=0,\ldots,d}
  \Vert  g_j^{(\ell)}  \Vert  _{L^\infty}
\nonumber
\\
&\lesssim &
(2^{-j}t )^{\frac {a-|\beta|}  2}
\sup_{\substack {\lambda\geq 1 \\ \ell\leq d}}
\lambda^{-\frac m2+\ell}  |\partial_\lambda^{\ell} f(\lambda)|,
\label{eq_pf_prop_mult_L1gj}
\end{eqnarray}
having used  \eqref{eq_pf_prop_mult_bdgj}.
This yields the estimate:
\begin{eqnarray*}
  \Vert   \partial_{X}^{(\beta)} q f(\mathcal{L})\delta   \Vert  _{L^1(G)}
&\lesssim  &
\sum_{j=1}^\infty 2^{j\frac m2}
(2^{-j}t )^{\frac {a-|\beta|}  2}
\sup_{\substack {\lambda\geq 1 \\ \ell\leq d}}
\lambda^{-\frac m2+\ell}  |\partial_\lambda^{\ell} f(\lambda)|
\\
&\lesssim  &
t ^{\frac {a-|\beta|}  2}
\sup_{\substack {\lambda\geq 1 \\ \ell\leq d}}
\lambda^{-\frac m2+\ell}  |\partial_\lambda^{\ell} f(\lambda)|,
\end{eqnarray*}
as long as $m-a+|\beta|<0$. 
This rough $L^1$-estimate 
 implies the estimate in \eqref{eq_pf_prop_mult2} in the case $b=0$ but 
is not enough to prove the case $b=-m+a$. 
We now present an argument making use of the almost orthogonality of the decomposition of $f(\mathcal{L})$.
More precisely, we will apply the Cotlar-Stein Lemma to the family of operators
$$
T_j:=2^{j \frac m2} T_{\mathcal{L}^{\frac b2} \{q g_j(2^{-j}t\mathcal{L})\delta \} }, 
$$
where $b=-m+a$, in the standard way.
Note that the properties of the subelliptic Sobolev spaces imply (see, e.g. \cite{CardonaRuzhansky2019I})
$$
  \Vert  \mathcal{L}^{\frac b2} \{q g_j(2^{-j}t\mathcal{L})\delta \}  \Vert  _{\star}
\leq 
\left(  \Vert  \mathcal{L}^{\lceil \frac b2\rceil} \{q g_j(2^{-j}t\mathcal{L})\delta \}   \Vert  _{\star}\right)^{\theta}
\left(  \Vert  \mathcal{L}^{\lfloor \frac b2\rfloor} \{q g_j(2^{-j}t\mathcal{L})\delta \}   \Vert  _{\star}
\right)^{1-\theta}
$$
with $\theta = \lfloor \frac b2\rfloor - \frac b2$
and we can bound the $  \Vert \cdot  \Vert  _{\star}$-norm 
with the $L^1$-norm given in  \eqref{eq_pf_prop_mult_L1gj}, summing up over $\beta$'s 
with $|\beta|=\lceil \frac b2\rceil$ or 
 $|\beta|=\lfloor \frac b2\rfloor$\footnote{In this proof, $\lceil b \rceil$ and 
 $\lfloor  b\rfloor$ denote the upper and lower integer parts of a real number $b$ respectively.}.
We obtain:
\begin{equation}
\label{eq_pf_prop_mult_bdTj}
  \Vert  \mathcal{L}^{\frac b2} \{q g_j(2^{-j}t\mathcal{L})\delta \}  \Vert  _{\star}
\lesssim_{q,b,m}
(2^{-j}t )^{\frac {a-b}  2}
\sup_{\substack {\lambda\geq 1 \\ \ell\leq d}}
\lambda^{-\frac m2+\ell}  |\partial_\lambda^{\ell} f(\lambda)|,
\end{equation}
and, as $q-b=m$, the operators $T_j$'s are uniformly bounded.
We also need to find a bound for the operator norm of 
$T_jT^*_k$ whose convolution kernel is given by
$$
2^{(j+k)\frac m2}
 \{\mathcal{L}^{\frac b2} q g_j(2^{-j}t\mathcal{L})\delta \} * \{\mathcal{L}^{\frac b2} q^* \bar g_k(2^{-k}t\mathcal{L})\delta \}.
 $$
 By observing that 
 \begin{align*}
     &\mathscr{F}_{G}[  \{\mathcal{L}^{\frac b2} q g_j(2^{-j}t\mathcal{L})\delta \} * \{\mathcal{L}^{\frac b2} q^* \bar g_k(2^{-k}t\mathcal{L})\delta \}]\\
     &= \mathscr{F}_{G}[\{\mathcal{L}^{\frac b2} q^* \bar g_k(2^{-k}t\mathcal{L})\delta \}]\mathscr{F}_{G}[ \{\mathcal{L}^{\frac b2} q g_j(2^{-j}t\mathcal{L})\delta \}\\
     &=\mathscr{F}_{G}[q^* \bar g_k(2^{-k}t\mathcal{L})\delta]\widehat{\mathcal{L}}^{\frac{b}{2}}\mathscr{F}_{G}[ q g_j(2^{-j}t\mathcal{L})\delta ]\widehat{\mathcal{L}}^{\frac{b}{2}}\\
      &=\mathscr{F}_{G}[q^* \bar g_k(2^{-k}t\mathcal{L})\delta]\widehat{\mathcal{L}}^{\frac{b+c}{2}}\widehat{\mathcal{L}}^{-\frac{c}{2}}\mathscr{F}_{G}[ q g_j(2^{-j}t\mathcal{L})\delta ]\widehat{\mathcal{L}}^{\frac{b}{2}},
 \end{align*} wehave the inequalities
 \begin{align*}
     \Vert \mathscr{F}_{G}[ & \{\mathcal{L}^{\frac b2} q g_j(2^{-j}t\mathcal{L})\delta \} * \{\mathcal{L}^{\frac b2} q^* \bar g_k(2^{-k}t\mathcal{L})\delta \}]\Vert_{\textnormal{op}}\\
     &\leq  \Vert \mathscr{F}_{G}[q^* \bar g_k(2^{-k}t\mathcal{L})\delta]\widehat{\mathcal{L}}^{\frac{b+c}{2}} \Vert_{\textnormal{op}}\Vert \widehat{\mathcal{L}}^{-\frac{c}{2}}\mathscr{F}_{G}[ q g_j(2^{-j}t\mathcal{L})\delta ]\widehat{\mathcal{L}}^{\frac{b}{2}} \Vert_{\textnormal{op}}\\
     &\lesssim  \Vert \mathscr{F}_{G}[q^* \bar g_k(2^{-k}t\mathcal{L})\delta]\widehat{\mathcal{L}}^{\frac{b+c}{2}} \Vert_{\textnormal{op}}\Vert \mathscr{F}_{G}[ q g_j(2^{-j}t\mathcal{L})\delta ]\widehat{\mathcal{L}}^{\frac{b}{2}}\widehat{\mathcal{L}}^{-\frac{c}{2}} \Vert_{\textnormal{op}}\\
      &=  \Vert \mathscr{F}_{G}[q^* \bar g_k(2^{-k}t\mathcal{L})\delta]\widehat{\mathcal{L}}^{\frac{b+c}{2}} \Vert_{\textnormal{op}}\Vert \mathscr{F}_{G}[ q g_j(2^{-j}t\mathcal{L})\delta ]\widehat{\mathcal{L}}^{\frac{b-c}{2}} \Vert_{\textnormal{op}},
 \end{align*}
 and as a consequence of this analysis we obtain 
 
 $$  \Vert  \mathcal{L}^{\frac b2} \{q g_j(2^{-j}t\mathcal{L})\delta \}  \Vert  _{\star} \lesssim \Vert \{\mathcal{L}^{\frac {b+c}2}  \ q g_j(2^{-j}t\mathcal{L})\delta \}    \Vert  _{\star} \Vert \{\mathcal{L}^{\frac {b-c}2}  q^* \bar g_k(2^{-k}t\mathcal{L})\delta  \Vert  _{\star}, $$
for any real number $c$.
The estimate for $  \Vert  \mathcal{L}^{\frac b2} \{q g_j(2^{-j}t\mathcal{L})\delta \}  \Vert  _{\star}$
in  \eqref{eq_pf_prop_mult_bdTj}
holds in fact for any $b\geq 0$ and by duality for any $b\in \mathbb{R}$.
Hence we can use it at $b\pm c$ to obtain
\begin{eqnarray*}
  \Vert  T_jT^*_k  \Vert  _{\mathscr{B}(L^2(G))} 
&\lesssim &
2^{(j+k)\frac m2}
  \Vert  \mathcal{L}^{\frac {b + c}2} \{q g_j(2^{-j}t\mathcal{L})\delta \}  \Vert  _{\star}
  \Vert  \mathcal{L}^{\frac {b -c }2} \{q g_k(2^{-k}t\mathcal{L})\delta \}  \Vert  _{\star}
\\
&\lesssim &
2^{(j+k)\frac m2}
t^{a-b}
2^{-j\frac {a- (b +c)}  2}
2^{-k\frac {a-(b-c)}  2}
\left(\sup_{\substack {\lambda\geq 1 \\ \ell\leq d}}
\lambda^{-\frac m2+\ell}  |\partial_\lambda^{\ell} f(\lambda)|\right)^2
\\
&\lesssim &
2^{(j-k)\frac c 2}t^{a-b}
\left(\sup_{\substack {\lambda\geq 1 \\ \ell\leq d}}
\lambda^{-\frac m2+\ell}  |\partial_\lambda^{\ell} f(\lambda)|\right)^2,
\end{eqnarray*}
having used $b=-m+a$.
We choose $c$ to be the sign of $j-k$. 
So we have proved that the hypotheses in (Cotlar-Stein)  Lemma \ref{CLemma}
 are satisfied. In consequence, we have proved   
\eqref{eq_pf_prop_mult2} for $b=-m+a$.  The proof  of Theorem  \ref{orders:theorems} is complete.
\end{proof}

\section{Evolution problems for subelliptic   operators}\label{GST}
The subelliptic pseudo-differential calculus developed for every subelliptic Laplacian will be applied in this section in a relevant problem of PDE, the global solvability of parabolic and hyperbolic/parabolic problems, in this case associated to subelliptic pseudo-differential operators.
More precisely, we will study the existence and uniqueness  for the following Cauchy problem 
\begin{equation}\label{PVI}(\textnormal{PVI}): \begin{cases}\frac{\partial v}{\partial t}=K(t,x,D)v+f ,& \text{ }v\in \mathscr{D}'((0,T)\times G),
\\v(0)=u_0, & \text{ } \end{cases}
\end{equation}where the initial data $u_0\in L^2(G),$ $K(t):=K(t,x,D)\in C([0,T], S^{m,\mathcal{L}}_{\rho,\delta}(G\times \widehat{G})),$ $f\in  L^2([0,T]\times G) \simeq L^2([0,T],L^2(G)) ,$ $m>0,$ and  a suitable positivity condition is imposed on $K.$ Indeed, we will assume that
\[ 
    \textnormal{Re}(K(t)):=\frac{1}{2}(K(t)+K(t)^*),\,\,0\leqslant t\leqslant T,\footnote{ This means that $A=K(t)$ is  strongly $\mathcal{L}$-elliptic.}
\]is $\mathcal{L}$-elliptic.   We also assume that $ \textnormal{Re}(K(t)),$ commutes with the symbol $\widehat{\mathcal{L}}$ of the sub-Laplacian $\mathcal{L}.$
Under such assumptions we will prove the existence and uniqueness of a solution $v\in C^1([0,T],L^2(G))\cap C([0,T],H^{m,L}(G)).$ We will start with the following energy estimate.
\begin{theorem}\label{energyestimate}
Let $K(t)=K(t,x,D)\in C([0,T], S^{m,\mathcal{L}}_{\rho,\delta}(G\times \widehat{G})),$   be a subelliptic pseudo-differential operator of order $m>0,$  $0\leqslant \delta<\rho\leqslant  1,$  and let us assume that $\textnormal{Re}(K(t))$ is $\mathcal{L}$-elliptic, for every $t\in[0,T]$ with $T>0.$ If  $v\in C^1([0,T], L^2(G) )  \cap C([0,T],H^{m,\mathcal{L}}(G)) ,$ then there exist $C,C'>0,$ such that
\begin{equation}
\Vert v(t) \Vert_{L^2(G)}\leqslant   \left(C\Vert v(0) \Vert^2_{L^2(G)}+C'\int\limits_{0}^T \Vert (\partial_{t}-K(\tau))v(\tau) \Vert^2_{L^2(G)}d\tau \right),  
\end{equation}holds for every $0\leqslant t\leqslant T.$ Moreover, we also have the estimate
\begin{equation}\label{Q*}
\Vert v(t) \Vert_{L^2(G)}\leqslant   \left(C\Vert v(T) \Vert^2_{L^2(G)}+C'\int\limits_{0}^T \Vert (\partial_{t}-K(\tau)^{*})v(\tau) \Vert^2_{L^2(G)}d\tau \right).  
\end{equation}
\end{theorem}
\begin{proof} 
Let $v\in C^1([0,T], L^2(G) )  \cap C([0,T],H^{m,\mathcal{L}}(G)) .$ Let us start by observing  that   $v\in  C([0,T],H^{\frac{m}{2},\mathcal{L}}(G))$ because of the embedding $H^{m,\mathcal{L}}\hookrightarrow H^{\frac{m}{2},\mathcal{L}}.$ This fact will be useful later because we will use the G\r{a}rding inequality applied to the operator $\textnormal{Re}(K(t).$ So, 
$v\in \textnormal{Dom}(\partial_{\tau}-K(\tau))$ for every $0\leqslant \tau\leqslant T.$ In view of the embedding $H^{m,\mathcal{L}}\hookrightarrow L^2(G),$ we also have that  $v\in C([0,T], L^2(G) ).$  Let us define $f(\tau):=Q(\tau)v(\tau),$ $Q(\tau):=(\partial_{\tau}-K(\tau)),$ for every $0\leqslant \tau\leqslant T.$ Observe that

\begin{align*}
   \frac{d}{dt}\Vert v(t) \Vert^2_{L^2(G)}&= \frac{d}{dt}\left(v(t),v(t)\right)_{L^2(G)}\\&=\left(\frac{d v(t)}{dt},v(t)\right)_{L^2(G)}+\left(v(t),\frac{d v(t)}{dt}\right)_{L^2(G)}\\
   &=\left(K(t)v(t)+f(t),v(t)\right)_{L^2(G)}+\left(v(t),K(t)v(t)+f(t)\right)_{L^2(G)}\\
    &=\left((K(t)+K(t)^{*})v(t),v(t)\right)_{L^2(G)}+2\textnormal{Re}(f(t), v(t))_{L^2(G)}\\
     &=(2\textnormal{Re}K(t)v(t),v(t))_{L^2(G)}+2\textnormal{Re}(f(t), v(t))_{L^2(G)}.
\end{align*}Now, from the subelliptic G\r{a}rding inequality, 
\begin{align}
    \textnormal{Re}(-K(t)v(t),v(t)) \geqslant C_1\Vert v(t)\Vert_{H^{\frac{m}{2},\mathcal{L}}(G)}-C_2\Vert v(t)\Vert_{L^2(G)}^2,
\end{align}and  the parallelogram law, we have
\begin{align*}
 2\textnormal{Re}(f(t), v(t))_{L^2(G)}&\leqslant 2\textnormal{Re}(f(t), v(t))_{L^2(G)}+\Vert f(t)\Vert_{L^2(G)}^2+\Vert v(t)\Vert_{L^2(G)}^2 \\
 &= \Vert f(t)+v(t)\Vert^2\leqslant \Vert f(t)+v(t)\Vert^2+\Vert f(t)-v(t)\Vert^2 \\
&= 2\Vert f(t)\Vert^2_{L^2(G)}+2\Vert v(t)\Vert^2_{L^2(G)},
\end{align*}
and consequently
\begin{align*}
   & \frac{d}{dt}\Vert v(t) \Vert^2_{L^2(G)}\\
   &\leqslant 2\left(C_2\Vert v(t)\Vert_{L^2(G)}^2-C_1\Vert v(t)\Vert_{H^{\frac{m}{2},\mathcal{L}}(G)}\right)+2\Vert f(t)\Vert^2_{L^2(G)}+2\Vert v(t)\Vert^2_{L^2(G)}.
\end{align*}  So, we have proved that
\begin{align*}
   \frac{d}{dt}\Vert v(t) \Vert^2_{L^2(G)}\lesssim  \Vert f(t)\Vert^2_{L^2(G)}+\Vert v(t)\Vert^2_{L^2(G)}.
\end{align*}By using Gronwall's Lemma we obtain the energy estimate
\begin{equation}
\Vert v(t) \Vert^2_{L^2(G)}\leqslant  \left(C\Vert v(0) \Vert_{L^2(G)}^2+C'\int\limits_{0}^T \Vert f(\tau) \Vert_{L^2(G)}^2d\tau \right),   
\end{equation}for every $0\leqslant t\leqslant T,$ and $T>0.$ To finish the proof, we can change the analysis above with $v(T-\cdot)$ instead of $v(\cdot),$ $f(T-\cdot)$ instead of $f(\cdot)$ and $Q^{*}=-\partial_{t}-K(t)^{*},$ (or equivalently $Q=\partial_{t}-K(t)$ ) instead of $Q^{*}=-\partial_{t}+K(t)^{*}$ (or equivalently $Q=\partial_{t}-K(t)$) using that $\textnormal{Re}(K(T-t)^*)=\textnormal{Re}(K(T-t))$ to deduce that
\begin{align*}
&\Vert v(T-t) \Vert^2_{L^2(G)}\\
&\leqslant \left(C\Vert v(T) \Vert^2_{L^2(G)}+C'\int\limits_{0}^{T} \Vert (-\partial_{t}+K(T-t)^{*})v(T-\tau) \Vert^2_{L^2(G)}d\tau \right)\\
&=   \left(C\Vert v(T) \Vert^2_{L^2(G)}+C'\int\limits_{0}^T \Vert (-\partial_{t}-K(t)^{*})v(s) \Vert^2_{L^2(G)}ds \right).
\end{align*}So, we conclude the proof.
\end{proof}

\begin{theorem}
Let $K(t)=K(t,x,D)\in C([0,T], S^{m,\mathcal{L}}_{\rho,\delta}(G\times \widehat{G})),$ $0\leqslant \delta<\rho\leqslant  1,$    be a subelliptic pseudo-differential operator of order $m>0,$ and let us assume that $\textnormal{Re}(K(t))$ is $\mathcal{L}$-elliptic, for every $t\in[0,T]$ with $T>0.$ Let   $f\in  L^2([0,T],L^2(G))$, and let $u_0\in L^2([0,T],L^2(G)).$ Then there exists a unique $v\in C^1([0,T], L^2(G) )  \cap C([0,T],H^{m,\mathcal{L}}(G)) $ solving \eqref{PVI}. Moreover, $v$ satisfies the energy estimate
\begin{equation}
\Vert v(t) \Vert^2_{L^2(G)}\leqslant  \left(C\Vert u_0 \Vert^2_{L^2(G)}+C'\Vert f \Vert^2_{L^2([0,T],L^2(G))} \right),
\end{equation}for every $0\leqslant t\leqslant T.$
\end{theorem}
\begin{proof}
    Let us denote by $Q=\partial_{t}-K(t),$ and let us introduce the spaces
    \begin{align*}
        E:=\{\phi\in  C^1([0,T], L^2(G) )  \cap C([0,T],H^{m,\mathcal{L}}(G)) :\phi(T)=0\},
    \end{align*}
    and $ Q^*E:=\{Q^*\phi:\phi\in E\}.$ Let us define the linear form $\beta\in (Q^*E)'$ by
    \[ 
    \beta(Q^*\phi):=\int\limits_{0}^T(f(\tau),\phi(\tau))d\tau+\frac{1}{i}(u_0,\phi(0)).
    \]
    From \eqref{Q*} we deduce  for every $\phi\in E$ that,
    \begin{align*}
    \Vert \phi(t) \Vert^2_{L^2(G)}&\leqslant   \left(C\Vert \phi(T) \Vert^2_{L^2(G)}+C'\int\limits_{0}^T \Vert (\partial_{t}-K(\tau)^{*})\phi(\tau) \Vert^2_{L^2(G)}d\tau \right).
    \end{align*} So, we have
    \begin{align*}
    |\beta(Q^*\phi)|&\leqslant \int\limits_{0}^T\Vert f(\tau)\Vert_{L^2(G)}\Vert \phi(\tau)\Vert_{L^2(G)} d\tau+\Vert u_0\Vert_{{L^2(G)}}\Vert \phi(0)\Vert_{L^2(G)}\\
    &\leqslant \Vert f\Vert_{L^2([0,T],L^2(G))}\Vert \phi\Vert_{L^2([0,T],L^2(G))}+\Vert u_0\Vert_{L^2(G)}\Vert \phi\Vert_{L^2(G)}\\
    &\lesssim_{T} (\Vert f\Vert_{L^2([0,T],L^2(G)}+\Vert u_0\Vert_{L^2(G)})\Vert Q^*\phi(\tau) \Vert_{L^2([0,T],L^2(G))},
    \end{align*}which shows that $\beta$ is a bounded functional on $\mathscr{T}:=Q^*E\cap L^2([0,T],L^2(G)),$ with $\mathscr{T}$ endowed with the topology induced by the  norm of $L^2([0,T],L^2(G)).$ By using the Hahn-Banach extension theorem, we can extends $\beta$ to a bounded functional $\tilde{\beta}$ on $L^2([0,T],L^2(G)),$ and by using the Riesz representation theorem, there exists $v\in (L^2([0,T],L^2(G))'=L^2([0,T],L^2(G)) ,$ such that
    \[ 
      \tilde{\beta}(\psi)=(v,\psi),\quad \psi\in   L^2([0,T],L^2(G)).
    \]In particular, for $\psi=Q^*\phi\in \mathscr{T},$ we have
    \[ 
      \tilde{\beta}(Q^* \phi)={\beta}(Q^*\phi)=(v,Q^*\phi),
    \]  Because, we can identify $C^{\infty}_{0}((0,T),\mathscr{D}(G))$ as a subspace of $E$ 
    \[ 
       C^{\infty}_{0}((0,T),C^\infty(G))\subset E=\{\phi\in  C^1([0,T], L^2(G) )  \cap C([0,T],H^{m,\mathcal{L}}(G)) :\phi(T)=0\},
    \] we have the identity
    \begin{align*}
      (f,\phi)=  \int\limits_{0}^T(f(\tau),\phi(\tau))d\tau= \int\limits_{0}^T(f(\tau),\phi(\tau))d\tau+\frac{1}{i}(u_0,\phi(0))=(v,Q^*\phi),
    \end{align*}for every $\phi \in C^{\infty}_{0}((0,T),C^\infty(G)). $ So, this implies that $v\in \textnormal{Dom}(Q^{**}).$ Because $Q^{**}=Q,$ we have that
    \begin{align*}
        (v,Q^*\phi)=(Qv,\phi)=(f,\phi),\,\,\forall\phi \in C^{\infty}_{0}((0,T),C^\infty(G)),
    \end{align*}which implies that $Qv=f.$ A routine argument of integration by parts shows that $v(0)=u_0.$ Now, in order to show the uniqueness of $v,$ let us assume that $u\in  C^1([0,T], L^2(G) )  \cap C([0,T],H^{m,\mathcal{L}}(G))$ is a solution of the problem
    \[  \begin{cases}\frac{\partial u}{\partial t}=K(t,x,D)u+f ,& \text{ }u\in \mathscr{D}'((0,T)\times G),
\\u(0)=u_0 .& \text{ } \end{cases}
\] Then $\omega:=v-u\in C^1([0,T], L^2(G) )  \cap C([0,T],H^{m,\mathcal{L}}(G))$ solves the problem
\[  \begin{cases}\frac{\partial \omega}{\partial t}=K(t,x,D)\omega,& \text{ }\omega\in \mathscr{D}'((0,T)\times G),
\\\omega(0)=0 ,& \text{ } \end{cases}
\]and from Theorem \ref{energyestimate}, $\Vert \omega(t)\Vert_{L^2(G)}=0,$ for all $0\leqslant t\leqslant T,$ and consequently, from the continuity in $t$ of the functions we have that $v(t,x)=u(t,x)$ for all $t\in [0,T]$ and a.e. $x\in G.$
\end{proof}

\section{Global Fourier Integral operators on compact Lie groups}\label{FIOCLG}
In this section we will study the $L^2$-boundedness of Fourier integral operators on compact Lie groups.  The motivation to include a study of the $L^2$-boundedness of Fourier integral operators is to make a generalisation of Theorem 10.5.5 in \cite{Ruz}, where it has been proved that the following condition
\[ 
  \sup_{(x,[\xi])\in G\times \widehat{G}}  \Vert X^{\alpha}_{x}\sigma(x,\xi)\Vert_{\textnormal{op}}<\infty,\quad |\alpha|\leqslant \left[\frac{n}{2}\right]+1,\,\,n=\dim(G),
\] implies the $L^2(G)$-boundedness of a densely defined extension of $A\equiv \textnormal{Op}(\sigma)$  on $C^\infty(G).$  We will use aditional to the notion of a global matrix-valued symbol, that of matrix-valued phase function. To do so, we will use   Cotlar Lemma, by exploiting in our case the almost-orthogonality of the decomposition of a Fourier integral operator in several pieces  induced by  every representation space.
\subsection{Matrix-valued phase functions on compact Lie groups}
First, we  introduce a global definition of Fourier integral operators on every compact Lie group $G$.
\begin{definition}[FIO]\label{DefiFIO}
A continuous linear operator $A:C^\infty(G)\rightarrow\mathscr{D}'(G)$ with Schwartz kernel $K_A\in C^{\infty}(G)\otimes \mathscr{D}'(G),$ is a global Fourier integral operator, if
\begin{itemize}
    \item there exists a  \textit{symbol} $ \sigma\equiv  \{\sigma(x,\xi)\in \mathbb{C}^{d_\xi\times d_\xi}\}_{(x,[\xi])\in G\times \widehat{G}}$ such that
    \[  \forall [\xi]\in \widehat{G},\quad\sigma(\cdot,[\xi])\equiv  \sigma(\cdot,\xi)\in C^{\infty}(G,\mathbb{C}^{d_\xi\times d_\xi}),
   \]
   \item and a \textit{phase function}  $\phi\equiv\{\phi(x,\xi)\in \mathbb{C}^{d_\xi\times d_\xi}\}_{(x,[\xi])\in G\times \widehat{G}}$ depending (possibly) on $x\in G,$  and satisfying  
   \[ 
     \forall [\xi]\in \widehat{G},\quad\phi(\cdot, [\xi])\equiv     \phi(\cdot,\xi)\in C^{\infty}(G,\mathbb{C}^{d_\xi\times d_\xi}),\,\,\,\phi(x,\xi)=\phi(x,\xi)^{*},
   \]
\end{itemize} such that, the Schwartz kernel of $A$ is defined by the distribution,
\[ 
    K_{A}(x,y)=\sum\limits_{[\xi]\in\widehat{G}}d_\xi\textnormal{\textbf{Tr}}(\xi(y)^*e^{i\phi(x,\xi)}\sigma(x,\xi)),
\] and the operator $A$ is defined via
\begin{equation}\label{defiFIO}
    Af(x)=\sum\limits_{[\xi]\in \widehat{G}}d_\xi\textnormal{\textbf{Tr}}(e^{i\phi(x,\xi)}\sigma(x,\xi)\widehat{f}(\xi)),
\end{equation} for every $f\in C^\infty(G).$
\end{definition}
\begin{remark}\label{r11}
To motivate the definition of the Schwartz Kernel of a Fourier integral operator, observe that under suitable conditions on $\phi$ and $\sigma,$ for $f\in C^\infty(G),$ we can use Fubini theorem to write 
\begin{align*}
    Af(x)&:=\int\limits_{G}K_{A}(x,y)f(y)dy\equiv \int\limits_{G}\sum\limits_{[\xi]\in \widehat{G}}d_\xi\textnormal{\textbf{Tr}}(\xi(y)^*e^{i\phi(x,\xi)}\sigma(x,\xi)) f(y)dy\\
    &=\sum\limits_{[\xi]\in \widehat{G}}d_\xi\textnormal{\textbf{Tr}}\left(\int\limits_{G}f(y)\xi(y)^*dy e^{i\phi(x,\xi)}\sigma(x,\xi)\right) \\
    &=\sum\limits_{[\xi]\in \widehat{G}}d_\xi\textnormal{\textbf{Tr}}(\widehat{f}(\xi) e^{i\phi(x,\xi)}\sigma(x,\xi))\\
    &=\sum\limits_{[\xi]\in \widehat{G}}d_\xi\textnormal{\textbf{Tr}}( e^{i\phi(x,\xi)}\sigma(x,\xi)\widehat{f}(\xi)).
\end{align*}
We will use the notation 
$
    A=\textnormal{FIO}(\sigma,\phi),
$
to denote the Fourier integral operator associated with the symbol $\sigma$ and with the phase function $\phi.$
\end{remark}
To motivate the definition of global Fourier integral operators we will review how the definition given in the Definition \ref{DefiFIO} extends  that of pseudo-differential operators and how it does appear to express solutions of some differential problems. We will make some remarks explaining both situations in detail.

\begin{example}[Pseudo-differential operators]\label{r1}
Let us assume that $A:C^{\infty}(G)\rightarrow \mathscr{D}'(G)$ is a continuous linear operator with symbol $\sigma,$
\[ 
    Af(x)=\sum\limits_{[\xi]\in \widehat{G}}d_\xi\textnormal{\textbf{Tr}}(\xi(x)\sigma(x,\xi)\widehat{f}(\xi)),\,\,\,\,f\in C^\infty(G).
\]Since the matrix $\xi(x)\in \mathbb{C}^{d_\xi\times d_\xi}$ is  unitary for every $x\in G,$ there exists a self-adjoint operator $\psi(x,\xi)\in  \mathbb{C}^{d_\xi\times d_\xi}$ such that $\xi(x)=e^{i\psi(x,\xi)}.$ Consequently, we have 
 \begin{equation}\label{FIOPSDO}
    Af(x)=\sum\limits_{[\xi]\in\widehat{G}}d_\xi\textnormal{\textbf{Tr}}(e^{i\psi(x,\xi)}\sigma(x,\xi)\widehat{f}(\xi)),\,\,\,\,f\in C^\infty(G).
\end{equation}Moreover, since $\textnormal{Spectrum}(\xi(x))\subset \{z\in \mathbb{C}:|z|=1\},$ in some basis $B_{\xi}$ of the representation space $\mathbb{C}^{d_\xi},$ $\xi(x)=\textnormal{diag}[e^{i\phi_{jj}(x,\xi)}]_{1\leqslant j\leqslant d_\xi}$ is a diagonal matrix with complex entries, and in this case on the same basis we have
\[ 
    \psi(x,\xi)\equiv \textnormal{diag}[\phi_{jj}(x,\xi)]_{1\leqslant j\leqslant d_\xi},\quad (x,\xi)\in G\times \widehat{G}.
\]To give a more precise description for the functions $\phi_{jj}(x,\xi),$ $1\leqslant j\leqslant d_\xi,$ let us consider $\mathbb{X}=\{X_{1},\cdots,X_n\}$ a basis for the Lie algebra $\mathfrak{g}$. Every $X_k$ is a left-invariant operator. Observe that
\[ 
    X_{k}\xi(x)=\textnormal{diag}[ie^{i\phi_{jj}(x,\xi)}(X_k\phi_{jj})(x,\xi)]_{j=1}^{d_\xi},
\] and 
\[ 
    X_{k}^{2}\xi(x)=\textnormal{diag}[-e^{i\phi_{jj}(x,\xi)}(X_{k}\phi_{jj}(x,\xi))^2+ie^{i\phi_{jj}(x,\xi)}(X_{k}^2\phi_{jj})(x,\xi)]_{j=1}^{d_\xi}.
\]Taking into account that $\xi(e_G)=I_{d_\xi\times d_\xi}=(\delta_{ij})_{1\leqslant i,j\leqslant d_\xi},$ it follows that $e^{i\phi_{jj}(e_G,\xi)}=1,$ for all $1\leqslant j\leqslant d_\xi.$  Consequently,
\begin{align*}
    \widehat{X}_{k}^2(\xi)&= X_{k}^{2}\xi(e_{G})\\
    &=\textnormal{diag}[-(X_{k}\phi_{jj}(e_G,\xi))^2+i(X_{k}^2\phi_{jj})(e_G,\xi)]_{j=1}^{d_\xi}\\
    &=\widehat{X}_{k}(\xi)\widehat{X}_{k}(\xi)={X}_{k}\xi(e_{G}){X}_{k}\xi(e_{G})\\
    &=\textnormal{diag}[-(X_{k}\phi_{jj}(e_G,\xi))^2]_{j=1}^{d_\xi},
\end{align*}which implies that
\begin{equation}
   ( X_{k}^2\phi_{jj})(e_G,\xi)=0,\,\,1\leqslant j\leqslant d_\xi.
\end{equation} A similar analysis implies that
\begin{equation}
   ( X_{k}^\ell\phi_{jj})(e_G,\xi)=0,\,\,1\leqslant j\leqslant d_\xi,\,\,\ell\geqslant    2.
\end{equation}
So, the Taylor expansion at $x=e_G,$ gives that for some smooth function $q_{j}$ vanishing with order $1$ at $x=e_{G},$ we have
\begin{equation}
  \phi_{jj}(x,\xi)=\phi_{jj}(e_{G},\xi)+\sum_{k=1}^{n}(X_{k}\phi_{jj})(e_{G},\xi) q_{k}(x).
\end{equation}Also, the condition $e^{i\phi_{jj}(e_G,\xi)}=1,$ implies that for some $n_{\xi,j}\in \mathbb{Z},$ $\phi_{jj}(e_G,\xi)=2\pi n_{\xi,j}.$ This analysis implies the following representation for the pseudo-differential operator $A,$
\begin{align*}
    Af(x)&=\sum\limits_{[\xi]\in \widehat{G}}d_\xi\textnormal{\textbf{Tr}}(\textnormal{diag}[ e^{i( 2\pi n_{\xi,j} +\sum_{k=1}^{n}(X_{k}\phi_{jj})(e_{G},\xi) q_{k}(x)})]_{j=1}^{d_\xi}\sigma(x,\xi)\widehat{f}(\xi))\\
    &=\sum\limits_{[\xi]\in \widehat{G}}d_\xi\textnormal{\textbf{Tr}}(\textnormal{diag}[ e^{i(\sum_{k=1}^{n}(X_{k}\phi_{jj})(e_{G},\xi) q_{k}(x)})]_{j=1}^{d_\xi}\sigma(x,\xi)\widehat{f}(\xi)).
\end{align*}
\end{example}

\begin{example}[Cauchy problem for the wave equation]\label{r2}
Let us consider the differential problem,
\begin{equation}\label{wave}
    \partial_{t}^2u(x,t)+\mathcal{A}u(x,t)=0,\,\,(x,t)\in G\times [T_{0},\infty),\,T_0\geqslant    0,
\end{equation}with initial  conditions $u(x,0)=f_0\in C^\infty(G)$ and $u_t(x,t)=f_1\in C^\infty(G),$ where $\mathcal{A}$ is a positive left-invariant operator. In this case, if $\lambda=0$ is an isolated point of the spectrum of $\mathcal{A},$ the solution $u(x,t)$ of \eqref{wave} can be written according to the representation,
\[ 
    u(x,t)=A_{t,0}f_{+}(x)+A_{t,1}f_{-}(x),\,\,(x,t)\in G\times [T_{0},\infty),
\]
where 
\[ 
f_{+}(x)=\frac{1}{2}(f_0-i\mathcal{A}^{-\frac{1}{2}}f_1),\,\,\,f_{-}(x)=\frac{1}{2}(f_0+i\mathcal{A}^{-\frac{1}{2}}f_1),
\] and the operators
\[ 
    A_{t,j}f_{\pm}(x)=\int\limits_{ {G}}\sum\limits_{ [\xi]\in \widehat{G}}d_\xi\textnormal{\textbf{Tr}}(\xi(y^{-1}x)e^{\pm it\sqrt{\widehat{\mathcal{A}}(\xi)}}\widehat{\mathcal{A}}(\xi)^{-\frac{j}{2}}) f_{\pm}(y) dy,\,\,j=0,1,
\] are global Fourier integral operators.
\end{example}

\subsection{$L^p$-boundedness of Fourier integral operators}
Now we will study the $L^p$-boundedness of Fourier integral operators. Motivated by the solution for the Cauchy problem for the wave equation, let us consider the case of Fourier integral operators where the phase function admits a factorisation of the form
\[ 
    e^{i\phi(x,\xi)}=\xi(x)e^{i\Phi(\xi)}\textnormal{   where   }\Phi(\xi)=\Phi(\xi)^*,\quad \forall [\xi]\in \widehat{G}.
\]Observe that the matrix $\xi(x)e^{i\Phi(\xi)}$ is  unitary for every $[\xi],$ and the existence of such $\phi$  follows.  More explicitly, for a Fourier integral operator of the form
\begin{equation}\label{wavetype}
 Af(x)=\sum\limits_{[\xi]\in \widehat{G}}d_\xi\textnormal{\textbf{Tr}}(\xi(x)e^{i\Phi(\xi)}\sigma(x,\xi)\widehat{f}(\xi)),\quad f\in C^\infty(G),
\end{equation}     we have the following $L^2$-boundedness theorem. We will denote by $\varkappa_Q$  the smallest even integer larger than $Q/2. $ 
\begin{theorem}\label{L2FIO2}
Let $G$ be a compact Lie group and let us denote by $Q$ the Hausdorff
dimension of $G$ associated to the control distance associated to the sub-Laplacian $\mathcal{L}=\mathcal{L}_X,$ where  $X= \{X_{1},\cdots,X_k\} $ is a system of vector fields satisfying the H\"ormander condition. Let  us consider the Fourier integral operator $A\equiv \textnormal{FIO}(\sigma, \phi):C^\infty(G)\rightarrow\mathscr{D}'(G)$ with symbol  $\sigma$ defined as in \eqref{wavetype}, where the function $\Phi:\widehat{G}\rightarrow \cup_{[\xi]\in \widehat{G}}\mathbb{C}^{d_\xi\times d_\xi},$ is such that $\Phi(\xi)=\Phi(\xi)^*$ for all $[\xi]\in \widehat{G}.$ Let us assume that $\sigma$ satisfies the following conditions,
 \begin{equation}
    \sup_{(x,[\xi])\in G\times \widehat{G}}  \Vert X^{\alpha_1}_{i_1}\cdots X^{\alpha_k}_{i_k}\sigma(x,\xi)\Vert_{\textnormal{op}}<\infty,\quad1\leqslant i_1\leqslant i_2\leqslant\cdots \leqslant i_{k}\leqslant k,
\end{equation}  
for all $ |\alpha|\leqslant \varkappa_Q.$ Then $A$ extends to a bounded linear operator on $L^2(G).$ 
\end{theorem}
\begin{proof}
 Let us write $\varkappa_Q=2\ell.$  Note that if $\sigma$ does not depend on $x\in G,$  then the statement  follows from the Plancherel theorem. Now, in the general case where $\sigma$ depends on $x,$ the idea is to use the subelliptic Sobolev embedding theorem. For every $z\in G,$ let us define the family of invariant operators $\{A_z\}_{z\in G},$ by
\[ 
    A_zf(x)=\sum_{[\xi]\in \widehat{G}}d_\xi\textnormal{\textbf{Tr}}[\xi(x)e^{i\Phi(\xi)}\sigma(z,\xi)\widehat{f}(\xi)],\,\,f\in C^\infty(G).
\] By the identity $A_xf(x)=Af(x),$ $x\in G,$  the subelliptic Sobolev embedding Theorem (see Remark \ref{sset}) gives
\begin{align*}
    \sup_{z\in G}|A_zf(x)|\lesssim  \sup_{z\in G}\Vert \mathcal{M}_{2\ell} A_zf\Vert_{L^2(G)}\lesssim 
    \sum_{1\leqslant i_1\leqslant i_2\leqslant\cdots \leqslant i_{k}\leqslant k\,,|\alpha|\leqslant 2\ell}\Vert  X^{\alpha_1}_{i_1,z}\cdots X^{\alpha_k}_{i_k,z}A_zf\Vert_{L^2(G)}.
\end{align*} Every operator $A_{z,\alpha}:=X^{\alpha_1}_{i_1,z}\cdots X^{\alpha_k}_{i_k,z}A_z$ has an invariant symbol given by  $$a_{z,\alpha}(\xi):=e^{i\Phi(\xi)}X^{\alpha_1}_{i_1,z}\cdots X^{\alpha_k}_{i_k,z} \sigma(z,\cdot),$$ and the estimates 
\[ 
    \sup_{(z,[\xi])\in G\times \widehat{G}}  \Vert e^{i\Phi(\xi)} X^{\alpha_1}_{i_1,z}\cdots X^{\alpha_k}_{i_k,z}\sigma(z,\xi)\Vert_{\textnormal{op}}<\infty,\quad1\leqslant i_1\leqslant i_2\leqslant\cdots \leqslant i_{k}\leqslant k,
\] are equivalent, indeed, to the following ones,
\[ 
   \sup_{(z,[\xi])\in G\times \widehat{G}}  \Vert X^{\alpha_1}_{i_1,z}\cdots X^{\alpha_k}_{i_k,z}a_{z,\alpha}(\xi)\Vert_{\textnormal{op}}<\infty,\quad1\leqslant i_1\leqslant i_2\leqslant\cdots \leqslant i_{k}\leqslant k,
\] 
because, using that $e^{i\Phi(\xi)}$ is unitary for all $[\xi]\in \widehat{G},$ one has that  $$\Vert e^{i\Phi(\xi)}X^{\alpha_1}_{i_1,z}\cdots X^{\alpha_k}_{i_k,z}a_{z,\alpha}(\xi)\Vert_{\textnormal{op}}=\Vert X^{\alpha_1}_{i_1,z}\cdots X^{\alpha_k}_{i_k,z}a_{z,\alpha}(\xi)\Vert_{\textnormal{op}}.$$ 
 Consequently the family of left-invariant operators $\{A_{z,\alpha}\}_{z\in G,\,|\alpha|\leqslant 2\ell},$ are uniformly bounded on $L^2(G)$. Moreover, for every $|\alpha|\leqslant 2\ell,$  the function $z\mapsto A_{z,\alpha},$ is a continuous function from $G$ into $\mathscr{B}(L^2(G)).$  The compactness of $G$ implies that 
\[ 
   \sup_{z\in G}\Vert A_{z,\alpha}\Vert_{\mathscr{B}(L^2(G))}= \Vert A_{z_0,\alpha}\Vert_{\mathscr{B}(L^2(G))}=\sup_{[\xi]\in \widehat{G}} \Vert X^{\alpha_1}_{i_1}\cdots X^{\alpha_k}_{i_k}\sigma(z_{0,\alpha},\xi)\Vert_{\textnormal{op}},
\] for some $z_{0,\alpha}\in G.$
Consequently, we can estimate the $L^2(G)$-norm of $Af,$ $f\in C^\infty(G),$ by, 
\begin{align*}
\Vert Af \Vert^2_{L^2(G)}&=\int\limits_{G}|A_xf(x)|^2dx\leqslant \int\limits_{G}\sup_{z\in G}|A_zf(x)|^2dx\\
&\lesssim \sum_{1\leqslant i_1\leqslant i_2\leqslant\cdots \leqslant i_{k}\leqslant k\,,|\alpha|\leqslant 2\ell} \int\limits_{G}\int\limits_{G} \vert X^{\alpha_1}_{i_1,z}\cdots X^{\alpha_k}_{i_k,z} A_zf(x)\vert^2 \, dzdx\\
& \lesssim \sum_{1\leqslant i_1\leqslant i_2\leqslant\cdots \leqslant i_{k}\leqslant k\,,|\alpha|\leqslant 2\ell} \int\limits_{G}\int\limits_{G}  X^{\alpha_1}_{i_1,z}\cdots X^{\alpha_k}_{i_k,z} A_zf(x)\vert^2 \, dxdz\\
&\lesssim \sum_{1\leqslant i_1\leqslant i_2\leqslant\cdots \leqslant i_{k}\leqslant k\,,|\alpha|\leqslant 2\ell}\sup_{z\in G}\Vert X^{\alpha_1}_{i_1,z}\cdots X^{\alpha_k}_{i_k,z} A_z \Vert^2_{\mathscr{B}(L^2(G))} \Vert f\Vert^2_{L^2(G)} \\
&=\sum_{1\leqslant i_1\leqslant i_2\leqslant\cdots \leqslant i_{k}\leqslant k\,,|\alpha|\leqslant 2\ell}\Vert X^{\alpha_1}_{i_1,z}\cdots X^{\alpha_k}_{i_k,z} A_{z_{0,\beta}} \Vert^2_{\mathscr{B}(L^2(G))}  \Vert f\Vert^2_{L^2(G)} . 
\end{align*}  So, we have proved the boundedness of $A$ on $L^2(G).$ 
\end{proof}

Now, we extend Theorem \ref{L2FIO2} to the $L^p$-case for $p\neq 2,$ (although the following theorem also absorbs the case $p=2$).  We will denote by $\varkappa_{Q,p}$  the smallest even integer larger than $Q/p, $ for $1<p<\infty.$ 
\begin{theorem}\label{LpFIO1}
Let $G$ be a compact Lie group and let us denote by $Q$ the Hausdorff
dimension of $G$ associated to the control distance associated to the sub-Laplacian $\mathcal{L}=\mathcal{L}_X,$ where  $X= \{X_{1},\cdots,X_k\} $ is a system of vector fields satisfying the H\"ormander condition. Let  us consider the Fourier integral operator $A\equiv \textnormal{FIO}(\sigma, \phi):C^\infty(G)\rightarrow\mathscr{D}'(G)$ with symbol  $\sigma$ defined as in \eqref{wavetype}, where  $\Phi:\widehat{G}\rightarrow \cup_{[\xi]\in \widehat{G}}\mathbb{C}^{d_\xi\times d_\xi},$ is a matrix-valued function. Let  $1<p<\infty$ and let $0<\rho\leq 1.$  Let us assume that $\sigma$ satisfies the following conditions,
 \begin{equation}\label{condFIO}
    \sup_{(x,[\xi])\in G\times \widehat{G}}  \Vert\widehat{\mathcal{M}}(\xi)^{m+\rho|\gamma|}\Delta_{\xi}^{\gamma}(e^{i\Phi(\xi)} X^{\alpha_1}_{i_1}\cdots X^{\alpha_k}_{i_k}\sigma(x,\xi))\Vert_{\textnormal{op}}<\infty,\quad \gamma\in \mathbb{N}_0^n,
\end{equation}  
for all  $1\leqslant i_1\leqslant i_2\leqslant\cdots \leqslant i_{k}\leqslant k,$ and $ |\alpha|\leqslant \varkappa_{Q,p}.$  Then $A$ extends to a bounded linear operator on $L^p(G),$ provided that
\begin{equation}
     m\geq m_{p}:=Q(1-\rho)\left|\frac{1}{p}-\frac{1}{2}\right|.
\end{equation}
 \end{theorem}
 \begin{remark}
 Let us observe that for $\rho=\frac{1}{Q},$ Theorem \ref{LpFIO1} implies that under the condition $m\geq (Q-1)\left|\frac{1}{p}-\frac{1}{2}\right|,$ the Fourier integral operator $A\equiv \textnormal{FIO}(\sigma, \phi):C^\infty(G)\rightarrow\mathscr{D}'(G)$ with symbol  $\sigma$ defined as in \eqref{wavetype}, and satisfying the family of inequalities in \eqref{condFIO}, extends a bounded operator from $L^p(G)$ to itself.
 \end{remark}
\begin{proof}
 Let us proceed as in the proof of Theorem  \ref{L2FIO2} and  let us write $\varkappa_{Q,p}=2\ell.$  Note that if $\sigma$ does not depend on $x\in G,$  then the statement  follows from Theorem \ref{parta2}, because in that case, $A$ is a Fourier multiplier whose symbol $e^{i\Phi}\sigma$ belongs to $S^{-m,\mathcal{L}}_{\rho,\delta}(G\times \widehat{G})$ with $m\geq m_{p}$. Now, in the general case where $\sigma$ depends on $x,$ the idea is to use the subelliptic $L^p$-Sobolev embedding theorem. For every $z\in G,$ let us define the family of invariant operators $\{A_z\}_{z\in G},$ by
\[ 
    A_zf(x)=\sum_{[\xi]\in \widehat{G}}d_\xi\textnormal{\textbf{Tr}}[\xi(x)e^{i\Phi(\xi)}\sigma(z,\xi)\widehat{f}(\xi)],\,\,f\in C^\infty(G).
\] By the identity $A_xf(x)=Af(x),$ $x\in G,$  the subelliptic $L^p$-Sobolev embedding Theorem (see Coulhon, Russ and  Tardivel-Nachef \cite[Page 288]{Couhlon}) gives
\begin{align*}
    \sup_{z\in G}|A_zf(x)|\lesssim  \sup_{z\in G}\Vert \mathcal{M}_{2\ell} A_zf\Vert_{L^p(G)}\lesssim 
    \sum_{1\leqslant i_1\leqslant i_2\leqslant\cdots \leqslant i_{k}\leqslant k\,,|\alpha|\leqslant 2\ell}\Vert  X^{\alpha_1}_{i_1,z}\cdots X^{\alpha_k}_{i_k,z}A_zf\Vert_{L^p(G)}.
\end{align*} Every operator $A_{z,\alpha}:=X^{\alpha_1}_{i_1,z}\cdots X^{\alpha_k}_{i_k,z}A_z$ has an invariant symbol given by  $$a_{z,\alpha}(\xi):=e^{i\Phi(\xi)}X^{\alpha_1}_{i_1,z}\cdots X^{\alpha_k}_{i_k,z} \sigma(z,\cdot),$$ and the estimates 
\[ 
    \sup_{(z,[\xi])\in G\times \widehat{G}}  \Vert \widehat{\mathcal{M}}(\xi)^{m+\rho|\gamma|}\Delta_{\xi}^{\gamma}(e^{i\Phi(\xi)} X^{\alpha_1}_{i_1,z}\cdots X^{\alpha_k}_{i_k,z}\sigma(z,\xi))\Vert_{\textnormal{op}}<\infty,
\] are equivalent, indeed, to the following ones,
\[ 
   \sup_{(z,[\xi])\in G\times \widehat{G}}  \Vert \widehat{\mathcal{M}}(\xi)^{m+\rho|\gamma|}\Delta_{\xi}^{\gamma}(e^{i\Phi(\xi)}X^{\alpha_1}_{i_1,z}\cdots X^{\alpha_k}_{i_k,z}a_{z,\alpha}(\xi))\Vert_{\textnormal{op}}<\infty.
\] 
 Consequently the family of left-invariant operators operators $\{A_{z,\alpha}\}_{z\in G,\,|\alpha|\leqslant 2\ell},$ are uniformly bounded on $L^p(G)$. Moreover, for every $|\alpha|\leqslant 2\ell,$  the function $z\mapsto A_{z,\alpha},$ is a continuous function from $G$ into $\mathscr{B}(L^p(G)).$  The compactness of $G$ implies that 
\[ 
   \sup_{z\in G}\Vert A_{z,\alpha}\Vert_{\mathscr{B}(L^p(G))}=  \Vert A_{z_0,\alpha}\Vert_{\mathscr{B}(L^p(G))}<\infty,
\] for some $z_{0,\alpha}\in G.$
Consequently, we can estimate the $L^p(G)$-norm of $Af,$ $f\in C^\infty(G),$ by, 
\begin{align*}
\Vert Af \Vert^p_{L^p(G)}&=\int\limits_{G}|A_xf(x)|^pdx\leqslant \int\limits_{G}\sup_{z\in G}|A_zf(x)|^pdx\\
& \lesssim     \sum_{1\leqslant i_1\leqslant i_2\leqslant\cdots \leqslant i_{k}\leqslant k\,,|\alpha|\leqslant 2\ell}\Vert  A_{z_{0,\alpha}} \Vert^p_{\mathscr{B}(L^p(G))}  \Vert f\Vert^p_{L^p(G)} . 
\end{align*}  So, we have proved the boundedness of $A$ on $L^p(G).$
\end{proof}

In   
order to motivate Theorem \ref{L2FIO}, let us recall the following $L^2$-estimate proved for Fourier integral operators on the torus in \cite[Theorem 4.14.2, page 415]{Ruz}. Here, $\{e_{j}\}_{j=1}^n$ is the canonical basis of $\mathbb{R}^n.$
\begin{theorem}\label{L2FIOtorus}
Let  us consider the Fourier integral operator $A\equiv \textnormal{FIO}(\sigma, \phi):C^\infty(\mathbb{T}^n)\rightarrow\mathscr{D}'(\mathbb{T}^n)$ with symbol  $\sigma$ and phase function $\phi,$ satisfying the following conditions
\begin{equation}\label{613torus}
  \sup_{(x,\xi)\in \mathbb{T}^n\times \mathbb{Z}^n } \vert \partial_{x}^\alpha\sigma(x,\xi)\vert<\infty,\,\,\,\, \sup_{j}\sup_{(x,\xi)\in \mathbb{T}^n\times \mathbb{Z}^n } \vert \partial_{x}^\alpha(\phi(x,\xi+e_{j})-\phi(x,\xi))\vert<\infty
\end{equation}for all $ |\alpha|\leqslant 2n+1,$  and 
\begin{equation}\label{maslovtorus}
    \vert \nabla_{x}\phi(x,\xi)-\nabla_{x}\phi(x,\xi')\vert\geqslant    C |\xi-\xi'|.
\end{equation} Then, $A$ extends to a bounded linear operator on $L^2(\mathbb{T}^n).$
\end{theorem}
The second condition in the right hand side of \eqref{613torus} together with \eqref{maslovtorus} implies the following condition (see \cite[page 417]{Ruz}),
\begin{equation}\label{maslovtorus2}
    \vert \nabla_{x}\phi(x,\xi)-\nabla_{x}\phi(x,\xi')\vert\asymp |\xi-\xi'|.
\end{equation} The $L^p$-boundedness of Fourier integral operators on the torus has been treated in \cite{CRS2018,CRS2021:}.
We will require a similar condition in the context of compact Lie groups. In the proof of Theorem \ref{L2FIO} we will use the following version of Cotlar-Stein Lemma due to Comech \cite{Comech}.
\begin{lemma}[Cotlar-Stein Lemma]\label{CLemma} Let $E,F$ be Hilbert spaces and let $\{T_i:E\rightarrow F\}_{i\in \mathbb{Z}},$ be a sequence of bounded operators satisfying the conditions of almost-orthogonality, this means that
\[ 
    \Vert T_{i}^*T_{i}\Vert_{\mathscr{B}(E,F)}\leqslant a(i,j),  \,\,\,\Vert T_{i}T_{i}^*\Vert_{\mathscr{B}(F,E)}\leqslant b(i,j),
\]where where $a(i,j)$ and $b(i,j)$ are non-negative symmetric functions on $\mathbb{Z}\times \mathbb{Z},$ which satisfy
\[ 
    M:=\sup_{i}\sum_{j}a(i,j)^{1-\varepsilon},\,N:=\sup_{i}\sum_{j}b(i,j)^{\varepsilon},\,\,\,0<\varepsilon<1.
\]Then, $T=\sum_{i}T_i:E\rightarrow F,$ extends to a bounded operator and 
\[ 
    \Vert T \Vert_{\mathscr{B}(E,F)}\leqslant (MN)^\frac{1}{2}.
\]

\end{lemma}
\begin{remark}[{On the gap between  the eigenvalues of the Laplacian}]
To analyse  the $L^2$-estimate for Fourier integral operators on compact Lie groups in Theorem \ref{L2FIO}, we will assume a condition about the gap between the eigenvalues for the Laplacian. So, we will assume that for every $\tau>0, $ there exists $\varepsilon_0\in (0,1),$ such that for every $[\xi],[\xi']\in \widehat{G},$ with $\lambda_{[\xi]}<\lambda_{[\xi']},$   we have 
\begin{equation}\label{gap}
 1-  \left(\frac{\lambda_{[\xi]}}{\lambda_{[\xi']}}\right)^\tau\geqslant    
   \varepsilon_0 .
\end{equation} This condition is easily verificable for many classes of compact Lie groups (e.g. the torus  $\mathbb{T}^n,$ $\textnormal{SU}(2),$ $\textnormal{SO}(3),$ etc).
\end{remark}
We will denote by $\mathbb{X}=\{X_{1},\cdots,X_n\}$ a basis for the Lie algebra $\mathfrak{g},$ and the corresponding gradient $ \nabla_{\mathbb{X}},$ defined by
\[ 
    \nabla_{\mathbb{X}}f(x)=(X_{1}f,\cdots,X_nf)\in C^{\infty}(G)\times \cdots \times C^{\infty}(G),\,\,\,f\in C^\infty(G).
\]

\begin{theorem}\label{L2FIO}
Let $G$ be a compact Lie group of topological dimension $n,$ and let  us consider the Fourier integral operator $A\equiv \textnormal{FIO}(\sigma, \phi):C^\infty(G)\rightarrow\mathscr{D}'(G)$ with symbol  $\sigma$ and phase function $\phi,$ satisfying the following conditions
\begin{equation}\label{613}
  \sup_{(x,[\xi])\in G\times \widehat{G}}  \Vert X^{\alpha}_{x}\sigma(x,\xi)\Vert_{\textnormal{op}}<\infty,
\end{equation}for all $ |\alpha|\leqslant 5n/2,$ and 
\begin{equation}\label{gap2}
    \vert \nabla_{\mathbb{X}}\phi_{jj}(x,\xi)-\nabla_{\mathbb{X}}\phi_{j'j'}(x,\xi')\vert\asymp |\lambda_{[\xi]}^\tau-\lambda_{[\xi']}^\tau|,\,\,\,1\leqslant j\leqslant d_\xi,
\end{equation}uniformly in $([\xi],[\xi'])\in \widehat{G}\times \widehat{G},$ for some $\tau>0.$ Let us assume the condition in \eqref{gap}. Then, $A$ extends to a bounded linear operator on $L^2(G).$
\end{theorem}
\begin{proof}To start the proof observe that by  Plancherel Theorem, it is enough to prove that the operator $S:\mathscr{S}(\widehat{G})\subset L^2(\widehat{G})\rightarrow L^2(G),$ defined by 
\[ 
    Sw(x)=\sum\limits_{[\xi]\in \widehat{G}}d_\xi\textnormal{\textbf{Tr}}(e^{i\phi(x,\xi)}\sigma(x,\xi)w(\xi)),\,\,w\in \mathscr{S}(\widehat{G}):=\mathscr{F}C^\infty(G),
\]
admits a bounded extension. Because $\phi(x,\xi)$ is a self-adjoint matrix for every $(x,[\xi]),$  in some basis $B_{\xi}$ of the representation space $\mathbb{C}^{d_\xi},$ $e^{i\phi(x,\xi)}=\textnormal{diag}[e^{i\phi_{jj}(x,\xi)}]_{ j=1 }^{d_\xi}$ is a diagonal matrix  where
\[ 
    \phi_{jj}(x,\xi),\,\,1\leqslant j\leqslant d_\xi,\quad (x,\xi)\in G\times \widehat{G},
\]are the real eigenvalues of $\phi(x,\xi).$ Now, by using the representation 
\begin{align*}
  & Sw(x)\\
   &=\sum\limits_{[\xi]\in \widehat{G}}\sum_{i,j,k=1}^{d_\xi}d_\xi[e^{i\phi(x,\xi)}]_{ij}\sigma(x,\xi)_{jk}w(\xi)_{ki}=\sum\limits_{[\xi]\in \widehat{G}}\sum_{j,k=1}^{d_\xi}d_\xi e^{i\phi_{jj}(x,\xi)}\sigma(x,\xi)_{jk}w(\xi)_{kj},
\end{align*} we can decompose the operator $S$ as,
\[ 
    Sw(x)=\sum\limits_{[\xi]\in \widehat{G}}\sum_{j,k=1}^{d_\xi}S_{[\xi],j,k}w(x),\quad S_{[\xi],j,k}w(x):=d_\xi e^{i\phi_{jj}(x,\xi)}\sigma(x,\xi)_{jk}w(\xi)_{kj}.
\] In order to use Cotlar Lemma (Lemma \ref{CLemma}), we will compute the adjoint of every operator $S_{[\xi],j,k}.$
So, from the identities
\begin{align*}
   & (S_{[\xi],j,k}w,v)_{L^2(G)} :=\int\limits_{G}d_\xi e^{i\phi_{jj}(x,\xi)}\sigma(x,\xi)_{jk}w(\xi)_{kj}\overline{v(x)}dx\\
    &=d_\xi w(\xi)_{k,j}\overline{\int\limits_{G} e^{-i\phi_{jj}(x,\xi)}  \overline{\sigma(x,\xi)}_{jk}v(x)dx}\\
    &=d_\xi w(\xi)_{k,j}\overline{(S_{[\xi],j,k}^{*}v)(\xi)},
\end{align*} where
\[ 
  (S^*_{[\xi],j,k}v) (\eta)_{u,v} := \int\limits_{G} e^{-i\phi_{jj}(x,\xi)}  \overline{\sigma(x,\xi)}_{jk}v(x)dx\cdot \delta_{([\eta],u,v),([\xi],k,j)},
\]we have
\begin{align*}
     &(w,S^*_{[\xi],j,k}v)_{L^2(\widehat{G})} \\&:=\sum_{[\eta]\in \widehat{G}}d_\eta \textnormal{\textbf{Tr}}(w(\eta)(S^*_{[\xi],j,k}v(\eta))^*)=\sum_{[\eta]\in \widehat{G}}\sum_{u,v=1}^{d_\eta} d_\eta w(\eta)_{uv}(S^*_{[\xi],j,k}v(\eta))^*_{v,u}\\
     &=\sum_{[\eta]\in \widehat{G}}\sum_{u,v=1}^{d_\eta}d_\eta w(\eta)_{uv}\overline{(S^*_{[\xi],j,k}v(\eta))_{u,v}}\cdot \delta_{([\eta],u,v),([\xi],k,j)}\\
     &=d_\xi w(\xi)_{k,j}\overline{(S_{[\xi],j,k}^{*}v)(\xi)}
\end{align*}which shows that $S^*_{[\xi],j,k},$ is the adjoint operator of $S_{[\xi],j,k}.$ The next step is to estimate the operators norms,
\[ 
    \Vert S^*_{[\xi],j,k}S_{[\xi],j,k} \Vert_{\mathscr{B}( L^2(\widehat{G})   )},\,\,\,\,\Vert S_{[\xi'],j,k}S^{*}_{[\xi'],j',k'} \Vert_{\mathscr{B}(L^2(G))}.
\]
Observe that for every $[\eta]\in \widehat{G},$ and $1\leqslant u,v\leqslant d_\eta,$
\begin{align*}
  &  (S^*_{[\xi],j,k}S_{[\xi'],j',k'}w)(\eta)_{u,v}= \int\limits_{G} e^{-i\phi_{jj}(x,\xi)}  \overline{\sigma(x,\xi)}_{jk}S_{[\xi'],j',k'}w(x)dx\cdot \delta_{([\eta],u,v),([\xi'],k',j')}\\
    &= d_{\xi'}\int\limits_{G} e^{-i\phi_{jj}(x,\xi)}  \overline{\sigma(x,\xi)}_{jk} e^{i\phi_{j'j'}(x,\xi')}\sigma(x,\xi')_{j'k'}w(\xi')_{k'j'}dx\cdot \delta_{([\eta],u,v),([\xi'],k',j')}\\
    &= d_{\xi'}\int\limits_{G} e^{i(\phi_{j'j'}(x,\xi')-\phi_{jj}(x,\xi))}  \overline{\sigma(x,\xi)}_{jk} \sigma(x,\xi')_{j'k'}w(\xi')_{k'j'}dx\cdot \delta_{([\eta],u,v),([\xi'],k',j')}\\
    &=d_{\eta}w(\eta)_{u,v}\int\limits_{G} e^{i(\phi_{j'j'}(x,\xi')-\phi_{jj}(x,\xi))}  \overline{\sigma(x,\xi)}_{jk} \sigma(x,\xi')_{j'k'}dx\cdot \delta_{([\eta],u,v),([\xi'],k',j')}\\
    &=d_{\eta}w(\eta)_{u,v}K_{[\xi],[\xi'],j,j',k,k'}(\eta,\eta)_{vu},
\end{align*}where we have defined
$$    K_{[\xi],[\xi'],j,j',k,k'}(\eta,\eta)_{vu}$$
$$:=  \int\limits_{G} e^{i(\phi_{j'j'}(x,\xi')-\phi_{jj}(x,\xi))}  \overline{\sigma(x,\xi)}_{jk} \sigma(x,\xi')_{j'k'}dx\cdot \delta_{([\eta],u,v),([\xi'],k',j')}.$$
Estimating the $L^2(\widehat{G})$-norm of $S^*_{[\xi],j,k}S_{[\xi'],j',k'}w,$ we have

\begin{align*}
    \Vert S^*_{[\xi],j,k}S_{[\xi'],j',k'}w\Vert_{L^2(\widehat{G})}^2&=\sum_{[\eta]\in \widehat{G}}d_\eta\Vert S^*_{[\xi],j,k}S_{[\xi'],j',k'}w(\eta)\Vert_{\textnormal{HS}}^2\\
    &=\sum_{[\eta]\in \widehat{G}}\sum_{u,v=1}^{d_\eta}d_\eta\vert S^*_{[\xi],j,k}S_{[\xi'],j',k'}w(\eta)_{uv}\vert^2\\
    &=d_{\xi'}\vert S^*_{[\xi],j,k}S_{[\xi'],j',k'}w(\xi')_{k'j'}\vert^2\\
    &=d_{\xi'}\times d_{\xi'}^2|w(\xi')_{k'j'}|^2|K_{[\xi],[\xi'],j,j',k,k'}(\xi',\xi')_{j'k'}|^2.
\end{align*}Because,
$$   d_{\xi'}|w(\xi')_{k'j'}|^2\leqslant  \sum_{[\eta]\in \widehat{G}}\sum_{u,v=1}^{d_\eta}d_\eta\vert w(\eta)_{uv}\vert^2=\Vert w \Vert^2_{L^2(\widehat{G})},   $$ we deduce that
\[ 
     \Vert S^*_{[\xi],j,k}S_{[\xi'],j',k'}w\Vert_{L^2(\widehat{G})}^2\leqslant d_{\xi'}^2|K_{[\xi],[\xi'],j,j',k,k'}(\xi',\xi')_{j'k'}|^2\Vert w \Vert^2_{L^2(\widehat{G})},
\] or equivalently,
\begin{equation}\label{toestimate}
     \Vert S^*_{[\xi],j,k}S_{[\xi'],j',k'}\Vert_{\mathscr{B}(L^2(\widehat{G}))}\leqslant d_{\xi'}|K_{[\xi],[\xi'],j,j',k,k'}(\xi',\xi')_{j'k'}|.
\end{equation} To estimate $|K_{[\xi],[\xi'],j,j',k,k'}(\xi',\xi')_{j'k'}|,$ we will use the argument of integration by parts. To do so, using the chain rule, observe that for every $X_k\in \mathbb{X},$
\begin{align*}
    X_{k}e^{i(\phi_{j'j'}(x,\xi')-\phi_{jj}(x,\xi))}=i\cdot e^{i(\phi_{j'j'}(x,\xi')-\phi_{jj}(x,\xi))}X_{k}(\phi_{j'j'}(x,\xi')-\phi_{jj}(x,\xi)).
\end{align*}Consequently, the action of the gradient field $\nabla_\mathbb{X}$ on the function $e^{i(\phi_{j'j'}(x,\xi')-\phi_{jj}(x,\xi))},$ is given by
\[ 
    \nabla_\mathbb{X}(e^{i(\phi_{j'j'}(x,\xi')-\phi_{jj}(x,\xi))})=i\cdot e^{i(\phi_{j'j'}(x,\xi')-\phi_{jj}(x,\xi))}\nabla_\mathbb{X}  (   \phi_{j'j'}(x,\xi')-\phi_{jj}(x,\xi)   ).
\]By considering the operator,

\[ 
    \mathbb{L}:=\frac{1}{i}\cdot\frac{ \nabla_\mathbb{X}  (   \phi_{j'j'}(x,\xi')-\phi_{jj}(x,\xi)   )    }{\sum_{r=1}^{n} |{X}_r  (   \phi_{j'j'}(x,\xi')-\phi_{jj}(x,\xi)   )\vert^2}\cdot \nabla_\mathbb{X}
\]defined by 
\[ 
     i\mathbb{L}f(x):=\frac{\sum_{s=1}^{n}  X_{s}(   \phi_{j'j'}(x,\xi')-\phi_{jj}(x,\xi)   )(X_{s}f(x))    }{\sum_{r=1}^{n} |{X}_r  (   \phi_{j'j'}(x,\xi')-\phi_{jj}(x,\xi)   )\vert^2},\,\,f\in C^\infty(G),
\] we have the following identity
\begin{align*}
     &\mathbb{L}(e^{i(\phi_{j'j'}(x,\xi')-\phi_{jj}(x,\xi))})\\
     &=\frac{1}{i}\cdot \frac{ \nabla_\mathbb{X}  (   \phi_{j'j'}(x,\xi')-\phi_{jj}(x,\xi)   )    }{\sum_{r=1}^{n} |{X}_r  (   \phi_{j'j'}(x,\xi')-\phi_{jj}(x,\xi)   )\vert^2}\cdot \nabla_\mathbb{X}(e^{i(\phi_{j'j'}(x,\xi')-\phi_{jj}(x,\xi)}) \\
     &=\frac{1}{i}\cdot \frac{ | \nabla_\mathbb{X}  (   \phi_{j'j'}(x,\xi')-\phi_{jj}(x,\xi)   ) |^2   }{\sum_{r=1}^{n} |{X}_r  (   \phi_{j'j'}(x,\xi')-\phi_{jj}(x,\xi)   )\vert^2}\cdot i\cdot e^{i(\phi_{j'j'}(x,\xi')-\phi_{jj}(x,\xi)}   )\\
     &= e^{i(\phi_{j'j'}(x,\xi')-\phi_{jj}(x,\xi)}.   
\end{align*} Now, by using integration by parts with the operator $\mathbb{L},$ for every $N\in \mathbb{N},$ we have,
\begin{align*}
     \int\limits_{G} e^{i(\phi_{j'j'}(x,\xi')-\phi_{jj}(x,\xi))} & \overline{\sigma(x,\xi)}_{jk} \sigma(x,\xi')_{j'k'}dx\\&
     =\int\limits_{G} e^{i(\phi_{j'j'}(x,\xi')-\phi_{jj}(x,\xi))}  \overline{\sigma(x,\xi)}_{jk} \sigma(x,\xi')_{j'k'}dx\\
     &=\int\limits_{G} \mathbb{L}^{N}( e^{i(\phi_{j'j'}(x,\xi')-\phi_{jj}(x,\xi))})  \overline{\sigma(x,\xi)}_{jk} \sigma(x,\xi')_{j'k'}dx\\
     &=\int\limits_{G}  e^{i(\phi_{j'j'}(x,\xi')-\phi_{jj}(x,\xi))} (\mathbb{L}^{t})^{N}[ \overline{\sigma(x,\xi)}_{jk} \sigma(x,\xi')_{j'k'}]dx.
\end{align*}Note that the transpose of $\mathbb{L},$ $\mathbb{L}^t$ is given by,
\begin{align*}
    \mathbb{L}^t &:=\frac{1}{i}\cdot\frac{ \nabla_\mathbb{X}  (   \phi_{j'j'}(x,\xi')-\phi_{jj}(x,\xi)   )    }{\sum_{r=1}^{n} |{X}_r  (   \phi_{j'j'}(x,\xi')-\phi_{jj}(x,\xi)   )\vert^2}\cdot \nabla_{\mathbb{X}'}\\
    &=\frac{1}{i}\cdot\frac{ \nabla_\mathbb{X}  (   \phi_{j'j'}(x,\xi')-\phi_{jj}(x,\xi)   )    }{ |\nabla_\mathbb{X}  (   \phi_{j'j'}(x,\xi')-\phi_{jj}(x,\xi)   ) |^2}\cdot \nabla_{\mathbb{X}'}
\end{align*}where $\nabla_{\mathbb{X}'}$ is defined by
\[ 
    \nabla_{\mathbb{X}'}f(x):=(-X_{1}f,\cdots,-X_nf)\in C^{\infty}(G)\times \cdots \times C^{\infty}(G),\,\,\,f\in C^\infty(G).
\] In order to estimate \eqref{toestimate}, using \eqref{gap2} we observe that 
\begin{equation}\label{themainestimate}
    |(\mathbb{L}^{t})^{N}[ \overline{\sigma(x,\xi)}_{jk} \sigma(x,\xi')_{j'k'}]|\lesssim \frac{1}{|\lambda_{[\xi']}^\tau-\lambda_{[\xi]}^\tau|^{N}}.
\end{equation} So, from \eqref{toestimate} and \eqref{themainestimate}  we would claim that
\begin{equation}\label{toestimate22}
     \Vert S^*_{[\xi],j,k}S_{[\xi'],j',k'}w\Vert_{\mathscr{B}(L^2(\widehat{G}))}\lesssim \frac{d_{\xi'}}{|\lambda_{[\xi']}^\tau-\lambda_{[\xi]}^\tau|^{N}}.
\end{equation} Indeed, for $N=1,$ by using the Cauchy-Schwarz inequality on $\mathbb{C}^n$  we have,
\begin{align*}
   &|(\mathbb{L}^{t})[ \overline{\sigma(x,\xi)}_{jk} \sigma(x,\xi')_{j'k'}]|\\
   &= \left|\frac{1}{i}\cdot\frac{ \nabla_\mathbb{X}  (   \phi_{j'j'}(x,\xi')-\phi_{jj}(x,\xi)   )    }{ |\nabla_\mathbb{X}  (   \phi_{j'j'}(x,\xi')-\phi_{jj}(x,\xi)   ) |^2}\cdot \nabla_{\mathbb{X}'}[ \overline{\sigma(x,\xi)}_{jk} \sigma(x,\xi')_{j'k'}]\right| \\
   &\leqslant \left|\frac{ \nabla_\mathbb{X}  (   \phi_{j'j'}(x,\xi')-\phi_{jj}(x,\xi)   )    }{ |\nabla_\mathbb{X}  (   \phi_{j'j'}(x,\xi')-\phi_{jj}(x,\xi)   ) |^2}\right|\cdot\left| \nabla_{\mathbb{X}'}[ \overline{\sigma(x,\xi)}_{jk} \sigma(x,\xi')_{j'k'}]\right|\\
    &= \frac{  \left| \nabla_{\mathbb{X}'}[ \overline{\sigma(x,\xi)}_{jk} \sigma(x,\xi')_{j'k'}]\right|  }{ |\nabla_\mathbb{X}  (   \phi_{j'j'}(x,\xi')-\phi_{jj}(x,\xi)   ) |}\\
     &= \frac{  \left( \sum_{r=1}^{n}|X_{r}[ \overline{\sigma(x,\xi)}_{jk} \sigma(x,\xi')_{j'k'}]|^2\right)^{\frac{1}{2}}  }{ |\nabla_\mathbb{X}  (   \phi_{j'j'}(x,\xi')-\phi_{jj}(x,\xi)   ) |}.
\end{align*}By the Leibniz rule, and using \eqref{613}, we can estimate
\begin{align*}
    &\left( \sum_{r=1}^{n}|X_{r}[ \overline{\sigma(x,\xi)}_{jk} \sigma(x,\xi')_{j'k'}]|^2\right)^{\frac{1}{2}}\\ &=\left( \sum_{r=1}^{n}|X_{r}[ \overline{\sigma(x,\xi)}_{jk}] \sigma(x,\xi')_{j'k'}]+[ \overline{\sigma(x,\xi)}_{jk}]X_{r} \sigma(x,\xi')_{j'k'}]|^2\right)^{\frac{1}{2}}\\
    &\lesssim \left( \sum_{r=1}^{n}[ 
    \Vert X_{r}[ \overline{\sigma(x,\xi)}]\Vert_{\textnormal{op}} \Vert \sigma(x,\xi')\Vert_{\textnormal{op}}+\Vert  \overline{\sigma(x,\xi)} \Vert_{\textnormal{op}}\Vert X_{r} \sigma(x,\xi')\Vert_{\textnormal{op}}]^2\right)^{\frac{1}{2}}\\
    &\lesssim (2n)^{\frac{1}{2}},
\end{align*} and consequently, we have
\begin{align}
    |(\mathbb{L}^{t})[ \overline{\sigma(x,\xi)}_{jk} \sigma(x,\xi')_{j'k'}]|\lesssim  \frac{1}{|\lambda_{[\xi']}^\tau-\lambda_{[\xi]}^\tau|},
\end{align}which proves \eqref{themainestimate} for $N=1.$ The general case $N>1,$ can be proved using induction. In order to use Lemma \ref{CLemma}, we need symmetry in the upper bound for the norm of $\Vert S^*_{[\xi],j,k}S_{[\xi'],j',k'}w\Vert_{\mathscr{B}(L^2(\widehat{G}))}.$ So, we trivially can deduce the following inequality \begin{equation}\label{toestimate222}
     \Vert S^*_{[\xi],j,k}S_{[\xi'],j',k'}\Vert_{\mathscr{B}(L^2(\widehat{G}))}\lesssim \frac{d_{\xi'}+d_\xi}{|\lambda_{[\xi']}^\tau-\lambda_{[\xi]}^\tau|^{N}},
\end{equation} from \eqref{toestimate22}. Let us assume for a moment the inequality
\begin{equation}\label{toestimate2223}
     \Vert S_{[\xi'],j',k'}S^*_{[\xi],j,k}\Vert_{\mathscr{B}(L^2(G))}\lesssim \frac{d_{\xi'}+d_\xi}{|\lambda_{[\xi']}^\tau-\lambda_{[\xi]}^\tau|^{N}}.
\end{equation}So, from Lemma \ref{CLemma}, it follows the $L^2(G)$-boundedness of $A,$  if we prove that
\begin{align*}
   \sup_{[\xi'],j',k'} \sum_{[\xi]\neq[\xi'] ,j,k}\frac{(d_{\xi'}+d_\xi)^{\frac{1}{2}}}{|\lambda_{[\xi']}^\tau-\lambda_{[\xi]}^\tau|^{\frac{N}{2}}}=\sup_{[\xi],j,k} \sum_{[\xi']\neq[\xi] ,j',k'}\frac{(d_{\xi'}+d_\xi)^{\frac{1}{2}}}{|\lambda_{[\xi]}^\tau-\lambda_{[\xi']}^\tau|^{\frac{N}{2}}}<\infty.
\end{align*}
For this, we will use the Weyl-eigenvalue counting Formula for the Laplacian (see e.g. in Remark \ref{weyl}, that in the case of the Laplacian,  $Q=n,$ $s=1$ in order  to deduce that $N(\lambda)=O(\lambda^n)$). First, we can split the sums as follows
\begin{align*}
    \sum_{[\xi]\neq[\xi'] ,j,k}\frac{(d_{\xi'}+d_\xi)^{\frac{1}{2}}}{|\lambda_{[\xi']}^\tau-\lambda_{[\xi]}^\tau|^{\frac{N}{2}}}&=\sum_{\lambda_{[\xi]}<\lambda_{[\xi']} ,j,k}\frac{(d_{\xi'}+d_\xi)^{\frac{1}{2}}}{|\lambda_{[\xi']}^\tau-\lambda_{[\xi]}^\tau|^{\frac{N}{2}}}+\sum_{\lambda_{[\xi]}>\lambda_{[\xi']} ,j,k}\frac{(d_{\xi'}+d_\xi)^{\frac{1}{2}}}{|\lambda_{[\xi']}^\tau-\lambda_{[\xi]}^\tau|^{\frac{N}{2}}}\\
    &:=I+II.
\end{align*}
To estimate $I,$ observe that
\begin{align*}
    I&=\sum_{\lambda_{[\xi]}<\lambda_{[\xi']} ,j,k}\frac{(d_{\xi'}+d_\xi)^{\frac{1}{2}}}{|\lambda_{[\xi']}^\tau-\lambda_{[\xi]}^\tau|^{\frac{N}{2}}}=\sum_{\lambda_{[\xi]}<\lambda_{[\xi']} ,j,k}\frac{(d_{\xi'}+d_\xi)^{\frac{1}{2}}}{(\lambda_{[\xi']}^\tau-\lambda_{[\xi]}^\tau)^{\frac{N}{2}}}\\
   &\lesssim \sum_{\lambda_{[\xi]}<\lambda_{[\xi']} ,j,k}\frac{(  \langle \xi'\rangle^{\frac{n}{2}}+\langle \xi\rangle^{\frac{n}{2}})^{\frac{1}{2}}}{(\lambda_{[\xi']}^\tau-\lambda_{[\xi]}^\tau)^{\frac{N}{2}}}=\sum_{\lambda_{[\xi]}<\lambda_{[\xi']} ,j,k}\frac{  \langle \xi'\rangle^{\frac{n}{4}}     }{(\lambda_{[\xi']}^\tau-\lambda_{[\xi]}^\tau))^{\frac{N}{2}}}\\
   &\lesssim \sum_{\lambda_{[\xi]}<\lambda_{[\xi']} ,j,k}\frac{  \lambda_{[\xi']}^{\frac{n}{4}}     }{(\lambda_{[\xi']}^\tau-\lambda_{[\xi]}^\tau))^{\frac{N}{2}}}=\sum_{\lambda_{[\xi]}<\lambda_{[\xi']} ,j,k}\frac{  \lambda_{[\xi']}^{\frac{n}{4} -\frac{N}{2}}     }{(1-\lambda_{[\xi']}^{-\tau}\lambda_{[\xi]}^\tau))^{\frac{N}{2}}}\\
   &=  \lambda_{[\xi']}^{\frac{n}{4} -\frac{N}{2}}     \sum_{\lambda_{[\xi]}<\lambda_{[\xi']} ,j,k}\frac{1}{(1-\lambda_{[\xi']}^{-\tau}\lambda_{[\xi]}^\tau))^{\frac{N}{2}}}\\
   &\lesssim \varepsilon_{0}^{-\frac{N}{2}}\lambda_{[\xi']}^{\frac{n}{4} -\frac{N}{2}}     \sum_{\lambda_{[\xi]}<\lambda_{[\xi']} }d_{\xi}^2\\
   &= O(\varepsilon_{0}^{-\frac{N}{2}}\lambda_{[\xi']}^{\frac{n}{4} -\frac{N}{2}}     \lambda_{[\xi']} ^{{n}})<\infty,
\end{align*} 
where we have used $\varepsilon_0$ from \eqref{gap}, and also that $\frac{n}{4}-\frac{N}{2}+n=\frac{5n}{4}-\frac{N}{2}\geq     0,$ or equivalently that $N\geqslant    \frac{5n}{2}.$ Now, we will estimate  $II$  as follows,
\begin{align*}
    &\sum_{\lambda_{[\xi]}>\lambda_{[\xi']} ,j,k}\frac{(d_{\xi'}+d_\xi)^{\frac{1}{2}}}{|\lambda_{[\xi']}^\tau-\lambda_{[\xi]}^\tau|^{\frac{N}{2}}}\lesssim \sum_{\lambda_{[\xi]}>\lambda_{[\xi']} ,j,k}\frac{d_\xi^{\frac{1}{2}}}{|\lambda_{[\xi']}^\tau-\lambda_{[\xi]}^\tau|^{\frac{N}{2}}}\\
   & \lesssim \sum_{\lambda_{[\xi]}>\lambda_{[\xi']} ,j,k}\frac{d_\xi^{\frac{1}{2}}}{|\lambda_{[\xi']}^\tau-\lambda_{[\xi]}^\tau|^{\frac{N}{2}}}\lesssim  \sum_{\lambda_{[\xi]}>\lambda_{[\xi']} ,j,k}\frac{\lambda_{[\xi]}^{\frac{n}{4}}}{|\lambda_{[\xi']}^\tau-\lambda_{[\xi]}^\tau|^{\frac{N}{2}}}.
\end{align*}Observe that
\[ 
    \sum_{\lambda_{[\xi]}>\lambda_{[\xi']} ,j,k}\frac{\lambda_{[\xi]}^{\frac{n}{4}}}{|\lambda_{[\xi']}^\tau-\lambda_{[\xi]}^\tau|^{\frac{N}{2}}}\leqslant \varepsilon_0^{-\frac{N}{2}}\sum_{\lambda_{[\xi]}>\lambda_{[\xi']} ,j,k}{\lambda_{[\xi]}^{\frac{n}{4}-\frac{N}{2}}}\leqslant \varepsilon_0^{-\frac{N}{2}}\sum_{[\xi]\in \widehat{G} }d_{\xi}^2{\lambda_{[\xi]}^{\frac{n}{4}-\frac{N}{2}}}<\infty,
\]
 provided that $N\geqslant    \frac{5n}{2}.$ In order to finish the proof we will estimate the norm $\Vert S_{[\xi],j,k}S^{*}_{[\xi'],j',k'} \Vert_{\mathscr{B}(L^2(\widehat{G}),L^2(G)}$ as we have assumed in \eqref{toestimate2223}. However, observe that 
 \begin{align*}
 & S_{[\xi],j,k}(S^{*}_{[\xi'],j',k'}v) (x)=d_{\xi}e^{i\phi_{jj}(x,\xi)}\sigma(x,\xi)_{jk}[S^{*}_{[\xi'],j',k'}v(\xi)]_{kj}\cdot\delta_{([\xi],j,k),([\xi'],j',k')}\\
  &=d_{\xi}e^{i\phi_{jj}(x,\xi)}\sigma(x,\xi)_{jk}\int\limits_{G}e^{-i\phi_{j'j'}(y,\xi)}\overline{\sigma(y,\xi')}_{j'k'}v(y)dy\cdot \delta_{([\xi],j,k),([\xi'],j',k')}.
 \end{align*}So, we can estimate
 \begin{align*}
  &\Vert S_{[\xi],j,k}S^{*}_{[\xi'],j',k'} v\Vert_{L^2(G)}   \\&\leq \Vert d_{\xi}e^{i\phi_{jj}(\cdot ,\xi)}\sigma(\cdot ,\xi)_{jk}\int\limits_{G}e^{-i\phi_{j'j'}(y,\xi)}\overline{\sigma(y,\xi')}_{j'k'}v(y)dy\cdot \delta_{([\xi],j,k),([\xi'],j',k')} \Vert_{L^2(G)}\\
  &\leqslant \sup_{x\in G,u,v}|\sigma(x,\xi)_{uv}|^{2}d_{\xi}\Vert v\Vert_{L^1(G)}\cdot \delta_{([\xi],j,k),([\xi'],j',k')}\\
  &\leqslant \sup_{x\in G,u,v}|\sigma(x,\xi)_{uv}|^{2}d_{\xi}\Vert v\Vert_{L^2(G)}\cdot\delta_{([\xi],j,k),([\xi'],j',k')}\\
  &\leqslant \sup_{x\in G}\Vert \sigma(x,\xi)\Vert_{\textnormal{op}}^{2}d_{\xi}\Vert v\Vert_{L^2(G)}\cdot\delta_{([\xi],j,k),([\xi'],j',k')}\\
  &\lesssim_{\sigma} (d_{\xi}+d_{\xi'})\Vert v\Vert_{L^2(G)}\cdot\delta_{([\xi],j,k),([\xi'],j',k')}.
 \end{align*} So, from the estimate  $$\Vert S_{[\xi],j,k}S^{*}_{[\xi'],j',k'} \Vert_{\mathscr{B}({L^2(G)})}  \lesssim_{\sigma}(d_{\xi}+d_{\xi'})\cdot\delta_{([\xi],j,k),([\xi'],j',k')}$$
 we can deduce \eqref{toestimate2223}. The proof is complete.
\end{proof}
\begin{remark}
In Theorem \ref{L2FIO}, we can replace the conditions \begin{equation}\label{613remark}
  \sup_{(x,[\xi])\in G\times \widehat{G}}  \Vert X^{\alpha}_{x}\sigma(x,\xi)\Vert_{\textnormal{op}}<\infty,\,\,\,\,\vert \nabla_{\mathbb{X}}\phi_{jj}(x,\xi)-\nabla_{\mathbb{X}}\phi_{j'j'}(x,\xi')\vert\asymp |\lambda_{[\xi]}^\tau-\lambda_{[\xi']}^\tau|,
\end{equation} by the following ones
\begin{equation}\label{613remark2}
  \sup_{(x,[\xi])\in G\times \widehat{G}}  \Vert \partial_{X}^{(\alpha)}\sigma(x,\xi)\Vert_{\textnormal{op}}<\infty,\,\,\vert \nabla\phi_{jj}(x,\xi)-\nabla\phi_{j'j'}(x,\xi')\vert\asymp |\lambda_{[\xi]}^\tau-\lambda_{[\xi']}^\tau|,
\end{equation} where $\nabla:=(\partial_{X_{1}},\cdots, \partial_{X_{n}}),$ is the gradient field defined by the local derivatives induced by the local coordinates.
\end{remark}

\section{Appendix I:  Sub-Laplacians on $\mathbb{S}^3,$ $\textnormal{SO}(4),$  $ \textnormal{SU}(3),$ and $\textnormal{Spin}(4) $}\label{examplessublaplacians} We will present some examples of sub-Laplacians on some compact Lie groups. By abuse of notation, we will use the same symbol to denote an element of the Lie algebra and the vector field on the group  obtained by left translation.
\begin{example}[Sub-Laplacians on $\textnormal{SU}(2)\cong \mathbb{S}^3$]\label{SU2} Let us consider the left-invariant first-order  differential operators $\partial_{+},\partial_{-},\partial_{0}: C^{\infty}(\textnormal{SU}(2))\rightarrow C^{\infty}(\textnormal{SU}(2)),$ called creation, annihilation, and neutral operators respectively, (see Definition 11.5.10 of \cite{Ruz}) and let us define 
\[ 
    X_{1}=-\frac{i}{2}(\partial_{-}+\partial_{+}),\, X_{2}=\frac{1}{2}(\partial_{-}-\partial_{+}),\, X_{3}=-i\partial_0,
\]where $X_{3}=[X_1,X_2],$ based in the commutation relations $[\partial_{0},\partial_{+}]=\partial_{+},$ $[\partial_{-},\partial_{0}]=\partial_{-},$ and $[\partial_{+},\partial_{-}]=2\partial_{0}.$ The system $X=\{X_1,X_2\}$ satisfies the H\"ormander condition at step $\kappa=2,$ and the Hausdorff dimension defined by the control distance associated to the sub-Laplacian   $\mathcal{L}_1=-X_1^2-X_2^2$  is $Q=4.$ In a similar way, we can define the sub-Laplacian $\mathcal{L}_2=-X_2^2-X_3^2$ associated to the system of vector fields $X'=\{X_2,X_3\},$ which also satisfies the H\"ormander condition of step $\kappa=2.$ 

\end{example}
\begin{example}[The sub-Laplacian on $\textnormal{SO}(4)$]
Consider the Lie group $\textnormal{SO}(4)=\{g\in \textnormal{GL}(3,\mathbb{R}):gg^t=I_4\equiv(\delta_{ij})_{1\leqslant i,j\leqslant 4},\, \textnormal{\bf{det}}(g)=1\},$ with Lie algebra $\mathfrak{so}(4).$ The latter consists of all skew-symmetric matrices. This Lie algebra is generated by matrices of the form
\[ 
    B^{ij}:=e_{i}e_{j}^t-e_{j}e_{i}^t,
\]where $\{e_{i}\}_{i=1}^{4}$ are the canonical vectors in $\mathbb{R}^4.$ Let us define $$ X_1:=B^{12},\, X^2:=B^{14},\, X_3:=B^{24},\, X_4:=B^{34}.$$ By writing  $$Z_1:=-[X_2,X_4]=B^{13},\textnormal{   and   }Z_{2}:=-[X_3,X_4]=B^{23},$$ from T{\SMALL{ABLE}} \ref{table1} the system of vector fields $X=\{X_1,X_2,X_3,X_4\}$ satisfies the H\"ormander condition of step $\kappa=2$ (see Berge and Grong \cite{Berge}). So, the Hausdorff dimension associated to the control distance associated to the sub-Laplacian
\[ 
    \mathcal{L}=-X_1^2-X_2^2-X_{3}^2-X_4^2,
\] can be computed from \eqref{Hausdorff-dimension} as follows.
\begin{align*}
    Q:&=\dim(H^1G)+\sum_{i=1}^{2} (i+1)(\dim H^{i+1}G-\dim H^{i}G )
    = 4+2(6-4)=8.
\end{align*}
\begin{table}[h]  
\caption{Commutators in $\textnormal{SO}(4)$ } 
\centering 
\begin{tabular}{l c c rrrrrrr} 
\hline\hline   
       & \,\, $X_1$ & \,\, $X_2$  & \,\, $X_3$ & \,\, $X_4$ & \,\, $Z_1$ & \,\, $Z_2$    \\
$X_1$  &  0 & $-X_3$  & $X_2$   & $0$ & $-Z_2$  & $Z_1$    \\

$X_2$  & $X_3$  &  0 & $-X_1$ & $-Z_1$ & $X_4$  & $0$     \\

$X_3$  &  $-X_2$ &  $X_1$ & 0 &  $-Z_2$&  $0$& $X_4$    \\

$X_4$  & $0$ & $Z_1$  & $Z_2$ & $0$ & $-X_2$ & $-X_3$    \\

$Z_1$  & $Z_2$ &  $-X_4$ & $0$ & $X_2$ & 0  & $-X_1$    \\

$Z_2$  & $-Z_1$ & $0$  & $-X_4$  & $X_3$ &  $X_1$ & 0  \\[1ex]  
\hline 
\end{tabular}  
\label{table1}
\end{table}

\end{example}
\begin{example}[The sub-Laplacian on $\textnormal{SU}(3)$] The special unitary group of $3\times 3$ complex matrices is defined by $$\textnormal{SU}(3)=\{g\in \textnormal{GL}(3,\mathbb{C}):gg^*=I_3\equiv(\delta_{ij})_{1\leqslant i,j\leqslant 3},\, \textnormal{\bf{det}}(g)=1\},$$
and its Lie algebra is given by
$$  \mathfrak{su}(3)=\{g\in \textnormal{GL}(3,\mathbb{C}):g+g^*=0,\, \textnormal{\bf{Tr}}(g)=0\}.$$ The inner product is defined by a multiple of the Killing form on $ \mathfrak{su}(3)$  given by $B(X,Y)=-1/2\textnormal{\bf{Tr}}[XY].$ The torus
$$ \mathbb{T}_{\textnormal{SU}(3)}=\{\textnormal{diag}[e^{i\theta_{1}},e^{i\theta_{2}},e^{i\theta_{3}}]:\theta_1+\theta_2+\theta_3=0,\,\theta_i\in \mathbb{R}\}  $$ is a maximal torus of $\textnormal{SU}(3),$ and its Lie algebra is given by
$$  \mathfrak{t}_{ \mathfrak{su}(3)}=\{\textnormal{diag}[i\theta_{1},i\theta_{2},i\theta_{3}]:\theta_1+\theta_2+\theta_3=0,\,\theta_i\in \mathbb{R}\}.  $$  The following vectors  
$$ T_{1}= \textnormal{diag}[-i,i,0],\quad T_{2}= \textnormal{diag}[-i/\sqrt{3},-i/\sqrt{3},2i/\sqrt{3}] $$ provide a basis for $ \mathfrak{t}_{ \mathfrak{su}(3)}.$ Completing this basis with the following vectors
\begin{align*}
  &X_{1}=  \begin{pmatrix}
0 & 1 &   0 \\
-1 & 0&  0 \\
0 & 0 & 0 
\end{pmatrix},\,
X_{2}=  \begin{pmatrix}
0 & i &   0 \\
i & 0&  0 \\
0 & 0 & 0 
\end{pmatrix},\,
\\
&X_{3}=  \begin{pmatrix}
0 & 0 &   0 \\
0 & 0&  1 \\
0 & -1 & 0 
\end{pmatrix},\,  X_{4}=  \begin{pmatrix}
0 & 0 &   0 \\
0 & 0&  -i \\
0 & -i & 0 
\end{pmatrix},\,
\\
    &X_{5}=  \begin{pmatrix}
0 & 0 &   1 \\
0 & 0&  0 \\
-1 & 0 & 0 
\end{pmatrix},\,
X_{6}=  \begin{pmatrix}
0 & 0 &   i \\
0 & 0&  0 \\
i & 0& 0 
\end{pmatrix},
\end{align*}
we obtain the  Gell-Mann system, which forms an orthonormal basis of $\mathfrak{su}(3).$ The system of vector fields $X=\{X_1,X_2,X_3,X_4,X_5,X_6\}$ satisfies the H\"ormander condition at step $\kappa=2,$ (see Domokos and Manfredi \cite{Domokos2}). Indeed, it can be deduced if we write
\[ 
  X_{7}=-[X_1,X_2]=  \begin{pmatrix}
-2i & 0 &   0 \\
0 & 2i&  0 \\
0 & 0 & 0 
\end{pmatrix},\,
\]
\[ 
    X_{8}=-[X_3,X_4]=  \begin{pmatrix}
0 & 0 &   0 \\
0 & 2i&  0 \\
0 & 0& -2i
\end{pmatrix}
\]
from T{\SMALL{ABLE}} \ref{table2}.
\begin{table}[h]  
\caption{Commutators in $\textnormal{SU}(3)$ } 
\centering 
\begin{tabular}{l c c rrrrrrr} 
\hline\hline   
       & \,\, $X_1$ & \,\, $X_2$  & \,\, $X_3$ & \,\, $X_4$ & \,\, $X_5$ & \,\, $X_6$ & \,\, $X_7$ &\,\, $X_8$    \\
$X_1$  &  0 & $-X_7$  & $X_5$   & $-X_6$ & $-X_3$  & $X_4$ & $4X_2$ & $2X_2$   \\

$X_2$  & $X_7$  &  0 & $X_6$ & $X_5$ & $-X_4$  & $-X_3$  & $-4X_1$ & $-2X_1$   \\

$X_3$  &  $-X_5$ &  $-X_6$ & 0 &  $-X_8$&  $X_1$& $X_2$ & $2X_4$  & $4X_4$   \\

$X_4$  & $X_6$ & $-X_5$  & $X_8$ & $0$ & $X_2$ & $-X_1$ &  $-2X_3$ & $-4X_3$   \\

$X_5$  & $X_3$ &  $X_4$ & $-X_1$ & $-X_2$ & 0  & $X_8-X_7$  & $2X_6$ & $-2X_6$   \\

$X_6$  & $-X_4$ & $X_3$  & $-X_2$  & $X_1$ &  $X_7-X_8$ & 0 & $-2X_5$ & $2X_5$    \\

$X_7$  & $-4X_2$ & $4X_1$  & $-2X_4$  & $2X_3$ & $-2X_6$ & $2X_5$ & 0 &  0 \\

$X_8$  & $2X_2$ &  $2X_1$ & $-4X_4$ & $4X_3$ & $2X_6 $  & $-2X_5$  & 0 & 0   \\[1ex]  
\hline 
\end{tabular}  
\label{table2}
\end{table}Observe that the Hausdorff dimension associated to the control distance associated to the sub-Laplacian
\[ 
    \mathcal{L}=-X_1^2-X_2^2-X_{3}^2-X_4^2-X_5^2-X_6^2,
\] can be computed from \eqref{Hausdorff-dimension} as follows.
\begin{align*}
    Q:&=\dim(H^1G)+ 2(\dim H^{2}G-\dim H^{1}G )
    = 6+2(8-6)=10.
\end{align*}

\end{example} 
\begin{example}[The sub-Laplacian on $\textnormal{Spin}(4)\cong \textnormal{SU}(2)\times \textnormal{SU}(2) $] Let us consider the Lie algebra of $\textnormal{SU}(2),$ $\mathfrak{su}(2)$ spanned by the following matrices
$$A=(1/\sqrt{2})\begin{pmatrix}
0 & i  \\
i & 0 
\end{pmatrix},B=(1/\sqrt{2})\begin{pmatrix}
0 & -1  \\
1 & 0 
\end{pmatrix}\,\textnormal{ and }C=(1/\sqrt{2})\begin{pmatrix}
i & 0  \\
0 & -i 
\end{pmatrix}.$$
On the Lie group $\textnormal{Spin}(4)\cong \textnormal{SU}(2)\times \textnormal{SU}(2) $ with Lie algebra $\mathfrak{spin}(4)\cong \mathfrak{su}(2)\times \mathfrak{su}(2),$ let us define the following vector fields,
\[ 
    X^{\pm}=A_1\pm A_2, \,Y^{\pm}=B_1\pm B_2,\,\,Z^{\pm}=C_1\pm C_2,
\]where
\[ 
    A_{1}=A\otimes I,\,A_{2}=I\otimes A,\,B_{1}=B\otimes I,\,B_{2}=I\otimes B,C_{1}=C\otimes I,\,C_{2}=I\otimes C.
\]For every real-number $c\in (-\infty,\infty),$ let $$X^c:=X^{-}+cX^{+}.$$

\begin{table}[h]  
\caption{Commutators in $\textnormal{Spin}(4)\cong \textnormal{SU}(2)\times \textnormal{SU}(2)$ } 
\centering 
\begin{tabular}{l c c rrrrrrr} 
\hline\hline   
       & \,\, $X^c$ & \,\, $Y^-$  & \,\, $Z^-$ & \,\, $X^+$ & \,\, $Y^+$ & \,\, $Z^+$ \\
       
$X^c$  & 0  & $Z^+ +cZ^{-}$  & $-Y^+ -cY^-$   & $0$ & $Z^-+ cZ^+$  & $-Y^- -cY^+$    \\

$Y^-$  & $-Z^+-cZ^-$  &  0 & $X^+$ & $-Z^{-}$ & $0$  & $X^c-cX^+$     \\

$Z^-$  &  $Y^++cY^-$ &  $-X^+$ & 0 &  $Y^-$&  $cX^+-X^c$& $0$    \\

$X^+$  & $0$ & $Z^-$  & $-Y^-$ & $0$ & $-Z^+$ & $-Y^+$    \\

$Y^+$  & $-Z^--cZ^+$ &  $0$ & $X^c-cX^+$ & $Z^+$ & 0  & $X^+$    \\

$Z^+$  & $Y^-+cY^+$ & $cX^+-X^c$  & $0$  & $Y^+$ &  $-X^+$ & 0      \\[1ex]  
\hline 
\end{tabular}  
\label{table22}
\end{table}
The system $X=\{X^c,Y^-,Z^-\}$ satisifes the H\"ormander condition, and we can consider the associated sub-Laplacian associated with $X,$
$$  \mathcal{L}=-(X^c)^2-(Y^-)^2-(Z^-)^2. $$ 

\end{example}

\section{Appendix II: Subelliptic Besov spaces}\label{Finalsect}
In this appendix  we present the description of subelliptic Besov spaces as in \cite{CardonaRuzhansky2019I}.
 On $\mathbb{R}^n,$ Besov spaces appears in \cite{Besov1,Besov2}. We refer the reader to Triebel \cite{Triebel1983,Triebel2006}, Furioli,  Melzi, and   Veneruso, \cite{furioli},  Nursultanov, Ruzhansky and Tikhonov \cite{NurRuzTikhBesov2015,NurRuzTikhBesov2017},  Peetre \cite{Peetre1,Peetre2}, and \cite{Cardona22}, for the analytic aspects of the theory of Besov spaces on $\mathbb{R}^n$ and others Lie groups. 

In order to define subelliptic Besov spaces we will use the notion of dyadic decompositions. 
Here, the sequence $\{\psi_{\tilde\ell}\}_{\tilde\ell\in \mathbb{N}_{0}}$ is a dyadic decomposition,  defined as follows:  we choose a function $\psi_0\in C^{\infty}_{0}(\mathbb{R}),$  $\psi_0(\lambda)=1,$  if $|\lambda|\leq 1,$ and $\psi_0(\lambda)=0,$ for $|\lambda|\geq 2.$ For every $j\geq 1,$ let us define $\psi_{j}(\lambda)=\psi_{0}(2^{-j}\lambda)-\psi_{0}(2^{-j+1}\lambda).$ For $\psi(\lambda):=\psi_0(\lambda)-\psi_{0}(2\lambda),$ $\psi_{j}(\lambda)=\psi(2^{-j}\lambda).$  In particular, we have
\begin{eqnarray}\label{deco1}
\sum_{\tilde\ell\in \mathbb{N}_{0}}\psi_{\tilde\ell}(\lambda)=1,\,\,\, \text{for every}\,\,\, \lambda>0.
\end{eqnarray}
We define the operator
\begin{equation}\label{functionofL''}
    \psi_{\tilde\ell}(D) f(x):=\sum_{ [\xi]\in \widehat{G}   }d_\xi \textnormal{\textbf{Tr}}[\xi(x) \psi_{\tilde\ell}((I_{d_\xi}+\widehat{ \mathcal{L}}(\xi))^\frac{1}{2})\widehat{f}(\xi) ],\,\,f\in C^\infty(G).
\end{equation} 
So, for $s\in\mathbb{R},$ $0<q<\infty,$ the subelliptic Besov space $B^{s,\mathcal{L}}_{p,q}(G)$ consists of those functions/distributions satisfying
\[ 
   \Vert f\Vert_{ B^{s,\mathcal{L}}_{p,q}(G)  } =\left( \sum_{\tilde\ell=0}^\infty  2^{\tilde\ell q s}\Vert \psi_{\tilde\ell}(D)f \Vert_{L^p(G)}^q \right)^{\frac{1}{q}}<\infty,
\] for $0<p\leqslant \infty,$ with the following modification
\[ 
   \Vert f\Vert_{ B^{s,\mathcal{L}}_{p,\infty}(G)  } = \sup_{\tilde\ell\in\mathbb{N}_0}  2^{\tilde\ell  s}\Vert \psi_{\tilde\ell}(D)f \Vert_{L^p(G)} <\infty,
\] when $q=\infty.$ In \cite{CardonaRuzhansky2019I}, the authors have described the subelliptic Besov spaces in terms of the matrix-valued quantization. We record it as follows (see Remark 4.2 in \cite{CardonaRuzhansky2019I}).
\begin{remark}[Fourier description for subelliptic Besov spaces]
If we write  $\widehat{\mathcal{M}}(\xi):=(\mathcal{M}\xi)(e_G)$ for the symbol of  the operator $(1+\mathcal{L})^{\frac{1}{2}},$ which  is given by,
\[ 
\widehat{\mathcal{M}}(\xi)  =\begin{bmatrix}
    (1+\nu_{11}(\xi)^2)^{\frac{1}{2}} & 0 & 0 & \dots  & 0 \\
     0 & (1+\nu_{22}(\xi)^2)^{\frac{1}{2}}  & 0 & \dots  & 0 \\
    \vdots & \vdots & \vdots & \ddots & \vdots \\
   0 & 0 &0 & \dots  & (1+\nu_{d_\xi d_\xi}(\xi)^2)^{\frac{1}{2}}
\end{bmatrix},
\] so that $\widehat{M}(\xi):= \textnormal{diag}[(1+\nu_{ii}(\xi)^2)^{\frac{1}{2}}]_{1\leqslant i\leqslant d_\xi} ,  $ then    $\psi_{\tilde\ell}(\xi)$ denotes the symbol of the operator $\psi_{\tilde\ell}(D),$ and we have
\[ 
    \psi_{\tilde\ell}(\xi)=\textnormal{diag}[\psi_{\tilde\ell}((1+\nu_{ii}(\xi)^2)^{\frac{1}{2}})]_{1\leqslant i\leqslant d_\xi},\,\,\tilde\ell\in\mathbb{N}_0,\,\,[\xi]\in \widehat{G}.
\]Then the  subelliptic Besov spaces  can be re-written as
\begin{equation}\label{q<inftyLaPgjgj}
   \Vert f\Vert_{ B^{s,\mathcal{L}}_{p,q}(G)  }^q =\sum_{\tilde\ell=0}^\infty  2^{\tilde\ell q s}\left\Vert  \sum_{ [\xi]\in \widehat{G}}d_\xi \textnormal{\textbf{Tr}}[\xi(x)\textnormal{diag}[\psi_{\tilde\ell}((1+\nu_{ii}(\xi)^2)^{\frac{1}{2}})] \widehat{f}(\xi) ]\right\Vert_{L^p(G)}^q <\infty,
\end{equation} for $0<p\leqslant \infty,$ and
\begin{equation}\label{q=inftyLaPhjkhk}
   \Vert f\Vert_{ B^{s,\mathcal{L}}_{p,\infty}(G)  } = \sup_{\tilde\ell\in\mathbb{N}_0}  2^{\tilde\ell  s}\left\Vert  \sum_{ [\xi]\in \widehat{G}   }d_\xi \textnormal{\textbf{Tr}}[\xi(x) \textnormal{diag}[\psi_{\tilde\ell}((1+\nu_{ii}(\xi)^2)^{\frac{1}{2}})]\widehat{f}(\xi) ]\right\Vert_{L^p(G)} <\infty,
\end{equation} for $q=\infty.$ 

\end{remark}

\begin{remark}
Subelliptic  Sobolev and Besov spaces may depend on the choice of a sub-Laplacian on a compact Lie group (see Remark \ref{Sdependence}). This is one of the contrasts to the case of graded Lie groups (see \cite[Chapter 4]{FischerRuzhanskyBook}).
\end{remark}
The following are the embedding properties proved in Theorem 4.3 of \cite{CardonaRuzhansky2019I}.
\begin{theorem}
Let $G$ be a compact Lie group and let us denote by  $Q$  the Hausdorff dimension of $G$ associated to the control distance associated to the sub-Laplacian $\mathcal{L}=-(X_1^2+\cdots +X_k^2),$ where the system of vector fields $X=\{X_i\}$ satisfies the H\"ormander condition.  Then 
\begin{itemize}
\item[(1)] ${B}^{r+\varepsilon,\mathcal{L}}_{p,q_1}(G)\hookrightarrow {B}^{r,\mathcal{L}}_{p,q_1}(G)\hookrightarrow {B}^{r,\mathcal{L}}_{p,q_2}(G)\hookrightarrow {B}^{r,\mathcal{L}}_{p,\infty}(G),$  $\varepsilon>0,$ $0<p\leqslant \infty,$ $0<q_{1}\leqslant q_2\leqslant \infty.$\\
\item[(2)]  ${B}^{r+\varepsilon,\mathcal{L}}_{p,q_1}(G)\hookrightarrow {B}^{r,\mathcal{L}}_{p,q_2}(G)$, $\varepsilon>0,$ $0<p\leqslant \infty,$ $1\leqslant q_2<q_1<\infty.$\\
\item[(3)]  ${B}^{r_1,\mathcal{L}}_{p_1,q}(G)\hookrightarrow {B}^{r_2,\mathcal{L}}_{p_2,q}(G),$  $1\leqslant p_1\leqslant p_2\leqslant \infty,$ $0<q<\infty,$ $r_{1}\in\mathbb{R},$ and $r_2=r_1- {Q}(\frac{1}{p_1}-\frac{1}{p_2}).$\\
\item[(4)] ${H}^{r,\mathcal{L}}(G)={B}^{r,\mathcal{L}}_{2,2}(G)$ and ${B}^{r,\mathcal{L}}_{p,p}(G)\hookrightarrow {L}_r^{p,\mathcal{L}}(G)\hookrightarrow {B}^{r,\mathcal{L}}_{p,2}(G),$ $1<p\leqslant 2.$\\
\item[(5)] ${B}^{r,\mathcal{L}}_{p,1}(G)\hookrightarrow L^{q}(G), $ $1\leqslant p\leqslant q\leqslant \infty,$ $r= {Q}(\frac{1}{p}-\frac{1}{q})$ and $L^{q}(G)\hookrightarrow {B}^{0,\mathcal{L}}_{q,\infty}(G)$ for $1<q\leqslant \infty.$
\end{itemize}
\end{theorem}
In order to compare  subelliptic Besov spaces with Besov spaces associated to the Laplacian, we  start by presenting the problem for Sobolev spaces (see Theorem 5.9 of \cite{CardonaRuzhansky2019I}). Here, $$\varkappa:=[n/2]+1,$$ is the smallest  integer larger than $n/2,$ $n=\dim(G).$ 
\begin{theorem}\label{SubHvsEllipH} Let $G$ be a compact Lie group and let us consider the sub-Laplacian $\mathcal{L}=-(X_1^2+\cdots +X_k^2),$ where the system of vector fields $X=\{X_i\}_{i=1}^k$ satisfies the H\"ormander condition of order $\kappa$.
 Then we have the continuous embeddings
\begin{equation}
L_s^{p}(G)\hookrightarrow L_s^{p,\mathcal{L}}(G) \hookrightarrow L_{ \frac{s}{\kappa}-\varkappa(1-\frac{1}{\kappa})\left|\frac{1}{2}-\frac{1}{p}\right|  }^{p}(G).
\end{equation} More precisely, for every $s\geq 0$ there exist constants $C_a>0$ and $C_{b}>0$ satisfying,
\begin{equation}\label{Embd1}
   C_a \Vert f\Vert_{ L_{  s_{\kappa,\varkappa} }^{p}(G) }\leq \Vert f\Vert_{L_s^{p,\mathcal{L}}(G)},\,\,f\in L_s^{p,\mathcal{L}}(G),
\end{equation}
where $s_{\kappa,\varkappa}:=\frac{s}{\kappa}-\varkappa(1-\frac{1}{\kappa})\left|\frac{1}{2}-\frac{1}{p}\right|,$ and 
\begin{equation}\label{Embd2}
    \Vert f\Vert_{ L_s^{p,\mathcal{L}}(G)}\leq C_b\Vert f\Vert_{L_s^{p}(G)},\,\,f\in L_s^{p}(G).
\end{equation}
Consequently, we have the following embeddings
\begin{equation}
L_{ -\frac{s}{\kappa}+\varkappa(1-\frac{1}{\kappa})\left|\frac{1}{2}-\frac{1}{p}\right| }^{p}(G)\hookrightarrow L_{-s}^{p,\mathcal{L}}(G) \hookrightarrow L_{-s}^{p}(G) .
\end{equation}
\end{theorem}

The following theorem (see Theorem 5.11 \cite{CardonaRuzhansky2019I}) shows some embedding properties between subelliptic Besov spaces and the Besov spaces associated to the Laplacian. 

\begin{theorem}\label{SubBesovvsEllipBesov} Let $G$ be a compact Lie group and let us consider the sub-Laplacian $\mathcal{L}=-(X_1^2+\cdots +X_k^2),$ where the system of vector fields $X=\{X_i\}_{i=1}^k$ satisfies the H\"ormander condition of order $\kappa$.
Let $s\geq 0,$ $0<q\leq \infty,$ and $1<p< \infty.$ Then we have the continuous embeddings
\begin{equation}
B^{s}_{p,q}(G)\hookrightarrow B^{s,\mathcal{L}}_{p,q}(G) \hookrightarrow B^{\frac{s}{\kappa}-\varkappa(1-\frac{1}{\kappa})\left|\frac{1}{2}-\frac{1}{p}\right|}_{p,q}(G).
\end{equation} More precisely, for every $s\geq 0$ there exist constants $C_a>0$ and $C_{b}>0$ satisfying,
\begin{equation}
   C_a \Vert f\Vert_{ B^{s_{\kappa,\varkappa}}_{p,q}(G) }\leq \Vert f\Vert_{B^{s,\mathcal{L}}_{p,q}(G)},\,\,f\in B^{s,\mathcal{L}}_{p,q}(G),
\end{equation}
where $s_{\kappa,\varkappa}:=\frac{s}{\kappa}-\varkappa(1-\frac{1}{\kappa})\left|\frac{1}{2}-\frac{1}{p}\right|,$ and 
\begin{equation}
    \Vert f\Vert_{B^{s,\mathcal{L}}_{p,q}(G)}\leq C_b\Vert f\Vert_{B^{s}_{p,q}(G)},\,\,f\in B^{s}_{p,q}(G).
\end{equation}
Consequently, we have the following embeddings
\begin{equation}
B^{-\frac{s}{\kappa}+\varkappa(1-\frac{1}{\kappa})\left|\frac{1}{2}-\frac{1}{p}\right|}_{p,q}(G)\hookrightarrow B^{-s,\mathcal{L}}_{p,q}(G) \hookrightarrow B^{-s}_{p,q}(G).
\end{equation}
\end{theorem}
Finally we summarise the action of the H\"ormander classes on subelliptic Sobolev and Besov spaces (See Theorems 1.2 and 1.3 in \cite{CardonaRuzhansky2019I}). 
\begin{theorem}[Fefferman Subelliptic Sobolev Theorem]\label{SubellipticLp11} Let $G$ be a compact Lie group of dimension $n.$
Let us assume that $\sigma\in \mathscr{S}^{-\nu}_{\rho,\delta}(G\times \widehat{G}))$ and let $0\leq \delta\leq \rho\leq 1,$ $\delta\neq 1.$  Then $A\equiv \sigma(x,D)$ extends to a bounded operator from $L^{p,\mathcal{L}}_{\vartheta}(G)$ to $L^p(G)$ provided that
\begin{equation}
    n(1-\min\{\rho,1/\kappa\})\left|\frac{1}{p}-\frac{1}{2}\right|-({\vartheta}/{\kappa})\leq  \nu.
\end{equation}  In particular, if $\sigma\in \mathscr{S}^{0}_{\rho,\delta}(G),$ the operator $A\equiv \sigma(x,D)$ extends to a bounded operator from $L^{p,\mathcal{L}}_{\vartheta}(G)$ to $L^p(G)$ with
\begin{equation}
    { n\kappa(1-\min\{\rho,1/\kappa\})}\left|\frac{1}{p}-\frac{1}{2}\right|\leq \vartheta .
\end{equation} 
\end{theorem}
\begin{theorem}[Fefferman Subelliptic Besov Theorem]\label{SubellipticBesov11}
Let us assume that $\sigma\in \mathscr{S}^{-\nu}_{\rho,\delta}(G\times \widehat{G}))$ and let $0\leq \delta\leq \rho\leq 1,$ $\delta\neq 1.$  Then $A\equiv \sigma(x,D)$ extends to a bounded operator from $B^{s+\vartheta,\mathcal{L}}_{p,q}(G)$   to $B^{s}_{p,q}(G)$ provided that
\begin{equation}
    n(1-\min\{\rho,1/\kappa\})\left|\frac{1}{p}-\frac{1}{2}\right|-({\vartheta}/{\kappa})\leq  \nu.
\end{equation}   In particular, if $\sigma\in \mathscr{S}^{0}_{\rho,\delta}(G),$ the operator $A\equiv \sigma(x,D)$ extends to a bounded operator from $B^{s+\vartheta,\mathcal{L}}_{p,q}(G)$   to $B^{s}_{p,q}(G)$  with
\begin{equation}
    { n\kappa(1-\min\{\rho,1/\kappa\})}\left|\frac{1}{p}-\frac{1}{2}\right|\leq \vartheta.
\end{equation} 
\end{theorem} 
Observe that Theorem  \ref{SubellipticLp11} and Theorem \ref{SubellipticBesov11} are analogies of the ones obtained in this work for the subelliptic H\"ormander classes (see Theorems \ref{parta} and \ref{parta2}). The mapping properties for the global calculus developed in \cite{Ruz} on $L^p(G),$  Sobolev spaces $L^p_{r}(G)$, Besov spaces $B^r_{p,q}(G)$ and subelliptic Sobolev and Besov spaces can be found in the references \cite[Chapter 10]{Ruz}, the works of the second author and J. Wirth \cite{RuzhanskyWirth2014,RuzhanskyWirth2015},  \cite{Cardona1,Cardona2,Cardona3,CardonaRuzhansky2019I} and the work of the second author and J. Delgado \cite{RuzhanskyDelgado2017}.

\section{Appendix III:   H\"ormander classes on graded Lie  groups}\label{AppeGraded}

It was proved in Theorem \ref{cor} and consequently in Corollary \ref{corC} that the following family of  seminorms, \begin{equation}\label{InI2XIII}
      p'_{\alpha,\beta,\rho,\delta,m,r}(a):= \sup_{(x,[\xi])\in G\times \widehat{G} }\Vert \widehat{ \mathcal{M}}(\xi)^{(\rho|\alpha|-\delta|\beta|-m-r)}\partial_{X}^{(\beta)} \Delta_{\xi}^{\alpha}a(x,\xi)\widehat{ \mathcal{M}}(\xi)^{r}\Vert_{\textnormal{op}} <\infty,
   \end{equation}
can be used to define the  subelliptic H\"ormander class ${S}^{m,\mathcal{L}}_{\rho,\delta}(G\times \widehat{G}).$ More precisely, we have proved that the following conditions are equivalent:
\begin{itemize}
    \item[(A').] $\forall \alpha,\beta\in \mathbb{N}_{0}^n, \forall r\in \mathbb{R}, $   $p'_{\alpha,\beta,r,m}(a)<\infty.$
    
    \item[(B').]  $\forall \alpha,\beta\in \mathbb{N}_{0}^n, $   $p'_{\alpha,\beta,0,m}(a)<\infty.$
    
    \item[(C').]  $\forall \alpha,\beta\in \mathbb{N}_{0}^n, $   $p'_{\alpha,\beta,m+\delta|\beta|-\rho|\alpha|,m}(a)<\infty.$
    \item[(D').]  $\forall \alpha,\beta\in \mathbb{N}_{0}^n, \exists r_0\in \mathbb{R}, $   $p'_{\alpha,\beta,r_0,m}(a)<\infty.$
    \item[(E').] $a\in {S}^{m,\mathcal{L}}_{\rho,\delta}(G\times \widehat{G}).$
\end{itemize}
The main goal of this appendix is to prove an analogy of Theorem \ref{cor} for the global H\"ormander classes  on arbitrary graded Lie groups developed by the second author and V. Fischer in \cite{FischerRuzhanskyBook}. The main result of this appendix is Theorem \ref{corgraded} where we prove that the following seminorm inequalities are equivalent: (see the next subsections for the notations and definitions) 
\begin{itemize}
    \item[(A).] $\forall \alpha,\beta\in \mathbb{N}_{0}^n, \forall \gamma\in \mathbb{R}, $   $p_{\alpha,\beta,\gamma,m}(\sigma)<\infty.$
    
    \item[(B).]  $\forall \alpha,\beta\in \mathbb{N}_{0}^n, $   $p_{\alpha,\beta,0,m}(\sigma)<\infty.$
    
    \item[(C).]  $\forall \alpha,\beta\in \mathbb{N}_{0}^n, $   $p_{\alpha,\beta,m+\delta|\beta|-\rho|\alpha|,m}(\sigma)<\infty.$
    \item[(D).]  $\forall \alpha,\beta\in \mathbb{N}_{0}^n, \exists \gamma_0\in \mathbb{R}, $   $p_{\alpha,\beta,\gamma_0,m}(\sigma)<\infty.$
     \item[(E).] $\sigma\in {S}^{m}_{\rho,\delta}(G\times \widehat{G}).$
\end{itemize}
In the case of the Heisenberg group it was proved in Theorem 6.5.1 of \cite[Page 479]{FischerRuzhanskyBook} that (A), (B), (C) and (D) are equivalent conditions and, in Section 5.5 of \cite[Page 479]{FischerRuzhanskyBook}, it was proved that on a arbitrary graded Lie group only (A), (B) and (C) are equivalent seminorms. Condition (D), is certainly, most useful. We will prove Theorem \ref{corgraded} by following the arguments  in Section \ref{Seccionsubelliptic}.

    \subsection{Homogeneous and graded Lie groups} 
        The notation and terminology of this appendix on the analysis of homogeneous Lie groups are mostly taken 
from Folland and Stein \cite{FollandStein1982}. For the theory of pseudo-differential operators we will follow the setting developed in \cite{FischerRuzhanskyBook} through  the notion of (operator-valued) global symbols. If $E,F$ are Hilbert spaces,  $\mathscr{B}(E,F)$ denotes the algebra of bounded linear operators from $E$ to $F,$ and also we will write $\mathscr{B}(E)=\mathscr{B}(E,E).$

    Let $G$ be a homogeneous Lie group. This means that $G$ is a connected and simply connected Lie group whose Lie algebra $\mathfrak{g}$ is endowed with a family of dilations $D_{r}^{\mathfrak{g}},$ $r>0,$ which are automorphisms on $\mathfrak{g}$  satisfying the following two conditions:
\begin{itemize}
\item For every $r>0,$ $D_{r}^{\mathfrak{g}}$ is a map of the form
$$ D_{r}^{\mathfrak{g}}=\textnormal{Exp}(rA) $$
for some diagonalisable linear operator $A\equiv \textnormal{diag}[\nu_1,\cdots,\nu_n]$ on $\mathfrak{g}.$
\item $\forall X,Y\in \mathfrak{g}, $ and $r>0,$ $[D_{r}^{\mathfrak{g}}X, D_{r}^{\mathfrak{g}}Y]=D_{r}^{\mathfrak{g}}[X,Y].$ 
\end{itemize}
We call  the eigenvalues of $A,$ $\nu_1,\nu_2,\cdots,\nu_n,$ the dilations weights or weights of $G$.  The homogeneous dimension of a homogeneous Lie group $G$ is given by  $$ Q=\textnormal{\textbf{Tr}}(A)=\nu_1+\cdots+\nu_n.  $$
The dilations $D_{r}^{\mathfrak{g}}$ of the Lie algebra $\mathfrak{g}$ induce a family of  maps on $G$ defined via
$$ D_{r}:=\exp_{G}\circ D_{r}^{\mathfrak{g}} \circ \exp_{G}^{-1},\,\, r>0, $$
where $\exp_{G}:\mathfrak{g}\rightarrow G$ is the usual exponential mapping associated to the Lie group $G.$ We refer to the family $D_{r},$ $r>0,$ as dilations on the group. If we write $rx=D_{r}(x),$ $x\in G,$ $r>0,$ then a relation on the homogeneous structure of $G$ and the Haar measure $dx$ on $G$ is given by $$ \int\limits_{G}(f\circ D_{r})(x)dx=r^{-Q}\int\limits_{G}f(x)dx. $$
    
A  Lie group is graded if its Lie algebra $\mathfrak{g}$ may be decomposed as the sum of subspaces $\mathfrak{g}=\mathfrak{g}_{1}\oplus\mathfrak{g}_{2}\oplus \cdots \oplus \mathfrak{g}_{s}$ such that $[\mathfrak{g}_{i},\mathfrak{g}_{j} ]\subset \mathfrak{g}_{i+j},$ and $ \mathfrak{g}_{i+j}=\{0\}$ if $i+j>s.$  Examples of such groups are the Heisenberg group $\mathbb{H}^n$ and more generally any stratified groups where the Lie algebra $ \mathfrak{g}$ is generated by $\mathfrak{g}_{1}$.  Here, $n$ is the topological dimension of $G,$ $n=n_{1}+\cdots +n_{s},$ where $n_{k}=\mbox{dim}\mathfrak{g}_{k}.$

A Lie algebra admitting a family of dilations is nilpotent, and hence so is its associated
connected, simply connected Lie group. The converse does not hold, i.e., not every
nilpotent Lie group is homogeneous (see Dyer \cite{Dyer1970}) although they exhaust a large class (see Johnson \cite[page 294]{Johnson1975}). Indeed, the main class of Lie groups under our consideration is that of graded Lie groups. A graded Lie group $G$ is a homogeneous Lie group equipped with a family of weights $\nu_j,$ all of them positive rational numbers. Let us observe that if $\nu_{i}=\frac{a_i}{b_i}$ with $a_i,b_i$ integer numbers,  and $b$ is the least common multiple of the $b_i's,$ the family of dilations 
$$ \mathbb{D}_{r}^{\mathfrak{g}}=\textnormal{Exp}(\log(r^b)A):\mathfrak{g}\rightarrow\mathfrak{g}, $$
have integer weights,  $\nu_{i}=\frac{a_i b}{b_i}. $ So, here we always assume that the weights $\nu_j,$ defining the family of dilations are non-negative integer numbers which allow us to assume that the homogeneous dimension $Q$ is a non-negative integer number. This is a natural context for the study of Rockland operators (see Remark 4.1.4 of \cite{FischerRuzhanskyBook}).

\subsection{Fourier analysis on nilpotent Lie groups}

Let $G$ be a simply connected nilpotent Lie group.  
Let us assume that $\pi$ is a continuous, unitary and irreducible  representation of $G,$ this means that,
\begin{itemize}
    \item $\pi\in \textnormal{Hom}(G, \textnormal{U}(H_{\pi})),$ for some separable Hilbert space $H_\pi,$ i.e. $\pi(xy)=\pi(x)\pi(y)$ and for the  adjoint of $\pi(x),$ $\pi(x)^*=\pi(x^{-1}),$ for every $x,y\in G.$
    \item The map $(x,v)\mapsto \pi(x)v, $ from $G\times H_\pi$ into $H_\pi$ is continuous.
    \item For every $x\in G,$ and for any subspace $W_\pi\subset H_\pi$ such that $\pi(x)W_{\pi}\subset W_{\pi},$ then $W_\pi=H_\pi$ or $W_\pi=\{0\}.$
\end{itemize} Let $\textnormal{Rep}(G)$ be the set of unitary, continuous and irreducible representations of $G.$ The relation, {\small{
\[ 
    \pi_1\sim \pi_2\textnormal{ if and only if, there exists } A\in \mathscr{B}(H_{\pi_1},H_{\pi_2}),\textnormal{ such that }A\pi_{1}(x)A^{-1}=\pi_2(x), 
\]}}for every $x\in G,$ is an equivalence relation and the unitary dual of $G,$ denoted by $\widehat{G}$ is defined via
$
    \widehat{G}:={\textnormal{Rep}(G)}/{\sim}.
$ Let us denote by $d\pi$ the Plancherel measure on $\widehat{G}.$ 
The Fourier transform of $f\in \mathscr{S}(G), $ (this means that $f\circ \textnormal{exp}_G\in \mathscr{S}(\mathfrak{g})$, with $\mathfrak{g}\simeq \mathbb{R}^{\dim(G)}$) at $\pi\in\widehat{G},$ is defined by 
\[ 
    \widehat{f}(\pi)=\int\limits_{G}f(x)\pi(x)^*dx:H_\pi\rightarrow H_\pi,\textnormal{   and   }\mathscr{F}_{G}:\mathscr{S}(G)\rightarrow \mathscr{S}(\widehat{G}):=\mathscr{F}_{G}(\mathscr{S}(G)).
\]

If we identify one representation $\pi$ with its equivalence class, $[\pi]=\{\pi':\pi\sim \pi'\}$,  for every $\pi\in \widehat{G}, $ the Kirillov trace character $\Theta_\pi$ defined by  $$(\Theta_{\pi},f):
=\textnormal{\textbf{Tr}}(\widehat{f}(\pi)),$$ is a tempered distribution on $\mathscr{S}(G).$ In particular, the identity
$
    f(e_G)=\int\limits_{\widehat{G}}(\Theta_{\pi},f)d\pi,
$
implies the Fourier inversion formula $f=\mathscr{F}_G^{-1}(\widehat{f}),$ where
\[ 
    (\mathscr{F}_G^{-1}\sigma)(x):=\int\limits_{\widehat{G}}\textnormal{\textbf{Tr}}(\pi(x)\sigma(\pi))d\pi,\,\,x\in G,\,\,\,\,\mathscr{F}_G^{-1}:\mathscr{S}(\widehat{G})\rightarrow\mathscr{S}(G),
\]is the inverse Fourier  transform. In this context, the Plancherel theorem takes the form $\Vert f\Vert_{L^2(G)}=\Vert \widehat{f}\Vert_{L^2(\widehat{G})}$,  where  $$L^2(\widehat{G}):=\int\limits_{\widehat{G}}H_\pi\otimes H_{\pi}^*d\pi,$$ is the Hilbert space endowed with the norm: $\Vert \sigma\Vert_{L^2(\widehat{G})}=(\int\limits_{\widehat{G}}\Vert \sigma(\pi)\Vert_{\textnormal{HS}}^2d\pi)^{\frac{1}{2}}.$

We will fix a homogeneous quasi-norm $|\cdot|$ on $G.$  This means that $|\cdot|$ is a non-negative function on $G,$ satisfying 
\[ 
    |x|=|x^{-1}|,\,\,\,r|x|=|D_r( x)|,\,\,\,\textnormal{ and }|x|=0 \textnormal{ if and only if  }x=e_{G},
\]
where $e_{G}$ is the identity element of $G.$ It satisfies a triangle
inequality with a constant: there exists a constant $\gamma\geq 1$ such that $|xy|\leq \gamma(|x|+|y|).$

\subsection{Homogeneous linear operators and Rockland operators} A linear operator $T:C^\infty(G)\rightarrow \mathscr{D}'(G)$ is homogeneous of  degree $\nu\in \mathbb{C}$ if for every $r>0$ the equality 
\[ 
T(f\circ D_{r})=r^{\nu}(Tf)\circ D_{r}
\]
holds for every $f\in \mathscr{D}(G). $
If for every representation $\pi\in\widehat{G},$ $\pi:G\rightarrow U(\mathcal{H}_{\pi}),$ we denote by $\mathcal{H}_{\pi}^{\infty}$ the set of smooth vectors, that is, the space of elements $v\in \mathcal{H}_{\pi}$ such that the function $x\mapsto \pi(x)v,$ $x\in \widehat{G},$ is smooth,  a Rockland operator is a left-invariant differential operator $\mathcal{R}$ which is homogeneous of positive degree $\nu=\nu_{\mathcal{R}}$ and such that, for every unitary irreducible non-trivial representation $\pi\in \widehat{G},$ $\pi(\mathcal{R})$ is injective on $\mathcal{H}_{\pi}^{\infty};$ $\sigma_{\mathcal{R}}(\pi)=\pi(\mathcal{R})$ is the symbol associated to $\mathcal{R}.$ It coincides with the infinitesimal representation of $\mathcal{R}$ as an element of the universal enveloping algebra. It can be shown that a Lie group $G$ is graded if and only if there exists a differential Rockland operator on $G.$ If the Rockland operator is formally self-adjoint, then $\mathcal{R}$ and $\pi(\mathcal{R})$ admit self-adjoint extensions on $L^{2}(G)$ and $\mathcal{H}_{\pi},$ respectively. Now if we preserve the same notation for their self-adjoint
extensions and we denote by $E$ and $E_{\pi}$  their spectral measures, we will denote by
$$ f(\mathcal{R}):=\int\limits_{-\infty}^{\infty}f(\lambda) dE(\lambda),\,\,\,\textnormal{and}\,\,\,\pi(f(\mathcal{R}))\equiv f(\pi(\mathcal{R})):=\int\limits_{-\infty}^{\infty}f(\lambda) dE_{\pi}(\lambda), $$ the operators defined by the functional calculus. 
In general, we will reserve the notation ${dE_A(\lambda)}_{0<\lambda<\infty}$ for the spectral measure associated with a positive and self-adjoint operator $A$ on a Hilbert space $H.$ 

We now recall a lemma on dilations on the unitary dual $\widehat{G},$ which will be useful in our analysis of spectral multipliers.   For the proof, see Lemma 4.3 of \cite{FischerRuzhanskyBook}.
\begin{lemma}\label{dilationsrepre}
For every $\pi\in \widehat{G}$ let us define  $D_{r}(\pi)(x)=\pi(rx)$ for every $r>0$ and $x\in G.$ Then, if $f\in L^{\infty}(\mathbb{R})$ then $f(\pi^{(r)}(\mathcal{R}))=f({r^{\nu}\pi(\mathcal{R})}).$
\end{lemma}

\subsection{Symbols of pseudo-differential operators}
 In order to present a consistent definition of pseudo-differential operators one developed in \cite{FischerRuzhanskyBook} (see the quantisation formula  \eqref{Quantization}),  a suitable class of spaces on the unitary dual $\widehat{G}$ acting in a suitable way with the set of smooth vectors $H_{\pi}^{\infty},$ on every representation space $H_{\pi}.$ Let now recall the main notions.
 \begin{definition}[Sobolev spaces on  smooth vectors] Let $\pi_1\in \textnormal{Rep}(G),$ and $a\in \mathbb{R}.$ We denote by $H_{\pi_1}^a,$ the Hilbert space obtained  as the completion of $H_{\pi_1}^\infty$ with respect to the norm
 \[ 
     \Vert v \Vert_{H_{\pi_1}^a}=\Vert\pi_1(1+\mathcal{R})^{\frac{a}{\nu}} v\Vert_{H_{\pi_1}},
 \]
 where $\mathcal{R}$ is  a positive Rockland operator on $G$ of homogeneous degree $\nu>0.$
 \end{definition}
 
 In order to introduce the general notion of a symbol as the one developed in \cite{FischerRuzhanskyBook}, we will use a suitable notion of operator-valued  symbols acting on smooth vectors. We introduce it as follows.
 
\begin{definition}
A $\widehat{G}$-field of operators $\sigma=\{\sigma(\pi):\pi\in \widehat{G}\}$ defined on smooth vectors {is defined} on the Sobolev space ${H}_\pi^a$ when for each representation $\pi_1\in \textnormal{Rep}(G),$ the operator $\sigma(\pi_1)$ is bounded from $H^a_{\pi_1}$ into $H_{\pi_{1}}$ in the sense that
\[ 
    \sup_{ \Vert v\Vert_{H_{\pi_1}^a}=1}\Vert \sigma(\pi_1)v \Vert<\infty.
\]
\end{definition}

We will consider those $\widehat{G}-$fields of operators with ranges in Sobolev spaces on smooth vectors.  We recall that the  Sobolev space $L^{2}_{a}(G)$ is  defined by the norm (see \cite[Chapter 4]{FischerRuzhanskyBook})
\begin{equation}\label{L2ab2}
    \Vert f \Vert_{L^{2}_{a}(G)}=\Vert (1+\mathcal{R})^{\frac{a}{\nu}}f\Vert_{L^2(G)},
\end{equation} for $s\in \mathbb{R}.$

\begin{definition}
A $\widehat{G}$-field of operators  defined on smooth vectors {with range} in the Sobolev space $H_{\pi}^a$ is a family of classes of operators $\sigma=\{\sigma(\pi):\pi\in \widehat{G}\}$ where
\[ 
    \sigma(\pi):=\{\sigma(\pi_1):H^{\infty}_{\pi_1}\rightarrow H_{\pi}^a,\,\,\pi_1\in \pi\},
\] for every $\pi\in \widehat{G}$ viewed as a subset of $\textnormal{Rep}(G),$ satisfying for every two elements $\sigma(\pi_1)$ and $\sigma(\pi_2)$ in $\sigma(\pi):$
\[ 
  \textnormal{If  }  \pi_{1}\sim \pi_2  \textnormal{  then  }   \sigma(\pi_1)\sim \sigma(\pi_2). 
\]
\end{definition}
  The following notion  will be useful in order to use the general theory of non-commutative integration (see e.g. Dixmier \cite{Dixmier1953}). 
\begin{definition}
A $\widehat{G}$-field of operators  defined on smooth vectors with range in the Sobolev space $H_\pi^a$ {is measurable}  when for some (and hence for any) $\pi_1\in \pi$ and any vector $v_{\pi_1}\in H_{\pi_1}^\infty,$ as $\pi\in \widehat{G},$ the resulting field $\{\sigma(\pi)v_\pi:\pi\in\widehat{G}\},$ 
is $d\pi$-measurable and
\[ 
    \int\limits_{\widehat{G} }\Vert v_\pi \Vert^2_{H_\pi^a}d\pi=\int\limits_{\widehat{G} }\Vert\pi(1+\mathcal{R})^{\frac{a}{\nu}} v_\pi \Vert^2_{H_\pi}d\pi<\infty.
\]
\end{definition}
\begin{remark}
We always assume that a $\widehat{G}$-field of operators  defined on smooth vectors {with range} in the Sobolev space $H_{\pi}^a$ is $d\pi$-measurable.
\end{remark} The $\widehat{G}$-fields of operators associated to Rockland operators can be defined as follows.
\begin{definition}
Let $L^2_a(\widehat{G})$ denote the space of fields of operators $\sigma$ with range in $H_\pi^a,$ that is,
\[ 
    \sigma=\{\sigma(\pi):H_\pi^\infty\rightarrow H_\pi^a\}, \textnormal{ with }\{\pi(1+\mathcal{R})^{\frac{a}{\nu}}\sigma(\pi):\pi\in \widehat{G}\}\in L^2(\widehat{G}),
\]for one (and hence for any) Rockland operator of homogeneous degree $\nu.$ We also denote
\[ 
    \Vert \sigma\Vert_{L^2_a(\widehat{G})}:=\Vert \pi(1+\mathcal{R})^{\frac{a}{\nu}}\sigma(\pi)\Vert_{L^2(\widehat{G})}.
\]
\end{definition} With the notation above, we will introduce some natural spaces which arise as spaces of  $\widehat{G}$-fields of operators.
\begin{definition}[The spaces $\mathscr{L}_{L}(L^2_a(G),L^2_b(G)),$ $\mathcal{K}_{a,b}(G)$ and $L^\infty_{a,b}(\widehat{G})$]
\hspace{0.1cm}
\begin{itemize}
    \item The space $\mathscr{L}_{L}(L^2_a(G),L^2_b(G)) $ consists of all  left-invariant operators $T$  such that  $T:L^2_a(G)\rightarrow L^2_b(G) $ extends to a bounded operator.
    \item The space  $\mathcal{K}_{a,b}(G)$ is the family of all right convolution kernels of elements in $  \mathscr{L}_{L}(L^2_a(G),L^2_b(G))  ,$ i.e. $k=T\delta\in \mathcal{K}_{a,b}(G)$ if and only if  $T\in
    \mathscr{L}_{L}(L^2_a(G),L^2_b(G)) .$ 
    \item We also define the space $L^\infty_{a,b}(\widehat{G})$ by the following condition:  $\sigma\in L^\infty_{a,b}(\widehat{G})$ if 
\[ 
    \Vert \pi(1+\mathcal{R})^{\frac{b}{\nu}}\sigma(\pi)\pi(1+\mathcal{R})^{-\frac{a}{\nu}} \Vert_{L^\infty(\widehat{G})}:=\sup_{\pi\in\widehat{G}}\Vert \pi(1+\mathcal{R})^{\frac{b}{\nu}}\sigma(\pi)\pi(1+\mathcal{R})^{-\frac{a}{\nu}} \Vert_{\mathscr{B}(H_\pi)}<\infty.
\]
\end{itemize} In this case $T_\sigma:L^2_a(G)\rightarrow L^2_b(G)$ extends to a bounded operator with 
\[ 
  \Vert \sigma\Vert _{L^\infty_{a,b}(\widehat{G})}= \Vert T_{\sigma} \Vert_{\mathscr{L}(L^2_a(G),L^2_b(G))},
\] and   $\sigma\in L^\infty_{a,b}(\widehat{G})$ if and only if $k:=\mathscr{F}_{G}^{-1}\sigma \in \mathcal{K}_{a,b}(G).$
\end{definition}
   With the previous definitions, we will introduce the type of symbols that we will use   further and under which the quantization formula make sense.
   
   \begin{definition}[Symbols and right-covolution kernels]\label{SRCK} A {symbol} is a field of operators $\{\sigma(x,\pi):H_\pi^\infty\rightarrow H_\pi,\,\,\pi\in\widehat{G}\},$ depending on $x\in G,$ such that 
   \[ 
       \sigma(x,\cdot)=\{\sigma(x,\pi):H_\pi^\infty\rightarrow H_\pi,\,\,\pi\in\widehat{G}\}\in L^\infty_{a,b}(\widehat{G})
   \]for some $a,b\in \mathbb{R}.$ The {right-convolution kernel} $k\in C^\infty(G,\mathscr{S}'(G))$ associated with $\sigma$ is defined, via the inverse Fourier transform on the group by
   \[ 
       x\mapsto k(x)\equiv k_{x}:=\mathscr{F}_{G}^{-1}(\sigma(x,\cdot)): G\rightarrow\mathscr{S}'(G).
   \]
   \end{definition}
   Definition \ref{SRCK} in this section allows us to establish the following theorem, which gives sense to the quantization of pseudo-differential operators in the graded setting (see Theorem 5.1.39 of \cite{FischerRuzhanskyBook}).
   \begin{theorem}\label{thetheoremofsymbol}
   Let us consider a symbol $\sigma$ and its associated right-convolution  kernel $k.$ For every $f\in \mathscr{S}(G),$ let us define the operator $A$ acting on $\mathscr{S}(G),$ via 
   \begin{equation}\label{pseudo}
       Af(x)=(f\ast k_{x})(x),\,\,\,\,\,x\in G.
   \end{equation}Then $Af\in C^\infty,$ and 
   \begin{equation}\label{Quantization}
       Af(x)=\int\limits_{\widehat{G}}\textnormal{\textbf{Tr}}(\pi(x)\sigma(x,\pi)\widehat{f}(\pi))d\pi.
   \end{equation}
   \end{theorem}
   
 Theorem   \ref{thetheoremofsymbol} motivates the following definition.
 \begin{definition}\label{DefiPSDO}
A continuous linear operator $A:C^\infty(G)\rightarrow\mathscr{D}'(G)$ with Schwartz kernel $K_A\in C^{\infty}(G)\widehat{\otimes}_{\pi} \mathscr{D}'(G),$ is a pseudo-differential operator, if
 there exists a  \textit{symbol}, which is a field of operators $\{\sigma(x,\pi):H_\pi^\infty\rightarrow H_\pi,\,\,\pi\in\widehat{G}\},$ depending on $x\in G,$ such that 
   \[ 
       \sigma(x,\cdot)=\{\sigma(x,\pi):H_\pi^\infty\rightarrow H_\pi,\,\,\pi\in\widehat{G}\}\in L^\infty_{a,b}(\widehat{G})
   \]for some $a,b\in \mathbb{R},$ such that, the Schwartz kernel of $A$ is given by 
\[ 
    K_{A}(x,y)=\int\limits_{\widehat{G}}\textnormal{\textbf{Tr}}(\pi(y^{-1}x)\sigma(x,\pi))d\pi=k_{x}(y^{-1}x).
\]
\end{definition}
Let $\mathcal{R}$ be a positive Rockland operator on a graded Lie group. Then $\mathcal{R}$  and $\pi(\mathcal{R}):=d\pi(\mathcal{R})$ (the infinitesimal representation of $\mathcal{R}$) are symmetric and  densely defined operators on $C^\infty_0 (G)$ and $H^\infty_\pi\subset H_\pi.$ We will denote by $\mathcal{R}$  and $\pi(\mathcal{R}):=d\pi(\mathcal{R})$ their self-adjoint extensions to $L^2(G)$ and $H_\pi$ respectively (see Proposition 4.1.5 and Corollary 4.1.16 of \cite[page 178]{FischerRuzhanskyBook}). 

\begin{remark}\label{symbolremark}
Let $\mathcal{R}$ be a positive  Rockland operator of  homogeneous degree $\nu$ on a graded Lie group $G.$ Every operator $  \pi(\mathcal{R}) $ has discrete spectrum (see ter Elst and Robinson \cite{TElst+Robinson}) admitting, by the spectral theorem, a basis  contained in its domain. In this case, $H_{\pi}^{\infty}\subset\textnormal{Dom}(\pi(\mathcal{R}))\subset H_{\pi},$ but in view of Proposition 4.1.5 and Corollary 4.1.16 of \cite[page 178]{FischerRuzhanskyBook}, every $ \pi(\mathcal{R})$ is densely defined and symmetric on $H^\infty_\pi,$ and this fact allows us to define the (restricted) domain of $\pi(\mathcal{R}),$ as \begin{equation}\label{RestrictedDomain}
    \textnormal{Dom}_{\textnormal{rest}}(\pi(\mathcal{R}))=H_{\pi}^{\infty}.   
\end{equation}Next, when we mention the domain of $\pi(\mathcal{R})$ we are referring to the restricted domain in  \eqref{RestrictedDomain}. This fact will be important, because, via the spectral theorem we can construct a basis for $H_{\pi},$ consisting of vectors in $\textnormal{Dom}_{\textnormal{rest}}(\pi(\mathcal{R}))=H_{\pi}^{\infty},$ where the operator $\pi(\mathcal{R})$ is diagonal. So, if  $B_\pi=\{e_{\pi,k}\}_{k=1}^\infty\subset H_{\pi}^{\infty},$ is  a basis such that $\pi(\mathcal{R})$  satisfies

$ \pi(\mathcal{R})e_{\pi,k}=\lambda_{\pi,k}e_{\pi,k},\,k\in \mathbb{N}, \,\,\,\pi \in \widehat{G}, $ for every $x\in G,$ the function $x\mapsto \pi(x)e_{\pi,k},$ is smooth and the family of functions
\begin{equation}\label{piij}
   \pi_{ij}:G\rightarrow \mathbb{C},\,\, \pi(x)_{ij}:=( \pi(x)e_{\pi,i},e_{\pi,j} )_{H_\pi},\,\,x\in G,
\end{equation} are smooth functions on $G.$  Consequently, for every continuous linear operator $A:C^\infty(G)\rightarrow C^\infty(G)$ we have \[ \{\pi_{ij}\}_{i,j=1}^\infty\subset \textnormal{Dom}(A)=C^\infty(G),\] for every $ \pi\in \widehat{G}. $
\end{remark} 

In view of Remark \ref{symbolremark}  we have the following theorem where we present the formula of a global symbol in terms of its corresponding pseudo-differential operator in the graded setting. The proof can be found in \cite[Section 3]{RuzhanskyDelgadoCardona2019}.

\begin{theorem}\label{SymbolPseudo}
Let  $\mathcal{R}$ be a positive  Rockland operator of  homogeneous degree $\nu$ on a graded Lie group $G.$ For every $\pi\in \widehat{G},$ let  $B_\pi=\{e_{\pi,k}\}_{k=1}^\infty\subset H_{\pi}^{\infty},$ be  a basis where the operator $\pi(\mathcal{R})$ is diagonal, i.e., 
\[ 
    \pi(\mathcal{R})e_{\pi,k}=\lambda_{\pi,k}e_{\pi,k},\,k\in \mathbb{N}, \,\,\,\pi \in \widehat{G}.
\] For every $x\in G,$ and $\pi\in \widehat{G},$ let us consider the functions $\pi(\cdot)_{ij}\in C^\infty(G)$ in \eqref{piij} induced by the coefficients of the matrix representation
 of $\pi(x)$ in the basis $B_\pi.$ If $A:C^\infty(G)\rightarrow C^\infty(G)$ is a continuous linear operator with symbol 
\begin{equation}\label{definition}
    \sigma:=\{\sigma(x,\pi)\in \mathscr{L}(H_{\pi}^\infty, H_\pi):\, x\in G,\,\pi\in \widehat{G}\},\, 
\end{equation} such that
\begin{equation}\label{quantization}
    Af(x)=\int\limits_{\widehat{G}}\textnormal{\textbf{Tr}}(\pi(x)\sigma(x,\pi)\widehat{f}(\pi))d\pi,
\end{equation} for every $f\in \mathscr{S}(G),$ and a.e. $(x,\pi),$ and if   $A\pi(x)$ is the  densely defined operator on $H_{\pi}^\infty,$ via
\begin{equation}\label{thesymbolofA}
  A\pi(x)\equiv  ((A\pi(x)e_{\pi,i},e_{\pi,j}))_{i,j=1}^\infty,\,\,\,(A\pi(x)e_{\pi,i},e_{\pi,j})=:(A\pi_{ij})(x),
\end{equation} then we have
\begin{equation}\label{formulasymbol}
    \sigma(x,\pi)=\pi(x)^*A\pi(x),
\end{equation} for every $x\in G,$ and a.e. $\pi\in \widehat{G}.$ 
\end{theorem}

 \subsection{Global H\"ormander classes $S^m_{\rho,\delta}$ on graded Lie groups}

The main tool in the construction of global H\"ormander classes is the notion of difference operators. 
 Indeed, for every smooth function $q\in C^\infty(G)$ and $\sigma\in L^\infty_{a,b}(G),$ where $a,b\in\mathbb{R},$ the difference operator $\Delta_q$ acts on $\sigma$ according to the formula (see Definition 5.2.1 of \cite{FischerRuzhanskyBook}),
 \[ 
     \Delta_q\sigma(\pi)\equiv [\Delta_q\sigma](\pi):=\mathscr{F}_{G}(qf)(\pi),\,\textnormal{ for  a.e.  }\pi\in\widehat{G},\textnormal{ where }f:=\mathscr{F}_G^{-1}\sigma\,\,.
 \]
We will reserve the notation $\Delta^{\alpha}$ for the difference operators defined by the functions $q_{\alpha}$ and $\tilde{q}_{\alpha}$ defined by $q_{\alpha}(x):=x^\alpha$ and $\tilde{q}_{\alpha}(x)=(x^{-1})^\alpha,$ respectively. In particular, we have the Leibniz rule,
\begin{equation}\label{differenceespcia;l}
    \Delta^{\alpha}(\sigma\tau)=\sum_{\alpha_1+\alpha_2=\alpha}c_{\alpha_1,\alpha_2}\Delta^{\alpha_1}(\sigma)\Delta^{\alpha_2}(\tau),\,\,\,\sigma,\tau\in L^\infty_{a,b}(\widehat{G}).
\end{equation} 
For our further analysis we will use the following property of the difference operators $\Delta^\alpha,$ (see e.g. \cite[page 20]{FischerFermanian-Kammerer2017}),
\begin{equation}\label{FischerFermanian-Kammerer2017}
    \Delta^{\alpha}(\sigma_{r\cdot})(\pi)=r^{|\alpha|}(\Delta^{\alpha}\sigma)(r\cdot \pi),\,\,\,r>0\,\,\,\pi\in\widehat{G},
\end{equation}where we have denoted \begin{equation}\label{eeeeeeeee}
    \sigma_{r\cdot}:=\{\sigma(r\cdot \pi):\pi\in\widehat{G}\},\,\,r\cdot \pi(x):=\pi(D_r(x)),\,\,x\in G. 
\end{equation}

\begin{definition}
In terms of difference operators, the global H\"ormander classes introduced in \cite{FischerRuzhanskyBook} can be defined as follows.
 Let $0\leq \delta,\rho\leq 1,$ and let $\mathcal{R}$ be a positive Rockland operator of homogeneous degree $\nu>0.$ If $m\in \mathbb{R},$ we say that the symbol $\sigma\in L^\infty_{a,b}(\widehat{G}), $ where $a,b\in\mathbb{R},$ belongs to the $(\rho,\delta)$-H\"ormander class of order $m,$ $S^m_{\rho,\delta}(G\times \widehat{G}),$ if for all $\gamma\in \mathbb{R},$ the following conditions
\begin{equation}\label{seminorm}
   p_{\alpha,\beta,\gamma,m}(\sigma)= \operatornamewithlimits{ess\, sup}_{(x,\pi)\in G\times \widehat{G}}\Vert \pi(1+\mathcal{R})^{\frac{\rho|\alpha|-\delta|\beta|-m-\gamma}{\nu}}[X_{x}^\beta \Delta^{\alpha}\sigma(x,\pi)] \pi(1+\mathcal{R})^{\frac{\gamma}{\nu}}\Vert_{\textnormal{op}}<\infty,
\end{equation}
hold true for all $\alpha$ and $\beta$ in $\mathbb{N}_0^n.$
\end{definition}
\begin{remark}\label{remarkgraded}
The resulting class $S^m_{\rho,\delta}(G\times \widehat{G}),$ does not depend on the choice of the Rockland operator $\mathcal{R}$ (see \cite[Page 306]{FischerRuzhanskyBook}). Moreover (see Theorem 5.5.20 of \cite{FischerRuzhanskyBook}), the following facts are equivalents: 
\begin{itemize}
    \item $\forall \alpha,\beta\in \mathbb{N}_{0}^n, \forall\gamma\in \mathbb{R}, $   $p_{\alpha,\beta,\gamma,m}(\sigma)<\infty.$
    
    \item  $\forall \alpha,\beta\in \mathbb{N}_{0}^n, $   $p_{\alpha,\beta,0,m}(\sigma)<\infty.$
    
    \item  $\forall \alpha,\beta\in \mathbb{N}_{0}^n, $   $p_{\alpha,\beta,m+\delta|\beta|-\rho|\alpha|,m}(\sigma)<\infty.$
     \item $\sigma\in {S}^{m}_{\rho,\delta}(G\times \widehat{G}).$
\end{itemize}
We will denote,
\begin{equation}
     \Vert \sigma\Vert_{k\,,\mathscr{S}^{m}_{\rho,\delta}}= \max_{|\alpha|+|\beta|\leq k}\{  p_{\alpha,\beta,0,m}(\sigma)\}.
\end{equation}
\end{remark}
By keeping in mind Remark \ref{remarkgraded}, we will improve  Theorem 5.5.20 of \cite{FischerRuzhanskyBook} proving  a characterization of  H\"ormander classes on graded Lie groups.
\begin{theorem}\label{corgraded}   Let $G$ be a graded Lie group of homogeneous dimension $Q,$ and let  $0\leqslant \delta,\rho\leqslant 1.$    The following conditions are equivalent: 
\begin{itemize}
    \item[(A).] $\forall \alpha,\beta\in \mathbb{N}_{0}^n, \forall \gamma\in \mathbb{R}, $   $p_{\alpha,\beta,\gamma,m}(\sigma)<\infty.$
    
    \item[(B).]  $\forall \alpha,\beta\in \mathbb{N}_{0}^n, $   $p_{\alpha,\beta,0,m}(\sigma)<\infty.$
    
    \item[(C).]  $\forall \alpha,\beta\in \mathbb{N}_{0}^n, $   $p_{\alpha,\beta,m+\delta|\beta|-\rho|\alpha|,m}(\sigma)<\infty.$
    \item[(D).]  $\forall \alpha,\beta\in \mathbb{N}_{0}^n, \exists \gamma_0\in \mathbb{R}, $   $p_{\alpha,\beta,\gamma_0,m}(\sigma)<\infty.$
    \item[(E).] $\sigma\in {S}^{m}_{\rho,\delta}(G\times \widehat{G}).$
\end{itemize}
\begin{proof} 
We only need to prove that 
$\textnormal{D}\Longrightarrow \textnormal{C}.$ Let us assume that $\textnormal{D}$ holds true for some $\gamma_0\in \mathbb{R}$ and let $\gamma\in \mathbb{R}$ be a real number. Let us assume first that $\gamma>\gamma_0.$ Let us define the operator $ \mathcal{Q}$ by the functional calculus in the following way
\[ 
    \mathcal{Q}:=(1+\mathcal{R})^{\frac{1}{\nu}}\equiv \int\limits_{0}^\infty (1+\lambda)^{\frac{1}{\nu}}dE_{\mathcal{R}}(\lambda),
\]where $\{dE_{\mathcal{R}}(\lambda)\}_{\lambda\geq 0}$ denotes the spectral resolution associated with $\mathcal{R}.$ Let us denote by $\{\pi(Q)\}$ the symbol of $Q,$ indexed by $\pi\in \widehat{G}$ except possibly on a subset of $\widehat{G}$ of null Plancherel measure. Let us note that the operator $A(\pi):=\pi(Q)^{(\gamma_0-\gamma)}$ is self-adjoint and bounded. Let us denote
\[ 
    X_{\alpha,\beta,\gamma_0}(x,\pi):=  \pi(Q)^{(\rho|\alpha|-\delta|\beta|-m-\gamma_0)}X_{x}^\beta \Delta^{\alpha}\sigma(x,\pi)\pi(Q)^{\gamma_0},
\]
which is a bounded operator on $H_{\pi}.$
From the Corach-Porta-Recht inequality \eqref{Recht2}, we have 
\begin{align*}
    &\Vert \pi(Q)^{(\rho|\alpha|-\delta|\beta|-m-\gamma)}X_x^\beta \Delta^{\alpha}\sigma(x,\pi)\pi(Q)^{\gamma}\Vert_{\textnormal{op}}\\
    &=\Vert \pi(Q)^{(\gamma_0-\gamma)} \pi(Q)^{(\rho|\alpha|-\delta|\beta|-m-\gamma_0)}X_x^\beta \Delta^{\alpha}\sigma(x,\pi)\pi(Q)^{\gamma_0}\pi(Q)^{(\gamma-\gamma_0)}\Vert_{\textnormal{op}} \\
    &=\Vert \pi(Q)^{(\gamma_0-\gamma)} X_{\alpha,\beta,\gamma_0}(x,\pi)\pi(Q)^{(\gamma-\gamma_0)}\Vert_{\textnormal{op}}\\
    &\leqslant \Vert A(\pi) \pi(Q)^{(\gamma_0-\gamma)}  X_{\alpha,\beta,\gamma_0}(x,\pi)\pi(Q)^{(\gamma-\gamma_0)}A(\pi)\\
    &+(1+A(\pi)^{2})^{\frac{1}{2}}\pi(Q)^{(\gamma_0-\gamma)} X_{\alpha,\beta,\gamma_0}(x,\pi)\pi(Q)^{(\gamma-\gamma_0)} (1+A(\pi)^{2})^{\frac{1}{2}} \Vert_{\textnormal{op}}\\
    &= \Vert \pi(Q)^{2(\gamma_0-\gamma)}  X_{\alpha,\beta,\gamma_0}(x,\pi)\\
   &+(1+\pi(Q)^{2(\gamma_0-\gamma)})^{\frac{1}{2}}\pi(Q)^{(\gamma_0-\gamma)} X_{\alpha,\beta,\gamma_0}(x,\pi)\pi(Q)^{(\gamma-\gamma_0)} (1+\pi(Q)^{2(\gamma_0-\gamma)})^{\frac{1}{2}} \Vert_{\textnormal{op}}.
    \end{align*} 
Taking into account that $\gamma_0-\gamma<0,$ and the functional calculus for real  powers of $\mathcal{R}$ and $I+\mathcal{R}$ (see  \cite[Page 319]{FischerRuzhanskyBook}) imply
\[ 
  \pi(Q)^{2(\gamma_0-\gamma)}\in {S}^{2(\gamma_0-\gamma)}_{1,0}(G\times\widehat{G}),\,\,\, \,(1+\pi(Q)^{2(\gamma_0-\gamma)})^{\frac{1}{2}} \pi(Q)^{(\gamma_0-\gamma)}\in {S}^{2(\gamma_0-\gamma)}_{1,0}(G\times\widehat{G}) ,
\]  
\[ 
    \pi(Q)^{(\gamma-\gamma_0)} (1+\pi(Q)^{2(\gamma_0-\gamma)})^{\frac{1}{2}} \in {S}^{0}_{1,0}(G\times\widehat{G}),
\]
from which we deduce that
\[ 
  \mathscr{D}_1:=  \sup_{\pi\in \widehat{G}} \Vert  \pi(Q)^{2(\gamma_0-\gamma)}\Vert_{\textnormal{op}}, \, \mathscr{D}_2:=  \sup_{\pi\in \widehat{G}} \Vert  (1+\pi(Q)^{2(\gamma_0-\gamma)})^{\frac{1}{2}} \pi(Q)^{2(\gamma_0-\gamma)}  \Vert_{\textnormal{op}}  <\infty,
\] and 
\[ 
  \mathscr{D}_3:=  \sup_{\pi\in \widehat{G}}\Vert \pi(Q)^{(\gamma-\gamma_0)} (1+\pi(Q)^{2(\gamma_0-\gamma)})^{\frac{1}{2}}\Vert_{\textnormal{op}}<\infty.
\]
Consequently, 
\begin{align*}
    &\Vert \pi(Q)^{(\rho|\alpha|-\delta|\beta|-m-\gamma)}X_x^\beta \Delta^{\alpha}\sigma(x,\pi)\pi(Q)^{\gamma}\Vert_{\textnormal{op}}\\
    &\leqslant  \Vert \pi(Q)^{2(\gamma_0-\gamma)}  X_{\alpha,\beta,\gamma_0}(x,\pi)\\
   &+(1+\pi(Q)^{2(\gamma_0-\gamma)})^{\frac{1}{2}}\pi(Q)^{(\gamma_0-\gamma)} X_{\alpha,\beta,\gamma_0}(x,\pi)\pi(Q)^{(\gamma-\gamma_0)} (1+\pi(Q)^{2(\gamma_0-\gamma)})^{\frac{1}{2}} \Vert_{\textnormal{op}}\\
   &\leqslant  \Vert \pi(Q)^{2(\gamma_0-\gamma)}\Vert_{\textnormal{op}} \Vert  X_{\alpha,\beta,\gamma_0}(x,\pi)\Vert_{\textnormal{op}}\\
   &+\Vert (1+\pi(Q)^{2(\gamma_0-\gamma)})^{\frac{1}{2}}\pi(Q)^{(\gamma_0-\gamma)}\Vert_{\textnormal{op}}\Vert X_{\alpha,\beta,\gamma_0}(x,\pi)\Vert_{\textnormal{op}}\\
   &\hspace{7cm}\times\Vert \pi(Q)^{(\gamma-\gamma_0)} (1+\pi(Q)^{2(\gamma_0-\gamma)})^{\frac{1}{2}} \Vert_{\textnormal{op}}\\
   &\leqslant (\mathscr{D}_1+\mathscr{D}_2\times \mathscr{D}_3)\Vert X_{\alpha,\beta,\gamma_0}(x,\pi)\Vert_{\textnormal{op}}.
\end{align*}This argument shows that $\textnormal{D}\Longrightarrow \textnormal{C}$ for $\gamma>\gamma_0.$ In the  case where $\gamma<\gamma_0,$ we can define $A(\pi)=\pi(Q)^{(\gamma-\gamma_0)}.$ By repeating the argument above we can  deduce that $\textnormal{D}\Longrightarrow \textnormal{C}$ for $\gamma<\gamma_0.$ Indeed, by using again the Corach-Porta-Recht inequality \eqref{Recht2}, we have 
\begin{align*}
    &\Vert \pi(Q)^{(\rho|\alpha|-\delta|\beta|-m-\gamma)}X_x^\beta \Delta^{\alpha}\sigma(x,\pi)\pi(Q)^{\gamma}\Vert_{\textnormal{op}}\\
    &=\Vert \pi(Q)^{(\gamma_0-\gamma)} \pi(Q)^{(\rho|\alpha|-\delta|\beta|-m-\gamma_0)}X_x^\beta \Delta^{\alpha}\sigma(x,\pi)\pi(Q)^{\gamma_0}\pi(Q)^{(\gamma-\gamma_0)}\Vert_{\textnormal{op}} \\
    &=\Vert \pi(Q)^{(\gamma_0-\gamma)} X_{\alpha,\beta,\gamma_0}(x,\pi)\pi(Q)^{(\gamma-\gamma_0)}\Vert_{\textnormal{op}}\\
    &\leqslant \Vert A(\pi) \pi(Q)^{(\gamma_0-\gamma)}  X_{\alpha,\beta,\gamma_0}(x,\pi)\pi(Q)^{(\gamma-\gamma_0)}A(\pi)\\
    &+(1+A(\pi)^{2})^{\frac{1}{2}}\pi(Q)^{(\gamma_0-\gamma)} X_{\alpha,\beta,\gamma_0}(x,\pi)\pi(Q)^{(\gamma-\gamma_0)} (1+A(\pi)^{2})^{\frac{1}{2}} \Vert_{\textnormal{op}}\\
    &= \Vert   X_{\alpha,\beta,\gamma_0}(x,\pi)\pi(Q)^{2(\gamma-\gamma_0)}\\
   &+(1+\pi(Q)^{2(\gamma-\gamma_0)})^{\frac{1}{2}}\pi(Q)^{(\gamma_0-\gamma)} X_{\alpha,\beta,\gamma_0}(x,\pi)\pi(Q)^{(\gamma-\gamma_0)} (1+\pi(Q)^{2(\gamma-\gamma_0)})^{\frac{1}{2}} \Vert_{\textnormal{op}}.
    \end{align*} 
Since $\gamma-\gamma_0$ is negative, and  by using again the functional calculus for real  powers of $\mathcal{R}$ and $I+\mathcal{R}$ (see  \cite[Page 319]{FischerRuzhanskyBook}) we have that
\[ 
  \pi(Q)^{2(\gamma-\gamma_0)}\in {S}^{{2}(\gamma-\gamma_0)}_{1,0}(G\times\widehat{G}), \,  \pi(Q)^{(\gamma-\gamma_0)} (1+\pi(Q)^{2(\gamma-\gamma_0)})^{\frac{1}{2}}        \in {S}^{{2(\gamma-\gamma_0)}}_{1,0}(G\times\widehat{G}) ,
\]  
\[ 
  (1+\pi(Q)^{2(\gamma-\gamma_0)})^{\frac{1}{2}} \pi(Q)^{(\gamma_0-\gamma)}   \in {S}^{0}_{1,0}(G\times\widehat{G}),
\]
and consequently we deduce that
\[ 
  \mathscr{D}'_1:=  \sup_{\pi\in \widehat{G}} \Vert   \pi(Q)^{2(\gamma-\gamma_0)}  \Vert_{\textnormal{op}}, \, \mathscr{D}_2':=  \sup_{\pi\in \widehat{G}} \Vert  (1+\pi(Q)^{2(\gamma-\gamma_0)})^{\frac{1}{2}} \pi(Q)^{(\gamma_0-\gamma)}    \Vert_{\textnormal{op}}  <\infty,
\] and 
\[ 
  \mathscr{D}_3':=  \sup_{\pi\in \widehat{G}}\Vert  \pi(Q)^{(\gamma-\gamma_0)} (1+\pi(Q)^{2(\gamma-\gamma_0)})^{\frac{1}{2}}    \Vert_{\textnormal{op}}<\infty.
\]
Consequently, 
\begin{align*}
    &\Vert \pi(Q)^{(\rho|\alpha|-\delta|\beta|-m-\gamma)}X_x^\beta \Delta^{\alpha}\sigma(x,\pi)\pi(Q)^{\gamma}\Vert_{\textnormal{op}}\\
    &\leqslant  \Vert   X_{\alpha,\beta,\gamma_0}(x,\pi) \pi(Q)^{2(\gamma-\gamma_0)}\\
   &+(1+\pi(Q)^{2(\gamma-\gamma_0)})^{\frac{1}{2}}\pi(Q)^{(\gamma_0-\gamma)} X_{\alpha,\beta,\gamma_0}(x,\pi)\pi(Q)^{(\gamma-\gamma_0)} (1+\pi(Q)^{2(\gamma-\gamma_0)})^{\frac{1}{2}} \Vert_{\textnormal{op}}\\
   &\leqslant  \Vert \pi(Q)^{2(\gamma-\gamma_0)}\Vert_{\textnormal{op}} \Vert  X_{\alpha,\beta,\gamma_0}(x,\pi)\Vert_{\textnormal{op}}\\
   &+\Vert (1+\pi(Q)^{2(\gamma-\gamma_0)})^{\frac{1}{2}}\pi(Q)^{(\gamma_0-\gamma)}\Vert_{\textnormal{op}}\Vert X_{\alpha,\beta,\gamma_0}(x,\pi)\Vert_{\textnormal{op}}\\
   &\hspace{7cm}\times\Vert \pi(Q)^{(\gamma-\gamma_0)} (1+\pi(Q)^{2(\gamma-\gamma_0)})^{\frac{1}{2}} \Vert_{\textnormal{op}}\\
   &\leqslant (\mathscr{D}_1'+\mathscr{D}_2'\times \mathscr{D}_3')\Vert X_{\alpha,\beta,\gamma_0}(x,\pi)\Vert_{\textnormal{op}}.
\end{align*}This argument shows that $\textnormal{D}\Longrightarrow \textnormal{C}$ for $\gamma_0>\gamma_0.$
The proof is complete.
\end{proof}
\end{theorem}
 \begin{remark}
In the case of the   $n$-dimensional Heisenberg group $G=\mathbb{H}_n,$ Theorem \ref{corgraded} was proved in \cite[Page 479]{FischerRuzhanskyBook}, relying on the  description of these H\"ormander classes in terms of the Shubin calculus. For general graded Lie groups, the equivalence of other conditions to condition (D) remained open. 
 \end{remark}
 
 We finish this subsection by also noting the following  Theorem \ref{MainTheoremGraded}, which is an extension of the classical Fefferman theorem to an arbitrary graded Lie group $G$ of homogeneous dimension $Q$ (see Theorem 1.2 in \cite{RuzhanskyDelgadoCardona2019}). 
\begin{theorem}\label{MainTheoremGraded}
Let $G$ be a graded Lie group of homogeneous dimension $Q.$ Let $A:C^\infty(G)\rightarrow\mathscr{D}'(G)$ be a pseudo-differential operator with symbol $\sigma\in S^{-m}_{\rho,\delta}(G\times \widehat{G} ),$ $0\leq \delta\leq \rho\leq 1,$ $\delta\neq 1.$ Then,
\begin{itemize}
    \item{\textnormal{(a)}} if ${m=\frac{Q  (1-\rho)  }{2}},$  then $A$ extends to a bounded operator from $L^\infty(G)$ to $ BMO(G),$ from the Hardy space $H^1(G)$ to $L^1(G),$ and  from $L^p(G)$ to $L^p(G)$ for all $1< p<\infty.$ 
    \item{\textnormal{(b)}} If $m\geq m_{p}:= Q(1-\rho)\left|\frac{1}{p}-\frac{1}{2}\right|,$ $1<p<\infty,$ then $A$ extends to a bounded operator from $ L^p(G)$ into $ L^p(G).$  
\end{itemize}
\end{theorem}
 \begin{remark}
 Theorem \ref{MainTheoremGraded} is an analogue of  Theorems \ref{parta} and \ref{parta2}.
  \end{remark}

\section{Appendix IV: Dependence of the subelliptic  calculus on the choice of sub-Laplacians} We have mentioned in Remark \ref{remarkgraded}, that in the case of a graded Lie group $G,$ 
the resulting class $S^m_{\rho,\delta}(G\times \widehat{G}),$ does not depend on the choice of the Rockland operator $\mathcal{R}.$ However, in the case of a compact Lie group we have mentioned without proof in Remark \ref{dependece1} that for  two sub-Laplacians, the corresponding subelliptic classes may not agree as we can see in the following remark for the case of $G=\textnormal{SU}(2)$. This may happen even when two sub-Laplacians are made from H\"ormander collections of  vector fields of the same step.
\begin{remark}\label{dependece2}Let us consider the positive sub-Laplacians $\mathcal{L}_1=-X_1^2-X_2^2$ and $\mathcal{L}_2=-X_2^2-X_3^2$ on $G=\textnormal{SU}(2)\cong \mathbb{S}^3,$ defined in Example \ref{SU2}.   The unitary dual of $\textnormal{SU}(2)$  can be identified as, (see \cite[Chapter 12]{Ruz})
\begin{equation}
\widehat{\textnormal{SU}}(2)\equiv \{ [t_{l}]:2l\in \mathbb{N}, d_{l}:=\dim t_{l}=(2l+1)\}.
\end{equation}
There are explicit formulae for $t_{l}$ as
functions of Euler angles in terms of the so-called Legendre-Jacobi polynomials, see \cite[Chapter 11]{Ruz}. In terms of the representations $t_l,$  it was shown in  \cite{RuzTurIMRN} that (by considering the positive Laplacian $\mathcal{L}_\textnormal{SU(2)}=-X_1^2-X_2^2-X_3^2$ on $\textnormal{SU(2)}$),
\begin{equation}\label{eq}
    \sigma_{\mathcal{L}_\textnormal{SU(2)}}(t_l)=\textnormal{diag}[ l(l+1)\delta_{mn}]_{m,n=-l}^{l},\quad \sigma_{X_3}(t_l)=\textnormal{diag}[ -in\delta_{mn}]_{m,n=-l}^{l},
\end{equation}where $\delta_{m,n}$ is the Kronecker-Delta. If a-priori we assume that  $S^{m,\mathcal{L}_1}_{\rho,\delta}(\textnormal{SU}(2)\times \frac{1}{2}\mathbb{N}_0)=S^{m,\mathcal{L}_2}_{\rho,\delta}(\textnormal{SU}(2)\times \frac{1}{2}\mathbb{N}_0),$  with $m\in \mathbb{R},$ and $0\leq \rho<\delta\leq 1,$ then we  would have that  $1+\mathcal{L}_2\in S^{2,\mathcal{L}_1}_{1,0}(\textnormal{SU}(2)\times \frac{1}{2}\mathbb{N}_0)=S^{2,\mathcal{L}_2}_{1,0}(\textnormal{SU}(2)\times \frac{1}{2}\mathbb{N}_0),$ which from the definition of the subelliptic classes implies that
\begin{equation}\label{this}
    \sup_{l\in \frac{1}{2}\mathbb{N}_0}\Vert (1+ \sigma_{\mathcal{L}_1}(t_l))^{-1} (1+ \sigma_{\mathcal{L}_2}(t_l)) \Vert_{\textnormal{op}}<\infty.
\end{equation}As a consequence of the Plancherel theorem, the previous inequality implies that $(1+\mathcal{L}_1)^{-1}(1+\mathcal{L}_2)$ is bounded on $L^2(\textnormal{SU}(2)).$ Let us note that
\begin{equation}
    (1+\mathcal{L}_1)^{-1}(1+\mathcal{L}_2)=(1+\mathcal{L}_1)^{-1}(1-X_2^2)-(1+\mathcal{L}_1)^{-1}X_{3}^2.
\end{equation}So, from the positivity of $(1+\mathcal{L}_1)^{-1},$ $(1-X_2^2)$ and $-(1+\mathcal{L}_1)^{-1}X_3^2$ we have that\footnote{Let $H$ be a Hilbert space. Then an operator $A:\textnormal{Dom}(A)\subset H\rightarrow H$ admitting a self-adjoint extension has a bounded extension, if and only if,  $\Vert A\Vert_{\mathscr{B}(H)}=\sup_{\Vert f\Vert_{H}=1}(Af,f)_{H}<\infty,$ where $(\cdot, \cdot)_H$ is the inner product of $H,$ and $\Vert \cdot\Vert_H$ the induced norm (see e.g. Weidmann \cite{Weid}). }
\begin{align*}
    \Vert (1+\mathcal{L}_1)^{-1}(1+\mathcal{L}_2)\Vert_{\mathscr{B}(L^2(\textnormal{SU}(2)))}&=\sup_{f\in L^2(\textnormal{SU}(2))}( (1+\mathcal{L}_1)^{-1}(1+\mathcal{L}_2)f,f)_{L^2(\textnormal{SU}(2))}\\
    &\geq \sup_{f\in L^2(\textnormal{SU}(2))}(-(1+\mathcal{L}_1)^{-1}X_3^2f,f)_{L^2(\textnormal{SU}(2))}\\
    &=\Vert (1+\mathcal{L}_1)^{-1}X_3^2\Vert_{\mathscr{B}(L^2(\textnormal{SU}(2)))},
\end{align*}which  implies that $(1+\mathcal{L}_1)^{-1}X_3^2,$ is bounded on  $L^2(\textnormal{SU}(2)),$ which indeed, is equivalent to say that 
\begin{equation}\label{su2}
    \sup_{l\in \frac{1}{2}\mathbb{N}_0}\Vert (1+ \sigma_{\mathcal{L}_1}(t_l))^{-1}  \sigma_{X_3^2}(t_l) \Vert_{\textnormal{op}}=\sup_{l\in \frac{1}{2}\mathbb{N}_0}\Vert (1+ \sigma_{\mathcal{L}_{\textnormal{SU}(2)}}(t_l)+\sigma_{X_3^2}(t_l))^{-1}  \sigma_{X_3^2}(t_l) \Vert_{\textnormal{op}}<\infty.
\end{equation}In terms of \eqref{eq}, \eqref{su2} implies that
\begin{align*}
   & \sup_{l\in \frac{1}{2}\mathbb{N}_0}\Vert (1+ \sigma_{\mathcal{L}_{\textnormal{SU}(2)}}(t_l)+\sigma_{X_3}(t_l))^{-1}  \sigma_{X_3}(t_l)^2 \Vert_{\textnormal{op}}\\
    &=  \sup_{l\in \frac{1}{2}\mathbb{N}_0}\Vert       \textnormal{diag}[ (l(l+1)-n^2)^{-1}n^{2}\delta_{mn}]_{m,n=-l}^{l}\Vert_{\textnormal{op}}   \asymp l\rightarrow\infty,\textnormal{ when  }l\rightarrow\infty.  
\end{align*}This shows that \eqref{this} does not hold. In consequence, $S^{2,\mathcal{L}_1}_{1,0}(\textnormal{SU}(2)\times \frac{1}{2}\mathbb{N}_0)\neq S^{2,\mathcal{L}_2}_{1,0}(\textnormal{SU}(2)\times \frac{1}{2}\mathbb{N}_0),$ which shows that the subelliptic calculus may depend on the choice of the sub-Laplacian.
\end{remark}
\begin{remark}\label{Sdependence}
As a consequence of the argument in Remark \ref{dependece2}  also subelliptic  Sobolev and Besov spaces may depend on the choice of a sub-Laplacian on a compact Lie group. Indeed, let us consider the case of $G=\textnormal{SU}(2),$ the sub-Laplacians $\mathcal{L}_1$ and $\mathcal{L}_2$ in Example \ref{SU2} and the subelliptic Sobolev spaces $L^{2,\mathcal{L}_1}_{-2}(\textnormal{SU}(2))$ and $L^{2,\mathcal{L}_2}_{-2}(\textnormal{SU}(2)).$ Let us define $f:=\mathscr{F}_{\textnormal{SU}(2)}^{-1}[\widehat{\mathcal{M}}_{2,2}],$ where $\mathcal{M}_{2,2}:=1+\mathcal{L}_2.$ Because
\[ 
   \sup_{l\in \frac{1}{2}\mathbb{N}_0} \Vert \widehat{\mathcal{M}}_{2,2}^{\,-1}(t_l)\widehat{f}(t_l) \Vert_{\textnormal{op}}=1,
\]we have that $f\in L^{2,\mathcal{L}_1}_{-2}(\textnormal{SU}(2)). $ If we assume that $L^{2,\mathcal{L}_1}_{-2}(\textnormal{SU}(2))=L^{2,\mathcal{L}_2}_{-2}(\textnormal{SU}(2)),$ then we could have that
\begin{equation}\label{finalremark}
 \sup_{l\in \frac{1}{2}\mathbb{N}_0} \Vert \widehat{\mathcal{M}}_{1,2}^{\,-1}(t_l)\widehat{f}(t_l) \Vert_{\textnormal{op}}<\infty,   
\end{equation}where  $\mathcal{M}_{1,2}:=1+\mathcal{L}_1.$ However, \eqref{finalremark} is equivalent to saying that $\Vert (1+\mathcal{L}_1)^{-1}(1+\mathcal{L}_2)\Vert_{\mathscr{B}(L^2(\textnormal{SU}(2)))}<\infty,$ which certainly, from Remark \ref{dependece2} is not possible. This analysis implies that $L^{2,\mathcal{L}_1}_{-2}(\textnormal{SU}(2))\neq L^{2,\mathcal{L}_2}_{-2}(\textnormal{SU}(2)).$ Because subelliptic Besov spaces can be obtained from the real interpolation between subelliptic Sobolev spaces (see Theorem 6.2 of \cite{CardonaRuzhansky2019I}) a similar argument as the one done in this remark, shows that subelliptic Besov spaces may depend on the choice of a sub-Laplacian on a compact Lie group. 
\end{remark}

  \par


\begin{thebibliography}{9}
  


\bibitem{Agrachev2008} Agrachev, A.,   Boscain, U.,   Gauthier, J. P., Rossi, F. The intrinsic hypoelliptic Laplacian and its heat kernel on unimodular
Lie groups, J. Funct. Anal. 255 (9), pp. 2190--2232, (2008).

\bibitem{Andruchow}   Andruchow, E., Corach, G.,  Stojanoff, D. Geometric operator inequalities, Linear Algebra Appl., 
258, pp. 295--310, (1997).

\bibitem{AF}  Asada, K., Fujiwara, D. On some oscillatory integral transformations in $L^2(\mathbb{R}^n),$ 
J. Math. (N.S.), Japan, \textbf{4}(2), pp. 299--361, (1978).

\bibitem{atiyabott1}  Atiyah, M. F., Bott, R. The index problem for manifolds with boundary.  Differential Analysis, Bombay Colloq.,  pp. 175--186 Oxford Univ. Press, London, (1964). 

\bibitem{atiyabott2}  Atiyah, M., Bott, R., Patodi, V. K. On the heat equation and the index theorem. Invent. Math. 19, pp. 279--330, (1973).

\bibitem{AS} Atiyah, M. F., Singer, I. M.  { The index of elliptic operators on compact manifolds.} Bull. Amer. Math. Soc. 69, pp. 422--433, (1963).

\bibitem{AS1}  Atiyah, M. F., Singer, I. M. The index of elliptic operators. I. Ann. of Math. (2) 87, pp. 484--530, (1968).

\bibitem{AS2}  Atiyah, M. F., Segal, G. The index of elliptic operators. II. Ann. of Math. (2) 87,  pp. 531--545, 531--545,  (1968).

\bibitem{AS3}  Atiyah, M. F., Singer, I. M. The index of elliptic operators. III. Ann. of Math. (2) 87, pp. 546--604, (1968). 


\bibitem{AS4}  Atiyah, M. F., Singer, I. M. The index of elliptic operators. IV. Ann. of Math. (2) 93,  pp. 119--138, (1971).

 \bibitem{AS5} Atiyah, M. F., Singer, I. M. The index of elliptic operators. V. Ann. of Math. (2) 93,  139--149, (1971).
 

\bibitem{Berge} Berge, S. M., Grong, E. A Lichnerowicz estimate for the spectral gap of a sub-Laplacian. (English summary)
Proc. Amer. Math. Soc. 147 , no. 12, pp. 5153--5166, (2019). 

\bibitem{Besov1} Besov, O. V.   On a family of function spaces. Embeddings theorems and applications [in Russian] Dokl. Akad. Nauk. SSSR., 126, pp. 1163--1165, (1959).

\bibitem{Besov2}  Besov, O. V.   On a family of function spaces in connection with  embeddings and extensions, [in Russian] Trudy. Mat. Inst. Steklov.  60, pp. 42--81, (1961).


\bibitem{Bathia}  Bhatia, R. Matrix Analysis. Springer-Verlag, New York, 1997.



\bibitem{Bismut2008}  Bismut, J. M.  The hypo elliptic Laplacian on a compact Lie group, J. Funct. Anal. 255, 2190--2232,  (2008).

\bibitem{BB} Bleecker, D., Booss-Bavnbek, B. Index theory with applications to mathematics and physics. International Press, Somerville, MA, 2013. 

\bibitem{Bramanti}  Bramanti, M. An invitation to hypoelliptic operators and H\"ormander's vector fields. Springer Briefs in Mathematics. Springer, Cham, 2014. xii+150 pp. ISBN: 978-3-319-02086-0; 978-3-319-02087-7.


\bibitem{Brezis} Brezis, H. Analyse fonctionnelle: th{\'e}orie et applications. Collection Math{\'e}matiques appliqu{\'e}es pour la ma{\^\i}trise. Ciarlet, P.G. and Lions, J.L. (Eds). Dunod, 1999. 

\bibitem{BrinkerWirth} Brinker, J., Wirth, J., Gelfand triples for the Kohn-Nirenberg quantization on homogeneous Lie groups. Advances in harmonic analysis and partial differential equations, 51--97, Trends Math., Birkhäuser/Springer, Cham. (2020
).
\bibitem{Calderon1} Calder\'on,  A. P.,     Vaillancourt, R. On   the   boundedness   of      pseudo-differential perators,  J.  Math.  Soc. Japan  23, pp. 374--378, (1971).

\bibitem{Calderon2} Calder\'on,  A. P.,     Vaillancourt, R.     A  class  of  bounded pseudo-differential  operators, Proc.  Nat. Acad.  Sci. USA  69,  pp. 1185--1187, (1972).

\bibitem{CalderonZygmund} Calder\'on, A. P., Zygmund, A. On singular integrals, Amer. J. Math., Vol. 78. 1956.

\bibitem{CMM} Carbonaro, A., Mauceri, G.,  Meda, S. $H^1$ and $BMO$ for certain locally doubling metric measure spaces, Ann. Sc. Norm. Super. Pisa Cl. Sci. (5)8, pp. 543--582, (2009).


\bibitem{Cardona1} Cardona, D. Besov continuity for Multipliers defined on compact Lie groups. Palest. J. Math.,   5(2), 35--44, (2016).  

\bibitem{Cardona2} Cardona, D. Besov continuity of pseudo-differential operators on compact Lie groups revisited.  C. R. Math. Acad. Sci. Paris  355(5), pp. 533--537,  (2017). 

\bibitem{Cardona22} Cardona, D. Nuclear pseudo-differential operators in Besov spaces on compact Lie groups. J. Fourier Anal. Appl. 23(5),  pp. 1238--1262, (2017).

\bibitem{Cardona3} Cardona, D. Continuity of pseudo-differential operators on Besov spaces on  compact homogeneous manifolds,  {J.  Pseudo-Differ. Oper.  Appl}., 9(4), pp. 861--880, (2018). 

\bibitem{control1} Cardona, D. Spectral inequalities for elliptic pseudo-differential operators on closed manifolds,  arXiv:2209.10690.

\bibitem{CardonaWodzicki} Cardona, D. The Wodzicki residue for pseudo-differential operators on compact Lie groups, to appear in Harmonic Analysis and Partial Differential Equations. Trends in Mathematics. 

\bibitem{Toft2023} Cardona, D., Chatzakou, M., Ruzhansky, M., Toft, J. Schatten-von Neumann properties for H\"ormander classes on compact Lie groups,  arXiv:2301.04044


\bibitem{CdC1} Cardona, D., Del Corral, C.   The Dixmier trace and the Wodzicki residue for pseudo-differential operators on compact manifolds, Rev. Integr. Temas. Mat.  38(1), pp. 67--79. (2020).

\bibitem{CdC2} Cardona, D., Del Corral, C. The Dixmier trace and the non-commutative residue for multipliers on compact manifolds. In: Georgiev V., Ozawa T., Ruzhansky M., Wirth J. (eds) Advances in Harmonic Analysis and Partial Differential Equations. Trends in Mathematics. Birkh\"auser, Cham. (2020).


\bibitem{CDR21b} Cardona, D. Delgado, J. Ruzhansky, M. Dixmier traces, Wodzicki residues, and determinants on compact Lie groups: the paradigm of the global quantisation.  arXiv:2105.14949



\bibitem{RuzhanskyDelgadoCardona2019} Cardona, D., Delgado, J., Ruzhansky, M. $L^p$-bounds for pseudo-differential operators on graded Lie groups.  Lp-bounds for pseudo-differential operators on graded Lie groups. J. Geom. Anal.   31, pp.  11603-11647, (2021).

\bibitem{RuzhanskyDelgadoCardona2022TracesDet} Cardona, D., Delgado, J., Ruzhansky, M. Determinants and Plemelj-Smithies formulas, Monatsh. Math.,  (199), pp. 459--482, (2022).

\bibitem{RuzhanskyDelgadoCardona2022JEv} Cardona, D., Delgado, J., Ruzhansky, M. Drift diffusion equations with fractional diffusion on compact Lie groups. J. Evol. Equ. No. 22,  84, (2022). 


\bibitem{RuzhanskyDelgadoCardona2022FC} Cardona, D., Delgado, J., Ruzhansky, M. Analytic functional calculus and G\r{a}rding inequality on graded Lie groups with applications to diffusion equations,  arXiv:2111.07469


\bibitem{control2} Cardona, D., Delgado, J., Ruzhansky, M. Estimates for sums of eigenfunctions of elliptic pseudo-differential operators on compact Lie groups,  arXiv:2209.12092.

\bibitem{localweyl} Cardona, D., Delgado, J., Ruzhansky, M. A note on the local Weyl formula on compact Lie groups,  arXiv:2210.00311.

\bibitem{Dyadic2023} Cardona, D., Delgado, J., Ruzhansky, M. Boundedness of the dyadic maximal function on graded Lie groups, arXiv:2301.08964.

\bibitem{Control3} Cardona D., Delgado J., Grajales, B., Ruzhansky M. Control of the Cauchy problem on Hilbert spaces: a global approach via symbol criteria, arXiv:2301.08999





\bibitem{CardonaFedericoRuzhansky} Cardona, D., Federico, S., Ruzhansky, M. Subelliptic sharp Gårding inequality on compact Lie groups,  arXiv:2110.00838


\bibitem{CdC3} Cardona, D. Kumar, V.,  Del Corral, C. Dixmier traces for discrete pseudo-differential operators. J. Pseudo-Differ. Oper. Appl.  11, pp. 647--656, (2020).


\bibitem{CRS2018} Cardona,  D., Messiouene, R.,  Senoussaoui, A.,  $L^p$-bounds for  Fourier integral operators on the torus. to appear in Complex Var. Elliptic Equ. arXiv:1807.09892

\bibitem{CRS2021:}Cardona, D., Messiouene, R., Senoussaoui, A. Periodic Fourier integral operators in Lp spaces. C. R. Math. Acad. Sci. Paris.  359 (5), pp. 547-553, (2021).



\bibitem{CardonaRuzhansky2017I} Cardona, D., Ruzhansky, M. Multipliers for Besov spaces on graded Lie groups.  C. R. Math. Acad. Sci. Paris.  355(4),  pp. 400--405, (2017).

\bibitem{CardonaRuzhansky2019I} Cardona, D., Ruzhansky, M. Boundedness of pseudo-differential operators in subelliptic Sobolev and Besov spaces on compact Lie groups, to appear in Complex Var. Elliptic Equ.  arXiv:1901.06825.

\bibitem{CardonaRuzhanksyFourierTriebel} Cardona, D., Ruzhansky, M. Fourier multipliers for Triebel-Lizorkin spaces on graded Lie groups,  arXiv:2101.05856.

\bibitem{CardonaRuzhanksyFourierTriebel2}  Cardona, D., Ruzhansky, M. Fourier multipliers for Triebel-Lizorkin spaces on compact Lie groups, Collect. Math.  73, pp.  477--504, (2022). 

\bibitem{CardonaRuzhansky2022FIO} Cardona, D., Ruzhansky, M. Sharpness of Seeger-Sogge-Stein orders for the weak (1,1) boundedness of Fourier integral operators,  119, pp. 189--198,  Arch. Math. (2022).

\bibitem{CardonaRuzhanskyOscil1} Cardona, D. Ruzhansky, M. Oscillating singular integral operators on graded Lie groups revisited,  arXiv:2201.12881

\bibitem{CardonaRuzhanskyOscil2} Cardona, D. Ruzhansky, M. Boundedness of oscillating singular integrals on Lie groups of polynomial growth,  arXiv:2201.12883

\bibitem{CardonaRuzhanskyOscil3} Cardona, D. Ruzhansky, M. Björk-Sjölin condition for strongly singular convolution operators on graded Lie groups, Math. Z. 302, pp. 1957--1981 (2022). 

\bibitem{CardonaRuzhanskyOscil4} Cardona, D. Ruzhansky, M. Oscillating singular integral operators on compact Lie groups revisited, Math. Z.  303, 26 (2023).




\bibitem{Connes94}  Connes A. Noncommutative geometry. Academic Press Inc., San Diego, CA (1994).

\bibitem{CorachPortaRecht90}   Corach, G., Porta, H.,    Recht, L. An operator inequality. Linear Algebra Appl. 142, pp. 153--158, (1990). 

\bibitem{CoRu1}  Coriasco, S.,  Ruzhansky, M. On the boundedness of Fourier integral operators on $L^p(\mathbb{R}^n),$ C. R. Math. Acad. Sci. Paris  {348}(15--16), pp. 847--851, (2010).

\bibitem{CoRu2} Coriasco, S.,  Ruzhansky, M.  Global $L^p$
continuity of Fourier integral operators, Trans.
Amer. Math. Soc. {366}(5), pp. 2575--2596, (2014).


\bibitem{Couhlon} Coulhon, T., Russ, E., Tardivel-Nachef, V., Sobolev algebras on Lie groups and Riemannian manifolds. Amer. J. Math. 123(2), pp. 283--342, (2001).


\bibitem{Chen} Chen, J., Fan, D. Central oscillating multipliers on compact Lie groups, Math. Z. 267, pp. 235-259, (2011).



\bibitem{CoifmandeGuzman} Coifman, R.R., De Guzm\'an, M. Singular integrals and multipliers on homogeneous spaces. Rev. un. Mat. Argentina, pp. 137--143, (1970).

\bibitem{CoifmanWeiss} Coifman, R., Weiss, G. Analyse harmonique non-commutative sur certains espaces homog\'enes. (French) \'Etude de certaines int\'egrales singuli\'eres. Lecture Notes in Mathematics,  242. Springer-Verlag, Berlin-New York, 1971. v+160 pp. 

\bibitem{Comech} Comech, A. Cotlar-Stein Almost Orthogonality Lemma, lecture note. Columbia University, New York, 1997. 

\bibitem{CowlingSikora}  Cowling, M.,  Sikora, A. A spectral multiplier theorem for a sublaplacian on SU(2). Math. Z. 238, pp. 1--36, (2001).

\bibitem{Christ}  Christ, M. Estimates for fundamental solutions of second-order subelliptic differential operators. Proc. Amer. Math. Soc. 105, pp. 166--172, (1989). 

\bibitem{ARLG} Dasgupta, A., Ruzhansky, M. The Gohberg lemma, compactness, and essential spectrum of operators on compact Lie groups. J. Anal. Math. 128, pp. 179--190, (2016).

\bibitem{Poincare} Delgado, J. A Poincar\'e determinant on the torus. J. Pseudo-differ. Oper. Appl. 13(3), No. 29, 13 pp. (2022). 

\bibitem{RuzhanskyDelgado2017} Delgado, J., Ruzhansky M., $L^p$-bounds for pseudo-differential operators on compact Lie groups,  J. Inst. Math. Jussieu 18(3), pp. 531--559, (2019). 

\bibitem{DelRuzTrace1} Delgado, J., Ruzhansky, M.  Fourier multipliers, symbols, and nuclearity on compact manifolds. J. Anal. Math. 135, no. 2, pp. 757--800, (2018).


\bibitem{DelRuzTrace11}  Delgado, J., Ruzhansky, M.  Schatten classes and traces on compact groups. Math. Res. Lett. 24, no. 4, pp. 979--1003, (2017).


\bibitem{DelRuzTrace111}  Delgado, J., Ruzhansky, M.   Kernel and symbol criteria for Schatten classes and r-nuclearity on compact manifolds. C. R. Math. Acad. Sci. Paris 352, no. 10, pp. 779--784, (2014).

\bibitem{DelRuzTrace1111}   Delgado, J., Ruzhansky, M. $L^p$-nuclearity, traces, and Grothendieck-Lidskii formula on compact Lie groups. J. Math. Pures Appl. (9) 102, no. 1, pp. 153--172, (2014).


\bibitem{deMoraes} de Moraes, W. A. A. Regularity of solutions to a Vekua-type equation on compact Lie groups Ann. Mat. Pura Appl. DOI: 10.1007/s10231-021-01120-7.

\bibitem{Dixmier} Dixmier, J.  Existence de traces non normales, C. R. Acad. Sci. Paris Series B, 262: 1107A--1108A, (1966).

\bibitem{Dixmier1953}  Dixmier, J. Formes lin\'eaires sur un anneau d'op\'erateurs. Bull. Soc. Math. France, 81, 9--39, (1953).


\bibitem{Dyer1970} Dyer, J. L. A nilpotent Lie algebra with nilpotent automorphism group. Bull. Amer. Math. Soc.
76, pp. 52--56, (1970).



\bibitem{Domokos} Domokos, A., Esquerra, R., Jaffa, B.,  Schulte, T. Subelliptic estimates on compact semisimple Lie groups. Nonlinear Analysis: Theory, Methods and Applications, 74(14), pp. 4642--4652, (2011).

\bibitem{Domokos2} Domokos, A., Manfredi,  J. $C^{1,\alpha}$-subelliptic regularity on $\textnormal{SU}(3)$ and compact, semi-simple Lie groups.
Anal. Math. Phys. 10, no. 1, Art. 4, 32 pp, (2020).


\bibitem{DuiHor} Duistermaat, J. J., H\"ormander, L. Fourier integral operators. II, Acta Math. \textbf{128}(3-4), pp. 183--269, (1972).


\bibitem{Duoandikoetxea2000} Duoandikoetxea, J. {Fourier Analysis. } {29}, American Mathematical Society, Providence (2000)

\bibitem{EbertWirth} Ebert, S., Wirth, J. Diffusive wavelets on groups and homogeneous spaces. Proc. Roy. Soc. Edinburgh Sect. A 141(3), pp. 497--520, (2011).

\bibitem{Eskin}   Eskin, G.I. Degenerate elliptic pseudodifferential equations of principal  type, Mat. Sb. (N.S.), {82}(124), pp. 585--628, (1970).



\bibitem{Fefferman1973} Fefferman, C., $L^p$-bounds for pseudo-differential operators, Israel J. Math. 14, pp. 413--417, (1973).


\bibitem{FedosovGolseLeichtnam} Fedosov, B. V., Golse, F., Leichtnam, E., Schrohe, E., The noncommutative residue for manifolds with boundary, J. Funct. Anal. 142, 1996, no 1, pp 1--31. DOI 10.1006/jfan.1996.0142; zbl 0877.58005; MR1419415.

\bibitem{Fischer2015} Fischer, V. Intrinsic pseudo-differential calculi on any compact Lie group. J. Funct. Anal.,
268, pp. 3404--3477, (2015).

\bibitem{Fischertraces1}  Fischer, V. Real trace expansions. Doc. Math. 24, pp. 2159--2202, (2019). 

\bibitem{Fischertraces2} Fischer, V. Local and global symbols on compact Lie groups. J. Pseudo-Differ. Oper. Appl.  doi:10.1007/s11868-019-00299-x., (2019).


\bibitem{FischerFermanian-Kammerer2017}   Fischer, V.   Fermanian-Kammerer, C. Defect measures on graded Lie groups. To appear in  Ann. Sc. Norm. Super. Pisa Cl. Sci. arXiv:1707.04002.


\bibitem{FischerRuzhanskyBook} Fischer V., Ruzhansky M., Quantization on nilpotent Lie groups, Progress in Mathematics,  314, Birkhauser, 2016. xiii+557pp.

\bibitem{FischerDiff} Fischer, V. Differential structure on the dual of a compact Lie group, arXiv:1610.06348.


 
\bibitem{FollandStein1982} Folland, G., Stein, E. Hardy Spaces on Homogeneous Groups, Princeton University
Press, Princeton, N.J., 1982.


\bibitem{Fuji} Fujiwara, A. Construction of the fundamental solution for the Schr\"odinger equations, Proc. Japan Acad. Ser. A Math. Sc, {55}(1), pp. 10--14, (1979).

\bibitem{furioli}  Furioli, G.,  Melzi, C.,  Veneruso, A. Littlewood-Paley decompositions and Besov spaces on Lie groups of polynomial growth. Math. Nachr.  279(9-10), pp. 1028--1040, (2006).

\bibitem{Garetto} Garetto, C., $L^p$ and Sobolev boundedness of pseudodifferential operators with
non-regular symbol: A regularisation approach. J. Math. Anal. Appl.   381, pp. 328--343, (2011). 


\bibitem{GarettoRuzhansky2015} Garetto, C.,  Ruzhansky, M. Wave equation for sum of squares on compact Lie groups,
J. Differential Equations.  258,  pp. 4324--4347, (2015).

\bibitem{Gilkey} Gilkey, P.  Invariance theory, the equation and the Atiyah--Singer index
theorem, Publish or Perish, Wilmington, 1984.

\bibitem{Goh60}  Gohberg, I. On the theory of multidimensional singular integral equations, Soviet Math. Dokl. 1, pp. 960--963, (1960).

\bibitem{GL}  Gordina, M., Laetsch, T. Sub-Laplacians on sub-Riemannian manifolds. Potential Anal. 44, no. 4, pp. 811--837, (2016). 

\bibitem{GRO} Grothendieck, A. Produits tensoriels topologiques et espaces nucl\'eaires, Memoirs
Amer. Math. Soc. 16, Providence, 1955 (Thesis, Nancy, 1953).


\bibitem{GrubbSchrohe} Grubb, G., Schrohe, E. Traces and quasi-traces on the Boutet de Monvel algebra. Ann. Inst. Fourier(Grenoble), 54(5), pp. 1641--1696  xvii, xxii., (2004).

\bibitem{HK16} Hassannezhad, A., Kokarev, G. Sub-Laplacian eigenvalue bounds on sub-Riemannian
manifolds. Ann. Scuola Norm. Sup. Pisa Cl. Sci. XVI(4), pp. 1049--1092, (2016).


\bibitem{HelfferNier} Helffer, B.,  Nier.,  F. Hypoelliptic estimates and spectral theory
for Fokker-Planck operators and Witten Laplacians, volume 1862 of
Lecture Notes in Mathematics. Springer-Verlag, Berlin, 2005.


\bibitem{Hirschman1956}  Hirschman, I. I., Multiplier transformations I,  Duke Math. J.  pp. 222--242, (1956).

\bibitem{Hi} Hirzebruch, F. Neue topologische methoden in der algebraischen Geometrie. Springer-Verlag, Berlin-G\"ottingen-Heidelberg,   1956.

\bibitem{Hormander1967} H\"ormander, L. Hypoelliptic second order differential equations, Acta Math.   119, pp. 147--171, (1967).

\bibitem{Hor71}   H\"ormander, L. Fourier integral operators. I,  Acta Math. \textbf{127}(1-2), pp. 79--183, (1971).

\bibitem{Hormander1985III} H\"ormander, L. { The Analysis of the linear partial differential operators} Vol. III. Springer-Verlag, (1985)



\bibitem{JerisonSanchezCalle}  Jerison, D.,  S\'anchez-Calle, A. Subelliptic second order differential operators, in Complex analysis III, Springer, pp. 46--77,  (1987).

\bibitem{Johnson1975} Johnson, J. W. Homogeneous Lie algebras and expanding automorphisms. Proc. Amer. Math. Soc. 48, 292--296, (1975).

\bibitem{KohnNirenberg1965}  Kohn, J.J., Nirenberg, L.   An algebra of pseudo‐differential operators. Commun. Pure and Appl. Math. 18, 269--305, (1965).

\bibitem{KV} Kontsevich, M., Vishik, S. Geometry of determinants of elliptic operators. In: Functional Analysis on the Eve of the 21st Century, Vol. 1 (New Brunswick, NJ, 1993), Progr. Math., vol. 131, pp. 173--197.Birkh\"auser, Boston (1995).


 \bibitem{Lesch} Lesch, M. On the Noncommutative Residue for Pseudodifferential Operators with log-Polyhomogeneous Symbols, Ann. Glob. Anal. Geom. 17, pp. 151--187, (1999).




\bibitem{M1} M\u{a}ntoiu, M., Ruzhansky, M. Quantizations on nilpotent Lie groups and algebras having flat coadjoint orbits. J. Geom. Anal. 29, no. 3, pp. 2823--2861, (2019). 



\bibitem{M2} M\u{a}ntoiu, M.,   Ruzhansky, M. Pseudo-differential operators, Wigner transform and Weyl systems on type I locally compact groups. Doc. Math. 22,  pp. 1539--1592, (2017).

\bibitem{Melrose}  Melrose, R. Propagation for the wave group of a positive subelliptic second-order differential operator, in: Hyperbolic equations and related topics (Katata/Kyoto, 1984), Academic Press, Boston, MA, pp. 181--192, (1986).

\bibitem{MartiniOtaziVallarino}  Martini, A.,  Ottazzi, A.,    Vallarino, M. A multiplier theorem for sub-Laplacians with drift on Lie groups, Rev. Mat. Iberoam. 35(5), pp. 1501--1534, (2019).

\bibitem{Miyachi}  Miyachi, A. On some estimates for the wave equation in $L^
p$ and $H^p$, J. Fac. Sci. Univ. Tokyo Sect. IA Math. {27}(2), pp. 331--354, (1998).

\bibitem{Mihlin} Mihlin, S. G., Singular integral equations, Uspehi Mat. Nauk,  3, So. 25, pp. 29--112, (1948); New York Univ., Courant Inst. Math. Sci., 1963. Amer. Math. Sac. translation  24,  (1950).

\bibitem{Montgomery}  Montgomery, R. A Tour of Sub-Riemannian Geometries, Their Geodesics and Applications, Math. Surveys
Monogr., vol. 91, Amer. Math. Soc. Providence, RI, (2002).

\bibitem{Mo}  Molahajloo, S. A characterization of compact pseudo-differential operators on $\mathbb{S}^1$, in Pseudo-differential Operators:  Analysis, Applications and Computations, Birkh\"auser/Springer Basel AG, Basel, pp. 25--29, (2011).

\bibitem{NagelStein78}  Nagel, A. Stein, E. M. Some new classes of pseudodifferential operators. Harmonic analysis in Euclidean spaces (Proc. Sympos. Pure Math., Williams Coll., Williamstown, Mass., 1978), Part 2, pp. 159--169, Proc. Sympos. Pure Math., XXXV, Part, Amer. Math. Soc., Providence, R.I., (1979).

\bibitem{NF} Noether, F. \"Uber eine Klasse singul\"arer Integralgleichungen. Math. Ann. 82, 42--63, (1921).




\bibitem{NurRuzTikhBesov2015} Nursultanov, E.,  Ruzhansky, M.,  Tikhonov, S. Nikolskii inequality and functional classes on compact Lie groups, Funct. Anal. Appl.  49, pp. 226--229, (2015).



\bibitem{NurRuzTikhBesov2017}  Nursultanov, E.,  Ruzhansky, M.,  Tikhonov S. Nikolskii inequality and Besov, Triebel-Lizorkin, Wiener and Beurling spaces on compact homogeneous manifolds, Ann. Sc. Norm. Super. Pisa Cl. Sci.,  XVI, pp. 981--1017, (2016).

\bibitem{Paycha} Paycha, S. Regularised Integrals, Sums and Traces, University Lecture Series, 59. American Mathematical Society, Providence, RI, An analytic point of view, (2012).

\bibitem{Peetre1} Peetre, J. Sur les espaces de Besov, C. R. Acad. Sci. Paris  264, pp. 281--283, (1967).

\bibitem{Peetre2} Peetre, J. Remarques sur les espaces de Besov. Le case $0<p<1,$ C. R. Acad. Sci. Paris. 277, pp. 947--950, (1973).

\bibitem{Peral} Peral,  J. C. $L^p$-estimates for the wave equation, J. Funct. Anal. {36}(1), pp. 114--145, (1980).

\bibitem{Person}  Persson, A. Compact linear mappings between interpolation spaces. Ark. Mat. 5, pp. 215--219, (1964).

\bibitem{P} Pietsch, A. Operator ideals. Mathematische Monographien, 16. VEB Deutscher Verlag der Wissenschaften, Berlin, 1978. 

\bibitem{P2} Pietsch, A. History of Banach spaces and linear operators. Birkh\"auser Boston, Inc., Boston, MA, 2007.


\bibitem{RodriguezRuzhansky2020} Rodriguez Torijano, C. A., Ruzhansky, M. Subelliptic wave equations with log-Lipschitz coefficients, arXiv:2007.09396. 

\bibitem{RothschildStein76} Rothschild, L. P.,   Stein, E. M. Hypoelliptic differential operators and nilpotent groups.
Acta Math.,  137(3-4), pp. 247--320, (1976).


\bibitem{M. Ruzhansky} Ruzhansky, M. Regularity theory of Fourier integral operators with complex phases and singularities of affine fibrations, Volume 131 of CWI Tract, Stichting Mathematisch Centrum, Centrum voor Wiskunde en Informatica, Amsterdam, 2001.

\bibitem{Ruz-Tok} Ruzhansky, M., Tokmagambetov, N., Nonharmonic analysis of boundary value problems, { Int. Math. Res. Notices}, (12), 3548--3615, (2016).

\bibitem{ProfRuzM:TokN:20017} Ruzhansky, M., Tokmagambetov, N. Nonharmonic analysis of boundary value problems without WZ condition, { Math. Model. Nat. Phenom.}, 12, 115--140, (2017).


\bibitem{Ruz}  Ruzhansky, M.,  Turunen, V. {Pseudo-differential Operators and Symmetries: Background Analysis and Advanced Topics } Birkh\"auser-Verlag, Basel, (2010).

\bibitem{RuzhanskyTurunenWirth2014} Ruzhansky, M., Turunen, V., Wirth J., H\"ormander class of pseudo-differential operators on compact Lie groups and global hypoellipticity, J. Fourier Anal. Appl.  20, pp. 476--499, (2014).

\bibitem{RuzTurIMRN} Ruzhansky, M., Turunen, V. Global quantization of pseudo-differential operators on compact Lie groups, SU(2) and 3-sphere, Int. Math. Res. Not. IMRN.   11, pp. 2439--2496, (2013).


\bibitem{RuzVR} Ruzhansky, M., Velasquez-Rodriguez, J. P., Non-harmonic Gohberg’s lemma, Gershgorin theory and heat equation on manifolds with boundary. Math. Nachr. 294, 1783--1820, (2021).

\bibitem{RuzhanskyWirth2014} Ruzhansky, M., Wirth, J. Global functional calculus for operators on compact Lie groups, J. Funct. Anal.  267,  144--172, (2014).

\bibitem{RuzhanskyWirth2015} Ruzhansky, M., Wirth, J. $L^p$ Fourier multipliers on compact Lie groups, Math. Z. 280, pp. 621--642, (2015).


\bibitem{RuzSugi01}   Ruzhansky, M., Sugimoto, M. Global $L^2$-boundedness theorems for a class of Fourier integral operators. Comm. Partial Differential Equations, \textbf{31}(4--6), pp. 547--569, (2006).

\bibitem{RuzSugi2019}  Ruzhansky, M., Sugimoto, M. A local-to-global boundedness argument and Fourier integral operators. J. Math. Anal. Appl. 473(2), pp. 892--904, (2019).

\bibitem{Scott} Scott, S. Traces and determinants of pseudodifferential operators, Oxford Mathematical Monographs, Oxford University Press, Oxford, 2010. zbl 1216.35192; MR2683288.

\bibitem{SSS91}  Seeger, A., Sogge, C. D.,   Stein, E. M. Regularity properties of Fourier integral operators, Ann. of Math. (2), {134}(2), pp. 231--251, (1991).

\bibitem{Seddik} Seddik, A. Some results related to the Corach-Porta-Recht inequality. Proc. Amer. Math. Soc. 129, no. 10, pp. 3009--3015, (2001).

\bibitem{Sukochev}  Sukochev, F., Usachev, A.  Dixmier traces and non--commutative analysis. J. Geom. Phys. 105, pp. 102--122, (2016). 

\bibitem{Tao} Tao, T. The weak-type $(1,1)$ of Fourier integral operators of order $-(n - 1)/2,$  J. Aust. Math. Soc. {76}(1), pp. 1--21, (2004).

\bibitem{Taylorbook1981} Taylor, M.  Pseudodifferential Operators, Princeton Univ. Press, Princeton, N.J., (1981).


\bibitem{taylorNC}Taylor, M. Noncommutative Harmonic Analysis. Mathematical Surveys and Monographs, vol. 22. American Mathematical Society, Providence, RI, (1986).


\bibitem{TElst+Robinson}
 { ter Elst, A. F. M.,  Robinson, D. W.}
 {Spectral estimates for positive {R}ockland operators},
 in {Algebraic groups and {L}ie groups},
  {Austral. Math. Soc. Lect. Ser.}
  {9},
    {195--213},
 {Cambridge Univ. Press}, (1997).


\bibitem{Triebel1983} Triebel, H. Theory of function spaces, vol. 78 of Monographs in Mathematics.
Birkh\"auser Verlag, Basel, (1983).

\bibitem{Triebel2006} Triebel, H. Theory of function spaces. III, volume 100 of Monographs in Mathematics.
Birkh\"auser Verlag, Basel, (2006).



\bibitem{VR} Velasquez-Rodriguez, J. P. On some spectral properties of pseudo-differential operators on $\mathbb{T}$.  J. Fourier Anal. Appl. 25, 2703–2732 (2019).

\bibitem{Wainger1965}  Wainger, S.  Special trigonometric series in $k$-dimensions, Mem. Amer. Math. Soc. 59, (1965).

\bibitem{Weid}  Weidmann, J. Linear operators in Hilbert spaces. Translated from the German by Joseph Sz\"ucs. Graduate Texts in Mathematics, 68. Springer-Verlag, New York-Berlin, (1980). 


\bibitem{Wodzicki} Wodzicki, M. Noncommutative Residue. I. Fundamentals, K-theory, Lecture Notes in Math., 1289, Springer, Berlin, 320--399, (1987).

\bibitem{Zhang} Zhang, Y.    Strichartz estimates for the Schr\"odinger flow on compact Lie groups. Analysis $\&$ PDE. 13, 1173--1219, (2020).
  \end{thebibliography}
\end{document}